\documentclass[10.5pt]{article}
\usepackage[utf8]{inputenc}
\usepackage{amsfonts}
\usepackage[symbol]{footmisc}
\usepackage{psfrag,epsf}
\usepackage[numbers]{natbib}
\usepackage{authblk}
\usepackage{caption}
\usepackage{fullpage}
\usepackage{graphicx, color, subcaption, float, hyperref, enumerate} 
\usepackage{ulem}
\hypersetup{
    colorlinks=true,
    linkcolor=blue, 
    citecolor=blue, 
}

\usepackage{amsmath,amssymb,mathtools,amsthm, bm, bbm}
\usepackage[ruled,vlined]{algorithm2e}

\numberwithin{equation}{section}
\mathtoolsset{showonlyrefs}

\allowdisplaybreaks

\newtheorem{thm}{Theorem}[section]

\newtheorem{cor}[thm]{Corollary}
\newtheorem{lemma}[thm]{Lemma}
\newtheorem{prop}[thm]{Proposition}
\newtheorem{definition}{Definition}[section]
\newtheorem{assumption}{Assumption}
\newtheorem*{remark}{Remark}
   
\newcommand{\PP}{\mathbb{P}}
\newcommand{\Var}{\textnormal{Var}}
\newcommand{\Cov}{\textnormal{Cov}}
\newcommand{\op}{\textnormal{op}}
\newcommand{\Imm}{\textnormal{Im}}
\newcommand{\Tr}{\textnormal{Tr}}
\newcommand{\dd}{\textnormal{d}}

\DeclareMathOperator*{\argmin}{arg\,min}

\newcommand{\mcl}{\mathcal}
\newcommand{\mcn}{\mathcal{N}}

\newcommand{\mbb}{\mathbb}

\newcommand{\ba}{\bm{a}}
\newcommand{\bb}{\bm{b}}

\newcommand{\bh}{\bm{h}}

\newcommand{\bu}{\bm{u}}
\newcommand{\bv}{\bm{v}}

\newcommand{\bx}{\bm{x}}

\newcommand{\by}{\bm{y}}

\newcommand{\fa}{\mathfrak{a}}
\newcommand{\fb}{\mathfrak{b}}
\newcommand{\fc}{\mathfrak{c}}

\newcommand{\bA}{\bm{A}}
\newcommand{\bB}{\bm{B}}

\newcommand{\bD}{\bm{D}}

\newcommand{\bF}{\bm{F}}
\newcommand{\bG}{\bm{G}}
\newcommand{\bH}{\bm{H}}
\newcommand{\bI}{\bm{I}}
\newcommand{\bJ}{\bm{J}}
\newcommand{\bM}{\bm{M}}

\newcommand{\bU}{\bm{U}}
\newcommand{\bV}{\bm{V}}
\newcommand{\bW}{\bm{W}}
\newcommand{\bX}{\bm{X}}

\newcommand{\bZ}{\bm{Z}}

\newcommand{\bbeta}{\bm{\beta}}

\newcommand{\bdelta}{\bm{\delta}}

\newcommand{\bep}{\bm{\epsilon}}
\newcommand{\bxi}{\bm{\xi}}
\newcommand{\bzeta}{\bm{\zeta}}

\newcommand{\bSigma}{\bm{\Sigma}}
\newcommand{\bLambda}{\bm{\Lambda}}
\newcommand{\btheta}{\bm{\theta}}

\newcommand{\R}{\mbb{R}}
\newcommand{\C}{\mbb{C}}

\newcommand{\E}{\mbb{E}}

\newcommand{\mI}{\mcl{I}}

\newcommand{\lcon}{\lesssim}

\newif\ifappendixrefs
\appendixrefstrue   

\newcommand{\appendixref}[2]{%
  \ifappendixrefs
    \ref{#1}%
  \else
    #2%
  \fi
}

\allowdisplaybreaks

\begin{document}
\title{\textbf{Generalization error of min-norm interpolators in transfer learning}}

\author[1,\textdagger]{Yanke Song}
\author[2,\textdagger]{Kenneth Gu}
\author[3,*]{Sohom Bhattacharya}
\author[1,*,$\ddagger$]{Pragya Sur}
\affil[1]{Department of Statistics, Harvard University}
\affil[2]{Department of Statistics, Stanford University}
\affil[3]{Department of Statistics, University of Florida}
\makeatletter
\renewcommand\AB@affilsepx{\\ \protect\Affilfont}
\renewcommand\Authsep{, }
\renewcommand\Authand{, }
\renewcommand\Authands{, }
\makeatother

\maketitle
\footnotetext[2]{co-first authors.}
\footnotetext[1]{co-senior authors.}
\footnotetext[3]{Emails: \href{mailto:ysong@g.harvard.edu}{\textit{ysong@g.harvard.edu}};  \href{mailto:kennygu@stanford.edu}{\textit{kennygu@stanford.edu}}; \href{mailto:bhattacharya.s@ufl.edu}{\textit{bhattacharya.s@ufl.edu}}; \href{mailto:pragya@fas.harvard.edu}{\textit{pragya@fas.harvard.edu}}}

\begin{abstract}
This paper establishes the generalization error of pooled min-$\ell_2$-norm interpolation in transfer learning, where data from diverse distributions are available. Min-norm interpolators arise naturally as implicit regularized limits of modern machine learning algorithms. Prior work has characterized their out-of-distribution risk when samples from the test distribution are unavailable during training. In many applications, however, limited test samples may be available at training time, yet properties of min-norm interpolation in this regime remain poorly understood. We address this gap by characterizing the bias and variance of pooled min-$\ell_2$-norm interpolation under both covariate shift and model shift. Our results yield several important implications. In certain cases under model shift, we show that adding data always hurts when the signal-to-noise ratio (SNR) is low. At higher SNR levels, transfer learning is beneficial provided the shift-to-signal ratio falls below a threshold that we characterize explicitly. Under covariate shift, we find that when the source sample size is small relative to the dimension, greater heterogeneity between domains reduces risk, and vice versa. While our model shift results are initially established for Gaussian designs, we extend them to more general designs through a universality argument. To illustrate the broader applicability of our technical tools beyond interpolation learning, we characterize the risk of a bias-corrected estimator that uses the pooled interpolator as an initialization and corrects the resulting bias with target data. On the technical side, we develop a novel anisotropic local law and a Lindeberg-swapping argument, yielding tools that may be of independent interest in random matrix theory and universality analysis. Finally, we supplement our theory with simulations demonstrating the finite-sample efficacy of our results.
\end{abstract}

\section{Introduction}
Recent advancements in deep learning methodology have uncovered a surprising phenomenon that defies conventional statistical wisdom: overfitting can yield remarkably effective generalization \cite{bartlett2020benign,belkin2019reconciling,belkin2019does,belkin2020two}. In the overparametrized regime, complex models that achieve zero-training error, i.e.~models that interpolate the training data, have gained popularity \cite{ghorbani2021linearized,hastie2022surprises,liang2020just}. However, as researchers grapple with increasingly large and heterogeneous datasets, the imperative  to effectively integrate diverse sources of information has become crucial. Transfer learning~\cite{torrey2010transfer} has emerged as a promising approach to address this challenge, allowing one to leverage knowledge from related datasets to boost performance on target problems. Yet, in the context of overparametrization, a crucial question arises: how should one leverage transfer learning in the presence of overfitting? Specifically, should diverse datasets be aggregated to construct a unified interpolator, or should they be handled individually? Intuition suggests that for closely related datasets, a pooled interpolator may yield superior performance compared to those derived from individual datasets. Conversely, for unrelated datasets, the opposite might be true. This paper investigates the interplay between interpolation risk and task relatedness, focusing particularly on overparametrized linear models and min-norm interpolation.

To achieve this, we explore two common ways datasets can differ: covariate shift, where the covariate distribution varies across different datasets, and model shift, where the conditional distribution of the outcome given the covariates differs across datasets. Covariate shift has been extensively studied through concepts like  likelihood ratios~\cite{ma2023optimally} and transfer exponents~\cite{hanneke2019value,kpotufe2021marginal,patil2024optimal}, among others. Similarly, a substantial body of literature has studied model shift problems~\cite{bastani2021predicting,li2022transfer,li2023transfer,maity2022minimax,tian2022transfer,patil2024optimal}. Other recent work has also studied overparametrized transfer learning in linear regression, including approaches that regularize toward a source-based estimator \cite{dar2021common} or transfer specific parameters learned from a source task \cite{dar2022double}.

However, the aforementioned literature has primarily considered statistical ideas where explicit penalties are added to loss functions to aid learning in high dimensions or non-parametric approaches. 
In contrast, modern machine learning has seen growing interest in and success with implicit regularization---the phenomenon whereby algorithms, under suitable initialization and step sizes, converge to predictors that generalize well despite overparameterization.  In this context, min-norm interpolators have emerged as an important class of predictors that arise commonly as implicit regularized limits of these algorithms \cite{bartlett2020benign,deng2022model,gunasekar2018characterizing,gunasekar2018implicit,liang2022precise,montanari2019generalization,muthukumar2020harmless,soudry2018implicit,zhang2005boosting,wang2022tight,zhou2022non}. However, properties of these interpolators under transfer learning are less well-studied. 

A recent work \cite{mallinar2024minimum} identified conditions on benign overfitting, i.e., data distribution choices such that min-norm interpolators are consistent under covariate shift. Another related work \cite{patil2024optimal} characterized the prediction risk of ridge regression and min-$\ell_2$-norm interpolators, while \cite{tripuraneni2021covariate} provided an asymptotic analysis of high-dimensional random feature regression under covariate shift. 
However, these works consider an out-of-distribution (OOD) setting, where no observation from the test distribution is utilized during training. 

Transfer learning is often required in settings where along with enormous source data, limited test data is indeed available, cf.~\cite{zhao2022construction}. 
Data from the test distribution can enhance transfer performance.
However, min-norm interpolators in this context are not as well-understood.  Some recent work has introduced and analyzed approaches that transfer specific parameters from an estimator fitted on source-only data \cite{dar2022double} or that interpolate the target data while minimizing distance to a source-only interpolator \cite{kim2026transfer}. However, these works study mechanisms where a source estimator is learned separately and then incorporated with separate target data, and their model shift analyses are restricted to isotropic Gaussian covariates. A closely related line of work studies data-level pooling in high-dimensional transfer learning in the underparametrized regime \cite{yang2020analysis}. However, the analyses in \cite{yang2020analysis} rely on properties of OLS in the underparametrized regime, and the model shift analysis in \cite{yang2020analysis} is similarly restricted to Gaussian covariates. In contrast, the behavior of early fusion mechanisms---where source and target samples are combined at the data level---in the overparametrized regime remains largely unexplored: it is unclear when one should pool source and target samples into a single interpolating estimator and how model and covariate shift affect the risk of such an estimator, especially under more general covariate distributions.

This paper addresses this crucial gap in the literature and characterizes the precise prediction performance of a pooled min-$\ell_2$-norm interpolator calculated using both source and target data under overparametrization.  
Our contributions are summarized as follows:
\begin{enumerate}
    \item[(i)] We consider a min-$\ell_2$-norm interpolator constructed by combining the source and target data (henceforth referred to as the pooled min-$\ell_2$-norm interpolator).
    We characterize the finite-sample bias and variance of the pooled min-$\ell_2$-norm interpolator $\hat\bbeta$, defined in \eqref{eqn:interpolator}, in an overparametrized regime. We consider the cases of  model shift and covariate shift (Theorems \ref{thm:model_shift} and \ref{thm:design_shift} respectively) and quantify  the generalization error of $\hat\bbeta$.
    In the presence of model shift, Theorem~\ref{thm:model_shift} characterizes the risk of $\hat\bbeta$ under general anisotropic covariance matrices; special examples are presented in Corollaries \ref{cor:model_shift_isotropic}--\ref{cor:REresult}, covering the cases of isotropic designs, spiked matrix models and random effects models. Our main theorem for covariate shift (Theorem \ref{thm:design_shift}) operates under an assumption of  simultaneous diagonalizability. We discuss possibilities of relaxing this assumption in Section~\ref{sec:future}(v) and Appendix~\appendixref{sec:sketch_nonsimul}{I}. 

    \item[(ii)] Our general results have a few important takeaways. For instance, our model shift results (Theorem~\ref{thm:model_shift}) characterize the precise threshold at which the pooled interpolator outperforms the target-only interpolator (Proposition~\ref{prop:model_shift_summary}). In the special case of isotropic covariance, we further provide a data-driven procedure for estimating this phase transition (Appendix~\appendixref{sec:choice_of_interpolator}{A}).
    In a similar spirit, our findings on covariate shift (Theorem~\ref{thm:design_shift}) uncover surprising dichotomies. In particular, we observe that increased degrees of covariate shift can in some cases enhance the performance of the pooled min-$\ell_2$-norm interpolator (while the opposite happens in other regimes). This property is entirely dependent on the dimensionalities of the observed data (see Proposition \ref{prop:cov_shift_example}). 

   \item[(iii)]  We prove a universality result (Theorem~\ref{thm:universality}), establishing that our asymptotic risk characterization of the pooled min-$\ell_2$-norm interpolator holds for a broad class of covariate distributions. This demonstrates that the transfer phenomena identified in our Gaussian-based analysis are robust to the choice of covariate distributions. Our result complements related work \cite{hu2023universality, gerace2024universality, dudeja2024universality, han2023universality, montanari2022universality, liang2022precise, lahiry2024universality} on the universality of regression estimators in high-dimensional settings.

   \item[(iv)] On the technical front, a significant mathematical contribution of our work is the derivation of an anisotropic law applicable to a broad class of random matrices (see Theorem \appendixref{thm:anisotropic_law}{F.1}). This complements previous anisotropic local laws \cite{knowles2017anisotropic,yang2020analysis} and may be of independent interest both in random matrix theory as well as in the study of other high-dimensional transfer learning problems, particularly those involving multiple covariance matrices. 
   
   Additionally, we establish how Lindeberg's replacement method can be combined with Gaussian interpolation to establish universality for resolvent-based functionals of mixtures of sample covariance matrices (Appendix \appendixref{sec:proof:universality}{G}). Our formulation allows for more precise control of expected derivative terms that arise along the interpolation path, thus permitting us to pass to the ridgeless limit. This complements previous approaches to proving universality \cite{chatterjee2006lindeberg,korado2011lindeberg,montanari2022universality} while also handling functionals of mixtures of sample covariance matrices, which prior techniques do not readily accommodate.
   
 \item[(v)] We interpret the pooled min-$\ell_2$-norm interpolator as a ridgeless limit of a pooled ridge estimator (see \eqref{eqn:ridgeless}). As part of our results, we characterize the exact bias and variance of the pooled ridge estimator under the aforementioned settings. This contributes independently to the expanding body of work on the generalization error of high-dimensional ridge regression~\cite{cheng2022high,hastie2022surprises}. Our work complements the out-of-distribution error characterized by recent works~\cite{mallinar2024minimum,patil2024optimal} where target data is not used in constructing the estimator. 

   \item[(vi)] Furthermore, we demonstrate that our framework accommodates both (1) settings beyond the well-specified linear model (Theorem~\ref{thm:misspecification}) and (2) estimators beyond the pooled min-$\ell_2$-norm interpolator (Theorem~\ref{thm:bias_corrected_risk}). In a misspecified linear model studied in the existing literature \cite{hastie2022surprises, dar2021common}, we characterize how misspecification alters the risk of the pooled min-$\ell_2$-norm interpolator under model shift. We also propose a two-step bias-corrected estimator, based on a held-out subset of target samples, and derive its exact asymptotic risk under model shift.
   \item[(vii)] We complement our theoretical findings with extensive simulations that demonstrate the finite-sample efficacy of our theory. Our results exhibit remarkable finite-sample performance under both model shift (Figures \ref{fig:isotropic_model_shift} and \ref{fig:spiked_model_shift}) and covariate shift (Figure \ref{fig:design_shift}), as well as across varying shift levels, signal-to-noise ratios, and dimensionality ratios. Additional simulations in Appendix \appendixref{sec:additional_simulations}{J} further validate the universality of our risk characterization for a range of synthetic and semi-synthetic covariates and illustrate the relative performance of our bias-corrected estimator. Our results complement analogous results for the underparametrized setting in \cite{yang2020analysis} to show that the generalization error of this interpolator exhibits the well-known double-descent phenomenon observed in the machine learning literature~\cite{belkin2019reconciling,belkin2019does}. 
\end{enumerate}
\noindent The remainder of the paper is structured as follows: In Section \ref{sec:setup}, we lay out our framework and model assumptions. Sections \ref{sec:model_shift} and \ref{sec:covariate_shift} provide detailed analyses of the generalization errors for model shift and covariate shift, respectively. Section \ref{sec:extensions} includes extensions of our work to a wide range of covariate distributions, to mildly misspecified settings, and to a two-step bias-corrected estimator. A proof outline is provided in Section \ref{sec:proof_outline}, followed by discussions and future directions in Section \ref{sec:future}. Auxiliary results concerning pooled ridge regression are included in Appendix \appendixref{sec:ridge}{C}, with detailed proofs of our main results deferred to the remaining appendices.

\section{Setup}\label{sec:setup}
In this section, we present our formal mathematical setup.
\subsection{Data model}
We consider the canonical transfer learning setting where random samples are observed from two different populations. Formally, we observe two datasets $(\by^{(1)}, \bX^{(1)})$ and $(\by^{(2)}, \bX^{(2)})$---one serves as the source data and the other the target. Typically, the former has a larger sample size than the latter. Some of our results extend to the general case of multiple source datasets; see Appendix \appendixref{sec:multiple}{B}. We assume that the observed data comes from linear models given by
\begin{equation}\label{eqn:model_linear}
\by^{(k)} = \bX^{(k)} \bbeta^{(k)} + \bep^{(k)}, k =1,2,  
\end{equation}
where $k=1$ and $k=2$ correspond to the source and target, respectively. Here, $\mathbf{y}^{(k)} \in \mathbb{R}^{n_k}$ denotes the response vectors, with $n_k$ representing the number of samples observed in population $k$, $\bX^{(k)} \in \mathbb{R}^{n_k \times p}$ denotes the observed covariate matrix from population $k$ with $p$ the number of features, $\boldsymbol{\beta}^{(k)} \in \mathbb{R}^p$ are signal vectors independent of $\bX^{(k)}$,
and $\boldsymbol{\epsilon}^{(k)} \in \mathbb{R}^{n_k}$ are noise vectors independent of both $\bX^{(k)}$ and $\boldsymbol{\beta}^{(k)}$. We denote $n = n_1 + n_2$ as the total number of samples.

Throughout the paper, we assume that the design matrices take the form
\begin{equation}\label{eqn:assumption_X}
\bX^{(k)} = \bZ^{(k)} (\bSigma^{(k)})^{1/2},
\end{equation}
where $\bSigma^{(k)} \in \R^{p \times p}$ are deterministic positive definite covariate matrices and $\bZ^{(k)} \in \R^{n_k \times p}$ are random matrices with independent entries. We impose a common set of regularity assumptions throughout the paper related to regularity of the covariance matrices $\bSigma^{(k)}$, the noise terms $\bep^{(k)}$, and the aspect ratios $p / n_1$ and $p / n_2$. Distributional assumptions on the design matrices $\bZ^{(k)}$ are stated separately, since different results require different levels of moment control. 

The precise assumptions are as follows.
\begin{assumption}[Shared regularity conditions]\label{as:shared}
    The random objects $\bZ^{(1)}$, $\bZ^{(2)}$, $\bep^{(1)}$, $\bep^{(2)}$ are mutually independent. Additionally, there exists a fixed constant $\tau \in (0,1)$ such that the following conditions hold for $k = 1,2$:
    \begin{enumerate}
        \item[(a)] The eigenvalues $\lambda_1^{(k)} \ge \cdots \ge \lambda_p^{(k)}$ of $\bSigma^{(k)}$ satisfy
        \begin{equation}\label{eqn:eigenvalue_bounded}
        \tau \leq \lambda_p^{(k)} \leq ... \leq \lambda_1^{(k)} \leq \tau^{-1}.
        \end{equation}
        \item[(b)] The entries of $\bep^{(k)}$ are independent, have mean zero, and have variance $\sigma^2 < \infty$.
        \item[(c)] The aspect ratios $\gamma_1 := p/n_1$, $\gamma_2 := p/n_2$, and $\gamma := p/n$ satisfy
        \begin{equation}\label{eqn:dimensions}
        1 + \tau \leq \gamma_1,\gamma_2,\gamma \leq  \tau^{-1}.
        \end{equation}
    \end{enumerate}
\end{assumption}
\begin{assumption}[Gaussian design]\label{as:design_gaussian}
    For $k = 1,2$, the entries of $\bZ^{(k)}$ are independent standard Gaussian random variables with mean zero and unit variance. 
\end{assumption}

\begin{assumption}[All-moments design]\label{as:design_all_moments}
    For $k = 1,2$, the entries of $\bZ^{(k)}$ are independent random variables with mean zero and variance one. Additionally, for every fixed integer $m \ge 1$, there exists a fixed constant $\tau_m \in(0,1)$ such that
    \begin{equation*}
        \sup_{i,j,k} \E[|Z_{ij}^{(k)}|^m] \le \tau_m^{-1} < \infty.
    \end{equation*}
\end{assumption}

\begin{assumption}[$4 + \epsilon$-moments design]\label{as:design_four_epsilon}
    For $k = 1,2$, the entries of $\bZ^{(k)}$ are independent random variables with mean zero and variance one. Additionally, there exists a fixed constant $\varphi > 4$ and a fixed constant $\tau_\varphi \in (0,1)$ such that 
    \begin{equation*}
        \sup_{i,j,k} \E[|Z_{ij}^{(k)}|^\varphi] \leq \tau^{-1}_\varphi < \infty.
    \end{equation*}
\end{assumption}
The following assumption will be used only for the model shift results in Section \ref{sec:model_shift}.
\begin{assumption}[Shared covariance]\label{as:shared_covariance}
    The source and target covariance matrices $\bSigma^{(1)}$ and $\bSigma^{(2)}$ are identical.
    When this holds, we denote the shared covariance by $\bSigma :=\bSigma^{(1)} = \bSigma^{(2)}$.
\end{assumption}
\begin{assumption}[Simultaneous diagonalizability]\label{as:simul_diag}
    The source and target covariance matrices $\bSigma^{(1)}$ and $\bSigma^{(2)}$ are simultaneously diagonalizable. That is, there exist an orthogonal matrix $\bV$ and two diagonal matrices $\bLambda^{(1)},\bLambda^{(2)}$ such that $\bSigma^{(1)} = \bV\bLambda^{(1)}\bV^\top$ and $\bSigma^{(2)} = \bV\bLambda^{(2)}\bV^\top$.
\end{assumption}
As displayed in \eqref{eqn:dimensions}, we are particularly interested in the overparametrized regime where $p>n$, partly because it is more prevalent in modern machine learning applications. Besides, the underparametrized regime, i.e. $p < n$, is better understood in the literature, cf. \cite{yang2020analysis} (they assumed $n_2 > p$, but their results remain valid for the general case $n>p$).

We remark that we do not require any additional assumptions on the signals $\bbeta^{(k)}$, except for their independence from $\bX^{(k)}$. In fact, our results later show that the convergence of the risks of interest is conditional on $\bbeta^{(k)}$, and depends on $\bbeta^{(k)}$ through their norms  and inner products (i.e., $\bbeta^{(k)}$ can also be taken to be arbitrary and deterministic). 

\subsection{Estimator}
Data fusion, the integration of information from diverse datasets, has been broadly categorized into three approaches in the literature: early, intermediate and late fusion \cite{baltruvsaitis2018multimodal,sidheekh2024credibility}. Early fusion combines information from multiple sources at the input level. Late fusion learns independently from each dataset and combines information at the end. 
Intermediate fusion achieves a middle ground and attempts to jointly learn information from both datasets in a more sophisticated manner than simply concatenating them. 
In this work, we consider a form of min-norm interpolation that captures both early and a class of intermediate fusion strategies. 
To achieve early fusion using min-norm interpolation, one would naturally consider the pooled min-$\ell_2$-norm interpolator, defined as 

\begin{equation}\label{eqn:interpolator}
\hat\bbeta = \argmin \left \{ \|\bb\|_2 \ \ \text{s.t.}\ \by^{(k)} = \bX^{(k)}\bb,\ k=1,2 \right \}.
\end{equation}
This interpolator admits the following analytical expression:

\begin{equation}\label{eqn:analytical_pinv}
\hat\bbeta = \left(\bX^{(1)\top}\bX^{(1)} + \bX^{(2)\top}\bX^{(2)}\right)^\dagger \left(\bX^{(1)\top}\by^{(1)} + \bX^{(2)\top}\by^{(2)}\right),
\end{equation}
where for any matrix $\bA$, we denote its pseudoinverse by $\bA^\dagger$. This interpolator is alternatively designated as the "ridgeless" estimator because it is the limit of the ridge estimator when the penalty term vanishes (noticed earlier by \cite{hastie2022surprises} in the context of a single training dataset):

\begin{equation}\label{eqn:ridgeless}
\begin{aligned}
\hat\bbeta
&= \lim_{\lambda\rightarrow 0}\left(\bX^{(1)\top}\bX^{(1)} + \bX^{(2)\top}\bX^{(2)}+n\lambda\bI\right)^{-1} \left(\bX^{(1)\top}\by^{(1)} + \bX^{(2)\top}\by^{(2)}\right)\\
&= \lim_{\lambda \rightarrow 0}\argmin_{\bb} \left\{\frac{1}{n}\|\by^{(1)} - \bX^{(1)}\bb\|_2^2 + \frac{1}{n}\|\by^{(2)} - \bX^{(2)}\bb\|_2^2 + \lambda \|\bb\|_2^2 \right\}.
\end{aligned}
\end{equation}

Coincidentally, although the pooled min-$\ell_2$-norm interpolator is motivated via early fusion, it can also capture a form of intermediate fusion as described below. If we consider a weighted ridge loss with weights $w_1,w_2$ on the source and target data respectively, then the ridgeless limit of this loss recovers the pooled min-$\ell_2$-norm interpolator. This can be seen as follows:  for \textit{any} two weighting coefficients $w_1,w_2 > 0$,
\begin{equation}\label{eqn:intermediate_fusion}
\begin{aligned}
\hat\bbeta &= \argmin \left \{ \|\bb\|_2 \ \ \text{s.t.}\ w_k\by^{(k)} = w_k\bX^{(k)}\bb,\ k=1,2 \right \}\\
&= \left(w_1\bX^{(1)\top}\bX^{(1)} + w_2\bX^{(2)\top}\bX^{(2)}\right)^\dagger \left(w_1\bX^{(1)\top}\by^{(1)} + w_2\bX^{(2)\top}\by^{(2)}\right)\\
&= \lim_{\lambda\rightarrow 0}\left(w_1\bX^{(1)\top}\bX^{(1)} + w_2\bX^{(2)\top}\bX^{(2)}+n\lambda\bI\right)^{-1} \left(w_1\bX^{(1)\top}\by^{(1)} + w_2\bX^{(2)\top}\by^{(2)}\right)\\
&= \lim_{\lambda \rightarrow 0}\argmin_{\bb} \left\{\frac{w_1}{n}\|\by^{(1)} - \bX^{(1)}\bb\|_2^2 + \frac{w_2}{n}\|\by^{(2)} - \bX^{(2)}\bb\|_2^2 + \lambda \|\bb\|_2^2 \right\},
\end{aligned}
\end{equation}
indicating that the interpolator above uniquely stands as the common solution of \textit{both} the min-$\ell_2$-norm interpolator and the ridgeless estimator when two datasets are weighted by $w_1,w_2$.

Furthermore, the pooled min-$\ell_2$-norm interpolator \eqref{eqn:interpolator} appears naturally as the limit of the gradient descent based linear estimator in the following way: Initialize $\bbeta_0=0$ and run gradient descent with iterates given by 
$$\bbeta_{t+1}= \bbeta_t + \eta \sum_{k=1}^2\bX^{(k)\top}(\by^{(k)} - \bX^{(k)}\bbeta_{t}), \qquad t=0,1,\ldots,$$
with $0< \eta \le \min\{\lambda_{\max}(\bX^{(1)\top}\bX^{(1)}), \lambda_{\max}(\bX^{(2)\top}\bX^{(2)})\}$,  
then $\lim_{t \rightarrow \infty} \bbeta_t= \hat \bbeta$.

In recent years, min-norm interpolation-type solutions and their statistical properties have been extensively studied in numerous contexts~\cite{bartlett2020benign,belkin2018understand,belkin2019does,belkin2020two,bunea2020interpolation,liang2020just,liang2022mehler,liang2022precise}. In fact, it is conjectured that this min-norm regularization is responsible for the superior statistical behavior of estimators in the overparametrized regime. In light of these results and the aforementioned abundant formulations, we view our work as a first step of understanding the behavior of this important family of estimators in transfer learning, under the lens of early and intermediate fusion.

We conclude this section by highlighting that computing such an interpolator requires only the summary statistics $\mathbf{X}^{(k)\top}\mathbf{X}^{(k)}$ and $\mathbf{X}^{(k)\top}\mathbf{y}^{(k)}$ for $k=1,2$, rather than entire datasets. This characteristic is particularly beneficial in fields such as medical or genomics research, where individual-level data may be highly sensitive and not permissible for sharing across institutions.

\subsection{Risk}
To assess the performance of $\hat \bbeta$, we focus on its prediction risk at a new (unseen) test point $(\by_0, \bx_0)$ drawn from the target distribution, i.e., $k=2$. 
This out-of-sample prediction risk (hereafter simply referred to as risk) is defined as
\begin{equation} \label{eq:define_risk}
R(\hat\bbeta;\bbeta^{(2)}) := \E\left[(\bx^{\top}_0\hat\bbeta - \bx^{\top}_0\bbeta^{(2)})^2\hspace{0.5ex}|\hspace{0.5ex}\bX\right] =  \E\left[\|\hat\bbeta - \bbeta^{(2)}\|_{\bSigma^{(2)}}^2\hspace{0.5ex}|\hspace{0.5ex}\bX\right],
\end{equation}
where $\|\bx\|_{\bSigma}^2:=\bx^\top\bSigma\bx$ and $\bX := (\bX^{(1)\top},\bX^{(2)\top})^\top$. Note that this risk has a $\sigma^2$ difference from the mean-squared prediction error for the new data point, which does not affect the relative performance and is therefore omitted. Also, this risk definition takes expectation over both the randomness in the new test point $(\by_0,\bx_0)$ and the randomness in the training noises $(\bep^{(1)},\bep^{(2)})$. The risk is conditional on the covariates $\bX$. Despite this dependence, we will use the notation $R(\hat\bbeta,\bbeta^{(2)})$ for conciseness since the context is clear. The risk admits a bias-variance decomposition:

\begin{equation}\label{eq:risk_analytical}
R(\hat\bbeta;\bbeta^{(2)}) = \underbrace{\|\E(\hat\bbeta|\bX)-\bbeta^{(2)}\|_{\bSigma^{(2)}}^2}_{B(\hat\bbeta;\bbeta^{(2)})} + \underbrace{\Tr[\Cov(\hat\bbeta|\bX)\bSigma^{(2)}]}_{V(\hat\bbeta;\bbeta^{(2)})}.
\end{equation}
Plugging in \eqref{eqn:ridgeless}, we immediately obtain the following result which we will crucially use for the remainder of the sequel:

\begin{lemma}\label{lemma:risk_analytical}
Under Assumption \ref{as:shared}, the min-norm interpolator \eqref{eqn:interpolator} has variance
\begin{equation}\label{eqn:variance_analytical}
V(\hat\bbeta;\bbeta^{(2)}) = \frac{\sigma^2}{n}\Tr(\hat\bSigma^\dagger\bSigma^{(2)})
\end{equation}
and bias
\begin{equation}\label{eqn:bias_analytical}
\begin{aligned}
B(\hat\bbeta;\bbeta^{(2)}) =& \underbrace{\bbeta^{(2)\top} (\hat\bSigma^\dagger\hat\bSigma - \bI)\bSigma^{(2)} (\hat\bSigma^\dagger\hat\bSigma - \bI) \bbeta^{(2)}}_{B_1(\hat\bbeta;\bbeta^{(2)})} \\
&+ 
\underbrace{\tilde\bbeta^\top (\frac{\bX^{(1)\top} \bX^{(1)}}{n})\hat\bSigma^\dagger \bSigma^{(2)} \hat\bSigma^\dagger (\frac{\bX^{(1)\top} \bX^{(1)}}{n})\tilde\bbeta}_{B_2(\hat\bbeta;\bbeta^{(2)})}\\
&+\underbrace{2\bbeta^{(2)\top} (\hat\bSigma^\dagger\hat\bSigma - \bI)\bSigma^{(2)} \hat\bSigma^\dagger (\frac{\bX^{(1)\top} \bX^{(1)}}{n})\tilde\bbeta}_{B_3(\hat\bbeta;\bbeta^{(2)})},
\end{aligned}
\end{equation}
where $\hat\bSigma = \bX^\top\bX/n$ is the (uncentered) sample  covariance matrix obtained on appending the source and target samples, and $\tilde\bbeta:=\bbeta^{(1)} - \bbeta^{(2)}$ is the signal shift vector.
\end{lemma}
\begin{proof}
The proof is straightforward by plugging in the respective definitions, and observing the fact that $(\bX^\top\bX)^\dagger\bX^\top\bX(\bX^\top\bX)^\dagger = (\bX^\top\bX)^\dagger$.
\end{proof}

In the next two sections, we will lay out our main results which involve precise risk characterizations of the pooled min-$\ell_2$-norm interpolator. Bias and variance terms, as decomposed in Lemma \ref{lemma:risk_analytical}, will be presented separately.

Throughout the rest of the paper, we denote $f(p) = O(g(p))$ if there exists a constant $C$ such that $|f(p)| \leq Cg(p)$ for large enough $p$. Moreover, we say an event $\mathcal{E}$ happens with high probability (w.h.p.) if $\PP(\mathcal{E})\rightarrow 1$ as $p \rightarrow \infty$.

\section{Model Shift}\label{sec:model_shift}
We first investigate the effect of model shift. More precisely, we assume the covariates $\bX^{(1)}$ and $\bX^{(2)}$ come from the same distribution, but the underlying signal changes from $\bbeta^{(1)}$ to $\bbeta^{(2)}$. We denote the shift in the signal parameter by $\tilde\bbeta := \bbeta^{(1)}-\bbeta^{(2)}$. We will first characterize the precise risk of the pooled min-$\ell_2$-norm interpolator and then compare it with that of the min-$\ell_2$-norm interpolator using solely the target data. 

\subsection{Risk under model shift}
Recall the formula for the risk $R(\hat \bbeta, \bbeta^{(2)})$ in \eqref{eq:risk_analytical} and its decomposition in Lemma \ref{lemma:risk_analytical}. We will split the risk into bias and variance following Lemma \ref{lemma:risk_analytical}, and characterize their behavior respectively.
Therefore, our contribution focuses on  characterizing the asymptotic behavior of $B_1(\hat\bbeta;\bbeta^{(2)})$, $B_2(\hat\bbeta;\bbeta^{(2)})$,  $B_3(\hat \bbeta;\bbeta^{(2)})$, and $V(\hat\bbeta;\bbeta^{(2)})$ as defined by \eqref{eqn:variance_analytical} and \eqref{eqn:bias_analytical}. 

The risk of the pooled min-$\ell_2$-norm interpolator depends on the geometry of the shared covariance matrix $\bSigma$, the model shift $\tilde \bbeta$, and the target signal $\bbeta^{(2)}$. Letting $\bSigma = \sum_{i=1}^p s_i \bm v_i \bm v_i^\top$ denote the eigendecomposition of the shared covariance matrix, we encode this geometry via the following signed measures on $\R_{\geq 0}$ for nonzero $\bbeta^{(2)}$ and $\tilde \bbeta$:
\begin{gather*}
    \hat G_n^{\bbeta^{(2)}}(s) = \frac{1}{\|\bbeta^{(2)}\|_2^2} \sum_{i=1}^p \langle \bbeta^{(2)}, \bm v_i\rangle^2 \bm 1_{\{s \geq s_i\}}, \quad \hat G_n^{\tilde \bbeta}(s) = \frac{1}{\|\tilde \bbeta\|_2^2} \sum_{i=1}^p \langle \tilde \bbeta, \bm v_i\rangle^2 \bm 1_{\{s \geq s_i\}} \\
    \hat G_n^{(b)}(s) = \frac{1}{\|\bbeta^{(2)}\|_2 \|\tilde \bbeta\|_2}\sum_{i=1}^p \langle \bbeta^{(2)}, \bm v_i\rangle  \langle \tilde \bbeta, \bm v_i\rangle  \bm 1_{\{s \geq s_i\}}, \quad 
    \hat H_n(s) = \frac{1}{p} \sum_{i=1}^p \bm 1_{\{s \geq s_i\}}.
\end{gather*}
When $\bbeta^{(2)} = 0$ or $\tilde \bbeta = 0$, we define the corresponding signed measures to be the Dirac measure on $\{0\}$.

These measures encode how the spectrum of the covariance matrix $\bSigma$ interacts with the target signal $\bbeta^{(2)}$ and the model shift $\tilde \bbeta$. While the empirical spectral distribution $\hat H_n$ simply records the eigenvalue distribution of the shared covariance $\bSigma$, the probability measures $\hat G_n^{\bbeta^{(2)}}$ and $\hat G_n^{\tilde \bbeta}$ weight this spectrum by the alignment between each eigenvector and the corresponding signal $\bbeta^{(2)}$ or shift $\tilde \bbeta$. The signed measure $\hat G_n^{(b)}$ then captures the directional alignment between $\bbeta^{(2)}$ and $\tilde \bbeta$ with respect to the eigenbasis $\bm v_1, \dots, \bm v_p$. Thus, these four measures translate the spectral geometry of $\bSigma$, $\tilde \bbeta$, and $\bbeta^{(2)}$ into inputs for the following asymptotic risk quantities:
\begin{definition}\label{def:cov_geometry}
    For $\gamma \in \R_{> 0}$, define $c_0(\gamma, \hat H_n) \in \R_{\ge0}$ to be the unique non-negative solution of
    \begin{equation}
        1 - \frac{1}{\gamma} = \int \frac{1}{1 + c_0 \gamma s} d\hat H_n(s).
    \end{equation}
    and then define 
    \begin{equation}
        c_1 := c_0 \frac{\int \frac{s^2}{(1 + c_0 \gamma s)^2} d\hat H_n(s)}{\int \frac{s}{(1 + c_0\gamma s)^2} d\hat H_n(s)}.
    \end{equation}
    Finally, we define 
    \begin{align}
        \mathcal B_1(\hat H_n, \hat G_n^{\bbeta^{(2)}}, \gamma) &:= \|\bbeta^{(2)}\|_2^2 (1 + \gamma c_1) \int \frac{s}{(1 + c_0 \gamma s)^2} d\hat G_n^{\bbeta^{(2)}}(s)\\
        \mathcal B_2(\hat H_n, \hat G_n^{\tilde \bbeta}, \gamma) &:= \|\tilde \bbeta\|_2^2 \biggl\{\biggl(\frac{n_1}{n}\biggr)^2 \int \frac{\gamma s(c_1 + \gamma c_0^2 s^2)}{(1 + c_0 \gamma s)^2} d\hat G_n^{\tilde \bbeta}(s) \\
        &\hspace{5em}+ \frac{n_1}{n}\biggl(1 - \frac{n_1}{n}\biggr) \frac{(1 + \gamma c_1) \int \frac{s^2}{(1 + \gamma c_0 s)^2} d\hat H_n(s)}{\gamma (\int \frac{s}{1 + \gamma c_0 s} d\hat H_n(s))^2}\int s \ d\hat G_n^{\tilde \bbeta}(s)\biggr\} \\
        \mathcal B_3(\hat H_n, \hat G_n^{(b)}, \gamma) &:= -\frac{2n_1}{n} \|\bbeta^{(2)}\|_2\|\tilde \bbeta\|_2 \int \frac{\gamma s (c_0 s- c_1)}{(1 +  c_0 \gamma s)^2} d\hat G_n^{(b)}(s) \\
        \mathcal V(\hat H_n,\gamma) &:= \sigma^2 \gamma c_1.
    \end{align}
\end{definition}
Our main result for model shift can then be stated in terms of these quantities.
\begin{thm}[Risk under model shift]\label{thm:model_shift}
Under Assumptions \ref{as:shared}, \ref{as:design_gaussian}, and \ref{as:shared_covariance}, with high probability over the randomness of $(\bZ^{(1)},\bZ^{(2)})$, we have
\begin{align}
    V(\hat \bbeta; \bbeta^{(2)}) &= \mathcal V(\hat H_n, \gamma) + O(p^{-1/12})  \label{eqn:model_shift_V}\\
    B_1(\hat \bbeta; \bbeta^{(2)}) &= \mathcal B_1(\hat H_n, \hat G_n^{\bbeta^{(2)}}, \gamma) + O(p^{-1/12} \|\bbeta^{(2)}\|_2^2) \label{eqn:model_shift_B1}\\
    B_2(\hat \bbeta; \bbeta^{(2)}) &= \mathcal B_2(\hat H_n, \hat G_n^{\tilde \bbeta}, \gamma) + O(p^{-1/12} \|\tilde \bbeta\|_2^2) \label{eqn:model_shift_B2}\\
    B_3(\hat \bbeta; \bbeta^{(2)}) &= \mathcal B_3(\hat H_n, \hat G_n^{(b)}, \gamma) + O(p^{-1/12} \|\bbeta^{(2)}\|_2 \|\tilde \bbeta\|_2) \label{eqn:model_shift_B3}
\end{align}
\end{thm}

\begin{remark}
    In Section \ref{sec:extensions}, we present a universality result for the out-of-sample prediction risk for the pooled min-$\ell_2$-norm interpolator from which we can conclude that the characterization in Theorem \ref{thm:model_shift} holds for a very general class of covariates (i.e., replacing Assumption \ref{as:design_gaussian} by Assumption \ref{as:design_all_moments}). Separately, in Appendix \appendixref{sec:multiple}{B}, we demonstrate that Theorem \ref{thm:model_shift} can be generalized to the analogous setting with multiple source datasets. 
\end{remark}

We remark that Theorem \ref{thm:model_shift} holds true for finite values of $n$ and $p$. It 
 characterizes the bias and variance of the pooled min-$\ell_2$-norm interpolator for the overparametrized regime $p>n$. The variance $V(\hat\bbeta;\bbeta^{(2)})$ is decreasing in $p/n$, but the bias $B(\hat\bbeta;\bbeta^{(2)})$ is not necessarily monotone. We provide extensive numerical examples of $R(\hat\bbeta;\bbeta^{(2)})$ in Section \ref{sec:model_shift} showing that the risk function still exhibits the well-known double-descent phenomenon. 
The main challenge in establishing Theorem \ref{thm:model_shift} lies in the analysis of $B_2(\hat\bbeta;\bbeta^{(2)})$ and $B_3(\hat\bbeta;\bbeta^{(2)})$. 
More details and explanations are provided in Section \ref{sec:proof_outline}. The theorem allows us to answer the following questions:
\begin{enumerate}
    \item[(i)] When does the pooled min-$\ell_2$-norm interpolator outperform the target-only min-$\ell_2$-norm interpolator? 
    \item[(ii)] What is the optimal target sample size that minimizes the generalization error of the pooled min-$\ell_2$-norm interpolator? 
\end{enumerate}
We study these questions for three example problems: isotropic designs (Section \ref{subsec:isotropic_model_shift}), spiked covariance designs (Section \ref{subsec:spiked_model_shift}), and random effects models (Section \ref{subsec:random_model_shift}). In each case, we first establish results under the assumption that $(\bZ^{(1)}, \bZ^{(2)})$ are both Gaussian random matrices but note that our universality result (Theorem \ref{thm:universality}) implies that the same results hold for very general classes of covariates.

\subsection{Risk under isotropic designs}\label{subsec:isotropic_model_shift}
In this section, we consider the special case where $\bSigma^{(1)} = \bSigma^{(2)} = \bm I$. Here, the following corollary of Theorem \ref{thm:model_shift} holds:
\begin{cor}\label{cor:model_shift_isotropic}
    In the setting where $\bSigma^{(1)} = \bSigma^{(2)} = \bm I$, with high probability over the randomness of $(\bZ^{(1)}, \bZ^{(2)})$, the risk terms decompose as   
    \begin{align}
        V(\hat\bbeta;\bbeta^{(2)}) &= \sigma^2 \frac{n}{p-n} + O(\sigma^2 p^{-1/12}),
    \\
        B_1(\hat\bbeta;\bbeta^{(2)}) &= \frac{p-n}{p}\|\bbeta^{(2)}\|_2^2 + O( p^{-1/12}\|\bbeta^{(2)}\|_2^2), \label{eqn:model_shift_isotropic_B1}
   \\
        B_2(\hat\bbeta;\bbeta^{(2)}) &= \frac{n_1(p-n_1)}{p(p-n)}\|\tilde\bbeta\|_2^2  + O(p^{-1/12}\|\tilde\bbeta\|_2^2), \label{eqn:model_shift_isotropic_B2}\\
        B_3(\hat\bbeta;\bbeta^{(2)}) &= 0.
    \end{align}
\end{cor}
In fact, it shows that, in the isotropic setting, the bias can be decomposed into two simple terms \eqref{eqn:model_shift_isotropic_B1} and \eqref{eqn:model_shift_isotropic_B2}, which depend only on the $\ell_2$-norm of $\bbeta^{(2)}$ and $\tilde \bbeta$ respectively. Therefore, under isotropic designs, the generalization error is independent of the angle between $\bbeta^{(1)}$ and $\bbeta^{(2)}$. 

Now, we compare the risk of $\hat\bbeta$, the pooled min-$\ell_2$-norm interpolator, with that of the interpolator using only the target dataset, defined as follows:

\begin{equation}\label{eqn:interpolator_single}
\hat\bbeta^{(2)} = \argmin \left\{ \|\bb\|_2 \ \ \text{s.t.}\ \by^{(2)} = \bX^{(2)}\bb \right\}.
\end{equation}

The risk of $\hat\beta^{(2)}$ has been previously studied:

\begin{prop}[Corollary of Theorem 1 in \cite{hastie2022surprises}]\label{prop:single_interpolator_risk_isotropic}
Consider the estimator $\hat\bbeta^{(2)}$ in \eqref{eqn:interpolator_single}. Under the same setting as Corollary \ref{cor:model_shift_isotropic} and assuming that $\|\bbeta^{(2)}\|_2^2$ is bounded, we have almost surely,
\begin{equation}\label{eqn:single_interpolator_risk_isotropic}
R(\hat\bbeta^{(2)};\bbeta^{(2)}) = \frac{p-n_2}{p}\|\bbeta^{(2)}\|_2^2 + \frac{n_2}{p-n_2}\sigma^2 + o(1)
\end{equation}
\end{prop}

The above proposition, along with our results on model shift (Theorem \ref{thm:model_shift} and Corollary \ref{cor:model_shift_isotropic}), allows us to characterize when pooled min-$\ell_2$-norm interpolator yields lower risk than the target-only min-$\ell_2$-norm interpolator, i.e. when data pooling leads to positive/negative transfer:

\begin{prop}[Impact of model shift]\label{prop:model_shift_summary}
Under the same setting as Corollary \ref{cor:model_shift_isotropic} and assuming further that $\|\bbeta^{(2)}\|_2^2,\|\tilde\bbeta\|_2^2$ are bounded, define the signal-to-noise ratio $\rm{SNR} := \frac{\|\bbeta^{(2)}\|_2^2}{\sigma^2}$ and the shift-to-signal ratio $\rm{SSR}:= \frac{\|\tilde\bbeta\|_2^2}{\|\bbeta^{(2)}\|_2^2}$. Then, the following statements hold with high probability.
\begin{enumerate}
    \item[(i)] If $\mathrm{SNR}\leq \frac{p^2}{(p-n)(p-n_2)}$, then
    \begin{equation}\label{eqn:model_shift_negative}
    R(\hat\bbeta^{(2)};\bbeta^{(2)}) \leq R(\hat\bbeta;\bbeta^{(2)}) + o(1).
    \end{equation}
    \item[(ii)] If $\mathrm{SNR}> \frac{p^2}{(p-n)(p-n_2)}$, then there exists \begin{equation}\label{eq:rho_cutoff}
        \rho := \frac{p-n}{p-n_1} - \frac{p^2}{(p-n_1)(p-n_2)} \cdot \frac{1}{\rm{SNR}}
    \end{equation} such that when $\rm{SSR}\geq \rho$, then \eqref{eqn:model_shift_negative} holds; when $\rm{SSR} < \rho$, then 
    \begin{equation}\label{eqn:model_shift_positive}
    R(\hat\bbeta;\bbeta^{(2)}) \leq R(\hat\bbeta^{(2)};\bbeta^{(2)}) + o(1).
    \end{equation}
\end{enumerate}
\end{prop}
\begin{proof}
From Theorem \ref{thm:model_shift} and Proposition \ref{prop:single_interpolator_risk_isotropic}, we know with high probability, 
$$
R(\hat\bbeta;\bbeta^{(2)}) - R(\hat\bbeta^{(2)};\bbeta^{(2)}) = -\frac{n_1}{p}\|\bbeta^{(2)}\|_2^2 + \frac{n_1p}{(p-n)(p-n_2)}\sigma^2 + \frac{n_1(p-n_1)}{p(p-n)}\|\tilde\bbeta\|_2^2+o(1).$$
The rest follows with simple algebra.
\end{proof}

The implications of Proposition \ref{prop:model_shift_summary} are as follows:
\begin{enumerate}
    \item[(i)] When the SNR is \textit{small}, i.e. $\rm{SNR} \le \frac{1}{(1-1/\gamma)(1-1/\gamma_2)}$, adding more data always hurts the pooled min-$\ell_2$-norm interpolator. Note that, even if the source and target populations have the exact same model settings, pooling the data induces higher risk for the interpolator. 
    \item[(ii)] If SNR is large, whether pooling the data helps or not depends on the level of model shift. Namely, if SSR exceeds $\rho$ defined by \eqref{eq:rho_cutoff}, the models are far apart resulting in negative transfer. Otherwise, pooling the datasets helps. In particular, if $\textnormal{SSR} \ge \frac{1-1/\gamma}{1-1/\gamma_1}$, then for \textit{ any values of $\textnormal{SNR}$}, data pooling worsens performance.
    \item[(iii)] Note that the $\textnormal{SNR}$ and $\textnormal{SSR}$ can be consistently estimated using the available samples. Based on this observation, in Appendix \appendixref{sec:choice_of_interpolator}{A}, we provide a data-driven method for deciding whether to use the pooled interpolator or the target-based interpolator.
\end{enumerate}

In sum, in the overparametrized regime, positive transfer happens when the SSR is low and the SNR is sufficiently high. 
When SSR is low, the source data is close to the target data, so pooling introduces little model shift bias. When SNR is high, the signal generally dominates the noise in the sample, so additional source samples can help stabilize the interpolation. In contrast, when SSR is high and SNR is low, adding source samples forces the interpolator to fit observations that are both noisy and generated from a signal far from the target of interest, thus leading to negative transfer.

From a practical perspective, these results suggest that the data-driven procedure in Appendix \appendixref{sec:choice_of_interpolator}{A} can be used to select the number of  source samples to pool with the target samples so as to approximately minimize the risk of the resulting linear interpolator.  

Figure \ref{fig:isotropic_model_shift} illustrates the interplay between the $\textnormal{SNR}$, $\textnormal{SSR}$, and the transfer performance as described in Proposition \ref{prop:model_shift_summary}. Figures \ref{fig:isotropic_model_shift}(a) and \ref{fig:isotropic_model_shift}(b) show that for a fixed source sample size $n_1$, the risk of the pooled interpolator $\hat \bbeta$ decreases relative to that of the target interpolator $\hat \bbeta^{(2)}$ as the tasks become more similar (as $\textnormal{SSR}$ decreases) or as the target signal becomes easier to estimate (as $\textnormal{SNR}$ increases).

For completeness, we also depict the  empirical performance of the pooled interpolator in the underparametrized regime (where $n_1 + n_2 > p$) along with the corresponding theoretical predictions from \cite{yang2020analysis}. Figures \ref{fig:isotropic_model_shift}(a) and \ref{fig:isotropic_model_shift}(b) show that, when combined with our results in the overparametrized setting, a double descent phenomenon appears, though the first descent in the overparametrized setting can fail to occur for sufficiently small $\textnormal{SNR}$ or sufficiently large $\textnormal{SSR}$, where pooling remains detrimental throughout the overparametrized regime.

Finally, Figures \ref{fig:isotropic_model_shift}(c) and \ref{fig:isotropic_model_shift}(d) demonstrate that, when the \textnormal{SNR} is sufficiently small, the pooled interpolator is outperformed by the target-only interpolator, regardless of the \textnormal{SSR}. However, we also note that, for especially small SNR (e.g., $\textnormal{SNR} = 1$ in Figure \ref{fig:isotropic_model_shift}(a)), the null estimator $\hat \bbeta_{\text{null}}$ can outperform both interpolators throughout the range of source sample sizes $n_1$. In these regimes with especially low SNR, this comparison should not be interpreted as identifying a practically useful transfer procedure: see Appendix \appendixref{sec:null_comparison}{J.4} for more. On the other hand, for larger \textnormal{SNR}, pooling improves performance when the \textnormal{SSR} is sufficiently small. The phase transitions observed in Figures \ref{fig:isotropic_model_shift}(a) and \ref{fig:isotropic_model_shift}(b) closely match the phase transitions predicted by our theory in Proposition \ref{prop:model_shift_summary}, thus providing strong finite-sample support for our results.

 \begin{figure}
    \centering
    \begin{subfigure}[b]{0.48\linewidth}
    \includegraphics[width=\linewidth]{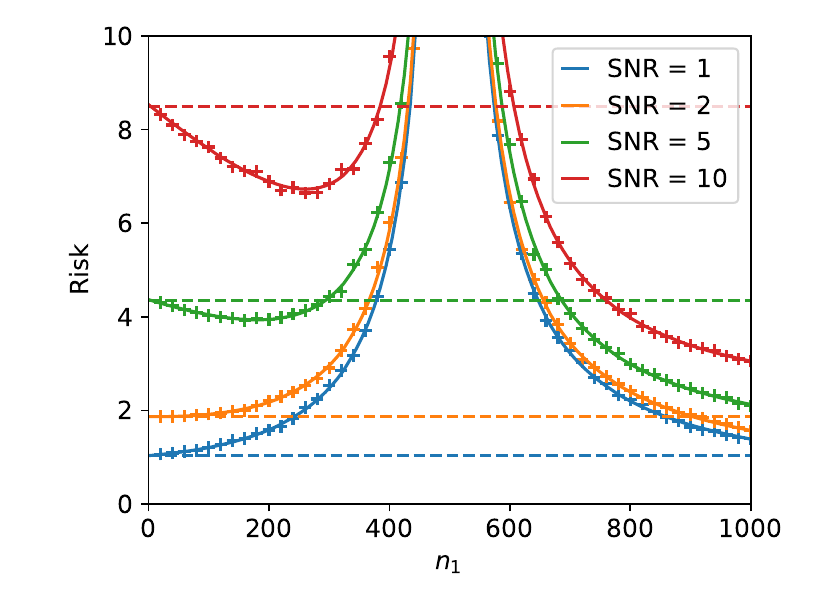}
    \caption{Fixing $\textnormal{SSR} = 0.2$}
    \end{subfigure}
    \begin{subfigure}[b]{0.48\linewidth}
    \includegraphics[width=\linewidth]{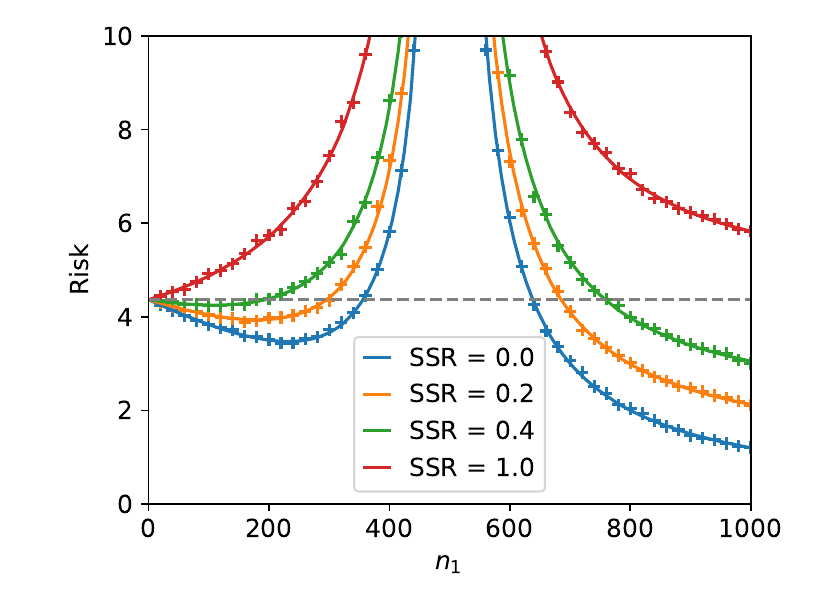}
    \caption{Fixing $\textnormal{SNR}=5$}
    \end{subfigure}
    \begin{subfigure}[b]{0.48\linewidth}
    \includegraphics[width=\linewidth]{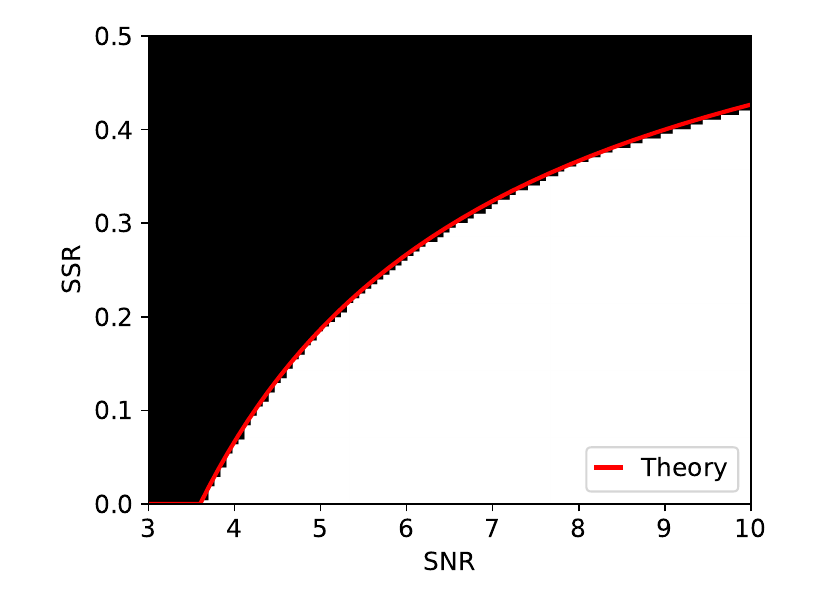}
    \caption{Pooled interpolator outperforms target-only interpolator}
    \end{subfigure}
    \begin{subfigure}[b]{0.48\linewidth}
    \includegraphics[width=\linewidth]{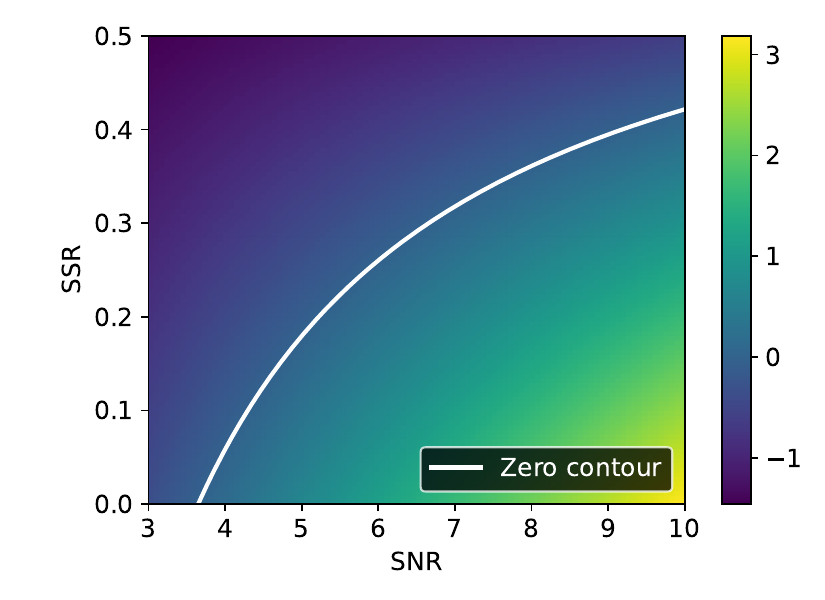}
    \caption{Risk difference $R(\hat \bbeta^{(2)};\bbeta^{(2)}) - R(\hat \bbeta; \bbeta^{(2)})$}
    \end{subfigure}
    \caption{Generalization error of the pooled min-$\ell_2$-norm interpolator under isotropic model shift with $n_2 = 100$ and $p = 600$. A fixed realization of $\bbeta^{(2)}$ and $\tilde \bbeta$ are used throughout, where $\bbeta^{(2)}$ is drawn uniformly from the sphere of radius $\sqrt{\textnormal{SNR}}$ and $\tilde \bbeta$ is drawn uniformly from the sphere of radius $\sqrt{\textnormal{SSR} \cdot \textnormal{SNR}}$. Isotropic Gaussian covariates and i.i.d. $\mcn(0,1)$ noise are redrawn across trials. Panels (a) and (b) show the risk of the pooled interpolator over varying source sample sizes $n_1$ for varying SNR and SSR, respectively. Solid curves denote theoretical predictions, obtained from Corollary \ref{cor:model_shift_isotropic} for $n_1 + n_2 < p$ and from \cite{yang2020analysis} for $n_1 + n_2 > p$. $+$ markers denote simulation averages across $50$ trials, and dashed horizontal lines denote the simulation average risk of the target-only interpolator. In Panel (c), the white region indicates where the pooled min-$\ell_2$-norm interpolator outperforms the target-only interpolator in a majority of $100$ trials with $n_1 = 300$; the red curve is the theoretical boundary from Proposition \ref{prop:model_shift_summary}. Panel (d) shows the risk difference $R(\hat \bbeta^{(2)};\bbeta^{(2)}) - R(\hat \bbeta; \bbeta^{(2)})$ estimated from $1000$ trials with $n_1 = 300$ for a range of SNR and SSR; the white curve is the corresponding estimated zero contour.}
    \label{fig:isotropic_model_shift}
\end{figure}

\subsection{Risk under spiked covariance designs}\label{subsec:spiked_model_shift}
In this section, we now consider the setting where $\bSigma^{(1)} = \bSigma^{(2)} = \bm I + \alpha \bm v\bm v^\top$ for $\alpha > 0$ and a unit vector $\bm v \in \R^p$. In this case, we obtain the following corollary of Theorem \ref{thm:model_shift}.

\begin{cor}\label{cor:model_shift_spiked_covariance}
    Let $\bm v \in \R^p$ be a unit vector, and let $\alpha \geq 0$ be a fixed non-negative real number.
    Further assuming that $\bSigma^{(1)} = \bSigma^{(2)} = \bm I + \alpha \bm v \bm v^\top$, with high probability over the randomness of $(\bm Z^{(1)}, \bm Z^{(2)})$, we have
    \begin{align*}
        V(\hat \bbeta; \bbeta^{(2)}) &= \sigma^2 \frac{n}{p-n} + O(p^{-1/12}) \\
        B_1(\hat \bbeta;\bbeta^{(2)}) &= \langle \bm v, \bbeta^{(2)}\rangle^2 \cdot \biggl(\frac{\gamma(\gamma-1)(1+\alpha)}{(\gamma + \alpha)^2} - \frac{p-n}{p}\biggr) + \frac{p-n}{p} \|\bbeta^{(2)}\|_2^2 \\
        &\quad + O(p^{-1/12} \|\bbeta^{(2)}\|_2^2) \\
        B_2(\hat \bbeta;\bbeta^{(2)}) &= \langle \bm v, \tilde \bbeta \rangle^2 \cdot \biggl(\biggl(\frac{n_1}{n}\biggr)^2 \biggl(\frac{(1+\alpha)(1 + \frac{(1+\alpha)^2}{\gamma-1})}{(\gamma-1)(1 + \frac{1+\alpha}{\gamma-1})^2} - \frac{1}{\gamma}\biggr) + \frac{n_1}{n}\biggl(1 - \frac{n_1}{n}\biggr) \frac{\alpha}{\gamma-1} \biggr)  \\
        &\quad + \frac{n_1(p-n_1)}{p(p-n)} \|\tilde \bbeta\|_2^2 +O( p^{-1/12} \|\tilde \bbeta\|_2^2) \\
        B_3(\hat \bbeta; \bbeta^{(2)}) &= -\frac{2n_1}{n} \langle \bm v, \bbeta^{(2)}\rangle \langle \bm v, \tilde \bbeta\rangle \frac{(\gamma-1)(1+\alpha)\alpha}{(\gamma+\alpha)^2} + O(p^{-1/12} \|\bbeta^{(2)}\|_2 \|\tilde \bbeta\|_2)
    \end{align*}
\end{cor}
 Corollary \ref{cor:model_shift_spiked_covariance} highlights how anisotropy enters the transfer learning problem. In Corollary \ref{cor:model_shift_isotropic}, the asymptotic out-of-sample prediction risk depended on $\bbeta^{(2)}$ and $\tilde \bbeta$ only through their norms $\|\bbeta^{(2)}\|_2$ and $\|\tilde \bbeta\|_2$ and was thus invariant to the relative positions of the signal vectors $\bbeta^{(1)}$ and $\bbeta^{(2)}$. In contrast, Corollary \ref{cor:model_shift_spiked_covariance} uncovers that even a rank-one perturbation from isotropy is sufficient to make the prediction risk depend on the target alignment $\frac{\langle \bm v, \bbeta^{(2)} \rangle}{\|\bbeta^{(2)}\|_2}$ and the shift alignment $\frac{\langle \bm v, \tilde \bbeta \rangle}{\|\tilde \bbeta\|_2}$, which suggests that the invariance of the prediction risk in the isotropic setting is particular to that setting.

Note that the  variance term in Corollary \ref{cor:model_shift_spiked_covariance} matches that of Corollary \ref{cor:model_shift_isotropic}; that is, the difference in the prediction risk between the two settings is driven by differences in the  bias expressions. In particular, each of the three terms in the bias decomposition can be expressed as one component that matches the corresponding bias term in Corollary \ref{cor:model_shift_isotropic} and one new component that depends on $\langle \bm v, \bbeta^{(2)}\rangle$, $\langle \bm v, \tilde \bbeta \rangle$, or both, in the cases of $B_1$, $B_2$, and $B_3$, respectively. Thus, when the spike $\bm v$ is orthogonal (or asymptotically orthogonal) to the target signal $\bbeta^{(2)}$ and the shift $\tilde \bbeta$, the risk in the spiked covariance setting reverts to the risk in the isotropic setting. On the other hand, when either the target signal $\bbeta^{(2)}$ or the shift $\tilde \bbeta$ is highly aligned with the spike $\bm v$, the resulting asymptotic risk can differ significantly from that of the isotropic setting.

Figure \ref{fig:spiked_model_shift} illustrates these findings in a simulation setting. In Figure \ref{fig:spiked_model_shift}(a), when the shift alignment $\frac{\langle \bm v,  \tilde \bbeta \rangle}{\|\tilde \bbeta\|_2}$ is fixed to be zero, increasing the magnitude of the alignment between the target $\bbeta^{(2)}$ and the spike $\bm v$ reduces the generalization error. Conversely, in Figure \ref{fig:spiked_model_shift}(b), when the target alignment $\frac{\langle \bm v, \bbeta^{(2)}\rangle}{\|\bbeta^{(2)}\|_2}$ is fixed to be zero, increasing the magnitude of the alignment between the shift $\tilde \bbeta$ and the spike $\bm v$ increases the generalization error. In both cases, the effects of alignment become more pronounced as the spike strength $\alpha$ grows. However, Figures \ref{fig:spiked_model_shift}(c) and \ref{fig:spiked_model_shift}(d) demonstrate that, when both the target $\bbeta^{(2)}$ and the shift $\tilde \bbeta$ have nonzero alignment with the spike $\bm v$, their effects are no longer separable. This interaction is partly driven by the growing $B_3$ term, which depends on the alignment between $\tilde \bbeta$ and $\bbeta^{(2)}$ through their one-dimensional projections onto the spike $\bm v$.

 \begin{figure}
    \centering
    \begin{subfigure}[b]{0.48\linewidth}
    \includegraphics[width=\linewidth]{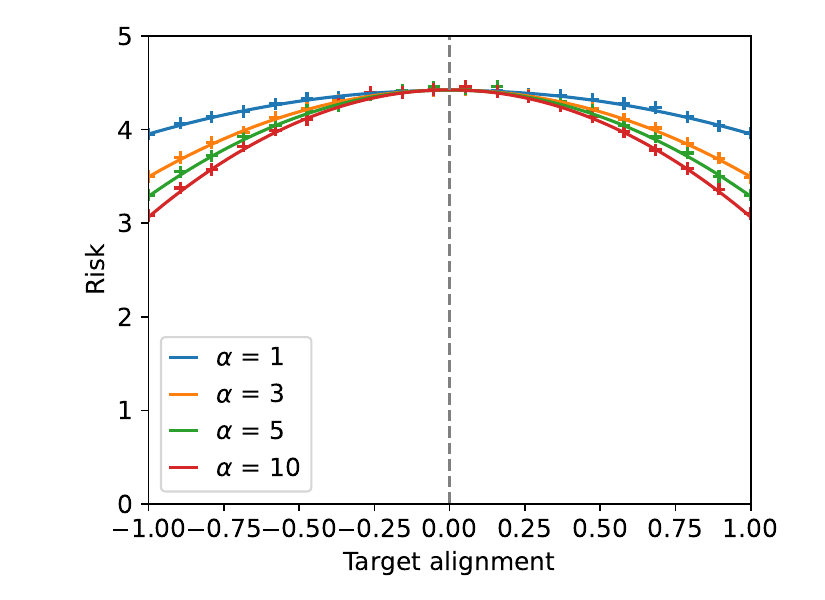}
    \caption{Fixing $\textnormal{shift alignment} = 0$}
    \end{subfigure}
    \begin{subfigure}[b]{0.48\linewidth}
    \includegraphics[width=\linewidth]{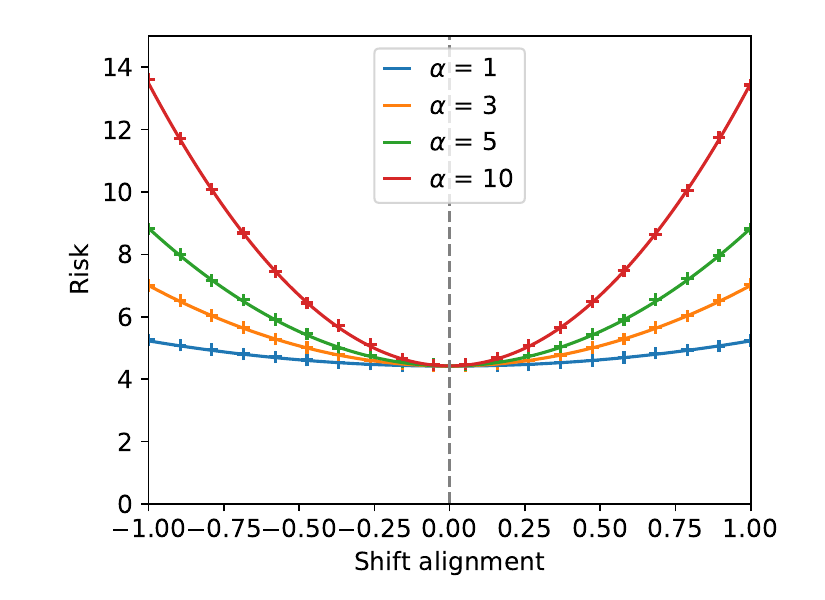}
    \caption{Fixing $\textnormal{target alignment}=0$}
    \end{subfigure}
    \begin{subfigure}[b]{0.48\linewidth}
    \includegraphics[width=\linewidth]{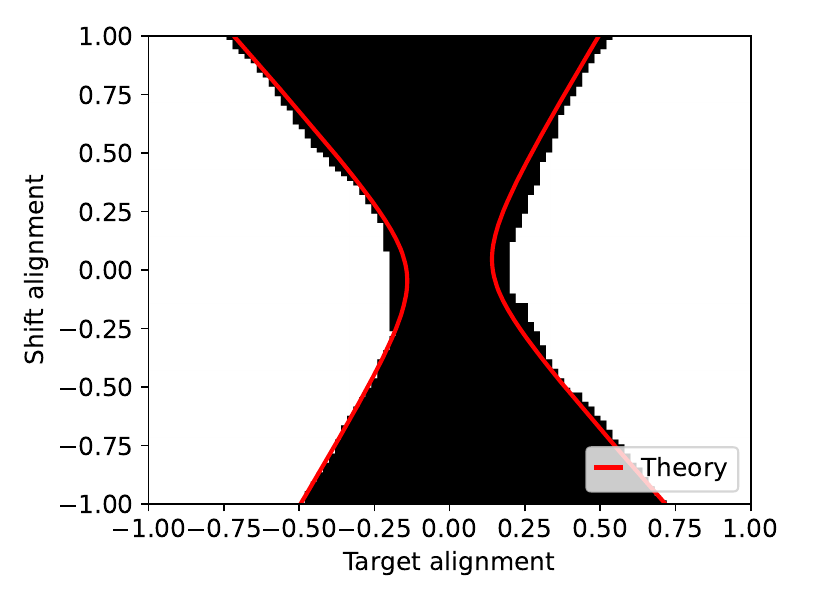}
    \caption{Pooled interpolator outperforms target-only interpolator}
    \end{subfigure}
    \begin{subfigure}[b]{0.48\linewidth}
    \includegraphics[width=\linewidth]{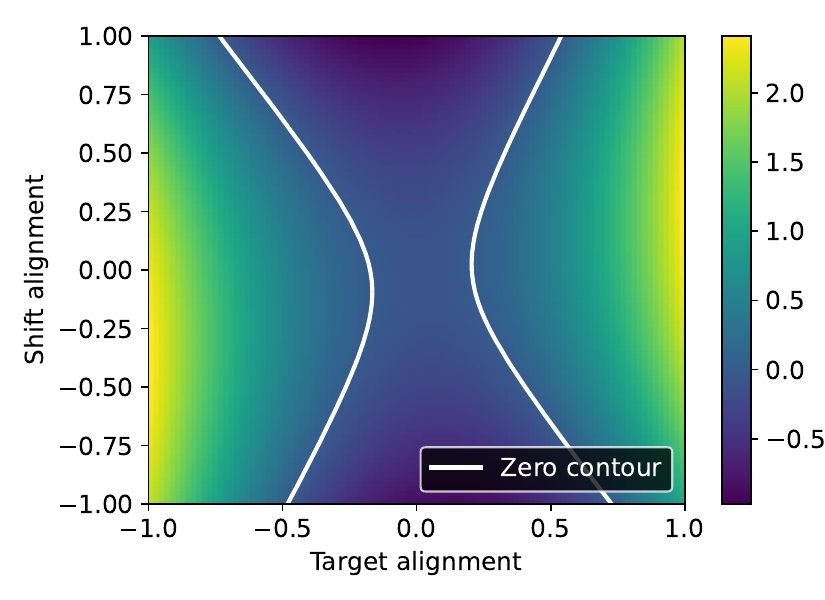}
    \caption{Risk difference $R(\hat \bbeta^{(2)};\bbeta^{(2)}) - R(\hat \bbeta; \bbeta^{(2)})$}
    \end{subfigure}
    \caption{Generalization error of the pooled min-$\ell_2$-norm interpolator under model shift with a rank-one spiked covariance structure, $n_1 = 300$, $n_2 = 100$, $p = 600$, $\textnormal{SNR} = 5$, and $\textnormal{SSR} = 0.2$. We let $\bSigma = \bm I + \alpha \bm v \bm v^\top$ for a fixed realization of $\bm v$ drawn uniformly from the unit sphere. 
    The target parameter $\bbeta^{(2)}$ has norm $\sqrt{\textnormal{SNR}}$ and the shift parameter $\tilde \bbeta$ has norm $\sqrt{\textnormal{SSR} \cdot \textnormal{SNR}}$ throughout. Gaussian covariates with covariance $\bSigma$ and i.i.d. $\mcn(0,1)$ noise are redrawn across trials. Target alignment is the cosine similarity between $\bbeta^{(2)}$ and $\bm v$, whereas shift alignment is the cosine similarity between $\tilde \bbeta$ and $\bm v$. Panels (a) and (b) show the risk of the pooled interpolator over varying target and shift alignments, respectively. Solid curves denote theoretical predictions, obtained from Corollary \ref{cor:model_shift_spiked_covariance}. $+$ markers denote simulation averages across $50$ trials, and dashed horizontal lines denote the simulation average risk of the target-only interpolator. In Panel (c), the white region indicates where the pooled min-$\ell_2$-norm interpolator outperforms the target-only interpolator in a majority of $100$ trials; the red curve is the theoretical boundary. Panel (d) shows the risk difference $R(\hat \bbeta^{(2)};\bbeta^{(2)}) - R(\hat \bbeta; \bbeta^{(2)})$ estimated from $100$ trials with $n_1 = 300$ for a range of target and shift alignments; the white curve is the corresponding estimated zero contour.}
    \label{fig:spiked_model_shift}
\end{figure}

\subsection{Risk under a random effects model}\label{subsec:random_model_shift}
In this section, we consider the setting where $\tilde\bbeta$ and $\bbeta^{(2)}$ are randomly drawn vectors. In particular, suppose that $\tilde \bbeta$, $\bbeta^{(2)}$ are i.i.d. uniform draws from the $p-1$ dimensional unit sphere. 
In this case, one can easily determine that the quantities $B_1(\hat \bbeta; \bbeta^{(2)})$ and $B_2(\hat \bbeta; \bbeta^{(2)})$ concentrate around the analogous tracial quantities, while $B_3(\hat \bbeta; \bbeta^{(2)})$ concentrates around $0$. The proof of Theorem \ref{thm:model_shift} then immediately yields the following corollary. 

\begin{cor}\label{cor:REresult}
    Suppose that $\tilde \bbeta$ and $\bbeta^{(2)}$ are drawn as i.i.d. random vectors from the uniform distribution over $\mathcal S^{p-1}$. Then, with high probability over the randomness of $(\tilde \bbeta, \bbeta^{(2)}, \bZ^{(1)}, \bZ^{(2)})$, we have
    \begin{align*}
    V(\hat \bbeta; \bbeta^{(2)}) &= \mathcal V(\hat H_n, \gamma) + O(p^{-1/12})  \\
    B_1(\hat \bbeta; \bbeta^{(2)}) &= \mathcal B_1(\hat H_n, \hat H_n, \gamma) + O(p^{-1/12} ) \\
    B_2(\hat \bbeta; \bbeta^{(2)}) &= \mathcal B_2(\hat H_n, \hat H_n, \gamma) + O(p^{-1/12} )\\
    B_3(\hat \bbeta; \bbeta^{(2)}) &= O(p^{-1/12} ) .
\end{align*}
 Notably, the asymptotic out-of-sample risk of the min-norm interpolator $\hat \bbeta$ is a function solely of the empirical spectral distribution $\hat H_n$ of $\bSigma$.
\end{cor}
Therefore, under the random effects model, the specific geometry of $\tilde \bbeta$ and $\bbeta^{(2)}$ becomes asymptotically irrelevant: the risk depends only on the spectrum of $\bSigma$ and not on any particular alignment with the eigenvectors of $\bSigma$.

This stands in direct contrast with the result in Corollary \ref{cor:model_shift_spiked_covariance}, where the risk could depend heavily on the projections of $\tilde \bbeta$ and $\bbeta^{(2)}$ in the spiked direction. Here, directions for $\tilde \bbeta$ and $\bbeta^{(2)}$ are sampled from an isotropic distribution, which leads to a result more reminiscent of Corollary \ref{cor:model_shift_isotropic}. However, the result here still depends on the spectrum of $\bSigma$ and thus allows us to probe how the covariance structure of the covariates influences the success of transfer procedures. 

\section{Covariate Shift}\label{sec:covariate_shift}
In this section, we characterize the bias and variance of the pooled min-$\ell_2$-norm interpolator under covariate shift. We assume the underlying signal $\bbeta^{(1)}=\bbeta^{(2)}=\bbeta$ stays the same for both datasets, whereas the population covariance matrix changes from $\bSigma^{(1)}$ to $\bSigma^{(2)}$. Under this condition, $B_2(\hat\bbeta;\bbeta^{(2)})=B_3(\hat\bbeta;\bbeta^{(2)})=0$ in \eqref{eqn:bias_analytical}. We then compare the remaining bias term $B_1(\hat\bbeta;\bbeta^{(2)})$ and the variance term $V(\hat\bbeta;\bbeta^{(2)})$ of the pooled min-$\ell_2$-norm interpolator with the target-based interpolator and investigate conditions that determine achievability of positive/negative transfer.

\subsection{Risk under simultaneously diagonalizable covariances}
We assume that the source and target covariance matrices $(\bSigma^{(1)},\bSigma^{(2)})$ are simultaneously diagonalizable, as defined below: 

\begin{definition}[Simultaneously diagonalizable]\label{def:simultaneously_diagonalizable}
For two symmetric matrices $\bSigma^{(1)}$, $\bSigma^{(2)}\in \R^{p\times p}$, we say they are simultaneously diagonalizable if there exist an orthogonal matrix $\bV$ and two diagonal matrices $\bLambda^{(1)},\bLambda^{(2)}$ such that $\bSigma^{(1)} = \bV\bLambda^{(1)}\bV^\top$ and $\bSigma^{(2)} = \bV\bLambda^{(2)}\bV^\top$. Under this condition, the generalization error depends on the joint empirical distribution as well as the geometry between $\bV$ and $\bbeta^{(2)}$. We encode these via the following probability distributions:
\begin{equation}\label{eqn:joint_ESD}
\begin{aligned}
\hat H_p(\lambda^{(1)},\lambda^{(2)}) &:= \frac{1}{p} \sum_{i=1}^p \bm{1}_{\{(\lambda^{(1)},\lambda^{(2)}) \ge (\lambda_{i}^{(1)},\lambda_i^{(2)})\}},\\
\hat G_p(\lambda^{(1)},\lambda^{(2)}) &:= \frac{1}{\|\bbeta^{(2)}\|_2^2} \sum_{i=1}^p \langle \bbeta^{(2)},\bv_i\rangle^2\bm{1}_{\{(\lambda^{(1)},\lambda^{(2)}) \ge (\lambda_{i}^{(1)},\lambda_i^{(2)})\}},
\end{aligned}
\end{equation}
where, $\lambda_i^{(k)} := \Lambda_{ii}^{(k)}$ for $k=1,2$, and $\bv_i$ is the $i$th common eigenvector (the $i$-th column of $\bV$).
\end{definition}

The simultaneous diagonalizability assumption has previously been used in the transfer learning literature (cf. \cite{mallinar2024minimum}). It  ensures that the difference in the covariate distributions of the   source and the target  is solely captured by the eigenvalues $\lambda_{i}^{(1)},\lambda_i^{(2)}$. Characterizing the generalization error when covariance matrices lack a common orthonormal basis appears to be fairly technical, and we defer this to future work. 

Note that in the typical overparametrized setup with a single dataset, the risk of the min-$\ell_2$-norm interpolator is expressed in terms of the empirical distribution of eigenvalues and the angle between the signal and the eigenvector of the covariance matrix (equation (9) in \cite{hastie2022surprises}). Consequently, it is  natural to expect that the generalization error of our estimator will depend on $\hat H_p$ and $\hat G_p$ defined by \eqref{eqn:joint_ESD}, which indeed proves to be the case. The following presents our main result concerning covariate shift.

\begin{thm}[Risk under covariate shift]\label{thm:design_shift}
Suppose Assumptions \ref{as:shared}, \ref{as:design_four_epsilon}, and \ref{as:simul_diag} hold, with joint spectral distributions for $\bSigma^{(1)}, \bSigma^{(2)}$ defined in \eqref{eqn:joint_ESD}. Then, with high probability over the randomness of $(\bZ^{(1)},\bZ^{(2)})$, we have
\begin{align}
V(\hat\bbeta;\bbeta^{(2)}) = -\sigma^2\gamma\int \frac{\lambda^{(2)}(\tilde a_3\lambda^{(1)}+\tilde a_4\lambda^{(2)})}{(\tilde a_1\lambda^{(1)}+\tilde a_2\lambda^{(2)}+1)^2} d \hat H_p(\lambda^{(1)},\lambda^{(2)}) + O(p^{-1/12})\label{eqn:design_shift_V}\\
B(\hat\bbeta;\bbeta^{(2)}) = \|\bbeta^{(2)}\|_2^2 \cdot\int \frac{\tilde b_3\lambda^{(1)}+(\tilde b_4+1)\lambda^{(2)}}{(\tilde b_1\lambda^{(1)}+\tilde b_2\lambda^{(2)}+1)^2} d \hat G_p(\lambda^{(1)},\lambda^{(2)}) + O(p^{-1/12})\label{eqn:design_shift_B}
\end{align}
where $(\tilde a_1,\tilde a_2,\tilde a_3,\tilde a_4)$ is the unique solution, with $\tilde a_1,\tilde a_2$ positive, to the following system of equations:
\begin{equation}\label{eqn:design_shift_a_eqns}
\begin{aligned}
0 &= 1 - \gamma \int \frac{\tilde a_1\lambda^{(1)}+\tilde a_2\lambda^{(2)}}{\tilde a_1\lambda^{(1)}+\tilde a_2\lambda^{(2)}+1} d\hat H_p(\lambda^{(1)},\lambda^{(2)}),\\
0 &= \frac{n_1}{n} - \gamma \int \frac{\tilde a_1\lambda^{(1)}}{\tilde a_1\lambda^{(1)}+\tilde a_2\lambda^{(2)}+1} d\hat H_p(\lambda^{(1)},\lambda^{(2)}),\\
\tilde a_1+\tilde a_2 &= -\gamma\int \frac{\tilde a_3\lambda^{(1)}+\tilde a_4\lambda^{(2)}}{(\tilde a_1\lambda^{(1)}+\tilde a_2\lambda^{(2)}+1)^2}d\hat H_p(\lambda^{(1)},\lambda^{(2)}),\\
\tilde a_1&= -\gamma\int \frac{\tilde a_3\lambda^{(1)}+\lambda^{(1)}\lambda^{(2)}(\tilde a_3\tilde a_2-\tilde a_4\tilde a_1)}{(\tilde a_1\lambda^{(1)}+\tilde a_2\lambda^{(2)}+1)^2}d\hat H_p(\lambda^{(1)},\lambda^{(2)}),
\end{aligned}
\end{equation}
and $(\tilde b_1,\tilde b_2,\tilde b_3,\tilde b_4)$ is the unique solution, with $\tilde b_1,\tilde b_2$ positive, to the following system of equations:
\begin{equation}\label{eqn:design_shift_b_eqns}
\begin{aligned}
0 &= 1 - \gamma \int \frac{\tilde b_1\lambda^{(1)}+\tilde b_2\lambda^{(2)}}{\tilde b_1\lambda^{(1)}+\tilde b_2\lambda^{(2)}+1} d\hat H_p(\lambda^{(1)},\lambda^{(2)}),\\
0 &= \frac{n_1}{n} - \gamma \int \frac{\tilde b_1\lambda^{(1)}}{\tilde b_1\lambda^{(1)}+\tilde b_2\lambda^{(2)}+1} d\hat H_p(\lambda^{(1)},\lambda^{(2)}),\\
0 &= \int \frac{\lambda^{(1)}(\tilde b_3-\tilde b_1\lambda^{(2)})+\lambda^{(2)}(\tilde b_4-\tilde b_2\lambda^{(2)})}{(\tilde b_1\lambda^{(1)}+\tilde b_2\lambda^{(2)} + 1)^2}d\hat H_p(\lambda^{(1)},\lambda^{(2)}),\\
0 &= \int \frac{\lambda^{(1)}(\tilde b_3-\tilde b_1\lambda^{(2)})+\lambda^{(1)}\lambda^{(2)}(\tilde b_3\tilde b_2-\tilde b_4\tilde b_1)}{(\tilde b_1\lambda^{(1)}+\tilde b_2\lambda^{(2)} + 1)^2}d\hat H_p(\lambda^{(1)},\lambda^{(2)}).
\end{aligned}    
\end{equation}
\end{thm}

Note that although $(\tilde a_1,\tilde a_2)=(\tilde b_1,\tilde b_2)$ in the displayed equations above, we intentionally used different notations for clarity, as they are derived through distinct routes in the proof. 

\subsection{Performance comparison example: when does heterogeneity help?}\label{sec:design_shift_example}
In this section, we illustrate the impact of covariate shift on the performance of our pooled min-$\ell_2$-norm interpolator. Since the closed-form expressions for the  bias and variance of the pooled min-$\ell_2$-norm interpolator are complicated (see \eqref{eqn:design_shift_V}, \eqref{eqn:design_shift_B}), we exhibit the impact through stylized examples. Specifically, we consider the setting where the target covariance $\bSigma^{(2)}=\mathbf{I}$ and the source covariance $\bSigma^{(1)}$ has pairs of reciprocal eigenvalues: $\bSigma^{(1)}:=\bM \in \mathcal{S}$, where $\mathcal{S}$ denotes the collection of all diagonal matrices with pairs of reciprocal eigenvalues. To elaborate, suppose that $p$ is  an even integer. In this case, $\bSigma^{(1)}$ is diagonal, with $\lambda_{p+1-i}^{(1)} = 1/\lambda_i^{(1)}$ for $i=1,...,p/2$. Note that $\mathbf{I} \in \mathcal{S}$. This setting leads to a simplification in the bias and the variance, and allows us to study the effects of covariate shift. 

We will analyze the risk of the pooled min-$\ell_2$-norm interpolator $R(\hat{\bbeta}; \bbeta^{(2)})$, defined by \eqref{eq:define_risk}, under this covariate shift model. To emphasize that the difference in the source and target feature covariances is captured by the matrix $\bM$, we denote this risk as  $\hat{R}(\bM)$ through the rest of the section.

Our first result studies the following question: under what conditions does heterogeneity reduce the risk for the pooled estimator?
To this end, we compare $\hat{R}(\bM), \bM \in \mathcal{S}$, with $\hat{R}(\mathbf{I})$, i.e.~the risk under covariate shift with the one without covariate shift.

\begin{prop}[Dependence on sample size]\label{prop:cov_shift_example}
Under the conditions of Theorem \ref{thm:design_shift}, and additionally assuming $\bSigma^{(2)}=\mathbf{I}$, the following holds with high probability:
\begin{enumerate}
\item[(i)] When $n_1 < \min\{p/2,p-n_2\}$, $\hat{R}(\bM) \leq \hat{R}(\mathbf{I})+o(1)$ for any $\bM \in \mathcal{S}$.
\item[(ii)] When $p/2 \leq n_1 <p-n_2$, $\hat{R}(\bM) \geq \hat{R}(\mathbf{I})+o(1)$ for any $\bM \in \mathcal{S}$.
\end{enumerate}
\end{prop}

Proposition \ref{prop:cov_shift_example} is proved in Appendix \appendixref{sec:cov-example}{F.4}. Its implications are as follows: in the overparametrized regime, heterogeneity leads to lower risk for the pooled min-$\ell_2$-norm interpolator if and only if the sample size from the source data is small relative to the dimension, specifically,   smaller than $p/2$. Note that since we consider the overparametrized regime, we always have $p > n_1+n_2$; in other words, $n_1 < p - n_2$. Thus, Proposition \ref{prop:cov_shift_example} gives an exhaustive characterization of when heterogeneity helps in terms of the relations between $p,n_1,n_2$. More generally, an additional question may be posed: how does the degree of heterogeneity affect the risk of the pooled min-$\ell_2$-norm interpolator?  Investigating this  is hard  even for the structured class of covariance matrices $\mathcal{S}$. Therefore, we consider the case when $\bSigma^{(1)}$ has only two eigenvalues: $\lambda_{p+1-i}^{(1)} = 1/\lambda_i^{(1)}=\kappa, \forall 1 \leq i \leq p/2$ for some $\kappa>1$. We denote this covariance matrix as $\bM(\kappa)$. Note that $\bM(1)=\bI$ and $\kappa$ here denotes the degree of heterogeneity. In this special setting, we can establish the following. 

\begin{prop}[Dependence on the degree of heterogeneity]\label{prop:cov_shift_kappa}
Under the conditions of Theorem \ref{thm:design_shift}, and additionally assuming $\bSigma^{(2)}=\mathbf{I}$, the following holds with high probability:
\begin{enumerate}
\item[(i)] When $n_1 < \min\{p/2,p-n_2\}$, $\hat{R}(\bM(\kappa_1)) \leq \hat{R}(\bM(\kappa_2))+o(1)$ for any $\kappa_1 > \kappa_2 >1$.
\item[(ii)] When $p/2 < n_1 <p-n_2$, $\hat{R}(\bM(\kappa_1)) \ge \hat{R}(\bM(\kappa_2))+o(1)$ for any $\kappa_1 > \kappa_2 >1$.
\item[(iii)] If $n_1 = \min\{p/2,p-n_2\}$, then $\hat{R}(\bM(\kappa))$ does not depend on $\kappa \ge 1$.
\end{enumerate}
\end{prop}

The key takeaways from Proposition \ref{prop:cov_shift_kappa} can be summarized as follows:
\begin{enumerate}
    \item[(i)] If the size of the source dataset is small, then the risk is a decreasing function of $\kappa$, i.e., more discrepancy between source and target dataset helps achieve lower risk for the pooled min-$\ell_2$-norm interpolator.
    \item[(ii)] When the size of the  source dataset is large, then larger heterogeneity hurts the pooled min-$\ell_2$-norm interpolator.
\end{enumerate}

Lastly, similar to Section \ref{subsec:isotropic_model_shift}, we compare here the pooled min-$\ell_2$-norm interpolator with the  target-based interpolator \eqref{eqn:interpolator_single}. Note that the risk of the latter under similar settings has been studied in Theorem 2 of \cite{hastie2022surprises} which we cite below:

\begin{prop}\label{prop:single_interpolator_risk_anisotropic}
Consider the estimator \eqref{eqn:interpolator_single}. Let $\hat H_p^{(2)}(\lambda^{(2)}),\hat G_p^{(2)}(\lambda^{(2)})$ be the marginal distribution of $\hat H_p(\lambda^{(1)},\lambda^{(2)}),\hat G_p(\lambda^{(1)},\lambda^{(2)})$ in \eqref{eqn:joint_ESD} on $\lambda^{(2)}$. Under the same setting as Theorem \ref{thm:design_shift}, and further assuming that entries of $\bZ^{(2)}$ have bounded $\phi$ moments for all $\phi\geq 2$. Then with high probability,
\begin{equation}\label{eqn:single_interpolator_risk_anisotropic}
\begin{aligned}
R(\hat\bbeta^{(2)};\bbeta^{(2)})=& \|\bbeta^{(2)}\|_2^2 \biggl(1 + \gamma_2 c_0 \frac{\int\frac{(\lambda^{(2)})^2}{(1+c_0\gamma_2 \lambda^{(2)})^2} d\hat H_p^{(2)}(\lambda^{(2)})}{\int\frac{\lambda^{(2)}}{(1+c_0\gamma_2 \lambda^{(2)})^2} d\hat H_p^{(2)}(\lambda^{(2)})} \biggr)  \int\frac{\lambda^{(2)}}{(1+c_0\gamma_2 \lambda^{(2)})^2} d\hat G_p^{(2)}(\lambda^{(2)})\\
&+\sigma^2\gamma_2 c_0\frac{\int\frac{(\lambda^{(2)})^2}{(1+c_0\gamma_2 \lambda^{(2)})^2} d\hat H_p^{(2)}(\lambda^{(2)})}{\int\frac{\lambda^{(2)}}{(1+c_0\gamma_2 \lambda^{(2)})^2} d\hat H_p^{(2)}(\lambda^{(2)})} + O(n^{-1/7}),
\end{aligned}
\end{equation}
where $c_0$ is the unique non-negative solution of
\begin{equation*}
1 - \frac{1}{\gamma} = \int\frac{\lambda^{(2)}}{1+c_0\gamma \lambda^{(2)}} d\hat H_p^{(2)}(\lambda^{(2)}).
\end{equation*}
\end{prop}

We remark that using similar proof techniques as Theorem \ref{thm:design_shift}, the additional bounded moments condition in Proposition \ref{prop:single_interpolator_risk_anisotropic} can be relaxed to Assumption \ref{as:design_four_epsilon}. We omit the details. Regretfully, although the precise risks for both interpolators are available from Theorem \ref{thm:design_shift} and Proposition \ref{prop:single_interpolator_risk_anisotropic}, one cannot obtain a simple result even for $\bSigma^{(2)}=\bI$ due to the complexity of expressions. Thus we make the comparison via solving the expressions numerically in the next section.

\subsection{Numerical examples under covariate shift}
Similar to Section \ref{subsec:isotropic_model_shift}, we illustrate the applicability of our results in finite samples and compare the performances of the pooled min-$\ell_2$-norm interpolator and its target-based counterpart. We pick the special setting of Proposition \ref{prop:cov_shift_kappa} where the target distribution has covariance $\bI$, while the source distribution has a covariance matrix with two distinct eigenvalues  $\kappa$ and $1/\kappa$, each having multiplicity $p/2$. The choice of covariance matrices ensures that both have equal determinants. In Figure \ref{fig:design_shift}, we illustrate how the generalization error varies with $n_1,n_2, p$, the SNR, and the covariate shift magnitude $\kappa$. Again, we include the generalization error of the OLS estimator for the underparametrized regime $p<n$ for completeness.

The key insights from the plots are summarized as follows. Figure \ref{fig:design_shift}(a) illustrates the generalization error of the pooled min-$\ell_2$-norm interpolator, showcasing the double descent phenomena across varying $p,n_1,n_2$. The generalization error of the target-based interpolator is represented by dotted horizontal lines. We observe that in the overparametrized regime, if $n_1$ is small, the pooled estimator consistently outperforms the target-based estimator. 
Figure \ref{fig:design_shift}(b) presents our findings from Proposition \ref{prop:cov_shift_kappa}. For small values of $n_1$, the generalization error decreases with the degree of heterogeneity, i.e., the value of $\kappa$. However, for $n_1> \min \{p/2, p-n_2\}$, higher values of $\kappa$ result in higher generalization error. A notable scenario occurs at $n_1 = \min \{p/2, p-n_2\}$, where the generalization error becomes independent of $\kappa$, causing all curves to intersect at a single point.
Finally, Figure \ref{fig:design_shift}(c) demonstrates the dependence of the generalization error on the SNR. Higher SNR values lead to significantly higher error when $n_1$ is small, but the curves tend to converge at the interpolation threshold, i.e., when the total sample size matches the number of covariates. As in Figure \ref{fig:design_shift}(a), the null estimator $\hat \bbeta_{\text{null}} = 0$ can outperform both interpolators throughout the range of source sample sizes $n_1$ for sufficiently small SNR (e.g., $\textnormal{SNR} = 1$ in Figure \ref{fig:design_shift}(c)). In such SNR regimes, this comparison should not be interpreted as identifying a practically useful transfer procedure: see Appendix \appendixref{sec:null_comparison}{J.4} for more. 

\begin{figure}
    \centering
    \begin{subfigure}[b]{0.48\linewidth}
    \includegraphics[width=\linewidth]{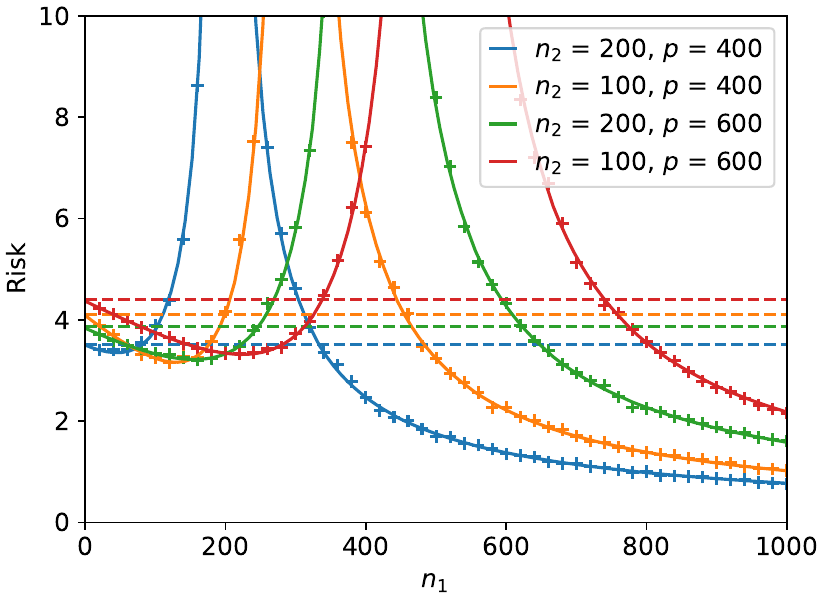}
    \caption{Fixing $\kappa=4,\textnormal{SNR}=5$}
    \end{subfigure}
    \begin{subfigure}[b]{0.48\linewidth}
    \includegraphics[width=\linewidth]{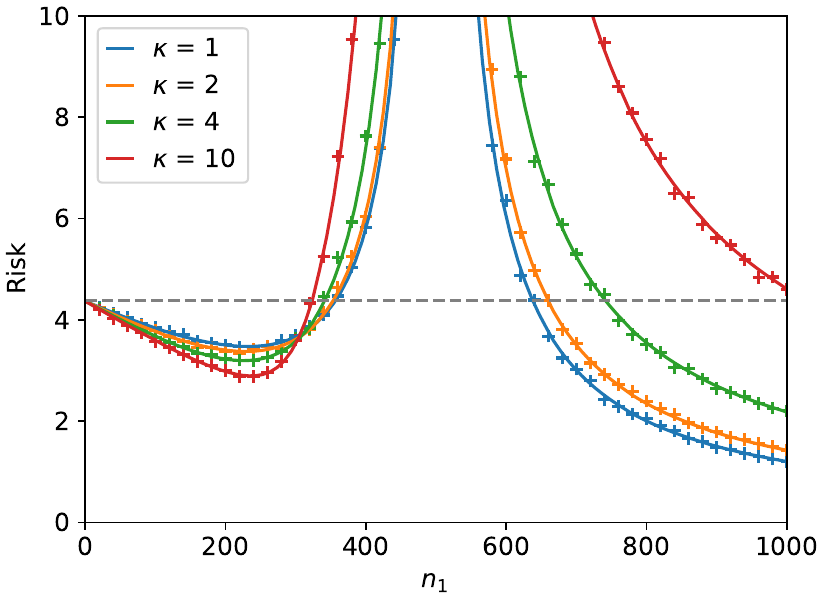}
    \caption{Fixing $n_2=100,p=600,\textnormal{SNR}=5$}
    \end{subfigure}
    \begin{subfigure}[b]{0.48\linewidth}
    \includegraphics[width=\linewidth]{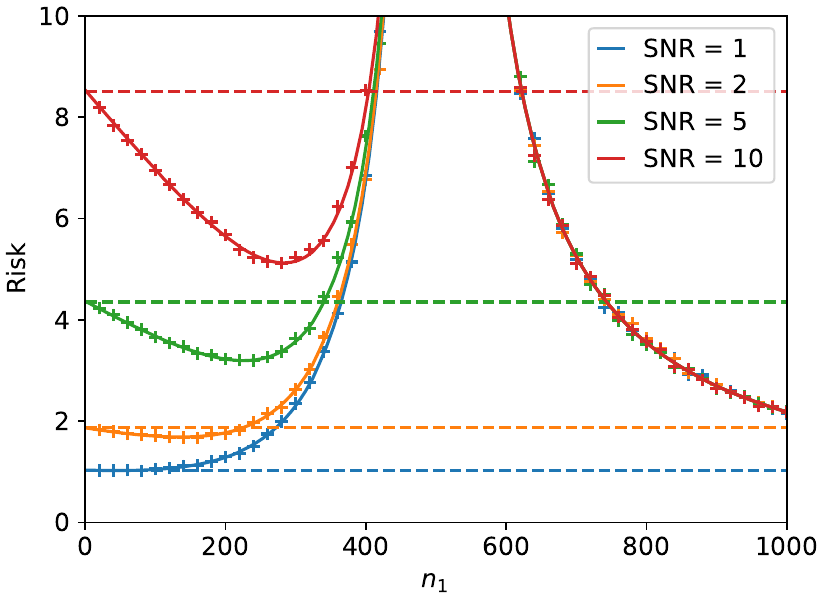}
    \caption{Fixing $n_2=100,p=600,\kappa=4$}
    \end{subfigure}
    \caption{Generalization error of the pooled min-$\ell_2$-norm interpolator under covariate shift. Solid lines: theoretically predicted values. $+$ marks: empirical values. Dotted horizontal line: Risk for min-$\ell_2$-norm interpolator using only the target data. Design choices: $\beta_i^{(1)}=\beta_i^{(2)} \sim \mcn(0,\sigma_\beta^2)$ where $\sigma_\beta^2=\textnormal{SNR} / p$, $\bSigma^{(2)} = \bI$ and $\bSigma^{(1)}$ has two distinct eigenvalues $\kappa$ and $1/\kappa$, $\bX^{(k)}$ has i.i.d. rows from $\mcn(0,\bSigma^{(k)})$. The signals and design matrices are then held fixed. We generate $50$ random $\epsilon_i \sim \mcn(0,1)$ and report the average empirical risks.}
    \label{fig:design_shift}
\end{figure}

\section{Beyond Gaussianity and Other Extensions}\label{sec:extensions}
In this section, we discuss several extensions to our above results.
\subsection{Universality of out-of-sample prediction risk} Our model shift result in
Theorem \ref{thm:model_shift} assumes that $(\bZ^{(1)}, \bZ^{(2)})$ consist of i.i.d. Gaussian entries (Assumption \ref{as:design_gaussian}). Here, we show that the same results hold when Assumption \ref{as:design_gaussian} is replaced by Assumption \ref{as:design_all_moments}.
Our universality result is the following.
\begin{thm}[Universality of the risk]\label{thm:universality}
    Suppose Assumptions \ref{as:shared}, \ref{as:design_all_moments}, and \ref{as:shared_covariance} hold. Let $(\tilde \bZ^{(1)}, \tilde \bZ^{(2)})$ consist of i.i.d. standard Gaussian entries, and let $\tilde \bX^{(k)} = \tilde \bZ^{(k)} (\bSigma^{(k)})^{1/2}$. For $j = 1,2,3$, let $\tilde B_j$ denote the respective bias term $B_j$ under $(\tilde \bX^{(1)}, \tilde \bX^{(2)})$, rather than $(\bX^{(1)}, \bX^{(2)})$. With high probability over $(\tilde \bZ^{(1)}, \tilde \bZ^{(2)}, \bZ^{(1)}, \bZ^{(2)})$, when $\|\bbeta^{(2)}\|_2 = O(1)$ and $\|\tilde \bbeta\|_2 = O(1)$, we have that
\begin{align}
|B_1(\hat \bbeta; \bbeta^{(2)}) - \tilde B_1(\hat \bbeta; \bbeta^{(2)})|
&= o(1) \quad &
|B_2(\hat \bbeta; \bbeta^{(2)}) - \tilde B_2(\hat \bbeta; \bbeta^{(2)})|
&= o(1) \\
|B_3(\hat \bbeta; \bbeta^{(2)}) - \tilde B_3(\hat \bbeta; \bbeta^{(2)})|
&= o(1) \quad &
|V(\hat \bbeta; \bbeta^{(2)}) - \tilde V(\hat \bbeta; \bbeta^{(2)})|
&= o(1).
\end{align}
\end{thm}
It immediately follows that the characterization of the out-of-sample prediction risk under model shift in Theorem \ref{thm:model_shift} also holds when Assumption \ref{as:design_gaussian} is replaced by Assumption \ref{as:design_all_moments}.

We complement Theorem \ref{thm:universality} with extensive finite-sample simulations in Appendix \appendixref{sec:university_sims}{J.1}. These simulations compare the theoretical risk predictions for the pooled min-$\ell_2$-norm interpolator against empirical risks under a variety of synthetic covariate distributions, as well as semi-synthetic covariates derived from real human genomics data. These simulations validate Theorem \ref{thm:universality} for several covariate distributions satisfying Assumptions \ref{as:shared} and \ref{as:design_all_moments}, and further suggest that the same risk predictions remain accurate for heavier-tailed covariates and real-world covariates that may not exactly satisfy these assumptions.

\subsection{Misspecified model}
Following \cite{hastie2022surprises}, we consider a misspecified model where the true model (for both the source and target data) is linear but we only observe a subset of the features. Formally, we assume that the data comes from linear models given by
\begin{equation}\label{eq:misspecified_model}
    \bm y^{(k)} = \bm X^{(k)} \bbeta^{(k)} + \bm W^{(k)} \btheta^{(k)} + \bep^{(k)}, k = 1,2
\end{equation}
where, once again, $k = 1$ denotes the source and $k = 2$ denotes the target. For $k = 1,2$, the matrix $\bm X^{(k)} \in \R^{n_k \times p}$ contains the observed covariates, while $\bm W^{(k)} \in \R^{n_k \times p_{m,k}}$ contains unobserved covariates. For each $k = 1,2$, we assume that $\bX^{(k)}$ and $\bW^{(k)}$ are independent, with covariance matrices $\bSigma^{(k)}$ and $\bSigma_m^{(k)}$, respectively.  
Then, for a test point $(\bm x_0, \bm w_0)$ drawn independently from the target distribution, we study the out-of-sample prediction risk for an estimator $\hat \bbeta$, defined in this case as follows: 
\begin{equation}
    R(\hat \bbeta; \bbeta^{(2)}, \btheta^{(2)}) = \E[(\bm x_0^\top \hat \bbeta - \bm x_0^\top \bbeta^{(2)} - \bm w_0^\top \btheta^{(2)})^2 \mid \bm X].
\end{equation}
This risk admits the following decomposition \cite[Lemma 2]{hastie2022surprises}:
\begin{align}
    R(\hat \bbeta; \bbeta^{(2)}, \btheta^{(2)}) &= \underbrace{\E[(\bm x_0^\top \hat \bbeta - \bm x_0^\top \bbeta^{(2)} - \E[\bm w_0 \mid \bm x_0]^\top \btheta^{(2)})^2 \mid \bm X]}_{R_1(\hat \bbeta; \bbeta^{(2)}, \btheta^{(2)})} \\
    &\qquad + \underbrace{\E[(\E[\bm w_0 \mid \bm x_0]^\top \btheta^{(2)} - \bm w_0^\top \btheta^{(2)})^2]}_{R_2(\hat \bbeta; \bbeta^{(2)}, \btheta^{(2)})}.
\end{align}
We generalize our results for this model as follows.
\begin{thm}[Risk under misspecification and model shift]\label{thm:misspecification}
    Assume the misspecified model \eqref{eq:misspecified_model}. Under Assumptions \ref{as:shared}, \ref{as:design_gaussian}, and \ref{as:shared_covariance} (with assumptions on $\bm W$ paralleling those of $\bm X$), 
    with high probability over the randomness of $(\bm Z^{(1)}, \bm Z^{(2)})$, we obtain that
    \begin{align}
        R_1(\hat \bbeta; \bbeta^{(2)}, \btheta^{(2)}) &=  \mathcal B_1(\hat H_n, \hat G_n^{\bbeta^{(2)}}, \gamma) + \mathcal B_2(\hat H_n, \hat G_n^{\tilde \bbeta}, \gamma) + \mathcal B_3(\hat H_n, \hat G_n^{(b)}, \gamma)  \\
        &\qquad + \biggl(1 + \frac{n_1}{n}\frac{\| \btheta^{(1)}\|_{\bSigma^{(1)}_m}^2}{\sigma^2} + \frac{n_2}{n} \frac{\|\btheta^{(2)}\|_{\bSigma^{(2)}_m}^2}{\sigma^2} \biggr) \mathcal V(\hat H_n, \gamma) \\
        &\qquad + O(p^{-1/12} (\|\bbeta^{(2)}\|_2 + \|\tilde \bbeta\|_2)^2) \\
        R_2(\hat \bbeta; \bbeta^{(2)}, \btheta^{(2)}) &= \|\btheta^{(2)}\|_{\bSigma^{(2)}_m}^2.
    \end{align}
\end{thm}
Thus, under this form of misspecification, the unobserved covariates influence the out-of-sample prediction risk through two additive components. The first component $R_1$ closely resembles the well-specified case, but includes an additional term reflecting the impact of misspecification through both the source and target datasets, thereby capturing errors arising from fitting the model on incomplete covariates.

The second component $R_2$ is new and depends solely on the unobserved covariates in the target distribution. It represents an irreducible error due to misspecification, capturing the discrepancy between the true response and its best approximation based solely on the observed covariates.

In Appendix \appendixref{sec:misspec_sim}{J.2}, we provide finite-sample simulations that display the risk of the pooled min-$\ell_2$-norm interpolator for different amounts of target and source misspecification. These simulations confirm that our theoretical predictions in Theorem \ref{thm:misspecification} closely match those seen in finite samples, thus again providing evidence for the  finite-sample validity of our predictions.

\subsection{Bias-corrected estimator}\label{sec:bias_corrected}
In the setting of model shift, we can also consider a modification of the min-norm interpolator \eqref{eqn:analytical_pinv} that corrects for the bias of  the combined min-norm interpolator with a holdout subset of the target dataset. That is, our \textit{bias-corrected} estimator---which we denote as $\hat \bbeta_{\text{BC}}$---is formed by a two-step procedure, parametrized by a regularization parameter $\lambda_\delta > 0$ and a holdout parameter $\kappa \in (0,1)$. We present this procedure in Algorithm \ref{alg:bias_corrected}.
\begin{algorithm}
    \caption{Bias-corrected estimator}
    \label{alg:bias_corrected}
    \KwIn{Observed data $\{(\bm y^{(k)}, \bm X^{(k)})\}_{k=1}^2$ and parameters $\lambda_\delta > 0$ and $\kappa \in (0,1)$.}
    Let $\mathcal{I} \subset [n_2]$ be  a random subset of size $\lfloor(1-\kappa) n_2\rfloor$\;
    Compute
    \begin{equation}
    \hat\bbeta_{\mathcal I} = \argmin \left \{ \|\bb\|_2 \ \ \text{s.t.}\ \by^{(1)} = \bX^{(1)}\bb\ \text{and}\ \by^{(2)}_{\mathcal I} = \bX^{(2)}_{\mathcal I} \bb\right \}
    \end{equation}
    as the first-stage min-norm interpolator\;
    Compute 
    \begin{equation*}
        \hat{\bdelta}_{\lambda_\delta} = \argmin_{\bdelta \in \R^p} \biggl\{ \|\bm y^{(2)}_{\mathcal{I}^c} - \bm X^{(2)}_{\mathcal{I}^c} (\hat \bbeta_{\mathcal I} + \bdelta)\|_2^2 + \lceil\kappa n_2\rceil\lambda_\delta \|\bdelta\|_2^2 \biggr\}
    \end{equation*}
    where $\lambda_\delta > 0$ is the regularizing parameter\;
    \KwOut{$\hat{\bbeta}_{\text{BC}} = \hat \bbeta_{\mathcal I} + \hat{\bdelta}_{\lambda_\delta}$}
\end{algorithm}
Note that the bias-corrected estimator $\hat \bbeta_{\text{BC}}$ is closely related to the estimator considered in \cite{dar2021common}, which regularizes the target estimator toward a source-only estimator. It is also conceptually related to the Transfer MNI estimator of \cite{kim2026transfer}, which first trains a source-only min-$\ell_2$-norm interpolator and then interpolates the target data while minimizing the $\ell_2$ distance to the source estimator. Our estimator and the estimators in \cite{dar2021common, kim2026transfer} all involve a two-step estimation procedure, where the second stage produces an estimator based on the first estimator. Our construction differs in that the reference estimator is itself obtained by pooling part of the target data with the source data.

We can then prove the following result for the out-of-sample prediction risk of the bias-corrected estimator:
\begin{thm}[Risk of the bias-corrected estimator]\label{thm:bias_corrected_risk}
    Fix $\kappa \in (0,1)$ and $\lambda_\delta > 0$. Let $\gamma_{\kappa} = p / \lceil \kappa n_2 \rceil$, and define $m_n(-\lambda_\delta)$ as the unique positive solution of
    \begin{equation}
        \gamma_\kappa \lambda_\delta m_n(-\lambda_\delta)^2 + (1-\gamma_\kappa + \lambda_\delta) m_n(-\lambda_\delta) - 1 = 0.
    \end{equation}
    Additionally, define $m_{n,1}(-\lambda_\delta)$ as
    \begin{equation*}
        m_{n,1}(-\lambda_\delta) = \frac{m_n(-\lambda_\delta) - \lambda_\delta m_n(-\lambda_\delta)^2}{1 + \gamma_\kappa \lambda_\delta m_n(-\lambda_\delta)^2}.
    \end{equation*}
    Suppose Assumptions \ref{as:shared} and \ref{as:design_gaussian} hold, and assume that $\bSigma^{(1)} = \bSigma^{(2)} = \bm I$ with  $\|\bbeta^{(2)}\|_2 = O(1)$ and $ \|\tilde \bbeta\|_2 = O(1)$. With high probability over the randomness of $(\bZ^{(1)}, \bZ^{(2)})$, we have
    \begin{align*}
        R(\hat \bbeta_{\text{BC}}; \bbeta^{(2)}) &= \frac{\lambda_\delta^2(1 + \gamma_\kappa m_{n,1}(-\lambda_\delta))}{(\lambda_\delta + 1 - \gamma_\kappa + \gamma_\kappa \lambda_\delta m_n(-\lambda_\delta))^2} \cdot R(\hat \bbeta_{\mathcal I}; \bbeta^{(2)}) \\
        &\qquad + \sigma^2 \gamma_\kappa \cdot \frac{\lambda_\delta + 1 - \gamma_\kappa + \gamma_\kappa \lambda_\delta m_n(-\lambda_\delta) - \lambda_\delta (1 + \gamma_\kappa m_{n,1}(-\lambda_\delta))}{(\lambda_\delta + 1 - \gamma_\kappa + \gamma_\kappa \lambda_\delta m_n(-\lambda_\delta))^2}  + o(1),
    \end{align*}
    where $R(\hat \bbeta_\mathcal I, \bbeta^{(2)})$ is the risk of the pooled interpolator with source sample size $n_1$ and target sample size $\lfloor (1-\kappa) n_2\rfloor$, as characterized in Theorem \ref{thm:model_shift}.
\end{thm}
Theorem \ref{thm:bias_corrected_risk} shows that, in the isotropic setting, the risk of the bias-corrected estimator is equal to the sum of two terms. The first term is the risk of the interpolator fitted with the source data and a split of the target data, multiplied by a shrinkage factor induced by the second-stage correction. The second term is the additional variance incurred by fitting the correction on the held-out target split. Therefore, the bias-corrected estimator improves over the interpolator in \eqref{eqn:interpolator} precisely when the reduction  in the first-stage risk outweighs the additional second-stage variance cost. 

In Appendix \appendixref{sec:bias_corected_sims}{J.3}, we provide finite-sample simulations that compare the bias-corrected estimator with the target-only and pooled min-$\ell_2$-norm interpolators. Across a range of SNR, SSR, sample-split parameter $\kappa$, and second-stage regularization parameter $\lambda_\delta$ values, the empirical risks closely track the theoretical values provided in Theorem \ref{thm:bias_corrected_risk}.

\section{Proof Outline}\label{sec:proof_outline}
In this section, we outline the key ideas behind the proofs of our main results in Sections \ref{sec:model_shift}, \ref{sec:covariate_shift}, and \ref{sec:extensions}, highlighting crucial auxiliary results along the way. All convergence statements will be presented in an approximate sense, with comprehensive proof details provided in the Appendix.

To circumvent the difficulty of dealing with the pseudo-inverse formula in \eqref{eqn:analytical_pinv}, we utilize the connection between the pooled min-$\ell_2$-norm interpolators and ridgeless least squares interpolation. This says that to study the interpolator, it suffices to study the following ridge estimator
\begin{equation}\label{eqn:ridge}
\hat\bbeta_\lambda = \left(\bX^{(1)\top}\bX^{(1)} + \bX^{(2)\top}\bX^{(2)}+n\lambda\bI\right)^{-1} \left(\bX^{(1)\top}\by^{(1)} + \bX^{(2)\top}\by^{(2)}\right),
\end{equation} 
and then study its $\lambda \rightarrow 0$ limit. 

In the remainder of this section, we focus on the proofs of Theorems \ref{thm:model_shift}, \ref{thm:design_shift}, and \ref{thm:universality}, as these results are the most technically challenging and intricate. We note that, while the proofs of Theorems \ref{thm:misspecification} and \ref{thm:bias_corrected_risk} require careful analysis, they reuse the same machinery rather than introducing new random matrix inputs. Therefore, we omit a discussion of these proofs here and defer them to Appendix \appendixref{sec:other_extension_proofs}{H}.

\subsection{Model shift and covariate shift}
Owing to the decomposition \eqref{eq:risk_analytical}, we analyze the bias and the variance terms separately. With this in mind, our proofs of Theorems \ref{thm:model_shift} and \ref{thm:design_shift} consist of the following four steps: 
\begin{enumerate}
    \item[(i)] For a fixed value of $\lambda>0$, we analyze the risk of the pooled ridge estimator defined by \eqref{eqn:ridge}, i.e., we show there exists $\mathcal{R}$ such that 
     $R(\hat\bbeta_{\lambda};\bbeta^{(2)}) \overset{p \rightarrow \infty}{\longrightarrow} \mathcal{R}(\hat\bbeta_{\lambda};\bbeta^{(2)}),$ where $R$ is defined as in \eqref{eq:define_risk}. [see Theorem \appendixref{thm:model_shift_ridge}{C.2} and Theorem \appendixref{thm:design_shift_ridge}{C.3} for the proof under model shift and covariate shift respectively].
    \item[(ii)] Show that with high probability, the difference between the (random) risks of the min-$\ell_2$-norm interpolator and the ridge estimator vanishes as $\lambda \rightarrow 0^+$, i.e., $R(\hat\bbeta_{\lambda};\bbeta^{(2)})  -R(\hat\bbeta;\bbeta^{(2)})\overset{\lambda \rightarrow 0^+}{\longrightarrow} 0$.
    \item[(iii)] Analyze the limit of the (deterministic) risk of $\hat \bbeta_\lambda$, i.e., $\mathcal{R}(\hat\bbeta_{\lambda};\bbeta^{(2)}) \overset{\lambda \rightarrow 0^+}{\longrightarrow} \mathcal{R}(\hat\bbeta;\bbeta^{(2)})$.
    \item[(iv)] Find a suitable relation between $\lambda$ and $p$ such that the aforementioned three convergence relationships are simultaneously satisfied.
\end{enumerate}

The most nontrivial parts of these results are to prove the convergence (i) and obtain the precise expression for $\mathcal{R}(\hat\bbeta_{\lambda};\bbeta^{(2)})$. For (ii), one can rewrite the difference $R(\hat\bbeta_{\lambda};\bbeta^{(2)})  -R(\hat\bbeta;\bbeta^{(2)})$ in terms of the eigenvalues of some symmetric random matrices such as $\bX^{(1)\top}\bX^{(1)} + \bX^{(2)\top}\bX^{(2)}$. Next, using high probability bounds of the largest and smallest eigenvalues (see Lemma \appendixref{lemma:ESD_bounded}{D.3} and Corollary \appendixref{cor:ESD_bounded_weaker}{D.4}), we are able to establish high probability convergence given by (ii) above. Part (iii) only involves deterministic quantities and therefore classical analysis techniques. Finally, Step (iv) reduces to  optimization problems over $\lambda$.

The rest of this section is devoted to outlining the proof for part (i). These steps for model shift and covariate shift are detailed in Appendix \appendixref{sec:proof:model_shift}{E} and Appendix \appendixref{sec:proof:design_shift}{F}, respectively.

\subsubsection{Model shift} 
After decomposing $R(\hat \bbeta_\lambda; \bbeta^{(2)})$ into the variance term and the three bias components, the terms requiring novel analyses are $B_2(\hat \bbeta_\lambda; \bbeta^{(2)})$ and $B_3(\hat \bbeta_\lambda; \bbeta^{(2)})$. These involve products of the source-only and pooled covariance matrices
\begin{equation*}
\frac{1}{n} \bm Z^{(1)\top}\bm Z^{(1)} \quad \text{and} \quad \frac{1}{n} \bm Z^{(1)\top}\bm Z^{(1)} + \frac{1}{n} \bm Z^{(2)\top}\bm Z^{(2)},
\end{equation*}
respectively. The key difficulty arises from this asymmetry: some terms depend only on the source design, while others involve both source and target designs. As a result, standard random matrix results---typically formulated for functions of a single sample covariance matrix---cannot be applied directly. This limitation also appears in related work: e.g., \cite{yang2020analysis} notes that this technical hurdle prevents their OLS analysis for the underparametrized model shift setting from generalizing beyond isotropic designs. Similarly, the analysis of a different transfer estimator in  \cite{dar2022double} relies heavily on properties of Wishart matrices and thus fails to extend to non-Gaussian covariates.

Our proof resolves this asymmetry by exploiting the exchangeability of the rows of the whitened design matrices $\bm Z^{(1)}$ and $\bm Z^{(2)}$. At the level of expectations, source-only insertions can be converted into weighted combinations of pooled-resolvent expressions. This reduction to a single covariance matrix is delicate because it must be performed inside anisotropic quadratic forms. Therefore, we carefully apply Gaussian concentration tools (e.g., Poincar\'e-type inequalities and the Efron-Stein inequality) to upgrade our expectation-level reductions to high-probability statements. Finally, we are able to apply anisotropic local laws and related resolvent techniques (e.g., \cite{knowles2017anisotropic, hastie2022surprises}) to analyze the reduced matrix expressions.

\subsubsection{Covariate shift}
For covariate shift, $B_2(\hat\bbeta_{\lambda},\bbeta^{(2)})=B_3(\hat \bbeta_\lambda, \bbeta^{(2)}) = 0$ since $\tilde\bbeta=0$, eliminating the need for dealing with quadratic terms with free addition. One can therefore relax the Gaussianity assumption on $(\bZ^{(1)},\bZ^{(2)})$ to the one in Theorem \ref{thm:design_shift}. However, new challenges arise with anisotropic covariances $(\bSigma^{(1)},\bSigma^{(2)})$. For conciseness, we illustrate the representative variance term, which is formulated as (see (F.2) 
in Section \appendixref{sec:proof:design_shift}{F}) 
\begin{equation}\label{eqn:outline_V}
\begin{aligned}
V(\hat\bbeta_\lambda;\bbeta^{(2)})=& \frac{\partial}{\partial\lambda} \left\{\frac{\sigma^2\lambda}{n}\Tr\left(\bSigma^{(2)}(\hat\bSigma+\lambda\bI)^{-1}\right)\right\}\\
=&\frac{\partial}{\partial\lambda} \left\{\frac{\sigma^2\lambda}{n}\Tr\left((\bLambda^{(2)})^{1/2}\left(\bW +\lambda\bI\right)^{-1}(\bLambda^{(2)})^{1/2}\right)\right\},
\end{aligned}
\end{equation}
where
$$\bW := (\bLambda^{(1)})^{1/2}\bV^\top\hat\bSigma_{\bZ}^{(1)}\bV(\bLambda^{(1)})^{1/2} +(\bLambda^{(2)})^{1/2}\bV^\top\hat\bSigma_{\bZ}^{(2)}\bV(\bLambda^{(2)})^{1/2}.$$

This simplification highlights the necessity of the jointly diagonalizable assumption for $(\bSigma^{(1)},\bSigma^{(2)})$. The core challenge, therefore, reduces to characterizing the behavior of $(\bW+\lambda\bI)^{-1}$ under the setting of Theorem \ref{thm:design_shift}. In other words, we aim to characterize the resolvent of $\bW$, a sum of two sample-covariance-like matrices with anisotropic covariances. To this end, we present Theorem \appendixref{thm:anisotropic_law}{F.1} (omitted here for the sake of space), a novel type of anisotropic local law \cite{knowles2017anisotropic}, which might be of independent interest. Convergence of the trace term in \eqref{eqn:outline_V} follows as a corollary of our anisotropic local law, and further analysis of the derivative yields the value of $V(\hat\bbeta_\lambda;\bbeta^{(2)})$.

\subsection{Universality}
Our proof of Theorem \ref{thm:universality} (detailed in Appendix \appendixref{sec:proof:universality}{G}) uses a Lindeberg replacement argument at the level of the full risk decomposition. We replace entries of the whitened design matrices $\bZ^{(k)}$ by their Gaussian counterparts and show that the distribution of each risk component is asymptotically unchanged, as $n,p \to \infty$. The nontrivial point is that the risk components are not smooth bounded functions of the data uniformly in the ridgeless regime, as many applications of Lindeberg replacement arguments often assume. Rather, they contain resolvents whose operator norms grow as $\lambda \downarrow 0$; additionally, as with the proofs of Theorems \ref{thm:model_shift} and \ref{thm:design_shift}, the analysis is complicated by the interactions between source-only, target-only, and pooled covariance terms.

Therefore, we first demonstrate universality of the risk of the ridge estimator $\hat \bbeta_\lambda$ for a fixed $\lambda > 0$. Using both the Sherman-Morrison formula and Taylor expansions, we reduce the replacement error to bounds on the derivatives of resolvent-based functionals. These bounds must be sharp enough to be summed over all replaced elements and must retain the explicit dependence on $\lambda$. Therefore, both our Taylor expansions and the subsequent derivative bounds must be carefully executed to maintain sufficient control over the error. Once we have controlled the replacement error for a given $\lambda > 0$, we take the limit $\lambda \downarrow 0$ to establish universality of the risk of the interpolator $\hat \bbeta$.

\section{Discussion and Future Directions}\label{sec:future}
In this paper, we characterize the generalization error of the pooled min-$\ell_2$-norm interpolator under overparametrization and distribution shift. 
Our study underscores how the degree of model shift or covariate shift influences the risk associated with the interpolator. Mathematically, we derive a novel anisotropic local law, 
which may be of independent interest in random matrix theory.
We conclude by identifying several avenues for future research and discussing some limitations of our current findings:
\begin{enumerate}
    \item[(i)]\textbf{Expanding to Multiple Datasets:} In Appendix \appendixref{sec:multiple}{B}, we establish the risk of the pooled min-$\ell_2$-norm interpolator with multiple source datasets under model shift. Extending our results under covariate shift to the case of multiple source datasets is equally important. We anticipate that this would be feasible by  extending our approach to incorporate more general anisotropic local laws. However, we defer this study to future work in the interest of space. 

\item[(ii)] \textbf{Combined Model and Covariate Shift:} 
In Sections \ref{sec:model_shift} and \ref{sec:covariate_shift}, we separately analyze model shift and covariate shift. A natural next step is to characterize the risk of the min-$\ell_2$-norm interpolator when both shifts occur simultaneously. Such settings are common in real-world transfer learning settings and have been partially explored in \cite{yang2020analysis} for the underparametrized regime.
Extending our precise risk characterizations to such settings would require significant new ideas for handling mixed random matrix terms involving multiple covariance structures and signal-shift directions.

\item[(iii)] \textbf{Comparison with Other Estimators:} Our pooled min-$\ell_2$-norm interpolator covers early and intermediate min-$\ell_2$-norm interpolation under a single umbrella thanks to the alternate representation \eqref{eqn:intermediate_fusion}. Additionally, our bias-corrected estimator $\hat \bbeta_{\text{BC}}$ in Section \ref{sec:bias_corrected} captures a non-interpolating intermediate-fusion estimator. However, a more complete picture of interpolation under transfer learning and overparametrization necessitates comprehensive comparisons with the following different types of estimators: 1) Late-fused min-$\ell_2$-norm interpolators, which combine min-$\ell_2$-norm interpolators that are individually obtained from different datasets; 2) Other types of interpolators including min-$\ell_1$-norm ones \cite{liang2022precise}; 3) Other estimators in overparametrized regimes including the penalized ones, beyond the ridge estimators covered in Appendix \appendixref{sec:ridge}{C}.

\item[(iv)] \textbf{Extension to Nonlinear Models:} Recent literature has explored the connection between linear models and more
complex models such as neural networks \cite{allen2019convergence,du2018gradient,jacot2018neural,mei2019mean}. To be precise, for i.i.d. data $(y_i, z_i)$, $i \le n$, $y_i \in \mathbb R$, $z_i \in \mathbb R^d$. Consider learning a neural network with
weights $\btheta \in \mathbb R^p$, $f(.; \btheta): \mathbb R^d \mapsto \mathbb R$, $z \mapsto f(z, \btheta)$, where the form of $f$ need not be related to the distribution of $(y_i, z_i)$. In some settings, the number of parameters $p$ is so large that training effectively changes $\btheta$ only by a
small amount with respect to a random initialization $\btheta_0 \in \mathbb R^p$. Then, given $\btheta_0$ such that $f(z, \btheta_0) \approx 0$, one obtains that the neural network can be approximated by $z \mapsto \nabla_{\btheta} f(z, \btheta_0)^\top \beta$, where we have reparametrized $\btheta=\btheta_0+ \bbeta$. Such models, despite being nonlinear in $z_i$s, are linear in $\bbeta$. Moving beyond high-dimensional linear regression, one might ask about the statistical consequences of such linearization in the context of transfer learning. We leave such directions for future research.
\item[(v)] \textbf{Relaxing Simultaneous Diagonalizability:} Our covariate shift result (Theorem \ref{thm:design_shift}) assumes that the source and target covariance matrices are simultaneously diagonalizable. A natural next step is to remove this assumption and allow non-commuting source-target covariance structures, which would allow us to extend our theory to covariate shift settings where eigenspaces change, in addition to eigenvalues. In Appendix \appendixref{sec:sketch_nonsimul}{I},  we sketch a potential route that involves iteratively applying anisotropic local laws, thus allowing us to isolate individual components. However, the required analysis  is substantially more delicate; we therefore defer a complete treatment to future work.
\end{enumerate}

\section*{Acknowledgements}
P.S. would like to acknowledge partial funding from NSF DMS-$21134$, NSF CAREER Award $2440824$ and the Office of
Naval Research Award N00014-26-1-2144.

\bibliography{main.bib}
\newpage
\appendix

\title{Supplement to ``Generalization error of min-norm interpolators in transfer learning''}

\section{A Data-Driven Choice of Interpolator}\label{sec:choice_of_interpolator}

Our findings in  Section \ref{subsec:isotropic_model_shift} highlight that, to achieve the least possible risk for pooled min-$\ell_2$-norm interpolator, 
one should neither always pool both datasets nor always use all samples in either datasets. In fact, if SNR and SSR were \textit{known}, it would be possible to identify subsets of source and target datasets, having sample size $\tilde n_1, \tilde n_2$ respectively, such that the risk of the pooled min-$\ell_2$-norm interpolator based on the total data of size $\tilde n_1 +\tilde n_2$ is minimized. 
To this end, the purpose of this section is two-fold: first to provide an estimation strategy for SNR and SSR, and second to invoke the estimator and provide a guideline for choosing appropriate sample size. 

For SNR and SSR 
estimation, we employ a simplified version of the recently introduced HEDE method \cite{song2024unpublished}. Let $\hat \bbeta_d^{(k)}$ denote the following debiased Lasso estimator \cite{javanmard2014hypothesis, bellec2019biasing}:
$$\hat{\bbeta}_d^{(k)} = \hat{\bbeta}_{\rm{L}}^{(k)} + \frac{\bX^{(k)\top} (\by^{(k)} - \bX^{(k)} \hat{\bbeta}_{\rm{L}}^{(k)})}{n_k - \|\hat{\bbeta}_{\rm{L}}^{(k)}\|_0}$$
that is based on the Lasso estimator for some fixed $\lambda_{\rm{L}}$:
$$\hat{\bbeta}_{\rm{L}}^{(k)} := \argmin_{\bb} \frac{1}{2n} \|\by^{(k)} - \bX^{(k)} \bb\|_2^2 + \frac{\lambda_{\textrm{L}}}{\sqrt{n_k}} \|\bb\|_1,$$
and let
$$\hat\tau_k^2 =\frac{ \|\by^{(k)} - \bX^{(k)} \hat{\bbeta}_{\textrm{L}}^{(k)}\|_2^2}{(n_k - \|\hat{\bbeta}_{\textrm{L}}^{(k)}\|_0)^2}$$
be the corresponding estimated variance for $k=1,2$. Then by Theorem 4.1 and Proposition 4.2 in \cite{song2024unpublished}, we know
\begin{align*}
    &\|\hat \bbeta_d^{(k)}\|_2^2 - p \hat \tau^2_k \rightarrow \|\bbeta^{(k)}\|_2^2,\ \ k=1,2, \\
    &\widehat{\Var}(\by^{(k)}) -( \|\hat \bbeta_d^{(k)}\|_2^2 - p \hat \tau^2_k) \rightarrow \sigma^2,\ \ k=1,2, \\
    &\|\hat \bbeta^{(1)}_d- \hat \bbeta^{(2)}_d\|_2^2 -p (\hat \tau^2_1+ \hat \tau^2_2) \rightarrow  \|\bbeta^{(1)}-\bbeta^{(2)}\|_2^2.
\end{align*}
where the last line invokes independence between the datasets.
Hence, we obtain the following consistent estimators for SNR and SSR:
\begin{align}\label{eq:snr_hat}
\widehat{\text{SNR}}:= \frac{\|\hat \bbeta_d^{(2)}\|_2^2 - p \hat \tau^2_2}{\widehat{\Var}(\by^{(2)}) -( \|\hat \bbeta_d^{(2)}\|_2^2 - p \hat \tau^2_2)} \rightarrow & \rm{SNR}, \nonumber\\
\widehat{\text{SSR}}:=\frac{\|\hat \bbeta^{(1)}_d- \hat \bbeta^{(2)}_d\|_2^2 -p (\hat \tau^2_1+ \hat \tau^2_2)}{\|\hat \bbeta_d^{(2)}\|_2^2 - p \hat \tau^2_2} \rightarrow & \rm{SSR}.
\end{align}

Given the estimators of SNR and SSR, denoted by $\widehat{\text{SNR}}$ and $\widehat{\text{SSR}}$ respectively, an important question arises: what is the optimal size of the target dataset for which the pooled min-$\ell_2$-norm interpolator achieves least mean-squared prediction error $R(\hat{\bbeta}; \bbeta^{(2)})$? This is answered by the following proposition:

\begin{cor}\label{prop:n2size}
Assume the conditions of Theorem \ref{thm:model_shift} holds. Denote by $\hat \bbeta_{t}$ the pooled min-$\ell_2$-norm interpolator based on $n_1$ samples from the source and $n_2$ samples from the target  where $n_2/p=t$. Then we have
\begin{equation}\label{eq:n2_size}
\argmin_{t: t>0, n_1+n_2 <p} R(\hat\bbeta_{t};\bbeta^{(2)}) = \left(p-n_1 - \sqrt{\frac{p^2}{\rm{SNR}}+ n_1 (p-n_1)\rm{SSR}}\right)_{+}, \quad n_2/p=t,
\end{equation}
    where $x_{+} =x$ if $x>0$ and $=0$ otherwise.
    \end{cor}
\begin{proof}
    Note that, by Theorem \ref{thm:model_shift}, we have $$R(\hat\bbeta;\bbeta^{(2)})= \sigma^2 \frac{n}{p-n} +\frac{p-n}{p}\|\bbeta^{(2)}\|_2^2+ \frac{n_1(p-n_1)}{p(p-n)}\|\tilde\bbeta\|_2^2 +o(1).$$ 
Plugging in $n_2=pt$ and differentiating yields the finishes the proof.
\end{proof}
    
This proposition validates our intuition for the problem: If SNR is low or SSR is large, large amount of target data is not necessary for achieving good generalization. In fact, given the estimators $\widehat{\rm{SNR}}$, $\widehat{\rm{SSR}}$, a data-dependent estimated optimal target data size is given by 
\begin{equation*}
    \hat n_{2,\text{opt}}=\left(p-n_1 - \sqrt{\frac{p^2}{\widehat{\rm{SNR}}}+ n_1(p-n_1) \widehat{\rm{SSR}}}\right)_{+}
\end{equation*}
Simulations in \cite{song2024unpublished} suggest that both $\widehat{\text{SNR}}$ and $\widehat{\text{SSR}}$ exhibit favorable finite-sample performance, which suggests favorable finite-sample performance of our data-dependent estimated optimal target data size $\hat n_{2, \text{opt}}$. We defer further analysis of the finite-sample properties of $\hat n_{2, \text{opt}}$ to future work.

\section{Extension to Multiple Source Datasets}\label{sec:multiple}
Our results for model shift, i.e., Theorem \ref{thm:model_shift}, can be extended to the general framework of multiple sources and one target datasets. Denote the target as $T$ and $K$ sources as $1,...,K$, and further denote the signal shift vectors as $\tilde\bbeta^{(k)} := \bbeta^{(k)} - \bbeta^{(T)}$. Suppose we are still in the overparametrized regime $p> n:=n_1+...+n_K+n_T$, and consider the pooled min-$\ell_2$-norm interpolator
\begin{equation}\label{eqn:multiple_joint_interpolator}
\hat\bbeta = \argmin \left\{\|\bb\|_2 \ \ \text{s.t.}\ \by^{(T)} = \bX^{(T)}\bb,\ \by^{(k)} = \bX^{(k)}\bb,\ k=1,...,K.\right\}
\end{equation}
The following assumption combines and generalizes Assumptions \ref{as:shared}, \ref{as:design_gaussian}, and \ref{as:shared_covariance} to this setting.
\begin{assumption}[Regularity conditions for multiple sources]\label{as:shared_multiple}
The following conditions all hold:
\begin{enumerate}
    \item[(a)] The random objects $\bZ^{(1)}$, $\bZ^{(2)}$, \dots, $\bZ^{(K)}$, $\bZ^{(T)}$, $\bep^{(1)}$, $\bep^{(2)}$, \dots, $\bep^{(K)}$, and $\bep^{(T)}$ are mutually independent.
        \item[(b)] The source and target covariance matrices $\bSigma^{(1)}$, \dots, $\bSigma^{(K)}$, and $\bSigma^{(T)}$ are identical. That is,
        \begin{equation*}
            \bSigma := \bSigma^{(1)} = \cdots = \bSigma^{(K)} = \bSigma^{(T)}.
        \end{equation*}
        \item[(c)] There exists a fixed constant $\tau \in (0,1)$ such that the eigenvalues $\lambda_1 \ge \cdots \ge \lambda_p$ of $\bSigma$ satisfy
        \begin{equation}
        \tau \leq \lambda_p \leq ... \leq \lambda_1 \leq \tau^{-1}.
        \end{equation}
        \item[(d)] For $i \in \{1,2, \dots, K\} \cup\{T\}$, the entries of $\bep^{(i)}$ are independent, have mean zero, and have variance $\sigma^2 < \infty$.
        \item[(e)] For $i \in \{1,2, \dots, K\} \cup\{T\}$, the entries of $\bZ^{(i)}$ are independent standard Gaussian random variables with mean zero and unit variance.
        \item[(f)] For $i \in \{1,2, \dots, K\} \cup\{T\}$, let $\gamma_i := p/n_i$, and let $\gamma := p / n$. The aspect ratios $\gamma_ := p/n_1$, $\gamma_2 := p/n_2$, and $\gamma := p/n$ satisfy
        \begin{equation}
        1 + \tau \leq \gamma_1,\gamma_2,\dots, \gamma_K, \gamma_T, \gamma \leq  \tau^{-1}.
        \end{equation}
    \end{enumerate}
\end{assumption}
Under these assumptions, the out-of-sample prediction risk of $\hat \bbeta$ on a new observation from the target distribution can similar be decomposed as
\begin{equation*}
    R(\hat \bbeta; \bbeta^{(T)}) = \underbrace{\|\E[\hat \bbeta \mid \bm X] - \bbeta^{(T)}\|_{\bSigma}^2}_{B(\hat \bbeta; \bbeta^{(T)})} + \underbrace{\Tr[\Cov(\hat \bbeta \mid X) \bSigma]}_{V(\hat \bbeta; \bbeta^{(T)})}.
\end{equation*}
Note that the pooled min-$\ell_2$-norm interpolator in \eqref{eqn:multiple_joint_interpolator} now takes the form
\begin{equation*}
    \hat \bbeta = \biggl(\bm X^{(T)\top} \bm X^{(T)} + \sum_{k=1}^K \bm X^{(k)\top} \bm X^{(k)} \biggr)^\dagger \biggl(\bm X^{(T)\top} \bm y^{(T)} + \sum_{k=1}^K \bm X^{(k)\top} \bm X^{(k)}\biggr).
\end{equation*}
and thus, we can further decompose
\begin{align*}
    V(\hat \bbeta; \bbeta^{(T)}) &= \frac{\sigma^2}{n} \Tr(\hat \bSigma^\dagger \bSigma), \\
    B(\hat \bbeta; \bbeta^{(T)}) &=  \underbrace{\bbeta^{(T)\top} (\hat \bSigma^\dagger \hat \bSigma - \bm I) \bSigma (\hat \bSigma^\dagger \hat \bSigma - \bm I) \bbeta^{(T)}}_{B_1(\hat \bbeta; \bbeta^{(T)})} \\
    &\qquad + \underbrace{\sum_{k=1}^{K} \tilde \bbeta^{(k)\top} \biggl(\frac{\bm X^{(k)\top} \bm X^{(k)}}{n}\biggr) \hat \bSigma^\dagger \bSigma \hat \bSigma^\dagger \biggl(\frac{\bm X^{(k)\top} \bm X^{(k)}}{n}\biggr)\tilde \bbeta^{(k)}}_{B_2(\hat \bbeta; \bbeta^{(T)})} \\
    &\qquad + \underbrace{\sum_{k=1}^{K} 2 \bbeta^{(T)\top} (\hat \bSigma^\dagger \hat \bSigma - \bm I) \bSigma \hat \bSigma^\dagger \biggl(\frac{\bm X^{(k)\top} \bm X^{(k)}}{n}\biggr) \tilde \bbeta^{(k)}}_{B_3(\hat \bbeta; \bbeta^{(T)})} \\
    &\qquad + \underbrace{\sum_{\substack{1 \leq k, \ell \leq K,\\ k \neq\ell}} \tilde \bbeta^{(k)\top} \biggl(\frac{\bm X^{(k)\top} \bm X^{(k)}}{n}\biggr) \hat \bSigma^\dagger \bSigma \hat \bSigma^\dagger \biggl(\frac{\bm X^{(\ell)\top} \bm X^{(\ell)}}{n}\biggr) \tilde \bbeta^{(\ell)}}_{B_4(\hat \bbeta; \bbeta^{(T)})}
\end{align*}
where $\hat \bSigma  = \bm X^\top \bm X / n $ is again the uncentered sample covariance matrix obtained by appending all source and target samples together and where $\tilde \bbeta^{(k)} = \bbeta^{(k)} - \bbeta^{(T)}$ for $k = 1, \dots, K$.

Like Theorem \ref{thm:model_shift}, the asymptotic risk of the pooled min-$\ell_2$-norm interpolator is characterized with the following signed measures that capture the geometry between $\bSigma$, $\bbeta^{(T)}$, and $\tilde \bbeta^{(k)}$ (for $k = 1,\dots, K$):
\begin{gather*}
    \hat G_n^{\bbeta^{(T)}}(s) = \frac{1}{\|\bbeta^{(T)}\|_2^2} \sum_{i=1}^p \langle \bbeta^{(T)}, \bm v_i\rangle^2 \bm 1_{\{s \geq s_i\}}, \quad \hat G_n^{\tilde \bbeta^{(k)}}(s) = \frac{1}{\|\tilde \bbeta^{(k)}\|_2^2} \sum_{i=1}^p \langle \tilde \bbeta^{(k)}, \bm v_i\rangle^2 \bm 1_{\{s \geq s_i\}} \\
    \hat G_n^{(b,k)}(s) = \frac{1}{\|\bbeta^{(T)}\|_2 \|\tilde \bbeta^{(k)}\|_2}\sum_{i=1}^p \langle \bbeta^{(T)}, \bm v_i\rangle  \langle \tilde \bbeta^{(k)}, \bm v_i\rangle  \bm 1_{\{s \geq s_i\}},  \\ 
    \hat G_n^{(k,\ell)}(s) =\frac{1}{\|\bbeta^{(k)}\|_2 \|\tilde \bbeta^{(\ell)}\|_2}\sum_{i=1}^p \langle \bbeta^{(k)}, \bm v_i\rangle  \langle \tilde \bbeta^{(\ell)}, \bm v_i\rangle  \bm 1_{\{s \geq s_i\}} \\
    \hat H_n(s) = \frac{1}{p} \sum_{i=1}^p \bm 1_{\{s \geq s_i\}}\\
\end{gather*}

Then, the following result characterizes the bias and variance of the pooled min-$\ell_2$-norm interpolator under this setting: 
\begin{thm}\label{thm:multiple_joint_interpolator}
    Let Assumptions \ref{as:shared}, \ref{as:design_gaussian}, and \ref{as:shared_covariance} hold (adjusted suitably for multiple source dataset $k=1,...,K$ and a target dataset $T$). Define
    \begin{align}
        {\mathcal B}_4(\hat H_n, \hat G_n^{(k,\ell)}, \gamma) &= \frac{n_\ell}{n-n_k} \cdot \frac{n_k}{n} \|\tilde \bbeta^{(\ell)}\|_2 \|\tilde \bbeta^{(k)}\|_2 \int \frac{\gamma s(c_1 + \gamma c_0^2 s^2)}{(1 + c_0 \gamma s)^2} d\hat G_n^{(k,\ell)}(s) \\
        &\qquad - \frac{n_\ell}{n-n_k} \mathcal B_2(\hat H_n, \hat G_n^{(k,\ell)}, \gamma).
    \end{align}
    Then, with high probability over the randomness of $(\bZ^{(1)},...,\bZ^{(K)},\bZ^{(T)})$, we have
    \begin{align}
        V(\hat\bbeta;\bbeta^{(T)}) &= \mathcal V(\hat H_n, \gamma) + O(p^{-1/12}),\\
        B_1(\hat\bbeta;\bbeta^{(T)}) &= \mathcal B_1(\hat H_n, \hat G_n^{\bbeta^{(T)}}, \gamma) + O(p^{-1/12} \|\bbeta^{(K)}\|_2^2),
    \\
        B_2(\hat\bbeta;\bbeta^{(T)}) &= \sum_{k=1}^{K} \mathcal B_2(\hat H_n, \hat G_n^{\tilde \bbeta^{(k)}}, \gamma)  + O( p^{-1/12}\|\tilde\bbeta^{(k)}\|_2^2) \\
        B_3(\hat \bbeta; \bbeta^{(T)}) &= \sum_{k=1}^{K} \mathcal B_3(\hat H_n, \hat G_n^{(b,k)}, \gamma) + O(p^{-1/12} \|\tilde \bbeta^{(k)}\|_2 \|\bbeta^{(T)}\|_2) \\
        &\qquad + \sum_{\substack{1 \leq k,\ell \leq K \\ k \neq \ell}} {\mathcal B}_4(\hat H_n, \hat G_n^{(k,\ell)}, \gamma) + O(p^{-1/12} \|\tilde \bbeta^{(k)}\|_2\|\tilde \bbeta^{(j)}\|_2).
    \end{align}
\end{thm}
Note that the results for $B_1(\hat \bbeta; \bbeta^{(T)})$, $B_2(\hat \bbeta; \bbeta^{(T)})$, $B_3(\hat \bbeta; \bbeta^{(T)})$, and $V(\hat \bbeta; \bbeta^{(T)})$ follow directly from Theorem \ref{thm:model_shift}. It thus suffices to analyze the remaining term $B_4(\hat \bbeta; \bbeta^{(T)})$, though the strategy and machinery we employ is identical to that of Theorem \ref{thm:model_shift}: we thus omit the full argument for brevity.

One can similarly verify that, in the isotropic and random effects settings, the asymptotic out-of-sample prediction risk depends only on the spectrum of the covariance matrix and the norms of the target $\bbeta^{(T)}$ and the shifts $\tilde \bbeta^{(1)}, \dots, \tilde \bbeta^{(K)}$, again emphasizing the relationship between the transfer geometry and covariate geometry in determining the out-of-sample prediction risk.

\section{Ridge Regression}\label{sec:ridge}
As mentioned in Section \ref{sec:proof_outline}, the risk of the pooled min-$\ell_2$-norm interpolator is closely related to that of the pooled ridge estimator in \eqref{eqn:ridge}, with a fixed penalty $\lambda >0$. In this section, we obtain exact generalization error of $\hat\bbeta_\lambda$ under the assumptions of Section \ref{sec:model_shift} and Section \ref{sec:covariate_shift}. These results are crucial intermediate results in our proofs of evaluating interpolator risk.

Similar to Lemma \ref{lemma:risk_analytical}, we first state the bias-variance decomposition of the prediction risk:

\begin{lemma}\label{lemma:risk_analytical_ridge}
Under Assumption \ref{as:shared}, the ridge estimator \eqref{eqn:ridge} has variance
\begin{equation}\label{eqn:variance_analytical_ridge}
V(\hat\bbeta_\lambda;\bbeta^{(2)}) = \frac{\sigma^2}{n}\Tr((\hat\bSigma+\lambda\bI)^{-2}\hat\bSigma\bSigma^{(2)})
\end{equation}
and bias
\begin{equation}\label{eqn:bias_analytical_ridge}
\begin{aligned}
B(\hat\bbeta_\lambda;\bbeta^{(2)})=&\underbrace{\lambda^2\bbeta^{(2)^\top} (\hat\bSigma+\lambda\bI)^{-1}\bSigma^{(2)} (\hat\bSigma+\lambda\bI)^{-1} \bbeta^{(2)}}_{B_1(\hat\bbeta_\lambda;\bbeta^{(2)})} \\
&+ \underbrace{\tilde\bbeta^\top (\frac{\bX^{(1)\top} \bX^{(1)}}{n})(\hat\bSigma+\lambda\bI)^{-1} \bSigma^{(2)}(\hat\bSigma+\lambda\bI)^{-1} (\frac{\bX^{(1)\top} \bX^{(1)}}{n})\tilde\bbeta}_{B_2(\hat\bbeta_\lambda;\bbeta^{(2)})}\\
&\underbrace{-2\lambda\bbeta^{(2)^\top} (\hat\bSigma+\lambda\bI)^{-1}\bSigma^{(2)} (\hat\bSigma+\lambda\bI)^{-1} (\frac{\bX^{(1)\top} \bX^{(1)}}{n})\tilde\bbeta}_{B_3(\hat\bbeta_\lambda;\bbeta^{(2)})}.
\end{aligned}
\end{equation}
\end{lemma}
In the subsections below, we present exact risk formulae under model shift and covariate shift, respectively.

\subsection{Model Shift}
We first define deterministic quantities that only depend on $\gamma_1$, $\gamma_2$, $\gamma$, $\lambda$, and $\bSigma$, then present our main results stating that these deterministic quantities are the limits of the bias and variance components of prediction risk.

\begin{definition}[Risk limit, ridge]\label{def:model_shift_ridge}
Recall the definitions of $\hat H_n$, $\hat G_n^{\tilde \bbeta}$, $\hat G_n^{\bbeta^{(2)}}$, and $\hat G_n^{(b)}$ from Definition \ref{def:cov_geometry}.
Define $m_n(-\lambda) = m_n(-\lambda; \hat H_n, \gamma)$ as the unique solution of
\begin{equation}
    m_n(-\lambda) = \int\frac{d\hat H_n(s)}{s(1 - \gamma + \gamma \lambda m_n(-\lambda)) + \lambda} 
\end{equation}
and then define $m_{n,1}(-\lambda) = m_{n,1}(-\lambda; \hat H_n, \gamma)$
\begin{equation}
    m_{n,1}(-\lambda) = \frac{\int \frac{s^2 (1 - \gamma + \gamma \lambda m_n(-\lambda))}{(s (1 - \gamma + \gamma \lambda m_n(-\lambda)) + \lambda)^2} d\hat H_n(s)}{1 + \gamma \lambda \int \frac{s}{(s(1 - \gamma + \gamma \lambda m_n(-\lambda)) + \lambda)^2} d \hat H_n(s)}.
\end{equation}
Next, define
\begin{align}
    \mathcal V(\lambda; \hat H_n, \gamma) &:= \sigma^2 \gamma \int \frac{s^2 (1-\gamma+\gamma \lambda^2 m_n'(-\lambda))}{(\lambda + s(1 - \gamma + \gamma \lambda m_n(-\lambda)))^2} d \hat H_n(s) \\
    \mathcal B_1(\lambda; \hat H_n, \hat G_n^{\bbeta^{(2)}}, \gamma) &:=  \|\bbeta^{(2)}\|^2  \int \frac{\lambda^2s(1 + \gamma m_{n,1}(-\lambda))}{(\lambda + (1 - \gamma + \gamma \lambda m_n(-\lambda))s)^2} d\hat G_n^{\bbeta^{(2)}}(s) \\
    \mathcal B_2(\lambda; \hat H_n, \hat G_n^{\tilde \bbeta}, \gamma) &:= \frac{n_1(n_1 - 1)}{n(n-1)} \|\tilde \bbeta\|_2^2 \\
    &\qquad \qquad \cdot\int \frac{s((1 -\gamma + \gamma \lambda m_n(-\lambda))^2 s^2 + \lambda^2 \gamma m_{n,1}(-\lambda))}{(\lambda + (1 - \gamma + \gamma \lambda m_n(-\lambda))s)^2} d\hat G^{\tilde \bbeta}_n(s)\\
     &\qquad + \biggl(\frac{n_1}{n} - \frac{n_1(n_1 - 1)}{n(n-1)}\biggr) \|\tilde \bbeta\|_2^2  \int s\; d\hat G^{\tilde \bbeta}_n(s) \\
     &\qquad \qquad \cdot \frac{\lambda^2 \gamma ( 1 + \gamma m_{n,1}(\lambda)) \int \frac{s^2}{(\lambda + (1 - \gamma + \gamma \lambda m_n(-\lambda))s)^2} d\hat H_n(s)}{(\lambda + \gamma \int \frac{\lambda s}{\lambda + (1 - \gamma + \gamma \lambda m_n(-\lambda)) s} d\hat H_n(s))^2}. \\
    \mathcal B_3(\lambda; \hat H_n, \gamma) &:= \frac{n_1}{n} \|\tilde \bbeta\|_2 \|\bbeta^{(2)}\|_2 \biggl(\int \frac{\lambda s}{\lambda + (1 - \gamma + \gamma \lambda m_n(-\lambda))s} \hat G_n^{(b)}(s) \\
    &\hspace{9em} + \int \frac{\lambda^2 (1 + \gamma m_{n,1}(-\lambda))s}{(\lambda + (1 - \gamma + \gamma \lambda m_n(-\lambda))s)^2} \hat G_n^{(b)}(s)\biggr).
\end{align}
\end{definition}

\begin{thm}[Risk limit under model shift, ridge]\label{thm:model_shift_ridge}
Suppose Assumptions \ref{as:shared}, \ref{as:design_gaussian}, and \ref{as:shared_covariance} hold. Let $s := \lambda(1+\alpha) > p^{-2/3+\epsilon}$, where $\epsilon>0$ is a small constant. Then, for any small constant $c>0$, with high probability over the randomness of $(\bZ^{(1)},\bZ^{(2)})$, we have
\begin{align}
    V(\hat\bbeta_\lambda;\bbeta^{(2)}) &= \mcl{V}(\lambda;\gamma) + O(\sigma^2 \lambda^{-2}p^{-1/2+c}),
\label{eqn:model_shift_ridge_V}\\
    B_1(\hat\bbeta_\lambda;\bbeta^{(2)}) &= \mcl{B}_1(\lambda;\gamma) + O(\lambda^{-1} p^{-1/2+c}\|\bbeta^{(2)}\|_2^2),
\label{eqn:model_shift_ridge_B1}\\
    B_2(\hat\bbeta_\lambda;\bbeta^{(2)}) &= \mcl{B}_2(\lambda;\gamma_1,\gamma_2)  + O(\lambda^{-4} p^{-1/2+c}\|\tilde\bbeta\|_2^2),\label{eqn:model_shift_ridge_B2}\\
    B_3(\hat\bbeta_\lambda;\bbeta^{(2)}) &= \mcl{B}_3(\lambda;\gamma_1,\gamma_2)  + O(\lambda^{-3} p^{-1/2+c}\|\tilde\bbeta\|_2 \|\bbeta^{(2)}\|_2).\label{eqn:model_shift_ridge_B3}
\end{align}
\end{thm}

\noindent In Theorem \ref{thm:model_shift_ridge}, both \eqref{eqn:model_shift_ridge_V} and \eqref{eqn:model_shift_ridge_B1} follow directly from \cite[Theorem 5]{hastie2022surprises}. We defer the proof of \eqref{eqn:model_shift_ridge_B2} and \eqref{eqn:model_shift_ridge_B3} to Appendix \ref{subsec:proof:model_shift_ridge}. 

\subsection{Covariate Shift}
\begin{thm}\label{thm:design_shift_ridge}
Suppose Assumptions \ref{as:shared}, \ref{as:design_four_epsilon}, and \ref{as:simul_diag} hold. Let $\lambda > p^{-1/7+\epsilon}$ where $\epsilon>0$ is a small constant. Then, for any small constant $c>0$, with high probability over the randomness of $(\bZ^{(1)},\bZ^{(2)})$, we have
\begin{align}
V(\hat\bbeta_\lambda;\bbeta^{(2)}) &= \sigma^2\gamma\int \frac{\lambda^{(1)}\lambda^{(2)}(a_1-a_3\lambda)+(\lambda^{(2)})^2(a_2-a_4\lambda)}{(a_1\lambda^{(1)}+a_2\lambda^{(2)}+\lambda)^2} d \hat H_p(\lambda^{(1)},\lambda^{(2)})\label{eqn:design_shift_ridge_V} \\
&\qquad + O(\sigma^2\lambda^{-9/2} p^{-1/2+c}),\\
B(\hat\bbeta_\lambda;\bbeta^{(2)}) &= \|\bbeta^{(2)}\|_2^2 \cdot \int \frac{b_3\lambda\lambda^{(1)}+(b_4+\lambda)\lambda\lambda^{(2)}}{(b_1\lambda^{(1)}+b_2\lambda^{(2)}+\lambda)^2} d \hat G_p(\lambda^{(1)},\lambda^{(2)}) \label{eqn:design_shift_ridge_B} \\
&\qquad + O(\lambda^{-3/2}p^{-1/4+c}\|\bbeta^{(2)}\|_2^2),
\end{align}
where $(a_1,a_2,a_3,a_4)$ is the unique solution, with $a_1,a_2$ positive, to the following system of equations:
\begin{equation}\label{eqn:design_shift_ridge_alpha_eqns}
\begin{aligned}
a_1+a_2 &= 1 - \gamma \int \frac{a_1\lambda^{(1)}+a_2\lambda^{(2)}}{a_1\lambda^{(1)}+a_2\lambda^{(2)}+\lambda} d\hat H_p(\lambda^{(1)},\lambda^{(2)}),\\
a_1 &= \frac{n_1}{n} - \gamma \int \frac{a_1\lambda^{(1)}}{a_1\lambda^{(1)}+a_2\lambda^{(2)}+\lambda} d\hat H_p(\lambda^{(1)},\lambda^{(2)}),\\
a_3+a_4 &= -\gamma\int \frac{\lambda^{(1)}(a_3\lambda-a_1)+\lambda^{(2)}(a_4\lambda-a_2)}{(a_1\lambda^{(1)}+a_2\lambda^{(2)}+\lambda)^2}d\hat H_p(\lambda^{(1)},\lambda^{(2)}),\\
a_3&= -\gamma\int \frac{\lambda^{(1)}(a_3\lambda-a_1)+\lambda^{(1)}\lambda^{(2)}(a_3a_2-a_4a_1)}{(a_1\lambda^{(1)}+a_2\lambda^{(2)}+\lambda)^2}d\hat H_p(\lambda^{(1)},\lambda^{(2)}),
\end{aligned}
\end{equation}
and $(b_1,b_2,b_3,b_4)$ is the unique solution, with $b_1,b_2$ positive, to the following system of equations:
\begin{equation}\label{eqn:design_shift_ridge_beta_eqns}
\begin{aligned}
b_1+b_2 &= 1 - \gamma \int \frac{b_1\lambda^{(1)}+b_2\lambda^{(2)}}{b_1\lambda^{(1)}+b_2\lambda^{(2)}+\lambda} d\hat H_p(\lambda^{(1)},\lambda^{(2)}),\\
b_1 &= \frac{n_1}{n} - \gamma \int \frac{b_1\lambda^{(1)}}{b_1\lambda^{(1)}+b_2\lambda^{(2)}+\lambda} d\hat H_p(\lambda^{(1)},\lambda^{(2)}),\\
b_3+b_4 &= -\gamma\int \frac{\lambda^{(1)}\lambda(b_3-b_1\lambda^{(2)})+\lambda^{(2)}\lambda(b_4-b_2\lambda^{(2)})}{(b_1\lambda^{(1)}+b_2\lambda^{(2)} + \lambda)^2}d\hat H_p(\lambda^{(1)},\lambda^{(2)}),\\
b_3 &= -\gamma\int \frac{\lambda^{(1)}\lambda(b_3-b_1\lambda^{(2)})+\lambda^{(1)}\lambda^{(2)}(b_3b_2-b_4b_1)}{(b_1\lambda^{(1)}+b_2\lambda^{(2)} + \lambda)^2}d\hat H_p(\lambda^{(1)},\lambda^{(2)}).
\end{aligned}    
\end{equation}
\end{thm}
Again, notice that $(a_1,a_2)=(b_1,b_2)$ above.

\section{Basic Tools}
\noindent We first collect the notations that we will use throughout the proof. Since $p,n_1,n_2$ has comparable orders, we use $p$ as the fundamental large parameter. All constants only depend on parameters introduced in Assumptions \ref{as:shared}, \ref{as:design_all_moments}, and \ref{as:design_four_epsilon}. For any matrix $\bX$, $\lambda_{\min}(\bA)$ denotes its smallest eigenvalue, $\lambda_{\max}(\bX) := \|\bX\|_{\op}$ denotes its largest eigenvalue (or equivalently the operator norm), $\lambda_i(\bX)$ denotes its $i$-th largest eigenvalue, and $\bX^\dagger$ denotes its pseudoinverse. We say an event $\mathcal{E}$ happens with high probability (w.h.p.) if $\PP(\mathcal{E})\rightarrow 1$ as $p \rightarrow \infty$. We write $f(p) = O(g(p))$ if there exists a constant $C$ such that $|f(p)| \leq Cg(p)$ for large enough $p$. For $f(p),g(p) \geq 0$, we also write $f(p) \lesssim g(p)$ and $g(p) = \Omega(f(p))$ if $f(p) = O(g(p))$, and $f(p) =\Theta( g(p))$ if $f(p) = O(g(p))$ and $g(p) = O(f(p))$. Whenever we talk about constants, we refer to quantities that does not depend on $\lambda,n_1,n_2,p,$ or random quantities such as $\bX^{(k)},\by^{(k)},\bep^{(k)},k=1,2$. They might depend on other constants such as $\tau$ in Assumption \ref{as:shared}, $\tau_m$ (for $m \ge 1)$ in Assumption \ref{as:design_all_moments}, and $\tau_\varphi$ in Assumption \ref{as:design_four_epsilon}. We often use $c,C$ to denote constants whose specific value might change from case to case. We use check symbols (e.g. $\check x,\check y$) to denote temporary notations (whose definitions might also change) for derivation simplicity. We further make the following definitions:

\begin{definition}[Overwhelming Probability]
We say an event $\mathcal{E}$ holds with overwhelming probability (w.o.p.) if for any constant $D>0$, $\PP(\mathcal{E}) \geq 1 - p^D$ for large enough $p$. Moreover, we say $\mathcal{E}$ holds with overwhelming probability in another event $\Omega$ if for any constant $D>0$, $\PP(\Omega \setminus \mathcal{E}) \geq 1 - p^D$ for large enough $p$.
\end{definition}

\begin{definition}[Stochastic domination, first introduced in \cite{erdHos2013averaging}]
Let $\xi^{(p)}$ and $\zeta^{(p)}$ be two $p$-dependent random variables. We say that $\xi$ is stochastically dominate by $\zeta$, denoted by $\xi \prec \zeta$ or $\xi = O_\prec (\zeta)$, if for any small constant $c > 0$ and large constant $D>0$, 
$$\PP(|\xi| > p^c|\zeta|) \leq p^{-D}$$
for large enough $p$. That is, $|\xi|\leq p^c|\zeta|$ w.o.p. for any $c>0$. If $\xi(u)$ and $\zeta(u)$ are functions of $u$ supported in $\mathcal{U}$, then we say $\xi(u)$ is stochastically dominated by $\zeta(u)$ uniformly in $\mathcal{U}$ if 
$$\PP(\bigcup_{u \in \mathcal{U}}\{|\xi(u)|>p^c|\zeta(u)|\})\leq p^{-D}.$$
\end{definition}

Note that since $(\log p)^C \prec 1$ for any constant $C>0$ and we are including $p^c$, we can safely ignore log terms. Also, for a random variable $\xi$ with finite moments up to any order, we have $|\xi|\prec 1$, since
$$\PP(|\xi|\geq p^c) \leq p^{-kc}\E|\xi|^k \leq p^{-D}$$ 
by Markov's inequality, when $k > D/c$.

The following lemma collect several algebraic properties of stochastic domination:

\begin{lemma}[Lemma 3.2 of \cite{alex2014isotropic}]\label{lemma:domination_algebra}
Let $\xi$ and $\zeta$ be two families of nonnegative random variables depending on some parameters $u \in \mathcal{U}$ and $v \in \mathcal{V}$. Let $C>0$ be an arbitrary constant.
\begin{enumerate}
    \item[(i)] \textbf{Sum.} Suppose $\xi(u,v) \prec \zeta(u,v)$ uniformly in $u \in \mathcal{U}$ and $v \in \mathcal{V}$. If $|\mathcal{V}| \leq p^C$, then $\sum_{v \in \mathcal{V}}\xi(u,v) \prec \sum_{v \in \mathcal{V}}\zeta(u,v)$ uniformly in $u$.
    \item[(ii)] \textbf{Product.} If $\xi_1(u) \prec \zeta_1(u)$ and $\xi_2(u) \prec \zeta_2(u)$, then $\xi_1(u)\xi_2(u) \prec \zeta_1(u)\zeta_2(u)$ uniformly in $u \in \mcl{U}$.
    \item[(iii)] \textbf{Expectation.} Suppose that $\Psi(u) \geq p^{-C}$ is a family of deterministic parameters, and $\xi(u)$ satisfies $\E\xi(u)^2 \leq p^{C}$. If $\xi(u) \prec \Psi(u)$ uniformly in $u$, then we also have $\E\xi(u) \prec \Psi(u)$ uniformly in $u$.
\end{enumerate}
\end{lemma}

\begin{definition}[Bounded Support]\label{def:bounded_support}
Let $Q>0$ be a ($p$-dependent) deterministic parameter. We say a random matrix $\bZ \in \R^{n \times p}$ satisfies the bounded support condition with $Q$ (or $\bZ$ has bounded support $Q$) if
\begin{equation}\label{eqn:bounded_support}
\max_{1\leq i \leq n,1\leq j \leq p} |Z_{ij}|\prec Q.
\end{equation}
\end{definition}

As showed above, if the entries of $\bZ$ have finite moments up to any order, then $\bZ$ has bounded support $Q=1$. More generally, if every entry of $\bZ$ has a finite $\varphi$-th moment as in  Assumption \ref{as:design_four_epsilon} and $n \leq p$, then using Markov's inequality and a union bound we have
\begin{equation}\label{eqn:logp_bounded_support}
\begin{aligned}
\PP\left(\max_{1 \leq i \leq n,1 \leq j \leq p}|Z_{ij}|\geq (\log p) p^{2/\varphi}\right) &\leq \sum_{i=1}^n \sum_{j=1}^p \PP\left(|Z_{ij}|\geq (\log p) p^{2/\varphi}\right)\\
&\lesssim \sum_{i=1}^n \sum_{j=1}^p \left[(\log p) p^{2/\varphi)}\right]^{-\varphi} \leq (\log p)^{-\varphi}.
\end{aligned}
\end{equation}
In other words, $\bZ$ has bounded support $Q=p^{2/\varphi}$ with high probability.

The following lemma provides concentration bounds for linear and quadratic forms of random variables with bounded support.

\begin{lemma}[Lemma 3.8 of \cite{erdHos2013spectral} and Theorem B.1 of \cite{erdHos2013delocalization}]\label{lemma:linear_quadratic_forms}
Let $(x_i), (y_i)$ be families of centered independent random variables, and $(a_i),(B_{ij})$ be families of deterministic complex numbers. Suppose the entries $x_i$ and $y_j$ have variance at most $1$, and satisfy the bounded support condition \eqref{eqn:bounded_support} for a deterministic parameter $Q \geq 1$. Then we have the following estimates:
\begin{align}
& \left|\sum_{i=1}^na_ix_i\right| \prec Q \max_{1\leq i\leq n}|a_i| + \left(\sum_{i=1}^n|a_i|^2\right)^{1/2},\label{eqn:linear_form}\\
& \left|\sum_{i,j=1}^n x_iB_{ij}y_j\right| \prec Q^2 B_d + Qn^{1/2}B_o + \left(\sum_{1\leq i,j\leq n} |B_{ij}|^2\right)^{1/2},\label{eqn:bilinear_form}\\
& \left|\sum_{i=1}^n (|x_i|^2 - \E|x_i|^2)B_{ii}\right| \prec Qn^{1/2}B_d, \label{eqn:quadratic_form_diag}\\
& \left |\sum_{1\leq i\neq j\leq n} x_iB_{ij}x_j\right| \prec Qn^{1/2}B_o + \left(\sum_{1\leq i\neq j\leq n}|B_{ij}|^2\right)^{1/2},\label{eqn:quadratic_form_offdiag}
\end{align}
where we denote $B_d:=\max_i|B_{ii}|$ and $B_o:=\max_{i\neq j}|B_{ij}|$. Moreover, if $x_i$ and $y_j$ have finite moments up to any order, then we have the following stronger estimates:
\begin{align}
& \left|\sum_{i=1}^na_ix_i\right| \prec \left(\sum_{i=1}^n|a_i|^2\right)^{1/2},\label{eqn:linear_form_strong}\\
& \left|\sum_{i,j=1}^n x_iB_{ij}y_j\right| \prec \left(\sum_{1\leq i,j\leq n} |B_{ij}|^2\right)^{1/2},\label{eqn:bilinear_form_strong}\\
& \left|\sum_{i=1}^n (|x_i|^2 - \E|x_i|^2)B_{ii}\right| \prec \left(\sum_{i=1}^n|B_{ii}|^2\right)^{1/2}, \label{eqn:quadratic_form_diag_strong}\\
& \left |\sum_{1\leq i\neq j\leq n} x_iB_{ij}x_j\right| \prec \left(\sum_{1\leq i\neq j\leq n}|B_{ij}|^2\right)^{1/2}.\label{eqn:quadratic_form_offdiag_strong}
\end{align}
\end{lemma}

The following lemma deals with the empirical spectral distribution of $\bZ^\top\bZ$ for $n < p$. Since $\bZ^\top\bZ$ is rank-deficient, its empirical spectral distribution has a peak at $0$. However, its nonzero eigenvalues are the same as the eigenvalues of $\bZ\bZ^\top$. Therefore, previous results for $\bZ\bZ^\top$ directly controls the positive empirical spectral distribution of $\bZ^\top\bZ$.

\begin{lemma}\label{lemma:ESD_bounded}
Suppose $1+\tau \leq p/n \leq \tau^{-1}$, and $\bZ\in \R^{n \times p}$ is a random matrix satisfying the same assumptions as $\bZ^{(2)}$ in Assumption \ref{as:design_four_epsilon} as well as the bounded support condition \eqref{eqn:bounded_support} for a deterministic parameter $Q$ such that $1 \leq Q \leq p^{1/2-c_Q}$ for a constant $c_Q > 0$. Then, we have that
\begin{equation}\label{eqn:ESD_bounded}
(\sqrt{p}-\sqrt{n})^2 - O_\prec(\sqrt{p}\cdot Q) \leq \lambda_n(\bZ^\top\bZ) \leq \lambda_1(\bZ^\top\bZ) \leq (\sqrt{p} + \sqrt{n})^2 + O_\prec(\sqrt{p} \cdot Q)
\end{equation}
\end{lemma}

\begin{proof}
First observe that $\lambda_1(\bZ^\top\bZ) = \lambda_1(\bZ\bZ^\top)$, and $\lambda_n(\bZ^\top\bZ) = \lambda_n(\bZ\bZ^\top)$. With this fact in mind, the case when $Q$ is of order $1$ follows from \cite[Theorem 2.10]{alex2014isotropic} and the case for general $Q$ follows from \cite[Lemma 3.11]{ding2018necessary}.
\end{proof}

Using a standard cut-off argument, the above results can be extended to random matrices with the weaker bounded moment assumptions:

\begin{cor}\label{cor:ESD_bounded_weaker}
Suppose $1+\tau \leq p/n \leq \tau^{-1}$, and $\bZ\in \R^{n \times p}$ is a random matrix satisfying the same assumptions as $\bZ^{(2)}$ in Assumption \ref{as:design_four_epsilon}, then \eqref{eqn:ESD_bounded} holds on a high probability event with $Q=p^{2/\varphi}$, where $\varphi$ is the constant in  Assumption \ref{as:design_four_epsilon}.
\end{cor}

\begin{proof}
The proof follows verbatim \cite[Corollary 23]{yang2020analysis} upon observing the fact that $\lambda_1(\bZ^\top\bZ) = \lambda_1(\bZ\bZ^\top)$, and $\lambda_n(\bZ^\top\bZ) = \lambda_n(\bZ\bZ^\top)$.
\end{proof}

We further cite two additional corollaries:

\begin{cor}[Corollary 24 of \cite{yang2020analysis}]\label{cor:linear_form_weaker}
Suppose $\bZ\in \R^{n \times p}$ is a random matrix satisfying the bounded moment condition in Assumption \ref{as:design_four_epsilon}. Then, for $Q=n^{2/\varphi}$, there exists a high probability event on which the following estimate holds for any deterministic vector $\bv \in \R^p$:
\begin{equation}\label{eqn:linear_form_weaker}
\left|\|\bZ \bv\|_2^2 - n\|\bv\|_2^2\right| \prec n^{1/2}Q\|\bv\|_2^2
\end{equation}
\end{cor}

We now provide two more results for bounding even moments of the operator norm of random matrices. Our first result holds under Assumption \ref{as:design_all_moments}, whereas our second result requires the stricter Assumption \ref{as:design_gaussian}.
\begin{lemma}\label{lm:matrix_op_bound}
    Let $\bm X \in \R^{n \times p}$ consist of i.i.d. centered random entries with moments of all orders. Then, for an integer $m \geq 1$,
    \begin{equation*}
        \E[\|\bm X\|_{\text{op}}^{2m}] \lcon (v \log v)^m
    \end{equation*}
    where $\lcon$ means that there exists some absolute constant $C_m$ (dependent on $m$) such that the inequality holds.
\end{lemma}
\begin{proof}
    Let $L > 0$ be some truncation level. Then, consider the truncated matrix
    \begin{equation*}
        \bm Y = (x_{ij} \bm 1[|x_{ij} \leq L|] - \E[x_{ij} \bm 1[|x_{ij} \leq L]])_{ij}
    \end{equation*}
    whose entries are bounded by $2L$ and also have moments of all orders. Then, we have
    \begin{equation*}
        \E[\|\bm X\|_{\text{op}}^{2m}] \leq \E[(\|\bm Y\|_{\text{op}} + \|\bm X - \bm Y\|_{\text{op}})^{2m}] \lcon \E[\|\bm Y\|_{\text{op}}^{2m}] + \E[\|\bm X - \bm Y\|_{\text{op}}^{2m}] 
    \end{equation*}
    where
    \begin{equation*}
        \bm X - \bm Y = (x_{ij} \bm 1[|x_{ij}| \geq L] - \E[x_{ij} \bm 1[|x_{ij}| \geq L]])_{ij}.
    \end{equation*}
    We bound each of these terms.

    We handle the first term via careful analysis using the matrix Bernstein inequality.
    Note that we can express
    \begin{equation*}
        \bm Y = \sum_{i,j} Y_{ij} e_i e_j^\top
    \end{equation*}
    where $e_i$ is the $i$th standard basis vector. Then, the matrices $Y_{ij} e_i e_j^\top$ are each independent random matrices with $\|Y_{ij} e_i e_j^\top\|_{\text{op}} \leq 2L$. Additionally, we can bound
    \begin{equation*}
        \|\E[\bm Y^\top \bm Y]\|_{\text{op}} = \biggl\| \sum_{i=1}^n \E[\bm Y_i \bm Y_i^\top] \biggr\|_{\text{op}} = \biggl\|\sum_{i=1}^n \E[Y_{11}^2] \bm I\biggr\|_{\text{op}} = \E[Y_{11}^2] n \leq \E[X_{11}^2] n
    \end{equation*}
    and similarly
    \begin{align*}
        \|\E[\bm Y \bm Y^\top]\|_{\text{op}} &= \biggl\| \sum_{j=1}^p \E[\bm Y_{\bullet j} \bm Y_{\bullet j}^\top] \biggr\|_{\text{op}} = \biggl\|\sum_{j=1}^p \E[Y_{11}^2] \bm I\biggr\|_{\text{op}} = \E[Y_{11}^2] p \leq \E[X_{11}^2] p
    \end{align*}
    Therefore, for $t \geq 0$, a matrix Bernstein bound \cite[Theorem 6.1.1]{tropp2015matrixconcentration} yields
    \begin{align*}
        \PP(\|\bm Y\|_{\text{op}} \geq t) &\leq \min\biggl(v \cdot \exp\biggl(\frac{-t^{2} / 2}{\E[X_{11}^2] (n + p) + 2L t / 3}\biggr), 1\biggr) \\
        &\leq \min\biggl(v \cdot \exp\biggl(-C \cdot \frac{t^2}{v + L t} \biggr), 1 \biggr)
    \end{align*}
    for some constant $C > 0$ that only depends on $\E[X_{11}^2]$, where we let $v = \max(n,p)$. 
    
    Letting
    $u^* = \frac{\sqrt{2 v \log v}}{\sqrt{C}}$, we can first bound
    \begin{align*}
        \E[\|\bm Y\|_{\text{op}}^{2m}] &= \int_0^\infty \PP(\|\bm Y\|_{\text{op}}^{2m} \geq t) dt \\
        &= \int_0^\infty \PP(\|\bm Y\|_{\text{op}} \geq t^{1/2m}) dt \\
        &= 2m\int_0^\infty u^{2m-1} \PP(\|\bm Y\|_{\text{op}} \geq u) dt \\
        &\leq 2m\int_0^\infty u^{2m-1} \min\biggl(v \cdot \exp\biggl(-C \cdot \frac{u^2}{v + L u} \biggr), 1 \biggr) du \\
        &=2 m\int_0^{u^*} u^{2m-1} du +2 m v \int_{u^*}^\infty u^{2m-1} \exp\biggl(-C \cdot \frac{u^2}{v + L u} \biggr) du \\
        &= (u^*)^{2m} + 2 m v \int_{u^*}^\infty u^{2m-1} \exp\biggl(-C \cdot \frac{u^2}{v + L u} \biggr) du
    \end{align*}
    Let $u^\dagger = v /L$, the point at which the behavior of the bound switches from subgaussian behavior to subexponential behavior. To handle this second term, we require that $\frac{u^\dagger}{L} = \frac{v}{L^2} \to \infty$:
    \begin{align*}
        &\int_{u^*}^\infty u^{2m-1} \exp\biggl(-C \cdot \frac{u^2}{v + Lu} \biggr) du \\
        &\leq \biggl|\int_{u^*}^{u^\dagger} u^{2m-1} \exp\biggl(-C \cdot \frac{u^2}{2v} \biggr) du \biggr| + \int_{u^\dagger}^\infty u^{2m-1} \exp\biggl(-C \cdot \frac{u}{2L} \biggr) du  \\
        &\leq \int_{u^*}^{\infty} u^{2m-1} \exp\biggl(-C \cdot \frac{u^2}{2v} \biggr) du  + \int_{u^\dagger}^\infty u^{2m-1} \exp\biggl(-C \cdot \frac{u}{2L} \biggr) du  \\
        &= 2^{m-1} \biggl(\frac{v}{C}\biggr)^m \int_{\frac{C(u^*)^2}{2v}}^\infty  w^{m-1} e^{-w} dw + \biggl(\frac{2L}{C}\biggr)^{2m} \int_{\frac{Cu^\dagger}{2L}}^\infty w^{2m-1} e^{-w} dw \\
        &\leq 2^{m-1} \biggl(\frac{v}{C}\biggr)^m \int_{\log v}^\infty  w^{m-1} e^{-w} dw + \biggl(\frac{2L}{C}\biggr)^{2m} \int_{\frac{Cv}{2L^2}}^\infty w^{2m-1} e^{-w} dw\\
        &\leq 2^{m-1} \biggl(\frac{v}{C}\biggr)^m \Gamma(m) \cdot \frac{e^m (\log v)^m}{m^m e^{\log v}} + \biggl(\frac{2L}{C}\biggr)^{2m} \Gamma(2m) \cdot \frac{e^{2m} (\frac{Cv}{2L^2})^{2m}}{(2m)^{2m} e^{Cv / 2L^2}} \\
        &\lcon v^m  \cdot \frac{(\log v)^m}{v} + L^{2m} \cdot \frac{(\frac{v}{L^2})^{2m}}{e^{Cv / 2L^2}}  \\
        &\lcon \frac{(v \log v)^m}{v} + \frac{v^{2m}}{L^{2m} e^{Cv/2L^2}} 
    \end{align*}
    Altogether, we thus have
    \begin{equation*}
        \E[\|\bm Y\|_{\text{op}}^{2m}] \lcon (v\log v)^m + \frac{v^{2m+1}}{L^{2m} e^{Cv / 2L^2}}
    \end{equation*}
    Now, we provide crude control over the latter term $\|\bm X - \bm Y\|_{\text{op}}^{2m}$. First, note that, for any $q > 0$,
    \begin{align*}
        \E[|x_{11}|^{q} \cdot \bm 1[|x_{11}| \geq L]] \leq \E[|x_{11}|^{q+r} L^{-r}]
    \end{align*}
    Then, we can bound
    \begin{align*}
        \E[\|\bm X - \bm Y\|_{\text{op}}^{2m}] &\leq \E[\|(x_{ij}\bm 1[|x_{ij}| \geq L] - \E[x_{ij}\bm 1[|x_{ij}| \geq L]])_{ij}\|_{F}^{2m}]  \\
        &\lcon \E[\|(x_{ij}\bm 1[|x_{ij}| \geq L])_{ij}\|_F^{2m}] + \|\E[(x_{ij}\bm 1[|x_{ij}| \geq L]])_{ij}\|_{F}^{2m} 
    \end{align*}
    For the first term, we can further bound
    \begin{align*}
        \E[\|(x_{ij}\bm 1[|x_{ij}| \geq L])_{ij}\|_F^{2m}] &= \E\biggl[\biggl(\sum_{i,j} (x_{ij} \bm 1[|x_{ij}| \geq L])^2\biggr)^m\biggr] \\
        &\leq (np)^{m-1} \sum_{i,j} \E[x_{ij}^{2m} \bm 1[|x_{11}| \geq L]] \\
        &= (np)^m \E[x_{11}^{2m} \bm 1[|x_{11}| \geq L]] \\
        &\leq v^{2m} \E[|x_{11}|^{2m+r}] L^{-r}
    \end{align*}
    and similarly
    \begin{align*}
         \|\E[(x_{ij}\bm 1[|x_{ij}| \geq L]])_{ij}\|_{F}^{2m} &= \biggl(\sum_{i,j} \E[x_{ij} \bm 1[|x_{ij}| \geq L]]^2\biggr)^m \\
         &= (np)^m \E[x_{11} \bm 1[|x_{11}| \geq L]]^{2m} \\
         &\leq v^{2m} \E[|x_{11}|^{1+r}]^{2m} L^{-2mr}
    \end{align*}
    Finally, let $L = \sqrt{\frac{Cv}{2\log v}}$ and $r = 2m$; then, the condition $\frac{v}{L^2} \to \infty$ is still satisfied (and thus our subexponential bound is valid).
    Putting together all of our bounds, we have
    \begin{align*}
        \E[\|\bm X\|_{\text{op}}^2] &\lcon (v \log v)^m +  \frac{v^{2m+1}}{L^{2m} e^{Cv/2L^2}} + \frac{v^{2m}}{L^r} + \frac{v^{2m}}{L^{2mr}} \\
        &\lcon (v \log v)^m .
    \end{align*}
\end{proof}
We can prove a tighter bound when the entries are Gaussian:
\begin{lemma}\label{lm:gaussian_matrix_norm_expec}
    Let $\bm Z \in \R^{n \times p}$ consist of i.i.d. $\mathcal{N}(0,1)$ entries. Then, there exists an absolute constant $C > 0$ such that
    \begin{equation*}
        \E[\| \bm Z\|_{\text{op}}^{2m}] \leq C \cdot \max(n, p)^m
    \end{equation*}
\end{lemma}
\begin{proof}
    Using \cite[Theorem 4.4.3]{vershynin2018probability}, there is a constant $C > 0$ such that, for any $t > 0$,
    \begin{equation*}
        \PP(\| \bm Z \|_{\text{op}} \geq C (\sqrt{n} + \sqrt{p} + t)) \leq 2 e^{-t^2}
    \end{equation*}
    Recalling $(C (\sqrt{n} + \sqrt{p} + t))^{2m} \leq 2^{2m-1} C^{2m} ((\sqrt{n} + \sqrt{p})^{2m} + t^{2m})$. Therefore, letting $C' = 2^{2m-1} C^{2m}$,
    \begin{equation*}
        \PP(\|\bm Z\|_{\text{op}}^{2m} \geq C' ((\sqrt{n} + \sqrt{p})^{2m} + t^{2m})) \leq \PP(\|\bm Z\|_{\text{op}}^{2m} \geq (C (\sqrt{n} + \sqrt{p} + t))^{2m}) \leq 2e^{-t^2}
    \end{equation*}
    or rather
    \begin{equation*}
        \PP(\|\bm Z \|_{\text{op}}^{2m} \geq C' (\sqrt{n} + \sqrt{p})^{2m} + r) \leq 2e^{-(r / C')^{1/m}}.
    \end{equation*}
    Now, we can bound
    \begin{align*}
        \E[\|\bm Z\|_{\text{op}}^{2m}] &= \int_0^\infty \PP(\|\bm Z||_{\text{op}}^{2m} \geq t) dt \\
        &= \int_0^{C'(\sqrt{n} + \sqrt{p})^{2m}} \PP(\|\bm Z\|_{\text{op}}^{2m} \geq t) dt + \int_{C'(\sqrt{n} + \sqrt{p})^{2m}}^\infty \PP(\|\bm Z\|_{\text{op}}^{2m} \geq t) dt   \\
        &\leq C'(\sqrt{n} + \sqrt{p})^{2m} + \int_0^\infty \PP(\|\bm Z\|_{\text{op}}^{2m} \geq C'(\sqrt{n} + \sqrt{p})^{2m} + t) dt \\
        &\leq C'(\sqrt{n} + \sqrt{p})^{2m} + \int_0^\infty 2 e^{-(t/C')^{1/m}} dt \\
        &= C'(\sqrt{n} + \sqrt{p})^{2m} + m! \cdot C'
    \end{align*}
    from which the result follows.
\end{proof}
We also require the following moment bound.
\begin{lemma}\label{lm:basic_S_moments}
    Let $\bm W$ be a random $\R^p$-valued random variable consisting of i.i.d. entries such that $\E[W_1] = 0$, $\Var(W_1) = 1$, and $\E[W_1^{2m}] < \infty$. For a vector $\bm u$, positive semidefinite matrix $\bm A$, and a positive integer $m \geq 2$, we have
    \begin{equation*}
        \E[\|\bm W\|_2^m] \lcon p^{m/2}, \quad \E[|\bm u^\top \bm W|^m] \lcon \|\bm u\|_2^m, \quad  , \quad \E[(\bm W^\top \bm A \bm W)^m] \lcon\|\bm A\|_{\text{op}}^m p^m
    \end{equation*}
    where $\lcon$ is taken with respect to $m$ and the moments of $W_1$.
\end{lemma}
\begin{proof}
    We begin by bounding
    \begin{equation*}
        \E[\|\bm W\|_2^m] = p^{m/2} \E\biggl[\biggl(\frac{1}{p} \sum_{i=1}^p W_i^2\biggr)^{m/2} \biggr] \leq p^{m/2} \E\biggl[\frac{1}{p} \sum_{i=1}^p W_i^m \biggr] = p^{m/2} \E[Z_1^m] \lcon p^{m/2}
    \end{equation*}
    Next, note that $\{u_i W_i\}_{i=1}^p$ are independent random variables and thus, by the Marcinkiewicz-Zygmund inequality, we have
    \begin{align*}
        \E[|\bm u^\top \bm W|^m] &= \E\biggl[ \biggl|\sum_{i=1}^p u_i W_i\biggr|^m\biggr] \lcon \biggl(\sum_{i=1}^p \E[|u_i W_i|^m]^{2/m}\biggr)^{m/2} = \E[|W_1|^m] \biggl(\sum_{i=1}^p u_i^2\biggr)^{m/2} \\
        &\lcon \|\bm u\|_2^m.
    \end{align*}
    Finally, we can use the first bound to yield
    \begin{equation*}
        \E[(\bm W^\top \bm A \bm W)^m] \leq \|\bm A\|_{\text{op}}^m \E[\|\bm W\|_2^{2m}] \lcon \|\bm A\|_{\text{op}}^m p^m
    \end{equation*}
\end{proof}

\section{Proof for Model Shift}\label{sec:proof:model_shift}
Similar to Section \ref{sec:proof:design_shift}, we first prove the ridge case (Theorem \ref{thm:model_shift_ridge}) and subsequently prove the min-norm interpolator case (Theorem \ref{thm:model_shift}).

\subsection{Proof of Theorem \ref{thm:model_shift_ridge}}\label{subsec:proof:model_shift_ridge}

We only prove \eqref{eqn:model_shift_ridge_B2} and \eqref{eqn:model_shift_ridge_B3} here as \eqref{eqn:model_shift_ridge_V} and \eqref{eqn:model_shift_ridge_B1} are direct corollaries of \cite[Theorem 5]{hastie2022surprises}. 

\subsubsection{Proof of Eq. \ref{eqn:model_shift_ridge_B2}}
Let $\hat \bSigma_Z = \bm Z^\top \bm Z / n$, and let $I \subseteq [n]$ be a random subset of size $n_1$ drawn without replacement. Then, by exchangeability of the rows of $\bm Z$, we have 
\begin{align}
    B_2(\hat \bbeta_\lambda; \bbeta^{(2)}) &= \tilde \bbeta^\top \bSigma^{1/2} \biggl(\frac{\bm Z^{(1)\top} \bm Z^{(1)}}{n}\biggr)  (\hat \bSigma_Z + \lambda \bSigma^{-1})^{-2} \biggl(\frac{\bm Z^{(1)\top} \bm Z^{(1)}}{n} \biggr) \bSigma^{1/2}\tilde \bbeta \\
    &\overset{d}{=} \tilde \bbeta^\top \bSigma^{1/2}\biggl(\frac{1}{n} \sum_{i \in I} Z_i Z_i^\top \biggr) (\hat \bSigma_Z + \lambda \bSigma^{-1})^{-2} \biggl(\frac{1}{n} \sum_{i \in I} Z_i Z_i^\top \biggr) \bSigma^{1/2} \tilde \bbeta \\
    &= \sum_{i \in I} \tilde \bbeta^\top \bSigma^{1/2}\biggl(\frac{Z_i Z_i^\top}{n}\biggr) (\hat \bSigma_Z + \lambda \bSigma^{-1})^{-2} \biggl( \frac{Z_i Z_i^\top}{n}\biggr) \bSigma^{1/2} \tilde \bbeta \\
    &\qquad + \sum_{i,j \in I; i \neq j} \tilde \bbeta^\top \bSigma^{1/2} \biggl(\frac{Z_i Z_i^\top}{n}\biggr)(\hat \bSigma_Z + \lambda \bSigma^{-1})^{-2} \biggl(\frac{Z_j Z_j^\top}{n} \biggr) \bSigma^{1/2} \tilde \bbeta.
\end{align}
We now introduce the following quantities:
\begin{align}
        B_2^F(\hat \bbeta_\lambda; \bbeta^{(2)}) &:=\tilde \bbeta^\top \bSigma^{1/2}\hat \bSigma_Z(\hat \bSigma_Z + \lambda \bSigma^{-1})^{-2} \hat \bSigma_Z \bSigma^{1/2} \tilde \bbeta \\
    B_2^P(\hat \bbeta_\lambda; \bbeta^{(2)}) &:=\sum_{i=1}^n \tilde \bbeta^\top \bSigma^{1/2}\biggl(\frac{Z_i Z_i^\top}{n}\biggr) (\hat \bSigma_Z + \lambda \bSigma^{-1})^{-2} \biggl( \frac{Z_i Z_i^\top}{n}\biggr) \bSigma^{1/2} \tilde \bbeta.
\end{align}
Then, by symmetry, it follows that
\begin{align}\label{eq:B2_decomp}
    \E[B_2(\hat \bbeta_\lambda; \bbeta^{(2)})] &= \frac{n_1(n_1 - 1)}{n(n-1)} \E[B_2^F(\hat \bbeta_\lambda; \bbeta^{(2)})] \\
    &\qquad + \biggl(\frac{n_1}{n} - \frac{n_1 (n_1 - 1)}{n(n-1)}\biggr) \E[B_2^P(\hat \bbeta_\lambda; \bbeta^{(2)})].
\end{align}
From the following lemma, it suffices to analyze $\E[B_2^F(\hat \bbeta_\lambda;\bbeta^{(2)})]$ and $\E[B_2^P(\hat \bbeta_\lambda;\bbeta^{(2)})]$:
\begin{lemma}\label{lm:b2_full_efron_stein}
    In the setting of Theorem \ref{thm:model_shift_ridge}, and for any small constant $c > 0$, we have 
    \begin{align}
        &\biggl|B_2(\hat \bbeta_\lambda; \bbeta^{(2)}) - \frac{n_1(n_1 - 1)}{n(n-1)} \E[B_2^F(\hat \bbeta_\lambda; \bbeta^{(2)})] - \biggl(\frac{n_1}{n} - \frac{n_1 (n_1 - 1)}{n(n-1)}\biggr) \E[B_2^P(\hat \bbeta_\lambda; \bbeta^{(2)})]\biggr| \\
        &= O(\lambda^{-4} n^{-1/2 + c})
    \end{align}
    with high probability over $(\bZ^{(1)}, \bZ^{(2)})$.
\end{lemma}
\begin{proof}
    By \eqref{eq:B2_decomp} and Chebyshev's inequality, it suffices to show that $\Var(B_2(\hat \bbeta_\lambda; \bbeta^{(2)})) = O(\lambda^{-8} n^{-1})$. We prove this bound by the Efron-Stein inequality: by symmetry, we have
    \begin{align*}
        &\Var(B_2(\hat \bbeta_\lambda; \bbeta^{(2)})) \\
        &\leq n_1 \E\biggl[\biggl(\tilde \bbeta^{\top} \bSigma^{1/2}\biggl(\biggl(\frac{\bm Z^{(1)\top} \bm Z^{(1)}}{n} \biggr)(\hat \bSigma_Z + \lambda \bSigma^{-1})^{-2} \biggl(\frac{\bm Z^{(1)\top} \bm Z^{(1)}}{n} \biggr) \\
        &\hspace{5em}- \biggl(\frac{\bm Z^{(1)'\top} \bm Z^{(1)'}}{n} \biggr)(\hat \bSigma_Z' + \lambda \bSigma^{-1})^{-2} \biggl(\frac{\bm Z^{(1)'\top} \bm Z^{(1)'}}{n} \biggr)\biggr) \bSigma^{1/2} \tilde \bbeta\biggr)^2  \biggr] \\
        &\qquad + n_2 \E\biggl[\biggl(\tilde \bbeta^\top\bSigma^{1/2}\biggl(\frac{\bm Z^{(1)\top} \bm Z^{(1)}}{n} \biggr)((\hat \bSigma_Z + \lambda \bSigma^{-1})^{-2} \\
        &\qquad \hspace{6em}- (\hat \bSigma_Z'' + \lambda \bSigma^{-1})^{-2} ) \biggl(\frac{\bm Z^{(1)\top} \bm Z^{(1)}}{n} \biggr) \bSigma^{1/2} \tilde \bbeta\biggr)^2  \biggr]
    \end{align*}
    where $\bm Z^{(1)'}$ and $\bm Z^{(2)'}$ involve replacing the first rows of $\bm Z^{(1)}$ and $\bm Z^{(2)}$ with independent copies, respectively, and $\hat \bSigma_Z' = \bm Z^{(1)'\top} \bm Z^{(1)'} / n + \bm Z^{(2)\top} \bm Z^{(2)} / n$ and $\hat \bSigma= \bm Z^{(1)\top} \bm Z^{(1)} / n + \bm Z^{(2)'\top} \bm Z^{(2)'} / n$.

    Now, letting $\hat \bSigma_Z^- = \hat \bSigma_Z - \frac{1}{n} \bm Z_1^{(1)} \bm Z_1^{(1)\top}$, $\bm u = \bSigma^{1/2} \tilde \bbeta$, $\bm A = (\hat \bSigma_Z^- + \lambda \bSigma^{-1})^{-1}$, and $\bm B = \frac{1}{n} \sum_{i=2}^{n_1} \bm Z_i^{(1)} \bm Z_i^{(1)\top}$, we can apply the Sherman-Morrison formula to yield
    \begin{align}
        &\E\biggl[\biggl(\tilde \bbeta^{\top} \bSigma^{1/2}\biggl(\biggl(\frac{\bm Z^{(1)\top} \bm Z^{(1)}}{n} \biggr)(\hat \bSigma_Z + \lambda \bSigma^{-1})^{-2} \biggl(\frac{\bm Z^{(1)\top} \bm Z^{(1)}}{n} \biggr) \\
        &\qquad - \biggl(\frac{\bm Z^{(1)'\top} \bm Z^{(1)'}}{n} \biggr)(\hat \bSigma_Z' + \lambda \bSigma^{-1})^{-2} \biggl(\frac{\bm Z^{(1)'\top} \bm Z^{(1)'}}{n} \biggr)\biggr) \bSigma^{1/2} \tilde \bbeta\biggr)^2  \biggr] \\
        &\lesssim \frac{1}{n^2}\E[(\bm u^\top \bm B \bm A^2 \bm Z_1^{(1)} \bm Z_1^{(1)\top} \bm u)^2] + \frac{1}{n^4} \E[(\bm u^\top \bm Z_1^{(1)}\bm Z_1^{(1)\top} \bm A^2 \bm Z_1^{(1)} \bm Z^{(1)\top} \bm u)^2] \\
        &\qquad + \frac{1}{n^2} \E[(\bm u^\top \bm B \bm A^2 \bm Z_1^{(1)} \bm Z_1^{(1)\top} \bm A \bm B \bm u)^2] + \frac{1}{n^4} \E[(\bm u^\top \bm Z_1^{(1)} \bm Z_1^{(1)\top} \bm A^2 \bm Z_1^{(1)} \bm Z_1^{(1)\top} \bm A \bm B \bm u)^2] \\
        &\qquad + \frac{1}{n^4} \E[(\bm u^\top \bm Z_1^{(1)} \bm Z_1^{(1)\top} \bm A \bm Z_1^{(1)} \bm Z_1^{(1)\top} \bm A^2 \bm B \bm u)^2] \\
        &\qquad + \frac{1}{n^6} \E[(\bm u^\top \bm Z_1^{(1)} \bm Z_1^{(1)\top} \bm A \bm Z_1^{(1)} \bm Z_1^{(1)\top} \bm A^2 \bm Z_1^{(1)} \bm Z_1^{(1)\top} \bm u)^2] \\
        &\qquad + \frac{1}{n^4} \E[(\bm u^\top \bm B \bm A \bm Z_1^{(1)} \bm Z_1^{(1)\top} \bm A^2 \bm Z_1^{(1)}\bm Z_1^{(1)\top} \bm A \bm B \bm u)^2] \\
        &\qquad + \frac{1}{n^6} \E[(\bm u^\top \bm Z_1^{(1)} \bm Z_1^{(1)\top} \bm A \bm Z_1^{(1)} \bm Z_1^{(1)\top} \bm A^2 \bm Z_1^{(1)}\bm Z_1^{(1)\top} \bm A \bm B \bm u)^2]  \\
        &\qquad + \frac{1}{n^8} \E[(\bm u^\top \bm Z_1^{(1)} \bm Z_1^{(1)\top} \bm A \bm Z_1^{(1)} \bm Z_1^{(1)\top} \bm A^2 \bm Z_1^{(1)}\bm Z_1^{(1)\top} \bm A \bm Z_1^{(1)} \bm Z_1^{(1)\top} \bm u)^2] \\
        &\lesssim\frac{\|\bSigma\|_{\text{op}}^4 \|\bm u\|_2^4}{\lambda^{4} n^2} + \frac{\|\bSigma\|_{\text{op}}^4 \|\bm u\|_2^4 p^2}{\lambda^4 n^4}  + \frac{\|\bSigma\|_{\text{op}}^6 \|\bm u\|_2^4}{\lambda^6 n^2}  + \frac{\|\bSigma\|_{\text{op}}^6 \|\bm u\|_2^4 p^2}{\lambda^6 n^4}  \\
        &\qquad + \frac{\|\bSigma\|_{\text{op}}^6\|\bm u\|_2^4 p^4}{\lambda^6 n^6}  + \frac{\|\bSigma\|_{\text{op}}^8 \|\bm u\|_2^4 p^2}{\lambda^8 n^4}  + \frac{\|\bSigma\|_{\text{op}}^8 \|\bm u\|_2^4 p^4}{\lambda^8 n^6}  + \frac{\|\bSigma\|_{\text{op}}^8\|\bm u\|_2^4 p^6}{\lambda^8 n^8}  
    \end{align}
    where the last step follows from repeatedly applying H\"older's inequality and Lemma \ref{lm:basic_S_moments}.

    Letting $\hat \bSigma_Z^{--} = \hat \bSigma_Z - \frac{1}{n} \bm Z_1^{(2)} \bm Z_1^{(2)\top}$,
    $\bm u = \bSigma^{1/2} \tilde \bbeta$, $\bm A = (\hat \bSigma_Z^{--} + \lambda \bSigma^{-1})^{-1}$, and $\bm B = \frac{\bm Z^{(1)\top} \bm Z^{(1)}}{n}$, we can similarly bound
    \begin{align*}
        &\E\biggl[\biggl(\tilde \bbeta^\top\bSigma^{1/2}\biggl(\frac{\bm Z^{(1)\top} \bm Z^{(1)}}{n} \biggr)((\hat \bSigma_Z + \lambda \bSigma^{-1})^{-2}  - (\hat \bSigma_Z'' + \lambda \bSigma^{-1})^{-2} )\\
        &\qquad \qquad \biggl(\frac{\bm Z^{(1)\top} \bm Z^{(1)}}{n} \biggr) \bSigma^{1/2} \tilde \bbeta\biggr)^2  \biggr] \\
        &\lesssim \frac{1}{n^2} \E[(\bm u ^\top \bm B \bm A^2 \bm Z_1^{(2)} \bm Z^{(2)\top}_1 \bm A \bm B \bm u)^2] + \frac{1}{n^4} \E[(\bm u^\top \bm B \bm A \bm Z_1^{(2)}\bm Z_1^{(2)\top} \bm A^2 \bm Z_1^{(2)} \bm Z_1^{(2)\top} \bm A \bm B \bm u)^2] \\
        &\leq \frac{\|\bSigma\|_{\text{op}}^6 \|\bm u\|_2^4}{\lambda^6 n^2} + \frac{\|\bSigma\|_{\text{op}}^8 \|\bm u\|_2^4 p^2}{\lambda^8 n^4}
    \end{align*}
    The lemma follows by combining these bounds.
\end{proof}
We now analyze $\E[B_2^F(\hat \bbeta_\lambda; \bbeta^{(2)})]$.
Note that we can express
\begin{align} 
    B_2^F(\hat \bbeta_\lambda; \bbeta^{(2)}) &= \tilde \bbeta^\top \bSigma^{1/2}\hat \bSigma_Z(\hat \bSigma_Z + \lambda \bSigma^{-1})^{-2} \hat \bSigma_Z \bSigma^{1/2} \tilde \bbeta \\
    &= \tilde \bbeta^\top \bSigma \tilde \bbeta - 2 \lambda \tilde \bbeta^\top \bSigma^{1/2} (\hat \bSigma_Z + \lambda \bSigma^{-1})^{-1} \bSigma^{-1/2} \tilde \bbeta \\
    &\qquad + \lambda^2 \tilde \bbeta^\top \bSigma^{-1/2} (\hat \bSigma_Z + \lambda \bSigma^{-1})^{-2} \bSigma^{-1/2} \tilde \bbeta \\
    &= \tilde \bbeta^\top \bSigma \tilde \bbeta - 2 \lambda \tilde \bbeta^\top \bSigma (\hat \bSigma + \lambda \bm I)^{-1} \tilde \bbeta  \\
    &\qquad + \lambda^2 \tilde \bbeta^\top (\hat \bSigma + \lambda \bm I)^{-1} \bSigma (\hat \bSigma + \lambda \bm I)^{-1} \tilde \bbeta. \label{eq:b2f_decomp}
\end{align}
\begin{lemma}\label{lm:b2f_full}
Define the quantity
\begin{equation}
    \mathcal B_2^F(\lambda; \hat H_n, \hat G_n^{\tilde \bbeta}) := \|\tilde \bbeta\|_2^2 \int \frac{s((1 -\gamma + \gamma \lambda m_n(-\lambda))^2 s^2 + \lambda^2 \gamma m_{n,1}(-\lambda))}{(\lambda + (1 - \gamma + \gamma \lambda m_n(-\lambda))s)^2}
\end{equation}
    In the setting of Theorem \ref{thm:model_shift_ridge}, for any small constant $c > 0$, we have
    \begin{equation}
        |\E[B_2^F(\hat \bbeta_\lambda; \bbeta^{(2)})] - \mathcal B_2^F(\lambda; \hat H_n, \hat G_n^{\tilde \bbeta})| = O(n^{-1/2+c} \lambda^{-1})
    \end{equation}
    with probability over $(\bZ^{(1)}, \bZ^{(2)})$.
\end{lemma}
\begin{proof}
    We first note that the same Efron-Stein bound and application of Chebyshev's inequality as Lemma \ref{lm:b2_full_efron_stein} yields
    \begin{equation}
        |B_2^F(\hat \bbeta_\lambda; \bbeta^{(2)}) - \E[B_2^F(\hat \bbeta_\lambda; \bbeta^{(2)})]| = O(n^{-1/2+c}\lambda^{-4})
    \end{equation}
    with high probability.
    Next, we can individually analyze each element of \eqref{eq:b2f_decomp}.
    First, we immediately have $\tilde \bbeta^\top \bSigma \tilde \bbeta = \|\tilde \bbeta\|_2^2 \int s\ d\hat G_{\tilde \bbeta,n}(s)$ by definition. Then, we have
    \begin{align*}
        -2\lambda \tilde \bbeta^\top \bSigma(\hat \bSigma + \lambda \bm I)^{-1} \tilde \bbeta &= -\|\tilde \bbeta\|_2^2 \int  \frac{2\lambda s}{\lambda + (1 - \gamma + \gamma \lambda m_n(-\lambda))s} d\hat G_{\tilde \bbeta,n}(s) \\
        &\qquad + O_\prec(n^{-1/2} \lambda^{-1}) \\
        \lambda^2 \tilde \bbeta^\top (\hat \bSigma + \lambda \bm I)^{-1} \bSigma (\hat \bSigma + \lambda \bm I)^{-1} \tilde \bbeta &= \|\tilde \bbeta\|_2^2 \int \frac{\lambda^2 (1 + \gamma m_{n,1}(-\lambda)) s}{(\lambda + (1 - \gamma + \gamma \lambda m_n(-\lambda))s)^2} d\hat G_{\tilde \bbeta,n}(s) \\
        &\qquad + O_\prec(n^{-1/2}\lambda^{-1})
    \end{align*}
    from applying the existing anisotropic local laws from \cite[Theorem 3.16]{knowles2017anisotropic} and \cite[Section A.2]{hastie2022surprises}, respectively. Combining these terms and simplifying the resulting integral expression reveals that
    \begin{equation*}
        |B_2^F(\hat \bbeta_\lambda; \bbeta^{(2)}) - \mathcal B_2^F(\lambda; \hat H_n, \hat G_n^{\tilde \bbeta})| = O_\prec(n^{-1/2} \lambda^{-1}).
    \end{equation*}
    Therefore, by the triangle inequality, we have
    \begin{equation*}
        |\E[B_2^F(\hat \bbeta_\lambda; \bbeta^{(2)})] - \mathcal B_2^F(\lambda; \hat H_n, \hat G_n^{\tilde \bbeta})| = O(n^{-1/2+c} \lambda^{-1})
    \end{equation*}
    with high probability.
\end{proof}

We now turn our attention to $\E[B_2^P(\hat \bbeta_\lambda; \bbeta^{(2)})]$.
Let $\bm A = (\hat \bSigma^-_Z + \lambda \bSigma^{-1})^{-1}$. Once again, by symmetry and by the Sherman-Morrison formula, we can simplify
\begin{align}\label{eq:symmetry_B2P}
    &\E[B_2^P(\hat \bbeta_\lambda; \bbeta^{(2)})] \\ &= \E\biggl[\sum_{i=1}^n \tilde \bbeta^\top \bSigma^{1/2}\biggl(\frac{Z_i Z_i^\top}{n}\biggr) (\hat \bSigma_Z + \lambda \bSigma^{-1})^{-2} \biggl( \frac{Z_i Z_i^\top}{n}\biggr) \bSigma^{1/2} \tilde \bbeta  \biggr] \\
    &= \frac{1}{n} \E\biggl[ (\tilde \bbeta^\top \bSigma^{1/2} Z_1)^2 \cdot Z_1^\top (\hat \bSigma_Z + \lambda \bSigma^{-1})^{-2} Z_1   \biggr] \\
    &= \frac{1}{n} \E[(\tilde \bbeta^\top \bSigma^{1/2} \bm Z_1)^2 \cdot \bm Z_1^\top \bm A^2 \bm Z_1] - \frac{2}{n^2} \E\biggl[(\tilde \bbeta^\top \bSigma^{1/2} \bm Z_1)^2 \cdot \frac{\bm Z_1^\top \bm A^2 \bm Z_1\bm Z_1^\top \bm A \bm Z_1}{1 + \frac{1}{n} \bm Z_1^\top \bm A \bm Z_1}\biggr] \\
    &\qquad + \frac{1}{n^3} \E\biggl[ (\tilde \bbeta^\top \bSigma^{1/2} \bm Z_1)^2 \cdot \frac{\bm Z_1^\top \bm A \bm Z_1 \bm Z_1^\top \bm A^2 \bm Z_1 \bm Z_1^\top \bm A \bm Z_1}{(1 + \frac{1}{n} \bm Z_1^\top \bm A \bm Z_1)^2} \biggr].
\end{align}
Note that, conditional on $\bm A$, the first term is a product of Gaussian quadratic forms, while the second and third terms are \textit{almost} equal to products of Gaussian quadratic forms. By replacing the denominators in the second and third terms by quantities independent of $\bm Z_1$ before applying formulas for the expectation of products of Gaussian quadratic forms, we yield the following result:
\begin{lemma}
    Define $B_2^{P,\star}(\hat \bbeta_\lambda; \bbeta^{(2)}) := \|\bSigma^{1/2} \tilde \bbeta\|_2^2 \frac{\frac{1}{n}\Tr[\bm A^2] }{(1 + \frac{1}{n} \Tr \bm A)^2}$. Then,
    \begin{equation}
        |\E[B_2^P(\hat \bbeta_\lambda; \bbeta^{(2)})] - \E[B_2^{P,\star}(\hat \bbeta_\lambda; \bbeta^{(2)})]| = O(\lambda^{-6} n^{-1}).
    \end{equation}
\end{lemma}
\begin{proof}
For the second term in \eqref{eq:symmetry_B2P}, we can apply H\"older's inequality and Lemma \ref{lm:basic_S_moments} to yield
\begin{align}\label{eq:trace_B2P_denom_2}
    &\biggl|\frac{2}{n^2} \E\biggl[(\tilde \bbeta^\top \bSigma^{1/2} \bm Z_1)^2 \cdot \frac{\bm Z_1^\top \bm A^2 \bm Z_1\bm Z_1^\top \bm A \bm Z_1}{1 + \frac{1}{n} \bm Z_1^\top \bm A \bm Z_1}\biggr] - \frac{2}{n^2} \E\biggl[(\tilde \bbeta^\top \bSigma^{1/2} \bm Z_1)^2 \cdot \frac{\bm Z_1^\top \bm A^2 \bm Z_1\bm Z_1^\top \bm A \bm Z_1}{1 + \frac{1}{n} \Tr \bm A}\biggr] \biggr| \\
    &\lesssim \frac{1}{n^2} \E\biggl[(\tilde \bbeta^\top \bSigma^{1/2} \bm Z_1)^2 \cdot \bm Z_1^\top \bm A^2 \bm Z_1 \bm Z_1^\top \bm A \bm Z_1 \cdot \biggl|\frac{1}{n} \bm Z_1^\top \bm A \bm Z_1 - \frac{1}{n} \Tr \bm A\biggr|\biggr] \\
    &\leq \frac{1}{n^2} \E[(\tilde \bbeta^\top \bSigma^{1/2} \bm Z_1)^{16}]^{1/8} \cdot \E[(\bm Z_1^\top \bm A^2 \bm Z_1)^4]^{1/4} \E[(\bm Z_1^\top \bm A \bm Z_1)^4]^{1/4} \\
    &\qquad \cdot \E\biggl[\biggl(\frac{1}{n} \bm Z_1^\top \bm A \bm Z_1 - \frac{1}{n} \Tr \bm A\biggr)^2\biggr]^{1/2} \\
    &\lesssim \frac{\|\bSigma\|_{\text{op}}^3 \|\bSigma^{1/2} \tilde \bbeta\|_2^2 p^2}{\lambda^3 n^2} \cdot \frac{1}{n} \Var(\bm Z_1^\top \bm A \bm Z_1)^{1/2} \\
    &=\frac{\|\bSigma\|_{\text{op}}^3 \|\bSigma^{1/2} \tilde \bbeta\|_2^2 p^2}{\lambda^3 n^2} \cdot \frac{\sqrt 2}{n} \Tr(\bm A^2)^{1/2}.
\end{align}
Similarly, for the third term in \eqref{eq:symmetry_B2P}, we can apply H\"older's inequality and Lemma \ref{lm:basic_S_moments}  to yield 
\begin{align}\label{eq:trace_B2P_denom_3}
    &\biggl|\frac{1}{n^3} \E\biggl[ (\tilde \bbeta^\top \bSigma^{1/2} \bm Z_1)^2 \cdot \frac{\bm Z_1^\top \bm A \bm Z_1 \bm Z_1^\top \bm A^2 \bm Z_1 \bm Z_1^\top \bm A \bm Z_1}{(1 + \frac{1}{n} \bm Z_1^\top \bm A \bm Z_1)^2} \biggr] \\
    &\qquad - \frac{1}{n^3} \E\biggl[ (\tilde \bbeta^\top \bSigma^{1/2} \bm Z_1)^2 \cdot \frac{\bm Z_1^\top \bm A \bm Z_1 \bm Z_1^\top \bm A^2 \bm Z_1 \bm Z_1^\top \bm A \bm Z_1}{(1 + \frac{1}{n} \Tr \bm A)^2} \biggr]\biggr| \\
    &\leq \frac{1}{n^3} \E\biggl[(\tilde \bbeta^\top \bSigma^{1/2} \bm Z_1)^2 (\bm Z_1^\top \bm A \bm Z_1)^2 (\bm Z_1^\top \bm A^2 \bm Z_1) \cdot\biggl|\frac{1}{n} \bm Z_1^\top \bm A \bm Z_1 - \frac{1}{n} \Tr \bm A  \biggr| \\
    &\qquad \cdot \biggl|2 + \frac{1}{n} \bm Z_1^\top \bm A \bm Z_1 + \frac{1}{n} \Tr\bm A \biggr| \biggr] \\
    &\leq \frac{1}{n^3} \E[(\tilde \bbeta^\top \bSigma^{1/2} \bm Z_1)^{16}]^{1/8} \E[(\bm Z_1^\top \bm A \bm Z_1)^{16}]^{1/8} \E[(\bm Z_1^\top \bm A^2 \bm Z_1)^8]^{1/8} \\
    &\qquad \cdot\E\biggl[\biggl(2 + \frac{1}{n} \bm Z_1^\top \bm A \bm Z_1 + \frac{1}{n} \Tr\bm A\biggr)^8 \biggr]^{1/8} \E\biggl[\biggl(\frac{1}{n} \bm Z_1^\top \bm A \bm Z_1 - \frac{1}{n} \Tr \bm A  \biggr)^2 \biggr]^{1/2} \\
    &\lesssim \frac{\|\bSigma\|_{\text{op}}^4 \|\bSigma^{1/2} \tilde \bbeta\|_2^2 p^3}{\lambda^4 n^3}  \biggl(1 + \frac{\|\bSigma\|_{\text{op}} p}{\lambda n} + \frac{\|\bSigma\|_{\text{op}}}{\lambda} \biggr) \cdot \frac{1}{n} \Var(\bm Z_1^\top \bm A \bm Z_1)^{1/2} \\
    &= \frac{\|\bSigma\|_{\text{op}}^4 \|\bSigma^{1/2} \tilde \bbeta\|_2^2 p^3}{\lambda^4 n^3}  \biggl(1 + \frac{\|\bSigma\|_{\text{op}} p}{\lambda n} + \frac{\|\bSigma\|_{\text{op}}}{\lambda} \biggr) \cdot \frac{\sqrt 2}{n} \Tr(\bm A^2)^{1/2}.
\end{align}
Now, recognizing that $\bm A$ is independent of $\bm Z_1$, an isotropic Gaussian vector, we can apply \cite[Theorem 5.1]{magnus1978quadatic} to compute
\begin{align}\label{eq:B2P_eval}
    &\frac{1}{n} \E[(\tilde \bbeta^\top \bSigma^{1/2} \bm Z_1)^2 \cdot \bm Z_1^\top \bm A^2 \bm Z_1] - \frac{2}{n^2} \E\biggl[(\tilde \bbeta^\top \bSigma^{1/2} \bm Z_1)^2 \cdot \frac{\bm Z_1^\top \bm A^2 \bm Z_1\bm Z_1^\top \bm A \bm Z_1}{1 + \frac{1}{n} \Tr \bm A}\biggr] \\
    &\qquad + \frac{1}{n^3} \E\biggl[ (\tilde \bbeta^\top \bSigma^{1/2} \bm Z_1)^2 \cdot \frac{\bm Z_1^\top \bm A \bm Z_1 \bm Z_1^\top \bm A^2 \bm Z_1 \bm Z_1^\top \bm A \bm Z_1}{(1 + \frac{1}{n} \Tr \bm A)^2} \biggr] \\
    &= \frac{\|\bSigma^{1/2} \tilde \bbeta\|_2^2}{n} \E[\Tr[\bm A^2]] + \frac{2}{n} \E[\tilde \bbeta^\top \bSigma^{1/2} \bm A^2\bSigma^{1/2} \tilde \bbeta ] \\
    &\qquad -\frac{2}{n^2} \E\biggl[\frac{1}{1 + \frac{1}{n} \Tr \bm A}\biggl( \|\bSigma^{1/2} \tilde \bbeta\|_2^2 \Tr[\bm A^2] \Tr[\bm A] + 2\tilde \bbeta^\top \bSigma^{1/2} \bm A^2 \bSigma^{1/2} \tilde \bbeta \cdot \Tr[\bm A] \\
    &\qquad \hspace{7em}+ 2\tilde \bbeta^\top \bSigma^{1/2} \bm A\bSigma^{1/2} \tilde \bbeta\cdot  \Tr[\bm A^2] + 2\|\bSigma^{1/2} \tilde \bbeta\|_2^2 \cdot \Tr[\bm A^3]\\
    &\qquad \hspace{7em}+ 8 \tilde \bbeta^\top \bSigma^{1/2} \bm A^3 \bSigma^{1/2} \tilde \bbeta  \biggr) \biggr] \\
    &\qquad + \frac{1}{n^3} \E\biggl[\frac{1}{(1 + \frac{1}{n} \Tr \bm A)^2} \biggl(\|\bSigma^{1/2} \tilde \bbeta\|_2^2 \Tr[\bm A]^2 \Tr[\bm A^2] \\
    &\hspace{4em}+ 8(\|\bSigma^{1/2} \tilde \bbeta\|_2^2 \Tr[\bm A^4] + 2 \Tr[\bm A] \cdot \tilde \bbeta^\top \bSigma^{1/2} \bm A^3\bSigma^{1/2} \tilde \bbeta \\
    &\hspace{6em}+ \Tr[\bm A^2] \cdot \tilde \bbeta^\top \bSigma^{1/2} \bm A^2\bSigma^{1/2} \tilde \bbeta) \\
    &\hspace{4em}+ 4(2 \tilde \bbeta^\top \bSigma^{1/2} \bm A \bSigma^{1/2} \tilde \bbeta \cdot \Tr[\bm A^3] + \tilde \bbeta^\top \bSigma^{1/2} \bm A^2 \bSigma^{1/2} \tilde \bbeta \cdot \Tr[\bm A^2]) \\
    &\hspace{4em}+ 2(2\| \bSigma^{1/2} \tilde \bbeta\|_2^2 \Tr[\bm A] \Tr[\bm A^3] + 2 \tilde \bbeta^\top \bSigma^{1/2} \bm A\bSigma^{1/2} \tilde \bbeta\cdot \Tr[\bm A] \Tr[\bm A^2] \\
    &\hspace{4em}+ \| \bSigma^{1/2} \tilde \bbeta\|_2^2 \Tr[\bm A^2]^2 + \tilde \bbeta^\top \bSigma^{1/2} \bm A^2 \bSigma^{1/2} \tilde \bbeta\cdot  \Tr[\bm A]^2)  \\
    &\hspace{4em}+ 48 \tilde \bbeta^\top \bSigma^{1/2} \bm A^4 \bSigma^{1/2} \tilde \bbeta \biggr)  \biggr] \\
    &= \|\bSigma^{1/2} \tilde \bbeta\|_2^2 \cdot  \E\biggl[\frac{\frac{1}{n}\Tr[\bm A^2] }{(1 + \frac{1}{n} \Tr \bm A)^2} \biggr] + O\biggl(\frac{\|\bSigma\|_{\text{op}}^3 \|\bSigma^{1/2} \tilde \bbeta\|_2^2}{\lambda^3 n} + \frac{\|\bSigma\|_{\text{op}}^4 \|\bSigma^{1/2} \tilde \bbeta\|_2^2}{\lambda^4 n}\biggr)
\end{align}
The result follows by combining \eqref{eq:symmetry_B2P}, \eqref{eq:trace_B2P_denom_2}, \eqref{eq:trace_B2P_denom_3}, and \eqref{eq:B2P_eval}.
\end{proof}
We wish to apply existing random matrix theory results to characterize the asymptotic behavior of $B_2^{P,\star}(\hat \bbeta_\lambda; \bbeta^{(2)}) = \|\bSigma^{1/2} \tilde \bbeta\|_2^2 \frac{\frac{1}{n}\Tr[\bm A^2] }{(1 + \frac{1}{n} \Tr \bm A)^2}$, but these results apply to the stochastic quantity, not its expectation. Therefore, we first require the following concentration result that allows us to return our attention to the stochastic setting:
\begin{lemma}\label{lm:b2p_conc}
In the setting of Theorem \ref{thm:model_shift_ridge}, for any small constant $c > 0$, we have
\begin{equation}
    |B_2^{P,\star}(\hat \bbeta_\lambda; \bbeta^{(2)}) - \E[B_2^{P,\star}(\hat \bbeta_\lambda; \bbeta^{(2)})]| = O(\lambda^{-3} n^{-1+c})
\end{equation}
with high probability over $(\bZ^{(1)}, \bZ^{(2)})$.
\end{lemma}
\begin{proof}
    By Chebyshev's inequality, it suffices to show that $\Var(B_2^{P,\star}(\hat \bbeta_\lambda; \bbeta^{(2)})) = O(\lambda^{-6} n^{-2})$. We proceed via the Gaussian Poincar\'e inequality. To simplify our notation, note that it suffices to consider bound the quantity
    \begin{equation}
        \Var\biggl(\frac{\frac{1}{n+1}\Tr[(\frac{\bm Z^\top \bm Z}{n+1} + \lambda \bSigma^{-1})^{-2}] }{(1 + \frac{1}{n+1} \Tr \bm [(\frac{\bm Z^\top \bm Z}{n+1} + \lambda \bSigma^{-1})^{-1}])^2}  \biggr).
    \end{equation}
    For fixed $i,j$, note that $\frac{\partial}{\partial Z_{ij}} \bm Z^\top \bm Z = e_j \bm Z_i^\top + \bm Z_i e_j^\top$. Then,
    we can compute
    \begin{align}
        &\frac{\partial}{\partial Z_{ij}} \Tr\biggl[\biggl(\frac{\bm Z^\top \bm Z}{n+1} + \lambda \bSigma^{-1} \biggr)^{-1} \biggr] \\
        &= -\frac{1}{n+1}\Tr\biggl[\biggl(\frac{\bm Z^\top \bm Z}{n+1} + \lambda \bSigma^{-1}\biggr)^{-1} (e_j \bm Z_i^\top + \bm Z_i e_j^\top) \biggl(\frac{\bm Z^\top \bm Z}{n+1} + \lambda \bSigma^{-1}\biggr)^{-1}\biggr] \\
        &= -\frac{2}{n+1}  \bm Z_i^\top \biggl(\frac{\bm Z^\top \bm Z}{n+1} + \lambda \bSigma^{-1}\biggr)^{-2} e_j.
    \end{align}
    and thus, by the chain rule,
    \begin{equation}
        \frac{\partial}{\partial Z_{ij}}\frac{1}{(1 + \frac{1}{n+1} \Tr[(\frac{\bm Z^\top \bm Z}{n+1} + \lambda \bSigma^{-1})^{-1}])^2} =- \frac{2}{n+1} \frac{\frac{\partial}{\partial Z_{ij}}  \Tr[(\frac{\bm Z^\top \bm Z}{n+1} + \lambda \bSigma^{-1})^{-1}]}{(1 + \frac{1}{n+1}  \Tr[(\frac{\bm Z^\top \bm Z}{n+1} + \lambda \bSigma^{-1})^{-1}])^3}.
    \end{equation}
    Similarly,
    \begin{align}
        &\frac{\partial}{\partial Z_{ij}} \Tr\biggl[\biggl(\frac{\bm Z^\top \bm Z}{n+1} + \lambda \bSigma^{-1}\biggr)^{-2}\biggr]\\
        &=- \frac{1}{n+1}\Tr\biggl[\biggl(\frac{\bm Z^\top \bm Z}{n+1} + \lambda \bSigma^{-1}\biggr)^{-2} (e_j \bm Z_i^\top + \bm Z_i e_j^\top) \biggl(\frac{\bm Z^\top \bm Z}{n+1} + \lambda \bSigma^{-1}\biggr)^{-1}\biggr] \\
        &\qquad -\frac{1}{n+1}\Tr\biggl[\biggl(\frac{\bm Z^\top \bm Z}{n+1} + \lambda \bSigma^{-1}\biggr)^{-1} (e_j \bm Z_i^\top + \bm Z_i e_j^\top) \biggl(\frac{\bm Z^\top \bm Z}{n+1} + \lambda \bSigma^{-1}\biggr)^{-2}\biggr] \\
        &= -\frac{4}{n+1} \bm Z_i^\top \biggl(\frac{\bm Z^\top \bm Z}{n+1} + \lambda \bSigma^{-1}\biggr)^{-3} e_j
    \end{align}
    By the product rule, we have
    \begin{align}
        &\frac{\partial}{\partial Z_{ij}} \frac{\Tr[(\frac{\bm Z^\top \bm Z}{n+1} + \lambda \bSigma^{-1})^{-2}] }{(1 + \frac{1}{n+1} \Tr \bm [(\frac{\bm Z^\top \bm Z}{n+1} + \lambda \bSigma^{-1})^{-1}])^2}  \\
        &= \frac{4}{n+1} \cdot \frac{1}{(1 + \frac{1}{n+1} \Tr \bm [(\frac{\bm Z^\top \bm Z}{n+1} + \lambda \bSigma^{-1})^{-1}])^2} \\
        &\qquad \cdot\biggl(\frac{\bm Z_i^\top (\frac{\bm Z^\top \bm Z}{n+1} + \lambda \bSigma^{-1} )^{-2} e_j}{1 + \frac{1}{n+1} \Tr \bm [(\frac{\bm Z^\top \bm Z}{n+1} + \lambda \bSigma^{-1})^{-1}]}-\bm Z_i^\top \biggl(\frac{\bm Z^\top \bm Z}{n+1} + \lambda \bSigma^{-1} \biggr)^{-3} e_j \biggr)
    \end{align}
    from which it follows that
    \begin{align}
        &\biggl\|\nabla_{\bm Z_i} \frac{\Tr[(\frac{\bm Z^\top \bm Z}{n+1} + \lambda \bSigma^{-1})^{-2}] }{(1 + \frac{1}{n+1} \Tr \bm [(\frac{\bm Z^\top \bm Z}{n+1} + \lambda \bSigma^{-1})^{-1}])^2}\biggr\|_2 \\
        &= \frac{4}{n+1} \cdot \frac{1}{(1 + \frac{1}{n+1} \Tr \bm [(\frac{\bm Z^\top \bm Z}{n+1} + \lambda \bSigma^{-1})^{-1}])^2} \\
        &\qquad \cdot \biggl|\frac{(\frac{\bm Z^\top \bm Z}{n+1} + \lambda \bSigma^{-1} )^{-2}}{1 + \frac{1}{n+1} \Tr \bm [(\frac{\bm Z^\top \bm Z}{n+1} + \lambda \bSigma^{-1})^{-1}]}-  \biggl(\frac{\bm Z^\top \bm Z}{n+1} + \lambda \bSigma^{-1} \biggr)^{-3} \biggr| \cdot \|  \bm Z_i\|_2 \\
        &\lesssim \frac{1}{n+1} \cdot \biggl(\frac{\|\bSigma\|_{\text{op}}^2}{\lambda^2} + \frac{\|\bSigma\|_{\text{op}}^3}{\lambda^3} \biggr) \|\bm Z_i\|_2.
    \end{align}
    Therefore, by the Gaussian Poincar\'e inequality, we have
    \begin{align}
        \Var\biggl(\|\bSigma^{1/2}& \tilde \bbeta\|_2^2 \cdot  \frac{\frac{1}{n+1}\Tr[(\frac{\bm Z^\top \bm Z}{n+1} + \lambda \bSigma^{-1})^{-2}] }{(1 + \frac{1}{n+1} \Tr \bm [(\frac{\bm Z^\top \bm Z}{n+1} + \lambda \bSigma^{-1})^{-1}])^2}  \biggr) \\
        &= \frac{\|\bSigma^{1/2} \tilde \bbeta\|_2^4 }{(n+1)^2}\Var\biggl(\frac{\Tr[(\frac{\bm Z^\top \bm Z}{n+1} + \lambda \bSigma^{-1})^{-2}] }{(1 + \frac{1}{n+1} \Tr \bm [(\frac{\bm Z^\top \bm Z}{n+1} + \lambda \bSigma^{-1})^{-1}])^2}  \biggr) \\
        &\leq \frac{\|\bSigma^{1/2} \tilde \bbeta\|_2^4 }{(n+1)^2} \sum_{i=1}^n \E\biggl[\biggl\|\nabla_{\bm Z_i}\frac{\Tr[(\frac{\bm Z^\top \bm Z}{n+1} + \lambda \bSigma^{-1})^{-2}] }{(1 + \frac{1}{n+1} \Tr \bm [(\frac{\bm Z^\top \bm Z}{n+1} + \lambda \bSigma^{-1})^{-1}])^2}  \biggr\|_2^2\biggr] \\
        &\lesssim \frac{\|\bSigma^{1/2} \tilde \bbeta\|_2^4 }{(n+1)^4} \biggl(\frac{\|\bSigma\|_{\text{op}}^4}{\lambda^4} + \frac{\|\bSigma\|_{\text{op}}^6}{\lambda^6} \biggr) \sum_{i=1}^n \E[ \|\bm Z_i\|_2^2] \\
        &\lesssim \frac{\|\bSigma^{1/2} \tilde \bbeta\|_2^4 }{(n+1)^2} \biggl(\frac{\|\bSigma\|_{\text{op}}^4}{\lambda^4} + \frac{\|\bSigma\|_{\text{op}}^6}{\lambda^6} \biggr)
    \end{align}
    which concludes the claim.
\end{proof}
Now, we can now apply results from random matrix theory to analyze $B_2^{P,\star}(\hat \bbeta_\lambda; \bbeta^{(2)})$. 
\begin{lemma}\label{lm:b2p_local}
    Define the quantities
    \begin{align}
        &\mathcal B_2^{P,\star,N}(\lambda;\hat H_n, \hat G_n^{\tilde \bbeta}, \gamma) :=  \lambda^2 \gamma ( 1 + \gamma m_{n,1}(\lambda)) \int \frac{s^2}{(\lambda + (1 - \gamma + \gamma \lambda m_n(-\lambda)s)^2} d\hat H_n(s) \\
        &\mathcal B_2^{P,\star,D}(\lambda;\hat H_n, \hat G_n^{\tilde \bbeta}, \gamma) :=\lambda + \gamma \int \frac{\lambda s}{\lambda + (1 - \gamma + \gamma \lambda m_n(-\lambda)) s} d\hat H_n(s) \\
        &\mathcal B_2^{P,\star}(\lambda;\hat H_n, \hat G_n^{\tilde \bbeta}, \gamma) := \|\tilde \bbeta\|_2^2 \frac{\mathcal{B}_2^{P,\star,N}}{(\mathcal B_2^{P,\star,D})^2} \int s\; d\hat G_{\tilde \bbeta,n}(s).
    \end{align}
    Then, for any small constant $c > 0$, we have
    \begin{equation}
        |B_2^{P,\star}(\hat \bbeta_\lambda; \bbeta^{(2)}) - \mathcal B_2^{P,\star}(\lambda; \hat H_n, \hat G_n^{\tilde \bbeta}, \gamma)| = p(n^{-1/2+c} \lambda^{-5})
    \end{equation}
    with high probability over $(\bZ^{(1)}, \bZ^{(2)})$.
\end{lemma}
\begin{proof}
    It is clear that (e.g., via a simple Sherman-Morrison argument) 
    \begin{equation*}
        \|\bSigma^{1/2} \tilde \bbeta\|_2^2 \cdot  \frac{\frac{1}{n}\Tr[\bm A^2] }{(1 + \frac{1}{n} \Tr \bm A)^2} \quad \text{and} \quad   \|\bSigma^{1/2} \tilde \bbeta\|_2^2 \cdot  \frac{\frac{1}{n}\Tr[\bSigma (\hat \bSigma + \lambda\bm I)^{-1} \bSigma (\hat \bSigma + \lambda \bm I)^{-1}] }{(1 + \frac{1}{n} \Tr[\bSigma (\hat \bSigma + \lambda \bm I)^{-1}])^2}
    \end{equation*}
    have the same asymptotic behavior.
    We can bound
    \begin{align*}
        &|B_2^{P,\star}(\lambda) - \mathcal B_2^{P,\star}(\lambda)| \\
        &\leq \biggl|\|\bSigma^{1/2} \tilde \bbeta\|_2^2 \cdot  \frac{\frac{\lambda^2}{n}\Tr[\bSigma (\hat \bSigma + \lambda\bm I)^{-1} \bSigma (\hat \bSigma + \lambda \bm I)^{-1}] }{(\lambda + \frac{\lambda}{n} \Tr[\bSigma (\hat \bSigma + \lambda \bm I)^{-1}])^2} \\
        &\hspace{6em}- \|\bSigma^{1/2} \tilde \bbeta\|_2^2 \cdot  \frac{\frac{\lambda^2}{n}\Tr[\bSigma (\hat \bSigma + \lambda\bm I)^{-1} \bSigma (\hat \bSigma + \lambda \bm I)^{-1}] }{(\mathcal B_2^{P,\star,D}(\lambda))^2}\biggr| \\
        &\qquad+\biggl|\|\bSigma^{1/2} \tilde \bbeta\|_2^2 \cdot  \frac{\frac{\lambda^2}{n}\Tr[\bSigma (\hat \bSigma + \lambda\bm I)^{-1} \bSigma (\hat \bSigma + \lambda \bm I)^{-1}] }{(\mathcal B_2^{P,\star,D}(\lambda))^2} - \|\bSigma^{1/2} \tilde \bbeta\|_2^2 \cdot  \frac{\mathcal B_2^{P,\star,N}(\lambda)}{(\mathcal B_2^{P,\star,D}(\lambda))^2}\biggr| \\
        &\leq \|\bSigma^{1/2} \tilde \bbeta\|_2^2 \cdot \frac{\lambda^2}{n}\Tr[\bSigma (\hat \bSigma + \lambda\bm I)^{-1} \bSigma (\hat \bSigma + \lambda \bm I)^{-1}]  \\
        &\hspace{6em} \cdot \biggl|  \frac{1}{(\lambda + \frac{\lambda}{n} \Tr[\bSigma (\hat \bSigma + \lambda \bm I)^{-1}])^2} - \frac{1}{(\mathcal B_2^{P,\star,D}(\lambda))^2}\biggr| \\
        &\qquad+ \frac{\|\bSigma^{1/2} \tilde \bbeta\|_2^2}{(\mathcal B_2^{P,\star,D}(\lambda))^2} \biggl|\frac{\lambda^2}{n}\Tr[\bSigma (\hat \bSigma + \lambda\bm I)^{-1} \bSigma (\hat \bSigma + \lambda \bm I)^{-1}] - \mathcal B_2^{P,\star,N}(\lambda)\biggr| \\
        &\leq \|\bSigma^{1/2} \tilde \bbeta\|_2^2 \|\bSigma\|_{\text{op}}^3 \lambda^{-4}  \cdot \biggl|   \lambda + \frac{\lambda}{n} \Tr[\bSigma (\hat \bSigma + \lambda \bm I)^{-1}] - \mathcal B_2^{P,\star,D}(\lambda)\biggr| \\
        &\qquad+ \|\bSigma^{1/2} \tilde \bbeta\|_2^2 \lambda^{-2} \biggl|\frac{\lambda^2}{n}\Tr[\bSigma (\hat \bSigma + \lambda\bm I)^{-1} \bSigma (\hat \bSigma + \lambda \bm I)^{-1}] - \mathcal B_2^{P,\star,N}(\lambda)\biggr| \\
        &= O_\prec\biggl(\frac{\|\tilde \bbeta\|_2^2 \|\bSigma\|_{\text{op}}^4}{n^{1/2} \lambda^5} + \frac{\|\tilde \bbeta\|_2^2 \|\bSigma\|_{\text{op}}}{n \lambda^3}\biggr)
    \end{align*}
    where the last step follows from \cite[Theorem 3.16]{knowles2017anisotropic}.
\end{proof}
We can now complete our analysis of $B_2^P$. 
\begin{lemma}\label{lm:b2p_full}
    In the setting of Theorem \ref{thm:model_shift_ridge}, for any small constant $c > 0$, we have
    \begin{equation}
        |\E[B_2^P(\hat \bbeta_\lambda; \bbeta^{(2)})] - \mathcal B_2^{P,\star}(\lambda; \hat H_n, \hat G_n^{\tilde \bbeta},\gamma)| = O(n^{-1/2+c} \lambda^{-5})
    \end{equation}
    with high probability over $(\bZ^{(1)}, \bZ^{(2)})$.
\end{lemma}
\begin{proof}
    The claim follows immediately by combining Lemmas \ref{lm:b2p_conc} and \ref{lm:b2p_local}.
\end{proof}
Combining Lemmas \ref{lm:b2_full_efron_stein},
\ref{lm:b2f_full},
and \ref{lm:b2p_full} yields \eqref{eqn:model_shift_ridge_B2}.

\subsubsection{Proof of Eq. \ref{eqn:model_shift_ridge_B3}}
Our proof of \eqref{eqn:model_shift_ridge_B3} follows a similar strategy. Define
\begin{equation}
    B_3^F(\hat \bbeta_\lambda; \bbeta^{(2)}) := \lambda \bbeta^{(2)\top} \bSigma^{-1/2} (\hat \bSigma_Z + \lambda \bSigma^{-1})^{-2} \hat \bSigma_Z \bSigma^{1/2} \tilde \bbeta.
\end{equation}
By exchangeability of the rows of $\bm Z$ and symmetry, we have
\begin{align}\label{eq:symmetry_B3}
    \E[B_3(\hat \bbeta_\lambda; \bbeta^{(2)})] &= \E\biggl[\lambda \bbeta^{(2)\top} \bSigma^{-1/2} (\hat \bSigma_Z + \lambda \bSigma^{-1})^{-2}\biggl(\frac{\bm Z^{(1)\top} \bm Z^{(1)}}{n} \biggr) \bSigma^{1/2} \tilde \bbeta  \biggr] \\
    &= \E\biggl[\lambda \bbeta^{(2)\top} \bSigma^{-1/2} (\hat \bSigma_Z + \lambda \bSigma^{-1})^{-2} \biggl(\frac1n \sum_{i \in I} \bm Z_i \bm Z_i^{\top}\biggr) \bSigma^{1/2} \tilde \bbeta \biggr] \\
    &= \frac{n_1}{n}\E[ B_3^F(\hat \bbeta_\lambda; \bbeta^{(2)})].
\end{align}
We similarly prove a concentration result that implies that it suffices to analyze $B_3^F(\hat \bbeta_\lambda; \bbeta^{(2)})$:
\begin{lemma}\label{lm:b3_conc}
    In the setting of Theorem \ref{thm:model_shift_ridge}, for any small constant $c > 0$, we  have
    \begin{equation}
        \biggl| B_3(\hat \bbeta_\lambda; \bbeta^{(2)}) - \frac{n_1}{n} B_3^F(\hat \bbeta_\lambda; \bbeta^{(2)}) \biggr| = O(\lambda^{-3} n^{-1/2+c})
    \end{equation}
    with high probability over $(\bZ^{(1)}, \bZ^{(2)})$.
\end{lemma}
\begin{proof}
    With \eqref{eq:symmetry_B3} in mind, it suffices to show that
    \begin{align}
        |B_3(\hat \bbeta_\lambda;\bbeta^{(2)}) - \E[B_3(\hat \bbeta_\lambda;\bbeta^{(2)}]| &= O(\lambda^{-3} n^{-1/2+c}) \label{eq:b3_conc} \\
        |B_3^F(\hat \bbeta_\lambda;\bbeta^{(2)}) - \E[B_3^F(\hat \bbeta_\lambda;\bbeta^{(2)}]| &= O(\lambda^{-3} n^{-1/2+c}) \label{eq:b3f_conc}
    \end{align}
    with high probability.
    We begin with \eqref{eq:b3_conc}.

    We proceed by an analogous Efron-Stein argument as in Lemma \ref{lm:b2_full_efron_stein}: that is,
    \begin{align}
        &\Var(B_3(\hat \bbeta_\lambda; \bbeta^{(2)})) \\
        &\leq n_1 \E\biggl[\biggl(\lambda \bbeta^{(2)\top} \bSigma^{-1/2} \biggl((\hat \bSigma_Z + \lambda \bSigma^{-1})^{-2} \biggl(\frac{\bm Z^{(1)\top} \bm Z^{(1)}}{n}\biggr) \\
        &\hspace{11em} - (\hat \bSigma_Z' + \lambda \bSigma^{-1})^{-2} \biggl(\frac{\bm Z^{(1)'\top} \bm Z^{(1)'}}{n}\biggr)\biggr) \bSigma^{1/2} \tilde \bbeta\biggr)^2\biggr] \\
        &\qquad + n_2 \E\biggl[\biggl(\lambda \bbeta^{(2)\top} \bSigma^{-1/2} ((\hat \bSigma_Z + \lambda \bSigma^{-1})^{-2}  - (\hat \bSigma_Z'' + \lambda \bSigma^{-1})^{-2} ) \\
        &\hspace{6em}\biggl(\frac{\bm Z^{(1)\top} \bm Z^{(1)}}{n}\biggr) \bSigma^{1/2} \tilde \bbeta\biggr)^2\biggr]
    \end{align}
    Letting $\bm u = \bSigma^{-1/2} \bbeta^{(2)}$, $\bm v = \bSigma^{1/2} \tilde \bbeta$, $\bm A = (\hat \bSigma_Z^- + \lambda \bSigma^{-1})^{-1}$, and $\bm B = \frac{1}{n} \sum_{i=2}^{n_1} \bm Z_i^{(1)} \bm Z_i^{(1)\top}$, the Sherman-Morrison formula yields
    \begin{align}
        &\E\biggl[\biggl(\lambda \bbeta^{(2)\top} \bSigma^{-1/2} \biggl((\hat \bSigma_Z + \lambda \bSigma^{-1})^{-2} \biggl(\frac{\bm Z^{(1)\top} \bm Z^{(1)}}{n}\biggr) \\
        &\hspace{8em}- (\hat \bSigma_Z' + \lambda \bSigma^{-1})^{-2} \biggl(\frac{\bm Z^{(1)'\top} \bm Z^{(1)'}}{n}\biggr)\biggr) \bSigma^{1/2} \tilde \bbeta\biggr)^2\biggr] \\
        &\lcon \frac{\lambda^2}{n^2} \E[(\bm u^\top \bm A^2 \bm Z_1^{(1)} \bm Z_1^{(1)\top} \bm A \bm B \bm v)^2] + \frac{\lambda^2}{n^2} \E[(\bm u^\top \bm A \bm Z_1^{(1)} \bm Z^{(1)\top} \bm A^2 \bm B \bm v)^2] \\
        &\qquad + \frac{\lambda^2}{n^4} \E[(\bm u^\top \bm A \bm Z_1^{(1)} \bm Z_1^{(1)\top} \bm A^2 \bm Z_1^{(1)} \bm Z_1^{(1)\top} \bm A \bm B \bm v)^2] \\
        &\qquad + \frac{\lambda^2}{n^2} \E[(\bm u^\top \bm A^2 \bm Z_1^{(1)}\bm Z_1^{(1)\top} \bm v)^2] + \frac{\lambda^2}{n^4} \E[(\bm u^\top \bm A^2 \bm Z_1^{(1)} \bm Z_1^{(1)\top} \bm A \bm Z_1^{(1)} \bm Z_1^{(1)\top} \bm v)^2] \\
        &\qquad + \frac{\lambda^2}{n^4} \E[(\bm u^\top \bm A \bm Z_1^{(1)} \bm Z^{(1)\top} \bm A^2 \bm Z_1^{(1)} \bm Z_1^{(1)\top} \bm v)^2] \\
        &\qquad + \frac{\lambda^2}{n^6} \E[(\bm u^\top \bm A \bm Z_1^{(1)} \bm Z_1^{(1)\top} \bm A^2 \bm Z_1^{(1)} \bm Z_1^{(1)\top} \bm A \bm Z_1^{(1)} \bm Z_1^{(1)\top} \bm v)^2] \\
        &\lcon \frac{\|\bSigma\|_{\text{op}}^6\|\bm u\|_2^2 \|\bm v\|_2^2}{\lambda^4 n^2} + \frac{\|\bSigma\|_{\text{op}}^8\|\bm u\|_2^2 \|\bm v\|_2^2 p^2}{\lambda^6 n^4} + \frac{\|\bSigma\|_{\text{op}}^4\|\bm u\|_2^2 \|\bm v\|_2^2}{\lambda^2 n^2}  \\
        &\qquad + \frac{\|\bSigma\|_{\text{op}}^6\|\bm u\|_2^2 \|\bm v\|_2^2 p^2}{\lambda^4 n^4} + \frac{\|\bSigma\|_{\text{op}}^8\|\bm u\|_2^2 \|\bm v\|_2^2 p^4}{\lambda^6 n^6}
    \end{align}
    where the last step follows from repeatedly applying H\"older's inequality and Lemma \ref{lm:basic_S_moments}.
    
    Letting $\bm u = \bSigma^{-1/2} \bbeta^{(2)}$, $\bm v = \bSigma^{1/2} \tilde \bbeta$, $\bm A = (\hat \bSigma^{--} + \lambda \bSigma^{-1})^{-1}$, and $\bm B = \frac{\bm Z^{(1)\top} \bm Z^{(1)}}{n}$, we can similarly bound
    \begin{align}
        &\E\biggl[\biggl(\lambda \bbeta^{(2)\top} \bSigma^{-1/2} ((\hat \bSigma_Z + \lambda \bSigma^{-1})^{-2}  - (\hat \bSigma_Z'' + \lambda \bSigma^{-1})^{-2} ) \biggl(\frac{\bm Z^{(1)\top} \bm Z^{(1)}}{n}\biggr) \bSigma^{1/2} \tilde \bbeta\biggr)^2\biggr] \\
        &\lcon \frac{\lambda^2}{n^2} \E[(\bm u^\top \bm A^2 \bm Z_1^{(1)} \bm Z_1^{(1)\top} \bm A \bm B \bm v)^2] + \frac{\lambda^2}{n^2} \E[(\bm u^\top \bm A \bm Z_1^{(1)} \bm Z^{(1)\top} \bm A^2 \bm B \bm v)^2] \\
        &\qquad + \frac{\lambda^2}{n^4} \E[(\bm u^\top \bm A \bm Z_1^{(1)} \bm Z_1^{(1)\top} \bm A^2 \bm Z_1^{(1)} \bm Z_1^{(1)\top} \bm A \bm B \bm v)^2] \\
        &\lcon \frac{\|\bSigma\|_{\text{op}}^6 \|\bm u\|_2^2 \|\bm v\|_2^2}{\lambda^4 n^2} + \frac{\|\bSigma\|_{\text{op}}^8\|\bm u\|_2^2 \|\bm v\|_2^2 p^2}{\lambda^6 n^4}.
    \end{align}
    Therefore, $\Var(B_3(\hat \bbeta_\lambda; \bbeta^{(2)})) = O(\lambda^{-6} n^{-1})$, from which  follows. \eqref{eq:b3f_conc} follows by the same process.
\end{proof}
As with Lemma \ref{lm:b2p_local}, we can now apply existing anisotropic local law results to analyze $B_3^F$.
\begin{lemma}\label{lm:b3f_local}
    In the setting of Theorem \ref{thm:model_shift_ridge}, we have
    \begin{equation}
        \biggl|B_3^F(\hat \bbeta_\lambda; \bbeta^{(2)}) - \frac{n}{n_1} \mathcal B_3(\lambda; \hat H_n, \hat G_n^{(b)}, \gamma)\biggr| =O_\prec(n^{-1/2 }\lambda^{-1})
    \end{equation}
\end{lemma}
\begin{proof}
    Note that
    \begin{align}
        B_3^F &= \lambda \bbeta^{(2)\top} \bSigma^{-1/2}(\hat \bSigma_Z + \lambda \bSigma^{-1})^{-2} \hat \bSigma_Z \bSigma^{1/2} \tilde \bbeta \\
        &= \lambda \bbeta^{(2)\top} (\hat \bSigma + \lambda \bm I)^{-1} \bSigma \tilde \bbeta - \lambda^2 \bbeta^{(2)\top} (\hat \bSigma + \lambda \bm I)^{-1} \bSigma (\hat \bSigma + \lambda \bm I)^{-1} \tilde \bbeta.
    \end{align}
    Now, by \cite[Theorem 3.16]{knowles2017anisotropic} and \cite[Section A.2]{hastie2022surprises}, we have
    \begin{align}
        &\biggl|\lambda \bbeta^{(2)\top} (\hat \bSigma + \lambda \bm I)^{-1} \bSigma\tilde \bbeta \\
        &\qquad - \|\bbeta^{(2)}\|_2 \|\tilde \bbeta\|_2 \int \frac{\lambda s}{\lambda + (1 - \gamma + \gamma \lambda m_n(-\lambda))s} d\hat G^{(b)}(s)\biggr| = O_\prec(n^{-1/2} \lambda^{-1})\\
        &\biggl| \lambda^2 \bbeta^{(2)\top} (\hat \bSigma + \lambda \bm I)^{-1} \bSigma (\hat \bSigma + \lambda \bm I)^{-1} \tilde \bbeta \\
        &\qquad -  \|\bbeta^{(2)}\|_2 \|\tilde \bbeta\|_2 \int \frac{\lambda^2 (1 + \gamma m_{n,1}(-\lambda))s}{(\lambda + (1 - \gamma + \gamma \lambda m_n(-\lambda))s)^2} d\hat G^{(b)}(s)\biggr| = O_\prec(\lambda^{-1} n^{-1})
    \end{align}
    The result follows immediately.
\end{proof}
Combining Lemmas \ref{lm:b3_conc} and \ref{lm:b3f_local} yields \eqref{eqn:model_shift_ridge_B3}.

\subsection{Proof of Theorem \ref{thm:model_shift}}\label{section:theorem_model_shift_final}
We only prove \eqref{eqn:model_shift_B2} and \eqref{eqn:model_shift_B3}, as \eqref{eqn:model_shift_V} and \eqref{eqn:model_shift_B1} follow immediately from \cite[Theorem 2]{hastie2022surprises}.  

\subsubsection{Finite-sample risk difference}\label{subsubsec:finite_sample_model_diff}
In this section, we provide high-probability bounds for the difference in the finite-sample risk terms for the pooled ridge estimator $\hat \bbeta_\lambda$ and for the pooled min-$\ell_2$-norm interpolator $\hat \bbeta$, as a function of the ridge parameter $\lambda$. Specifically, we control the differences $|B_2(\hat \bbeta_\lambda; \bbeta^{(2)}) - B_2(\hat \bbeta; \bbeta^{(2)})|$ and $|B_3(\hat \bbeta_\lambda; \bbeta^{(2)}) - B_3(\hat \bbeta; \bbeta^{(2)})|$. 

Throughout this section, we denote by $\bm X / \sqrt{n} = \bm U \bm D \bm V^\top$ the full singular value decomposition of $\bm X / \sqrt{n}$, where $\bm U \in \R^{n \times n}$ and $\bm V \in \R^{p \times p}$ are orthogonal matrices. Let $d_1 \le \cdots \le d_n$ denote the singular values, and let $d_{\text{nz}}$ denote the smallest \textit{nonzero} singular value.

We first establish a few useful bounds.
\begin{lemma}\label{lm:pseudoinverse_ridge_diff_bound}
    With $d_{\mathrm{nz}}$ as defined above, we have
    \begin{align*}
        \biggl\|((\hat \bSigma + \lambda \bm I)^{-1} - \hat \bSigma^\dagger) \frac{\bm X^{(1)\top}}{\sqrt n}\biggr\|_{\mathrm{op}} &\le  \frac{\lambda}{d_{\mathrm{nz}}^3}, \\
        \|\lambda (\hat \bSigma + \lambda\bm I)^{-1} - (\bm I - \hat \bSigma^\dagger \bSigma)) \|_{\mathrm{op}} &\le \frac{\lambda}{d_{\mathrm{nz}}^2}.
    \end{align*}
\end{lemma}
\begin{proof}
    Note that $\bm X^{(1)\top}$ is a submatrix of $\bm X^\top$ and thus
    \begin{align*}
        \biggl\|((\hat \bSigma + \lambda \bm I)^{-1} - \hat \bSigma^\dagger) \frac{\bm X^{(1)\top}}{\sqrt n}\biggr\|_{\mathrm{op}} &\le \biggl\|((\hat \bSigma + \lambda \bm I)^{-1} - \hat \bSigma^\dagger) \frac{\bm X^\top}{\sqrt n}\biggr\|_{\mathrm{op}} \\
        &= \|((\bm D^\top \bm D + \lambda \bm I)^{-1} - (\bm D^\top \bm D)^\dagger) \bm D^\top \|_{\text{op}}  \\
        &= \max_{d_i > 0} \frac{\lambda}{d_i (d_i^2 + \lambda)} \\
        &\le \frac{\lambda}{d_{\text{nz}}^3}
    \end{align*}
    which establishes the first claim. Next, note that
    \begin{align*}
        &\|\lambda (\hat \bSigma + \lambda\bm I)^{-1} - (\bm I - \hat \bSigma^\dagger \bSigma)\|_{\mathrm{op}}  \\
        &= \| \lambda (\bm D^\top \bm D + \lambda \bm I)^{-1} - (\bm I - (\bm D^\top \bm D)^\dagger (\bm D^\top \bm D))\|_{\text{op}} \\
        &= \max_{d_i > 0} \frac{\lambda}{d_i^2 + \lambda} \\
        &\le \frac{\lambda}{d_{\text{nz}}^2}
    \end{align*}
    which establishes the second claim.
\end{proof}

\begin{lemma}\label{lm:pseudoinverse_ridge_indiv_bound}
    With $d_{\mathrm{nz}}$ as defined above, we have
    \begin{align*}
        \biggl\|(\hat \bSigma + \lambda \bm I)^{-1} \frac{\bX^{(1)\top}}{\sqrt{n}}\biggr\|_{\mathrm{op}} \le 1/d_{\mathrm{nz}}, \quad \biggl\|\hat \bSigma^\dagger \frac{\bX^{(1)\top}}{\sqrt{n}}\biggr\|_{\mathrm{op}} \le 1/d_{\mathrm{nz}}.\\
    \end{align*}
\end{lemma}
\begin{proof}
    Like above, we have
    \begin{align*}
        \biggl\|(\hat \bSigma + \lambda \bm I)^{-1} \frac{\bX^{(1)\top}}{\sqrt{n}}\biggr\|_{\mathrm{op}} &\le \biggl\|(\hat \bSigma + \lambda \bm I)^{-1} \frac{\bX^\top}{\sqrt{n}}\biggr\|_{\mathrm{op}} = \|(\bm D^\top \bm D + \lambda \bm I)^{-1}  \bm D^\top \|_{\text{op}}  \\
        &= \max_{d_i > 0} \frac{d_i}{d_i^2 + \lambda} \le \max_{d_i > 0} \frac{1}{d_i} \le 1/d_{\text{nz}}
    \end{align*}
    which yields the first claim. Similarly, 
    \begin{align*}
        \biggl\|\hat \bSigma^\dagger \frac{\bX^{(1)\top}}{\sqrt{n}}\biggr\|_{\mathrm{op}} &\le \biggl\|\hat \bSigma^\dagger \frac{\bX^{\top}}{\sqrt{n}}\biggr\|_{\mathrm{op}} = \|(\bm D^\top \bm D )^\dagger  \bm D^\top \|_{\text{op}}  = \max_{d_i > 0} \frac{1}{d_i} = 1 / d_{\text{nz}}
    \end{align*}
    which yields the second claim.
\end{proof}

Now, we control the difference $|B_2(\hat \bbeta_\lambda; \bbeta^{(2)}) - B_2(\hat \bbeta; \bbeta^{(2)})|$.
\begin{lemma}\label{lm:b2_ridge_limit_bound}
    We have
    \begin{equation*}
        |B_2(\hat \bbeta_\lambda; \bbeta^{(2)}) - B_2(\hat \bbeta; \bbeta^{(2)})| = O_\prec( \lambda \|\tilde \bbeta\|_2^2).
    \end{equation*}
\end{lemma}
\begin{proof}
    Using the triangle inequality, we have
    \begin{align*}
        &|B_2(\hat \bbeta_\lambda; \bbeta^{(2)}) - B_2(\hat \bbeta; \bbeta^{(2)})| \\
        &= \biggl|\tilde \bbeta^\top \biggl(\frac{\bm X^{(1)\top} \bm X^{(1)}}{n} \biggr) (\hat \bSigma + \lambda \bm I)^{-1} \bSigma (\hat \bSigma + \lambda \bm I)^{-1} \biggl(\frac{\bm X^{(1)\top} \bm X^{(1)}}{n} \biggr)  \tilde \bbeta \\
        &\qquad - \tilde \bbeta^\top \biggl(\frac{\bm X^{(1)\top} \bm X^{(1)}}{n} \biggr) \hat \bSigma^\dagger \bSigma \hat \bSigma^\dagger \biggl(\frac{\bm X^{(1)\top} \bm X^{(1)}}{n} \biggr)  \tilde \bbeta \biggr| \\
        &\leq \biggl|\tilde \bbeta^\top \biggl(\frac{\bm X^{(1)\top} \bm X^{(1)}}{n} \biggr) (\hat \bSigma + \lambda \bm I)^{-1} \bSigma ((\hat \bSigma + \lambda \bm I)^{-1} - \hat \bSigma^\dagger) \biggl(\frac{\bm X^{(1)\top} \bm X^{(1)}}{n} \biggr)  \tilde \bbeta  \biggr| \\
        &\qquad + \biggl| \tilde \bbeta^\top \biggl(\frac{\bm X^{(1)\top} \bm X^{(1)}}{n} \biggr) ((\hat \bSigma + \lambda \bm I)^{-1} - \hat \bSigma^\dagger) \bSigma \hat \bSigma^\dagger \biggl(\frac{\bm X^{(1)\top} \bm X^{(1)}}{n} \biggr)  \tilde \bbeta \biggr| \\
        &\leq \|\tilde \bbeta\|_2^2 \cdot \|\bSigma\|_{\text{op}}  \cdot \biggl\| \frac{\bm X^{(1)}}{\sqrt{n}}\biggr\|_{\text{op}}^2 \cdot \biggr\|((\hat \bSigma + \lambda \bm I)^{-1} - \hat \bSigma^\dagger) \frac{\bm X^{(1)\top}}{\sqrt n}  \biggl\|_{\text{op}}  \\
        &\qquad \cdot \biggl(\biggl\|(\hat \bSigma + \lambda \bm I)^{-1} \frac{\bm X^{(1)}}{\sqrt n}\biggr\|_{\text{op}} + \biggl\|\hat \bSigma^\dagger \frac{\bm X^{(1)}}{\sqrt n}\biggr\|_{\text{op}}  \biggr) \\
        &\lesssim \|\tilde \bbeta\|_2^2 \cdot \|\bSigma\|_{\text{op}} \cdot  \frac{\lambda}{d_{\text{nz}}^4}\\
        &= O_\prec(\lambda \|\tilde \bbeta\|_2^2 \cdot \|\bSigma\|_{\text{op}})
    \end{align*}
    where the last two lines follow from Lemmas \ref{lemma:ESD_bounded}, \ref{lm:pseudoinverse_ridge_diff_bound}, and \ref{lm:pseudoinverse_ridge_indiv_bound}. 
\end{proof}
We can then similarly control the difference $|B_3(\hat \bbeta_\lambda; \bbeta^{(2)}) - B_3(\hat \bbeta; \bbeta^{(2)})|$.
\begin{lemma}\label{lm:b3_ridge_limit_bound}
    We have
    \begin{equation*}
        |B_3(\hat \bbeta_\lambda; \bbeta^{(2)}) - B_3(\hat \bbeta; \bbeta^{(2)})| = O_\prec( \lambda \|\bbeta^{(2)}\|_2 \|\tilde \bbeta\|_2).
    \end{equation*}
\end{lemma}
\begin{proof}
    Using the triangle inequality, we have
    \begin{align*}
        &\frac12 |B_3(\hat \bbeta_\lambda; \bbeta^{(2)}) - B_3(\hat \bbeta; \bbeta^{(2)})| \\
        &= \biggl|\lambda \bbeta^{(2)\top} (\hat \bSigma + \lambda \bm I)^{-1} \bSigma (\hat \bSigma + \lambda \bm I)^{-1} \biggl(\frac{\bm X^{(1)\top} \bm X^{(1)}}{n} \biggr)  \tilde \bbeta \\
        &\qquad - \bbeta^{(2)\top} (\bm I - \hat \bSigma^\dagger \hat \bSigma) \bSigma \hat \bSigma^\dagger \biggl(\frac{\bm X^{(1)\top} \bm X^{(1)}}{n} \biggr)  \tilde \bbeta \biggr| \\
        &\leq \biggl|\bbeta^{(2)\top} (\lambda (\hat \bSigma + \lambda \bm I)^{-1} - (\bm I - \hat \bSigma^\dagger \hat \bSigma))   \bSigma (\hat \bSigma + \lambda \bm I)^{-1}  \biggl(\frac{\bm X^{(1)\top} \bm X^{(1)}}{n} \biggr)  \tilde \bbeta  \biggr| \\
        &\qquad + \biggl| \bbeta^{(2)\top} (\bm I - \hat \bSigma^\dagger \hat \bSigma) \bSigma ((\hat \bSigma + \lambda \bm I)^{-1} - \hat \bSigma^\dagger) \biggl(\frac{\bm X^{(1)\top} \bm X^{(1)}}{n} \biggr)  \tilde \bbeta \biggr| \\
        &\leq \|\bbeta^{(2)}\|_2 \|\tilde \bbeta\|_2 \|\bSigma\|_{\text{op}} \cdot \|\lambda (\hat \bSigma + \lambda \bm I)^{-1} - (\bm I - \hat \bSigma^\dagger \hat \bSigma)\|_{\text{op}} \\
        &\hspace{4em}\cdot \biggl\|(\hat \bSigma + \lambda \bm I)^{-1} \frac{\bm X^{(1)\top}}{\sqrt{n}}\biggr\|_{\text{op}} \cdot \biggl\| \frac{\bm X^{(1)}}{\sqrt{n}}\biggr\|_{\text{op}} \\
        &\qquad + \|\bbeta^{(2)}\|_2 \|\tilde \bbeta\|_2 \|\bSigma\|_{\text{op}} \cdot \|\bm I - \hat \bSigma^\dagger \hat \bSigma\|_{\text{op}} \cdot \biggl\|((\hat \bSigma + \lambda \bm I)^{-1} - \hat \bSigma^\dagger) \frac{\bX^{(1)\top}}{\sqrt{n}}\biggr\|_{\text{op}} \biggl\| \frac{\bm X^{(1)}}{\sqrt{n}}\biggr\|_{\text{op}} \\
        &\lesssim \|\bbeta^{(2)}\|_2 \|\tilde \bbeta\|_2 \|\bSigma\|_{\text{op}} \cdot \frac{\lambda}{d_{\mathrm{nz}}^3} \\
        &= O_\prec( \lambda \|\bbeta^{(2)}\|_2 \|\tilde \bbeta\|_2 \|\bSigma\|_{\text{op}})
    \end{align*}
    where again the last two lines follow from Lemmas \ref{lemma:ESD_bounded}, \ref{lm:pseudoinverse_ridge_diff_bound}, and \ref{lm:pseudoinverse_ridge_indiv_bound}. 
\end{proof}

\subsubsection{Asymptotic risk difference}
In this section, we provide numerical bounds for the difference between asymptotic risk terms for the pooled ridge estimator $\hat \bbeta_\lambda$ and for the pooled min-$\ell_2$-norm interpolator $\hat \bbeta$, as a function of the ridge parameter $\lambda$. That is, we control the differences $|\mathcal B_2(\lambda; \hat H_n, \hat G_n^{\tilde \bbeta}) - \mathcal B_2( \hat H_n, \hat G_n^{\tilde \bbeta})|$ and $|\mathcal B_3(\lambda; \hat H_n, \hat G_n^{(b)}) - \mathcal B_3( \hat H_n, \hat G_n^{(b)})|$.

We first recall that, per \cite[Section A.2.1]{hastie2022surprises}, we have that
\begin{align*}
    m_n(-\lambda) &= (1 - \gamma^{-1}) \frac{1}{\lambda} + c_0 + O_*(\lambda) \\
    m_{n,1}(-\lambda) &= c_1 + O_*(\lambda)
\end{align*}
where $O_*(\lambda)$ is a term bounded by $\lambda \cdot C(\tau)$, where $C(\tau)$ is a constant only depending on $\tau$. Importantly, the above bound implies that
\begin{equation*}
    \frac{1 + \gamma + \gamma \lambda m_n(-\lambda)}{\lambda} = \gamma c_0 + O_*(\lambda).
\end{equation*}
Using these facts, we can establish the desired bounds below.

\begin{lemma}
\begin{align*}
    |\mathcal B_2(\lambda; \hat H_n, \hat G_n^{\tilde \bbeta}) - \mathcal B_2( \hat H_n, \hat G_n^{\tilde \bbeta})| = O_*(\lambda \|\tilde \bbeta\|_2^2) + O(n^{-1} \|\tilde \bbeta\|_2^2)
\end{align*}
\end{lemma}
\begin{proof}
    The above bounds yield
    \begin{align*}
        &\biggl|\frac{\gamma s (c_1 + \gamma c_0^2 s^2)}{(1 + c_0 \gamma s)^2} - \frac{s((1 -\gamma + \gamma \lambda m_n(-\lambda))^2 s^2 + \lambda^2 \gamma m_{n,1}(-\lambda))}{(\lambda + (1 - \gamma + \gamma \lambda m_n(-\lambda))s)^2} \biggr| = O_*(\lambda)
    \end{align*}
    uniformly in $s \in \text{supp}(\hat G_n^{\tilde \bbeta})$ and thus, recalling that $\hat G_n^{\tilde \bbeta}(\R) = 1$, we have
    \begin{align*}
        &\biggl|\int \frac{\gamma s (c_1 + \gamma c_0^2 s^2)}{(1 + c_0 \gamma s)^2} d\hat G_n^{\tilde \bbeta}(s) - \int \frac{s((1 -\gamma + \gamma \lambda m_n(-\lambda))^2 s^2 + \lambda^2 \gamma m_{n,1}(-\lambda))}{(\lambda + (1 - \gamma + \gamma \lambda m_n(-\lambda))s)^2} d\hat G^{\tilde \bbeta}_n(s) \biggr| \\
        &= O_*(\lambda).
    \end{align*}
    The same analysis yields
    \begin{align*}
        &\biggl|\frac{\lambda^2 \gamma ( 1 + \gamma m_{n,1}(\lambda)) \int \frac{s^2}{(\lambda + (1 - \gamma + \gamma \lambda m_n(-\lambda)s)^2} d\hat H_n(s)}{(\lambda + \gamma \int \frac{\lambda s}{\lambda + (1 - \gamma + \gamma \lambda m_n(-\lambda)) s} d\hat H_n(s))^2} - \frac{(1 + \gamma c_1) \int \frac{s^2}{(1 + \gamma c_0 s)^2} d\hat H_n(s)}{\gamma (\int \frac{s}{1 + \gamma c_0 s} d\hat H_n(s))^2}\biggr| \\
        &= O_*(\lambda).
    \end{align*}
    We can also simply bound
    \begin{equation*}
        \biggl|\frac{n_1(n_1 - 1)}{n(n-1)} - \frac{n_1^2}{n^2}\biggr| = O(n^{-1}), \quad  \biggl|\frac{n_1 (n - n_1)}{n(n-1)} - \frac{n_1}{n}\biggl(1 - \frac{n_1}{n} \biggr)\biggr| = O(n^{-1})
    \end{equation*}
    after which the conclusion follows by combining these bounds.
\end{proof}
\begin{lemma}
\begin{equation*}
    |\mathcal B_3(\lambda; \hat H_n, \hat G_n^{(b)}) - \mathcal B_3(\hat H_n, \hat G_n^{(b)})| = O_*(\lambda \|\tilde \bbeta\|_2 \|\bbeta^{(2)}\|_2)
\end{equation*}
\end{lemma}
\begin{proof}
Once again, the above bounds yield
\begin{align*}
    \biggl|\frac{\gamma s (c_0 s- c_1)}{(1 +  c_0 \gamma s)^2} - \frac{\lambda s(\lambda + (1 - \gamma + \gamma \lambda m_n(-\lambda))s) + \lambda^2 (1 + \gamma m_{n,1}(-\lambda))s}{(\lambda + (1 - \gamma + \gamma \lambda m_n(-\lambda))s)^2}\biggr| =O_*(\lambda)
\end{align*}
uniformly in $s \in \text{supp}(\hat G_n^{(b)})$. Now, note that the total variation of the signed measure $\hat G_n^{(b)}$ can be bounded by Cauchy-Schwarz as
\begin{align*}
    \|\hat G_n^{(b)}\|_{\text{TV}} &= \frac{\sum_{i=1}^p |\langle \bbeta^{(2)}, \bm v_i\rangle| \cdot |\langle \tilde \bbeta, \bm v_i \rangle |}{\|\tilde \bbeta\|_2 \|\bbeta^{(2)}\|_2} \\
    &\le \biggl(\frac{1}{\|\tilde \bbeta\|_2^2} \sum_{i=1}^p \langle \tilde \bbeta, \bm v_i\rangle^2 \biggr) \biggl(\frac{1}{\|\bbeta^{(2)}\|_2^2} \sum_{i=1}^p \langle\bbeta^{(2)}, \bm v_i\rangle^2 \biggr)  \\
    &= 1
\end{align*}
which implies
\begin{align*}
    &\biggl|\int \frac{\gamma s (c_0 s- c_1)}{(1 +  c_0 \gamma s)^2} d\hat G_n^{(b)}(s) \\
    &\qquad - \int \frac{\lambda s(\lambda + (1 - \gamma + \gamma \lambda m_n(-\lambda))s) + \lambda^2 (1 + \gamma m_{n,1}(-\lambda))s}{(\lambda + (1 - \gamma + \gamma \lambda m_n(-\lambda))s)^2} d\hat G_n^{(b)}(s)\biggr| \\
    &= O_*(\lambda)
\end{align*}
from which the conclusion follows.
\end{proof}
\subsubsection{Proof of Eq. \ref{eqn:model_shift_B2}}
Take $\lambda = p^{-1/10}$. Then, for any small $c > 0$, we have
\begin{align*}
    |B_2(\hat \bbeta;\bbeta^{(2)}) - \mathcal B_2(\hat H_n, \hat G_n^{\tilde \bbeta})| &\le |B_2(\hat \bbeta;\bbeta^{(2)}) - B_2(\hat \bbeta_\lambda; \bbeta^{(2)})| \\
    &\qquad + |B_2(\hat \bbeta_\lambda; \bbeta^{(2)}) - \mathcal B_2(\lambda; \hat H_n, \hat G_n^{\tilde \bbeta})| \\
    &\qquad + |\mathcal B_2(\lambda; \hat H_n, \hat G_n^{\tilde \bbeta}) - \mathcal B_2(\hat H_n, \hat G_n^{\tilde \bbeta})| \\
    &= O(\lambda^{-4} p^{-1/2+c} \|\tilde \bbeta\|_2^2) + O_*(\lambda \|\tilde \bbeta\|_2^2) + O(n^c \lambda  \|\tilde \bbeta\|_2^2) \\
    &= O(p^{-1/10+c} \|\tilde \bbeta\|_2^2).
\end{align*}
with high probability.

\subsubsection{Proof of Eq. \ref{eqn:model_shift_B3}}
Take $\lambda = p^{-1/10}$. Then, for any small $c > 0$, we have
\begin{align*}
    |B_3(\hat \bbeta;\bbeta^{(2)}) - \mathcal B_3(\hat H_n, \hat G_n^{(b)})| &\le |B_3(\hat \bbeta;\bbeta^{(2)}) - B_3(\hat \bbeta_\lambda; \bbeta^{(2)})| \\
    &\qquad + |B_3(\hat \bbeta_\lambda; \bbeta^{(2)}) - \mathcal B_3(\lambda; \hat H_n, \hat G_n^{(b)})| \\
    &\qquad + |\mathcal B_3(\lambda; \hat H_n, \hat G_n^{(b)}) - \mathcal B_3(\hat H_n, \hat G_n^{(b)})| \\
    &= O(\lambda^{-3} p^{-1/2+c} \|\tilde \bbeta\|_2 \|\bbeta^{(2)}\|_2) + O_*(\lambda \|\tilde \bbeta\|_2 \|\bbeta^{(2)}\|_2) \\
    &\qquad + O(n^c \lambda \|\bbeta^{(2)}\|_2\|\tilde \bbeta\|_2)\\
    &= O(p^{-1/10+c} \|\tilde \bbeta\|_2 \|\bbeta^{(2)}\|_2).
\end{align*}
with high probability.

\section{Proof for Covariate Shift}\label{sec:proof:design_shift}
We first prove the generalization error for ridge regression under covariate shift (Theorem \ref{thm:design_shift_ridge}), and then obtain the generalization error for the min-norm interpolator by letting $\lambda \rightarrow 0$. Recall the expression for bias and variance was given by \eqref{eqn:variance_analytical_ridge} and \eqref{eqn:bias_analytical_ridge}. Note that $\tilde\bbeta=\bm{0}$ here. Define
\begin{equation}\label{eqn:def:W}
\bW := (\bLambda^{(1)})^{1/2}\bV^\top\hat\bSigma_{\bZ}^{(1)}\bV(\bLambda^{(1)})^{1/2} +(\bLambda^{(2)})^{1/2}\bV^\top\hat\bSigma_{\bZ}^{(2)}\bV(\bLambda^{(2)})^{1/2} 
\end{equation}

Taking the antiderivative, we have
\begin{equation}\label{eqn:variance:analytical_ridge_antider}
\begin{aligned}
&V(\hat\bbeta_\lambda;\bbeta^{(2)})\\
=& \frac{\partial}{\partial\lambda} \left\{\frac{\sigma^2\lambda}{n}\Tr\left(\bSigma^{(2)}(\hat\bSigma+\lambda\bI)^{-1}\right)\right\}\\
=&\frac{\partial}{\partial\lambda} \left\{\frac{\sigma^2\lambda}{n}\Tr\left(\bSigma^{(2)}\left((\bSigma^{(1)})^{1/2}\hat\bSigma_{\bZ}^{(1)}(\bSigma^{(1)})^{1/2} +(\bSigma^{(2)})^{1/2}\hat\bSigma_{\bZ}^{(2)}(\bSigma^{(2)})^{1/2} +\lambda\bI\right)^{-1}\right)\right\}\\
=&\frac{\partial}{\partial\lambda} \left\{\frac{\sigma^2\lambda}{n}\Tr\left((\bLambda^{(2)})^{1/2}\left(\bW +\lambda\bI\right)^{-1}(\bLambda^{(2)})^{1/2}\right)\right\},
\end{aligned}
\end{equation}
and
\begin{equation}\label{eqn:bias:analytical_bias_antider}
\begin{aligned}
&B(\hat\bbeta_\lambda,\bbeta^{(2)})\\
=& -\frac{\partial}{\partial \eta} \left\{\lambda\bbeta^{(2)\top}(\hat\bSigma+\lambda\bI+\lambda\eta\bSigma^{(2)})^{-1}\bbeta^{(2)}\right\}\bigg |_{\eta=0}\\
=&  -\frac{\partial}{\partial \eta} \Bigl\{\lambda\bbeta_\eta^{(2)\top}\Bigl((\bLambda_\eta^{(1)})^{1/2}\bV^\top\hat\bSigma_{\bZ}^{(1)}\bV(\bLambda_\eta^{(1)})^{1/2} \\
=&\hspace{8em}+(\bLambda_\eta^{(2)})^{1/2}\bV^\top\hat\bSigma_{\bZ}^{(2)}\bV(\bLambda_\eta^{(2)})^{1/2} +\lambda\bI\Bigr)^{-1}\bbeta_\eta^{(2)}\Bigr\}\bigg |_{\eta=0},
\end{aligned}
\end{equation}
where $\bLambda_\eta^{(k)} := \bLambda^{(k)}(\bI+\eta\bLambda^{(2)})^{-1}$ for $k=1,2, \bbeta_\eta^{(2)} := (\bI+\eta\bLambda^{(2)})^{-1/2}\bV^\top\bbeta^{(2)}$, and $\hat\bSigma_{\bZ}^{(k)}:= \frac{1}{n}\bZ^{(k)\top}\bZ^{(k)}$ where $\bZ^{(k)}$ satisfies Assumptions \ref{as:shared} and \ref{as:design_four_epsilon}. Here $\bV,\bLambda^{(1)},\bLambda^{(2)}$ are defined in Definition \ref{def:simultaneously_diagonalizable}.

\subsection{Prototype Problem}
Both the bias and variance terms above \eqref{eqn:bias:analytical_bias_antider}, \eqref{eqn:variance:analytical_ridge_antider} involves characterizing the asymptotic behavior of $(\bW+\lambda\bI)^{-1}$ (i.e. the resolvent of $\bW$ at $-\lambda$) where $\bW$ is defined by \eqref{eqn:def:W}.

A powerful tool to handle this type of problems is known as the anisotropic local law. \cite{knowles2017anisotropic} first introduced this technique and effectively characterized the behavior of $\left((\bLambda^{(1)})^{1/2}\bV^\top\hat\bSigma_{\bZ}^{(1)}\bV(\bLambda^{(1)})^{1/2} +z\bI\right)^{-1}$ for a wide range of $z\in \C$, whereas \cite{yang2020analysis} generalized the result to $\left((\bLambda^{(1)})^{1/2}\bV^\top\hat\bSigma_{\bZ}^{(1)}\bV(\bLambda^{(1)})^{1/2} +\hat\bSigma_{\bZ}^{(2)} +z\bI\right)^{-1}$ at $z$ around $0$.

In this section, we state and prove a more general local law which is one of the major technical contributions of this manuscript. This local law will serve as the basis for studying the limit of \eqref{eqn:bias:analytical_bias_antider} and \eqref{eqn:variance:analytical_ridge_antider}. The following subsections follows \cite[Section B.3]{yang2020analysis} closely.

\subsubsection{Resolvent and Local Law}
We can write $\bW = \bF\bF^\top$ where $\bF \in \R^{p \times n}$ is given by
\begin{equation}\label{eqn:def:F}
\bF := n^{-1/2}[ (\bLambda^{(1)})^{1/2}\bV^\top\bZ^{(1)\top},(\bLambda^{(2)})^{1/2}\bV^\top\bZ^{(2)\top}].
\end{equation}

\begin{definition}[Self-adjoint linearization and resolvent]\label{def:resolvent}
We define the following $(p+n)\times (p+n)$ symmetric block matrix
\begin{equation}\label{eqn:def:H}
\bH := \begin{pmatrix}
\bm{0} & \bF \\ \bF^\top & \bm{0}
\end{pmatrix}
\end{equation}
and its resolvent as
\begin{equation}\label{eq:define_G}
\bG(z) := \left[\bH - \begin{pmatrix}
z \bI_{p} & \bm{0} \\ \bm{0} & \bI_{n}
\end{pmatrix}\right]^{-1}, z \in \C
\end{equation}
as long as the inverse exists. We further define the following (weighted) partial traces
\begin{equation}\label{eqn:def:partial_traces}
\begin{aligned}
m_{01}(z) := \frac{1}{p} \sum_{i \in \mcl{I}_0}\lambda_i^{(1)}G_{ii}(z), \ \ & m_{02}(z) := \frac{1}{p} \sum_{i \in \mcl{I}_0}\lambda_i^{(2)}G_{ii}(z),\\ m_1(z) := \frac{1}{n_1} \sum_{\mu \in \mcl{I}_1}G_{\mu\mu}(z), \ \ & m_2(z) := \frac{1}{n_2} \sum_{\nu \in \mcl{I}_2}G_{\nu\nu}(z),
\end{aligned}
\end{equation}
where $\mcl{I}_i, i=0,1,2$, are index sets defined as
\[
\mcl{I}_0:=[\![1,p]\!],\ \ \mcl{I}_1:=[\![p+1,p+n_1]\!],\ \ \mcl{I}_0:=[\![p+n_1+1,p+n_1+n_2]\!].
\]
\end{definition}
We will consistently use $i,j\in\mcl{I}_0$ and $\mu,\nu\in \mcl{I}_1\cup \mcl{I}_2$. Correspondingly, the indices are labeled as 
\[\bZ^{(1)}=\left[Z_{\mu i}^{(1)}:i\in\mcl{I}_0,\mu\in\mcl{I}_1\right],\ \ \bZ^{(2)}=\left[Z_{\nu i}^{(2)}:i\in\mcl{I}_0,\nu\in\mcl{I}_2\right].
\]

We further define the set of all indices as $\mcl{I}:=\mI_0\cup \mI_1\cup \mI_2$, and label indices as $\mathfrak{a},\fb,\fc$, etc.

Using Schur complement formula, \eqref{eq:define_G} becomes
\begin{equation}\label{eqn:G_inverse_schur}
\bG(z) = \begin{pmatrix}
(\bF\bF^\top-z\bI)^{-1} & (\bF\bF^\top-z\bI)^{-1}\bF\\
\bF^\top (\bF\bF^\top-z\bI)^{-1} & z(\bF^\top\bF-z\bI)^{-1}
\end{pmatrix}
\end{equation}
Notice that the upper-left block of $\bG$ is directly related to the resolvent of $\bW =\bF\bF^\top$ that we are interested in. However, rather than studying $(\bF\bF^\top-z\bI)^{-1}$ directly, we work with $\bH$ as it is a linear function of $\bZ^{(1)},\bZ^{(2)}$ and is therefore easier to work with.

We now define the limit of $\bG$ as
\begin{equation}\label{eqn:G_limit}
\mathfrak{G}(z) := \begin{pmatrix}
[a_1(z)\bLambda^{(1)}+a_2(z)\bLambda^{(2)}-z]^{-1} & 0 & 0 \\
0 & -\frac{n}{n_1}a_1(z)\bI_{n_1} & 0 \\
0 & 0 & -\frac{n}{n_2}a_2(z)\bI_{n_2}
\end{pmatrix}
\end{equation}
where $a_1(z)$ and $a_2(z)$ are the unique solution to the following system of equations (we omit the dependence on $z$ for conciseness):
\begin{equation}\label{eqns:prototype_system}
\begin{aligned}
a_1+a_2 &= 1 - \gamma \int \frac{a_1\lambda^{(1)}+a_2\lambda^{(2)}}{a_1\lambda^{(1)}+a_2\lambda^{(2)}-z} d\hat H_p(\lambda^{(1)},\lambda^{(2)}),\\
a_1 &= \frac{n_1}{n} - \gamma \int \frac{a_1\lambda^{(1)}}{a_1\lambda^{(1)}+a_2\lambda^{(2)}-z} d\hat H_p(\lambda^{(1)},\lambda^{(2)}), 
\end{aligned}
\end{equation}
such that $\Imm(a_1) \leq 0$ and $\Imm(a_2) \leq 0$ whenever $\Imm(z) > 0$. The existence and uniqueness of the above system will be proved in Lemma \ref{lemma:existence_solution}.

We now state our main result, Theorem \ref{thm:anisotropic_law}, which characterizes the convergence (rate) of $\bG(z)$ to $\mathfrak{G}(z)$ at $z=-\lambda$. This Theorem is a more general extension in the family of \textit{anisotropic local laws} \cite{knowles2017anisotropic}. In specific, it extends \cite[Theorem 27]{yang2020analysis} to allow a general $\bSigma^{(2)}$ control at $z=-\lambda\neq 0$. For a fix small constant $\tau>0$, we define a domain of $z$ as
\begin{equation}\label{eqn:def:spectral_domain}
\bD \equiv \bD(\tau):= \{z = E + i\eta \in \C_+: |z+\lambda| \leq (\log p)^{-1}\lambda\}.
\end{equation}

\begin{thm}\label{thm:anisotropic_law}
Suppose Assumptions \ref{as:shared}, \ref{as:design_four_epsilon}, and \ref{as:simul_diag} hold, with joint spectral distribution defined in \eqref{eqn:joint_ESD}. Suppose also that $\bZ^{(1)}$ and $\bZ^{(2)}$ satisfy the bounded support condition \eqref{eqn:bounded_support} with $Q=p^{2/\varphi}$. Then the following local laws hold for $\lambda > p^{-1/7+\epsilon}$ with a small constant $\epsilon$:
\begin{enumerate}
    \item[(i)] \textbf{Averaged local law}: We have
    \begin{align}
    \sup_{z\in \bD}\left|p^{-1}\sum_{i\in\mcl{I}_0}\lambda_{i}^{(k)}[G_{ii}(z)-\mathfrak{G}_{ii}(z)]\right| \prec p^{-1}\lambda^{-7}Q,\ \ k=1,2.\label{eqn:averaged_law}
    \end{align}
    \item[(ii)] \textbf{Anisotropic local law}: For any deterministic unit vectors $\bu,\bv\in \R^{p+n}$, we have
    \begin{equation}\label{eqn:anisotropic_law}
    \sup_{z\in \bD}\left|\bu^\top[\bG(z)-\mathfrak{G}(z)]\bv\right| \prec p^{-1/2}\lambda^{-3} Q.
    \end{equation}
\end{enumerate}
\end{thm}

The rest of this section is devoted to the proof of Theorem \ref{thm:anisotropic_law}. We focus on the case of $\lambda \in (p^{-1/7+\epsilon},1)$, since the case for $\lambda \geq 1$ can be trivially extended by dropping all negative powers of $\lambda$ in the proof below.

\subsubsection{Self-consistent equations}
In this subsection, we show that the self-consistent equations \eqref{eqns:prototype_system} has a unique solution for any $z \in \bD$. For simplicity, we further define 

\begin{equation}\label{eq:define_r1_r2}
    r_1:=n_1/n, r_2 := n_2 / n.
\end{equation}

We first show that \eqref{eqns:prototype_system} has a unique solution at $z=-\lambda$, where \eqref{eqns:prototype_system} reduces to the following system (equivalent to the first two equations of \eqref{eqn:design_shift_ridge_alpha_eqns}):
\begin{equation}\label{eqn:fg_prototype_system}
\begin{aligned}
f(a_1,a_2) &:= a_1 - r_1 + \gamma \int \frac{a_1\lambda^{(1)}}{a_1\lambda^{(1)}+a_2\lambda^{(2)}+\lambda} d\hat H_p(\lambda^{(1)},\lambda^{(2)}) = 0,\\
g(a_1,a_2) &:= a_2 - r_2 + \gamma \int \frac{a_2\lambda^{(2)}}{a_1\lambda^{(1)}+a_2\lambda^{(2)}+\lambda} d\hat H_p(\lambda^{(1)},\lambda^{(2)}) = 0. 
\end{aligned}
\end{equation}

Focusing on $f$ and consider any fixed $a_2>0$ for a small enough constant $c$. We know 
$$f(C\lambda,a_2) > -r_1 +\gamma \int \frac{C\lambda\lambda^{(1)}}{C\lambda\lambda^{(1)} + a_2\lambda^{(2)}+\lambda}d\hat H_p(\lambda^{(1)},\lambda^{(2)})>0, $$ if $C$ is large enough,
$$f(c\min\{\lambda,r_1\},a_2) < c\min\{\lambda,r_1\}-r_1 +\gamma \int \frac{c\lambda\lambda^{(1)}}{\lambda}d\hat H_p(\lambda^{(1)},\lambda^{(2)})<0, $$ if $c$ is small enough, and 
$$\frac{\partial}{\partial a_1} f(a_1,a_2) = 1 + \gamma \int \frac{\lambda^{(1)}(a_2\lambda^{(2)}+\lambda)}{(a_1\lambda^{(1)}+a_2\lambda^{(2)}+\lambda)^2}d\hat H_p(\lambda^{(1)},\lambda^{(2)}) > 0.$$
for any $a_1>0$. Together with the obvious continuity, we know, given $a_2$, there exists a unique $a_1 \in (c\min\{\lambda,r_1\},C\lambda)$ such that $f(a_1,a_2) = 0$. Define the root as $$a_1 := \chi(a_2;f).$$ By implicit function theorem, we know the function $\chi$ is continuous, and
$$\chi'(a_2;f) + \gamma \int \frac{\chi'(a_2;f)\lambda^{(1)}(a_2\lambda^{(2)}+\lambda)-\chi(a_2;f)\lambda^{(1)}\lambda^{(2)}}{(a_1\lambda^{(1)}+a_2\lambda^{(2)}+\lambda)^2}d\hat H_p(\lambda^{(1)},\lambda^{(2)})=0,$$
which, after straightforward rearranging, yields $\chi'(a_2;f)>0$.

Now focus on $g(\chi(a_2;f),a_2)$. Since $\chi(a_2;f) \in (c\min\{\lambda,r_2\} ,C\lambda)$ for any $a_2>0$, similar to our derivation above, we have 
$$g(\chi(c\min\{\lambda,r_2\};f),c\min\{\lambda,r_1\})<0, g(\chi(C\lambda;f),C\lambda)>0$$ for $c$ small enough and $C$ large enough, and $\frac{\dd}{\dd a_2} g(\chi(a_2;f),a_2)>0$. Therefore there exists a unique $a_2\in (c\min\{\lambda,r_2\},C\lambda)$ such that $g(\chi(a_2;f),a_2)=0$. Hence, there exists a unique pair of positive $(a_1,a_2)$ satisfying \eqref{eqn:fg_prototype_system}, such that
\begin{equation}\label{eqn:a1_a2_bounded}
(a_1,a_2) \in (c\min\{\lambda,r_1\},C\lambda)\times (c\min\{\lambda,r_2\},C\lambda)
\end{equation}

Next, we prove the existence and uniqueness of the solution to \eqref{eqns:prototype_system} for a general $z \in \bD$, where $\bD$ is defined by \eqref{eqn:def:spectral_domain}, presented by the following lemma:

\begin{lemma}\label{lemma:existence_solution}
There exists constants $c,c_1,C>0$ such that the following statements hold. There exists a unique solution $(a_1(z),a_2(z))$ of \eqref{eqns:prototype_system} under the conditions
\begin{equation}\label{eqn:delta_def_bound}
|z+\lambda| \leq c\lambda, |a_1(-\lambda)-a_1(z)|+ |a_2(-\lambda)-a_2(z)| \leq c_1\lambda.
\end{equation}
Moreover, the solution satisfies
\begin{equation}\label{eqn:delta_bounded_by_z}
|a_1(-\lambda)-a_1(z)|+ |a_2(-\lambda)-a_2(z)| \leq C|z+\lambda|.
\end{equation}
\end{lemma}
\begin{proof}
The proof is based on contraction principle. Define the more general form of \eqref{eqn:fg_prototype_system} as
\begin{equation}\label{eqn:fgz_prototype_system}
\begin{aligned}
f_z(a_1,a_2) &:= a_1 - r_1 + \gamma \int \frac{a_1\lambda^{(1)}}{a_1\lambda^{(1)}+a_2\lambda^{(2)}-z} d\hat H_p(\lambda^{(1)},\lambda^{(2)}) = 0,\\
g_z(a_1,a_2) &:= a_2 - r_2 + \gamma \int \frac{a_2\lambda^{(2)}}{a_1\lambda^{(1)}+a_2\lambda^{(2)}-z} d\hat H_p(\lambda^{(1)},\lambda^{(2)}) = 0. 
\end{aligned}
\end{equation}
Define the vector
$$\bF_z(a_1,a_2):=[f_z(a_1,a_2),g_z(a_1,a_2)].$$ Denote the (potential) solution to $\bF=\bm{0}$ as $\ba(z):=[a_1(z),a_2(z)]$. By definition, we know $\bF_z(\ba(z))=0$. Therefore, defining $\bdelta(z):=\ba(z)-\ba(-\lambda)$ we have the following identity:
\begin{align*}
\bdelta(z) \equiv& -\bJ_z(\ba(-\lambda))^{-1}\left(\bF_z(\ba(-\lambda))-\bF_{-\lambda}(\ba(-\lambda))\right) \\
&- \bJ_z(\ba(-\lambda))^{-1}\left(\bF_z(\ba(-\lambda)+\bdelta(z))-\bF_z(\ba(-\lambda)) -  \bJ_z(\ba(-\lambda))\bdelta(z)\right),
\end{align*}
where $\bJ_z(\ba)$ is the Jacobian of $\bF_z$ at $\ba$:
\begin{equation}\label{eqn:def:Jacobian}
\begin{aligned}
&\bJ_z(a_1,a_2) \\&:= \begin{pmatrix}
\partial_1 f_z(a_1,a_2) & \partial_2 f_z(a_1,a_2)\\
\partial_1 g_z(a_1,a_2) & \partial_2 g_z(a_1,a_2)
\end{pmatrix}\\
& = \begin{pmatrix}
1 + \gamma \int \frac{\lambda^{(1)}(a_2\lambda^{(2)}-z)}{(a_1\lambda^{(1)}+a_2\lambda^{(2)}-z)^2}d\hat H_p(\lambda^{(1)},\lambda^{(2)}) &  -\gamma \int \frac{a_1\lambda^{(1)}\lambda^{(2)}}{(a_1\lambda^{(1)}+a_2\lambda^{(2)}-z)^2}d\hat H_p(\lambda^{(1)},\lambda^{(2)}) \\
 -\gamma \int \frac{a_2\lambda^{(1)}\lambda^{(2)}}{(a_1\lambda^{(1)}+a_2\lambda^{(2)}-z)^2}d\hat H_p(\lambda^{(1)},\lambda^{(2)}) & 1 + \gamma \int \frac{\lambda^{(2)}(a_1\lambda^{(1)}-z)}{(a_1\lambda^{(1)}+a_2\lambda^{(2)}-z)^2}d\hat H_p(\lambda^{(1)},\lambda^{(2)}) \\
\end{pmatrix}\\
& := \begin{pmatrix}
1 + \check x_2(z) + \check y_1(z) & -\check x_1(z)\\
-\check x_2(z) & 1 + \check x_1(z) + y_2(z)\\
\end{pmatrix},
\end{aligned}
\end{equation}
where for $k=1,2$, $\partial_k$ denotes the partial derivative with respect to the $k$-th argument, and \begin{align*}
    &\check x_k(z) := \gamma \int \frac{a_k\lambda^{(1)}\lambda^{(2)}}{(a_1\lambda^{(1)}+a_2\lambda^{(2)}-z)^2}d\hat H_p(\lambda^{(1)},\lambda^{(2)}),\\
    &\check y_k(z) := -\gamma \int \frac{\lambda^{(k)}z}{(a_1\lambda^{(1)}+a_2\lambda^{(2)}-z)^2}d\hat H_p(\lambda^{(1)},\lambda^{(2)}).
\end{align*}

Notice that we omit the dependence on $-\lambda$ for simplicity without confusion, as only $\ba(-\lambda)$ appear from now on. Taking this inspiration, we define the iterative sequence $\bdelta^{(k)}(z)$ such that $\bdelta^{(0)}(z) = \bm{0}$ and
\begin{equation}\label{eqn:iterative_delta}
\begin{aligned}
&\bdelta^{(k+1)}(z) := \bh_z(\bdelta^{(k)}(z))\\
=&  -\bJ_z(\ba(-\lambda))^{-1}\left(\bF_z(\ba(-\lambda))-\bF_{-\lambda}(\ba(-\lambda))\right) \\
&- \bJ_z(\ba(-\lambda))^{-1}\left(\bF_z(\ba(-\lambda)+\bdelta^{(k)}(z))-\bF_z(\ba(-\lambda)) -  \bJ_z(\ba(-\lambda))\bdelta^{(k)}(z)\right).
\end{aligned}
\end{equation}
We want to show that $\bh_z$ is a contractive self-mapping. Let $\check z := z + \lambda$. We first deal with the first summand of \eqref{eqn:iterative_delta}. Since $|\check z| \leq c\lambda$, we have
\begin{align*}
&\left|\frac{\partial}{\partial \check z} \check x_1(z)\right|\\
=& \left| 2\gamma \int \frac{a_1(z)\lambda^{(1)}\lambda^{(2)}}{(a_1(z)\lambda^{(1)}+a_2(z)\lambda^{(2)}+\lambda-\check z)^3}d\hat H_p(\lambda^{(1)},\lambda^{(2)}) \right|\\
\leq & 2\gamma \int \left| \frac{a_1(z)\lambda^{(1)}\lambda^{(2)}}{(a_1(z)\lambda^{(1)}+a_2(z)\lambda^{(2)}+\lambda-\check z)^3} \right|d\hat H_p(\lambda^{(1)},\lambda^{(2)})\\
= & O(\lambda^{-2})
\end{align*}
if $c$ is small enough. Further, a corollary of \eqref{eqn:a1_a2_bounded} yields
\begin{align*}
\check x_1(-\lambda) &:=\gamma \int \frac{a_1\lambda^{(1)}\lambda^{(2)}}{(a_1\lambda^{(1)}+a_2\lambda^{(2)}+\lambda)^2}d\hat H_p(\lambda^{(1)},\lambda^{(2)}) =\Theta(\lambda^{-1}),
\end{align*} 
which, together with the equation above and the fact that $|\check z| \leq c \lambda$, yields
$$|\check x_1(-\lambda)-\check x_1(z)| = O(c\lambda^{-1})$$
Using similar derivations for $\check x_2,\check y_1,\check y_2$, we obtain that for $k=1,2$,
\begin{equation}\label{eqn:checkxy_bounded}
\begin{aligned}
&\check x_k(-\lambda) =\Theta(\lambda^{-1}),\quad  \check y_k(-\lambda)=\Theta(\lambda^{-1}),\\
&|\check x_k(-\lambda)-\check x_k(z)| = O(c\lambda^{-1}),\ \ |\check y_k(-\lambda)-\check y_k(z)| = O(c\lambda^{-1}).
\end{aligned}
\end{equation}
Let $\check \lambda_{\pm}$ denote the two eigenvalues of $\bJ$. Since $\bJ$ is a $2\times 2$ matrix, it is easy to check that
\begin{align*}
    \check\lambda_{\pm}(\bJ_{z}(\ba(-\lambda)))&:= \frac{2+\check x_1(z) + \check x_2(z) + \check y_1(z) + \check y_2(z)}{2}\\ &\pm \frac{\sqrt{(\check x_2(z) + \check y_1(z) - \check x_1(z) - \check y_2(z))^2 + 4\check x_1(z)\check x_2(z)}}{2}.
\end{align*}

At $z=-\lambda$, we know 
\begin{align*}
    &(\check x_1(z) + \check x_2(z) + \check y_1(z) + \check y_2(z))^2 -(\check x_2(z) + \check y_1(z) - \check x_1(z) - \check y_2(z))^2 - 4\check x_1(z)\check x_2(z) \\
    &= 4\check x_1(-\lambda)\check y_1(-\lambda)+4\check x_2(-\lambda)\check y_2(-\lambda)+4\check y_1(-\lambda)\check y_2(-\lambda)=\Theta(\lambda^{-2}),
\end{align*} and $\check x_1(z) + \check x_2(z) + \check y_1(z) + \check y_2(z) =\Theta(\lambda^{-1})$. Therefore we obtain $$\check \lambda_{\pm}(\bJ_z(\ba(-\lambda)))=\Theta(\lambda^{-1}).$$

Now consider a general $|\check z| \leq c\lambda$. By \eqref{eqn:checkxy_bounded}, it is straightforward to derive that $$|\check\lambda_{\pm}(\bJ_{z}(\ba(-\lambda)))-\check\lambda_{\pm}(\bJ_{-\lambda}(\ba(-\lambda)))| = O(c\lambda^{-1}).$$ Therefore we have $$|\check\lambda_{\pm}(\bJ_{z}(\ba(-\lambda)))| \geq |\check\lambda_{\pm}(\bJ_{-\lambda}(\ba(-\lambda)))| - O(c\lambda^{-1}) = \Theta(\lambda^{-1})-O(c\lambda^{-1}) = \Theta(\lambda^{-1}),$$ if $c$ is small enough, which yields
\begin{equation}\label{eqn:Jinv_op_bounded}
\|\bJ_z(\ba(-\lambda))^{-1}\|_{\op} = O(\lambda).
\end{equation}
Using similar derivations as \eqref{eqn:checkxy_bounded}, we obtain that for $|\check z|\leq c\lambda$,
\begin{equation}\label{eqn:fg_diff_bounded}
\begin{aligned}
&\frac{\partial}{\partial \check z}f_{z}(\ba(-\lambda)) =  O(\lambda^{-1}), \qquad \frac{\partial}{\partial \check z}g_{z}(\ba(-\lambda)) =  O(\lambda^{-1}),\\
&|f_{-\lambda}(\ba(-\lambda)) - f_z(\ba(-\lambda))| = O(\lambda^{-1}|\check z|), |g_{-\lambda}(\ba(-\lambda)) - g_z(\ba(-\lambda))| = O(\lambda^{-1}|\check z|),
\end{aligned}
\end{equation}
which yields 
\begin{equation}\label{eqn:F_term1_bounded}
\|\bF_z(\ba(-\lambda))-\bF_{-\lambda}(\ba(-\lambda))\|_2 = O(\lambda^{-1}|\check z|)
\end{equation}
Now we deal with the second term of $\bh_z$ \eqref{eqn:iterative_delta}. Using \eqref{eqn:def:Jacobian} and similar derivations as \eqref{eqn:checkxy_bounded}, we know there exists a constant $c_1$ such that for $z\leq c\lambda$, $\sqrt{|\delta_1|^2+|\delta_2|^2} \leq c_1\lambda$ and $\check k_1,\check k_2,\check k_3=1,2$
\begin{equation}\label{eqn:fg_der_bounded}
\begin{aligned}
& \partial_{\check k_1\check k_2} f_z(a_1,a_2) = \Theta(\lambda^{-2}), \quad
\partial_{\check k_1\check k_2} g_z(a_1,a_2) = \Theta(\lambda^{-2}),\\
& \partial_{\check k_1\check k_2\check k_3} f_z(a_1+\delta_1,a_2+\delta_2) = O(\lambda^{-3}), \quad 
\partial_{\check k_1\check k_2\check k_3} g_z(a_1+\delta_1,a_2+\delta_2) = O(\lambda^{-3}),
\end{aligned}
\end{equation}
As a corollary, for $\bdelta_1,\bdelta_2$ satisfying $\|\bdelta_1\|_2 \leq c_1\lambda,\|\bdelta_2\|_2 \leq c_1\lambda$,
the difference of the second term of $\bh_z$ reads
\begin{equation}\label{eqn:F_term2_bounded}
\begin{aligned}
& \left\|\bJ_z(\ba(-\lambda))^{-1}\left(\bF_z(\ba(-\lambda)+\bdelta_1) - \bF_z(\ba(-\lambda)+\bdelta_2)\right)- (\bdelta_1-\bdelta_2) \right\|_2 \\
\leq & \|\bJ_z(\ba(-\lambda))\|_{\op}\cdot\|\bF_z(\ba(-\lambda)+\bdelta_1) - \bF_z(\ba(-\lambda)+\bdelta_2) - \bJ_z(\ba(-\lambda))(\bdelta_1-\bdelta_2)\|_2\\
=& O(\lambda) \cdot \|\bdelta_1-\bdelta_2\|_2 \cdot O(c_1\lambda^{-1})\\
= & O(c_1\|\bdelta_1-\bdelta_2\|_2).
\end{aligned}
\end{equation}
Consequently, for $\|\bdelta\|_2 \leq c_1\lambda$, plugging in $\bdelta_1=\bdelta$ and $\bdelta_2=\bm{0}$ to the equation above and combining \eqref{eqn:Jinv_op_bounded}, \eqref{eqn:F_term1_bounded}, we obtain that
$$\|\bh_z(\bdelta)\|_2 \leq O(|\check z|) + O(c_1\|\bdelta\|_2) = O(c\lambda) + O(c_1^2\lambda) \leq c_1\lambda$$
if we first choose $c_1$ small enough such that $O(c_1^2\lambda)\leq \frac{1}{2}c_1\lambda$ and then choose $c$ small enough such that $O(c) \leq \frac{1}{2}c_1$. 

Therefore $\bh_z:B_{c_1} \rightarrow B_{c_1}$ is a self-mapping on the ball $B_{c_1} := \{\bdelta \in \C^2: \|\bdelta\|\leq c_1\}$. Moreover, from \eqref{eqn:iterative_delta} and \eqref{eqn:F_term2_bounded}, we know $$\|\bdelta^{(k+1)} - \bdelta^{(k)}\|_2 \leq O(c_1\|\bdelta^{(k)}-\bdelta^{(k-1)}\|_2) \leq \frac{1}{2}\|\bdelta^{(k)}-\bdelta^{(k-1)}\|_2$$ for the $c_1$ chosen. This means that $\bh_z$ restricted to $B_{c_1}$ is a contraction. By contraction mapping theorem, $\bdelta^* := \lim_{k\rightarrow \infty} \bdelta^{(k)}$ exists and $\ba(-\lambda)+\bdelta$ is a unique solution to the equation $\bF_z(\ba(z)) = 0$ subject to the condition $\|\bdelta\|\leq c_1$. \eqref{eqn:delta_def_bound} is then obtained by noticing the final fact that for all $\check \delta_1,\check \delta_2$,
\begin{equation}\label{eqn:vector_norm_equivalent}
|\check \delta_1| + |\check \delta_2| \leq \sqrt{2(|\check \delta_1|^2+|\check \delta_2|^2)}
\end{equation}

We are now left with proving \eqref{eqn:delta_bounded_by_z}. Using \eqref{eqn:Jinv_op_bounded}, \eqref{eqn:F_term1_bounded} and $\bdelta^{(0)}=\bm{0}$, we can obtain from \eqref{eqn:iterative_delta} that $\|\bdelta^{(1)}\|_2 \leq O(|\check z|)$, which means there exists $C_0$ such that $\|\bdelta^{(1)}-\bdelta^{(0)}\|_2 \leq C_1 |\check z|$. Then by the contraction mapping, we have
$$\|\bdelta^*\| \leq \sum_{k=0}^\infty \|\bdelta^{(k+1)}-\bdelta^{(k)}\|_2 \leq 2 C_1|\check z|,$$
which, together with \eqref{eqn:vector_norm_equivalent}, yields \eqref{eqn:delta_bounded_by_z}.
\end{proof}

As a by-product, we obtain the following stability result which is useful for proving Theorem \ref{thm:anisotropic_law}. It roughly says that if we replace the $0$'s on the right hand sides of \eqref{eqn:fgz_prototype_system} with some small errors, the resulting solutions will still be close to $a_1(z),a_2(z)$. We further make the following rescaling, which will prove handy in later derivations:
\begin{equation}\label{eqn:rescale_m_a}
m_{1c}(z) := -r_1^{-1}a_1(z),\ \ m_{2c}(z) := -r_2^{-1} a_2(z),
\end{equation}
where $r_1,r_2$ are defined by \eqref{eq:define_r1_r2}. Combining \eqref{eqn:rescale_m_a} with \eqref{eqn:a1_a2_bounded} yields
\begin{equation}\label{eqn:m12_bounded}
c\lambda \leq -m_{1c}(-\lambda) \leq C\lambda,\ \ c\lambda \leq -m_{2c}(-\lambda) \leq C\lambda
\end{equation}
for some (potentially different from before) constants $c,C > 0$.

\begin{lemma}\label{lemma:consistent_equation_with_error}
There exists constants $c,c_1,C>0$ such that the following statements hold. Suppose $|z+\lambda|\leq c\lambda$, and $m_{1}(z),m_{2}(z)$ are analytic functions of $z$ such that
\begin{equation}\label{eqn:delta_bounded_perturbed}
|m_{1}(z)-m_{1c}(-\lambda)| + |m_{2}(z)-m_{2c}(-\lambda)| \leq c_1\lambda.
\end{equation}
Moreover, assume that $(m_{1},m_{2})$ satisfies the system of equations
\begin{equation}\label{eqn:system_eqn_m}
\begin{aligned}
\frac{1}{m_{1}} + 1 - \gamma \int \frac{\lambda^{(1)}}{r_1m_{1}\lambda^{(1)}+r_2m_{2}\lambda^{(2)}+z} d\hat H_p(\lambda^{(1)},\lambda^{(2)}) &= \mcl{E}_1,\\
\frac{1}{m_{2}} + 1 - \gamma \int \frac{\lambda^{(2)}}{r_1m_{1}\lambda^{(1)}+r_2m_{2}\lambda^{(2)}+z} d\hat H_p(\lambda^{(1)},\lambda^{(2)}) &= \mcl{E}_2,
\end{aligned}    
\end{equation}
for some (deterministic or random) errors such that $|\mcl{E}_1|+|\mcl{E}_2| \leq (\log n)^{-1/2}$. Then we have
$$|m_{1}(z)-m_{1c}(z)| + |m_{2}(z)-m_{2c}(z)|\leq C(|\mcl{E}_1|+|\mcl{E}_2|)\lambda.$$
\begin{proof}
We connect with Lemma \ref{lemma:existence_solution}. Notice that $a_{1\epsilon}(z):=-r_1m_{1}(z),a_{2\epsilon}(z):=-r_2m_{2}(z)$ satisfies the following:
\begin{equation}\label{eqn:fgz_prototype_system_perturbed}
\begin{aligned}
f_z(a_{1\epsilon},a_{1\epsilon}) &:= a_{1\epsilon} - r_1 + \gamma \int \frac{a_{1\epsilon}\lambda^{(1)}}{a_{1\epsilon}\lambda^{(1)}+a_{2\epsilon}\lambda^{(2)}-z} d\hat H_p(\lambda^{(1)},\lambda^{(2)}) = \check{\mcl{E}}_1,\\
g_z(a_{1\epsilon},a_{2\epsilon}) &:= a_{2\epsilon} - r_2 + \gamma \int \frac{a_{1\epsilon}\lambda^{(2)}}{a_{1\epsilon}\lambda^{(1)}+a_{2\epsilon}\lambda^{(2)}-z} d\hat H_p(\lambda^{(1)},\lambda^{(2)}) = \check{\mcl{E}}_2. 
\end{aligned}
\end{equation}
with $|\tilde{\mcl{E}}_1| + |\tilde{\mcl{E}}_2| \lesssim |{\mcl{E}}_1|+|{\mcl{E}}_2|$. We then subtract \eqref{eqn:fgz_prototype_system} from \eqref{eqn:fgz_prototype_system_perturbed} and consider the contraction principle for $\bdelta(z) := \ba_{\epsilon}(z) - \ba(z)$. The rest of the proof is exactly the same as the one for \ref{lemma:existence_solution} and we omit the details.
\end{proof}
\end{lemma}

\subsubsection{Multivariate Gaussian Case}
In this section we prove the idealized setting where the entries of $\bZ^{(1)}$ and $\bZ^{(2)}$ are i.i.d. Gaussian. By the rotational invariance of multivariate Gaussian distribution, we have
\begin{equation}\label{eqn:gaussian_rotational_invariance}
\bZ^{(1)}\bV(\bLambda^{(1)})^{1/2}, \bZ^{(2)}\bV(\bLambda^{(2)})^{1/2} \overset{d}{=} \bZ^{(1)}(\bLambda^{(1)})^{1/2}, \bZ^{(2)}(\bLambda^{(2)})^{1/2}.
\end{equation}
Notice that since the entries of $\bZ^{(1)},\bZ^{(2)}$ are i.i.d. Gaussian, they have bounded support $Q=1$ by the remark after Definition \ref{def:bounded_support}. Hence  
we can use the resolvent method (cf. \cite{alex2014isotropic}) to prove the following proposition.

\begin{prop}\label{prop:anisotropic_law_gaussian}
Under the setting of Theorem \ref{thm:anisotropic_law}, assume further that the entries of $\bZ^{(1)}$ and $\bZ^{(2)}$ are i.i.d. Gaussian random variables. Then the estimates \eqref{eqn:averaged_law} and \eqref{eqn:anisotropic_law} hold with $Q=1$.
\end{prop}
Proposition \ref{prop:anisotropic_law_gaussian} is a corollary of the following entrywise local law.

\begin{lemma}[Entrywise Local Law]\label{lemma:entrywise_law_gaussian}
Recall the definition of $\bG$ in \eqref{eqn:G_inverse_schur} and $\mathfrak{G}$ in \eqref{eqn:G_limit}. Under the setting of Proposition \ref{prop:anisotropic_law_gaussian}, the averaged local laws \eqref{eqn:averaged_law} and the following entrywise local law hold with $Q=1$:
\begin{equation}\label{eqn:entrywise_law_gaussian} 
\sup_{z\in \bD}\max_{\mathfrak{a},\mathfrak{b}\in \mathcal{I}}\left|\bG_{\fa\fb}(z) - \mathfrak{G}_{\fa\fb}(z)\right| \prec p^{-1/2}\lambda^{-3}.
\end{equation}
\end{lemma}

\begin{proof}[Proof of Proposition \ref{prop:anisotropic_law_gaussian}]
With estimate \eqref{eqn:entrywise_law_gaussian}, we can use the polynomialization method in \cite[Section 5]{alex2014isotropic} to prove \eqref{eqn:anisotropic_law} with $Q=1$. The proof is exactly the same with minor notation differences. We omit the details.    
\end{proof}

In the remaining of this subsection we prove Lemma \ref{lemma:entrywise_law_gaussian}, where the resolvent in Definition \ref{def:resolvent} becomes 
\begin{equation}\label{eqn:resolvent_gaussian}
\bG(z) = \begin{pmatrix}
-z\bI_p & n^{-1/2}(\bLambda^{(1)})^{1/2}\bZ^{(1)\top} & n^{-1/2}(\bLambda^{(2)})^{1/2}\bZ^{(2)\top} \\ 
n^{-1/2}\bZ^{(1)}(\bLambda^{(1)})^{1/2} & -\bI & \bm{0}\\
n^{-1/2}\bZ^{(2)}(\bLambda^{(2)})^{1/2} & \bm{0} & -\bI
\end{pmatrix}^{-1}.
\end{equation}

Below we introduce resolvent minors to deal with the inverse using Schur's complement formula.

\begin{definition}[Resolvent minors]\label{def:resolvent_minors}
Given a $(p+n)\times (p+n)$ matrix $\bA$ and $\mathfrak{c} \in \mcl{I}$, the minor of $\bA$ after removing the $\mathfrak{c}$-th row and column is a $(p+n-1)\times (p+n-1)$ matrix denoted by $\bA^{(\mathfrak{c})}:=[\bA_{\mathfrak{a}\mathfrak{b}}:\mathfrak{a},\fb \in \mcl{I}\setminus \{\fc\}]$. We keep the names of indices when defining $\bA^{(\fc)}$, i.e. $\bA_{\fa\fb}^{(\fc)}=\bA_{\fa\fb}$ for $\fa,\fb \neq \fc$. Correspondingly, we define the resolvent minor of $\bG(z)$ by
$$\bG^{(\fc)}(z) := \left[\begin{pmatrix}
-z\bI_p & n^{-1/2}(\bLambda^{(1)})^{1/2}\bZ^{(1)\top} & n^{-1/2}(\bLambda^{(2)})^{1/2}\bZ^{(2)\top} \\ 
n^{-1/2}\bZ^{(1)}(\bLambda^{(1)})^{1/2} & -\bI & \bm{0}\\
n^{-1/2}\bZ^{(2)}(\bLambda^{(2)})^{1/2} & \bm{0} & -\bI
\end{pmatrix}^{(\fc)}\right]^{-1}.$$
We further define the partial traces $m^{(\fc)}(z),m^{(\fc)}(z),m^{(\fc)}(z),m^{(\fc)}(z)$ by replacing $\bG(z)$ with $\bG^{(\fc)}(z)$ in \eqref{eqn:def:partial_traces}. For convenience, we adopt the notation $\bG_{\fa\fb}^({\fc)}=0$ if $\ba=\fc$ or $\fb=\fc$.
\end{definition}

The following formulae are identical to those in \cite[Lemma 33]{yang2020analysis}. In fact they can all be proved using Schur's complement formula.

\begin{lemma}
We have the following resovent identities.
\begin{enumerate}[(i)]
    \item For $i \in \mcl{I}_0$ and $\mu \in \mcl{I}_1 \cup \mcl{I}_2$, we have
    \begin{equation}\label{eqn:complement_ii}
        \frac{1}{G_{ii}} = -z - \left(\bF \bG^{(i)}\bF^\top\right)_{ii},\ \ \frac{1}{G_{\mu\mu}} = -1 - \left(\bF^\top\bG^{(\mu)}\bF\right)_{\mu\mu}.
    \end{equation}
    \item For $i \in \mcl{I}_0,\mu \in \mcl{I}_1\cup \mcl{I}_2,\fa \in \mcl{I} \setminus \{i\}$, and $\fb \in \mcl{I} \setminus \{\mu\}$, we have
    \begin{equation}\label{eqn:complement_ia}
        G_{i\fa} = -G_{ii}\left(\bF\bG^{(i)}\right)_{i\fa},\ \ G_{\mu\fb} = -G_{\mu\mu}\left(\bF^\top\bG^{(\mu)}\right)_{\mu\fb}.
    \end{equation}
    \item For $\fc \in \mcl{I}$ and $\fa,\fb \in \mcl{I}\setminus \{\fc\}$, we have
    \begin{equation}\label{eqn:complement_ab}
        G_{\fa\fb}^{(\fc)} = G_{\fa\fb} - \frac{G_{\fa\fc}G_{\fc\fb}}{G_{\fc\fc}}
    \end{equation}
\end{enumerate}
\end{lemma}

The next lemma provides a priori estimate on the resolvent $\bG(z)$ for $z\in \bD$.

\begin{lemma}\label{lemma:close_Gz}
Under the setting of Theorem \ref{thm:anisotropic_law}, there exists a constant $C$ such that with overwhelming probability the following estimates hold uniformly in $z,z' \in \bD$:
\begin{equation}\label{eqn:bounded_Gz}
    \|\bG(z)\|_{\op} \leq C\lambda^{-1},
\end{equation}
and
\begin{equation}\label{eqn:close_Gz}
    \|\bG(z)-\bG(z')\|_{\op} \leq C\lambda^{-1}|z-z'|.
\end{equation}
\end{lemma}
\begin{proof}
Let
\begin{equation}\label{eqn:F_svd}
    \bF = \sum_{k=1}^n \sqrt{\check \mu_k} \check{\bm{\xi}}_k \check{\bzeta}_k^\top,\ \ \check \mu_1 \geq \check \mu_2 \geq ... \geq \check \mu_n \geq 0=\check\mu_{n+1}=...=\check\mu_p
\end{equation}
be a singular value decomposition of $\bF$, where $\{\check{\bxi}_k\}_{k=1}^p$ are the left singular vectors and $\{\check{\bzeta}_k\}_{k=1}^n$ are the right singular vectors. Then using \eqref{eqn:G_inverse_schur}, we know for $i,j \in \mcl{I}_0,\mu,\nu \in \mcl{I}_1\cup \mcl{I}_2$, 
\begin{equation}\label{eqn:G_spectral_F}
G_{ij} = \sum_{k=1}^p \frac{\check\xi_k(i)\check \xi_k^\top(j)}{\check \mu_k-z},\ \ G_{\mu\nu} = z\sum_{k=1}^n \frac{\check{\zeta}_k(\mu)\check{\zeta}_k^\top(\nu)}{\check\mu_z-z},\ \ G_{i\mu}=G_{\mu i} = \sum_{k=1}^p \frac{\sqrt{\check\mu_k}\check{\bxi}_k(i)\check{\bzeta}_k^\top(\mu)}{\check{\mu}_k-z}
\end{equation}
Since $z\in \bD$, we know
$$\inf_{z\in \bD} \min_{1 \leq k \leq p} |\mu_k - z| \geq C\lambda$$
for some $C$. Combining this fact with \eqref{eqn:G_spectral_F}, we readily conclude \eqref{eqn:bounded_Gz} and \eqref{eqn:close_Gz}.
\end{proof}

Now we are ready to prove Lemma \ref{lemma:entrywise_law_gaussian}.
\begin{proof}[Proof of Lemma \ref{lemma:entrywise_law_gaussian}]
In the setting of Lemma \ref{lemma:entrywise_law_gaussian}, we have \eqref{eqn:gaussian_rotational_invariance} and \eqref{eqn:resolvent_gaussian}. Similar to \cite{yang2020analysis}, we divide our proof into four steps:

\textbf{Step 1: Large deviation estimates.} In this step, we provide large deviation estimates on the off-diagonal entries of $\bG$. We introduce
$$\mcl{Z}_{\fa} := (1-\E_{\fa})[(G_{\fa\fa})^{-1}],\ \ \fa \in \mcl{I}$$,
where $\E_{\fa} := \E[\cdot |\bH^{(\fa)}]$ denotes the partial expectation over the entries in the $\fa$-th row and column of $\bH$. Using \eqref{eqn:complement_ii}, we have for $i\in \mcl{I}_0$,
\begin{equation}\label{eqn:Z_i}
\begin{aligned}
\mcl{Z}_i =& \frac{\lambda_i^{(1)}}{n} \sum_{\mu,\nu \in \mI_1}G_{\mu\nu}^{(i)}\left(\delta_{\mu\nu} - Z_{\mu i}^{(1)}Z_{\nu i}^{(1)}\right) + \frac{\lambda_i^{(2)}}{n} \sum_{\mu,\nu \in \mI_2}G_{\mu\nu}^{(i)}\left(\delta_{\mu\nu} - Z_{\mu i}^{(2)}Z_{\nu i}^{(2)}\right) \\
&- \frac{\sqrt{\lambda_i^{(1)}\lambda_i^{(2)}}}{n} \sum_{\mu,\nu \in \mI_1}G_{\mu\nu}^{(i)}Z_{\mu i}^{(1)}Z_{\nu i}^{(1)},
\end{aligned}
\end{equation}
and for $\mu \in \mI_1$ and $\nu \in \mI_2$,
\begin{equation}
\begin{aligned}\label{eqn:Z_mu}
\mcl{Z}_\mu &= \frac{1}{n} \sum_{i,j\in\mI_0} \sqrt{\lambda_i^{(1)}\lambda_j^{(1)}} G_{ij}^{(\mu)}\left(\delta_{ij}-Z_{\mu i}^{(1)}Z_{\mu j}^{(1)}\right),\\
\mcl{Z}_\nu &= \frac{1}{n} \sum_{i,j\in\mI_0} \sqrt{\lambda_i^{(2)}\lambda_j^{(2)}} G_{ij}^{(\nu)}\left(\delta_{ij}-Z_{\nu i}^{(1)}Z_{\nu j}^{(1)}\right).
\end{aligned}
\end{equation}
We further introduce the random error
\begin{equation}\label{eqn:def:off_diag}
\check \Lambda_o := \max_{\fa \neq \fb} \left| G_{\fa\fa}^{-1} G_{\fa\fb}\right|,
\end{equation}
which controls the size of the largest off-diagonal entries.

\begin{lemma}\label{lemma:off_diag_bounded}
In the setting of Proposition \ref{prop:anisotropic_law_gaussian}, the following holds uniformly in $z \in \bD$:
\begin{equation}\label{eqn:onoff_diag_bounded}
\check \Lambda_0 + \max_{\fa\in\mI}|\mcl{Z}_{\fa}| \prec p^{-1/2}\lambda^{-1}
\end{equation}
\end{lemma}
\begin{proof}
Notice that for $\fa \in \mI$, $\bH^{(\fa)}$ and $\bG^{(\fa)}$ also satisfy the assumptions of Lemma \ref{lemma:close_Gz}. Thus the estimates \eqref{eqn:bounded_Gz} and \eqref{eqn:close_Gz} also hold for $\bG^{(a)}$ with overwhelming probability. For any $i \in \mI_0$, since $\bG^{(i)}$ is independent of the entries in the $i$-th row and column of $\bH$, we can combine \eqref{eqn:bilinear_form_strong}, \eqref{eqn:quadratic_form_diag_strong} and \eqref{eqn:quadratic_form_offdiag_strong} with \eqref{eqn:Z_i} to get
\begin{align*}
|\mcl{Z}_i| \lesssim & \frac{1}{n} \sum_{k=1}^2 \left| \sum_{\mu,\nu \in \mI_k}G_{\mu\nu}^{(i)}\left(\delta_{\mu\nu} - Z_{\mu i}^{(k)}Z_{\nu i}^{(k)}\right)\right| + \frac{1}{n} \left|\sum_{\mu,\nu \in \mI_1}G_{\mu\nu}^{(i)}Z_{\mu i}^{(1)}Z_{\nu i}^{(1)}\right|\\
\prec & \frac{1}{n} \left( \sum_{\mu,\nu \in \mI_1 \cup \mI_2} |G_{\mu\nu}^{(i)}|^2\right)^{1/2} \prec p^{-1/2}\lambda^{-1},
\end{align*}
where in the last step we used \eqref{eqn:bounded_Gz} to get that for $\mu \in \mI_1\cup \mI_2$, w.h.p.,
\begin{equation}\label{eqn:bounded_rowwise_G}
\sum_{\nu \in \mI_1\cup \mI_2} |G_{\mu\nu}^{(i)}|^2 \leq \sum_{\fa\in \mI}|G_{\mu\fa}^{(i)}|^2 = \left(G^{(i)}G^{(i)*}\right)_{\mu\mu} = O(\lambda^{-2}).
\end{equation}
The same bound for $|\mcl{Z}_\mu|$ and $|\mcl{Z}_\nu|$ can be obtained by combining \eqref{eqn:bilinear_form_strong}, \eqref{eqn:quadratic_form_diag_strong} and \eqref{eqn:quadratic_form_offdiag_strong} with \eqref{eqn:Z_mu}. This gives $\max_{\fa\in\mI}|\mcl{Z}_{\fa}| \prec p^{-1/2}\lambda^{-1}$.

Next we consider $\check \lambda_o$. For $i \in \mI_0$ and $\fa \in \mI \setminus \{i\}$, using \eqref{eqn:complement_ia}, \eqref{eqn:linear_form_strong} and \eqref{eqn:bounded_Gz}, we obtain that
\begin{align*}
    |G_{ii}^{-1}G_{ia}| \lesssim & n^{-1/2}\left| \sum_{\mu \in \mI_1}Z_{\mu i}^{(1)}G_{\mu\fa}^{(i)}\right| + n^{-1/2}\left| \sum_{\mu \in \mI_2}Z_{\mu i}^{(2)}G_{\mu\fa}^{(i)}\right| \\
    \prec & n^{-1/2} \left( \sum_{\mu \in \mI_1\cup \mI_2}|G_{\mu\fa}^{(1)}|^2\right)^{1/2} \prec p^{-1/2}\lambda^{-1}
\end{align*}
The exact same bound for $|G_{\mu\mu}^{-1}G_{\mu\fb}|$ with $\mu \in \mI_1\cup\mI_2$ and $\fb \in \mI \setminus \{\mu\}$ can be obtained with similar argument. Thus $\check \Lambda_o \prec p^{-1/2}\lambda^{-1}$.
\end{proof}

Combining \eqref{eqn:onoff_diag_bounded} with \eqref{eqn:bounded_Gz}, we immediately obtain \eqref{eqn:entrywise_law_gaussian} for off-diagonals $\fa\neq\fb$.

\textbf{Step 2: Self-consistent equations.} In this step we show that $(m_{1}(z),m_{2}(z))$, defined in \eqref{eqn:def:partial_traces} satisfies the system of approximate self-consistent equations \eqref{eqn:system_eqn_m}. By \eqref{eqn:delta_bounded_by_z} and \eqref{eqn:rescale_m_a}, we know the following estimate hold for $z \in \bD$:
$$|m_{1c}(z)-m_{1c}(-\lambda)| \lesssim (\log p)^{-1}\lambda,\ \  |m_{2c}(z)-m_{2c}(-\lambda)| \lesssim (\log p)^{-1}\lambda.$$
Combining with \eqref{eqn:m12_bounded}, we know that uniformly in $z \in \bD$, 
\begin{equation}\label{eqn:m12c_z_bounded}
|m_{1c}(z)| \sim |m_{2c}(z)| \sim \lambda,\ \ |z + \lambda_i^{(1)}r_1m_{1c}(z) + \lambda_i^{(2)}r_2m_{2c}(z)| \sim \lambda.
\end{equation}
Further, combining \eqref{eqn:fgz_prototype_system} with \eqref{eqn:rescale_m_a}, we get that uniformly in $z\in \bD$, for $k=1,2$
\begin{equation}\label{m12_z_inv_bounded}
|1+\gamma  m_{0kc}(z)| = |m_{kc}^{-1}(z)| \sim \lambda^{-1},
\end{equation}
where we denote
\begin{equation}\label{eqn:def:m_0kc}
 m_{0kc}(z) := -\int \frac{\lambda^{(k)}}{z + \lambda^{(1)}m_{1c}(z) + \lambda^{(2)}m_{2c}(z)}d\hat H_p(\lambda^{(1)},\lambda^{(2)}).
\end{equation}

We will later see in fact that $m_{0kc}(z)$ are the asymptotic limits of $m_{0k}(z)$ defined in \eqref{eqn:def:partial_traces}. Applying \eqref{eqn:m12c_z_bounded} to \eqref{eqn:G_limit} and using \eqref{eqn:rescale_m_a}, we get that
\begin{equation}\label{eqn:G_limit_aa_bounded}
|\mathfrak{G}_{ii}(z)| \sim \lambda^{-1},\ \ |\mathfrak{G}_{\mu\mu}(z)| \sim \lambda,\ \ \text{uniformly for $z\in \bD$ and $i \in \mI_0,\mu \in \mI_1\cup \mI_2$}.
\end{equation}

We now define the critical $z$-dependent event:
\begin{equation}\label{eqn:z_event}
\Xi(z) := \left\{ |m_1(z) - m_{1c}(-\lambda)| + |m_1(z) - m_{1c}(-\lambda)| \leq (\log p)^{-1/2} \lambda\right\}.
\end{equation}

\eqref{eqn:m12c_z_bounded} yields that on $\Xi(z)$, 
\begin{equation}\label{eqn:m12_z_bounded}
|m_1(z)| \sim |m_2(z)| \sim \lambda,\ \ |z+ \lambda_i^{(1)}r_1m_{1}(z) + \lambda_i^{(2)}r_2m_{2}(z)| \sim \lambda.
\end{equation}

We now propose the following key lemma, which introduces the approximate self-consistent equations for $(m_1(z),m_2(z))$ on $\Xi(z)$.
\begin{lemma}\label{lemma:approximate_consistent_equation_event}
Under the setting of Lemma \ref{lemma:entrywise_law_gaussian}, the following estimates hold uniformly in $z \in \bD$ for $k=1,2$:
\begin{equation}\label{eqn:approximate_consistent_equation_event}
\begin{aligned}
\bm{1}(\Xi) &\left| \frac{1}{m_k} + 1 -\gamma \int \frac{\lambda^{(k)}}{z+\lambda^{(1)}r_1m_1+\lambda^{(2)}r_2m_2} d\hat H_p(\lambda^{(1)},\lambda^{(2)}) \right| \\
&\prec p^{-1}\lambda^{-5} + p^{-1/2}\lambda^{-4}\Theta +  |\check{\mcl{Z}}_{0k}| + |\check{\mcl{Z}}_{k}|, 
\end{aligned}
\end{equation}
where we denote
\begin{equation}\label{eqn:def:Theta}
\Theta := |m_1(z) - m_{1c}(z)| + |m_2(z)-m_{2c}(z)|
\end{equation}
and
\begin{equation}\label{eqn:def:check_z}
\begin{aligned}
\check{\mcl{Z}}_{0k}&:= \frac{1}{p} \sum_{i\in \mI_0} \frac{\lambda_i^{(k)}\mcl{Z}_i}{(z+\lambda_i^{(1)}r_1m_{1c}+\lambda_i^{(2)}r_2m_{2c})^2},\\
\check{\mcl{Z}}_1 &:= \frac{1}{n_1} \sum_{\mu \in \mI_1} \mcl{Z}_\mu,\ \ \check{\mcl{Z}}_2 := \frac{1}{n_2} \sum_{\nu \in \mI_2} \mcl{Z}_\nu.
\end{aligned}
\end{equation}
\end{lemma}
\begin{proof}
Using \eqref{eqn:complement_ii},\eqref{eqn:Z_i} and \eqref{eqn:Z_mu}, we have
\begin{align}
\frac{1}{G_{ii}} &=  -z - \lambda_i^{(1)}r_1m_1 - \lambda_i^{(2)}r_2m_2 + \mcl{E}_i,\ \ \text{for $i \in \mI_0$}, \label{eqn:G_ii_inv}\\
\frac{1}{G_{\mu\mu}} &= -1 - \gamma m_{01} + \mcl{E}_\mu,\ \ \text{for $\mu \in \mI_1$},\label{eqn:G_mumu_inv}\\
\frac{1}{G_{\nu\nu}} &= -1 - \gamma m_{02} + \mcl{E}_\nu,\ \ \text{for $\nu \in \mI_2$},\label{eqn:G_nunu_inv}
\end{align}
where we denote (recall \eqref{eqn:def:partial_traces} and Definition \ref{def:resolvent_minors})
$$\mcl{E}_i := \mcl{Z}_i + \lambda_i^{(1)}r_1(m_1 - m_1^{(i)}) + r_2(m_2-m_2^{(i)}),$$
and 
$$\mcl{E}_\mu := \mcl{Z}_\mu +\gamma (m_{01}-m_{01}^{(\mu)}),\ \ \mcl{E}_\nu := \mcl{Z}_\nu +\gamma(m_{02}-m_{02}^{(\nu)}).$$

Using equations \eqref{eqn:complement_ab}, \eqref{eqn:def:off_diag} and \eqref{eqn:onoff_diag_bounded}, we know
\begin{equation}\label{eqn:m1_minor_close}
|m_1 - m_i^{(i)}| \leq  \frac{1}{n_1} \sum_{\mu \in \mI_1} \left|\frac{G_{\mu i}G_{i\mu}}{G_{ii}}\right| \leq |\check \Lambda_0|^2|G_{ii}| \prec p^{-1}\lambda^{-3}.
\end{equation}
Similarly, we also have
\begin{equation}\label{eqn:m2_minor_close}
|m_2 - m_2^{(i)}| \prec p^{-1}\lambda^{-3},\ \ |m_{01} - m_{01}^{(\mu)}| \prec p^{-1}\lambda^{-3},\ \ |m_{02} - m_{02}^{(\nu)}| \prec p^{-1}\lambda^{-3},
\end{equation}
for $i \in \mI_0,\mu \in \mI_1,\nu \in \mI_2$. Combining \eqref{eqn:m1_minor_close} and \eqref{eqn:m2_minor_close} with \eqref{eqn:onoff_diag_bounded}, we have
\begin{equation}\label{eqn:max_E_bounded}
\max_{i \in \mI_0}|\mcl{E}_i| + \max_{\mu \in \mI_1 \cup \mI_2} |\mcl{E}_\mu| \prec p^{-1/2}\lambda^{-1}.
\end{equation}

Now from equation \eqref{eqn:G_ii_inv}, we know that on $\Xi$,
\begin{equation}\label{eqn:bound_Gii_Theta}
\begin{aligned}
G_{ii} &= -\frac{1}{z+\lambda_i^{(1)}r_1m_1+\lambda_{i}^{(2)}r_2m_2} - -\frac{\mcl{E}_i}{(z+\lambda_i^{(1)}r_1m_1+\lambda_{i}^{(2)}r_2m_2)^2} + O_\prec (p^{-1}\lambda^{-5})\\
&= -\frac{1}{z+\lambda_i^{(1)}r_1m_1+\lambda_{i}^{(2)}r_2m_2} - -\frac{\mcl{Z}_i}{(z+\lambda_i^{(1)}r_1m_1+\lambda_{i}^{(2)}r_2m_2)^2} \\
&\qquad + O_\prec (p^{-1}\lambda^{-5}+p^{-1/2}\lambda^{-4}\Theta)
\end{aligned}
\end{equation}
where in the first step we use \eqref{eqn:max_E_bounded} and \eqref{eqn:m12_z_bounded} on $\Xi$, and in the second step we use \eqref{eqn:def:Theta}, \eqref{eqn:m1_minor_close} and \eqref{eqn:onoff_diag_bounded}. Plugging \eqref{eqn:bound_Gii_Theta} into the definitions of $m_{0k}$ in \eqref{eqn:def:partial_traces} and using \eqref{eqn:def:check_z}, we get that on $\Xi$, for $k=1,2$,
\begin{equation}\label{eqn:bound_m0k_Theta}
\begin{aligned}
m_{0k} &= -\int \frac{\lambda^{(k)}}{z+\lambda^{(1)}r_1m_1+\lambda^{(2)}r_2m_2} d\hat H_p(\lambda^{(1)},\lambda^{(2)}) - \check{\mcl{Z}}_{0k} \\
&\qquad + O_\prec (p^{-1}\lambda^{-5}+p^{-1/2}\lambda^{-4}\Theta).
\end{aligned}
\end{equation}
Comparing the equations above with \eqref{eqn:def:m_0kc} and applying \eqref{eqn:def:check_z} and \eqref{eqn:onoff_diag_bounded}, we get
\begin{equation}\label{eqn:m0k_m0kc_close_logp}
|m_{01}(z)-m_{01c}(z)| + |m_{02}(z)-m_{02c}(z)| \lesssim (\log p)^{-1/2}\lambda^{-1}, \ \ \text{w.o.p.} \ \ \text{on} \ \ \Xi.
\end{equation}
Together with equation \eqref{m12_z_inv_bounded}, we have
\begin{equation}\label{eqn:gamma_m0k_bounded}
|1+\gamma m_{01}(z)| \sim \lambda^{-1},\ \ |1+\gamma m_{01}(z)| \sim \lambda^{-1} \ \ \text{w.o.p.} \ \ \text{on} \ \ \Xi.
\end{equation}
With a similar argument, from equations \eqref{eqn:G_mumu_inv}, we know that on $\Xi$,
\begin{equation}\label{eqn:Gmumu_close_Z}
G_{\mu\mu} = -\frac{1}{1+\gamma m_{0k}} - \frac{\mcl{Z}_\mu}{(1+\gamma m_{0k})^2} + O_\prec(p^{-1}\lambda^{-1}),\ \ \mu \in \mI_k,\ \ k=1,2,
\end{equation}
where we used \eqref{eqn:m1_minor_close}, \eqref{eqn:max_E_bounded} and \eqref{eqn:gamma_m0k_bounded}. Taking the average of \eqref{eqn:Gmumu_close_Z} over $\mu \in \mI_k$, we get that on $\Xi$,
\begin{equation}\label{eqn:m1_close_check_z}
m_k = -\frac{1}{1+\gamma m_{0k}} - \frac{\check{\mcl{Z}}_{k}}{(1+\gamma m_{0k})^2} + O_\prec(p^{-1}\lambda^{-1}),\ \ k=1,2,
\end{equation}
which implies that on $\Xi$, 
\begin{equation}\label{eqn:m1_inv_close_check_z}
\frac{1}{m_k} + 1 + \gamma m_{0k} \prec p^{-1}\lambda^{-3} + |\check{\mcl{Z}_{k}}|,
\end{equation}
where we used \eqref{eqn:m12_z_bounded} and \eqref{eqn:gamma_m0k_bounded}. Finally, plugging \eqref{eqn:bound_m0k_Theta} into \eqref{eqn:m1_inv_close_check_z}, we get \eqref{eqn:approximate_consistent_equation_event}.
\end{proof}

\textbf{Step 3: Entrywise Local Law.} In this step, we show that the event $\Xi(z)$ in \eqref{eqn:z_event} actually holds with high probability, and then we apply Lemma \ref{lemma:approximate_consistent_equation_event} to conclude the entrywise local law \eqref{eqn:entrywise_law_gaussian}.
We first claim that it suffices to show
\begin{equation}\label{eqn:m1_m1c_lambda_close}
|m_1(-\lambda)-m_{1c}(-\lambda)| + |m_2(-\lambda)-m_{2c}(-\lambda)| \prec p^{-1/2}\lambda^{-2}.
\end{equation}
In fact, by \eqref{eqn:G_spectral_F} and similar arguments as in the proof of Lemma \ref{lemma:close_Gz}, we know that uniformly for $z \in \bD$, w.o.p.,
\begin{equation*}
|m_1(z)-m_1(-\lambda)| + |m_2(z)-m_2(-\lambda)| \lesssim |z+\lambda| \leq (\log p)^{-1}\lambda.
\end{equation*}
Thus, if \eqref{eqn:m1_m1c_lambda_close} holds, we can obtain with triangle inequality that
\begin{equation}\label{eqn:Xi_holds}
\sup_{z\in\bD} (|m_1(z)-m_{1c}(-\lambda)| + |m_2(z)-m_{2c}(-\lambda)|) \lesssim (\log p)^{-1} \lambda \ \ \text{w.o.p.,}
\end{equation}
which confirms that $\Xi$ holds w.o.p., and also verifies the condition \eqref{eqn:delta_bounded_perturbed} of Lemma \ref{lemma:consistent_equation_with_error}. Now applying Lemma \ref{lemma:consistent_equation_with_error} to \eqref{eqn:approximate_consistent_equation_event}, we get
\begin{equation*}
\begin{aligned}
\Theta(z) &= |m_1(z) - m_{1c}(z)| + |m_2(z)-m_{2c}(z)|\\
&\prec p^{-1}\lambda^{-4} + p^{-1/2}\lambda^{-3}\Theta + \lambda(|\check{\mcl{Z}}_1| + |\check{\mcl{Z}}_2| + |\check{\mcl{Z}}_{01}| + |\check{\mcl{Z}}_{02}|),
\end{aligned}
\end{equation*}
which implies 
\begin{equation}\label{eqn:Theta_bounded}
\Theta(z) \prec p^{-1}\lambda^{-4} + \lambda(|\check{\mcl{Z}}_1| + |\check{\mcl{Z}}_2| + |\check{\mcl{Z}}_{01}| + |\check{\mcl{Z}}_{02}|) \prec p^{-1/2}\lambda^{-2}
\end{equation}
uniformly for $z \in \bD$, where we used \eqref{eqn:onoff_diag_bounded}. On the other hand, with equation \eqref{eqn:Gmumu_close_Z} and \eqref{eqn:m1_close_check_z}, we know
\begin{equation*}
\max_{\mu\in \mI_1} |G_{\mu\mu}(z) - m_1(z)| + \max_{\nu\in \mI_2} |G_{\nu\nu}(z) - m_2(z)| \prec p^{-1}\lambda^{-1},
\end{equation*}
which, combined with \eqref{eqn:Theta_bounded}, yields
\begin{equation}\label{def:gmumu_m1c_close}
\max_{\mu\in \mI_1} |G_{\mu\mu}(z) - m_{1c}(z)| + \max_{\nu\in \mI_2} |G_{\nu\nu}(z) - m_{2c}(z)| \prec p^{-1/2}\lambda^{-2}.
\end{equation}
Next, plugging \eqref{eqn:Theta_bounded} into \eqref{eqn:bound_Gii_Theta} and recalling \eqref{eqn:G_limit} and \eqref{eqn:rescale_m_a}, we get that
\begin{equation*}
\max_{i \in \mI_0}|G_{ii}(z) - \mathfrak{G}_{ii}(z)| \prec p^{-1/2}\lambda^{-3}.
\end{equation*}
Together with \eqref{def:gmumu_m1c_close}, we have the diagonal estimate
\begin{equation}\label{eqn:entrywise_gaussian_diagonal}
\max_{\fa \in \mI}|G_{\fa\fa}(z) - \mathfrak{G}_{\fa\fa}(z)| \prec p^{-1/2}\lambda^{-3},
\end{equation}
which, when combined with the previously established off-diagonal estimate, we conclude the entrywise local law \eqref{eqn:entrywise_law_gaussian}.

Thus we are left with proving \eqref{eqn:m1_m1c_lambda_close}. Using \eqref{eqn:G_spectral_F} and \eqref{eqn:bounded_Gz}, we know $m_{0k}(-\lambda) \sim \lambda^{-1}$ for $k=1,2$. Therefore we have
\begin{equation}\label{eqn:gamma_m0k_lambda_bounded}
1 + \gamma m_{0k}(-\lambda) \sim \lambda^{-1},\ \ k=1,2.
\end{equation}
Combining these estimates with \eqref{eqn:G_mumu_inv}, \eqref{eqn:G_nunu_inv} and \eqref{eqn:max_E_bounded}, we conclude that \eqref{eqn:m1_close_check_z} hold at $z=-\lambda$ without requiring $\Xi(0)$. This further gives that w.o.p.,
\begin{align*}
    |\lambda_i^{(1)}r_1m_1(-\lambda)+\lambda_i^{(2)}r_2m_2(-\lambda)| &= \left|\frac{\lambda_i^{(1)}r_1}{1+\gamma m_{01}(-\lambda)} + \frac{\lambda_i^{(2)}r_1}{1+\gamma m_{02}(-\lambda)} + O_\prec(p^{-1/2}\lambda^{-1})\right| \\
    &\sim \lambda.
\end{align*}
Combining the above estimate with \eqref{eqn:G_ii_inv} and \eqref{eqn:max_E_bounded}, we obtain that \eqref{eqn:bound_m0k_Theta} also hold at $z=-\lambda$ without requiring $\Xi(0)$. Finally, plugging \eqref{eqn:bound_m0k_Theta} into \eqref{eqn:m1_inv_close_check_z}, we conclude that \eqref{eqn:approximate_consistent_equation_event} hold at $z=0$.

Also, combining \eqref{eqn:m1_close_check_z} at $z=-\lambda$ with \eqref{eqn:gamma_m0k_lambda_bounded}, we know there exists constants $c,C>0$ such that 
\begin{equation}\label{eqn:_mk_lambda_bounded}
c\lambda \leq a_{1\epsilon}(-\lambda):= -r_1m_1(-\lambda) \leq C\lambda,\ \ c\lambda \leq a_{2\epsilon}(-\lambda):= -r_2m_2(-\lambda) \leq C\lambda,\ \ \text{w.o.p.}
\end{equation}
Consequently, we know $(a_1(-\lambda),a_2(-\lambda))$ satisfies \eqref{eqn:fg_prototype_system} and $(a_{1\epsilon}(-\lambda),a_{2\epsilon}(-\lambda))$ satisfies a similar system with right-hand-sides replaced by $O_\prec(p^{-1/2}\lambda^{-3})$. Denote $\bF(\ba)=[f(a_1,a_2),g(a_1,a_2)]$ for $\ba = (a_1,a_2)$ and $f,g$ defined in \eqref{eqn:fg_prototype_system}, and further denote $\bJ$ to be the Jacobian of $\bF$. Similar to the derivation of \eqref{eqn:Jinv_op_bounded}, because of \eqref{eqn:_mk_lambda_bounded} and \eqref{eqn:a1_a2_bounded}, we know $\|\bJ^{-1}\|_{\op} = O(\lambda)$ at both $(a_1(-\lambda),a_2(-\lambda))$ and $(a_{1\epsilon}(-\lambda),a_{2\epsilon}(-\lambda))$. Hence we obtain that 
\begin{equation*}
|m_1(-\lambda)-m_{1c}(-\lambda)| \prec p^{-1/2}\lambda^{-2},\ \ |m_2(-\lambda)-m_{2c}(-\lambda)| \prec p^{-1/2}\lambda^{-2},
\end{equation*}
thereby yielding \eqref{eqn:m1_m1c_lambda_close}.

\textbf{Step 4: Averaged Local Law.} Now we prove the averaged local laws \eqref{eqn:averaged_law}. We have the following \textit{fluctuation averaging} lemma:

\begin{lemma}[Fluctuation averaging]\label{lemma:fluctuation_averaging}
Under the setting of Proposition \ref{prop:anisotropic_law_gaussian}, suppose the entrywise law holds uniformly in $z \in \bD$, then we have
\begin{equation}\label{eqn:check_z_bounded_strong}
|\check{\mcl{Z}}_1| + |\check{\mcl{Z}}_2| + |\check{\mcl{Z}}_{01}| + |\check{\mcl{Z}}_{02}| \prec p^{-1}\lambda^{-6}
\end{equation}
uniformly in $z\in\bD$.
\end{lemma}
\begin{proof}
The proof exactly follows \cite[Theorem 4.7]{erdHos2013local}.
\end{proof}

Now plugging in \eqref{eqn:check_z_bounded_strong} into \eqref{eqn:Theta_bounded}, we get
\begin{equation}\label{eqn:m1_m1c_close_strong}
|m_1(z)-m_{1c}(z)|+|m_2(z)-m_{2c}(z)| \prec p^{-1}\lambda^{-5}.
\end{equation}
Then, substracting \eqref{eqn:def:m_0kc} from \eqref{eqn:bound_m0k_Theta} and applying \eqref{eqn:check_z_bounded_strong} and \eqref{eqn:m1_m1c_close_strong}, we get
\begin{equation*}
|m_{0k}(z)-m_{0kc}(z)| \prec p^{-1}\lambda^{-7},
\end{equation*}
which is exactly the averaged local law.
\end{proof}

\subsubsection{Non-Gaussian extension}
Now with Proposition \ref{prop:anisotropic_law_gaussian}, we can extend from Gaussian random matrices to generally distributed random matrices using the exact same proof as \cite[Sections B.3.4 and B.3.5]{yang2020analysis}, since it longer involves tracking powers of $\lambda$, we omit the details here.

\subsection{Proof of Theorem \ref{thm:design_shift_ridge}}
First, using a standard truncation argument (cf.~\cite[proof of Theorem 3]{yang2020analysis}), we know that Theorem \ref{thm:anisotropic_law} holds for $Q=p^{2/\varphi}$ with high probability if $(\bZ^{(1)},\bZ^{(2)})$ now satisfy the bounded moment condition in Assumption \ref{as:design_four_epsilon}.

We first prove the variance term \eqref{eqn:design_shift_ridge_V}. Define the auxiliary variable
\begin{equation}\label{eqn:def:ridge_check_V}
\begin{aligned}
\check V(\lambda) &:= \frac{\sigma^2\lambda}{n}\Tr\biggl((\bLambda^{(2)})^{1/2}\Bigl((\bLambda^{(1)})^{1/2}\bV^\top\hat\bSigma_{\bZ}^{(1)}\bV(\bLambda^{(1)})^{1/2} \\
&\hspace{10em}+(\bLambda^{(2)})^{1/2}\bV^\top\hat\bSigma_{\bZ}^{(2)}\bV(\bLambda^{(2)})^{1/2} +\lambda\bI\Bigr)^{-1}(\bLambda^{(2)})^{1/2}\biggr) \\
&= \frac{\sigma^2\lambda}{n}\sum_{i \in \mI_0} \lambda_i^{(2)}G_{ii}(-\lambda).
\end{aligned}
\end{equation}
Since \eqref{eqn:fgz_prototype_system} reduces to the first two equations in \eqref{eqn:design_shift_ridge_alpha_eqns} at $z=-\lambda$, then by \eqref{eqn:averaged_law}, with high probability,
\begin{equation*}
\begin{aligned}
\check V(\lambda) &= \frac{\sigma^2\lambda}{n}  \sum_{i \in \mI_0}\mathfrak{G}_{ii}(-\lambda) + O(p^{-1}\lambda^{-6})\\
&= \sigma^2\gamma \int \frac{\lambda\lambda^{(2)}}{a_1\lambda^{(1)}+a_2\lambda^{(2)}+\lambda} d\hat H_p(\lambda^{(1)},\lambda^{(2)}) + O(p^{-1}\lambda^{-6})  \\
&:=\overline{V}(\lambda)+O(p^{-1}\lambda^{-6})
\end{aligned}
\end{equation*}
where $(a_1,a_2)$ is defined in \eqref{eqn:design_shift_ridge_alpha_eqns}.

Fix $t>0$. Let $\bA := \hat\bSigma + \lambda\bI, \bB := \hat\bSigma +(\lambda+t\lambda)\bI$. Then from \eqref{eqn:variance:analytical_ridge_antider},
\begin{align*}
&\left|V(\hat\bbeta_{\lambda};\bbeta^{(2)}) - \frac{1}{t\lambda}(\check V(\lambda+t\lambda)-\check V(\lambda))\right|\\
=& \left|\frac{\sigma^2}{n} \Tr\left(\bSigma^{(2)} \hat\bSigma \left( \bA^{-2} - \frac{1}{t\lambda}(\bA^{-1} - \bB^{-1})\right)\right)\right|\\
\lesssim& \left\| \bSigma^{(2)} \hat\bSigma \left( \bA^{-2} - \frac{1}{t\lambda}(\bB^{-1}(\bB-\bA)\bA^{-1})\right)\right\|_{\op}\\
\lesssim& \left\| \hat\bSigma \left( \bA^{-2} - \bB^{-1}\bA^{-1}\right)\right\|_{\op}\\
= & t\lambda \|\hat\bSigma \bB^{-1}\bA^{-2}\|_{\op}\\
=& O(t\lambda^{-2})
\end{align*}
with high probability, where in the last line we used Corollary \ref{cor:ESD_bounded_weaker}.

Similarly (but without randomness), we have
\begin{align*}
&\left|\sigma^2\gamma\int \frac{\lambda^{(1)}\lambda^{(2)}(a_1-a_3\lambda)+(\lambda^{(2)})^2(a_2-a_4\lambda)}{(a_1\lambda^{(1)}+a_2\lambda^{(2)}+\lambda)^2} d \hat H_p(\lambda^{(1)},\lambda^{(2)}) - \frac{1}{t\lambda}(\overline V(\lambda+t\lambda)-\overline V(\lambda))\right|\\
&=O(t\lambda^{-2}),
\end{align*}
where $(a_3,a_4):=\frac{\partial}{\partial\lambda}(a_1,a_2)$ satisfies the last two equations in \eqref{eqn:design_shift_ridge_alpha_eqns}. Combining the previous two displays and setting $t=p^{-1/2}\lambda^{-5/2}$, we get
\begin{align*}
&\left|V(\hat\bbeta_{\lambda};\bbeta^{(2)})-\sigma^2\gamma\int \frac{\lambda^{(1)}\lambda^{(2)}(a_1-a_3\lambda)+(\lambda^{(2)})^2(a_2-a_4\lambda)}{(a_1\lambda^{(1)}+a_2\lambda^{(2)}+\lambda)^2} d \hat H_p(\lambda^{(1)},\lambda^{(2)}) \right|\\
\leq & \frac{1}{t\lambda} \left( |\check V(\lambda)-\overline V(\lambda)| + |\check V(\lambda+t\lambda)-\overline V(\lambda+t\lambda)|\right) + O(t\lambda^{-2})\\
=& O(t^{-1}p^{-1}\lambda^{-7}) + O(t\lambda^{-2})\\
=& O(p^{-1/2}\lambda^{-9/2}) \prec p^{-1/2+c}\lambda^{-9/2}
\end{align*}
for any small constant $c>0$, thereby obtaining \eqref{eqn:design_shift_ridge_V}.

Similarly, for the bias term, define the auxiliary variables
\begin{equation}\label{def:check_B}
\begin{aligned}
\check B(\eta) &=\lambda\bbeta_\eta^{(2)\top}\Bigl((\bLambda_\eta^{(1)})^{1/2}\bV^\top\hat\bSigma_{\bZ}^{(1)}\bV(\bLambda_\eta^{(1)})^{1/2}\\
&\hspace{6em}+(\bLambda_\eta^{(2)})^{1/2}\bV^\top\hat\bSigma_{\bZ}^{(2)}\bV(\bLambda_\eta^{(2)})^{1/2} +\lambda\bI\Bigr)^{-1}\bbeta_\eta^{(2)}, \\
\overline B(\eta) &= \|\bbeta^{(2)}\|_2^2 \cdot \int \frac{\lambda}{b_1\lambda^{(1)}+b_2\lambda^{(2)}+\lambda(1+\eta\lambda^{(2)})} d \hat G_p(\lambda^{(1)},\lambda^{(2)}),
\end{aligned}
\end{equation}
where $(b_1,b_2)$ is defined in \eqref{eqn:design_shift_ridge_beta_eqns}, then we know by \eqref{eqn:anisotropic_law} that with high probability,
\begin{equation*}
\check B(\eta) = \overline B(\eta) + O(p^{-1/2}\lambda^{-3})
\end{equation*}

Let $\eta > 0$ and denote $\bA := \hat\bSigma + \lambda\bI, \bB := \hat\bSigma +\lambda\bI+\lambda\eta\bSigma^{(2)}$. Then from \eqref{eqn:bias:analytical_bias_antider},
\begin{align*}
&\left| B(\hat\bbeta_{\lambda};\bbeta^{(2)}) - \frac{1}{\eta}(\check B(0)-\check B(\eta))\right|\\
=& \left|\lambda\bbeta^{(2)\top}\left(\lambda\bA^{-1}\bSigma^{(2)}\bA^{-1}-\frac{1}{\eta}(\bA^{-1}-\bB^{-1})\right)\bbeta^{(2)} \right|\\
\lesssim & \lambda \left\| \left(\lambda\bA^{-1}\bSigma^{(2)}\bA^{-1}-\frac{1}{\eta}(\bB^{-1}(\bB-\bA)\bA^{-1})\right) \right\|_{\op}\\
=& \lambda^2 \|(\bA^{-1}-\bB^{-1})\bSigma^{(2)}\bA^{-1}\|_{\op}\\
=& \lambda^3\eta \|\bB^{-1}\bSigma^{(2)}\bA^{-1}\bSigma^{(2)}\bA^{-1}\|_{\op}\\
=& O(\eta).
\end{align*}
with high probability. Similarly, we have
\begin{equation*}
\left| \|\bbeta^{(2)}\|_2^2 \cdot \int \frac{b_3\lambda\lambda^{(1)}+(b_4+\lambda)\lambda\lambda^{(2)}}{(b_1\lambda^{(1)}+b_2\lambda^{(2)}+\lambda)^2} d \hat G_p(\lambda^{(1)},\lambda^{(2)}) - \frac{1}{\eta}(\overline B(0) - \overline B(\eta))\right| + O(\eta),
\end{equation*}
where $(b_3,b_4):=\frac{\partial}{\partial \eta}(b_1,b_2)$ satisties the last two equations in \eqref{eqn:design_shift_ridge_beta_eqns}. Combining the previous two displays and setting $\eta = p^{-1/4}\lambda^{-3/2}$, we get
\begin{align*}
& \left|B(\hat\bbeta_{\lambda};\bbeta^{(2)}) -  \|\bbeta^{(2)}\|_2^2 \cdot \int \frac{b_3\lambda\lambda^{(1)}+(b_4+\lambda)\lambda\lambda^{(2)}}{(b_1\lambda^{(1)}+b_2\lambda^{(2)}+\lambda)^2} d \hat G_p(\lambda^{(1)},\lambda^{(2)})\right|\\
\leq & \frac{1}{\eta} \left( |\check B(0)-\overline B(0)| + |\check B(\eta) - \overline B(\eta)| \right) + O(\eta)\\
=& O(\eta^{-1}p^{-1/2}\lambda^{-3}) + O(\eta)\\
=& O(p^{-1/4}\lambda^{-3/2}) \prec p^{-1/4+c}\lambda^{-3/2}
\end{align*}
for any small constant $c>0$, thereby obtaining \eqref{eqn:design_shift_ridge_B}.

\subsection{Proof of Theorem \ref{thm:design_shift}}\label{subsec:proof:design_shift}
The proof involves the following three components.
\begin{enumerate}
 \item[(i)] First, we prove that the limiting risk of the ridge estimator is close to the limiting risk of the interpolator.
 
 We first consider the limits of the variances, which are defined respectively as
 \begin{equation}\label{eqn:def:design_shift_limits_V}
 \begin{aligned}
 \mathcal{V}^{(d)}(\hat\bbeta) &:= -\sigma^2\gamma\int \frac{\lambda^{(2)}(\tilde a_3\lambda^{(1)}+\tilde a_4\lambda^{(2)})}{(\tilde a_1\lambda^{(1)}+\tilde a_2\lambda^{(2)}+1)^2} d \hat H_p(\lambda^{(1)},\lambda^{(2)})\\
 \mathcal{V}^{(d)}(\hat\bbeta_\lambda)&:= \sigma^2\gamma\int \frac{\lambda^{(1)}\lambda^{(2)}(a_1-a_3\lambda)+(\lambda^{(2)})^2(a_2-a_4\lambda)}{(a_1\lambda^{(1)}+a_2\lambda^{(2)}+\lambda)^2} d \hat H_p(\lambda^{(1)},\lambda^{(2)})
 \end{aligned}
 \end{equation}

 To this end, we first show that 
 \begin{equation}\label{eqn:a_tildea_close}
     \check a_k := a_k / \lambda = \tilde a_k + O(\lambda),\ \ k=1,2.
 \end{equation}
 From \eqref{eqn:a_tildea_close}, we know the first two equations of \eqref{eqn:design_shift_ridge_alpha_eqns} becomes 
 \begin{equation}\label{eqn:design_shift_ridge_alpha_eqns_rescaled}
 \begin{aligned}
  \check f(\check a_1,\check a_2) &:= r_1 - \gamma \int \frac{\check a_1 \lambda^{(1)}}{\check a_1 \lambda^{(1)} + \check a_2 \lambda^{(2)} + 1} d \hat H_p(\lambda^{(1)},\lambda^{(2)})= \check a_1 \lambda,\\
  \check g(\check a_1,\check a_2) &:= r_2 - \gamma \int \frac{\check a_2 \lambda^{(1)}}{\check a_1 \lambda^{(1)} + \check a_2 \lambda^{(2)} + 1} d \hat H_p(\lambda^{(1)},\lambda^{(2)}) = \check a_2 \lambda.\\
 \end{aligned}
 \end{equation}
 Let $\check \bF(\check \ba) :=(\check f(\check \ba),\check g(\check \ba))$ where $\check \ba:=(\check a_1,\check a_2)$. From \eqref{eqn:a_tildea_close} and \eqref{eqn:a1_a2_bounded}, we know there exists some constants $c,C>0$ such that $c \leq \check a_1,\check a_2 \leq C$. This yields $\check \bF(\check \ba) = (\check a_1\lambda,\check a_2\lambda) \sim \lambda$. Also, similar to how we derived \eqref{eqn:a1_a2_bounded}, we can also show that there exists constants $c,C>0$ such that 
 \begin{equation}\label{eqn:tilde_a_bounded}
     c \leq \tilde a_1,\tilde a_2 \leq C
 \end{equation}. On the other hand, similar to how we obtained \eqref{eqn:def:Jacobian} and \eqref{eqn:Jinv_op_bounded}, let $\check \bJ$ be the Jacobian of $\check \bF$, we can similarly show that $\|\bJ(\ba)\|_{\op} = O(1)$ for all $\ba \in [c,C]^2$. Together with the fact that $\check \bF(\tilde \ba) = 0$ from \eqref{eqn:design_shift_a_eqns}, we have obtained \eqref{eqn:a_tildea_close}.

Now taking the derivative of  \eqref{eqn:design_shift_ridge_alpha_eqns_rescaled} with respect to $\lambda$, we get
\begin{equation}\label{eqn:design_shift_ridge_alpha_eqns_der}
\begin{aligned}
\check a_1 + \check a_1'\lambda &= -\gamma\int \frac{\check a_1'\lambda^{(1)}+\lambda^{(1)}\lambda^{(2)}(\check a_1'\check a_2-\check a_2'\check a_1)}{(\check a_1\lambda^{(1)}+\check a_2\lambda^{(2)}+1)^2}d\hat H_p(\lambda^{(1)},\lambda^{(2)}),\\
\check a_2 + \check a_2'\lambda &= -\gamma\int \frac{\check a_2'\lambda^{(2)}+\lambda^{(1)}\lambda^{(2)}(\check a_2'\check a_1-\check a_1'\check a_2)}{(\check a_1\lambda^{(1)}+\check a_2\lambda^{(2)}+1)^2}d\hat H_p(\lambda^{(1)},\lambda^{(2)}),
\end{aligned}
\end{equation}

where $(a_1',a_2')$ represents respective derivatives with respect to $\lambda$. Since $c \leq \check a_1,\check a_2 \leq C$, we know $(\check a_1',\check a_2') = O(\lambda^{-1})$, which means the RHS of \eqref{eqn:design_shift_ridge_alpha_eqns_der} is $O(1)$. Thus the LHS is also $O(1)$, which yields  $(\check a_1',\check a_2') = O(1) \Rightarrow (\check a_1'\lambda,\check a_2'\lambda)=O(\lambda)$. Now combining \eqref{eqn:design_shift_ridge_alpha_eqns_der} and the last two equations of \eqref{eqn:design_shift_a_eqns}, we can use similar analysis to conclude that 
\begin{equation}\label{eqn:tildea3_checka1p_close}
    \check a_1' = \tilde a_3 + O(\lambda),\ \ \check a_2' = \tilde a_4 + O(\lambda)
\end{equation}

Finally, from the second equation of \eqref{eqn:def:design_shift_limits_V} and the fact that $a_k = \check a_k \lambda \Rightarrow a_{k+2}:=a_k' = \check a_k+\check a_k' \lambda$, we know
\begin{equation*}
 \mathcal{V}^{(d)}(\hat\bbeta_\lambda):= -\sigma^2\gamma\int \frac{\lambda^{(2)}(\check a_1'\lambda^{(1)}+\check a_2'\lambda^{(2)})}{(\check a_1\lambda^{(1)}+\check a_2\lambda^{(2)}+1)^2} d \hat H_p(\lambda^{(1)},\lambda^{(2)}),
\end{equation*}
which, when combined with \eqref{eqn:a_tildea_close} and \eqref{eqn:tildea3_checka1p_close}, yields 
\begin{equation}\label{eqn:design_shift_V_limits_close}
\mathcal{V}^{(d)}(\hat\bbeta_\lambda) = \mathcal{V}^{(d)}(\hat\bbeta) + O(\lambda).
\end{equation}

Using similar analysis we can also show that the bias limits, defined respectively as
\begin{equation}\label{eqn:def:design_shift_limits_B}
\begin{aligned}
\mcl{B}^{(d)}(\hat\bbeta) &:= \|\bbeta^{(2)}\|_2^2 \cdot\int \frac{\tilde b_3\lambda^{(1)}+(\tilde b_4+1)\lambda^{(2)}}{(\tilde b_1\lambda^{(1)}+\tilde b_2\lambda^{(2)}+1)^2} d \hat G_p(\lambda^{(1)},\lambda^{(2)}) \\
\mcl{B}^{(d)}(\hat\bbeta_\lambda)&:=\|\bbeta^{(2)}\|_2^2 \cdot \int \frac{b_3\lambda\lambda^{(1)}+(b_4+\lambda)\lambda\lambda^{(2)}}{(b_1\lambda^{(1)}+b_2\lambda^{(2)}+\lambda)^2} d \hat G_p(\lambda^{(1)},\lambda^{(2)}),
\end{aligned}
\end{equation}
are close:
\begin{equation}\label{eqn:design_shift_B_limits_close}
\mathcal{B}^{(d)}(\hat\bbeta_\lambda) = \mathcal{B}^{(d)}(\hat\bbeta) + O(\lambda).
\end{equation}

\item[(ii)] Next, we show that the ridge risk is close to the interpolator risk with high probability. \par
We begin by showing $V(\hat\bbeta;\bbeta^{(2)})$ and $V(\hat\bbeta_\lambda;\bbeta^{(2)}) $ are \textit{close}, where the quantities are defined in \eqref{eqn:variance_analytical} and \eqref{eqn:variance_analytical_ridge} respectively. To this end, denote the singular value decomposition of $\hat \bSigma$ by $\hat \bSigma= \bU \bD \bU^\top$. Then the quantities \eqref{eqn:variance_analytical} and \eqref{eqn:variance_analytical_ridge} can be rewritten as
\begin{align*}
    V(\hat\bbeta;\bbeta^{(2)})&= \frac{\sigma^2}{n} \Tr(\bU\bD^{-1} \bm{1}_{\bD>0}\bU^\top\bSigma^{(2)})\\
    V(\hat\bbeta_\lambda;\bbeta^{(2)})&= \frac{\sigma^2}{n} \Tr(\bU(\bD+\lambda \bI)^{-2} \bD \bU^\top \bSigma^{(2)}),
\end{align*}
where $\bm{1}_{\bD>0}$ is the diagonal matrix with $(i,i)$-th entry equals $1$ if $D_{ii}>0$ and $0$ otherwise. Therefore, we have,

\begin{equation}
\begin{aligned}\label{eqn:design_shift_V_originals_close}
    &|V(\hat\bbeta;\bbeta^{(2)})-V(\hat\bbeta_\lambda;\bbeta^{(2)})|\\
    \le& \frac{\sigma^2}{n} \|\bU^\top \bSigma^{(2)} \bU\|_{\op} \Tr(\bD^{-1} \bm{1}_{\bD>0}-(\bD+\lambda \bI)^{-2} \bD)\\
    \lesssim& \frac{\sigma^2}{n}\Tr(\bD^{-1} \bm{1}_{\bD>0}-(\bD+\lambda \bI)^{-2} \bD) \\
    \lesssim& \frac{\lambda \sigma^2}{\lambda^2_{\min}(\hat\Sigma)}= O(\lambda),
\end{aligned}
\end{equation}
where $\lambda_{\min}$ denotes the minimum nonzero eigenvalue. Here, the second inequality follows from the fact that $\|\bSigma^{(2)}\|_{\op} \lesssim 1$, the third inequality holds because $x^{-1}-(x+\lambda)^{-2}x\le 2\lambda/x^2$ for $x>0$, and the final inequality is given by the following lemma.

\begin{lemma}\label{lemma:min_eigenvalue_heterogeneous}
Consider the same setup as Theorem \ref{thm:design_shift}. Recall $\hat\bSigma = \frac{1}{n}(\bX^{(1)\top}\bX^{(1)} + \bX^{(2)\top}\bX^{(2)})$. Let $\lambda_{\min}(\hat\bSigma)$ denote the smallest nonzero eigenvalue of $\hat\bSigma$, then there exists a positive constant $c$ such that with high probability, 
\begin{equation}\label{eqn:min_eval_heterogeneous}
\lambda_{\min}^2(\hat\bSigma) \ge c
\end{equation}
\end{lemma}
\begin{proof}
We know $\lambda_{\min}(\hat\bSigma )= \lambda_{\min}\left(\frac{1}{n}\bX\bX^\top\right)$ where $$\bX^\top = [(\bSigma^{(1)})^{1/2}\bZ^{(1)\top},(\bSigma^{(2)})^{1/2}\bZ^{(2)\top}]$$ 
and $\bZ^{(1)},\bZ^{(2)}$ satisfies Assumption \ref{as:design_four_epsilon}. From \cite[Theorem 3.12]{brailovskaya2022universality}, we know that with high probability, 
\begin{equation}\label{eqn:X_free_eval_close}
\lambda_{\min}\left(\frac{1}{n}\bX\bX^\top\right) \geq \lambda_{\min}\left(\frac{1}{n}\bX_{G}\bX_{G}^\top\right) - o(1),
\end{equation}
where $\bX_G^\top = [(\bSigma^{(1)})^{1/2}\bZ_G^{(1)\top},(\bSigma^{(2)})^{1/2}\bZ_G^{(2)\top}]$ and $\bZ_{G}^{(1)},\bZ_G^{(2)}$ have i.i.d. standard Gaussian entries instead. Since we assumed $\bSigma^{(1)}$ and $\bSigma^{(2)}$ are simultaneously diagonalizable, we further have, by rotational invariance, 
$$\lambda_{\min}\left(\frac{1}{n}\bX_G\bX_G^\top\right) = \lambda_{\min} \left(\frac{1}{n}\check \bZ\check\bZ^\top\right)$$ where $\check\bZ^\top = [(\bLambda^{(1)})\bZ^{(1)\top},(\bLambda^{(2)})\bZ^{(2)\top}]$ has (blockwise) diagonal covariances. From \cite[Corollary 1.5]{parmaksiz2025computingextremesingularvalues}, we obtain that with high probability, $\lambda_{\min}^2\left(\frac{1}{n}\check \bZ\check\bZ^\top\right) = \check \lambda + o(1)$ where $\check \lambda$ is the solution to the following variational problem:
\begin{equation}\label{eqn:variational}
\begin{aligned}
\check \lambda &:= \sup_{x_1,...,x_n < 0} \min_{1\leq i \leq n} \left\{ \frac{1}{x_i} + \sum_{j=1}^p \frac{\frac{1}{n}\lambda_j^{(1)}}{1 - \sum_{k=1}^{n_1} \frac{1}{n}\lambda_j^{(1)}x_k + \sum_{k=n_1+1}^{n}\frac{1}{n}\lambda_{j}^{(2)}x_k}\right\}  \\
&=\sup_{\alpha_1,\alpha_2>0} \min_{k=1,2} \left\{-\frac{r_k}{\alpha_k}+\gamma \int \frac{\lambda^{(k)}}{\alpha_1\lambda^{(1)}+\alpha_2\lambda^{(2)}+1} d\hat H_p(\lambda^{(1)},\lambda^{(2)})\right\},
\end{aligned}
\end{equation}
where recall that $r_k=n_k/n$, $k=1,2$ and $d \hat H_p$ is defined by \eqref{eqn:joint_ESD}. The second equality above the fact that the first line is maximized at $x_1=...=x_{n_1},x_{n_1+1}=...=x_{n}$, and $\alpha_1:=-r_1x_1,\alpha_2:=-r_2x_{n_1+1}$.

Recall the definition of $(\tilde a_1,\tilde a_2)$ in \eqref{eqn:design_shift_a_eqns}. For $k=1,2$, we set $\alpha_k = 2\tilde a_k$, we have
\begin{align*}
&-\frac{r_k}{\alpha_k}+\gamma \int \frac{\lambda^{(k)}}{\alpha_1\lambda^{(1)}+\alpha_2\lambda^{(2)}+1} d\hat H_p(\lambda^{(1)},\lambda^{(2)})\\
=&-\frac{1}{\alpha_k}\left( r_k - \gamma \int \frac{\tilde a_k\lambda^{(k)}}{\tilde a_1\lambda^{(1)}+\tilde a_2\lambda^{(2)}+\frac{1}{2}} d\hat H_p(\lambda^{(1)},\lambda^{(2)})\right)\\
=& \frac{\gamma}{\tilde a_k} \int \frac{\frac{1}{2}\tilde a_k \lambda^{(k)}}{(\tilde a_1\lambda^{(1)}+\tilde a_2\lambda^{(2)}+\frac{1}{2})(\tilde a_1\lambda^{(1)}+\tilde a_2\lambda^{(2)}+1)}d\hat H_p(\lambda^{(1)},\lambda^{(2)})\\
\geq & c
\end{align*}
for some constant $c>0$, where the last line follows from \eqref{eqn:tilde_a_bounded}. Now combining with \eqref{eqn:variational} and \eqref{eqn:X_free_eval_close}, we obtain \eqref{eqn:min_eval_heterogeneous}, thereby finishing the proof.
\end{proof}

Now we turn to show $B(\hat\bbeta_\lambda;\bbeta^{(2)})$ and $B(\hat\bbeta_\lambda;\bbeta^{(2)})$ are \textit{close}, where the quantities are defined via \eqref{eqn:bias_analytical} and \eqref{eqn:bias_analytical_ridge} respectively. Since $\tilde\bbeta=0$, we only need to compare $B_1(\hat\bbeta_\lambda;\bbeta^{(2)})$ and $B_1(\hat\bbeta_\lambda;\bbeta^{(2)})$. Note that,
\begin{align*}
    B_1(\hat\bbeta;\bbeta^{(2)}) &= \bbeta^{(2)^\top} (\hat\bSigma^\dagger\hat\bSigma - \bI)\bSigma^{(2)} (\hat\bSigma^\dagger\hat\bSigma - \bI) \bbeta^{(2)}\\
    &= \bbeta^{(2)^\top}  \bU \bm{1}_{\bD=0} \bU^\top \bSigma^{(2)} \bU \bm{1}_{\bD=0} \bU^\top \bbeta^{(2)}\\
    &:= \bbeta^{(2)^\top}  \bU \bm{1}_{\bD=0} \bA \bm{1}_{\bD=0} \bU^\top \bbeta^{(2)}= \|\bA^{1/2} \bm{1}_{\bD=0} \bU^\top \bbeta^{(2)}\|^2_2,
\end{align*} 
where we defined $\bA=  \bU^\top \bSigma^{(2)} \bU$. Also,  
\begin{align*}
    B_1(\hat\bbeta_\lambda;\bbeta^{(2)})&= \lambda^2\bbeta^{(2)^\top} (\hat\bSigma+\lambda\bI)^{-1}\bSigma^{(2)} (\hat\bSigma+\lambda\bI)^{-1} \bbeta^{(2)}\\
    &=\lambda^2\bbeta^{(2)^\top} \bU (\bD+\lambda \bI)^{-1}\bA (\bD+\lambda \bI)^{-1} \bU^\top \bbeta^{(2)}\\
    &= \|\bA^{1/2} \lambda(\bD+\lambda \bI)^{-1} \bU^\top \bbeta^{(2)}\|^2_2.
\end{align*}
Therefore, we have 
\begin{align*}
   |B^{1/2}_1(\hat\bbeta;\bbeta^{(2)}) -B^{1/2}_1(\hat\bbeta_\lambda;\bbeta^{(2)})| &\le \|\bA^{1/2}(\bm{1}_{\bD=0}- \lambda(\bD+\lambda \bI)^{-1}) \bU^\top \bbeta^{(2)}\|_2\\
   & \lesssim \|\bA\|^{1/2}_{\op}\|\lambda(\bD+\lambda \bI)^{-1})\bm{1}_{\bD>0}\|_2\\
   &\lesssim \frac{\lambda}{\lambda_{\min}(\hat\Sigma)} =O(\lambda),
\end{align*}
 where the second inequality holds as $\|\bA\|_{\op}=\|\bSigma^{(2)}\|_{\op}\lesssim 1$ and the last inequality follows from Lemma \eqref{lemma:min_eigenvalue_heterogeneous}. Moreover, since $B_1(\hat\bbeta_\lambda;\bbeta^{(2)}), B_1(\hat\bbeta;\bbeta^{(2)}) \lesssim 1$, we have 
\begin{equation}\label{eqn:design_shift_B_originals_close}
     |B_1(\hat\bbeta;\bbeta^{(2)}) -B_1(\hat\bbeta_\lambda;\bbeta^{(2)})| 
     = O(\lambda).
 \end{equation}
 
\item [(iii)] Now we combine previous displays to obtain the desired results. Combining \eqref{eqn:design_shift_ridge_V}, \eqref{eqn:design_shift_V_limits_close} and \eqref{eqn:design_shift_V_originals_close}, we get
\begin{align*}
V(\hat\bbeta;\bbeta^{(2)}) &= -\sigma^2\gamma\int \frac{\lambda^{(2)}(\tilde a_3\lambda^{(1)}+\tilde a_4\lambda^{(2)})}{(\tilde a_1\lambda^{(1)}+\tilde a_2\lambda^{(2)}+1)^2} d \hat H_p(\lambda^{(1)},\lambda^{(2)}) \\
&\qquad + O(\lambda^{-9/2}p^{-1/2+c})+O(\lambda).
\end{align*}
Combining \eqref{eqn:design_shift_ridge_B}, \eqref{eqn:design_shift_B_limits_close} and \eqref{eqn:design_shift_B_originals_close}, we get
\begin{align*}
B(\hat\bbeta;\bbeta^{(2)}) &= \|\bbeta^{(2)}\|_2^2 \cdot\int \frac{\tilde b_3\lambda^{(1)}+(\tilde b_4+1)\lambda^{(2)}}{(\tilde b_1\lambda^{(1)}+\tilde b_2\lambda^{(2)}+1)^2} d \hat G_p(\lambda^{(1)},\lambda^{(2)}) \\
&\qquad + O(\lambda^{-3/2}p^{-1/4+c}) + O(\lambda).
\end{align*}
Taking $\lambda=p^{-1/12}$ and choosing $c$ small enough completes the proof.
\end{enumerate}

\subsection{Proof of Proposition \ref{prop:cov_shift_example}}\label{sec:cov-example}
Assume throughout the example that $n_1+n_2<p$. Here, we want to compare $\hat R (\bI)$ with $\hat R (\bM), \bM \in \mcl{S}$. To this end, plugging in $\bSigma^{(1)}=\bSigma^{(2)}=\bI$ into Theorem \ref{thm:design_shift} yields 

\begin{equation}\label{eq:r_i}
    \hat{R}(\bI)=\sigma^2\cdot\frac{n}{p-n} + \|\bbeta^{(2)}\|_2^2 \cdot \frac{p-n}{p} + o(1),
\end{equation}
with high probability, where the first term is the variance and the second term is the bias.

Now consider any $\bSigma^{(1)} =\bM \in \mcl{S}$. Plugging in $\bSigma^{(2)}=\bI$ and combining \eqref{eqn:design_shift_B} and the third equation of \eqref{eqn:design_shift_b_eqns}, we obtain that the bias is also $\|\bbeta^{(2)}\|_2^2 \cdot \frac{p-n}{p}$. Therefore we are left with comparing the variance.

Plugging in $\bSigma^{(2)}=\bI$ and combining \eqref{eqn:design_shift_V} and the third equation of \eqref{eqn:design_shift_a_eqns}, we obtain that the variance is $\sigma^2 \cdot (\tilde a_1+\tilde a_2)$. Moreover, the first and second equations of \eqref{eqn:design_shift_a_eqns} 
yields $\tilde a_2 = \frac{n_2}{p-n}$. Therefore, using \eqref{eq:r_i}, the proof is complete upon showing $\tilde a_1 \leq \frac{n_1}{p-n}$ if and only if $n_1 \leq p/2$. 

Since $\tilde a_2 = \frac{n_2}{p-n}$, the second equation of \eqref{eqn:design_shift_a_eqns} yields
\begin{equation}\label{eq:a1_example}
    1 = \sum_{i=1}^{p/2}\left[ \frac{1}{\tilde a_1(p-n)\lambda_i^{(1)}+p-n_1} + \frac{1}{\tilde a_1(p-n)/\lambda_i^{(1)}+p-n_1}\right].
\end{equation}
We know the right hand side is a decreasing function of $\tilde a_1$. Hence $\tilde a_1 \leq \frac{n_1}{p-n}$ if and only if we have
$$1 \geq \sum_{i=1}^{p/2}\left[ \frac{1}{n_1\lambda_i^{(1)}+p-n_1} + \frac{1}{n_1/\lambda_i^{(1)}+p-n_1}\right]:=\psi(n_1).$$
Note by straightforward calculation, $\psi(0)=\psi(p/2) = 1$, $\psi'(p/2)>0$, and $\psi''(x)>0$ for all $0 \leq x \leq p$. Therefore, we have $\tilde a_1 \leq \frac{n_1}{p-n}$ if and only if $n_1 \leq p/2$. 

\subsection{Proof of Proposition \ref{prop:cov_shift_kappa}}\label{sec:cov-shift_kappa}
Here, we want to study the effect of $\kappa$ in the quantity $\hat{R}(\bSigma(\kappa))$.

To analyze the quantity $\hat{R}(\bSigma(\kappa))$, first note that, by the proof of Proposition \ref{prop:cov_shift_example}, the bias stays identical for any $\kappa$ and the variance is $\sigma^2(\tilde a_1+ \tilde a_2)$, where $\tilde a_1,\tilde a_2$ are defined by \eqref{eqn:design_shift_a_eqns}. Moreover, again by proof of Proposition \ref{prop:cov_shift_example}, one has $\tilde a_2=\frac{n_2}{p-n_2}$. Hence, we are left with analyzing $\tilde a_1$ given by \eqref{eq:a1_example} which simplifies here as

\begin{align*}
    &\frac{2}{p}= \left[ \frac{1}{\tilde a_1(p-n)\kappa+p-n_1} + \frac{1}{\tilde a_1(p-n)/\kappa+p-n_1}\right]\\
    \implies& \tilde a^2_1 \frac{2(p-n)^2}{p}+ \frac{2(p-n_1)^2}{p}+\tilde a_1 \frac{2(p-n)(p-n_1)}{p} (\kappa+\frac{1}{\kappa}) \\
    &= \tilde a_1 (p-n) (\kappa+\frac{1}{\kappa}) +2(p-n_1)\\
    \implies&  2\tilde a_1^2 (p-n)^2+ \tilde a_1 (p-n)(p-2n_1)(\kappa+\frac{1}{\kappa}) - 2n_1(p-n_1) \\
    &= 0
\end{align*}
If $n_1 <p/2$, the linear term in the above quadratic equation is positive, implying $\tilde a_1$ is smaller for larger values of $(\kappa+\frac{1}{\kappa})$. Hence, $\tilde a_1$ is a decreasing function in $\kappa$, implying $\hat R(\bSigma(\kappa))$ decreases as $\kappa$ increases.

Similarly, if $n_1 <p/2$, the linear term in the above quadratic equation is negative, and hence $\hat R(\bSigma(\kappa))$ increases as $\kappa$ increases.

Finally, if $n_1=p/2$, the linear term vanishes leading to the fact that $\hat R(\bSigma(\kappa))$ is agnostic of $\kappa$.

\section{Proof for Universality}\label{sec:proof:universality}
In this theorem, we prove Theorem \ref{thm:universality}. 

Continuing from the notation introduced in Theorem \ref{thm:universality}, we write the bias components as functions of the underlying whitened covariates. Specifically, for $j = 1,2,3$, we denote by $B_{\lambda,j}(\bZ^{(1)}, \bZ^{(2)})$ denote the respective term in \eqref{eqn:bias_analytical_ridge} computed from $\bX^{(k)}=\bZ^{(k)}(\bSigma^{(k)})^{1/2}$, $k=1,2$. More generally, for $\bW^{(1)} \in \{\bZ^{(1)}, \tilde \bZ^{(1)}\}$ and $\bW^{(2)} \in \{\bZ^{(2)}, \tilde \bZ^{(2)}\}$, we let  $B_{\lambda, j}(\bW^{(1)},\bW^{(2)})$ denote the same bias term evaluated by replacing $\bZ^{(k)}$ by $\bW^{(k)}$ for $k = 1,2$. We suppress the dependence of these quantities on the ridge parameter $\lambda$ whenever it is clear from context. Additionally, with slight abuse of notation, we denote by $B_{0,j}(\bW^{(1)}, \bW^{(2)})$ the analogous ridgeless quantity in \eqref{eqn:bias_analytical}. Finally, the variance terms $\bm V_{\lambda}(\bm W^{(1)}, \bW^{(2)})$ and $\bm V_{0}(\bm W^{(1)}, \bW^{(2)})$ are defined analogously, using \eqref{eqn:variance_analytical_ridge} and \eqref{eqn:variance_analytical}.

Our proof of universality uses the Lindeberg swapping method, where we individually swap each row of $\bm Z$ with the corresponding row of $\tilde{\bm Z}$ and bound the resulting change in the out-of-sample prediction risk. To bound this difference, we bound the derivative of the matrix functional along a Gaussian interpolation path.

Using the arguments from \cite{hu2023universality, liang2022precise}, it suffices to show that, for $j = 1,2,3$,
\begin{equation*}
    |\E[\varphi(B_j(\bm Z^{(1)}, \bm Z^{(2)}))] - \E[\varphi(B_j(\tilde{\bm Z}^{(1)}, \tilde{\bm Z}^{(2)}))]| \to 0
\end{equation*}
for any bounded function $\varphi$ with bounded first and second derivatives. 

We begin by proving the following result, which we later use to bound the analogous ridgeless quantities.
\begin{prop}\label{pp:universality_ridge}
    Let $\lambda > 0$. Then, for any bounded function $\varphi$ with bounded first and second derivatives,
    \begin{align*}
        |&\E[\varphi(B_{\lambda,1}(\bm Z^{(1)}, \bm Z^{(2)}))] - \E[\varphi(B_{\lambda,1}(\tilde{\bm Z}^{(1)}, \tilde{\bm Z}^{(2)}))]| \\
        &\lcon \lambda^{-2} n^{-1/2}\|\bbeta^{(2)}\|_2^2(  (1   + \lambda^{-1} n^{-1/2} + \lambda^{-2} + \lambda^{-4} n^{-1/2} + \lambda^{-6} n^{-1/2}) \\
        &\hspace{10em}+  n^{-1/2}(1 + \lambda^{-2}) \|\bbeta^{(2)}\|_2^2)  \\
    |&\E[\varphi(B_{\lambda,2}(\bm Z^{(1)}, \bm Z^{(2)}))] - \E[\varphi(B_{\lambda,2}(\tilde{\bm Z}^{(1)}, \tilde{\bm Z}^{(2)}))]| \\
    &\lcon n^{-1} \log^2 n \lambda^{-4}  \|\tilde \bbeta\|_2^4  + n^{-1} \log^4 n (\lambda^{-6} + \lambda^{-8}) \|\tilde \bbeta\|_2^4  \\
    &\qquad + n^{-1/2} \lambda^{-2} \|\tilde \bbeta\|_2^2 (1 + n^{-1/2} \lambda^{-7} ) \\
    &\qquad + n^{-1/2} \log n \lambda^{-3} \|\tilde \bbeta\|_2^2  \\
    &\qquad + n^{-1/2} \log^2 n \lambda^{-3} \|\tilde \bbeta\|_2^2 (\lambda^{-1} + \lambda^{-3} + n^{-1/2} (\lambda^{-2} + \lambda^{-5} + \lambda^{-7})) \\
    |&\E[\varphi(B_{\lambda,3}(\bm Z^{(1)}, \bm Z^{(2)}))] - \E[\varphi(B_{\lambda,3}(\tilde{\bm Z}^{(1)}, \tilde{\bm Z}^{(2)}))]| \\
    &\lcon \lambda^{-2} n^{-1} \log^2 n \cdot \|\bbeta^{(2)}\|_2^2 \|\tilde \bbeta\|_2^2  (1 + \lambda^{-2} + \lambda^{-4}) \\
    &\qquad + \lambda^{-3} n^{-1/2} \log n \cdot \|\bbeta^{(2)}\|_2\|\tilde \bbeta\|_2 (1 + \lambda^{-1} + \lambda^{-2} + (\lambda^{-3} + \lambda^{-4}+ \lambda^{-6}) n^{-1/2}) \\
    |&\E[\varphi(V_\lambda(\bZ^{(1)}, \bZ^{(2)}))] - \E[\varphi(V_\lambda(\tilde \bZ^{(1)}, \tilde \bZ^{(2)}))]| \\
    &\lcon \sigma^2 n^{-1/2} \log n \cdot \lambda^{-4} (1 + \lambda^{-2} + n^{-1/2}(\lambda^{-1} + \lambda^{-4} + \lambda^{-6})) + \sigma^2 n^{-1/2} \lambda^{-3} \\
    &\qquad + \sigma^4 n^{-1} \log^2n \cdot (\lambda^{-6} + \lambda^{-8}) + \sigma^4 n^{-1} \lambda^{-4}. 
    \end{align*}
\end{prop}
We begin by proving the result for the bias terms and will later conclude with a similar proof for the variance term.

Let $\bm X^{(i),k}  \in \R^{n_i \times p}$ be the matrix whose first $k$ rows match the first $k$ rows of $\tilde{\bm X}^{(i)}$ and whose last $n_i - k$ rows match those of $\bm X^{(i)}$. Then, let $\bm X^{(i),\setminus k} \in \R^{(n_i - 1) \times p}$ be equal to the matrix $\bm X^{(i),k}$ with the $k$th row removed. Note that $\bm X^{(i),\setminus k}$ is also equal to the matrix $\bm X^{(i),k-1}$ with the $k$th row removed. Then, we can define
\begin{gather*}
    \hat \bSigma_k = \frac{1}{n} \bm X^{(1),k\top} \bm X^{(1),k} + \frac{1}{n} \bm X^{(2)\top} \bm X^{(2)} \\
    \hat \bSigma_{\setminus k} = \frac{1}{n} \bm X^{(1),\setminus k\top} \bm X^{(1),\setminus k} + \frac{1}{n} \bm X^{(2)\top} \bm X^{(2)} \\
    \tilde{\bSigma}_{k} = \frac{1}{n} \tilde{\bm X}^{(1)\top} \tilde{\bm X}^{(1)} + \frac{1}{n} \bm X^{(2),k\top} \bm X^{(2),k} \\
    \tilde{\bSigma}_{\setminus k} = \frac{1}{n} \tilde{\bm X}^{(1)\top} \tilde{\bm X}^{(1)} + \frac{1}{n} \bm X^{(2),\setminus k\top} \bm X^{(2),\setminus k}
\end{gather*}
For $k = 1, \dots, n_1$, an application of Taylor's theorem yields
\begin{align*}
    \varphi(B_j(\bm Z^{(1),k}, \bm Z^{(2)})) &= \varphi(B_j(\bm Z^{(1),\setminus k}, \bm Z^{(2)})) \\
    &\qquad + \varphi'(B_j(\bm Z^{(1),\setminus k}, \bm Z^{(2)})) (B_j(\bm Z^{(1),k}, \bm Z^{(2)}) - B_j(\bm Z^{(1),\setminus k}, \bm Z^{(2)})) \\
    &\qquad + \frac
    1   2\varphi''(\tilde B^{(1),k})(B_j(\bm Z^{(1),k}, \bm Z^{(2)}) - B_j(\bm Z^{(1),\setminus k}, \bm Z^{(2)}))^2.
\end{align*}
where $\tilde B^{(1),k}$ is some value between  $B_1(\bm Z^{(1),k}, \bm Z^{(2)})$ and $B_1(\bm Z^{(1),\setminus k}, \bm Z^{(2)})$. An analogous application yields
\begin{align*}
    \varphi(B_j(\bm Z^{(1),k-1}, \bm Z^{(2)})) &= \varphi(B_j(\bm Z^{(1),\setminus k}, \bm Z^{(2)})) \\
    &\quad + \varphi'(B_j(\bm Z^{(1),\setminus k}, \bm Z^{(2)})) (B_j(\bm Z^{(1),k-1}, \bm Z^{(2)}) - B_j(\bm Z^{(1),\setminus k}, \bm Z^{(2)})) \\
    &\quad + \frac
    1   2\varphi''(\dot B^{(1),k})(B_j(\bm Z^{(1),k-1}, \bm Z^{(2)}) - B_j(\bm Z^{(1),\setminus k}, \bm Z^{(2)}))^2.
\end{align*}
for some value $\dot B^{(1),k}$ between $B_1(\bm Z^{(1),k-1}, \bm Z^{(2)})$ and $B_1(\bm Z^{(1),\setminus k}, \bm Z^{(2)})$. Subtracting the first expression from the latter expression then yields
\begin{align*}
    \varphi(B_j(\bm Z^{(1),k},& \bm Z^{(2)})) - \varphi(B_j(\bm Z^{(1),k-1}, \bm Z^{(2)})) \\
    &= \varphi'(B_j(\bm Z^{(1),\setminus k}, \bm Z^{(2)}))(B_j(\bm Z^{(1),k}, \bm Z^{(2)}) - B_j(\bm Z^{(1),k-1}, \bm Z^{(2)})) \\
    &\qquad + \frac
    1   2\varphi''(\tilde B^{(1),k})(B_j(\bm Z^{(1),k}, \bm Z^{(2)}) - B_j(\bm Z^{(1),\setminus k}, \bm Z^{(2)}))^2 \\
    &\qquad - \frac
    1   2\varphi''(\dot B^{(1),k})(B_j(\bm Z^{(1),k-1}, \bm Z^{(2)}) - B_j(\bm Z^{(1),\setminus k}, \bm Z^{(2)}))^2.
\end{align*}
The analogous procedure yields
\begin{align*}
    \varphi(B_j(\tilde{\bm Z},& \bm Z^{(2),k})) - \varphi(B_j(\bm Z^{(1)}, \bm Z^{(2),k-1})) \\
    &= \varphi'(B_j(\bm Z^{(1)}, \bm Z^{(2),\setminus k}))(B_j(\bm Z^{(1)}, \bm Z^{(2),k}) - B_j(\bm Z^{(1)}, \bm Z^{(2),k-1})) \\
    &\qquad + \frac
    1   2\varphi''(\tilde B^{(2),k})(B_j(\bm Z^{(1)}, \bm Z^{(2),k}) - B_j(\bm Z^{(1)}, \bm Z^{(2),\setminus k}))^2 \\
    &\qquad - \frac
    1   2\varphi''(\dot B^{(2),k})(B_j(\bm Z^{(1)}, \bm Z^{(2),k-1}) - B_j(\bm Z^{(1)}, \bm Z^{(2),\setminus k}))^2.
\end{align*}
for each $k = 1,\dots, n_2$.
Importantly, for each $k = 1, \dots, n_1$, the quantity $B_j(\bm Z^{(1),\setminus k}, \bm Z^{(2)})$ is independent of $\bm Z^{(1),k}, \tilde{\bm Z}^{(1),k}$. Similarly, for each $k = 1, \dots, n_2$, the quantity $B_j(\tilde{\bm Z}^{(1)}, \bm Z^{(2),\setminus k})$ is independent of $\bm Z^{(2),k}, \tilde{\bm Z}^{(2),k}$. Therefore, we have
\begin{align*}
    &\varphi(B_j(\bm Z^{(1)}, \bm Z^{(2)})) - \varphi(B_j(\tilde{\bm Z}^{(1)}, \bm Z^{(2)})) 
    \\&= \sum_{k=1}^{n_1} [\varphi(B_j(\bm Z^{(1),k}, \bm Z^{(2)})) - \varphi(B_j(\bm Z^{(1),k-1}, \bm Z^{(2)}))]  \\
    &\varphi(B_j(\tilde{\bm Z}^{(1)}, \bm Z^{(2)})) - \varphi(B_j(\tilde{\bm Z}^{(1)}, \tilde{\bm Z}^{(2)})) 
    \\
    &= \sum_{k=1}^{n_2} [\varphi(B_j(\tilde{\bm Z}^{(1)}, \bm Z^{(2),k})) - \varphi(B_j(\tilde{\bm Z}^{(1)}, \bm Z^{(2),k-1}))]
\end{align*}
Therefore, letting $\E_{1,k}$ be expectation with respect to $\bm X^{(1),k}, \tilde{\bm X}^{(1),k}$ and letting $\E_{2,k}$ be expectation with respect to $\bm X^{(2),k}, \tilde{\bm X}^{(2),k}$, we can apply the triangle inequality and the above reasoning to yield
\begin{align*}
&\bigl|\E\bigl[\varphi\bigl(B_j(\bm Z^{(1)}, \bm Z^{(2)})\bigr)\bigr] - \E\bigl[\varphi\bigl(B_j(\tilde{\bm Z}^{(1)}, \bm Z^{(2)})\bigr)\bigr] \bigr| \\
&\le \sum_{k=1}^{n_1}
\bigl| \E\bigl[\varphi\bigl(B_j(\bm Z^{(1),k}, \bm Z^{(2)})\bigr)\bigr] - \E\bigl[\varphi\bigl(B_j(\bm Z^{(1),k-1}, \bm Z^{(2)})\bigr)\bigr] \bigr| \\
&\le \|\varphi'\|_\infty \sum_{k=1}^{n_1} \E\Bigl[\bigl|\E_{1,k}\bigl[B_j(\bm Z^{(1),k}, \bm Z^{(2)})-B_j(\bm Z^{(1),k-1}, \bm Z^{(2)})\bigr] \bigr|\Bigr] \\
&\quad
+ \frac{\|\varphi''\|_\infty}{2} \sum_{k=1}^{n_1} \E\Bigl[\E_{1,k}\bigl[\bigl(B_j(\bm Z^{(1),k}, \bm Z^{(2)})-B_j(\bm Z^{(1),\setminus k}, \bm Z^{(2)})\bigr)^2\bigr]\Bigr] \\
&\quad
+ \frac{\|\varphi''\|_\infty}{2} \sum_{k=1}^{n_1}
    \E\Bigl[
        \E_{1,k}\bigl[
            \bigl(
                B_j(\bm Z^{(1),k-1}, \bm Z^{(2)})
                -
                B_j(\bm Z^{(1),\setminus k}, \bm Z^{(2)})
            \bigr)^2
        \bigr]
    \Bigr].
\end{align*}
and similarly
\begin{align*}
&\bigl|
    \E\bigl[\varphi\bigl(B_j(\tilde{\bm Z}^{(1)}, \bm Z^{(2)})\bigr)\bigr]
    -
    \E\bigl[\varphi\bigl(B_j(\tilde{\bm Z}^{(1)}, \tilde{\bm Z}^{(2)})\bigr)\bigr]
\bigr| \\
&\le \sum_{k=1}^{n_2}
\bigl|
    \E\bigl[\varphi\bigl(B_j(\tilde{\bm Z}^{(1)}, \bm Z^{(2),k})\bigr)\bigr]
    -
    \E\bigl[\varphi\bigl(B_j(\tilde{\bm Z}^{(1)}, \bm Z^{(2),k-1})\bigr)\bigr]
\bigr| \\
&\le \|\varphi'\|_\infty \sum_{k=1}^{n_2}
    \E\Bigl[
        \bigl|
            \E_{2,k}\bigl[
                B_j(\tilde{\bm Z}^{(1)}, \bm Z^{(2),k})
                -
                B_j(\tilde{\bm Z}^{(1)}, \bm Z^{(2),k-1})
            \bigr]
        \bigr|
    \Bigr] \\
&\quad
+ \frac{\|\varphi''\|_\infty}{2} \sum_{k=1}^{n_2}
    \E\Bigl[
        \E_{2,k}\bigl[
            \bigl(
                B_j(\tilde{\bm Z}^{(1)}, \bm Z^{(2),k})
                -
                B_j(\tilde{\bm Z}^{(1)}, \bm Z^{(2),\setminus k})
            \bigr)^2
        \bigr]
    \Bigr] \\
&\quad
+ \frac{\|\varphi''\|_\infty}{2} \sum_{k=1}^{n_2}
    \E\Bigl[
        \E_{2,k}\bigl[
            \bigl(
                B_j(\tilde{\bm Z}^{(1)}, \bm Z^{(2),k-1})
                -
                B_j(\tilde{\bm Z}^{(1)}, \bm Z^{(2),\setminus k})
            \bigr)^2
        \bigr]
    \Bigr].
\end{align*}
If both of these terms converge to zero, then the triangle inequality immediately implies that the term $|\E[\varphi(B_j(\bm Z^{(1)}, \bm Z^{(2)}))] - \E[\varphi(B_j(\tilde{\bm Z}^{(1)}, \tilde{\bm Z}^{(2)}))]|$ converges to zero, which yields the desired universality result.

\subsection{Reduction via Sherman-Morrison}
Now, we can expand each of the terms from above. We will show that bounding these differences reduces to controlling the behavior of eight relatively simple terms.

By Sherman-Morrison, we can express
\begin{align*}
    (\hat \bSigma_k + \lambda \bm I)^{-1} &= (\hat \bSigma_{\setminus k} + \tfrac{1}{n} \tilde{\bm X}_k^{(1)} \tilde{\bm X}_k^{(1)\top} + \lambda \bm I)^{-1} \\
    &= (\hat \bSigma_{\setminus k} + \lambda \bm I)^{-1} - \frac{1}{n} \frac{(\hat \bSigma_{\setminus k} + \lambda \bm I)^{-1} \tilde{\bm X}_k^{(1)} \tilde{\bm X}_k^{(1)\top} (\hat \bSigma_{\setminus k} + \lambda \bm I)^{-1}}{1 + \frac{1}{n} \tilde{\bm X}_k^{(1)\top} (\hat \bSigma_{\setminus k} + \lambda \bm I)^{-1} \tilde{\bm X}_k^{(1)}}.
\end{align*}
\subsubsection{Expanding \texorpdfstring{$B_1$}{B1} terms}
Recall that
\begin{equation*}
    B_1(\bm Z^{(1)}, \bm Z^{(2)}) = \lambda^2 \bbeta^{(2)\top} (\hat \bSigma+ \lambda \bm I)^{-1} \bSigma^{(2)} (\hat \bSigma+ \lambda \bm I)^{-1} \bbeta^{(2)}
\end{equation*}
and therefore we can express
\begin{align*}
    &B_1(\bm Z^{(1),k}, \bm Z^{(2)}) \\&=\lambda^2 \bbeta^{(2)\top} \biggl((\hat \bSigma_{\setminus k} + \lambda \bm I)^{-1} - \frac{1}{n} \frac{(\hat \bSigma_{\setminus k} + \lambda \bm I)^{-1} \tilde{\bm X}_k^{(1)} \tilde{\bm X}_k^{(1)\top} (\hat \bSigma_{\setminus k} + \lambda \bm I)^{-1}}{1 + \frac{1}{n} \tilde{\bm X}_k^{(1)\top} (\hat \bSigma_{\setminus k} + \lambda \bm I)^{-1} \tilde{\bm X}_k^{(1)}}\biggr)^{-1}  \\
    &\qquad \bSigma^{(2)}\biggl((\hat \bSigma_{\setminus k} + \lambda \bm I)^{-1} - \frac{1}{n} \frac{(\hat \bSigma_{\setminus k} + \lambda \bm I)^{-1} \tilde{\bm X}_k^{(1)} \tilde{\bm X}_k^{(1)\top} (\hat \bSigma_{\setminus k} + \lambda \bm I)^{-1}}{1 + \frac{1}{n} \tilde{\bm X}_k^{(1)\top} (\hat \bSigma_{\setminus k} + \lambda \bm I)^{-1} \tilde{\bm X}_k^{(1)}}\biggr) \bbeta^{(2)}
\end{align*}
from which we get
\begin{align*}
   & B_1(\bm Z^{(1),k}, \bm Z^{(2)}) - B_1(\bm Z^{(1),\setminus k}, \bm Z^{(2)})\\
    &= -2\lambda^2 \bbeta^{(2)\top} \biggl(\frac{1}{n} \frac{(\hat \bSigma_{\setminus k} +  \lambda \bm I)^{-1} \tilde{\bm X}^{(1)}_k \tilde{\bm X}^{(1)\top}_k (\hat \bSigma_{\setminus k} + \lambda \bm I)^{-1}}{1 +  \frac{1}{n}\tilde{\bm X}^{(1)\top}_k (\hat \bSigma_{\setminus k} + \lambda \bm I)^{-1} \tilde{\bm X}^{(1)}_k}  \biggr) \bSigma^{(2)} (\hat \bSigma_{\setminus k} + \lambda \bm I)^{-1} \bbeta^{(2)} \\
    &\qquad + \lambda^2 \bbeta^{(2)\top} \biggl(\frac{1}{n} \frac{(\hat \bSigma_{\setminus k} +  \lambda \bm I)^{-1} \tilde{\bm X}^{(1)}_k \tilde{\bm X}^{(1)\top}_k (\hat \bSigma_{\setminus k} + \lambda \bm I)^{-1}}{1 +  \frac{1}{n}\tilde{\bm X}^{(1)\top}_k (\hat \bSigma_{\setminus k} + \lambda \bm I)^{-1} \tilde{\bm X}^{(1)}_k} \biggr) \bSigma^{(2)} \\
    &\qquad \qquad \cdot \biggl( \frac{1}{n} \frac{(\hat \bSigma_{\setminus k} +  \lambda \bm I)^{-1} \tilde{\bm X}^{(1)}_k \tilde{\bm X}^{(1)\top}_k (\hat \bSigma_{\setminus k} + \lambda \bm I)^{-1}}{1 +  \frac{1}{n}\tilde{\bm X}^{(1)\top}_k (\hat \bSigma_{\setminus k} + \lambda \bm I)^{-1} \tilde{\bm X}^{(1)}_k} \biggr) \bbeta^{(2)} \\
    &:= \frac{\bm u^{(a,1)\top} \tilde{\bm Z}_k^{(1)} \tilde{\bm Z}_k^{(1)\top} \bm v^{(a,1)}}{1 + \tilde{\bm Z}_k^{(1)\top} \bm A^{(a,1)} \tilde{\bm Z}_k^{(1)}} + \frac{\bm u^{(a,2)\top} \tilde{\bm Z}_k^{(1)} (\tilde{\bm Z}_k^{(1)\top} \bm A^{(a,2)} \tilde{\bm Z}_k^{(1)}) \tilde{\bm Z}_k^{(1)\top} \bm v^{(a,2)}}{(1 + \tilde{\bm Z}_k^{(1)\top} \bm B^{(a,2)} \tilde{\bm Z}_k^{(1)})^2}
\end{align*}
and similarly 
\begin{align*}
   & B_1(\bm Z^{(1),k}, \bm Z^{(2)}) - B_1(\bm Z^{(1),\setminus k}, \bm Z^{(2)})\\
    &:= \frac{\bm u^{(a,1)\top} \bm Z_k^{(1)} \bm Z_k^{(1)\top} \bm v^{(a,1)}}{1 + \bm Z_k^{(1)\top} \bm A^{(a,1)} \bm Z_k^{(1)}} + \frac{\bm u^{(a,2)\top} \bm Z_k^{(1)} (\bm Z_k^{(1)\top} \bm A^{(a,2)} \bm Z_k^{(1)}) \bm Z_k^{(1)\top} \bm v^{(a,2)}}{(1 + \bm Z_k^{(1)\top} \bm B^{(a,2)} \bm Z_k^{(1)})^2}
\end{align*}
for vectors $\bm u^{(a,1)}$, $\bm v^{(a,1)}$, $\bm u^{(a,2)}$, $\bm v^{(a,2)}$ and matrices $\bm A^{(a,1)}$, $\bm A^{(a,2)}$, $\bm A^{(a,3)}$ that are all independent of $\bm Z_k^{(1)}, \tilde{\bm Z}_k^{(1)}$. Specifically, we have
\begin{gather*}
    \bm u^{(a,1)} = -\frac{2\lambda^2}{n} (\bSigma^{(1)})^{1/2} (\hat \bSigma_{\setminus k} + \lambda \bm I)^{-1}\bbeta^{(2)} \\ \bm v^{(a,1)} =  (\bSigma^{(1)})^{1/2} (\hat \bSigma_{\setminus k} + \lambda \bm I)^{-1}\bSigma^{(2)} (\hat \bSigma_{\setminus k} + \lambda \bm I)^{-1} \bbeta^{(2)} \\
    \bm A^{(a,1)} = \frac
    1   n(\bSigma^{(1)})^{1/2} (\hat \bSigma_{\setminus k} + \lambda \bm I)^{-1} (\bSigma^{(1)})^{1/2} \\
    \bm u^{(a,2)} = \frac{\lambda^2}{n^2} (\bSigma^{(1)})^{1/2} (\hat \bSigma_{\setminus k} + \lambda \bm I)^{-1} \bbeta^{(2)}, \quad \bm v^{(a,2)} = (\bSigma^{(1)})^{1/2} (\hat \bSigma_{\setminus k} + \lambda \bm I)^{-1} \bbeta^{(2)} \\
    \bm A^{(a,2)} = (\bSigma^{(1)})^{1/2} (\hat \bSigma_{\setminus k} + \lambda \bm I)^{-1} \bSigma^{(2)} (\hat \bSigma_{\setminus k} + \lambda \bm I)^{-1} (\bSigma^{(1)})^{1/2} \\
    \bm B^{(a,2)} = \frac
    1   n(\bSigma^{(1)})^{1/2} (\hat \bSigma_{\setminus k} + \lambda \bm I)^{-1} (\bSigma^{(1)})^{1/2}
\end{gather*}
We similarly have
\begin{align*}
    &B_1(\tilde{\bm Z}^{(1)}, \bm Z^{(2),k}) \\&=\lambda^2 \bbeta^{(2)\top} \biggl((\tilde \bSigma_{\setminus k} + \lambda \bm I)^{-1} - \frac{1}{n} \frac{(\tilde \bSigma_{\setminus k} + \lambda \bm I)^{-1} \tilde{\bm X}_k^{(2)} \tilde{\bm X}_k^{(2)\top} (\tilde \bSigma_{\setminus k} + \lambda \bm I)^{-1}}{1 + \frac{1}{n} \tilde{\bm X}_k^{(2)\top} (\tilde \bSigma_{\setminus k} + \lambda \bm I)^{-1} \tilde{\bm X}_k^{(2)}}\biggr)^{-1} \bSigma^{(2)} \\
    &\qquad \biggl((\tilde \bSigma_{\setminus k} + \lambda \bm I)^{-1} - \frac{1}{n} \frac{(\tilde \bSigma_{\setminus k} + \lambda \bm I)^{-1} \tilde{\bm X}_k^{(2)} \tilde{\bm X}_k^{(2)\top} (\tilde \bSigma_{\setminus k} + \lambda \bm I)^{-1}}{1 + \frac{1}{n} \tilde{\bm X}_k^{(1)\top} (\tilde \bSigma_{\setminus k} + \lambda \bm I)^{-1} \tilde{\bm X}_k^{(1)}}\biggr) \bbeta^{(2)}
\end{align*}
from which we get
\begin{align*}
    &B_1(\tilde{\bm Z}^{(1)}, \bm Z^{(2),k}) - B_1(\tilde{\bm Z}^{(1)}, \bm Z^{(2),\setminus k}) \\
    &= - 2 \lambda^2 \bbeta^{(2)\top} \biggl( \frac{1}{n} \frac{(\tilde \bSigma_{\setminus k} + \lambda \bm I)^{-1} \tilde{\bm X}_k^{(2)} \tilde{\bm X}_k^{(2)\top} (\tilde \bSigma_{\setminus k} + \lambda \bm I)^{-1}}{1 + \frac{1}{n} \tilde{\bm X}_k^{(1)\top} (\tilde \bSigma_{\setminus k} + \lambda \bm I)^{-1} \tilde{\bm X}_k^{(1)}}\biggr) \bSigma^{(2)} (\tilde{\bSigma}_{\setminus k} + \lambda \bm I)^{-1} \bbeta^{(2)} \\
    &\qquad + \lambda^2 \bbeta^{(2)\top} \biggl(\frac{1}{n} \frac{(\tilde \bSigma_{\setminus k} + \lambda \bm I)^{-1} \tilde{\bm X}_k^{(2)} \tilde{\bm X}_k^{(2)\top} (\tilde \bSigma_{\setminus k} + \lambda \bm I)^{-1}}{1 + \frac{1}{n} \tilde{\bm X}_k^{(1)\top} (\tilde \bSigma_{\setminus k} + \lambda \bm I)^{-1} \tilde{\bm X}_k^{(1)}}\biggr) \bSigma^{(2)} \\
    &\qquad \qquad \biggl(\frac{1}{n} \frac{(\tilde \bSigma_{\setminus k} + \lambda \bm I)^{-1} \tilde{\bm X}_k^{(2)} \tilde{\bm X}_k^{(2)\top} (\tilde \bSigma_{\setminus k} + \lambda \bm I)^{-1}}{1 + \frac{1}{n} \tilde{\bm X}_k^{(1)\top} (\tilde \bSigma_{\setminus k} + \lambda \bm I)^{-1} \tilde{\bm X}_k^{(1)}}\biggr) \bbeta^{(2)} \\
    &:= \frac{\bm u^{(A,1)\top} \tilde{\bm Z}_k^{(2)} \tilde{\bm Z}_k^{(2)\top} \bm v^{(A,1)}}{1 + \tilde{\bm Z}_k^{(2)\top} \bm A^{(A,1)} \tilde{\bm Z}_k^{(2)}} + \frac{\bm u^{(A,2)\top} \tilde{\bm Z}_k^{(2)} (\tilde{\bm Z}_k^{(2)\top} \bm A^{(A,2)} \tilde{\bm Z}_k^{(2)}) \tilde{\bm Z}_k^{(2)\top} \bm v^{(A,2)}}{(1 + \tilde{\bm Z}_k^{(2)\top} \bm B^{(A,2)} \tilde{\bm Z}_k^{(2)})^2} \\
\end{align*}
and similarly
\begin{align*}
   &B_1(\tilde{\bm Z}^{(1)}, \bm Z^{(2),k-1}) - B_1(\tilde{\bm Z}^{(1)}, \bm Z^{(2),\setminus k})\\
    &:= \frac{\bm u^{(A,1)\top} \bm Z_k^{(2)} \bm Z_k^{(2)\top} \bm v^{(A,1)}}{1 + \bm Z_k^{(2)\top} \bm A^{(A,1)} \bm Z_k^{(2)}} + \frac{\bm u^{(A,2)\top} \bm Z_k^{(2)} (\bm Z_k^{(2)\top} \bm A^{(A,2)} \bm Z_k^{(2)}) \bm Z_k^{(2)\top} \bm v^{(A,2)}}{(1 + \bm Z_k^{(2)\top} \bm B^{(A,2)} \bm Z_k^{(2)})^2}
\end{align*}
where
\begin{gather*}
    \bm u^{(A,1)} = -\frac{2\lambda^2}{n} (\bSigma^{(2)})^{1/2} (\tilde \bSigma_{\setminus k} + \lambda \bm I)^{-1}\bbeta^{(2)}, \\ \bm v^{(A,1)} =  (\bSigma^{(2)})^{1/2} (\tilde \bSigma_{\setminus k} + \lambda \bm I)^{-1}\bSigma^{(2)} (\tilde \bSigma_{\setminus k} + \lambda \bm I)^{-1} \bbeta^{(2)} \\
    \bm A^{(A,1)} = \frac
    1   n(\bSigma^{(2)})^{1/2} (\tilde \bSigma_{\setminus k} + \lambda \bm I)^{-1} (\bSigma^{(2)})^{1/2} \\
    \bm u^{(A,2)} = \frac{\lambda^2}{n^2} (\bSigma^{(2)})^{1/2} (\tilde \bSigma_{\setminus k} + \lambda \bm I)^{-1} \bbeta^{(2)}, \quad \bm v^{(A,2)} = (\bSigma^{(2)})^{1/2} (\tilde \bSigma_{\setminus k} + \lambda \bm I)^{-1} \bbeta^{(2)} \\
    \bm A^{(A,2)} = (\bSigma^{(2)})^{1/2} (\tilde \bSigma_{\setminus k} + \lambda \bm I)^{-1} \bSigma^{(2)} (\tilde \bSigma_{\setminus k} + \lambda \bm I)^{-1} (\bSigma^{(2)})^{1/2} \\
    \bm B^{(A,2)} = \frac
    1   n(\bSigma^{(2)})^{1/2} (\tilde \bSigma_{\setminus k} + \lambda \bm I)^{-1} (\bSigma^{(2)})^{1/2}
\end{gather*}
\subsubsection{Expanding \texorpdfstring{$B_3$}{B3} terms}
Now, recall that
\begin{equation*}
    B_3(\bm Z^{(1)}, \bm Z^{(2)}) = - 2\lambda \bbeta^{(2)\top} (\hat \bSigma+ \lambda \bm I)^{-1} \bSigma^{(2)} (\hat \bSigma+ \lambda \bm I)^{-1} \biggl(\frac{\bm X^{(1)\top} \bm X^{(1)}}{n}\biggr) \tilde \bbeta
\end{equation*}
Therefore, using the result from earlier,
\begin{align*}
    &B_3(\bm Z^{(1), k}, \bm Z^{(2)}) \\&= -2 \lambda \bbeta^{(2)\top} (\hat \bSigma_{k} + \lambda \bm I)^{-1} \bSigma^{(2)} (\hat \bSigma_{k} + \lambda \bm I)^{-1}\biggl(\frac{\bm X^{(1), k\top}\bm X^{(1), k}}{n}\biggr) \tilde \bbeta \\
    &=-2\lambda \bbeta^{(2)\top} \biggl((\hat \bSigma_{\setminus k} + \lambda \bm I)^{-1} - \frac{1}{n} \frac{(\hat \bSigma_{\setminus k} + \lambda \bm I)^{-1} \tilde{\bm X}_k^{(1)} \tilde{\bm X}_k^{(1)\top} (\hat \bSigma_{\setminus k} + \lambda \bm I)^{-1}}{1 + \frac{1}{n} \tilde{\bm X}_k^{(1)\top} (\hat \bSigma_{\setminus k} + \lambda \bm I)^{-1} \tilde{\bm X}_k^{(1)}} \biggr)  \\
    &\qquad \cdot \bSigma^{(2)} \biggl((\hat \bSigma_{\setminus k} + \lambda \bm I)^{-1} - \frac{1}{n} \frac{(\hat \bSigma_{\setminus k} + \lambda \bm I)^{-1} \tilde{\bm X}_k^{(1)} \tilde{\bm X}_k^{(1)\top} (\hat \bSigma_{\setminus k} + \lambda \bm I)^{-1}}{1 + \frac{1}{n} \tilde{\bm X}_k^{(1)\top} (\hat \bSigma_{\setminus k} + \lambda \bm I)^{-1} \tilde{\bm X}_k^{(1)}} \biggr) \\
    &\qquad \cdot \biggl(\frac{\bm  X^{(1),\setminus k\top} \bm X^{(1),\setminus k}}{n} + \frac{\tilde{\bm X}^{(1)}_k \tilde{\bm X}^{(1)\top}_k}{n} \biggr)\tilde \bbeta
\end{align*}
and therefore
\begin{align*}
    &B_3(\bm Z^{(1),k}, \bm Z^{(2)}) - B_3( \bm Z^{(1),\setminus k}, \bm Z^{(2)}) \\
    &= -2 \lambda \bbeta^{(2)\top} (\hat \bSigma_{\setminus k} + \lambda \bm I)^{-1} \bSigma^{(2)} (\hat \bSigma_{\setminus k} + \lambda \bm I)^{-1} \biggl(\frac{\tilde{\bm X}_k^{(1)} \tilde{\bm X}_k^{(1)\top}}{n} \biggr) \tilde \bbeta \\
    &\qquad + 2 \lambda \bbeta^{(2)\top} (\hat \bSigma_{\setminus k} + \lambda \bm I)^{-1} \bSigma^{(2)}  \\
    &\hspace{4em} \biggl(\frac{1}{n} \frac{(\hat \bSigma_{\setminus k} +  \lambda \bm I)^{-1} \tilde{\bm X}^{(1)}_k \tilde{\bm X}^{(1)\top}_k (\hat \bSigma_{\setminus k} + \lambda \bm I)^{-1}}{1 +  \frac{1}{n}\tilde{\bm X}^{(1)\top}_k (\hat \bSigma_{\setminus k} + \lambda \bm I)^{-1} \tilde{\bm X}^{(1)}_k}\biggr)\biggl(\frac{\bm X^{(1),\setminus k\top} \bm X^{(1),\setminus k}}{n}\biggr) \tilde \bbeta \\
    &\qquad + 2 \lambda \bbeta^{(2)\top} (\hat \bSigma_{\setminus k} + \lambda \bm I)^{-1} \bSigma^{(2)}  \\
    &\hspace{4em}\biggl(\frac{1}{n} \frac{(\hat \bSigma_{\setminus k} +  \lambda \bm I)^{-1} \tilde{\bm X}^{(1)}_k \tilde{\bm X}^{(1)\top}_k (\hat \bSigma_{\setminus k} + \lambda \bm I)^{-1}}{1 +  \frac{1}{n}\tilde{\bm X}^{(1)\top}_k (\hat \bSigma_{\setminus k} + \lambda \bm I)^{-1} \tilde{\bm X}^{(1)}_k}\biggr)\biggl(\frac{\tilde {\bm X}^{(1)}_k \tilde {\bm X}_k^{(1)\top}}{n}\biggr) \tilde \bbeta \\
    &\qquad + 2 \lambda \bbeta^{(2)\top}\biggl(\frac{1}{n} \frac{(\hat \bSigma_{\setminus k} +  \lambda \bm I)^{-1} \tilde{\bm X}^{(1)}_k \tilde{\bm X}^{(1)\top}_k (\hat \bSigma_{\setminus k} + \lambda \bm I)^{-1}}{1 +  \frac{1}{n}\tilde{\bm X}^{(1)\top}_k (\hat \bSigma_{\setminus k} + \lambda \bm I)^{-1} \tilde{\bm X}^{(1)}_k}\biggr) \bSigma^{(2)} \\
    &\hspace{4em}(\hat \bSigma_{\setminus k} + \lambda \bm I)^{-1} \biggl(\frac{\bm X^{(1),\setminus k\top} \bm X^{(1),\setminus k}}{n}\biggr) \tilde \bbeta \\
    &\qquad + 2 \lambda \bbeta^{(2)\top} \biggl(\frac{1}{n} \frac{(\hat \bSigma_{\setminus k} +  \lambda \bm I)^{-1} \tilde{\bm X}^{(1)}_k \tilde{\bm X}^{(1)\top}_k (\hat \bSigma_{\setminus k} + \lambda \bm I)^{-1}}{1 +  \frac{1}{n}\tilde{\bm X}^{(1)\top}_k (\hat \bSigma_{\setminus k} + \lambda \bm I)^{-1} \tilde{\bm X}^{(1)}_k}\biggr) \bSigma^{(2)} \\
    &\hspace{4em} (\hat \bSigma_{\setminus k} + \lambda \bm I)^{-1} \biggl(\frac{\tilde {\bm X}^{(1)}_k \tilde {\bm X}_k^{(1)\top}}{n}\biggr) \tilde \bbeta \\
    &\qquad - 2 \lambda \bbeta^{(2)\top} \biggl(\frac{1}{n} \frac{(\hat \bSigma_{\setminus k} +  \lambda \bm I)^{-1} \tilde{\bm X}^{(1)}_k \tilde{\bm X}^{(1)\top}_k (\hat \bSigma_{\setminus k} + \lambda \bm I)^{-1}}{1 +  \frac{1}{n}\tilde{\bm X}^{(1)\top}_k (\hat \bSigma_{\setminus k} + \lambda \bm I)^{-1} \tilde{\bm X}^{(1)}_k}\biggr) \bSigma^{(2)} \\
    &\hspace{4em} \biggl(\frac{1}{n} \frac{(\hat \bSigma_{\setminus k} +  \lambda \bm I)^{-1} \tilde{\bm X}^{(1)}_k \tilde{\bm X}^{(1)\top}_k (\hat \bSigma_{\setminus k} + \lambda \bm I)^{-1}}{1 +  \frac{1}{n}\tilde{\bm X}^{(1)\top}_k (\hat \bSigma_{\setminus k} + \lambda \bm I)^{-1} \tilde{\bm X}^{(1)}_k}\biggr) \biggl(\frac{\tilde {\bm X}^{(1),\setminus k\top } \tilde {\bm X}^{(1), \setminus k}}{n}\biggr) \tilde \bbeta \\
    &\qquad - 2 \lambda \bbeta^{(2)\top} \biggl(\frac{1}{n} \frac{(\hat \bSigma_{\setminus k} +  \lambda \bm I)^{-1} \tilde{\bm X}^{(1)}_k \tilde{\bm X}^{(1)\top}_k (\hat \bSigma_{\setminus k} + \lambda \bm I)^{-1}}{1 +  \frac{1}{n}\tilde{\bm X}^{(1)\top}_k (\hat \bSigma_{\setminus k} + \lambda \bm I)^{-1} \tilde{\bm X}^{(1)}_k}\biggr) \bSigma^{(2)} \\
    &\hspace{4em} \biggl(\frac{1}{n} \frac{(\hat \bSigma_{\setminus k} +  \lambda \bm I)^{-1} \tilde{\bm X}^{(1)}_k \tilde{\bm X}^{(1)\top}_k (\hat \bSigma_{\setminus k} + \lambda \bm I)^{-1}}{1 +  \frac{1}{n}\tilde{\bm X}^{(1)\top}_k (\hat \bSigma_{\setminus k} + \lambda \bm I)^{-1} \tilde{\bm X}^{(1)}_k}\biggr) \biggl(\frac{\tilde {\bm X}^{(1)}_k \tilde {\bm X}_k^{(1)\top}}{n}\biggr) \tilde \bbeta \\
    &:= \bm u^{(b,1)\top} \tilde{\bm Z}_k^{(1)} \tilde{\bm Z}_k^{(1)\top} \bm v^{(b,1)} 
    + \frac{\bm u^{(b,2)\top} \tilde{\bm Z}_k^{(1)} \tilde{\bm Z}_k^{(1)\top} \bm v^{(b,2)}}{1 + \tilde{\bm Z}_k^{(1)\top} \bm A^{(b,2)} \tilde{\bm Z}_k^{(1)}} 
    ]] \\
    &\qquad + \frac{\bm u^{(b,3)\top} \tilde{\bm Z}_k^{(1)} (\tilde{\bm Z}_k^{(1)\top} \bm A^{(b,3)} \tilde{\bm Z}_k^{(1)}) \tilde{\bm Z}_k^{(1)\top} \bm v^{(b,3)}}{1 + \tilde{\bm Z}_k^{(1)\top} \bm B^{(b,3)} \tilde{\bm Z}_k^{(1)}} \\
    &\qquad + \frac{\bm u^{(b,4)\top} \tilde{\bm Z}_k^{(1)}  \tilde{\bm Z}_k^{(1)\top} \bm v^{(b,4)}}{1 + \tilde{\bm Z}_k^{(1)\top} \bm A^{(b,4)} \tilde{\bm Z}_k^{(1)}} + \frac{\bm u^{(b,5)\top} \tilde{\bm Z}_k^{(1)} (\tilde{\bm Z}_k^{(1)\top} \bm A^{(b,5)} \tilde{\bm Z}_k^{(1)}) \tilde{\bm Z}_k^{(1)\top} \bm v^{(b,5)}}{1 + \tilde{\bm Z}_k^{(1)\top} \bm B^{(b,5)} \tilde{\bm Z}_k^{(1)}} \\
    &\qquad + \frac{\bm u^{(b,6)\top} \tilde{\bm Z}_k^{(1)} (\tilde{\bm Z}_k^{(1)\top} \bm A^{(b,6)} \tilde{\bm Z}_k^{(1)}) \tilde{\bm Z}_k^{(1)\top} \bm v^{(b,6)}}{(1 + \tilde{\bm Z}_k^{(1)\top} \bm B^{(b,6)} \tilde{\bm Z}_k^{(1)})^2} \\
    &\qquad + \frac{\bm u^{(b,7)\top} \tilde{\bm Z}_k^{(1)} (\tilde{\bm Z}_k^{(1)\top} \bm A^{(b,7)} \tilde{\bm Z}_k^{(1)})(\tilde{\bm Z}_k^{(1)\top} \bm B^{(b,7)} \tilde{\bm Z}_k^{(1)}) \tilde{\bm Z}_k^{(1)\top} \bm v^{(b,7)}}{(1 + \tilde{\bm Z}_k^{(1)\top} \bm C^{(b,7)} \tilde{\bm Z}_k^{(1)})^2}
\end{align*}
and similarly
\begin{align*}
    &B_3(\bm Z^{(1),k-1}, \bm Z^{(2)}) - B_3( \bm Z^{(1),\setminus k}, \bm Z^{(2)}) \\
    &:= \bm u^{(b,1)\top} {\bm Z}_k^{(1)} {\bm Z}_k^{(1)\top} \bm v^{(b,1)} 
    + \frac{\bm u^{(b,2)\top} {\bm Z}_k^{(1)} {\bm Z}_k^{(1)\top} \bm v^{(b,2)}}{1 + {\bm Z}_k^{(1)\top} \bm A^{(b,2)} {\bm Z}_k^{(1)}} 
    \\
    &\qquad + \frac{\bm u^{(b,3)\top} {\bm Z}_k^{(1)} ({\bm Z}_k^{(1)\top} \bm A^{(b,3)} {\bm Z}_k^{(1)}) {\bm Z}_k^{(1)\top} \bm v^{(b,3)}}{1 + {\bm Z}_k^{(1)\top} \bm B^{(b,3)} {\bm Z}_k^{(1)}} \\
    &\qquad + \frac{\bm u^{(b,4)\top} {\bm Z}_k^{(1)}  {\bm Z}_k^{(1)\top} \bm v^{(b,4)}}{1 + {\bm Z}_k^{(1)\top} \bm A^{(b,4)} {\bm Z}_k^{(1)}} + \frac{\bm u^{(b,5)\top} {\bm Z}_k^{(1)} ({\bm Z}_k^{(1)\top} \bm A^{(b,5)} {\bm Z}_k^{(1)}) {\bm Z}_k^{(1)\top} \bm v^{(b,5)}}{1 + {\bm Z}_k^{(1)\top} \bm B^{(b,5)} {\bm Z}_k^{(1)}} \\
    &\qquad + \frac{\bm u^{(b,6)\top} {\bm Z}_k^{(1)} ({\bm Z}_k^{(1)\top} \bm A^{(b,6)} {\bm Z}_k^{(1)}) {\bm Z}_k^{(1)\top} \bm v^{(b,6)}}{(1 + {\bm Z}_k^{(1)\top} \bm B^{(b,6)} {\bm Z}_k^{(1)})^2} \\
    &\qquad + \frac{\bm u^{(b,7)\top} {\bm Z}_k^{(1)} ({\bm Z}_k^{(1)\top} \bm A^{(b,7)} {\bm Z}_k^{(1)})({\bm Z}_k^{(1)\top} \bm B^{(b,7)} {\bm Z}_k^{(1)}) {\bm Z}_k^{(1)\top} \bm v^{(b,7)}}{(1 + {\bm Z}_k^{(1)\top} \bm C^{(b,7)} {\bm Z}_k^{(1)})^2}
\end{align*}
where
\begin{gather*}
    \bm u^{(b,1)} = -\frac{2\lambda}{n} (\bSigma^{(1)})^{1/2} (\hat \bSigma_{\setminus k} + \lambda \bm I)^{-1} \bSigma^{(2)} (\hat \bSigma_{\setminus k} + \lambda \bm I)^{-1}\bbeta^{(2)}, \quad \bm v^{(b,1)} = (\bSigma^{(1)})^{1/2} \tilde \bbeta \\
    \bm u^{(b,2)} = \frac{2\lambda}{n^2} (\bSigma^{(1)})^{1/2} (\hat \bSigma_{\setminus k} + \lambda \bm I)^{-1} \bSigma^{(2)} (\hat \bSigma_{\setminus k} + \lambda \bm I)^{-1} \bbeta^{(2)}\\
    \bm v^{(b,2)} = (\bSigma^{(1)})^{1/2} (\hat \bSigma_{\setminus k} + \lambda \bm I)^{-1} (\bm X^{(1), \setminus k\top} \bm X^{(1),\setminus k}) \tilde \bbeta \\
    \bm A^{(b,2)}= \frac{1}{n} (\bSigma^{(1)})^{1/2}(\hat \bSigma_{\setminus k} + \lambda \bm I)^{-1} (\bSigma^{(1)})^{1/2} \\
    \bm u^{(b,3)} = \frac{2\lambda}{n^2} (\bSigma^{(1)})^{1/2} (\hat \bSigma_{\setminus k} + \lambda \bm I)^{-1} \bSigma^{(2)} (\hat \bSigma_{\setminus k} + \lambda\bm I)^{-1} \bbeta^{(2)}, \quad \bm v^{(b,3)} = (\bSigma^{(1)})^{1/2} \tilde \bbeta \\
    \bm A^{(b,3)} = (\bSigma^{(1)})^{1/2} (\hat \bSigma_{\setminus k} + \lambda \bm I)^{-1} (\bSigma^{(1)})^{1/2}, \quad \bm B^{(b,3)} = \frac{1}{n} (\bSigma^{(1)})^{1/2} (\hat \bSigma_{\setminus k} + \lambda \bm I)^{-1} (\bSigma^{(1)})^{1/2} \\
    \bm u^{(b,4)} = \frac{2\lambda}{n^2} (\bSigma^{(1)})^{1/2} (\hat \bSigma_{\setminus k} + \lambda \bm I)^{-1} \bbeta^{(2)} \\
    \bm v^{(b,4)} = (\bSigma^{(1)})^{1/2} (\hat \bSigma_{\setminus k} + \lambda \bm I)^{-1} \bSigma^{(2)} (\hat \bSigma_{\setminus k} + \lambda \bm I)^{-1} (\bm X^{(1),\setminus k\top} \bm X^{(1),\setminus k}) \tilde \bbeta \\
    \bm A^{(b,4)} = \frac{1}{n} (\bSigma^{(1)})^{1/2} (\hat \bSigma_{\setminus k} + \lambda \bm I)^{-1} (\bSigma^{(1)})^{1/2} \\
    \bm u^{(b,5)} = \frac{2\lambda}{n^2} (\bSigma^{(1)})^{1/2} (\hat \bSigma_{\setminus k} + \lambda \bm I)^{-1}  \bbeta^{(2)}, \quad \bm v^{(b,5)} = (\bSigma^{(1)})^{1/2} \tilde \bbeta \\
    \bm A^{(b,5)} = (\bSigma^{(1)})^{1/2} (\hat \bSigma_{\setminus k} + \lambda \bm I)^{-1}  \bSigma^{(2)} (\hat \bSigma_{\setminus k} + \lambda \bm I)^{-1} (\bSigma^{(1)})^{1/2} \\
    \bm B^{(b,5)} = \frac{1}{n} (\bSigma^{(1)})^{1/2} (\hat \bSigma_{\setminus k} + \lambda \bm I)^{-1} (\bSigma^{(1)})^{1/2}  \\
    \bm u^{(b,6)} = -\frac{2\lambda}{n^3} (\bSigma^{(1)})^{1/2} (\hat \bSigma_{\setminus k} + \lambda \bm I)^{-1}  \bbeta^{(2)} \\ \bm v^{(b,6)} = (\bSigma^{(1)})^{1/2} (\hat \bSigma_{\setminus k} + \lambda \bm I)^{-1} (\bm X^{(1),\setminus k\top} \bm X^{(1),\setminus k}) \tilde \bbeta \\
    \bm A^{(b,6)} = (\bSigma^{(1)})^{1/2} (\hat \bSigma_{\setminus k} + \lambda \bm I)^{-1}  \bSigma^{(2)} (\hat \bSigma_{\setminus k} + \lambda \bm I)^{-1} (\bSigma^{(1)})^{1/2} \\ \bm B^{(b,6)} = \frac{1}{n} (\bSigma^{(1)})^{1/2} (\hat \bSigma_{\setminus k} + \lambda \bm I)^{-1} (\bSigma^{(1)})^{1/2} \\
    \bm u^{(b,7)} = -\frac{2\lambda}{n^3} (\bSigma^{(1)})^{1/2} (\hat \bSigma_{\setminus k} + \lambda \bm I)^{-1}  \bbeta^{(2)}, \quad \bm v^{(b,7)} = (\bSigma^{(1)})^{1/2} \tilde \bbeta \\
    \bm A^{(b,7)} = (\bSigma^{(1)})^{1/2} (\hat \bSigma_{\setminus k} + \lambda \bm I)^{-1}  \bSigma^{(2)} (\hat \bSigma_{\setminus k} + \lambda \bm I)^{-1} (\bSigma^{(1)})^{1/2} \\ \bm B^{(b,7)} = (\bSigma^{(1)})^{1/2} (\hat \bSigma_{\setminus k} + \lambda \bm I)^{-1} (\bSigma^{(1)})^{1/2} \\
    \bm C^{(b,7)} = \frac{1}{n} (\bSigma^{(1)})^{1/2} (\hat \bSigma_{\setminus k} + \lambda \bm I)^{-1} (\bSigma^{(1)})^{1/2}
\end{gather*}
are all independent of $\bm Z_k^{(1)}, \tilde{\bm Z}_k^{(1)}$.

Similarly, we can express 
\begin{align*}
    &B_3(\tilde{\bm Z}^{(1)}, \bm Z^{(2),k}) \\&= -2 \lambda \bbeta^{(2)\top} (\tilde \bSigma_{k} + \lambda \bm I)^{-1} \bSigma^{(2)} (\tilde \bSigma_{k} + \lambda \bm I)^{-1}\biggl(\frac{\tilde{\bm X}^{(1)\top}\tilde{\bm X}^{(1)}}{n}\biggr) \tilde \bbeta \\
    &=-2\lambda \bbeta^{(2)\top} \biggl((\tilde \bSigma_{\setminus k} + \lambda \bm I)^{-1} - \frac{1}{n} \frac{(\tilde \bSigma_{\setminus k} + \lambda \bm I)^{-1} \tilde{\bm X}_k^{(2)} \tilde{\bm X}_k^{(2)\top} (\tilde \bSigma_{\setminus k} + \lambda \bm I)^{-1}}{1 + \frac{1}{n} \tilde{\bm X}_k^{(2)\top} (\tilde \bSigma_{\setminus k} + \lambda \bm I)^{-1} \tilde{\bm X}_k^{(2)}} \biggr) \bSigma^{(2)} \\
    &\qquad \cdot  \biggl((\tilde \bSigma_{\setminus k} + \lambda \bm I)^{-1} - \frac{1}{n} \frac{(\tilde \bSigma_{\setminus k} + \lambda \bm I)^{-1} \tilde{\bm X}_k^{(2)} \tilde{\bm X}_k^{(2)\top} (\tilde \bSigma_{\setminus k} + \lambda \bm I)^{-1}}{1 + \frac{1}{n} \tilde{\bm X}_k^{(2)\top} (\tilde \bSigma_{\setminus k} + \lambda \bm I)^{-1} \tilde{\bm X}_k^{(2)}} \biggr) \biggl(\frac{\tilde{\bm X}^{(1)\top} \tilde{\bm X}^{(1)}}{n}\biggr)\tilde \bbeta
\end{align*}
and therefore
\begin{align*}
    &B_3(\tilde{\bm Z}^{(1)}, \bm Z^{(2),k}) - B_3(\tilde{\bm Z}^{(1)}, \bm Z^{(2),\setminus k}) \\
    &= 2\lambda \bbeta^{(2)\top} (\tilde \bSigma_{\setminus k} + \lambda \bm I)^{-1} \bSigma^{(2)} \\
    &\hspace{4em}\biggl(\frac{1}{n} \frac{(\tilde \bSigma_{\setminus k} + \lambda \bm I)^{-1} \tilde{\bm X}_k^{(2)} \tilde{\bm X}_k^{(2)\top} (\tilde \bSigma_{\setminus k} + \lambda \bm I)^{-1}}{1 + \frac{1}{n} \tilde{\bm X}_k^{(2)\top} (\tilde \bSigma_{\setminus k} + \lambda \bm I)^{-1} \tilde{\bm X}_k^{(2)}} \biggr) \biggl(\frac{\tilde{\bm X}^{(1)\top} \tilde{\bm X}^{(1)}}{n} \biggr) \tilde \bbeta  \\
    &\qquad + 2 \lambda \bbeta^{(2)\top} \biggl(\frac{1}{n} \frac{(\tilde \bSigma_{\setminus k} + \lambda \bm I)^{-1} \tilde{\bm X}_k^{(2)} \tilde{\bm X}_k^{(2)\top} (\tilde \bSigma_{\setminus k} + \lambda \bm I)^{-1}}{1 + \frac{1}{n} \tilde{\bm X}_k^{(2)\top} (\tilde \bSigma_{\setminus k} + \lambda \bm I)^{-1} \tilde{\bm X}_k^{(2)}} \biggr) \bSigma^{(2)} \\
    &\hspace{4em}(\tilde \bSigma_{\setminus k} + \lambda \bm I)^{-1} \biggl(\frac{\tilde{\bm X}^{(1)\top} \tilde{\bm X}^{(1)}}{n} \biggr) \tilde \bbeta \\
    &\qquad - 2 \lambda \bbeta^{(2)\top} \biggl( \frac{1}{n} \frac{(\tilde \bSigma_{\setminus k} + \lambda \bm I)^{-1} \tilde{\bm X}_k^{(2)} \tilde{\bm X}_k^{(2)\top} (\tilde \bSigma_{\setminus k} + \lambda \bm I)^{-1}}{1 + \frac{1}{n} \tilde{\bm X}_k^{(2)\top} (\tilde \bSigma_{\setminus k} + \lambda \bm I)^{-1} \tilde{\bm X}_k^{(2)}}\biggr) \\
    &\hspace{4em}\bSigma^{(2)} \biggl( \frac{1}{n} \frac{(\tilde \bSigma_{\setminus k} + \lambda \bm I)^{-1} \tilde{\bm X}_k^{(2)} \tilde{\bm X}_k^{(2)\top} (\tilde \bSigma_{\setminus k} + \lambda \bm I)^{-1}}{1 + \frac{1}{n} \tilde{\bm X}_k^{(2)\top} (\tilde \bSigma_{\setminus k} + \lambda \bm I)^{-1} \tilde{\bm X}_k^{(2)}}\biggr) \biggl(\frac{\tilde{\bm X}^{(1)\top} \tilde{\bm X}^{(1)}}{n} \biggr) \tilde \bbeta \\
    &:= \frac{\bm u^{(B,1)\top} \tilde{\bm Z}_k^{(2)} \tilde{\bm Z}_k^{(2)\top} \bm v^{(B,1)}}{1 + \tilde{\bm Z}_k^{(2)\top} \bm A^{(B,1)} \tilde{\bm Z}_k^{(2)}} + 
    \frac{\bm u^{(B,2)\top} \tilde{\bm Z}_k^{(2)} \tilde{\bm Z}_k^{(2)\top} \bm v^{(B,2)}}{1 + \tilde{\bm Z}_k^{(2)\top} \bm A^{(B,2)} \tilde{\bm Z}_k^{(2)}} \\
    &\qquad +
    \frac{\bm u^{(B,3)\top} \tilde{\bm Z}_k^{(2)} (\tilde{\bm Z}_k^{(2)\top} \bm A^{(B,3)} \tilde{\bm Z}_k^{(2)}) \tilde{\bm Z}_k^{(2)\top} \bm v^{(B,3)}}{(1 + \tilde{\bm Z}_k^{(2)\top} \bm B^{(B,3)} \tilde{\bm Z}_k^{(2)})^2}
\end{align*}
and similarly
\begin{align*}
    &B_3(\tilde{\bm Z}^{(1)}, \bm Z^{(2),k-1}) - B_3(\tilde{\bm Z}^{(1)}, \bm Z^{(2),\setminus k}) \\
    &= \frac{\bm u^{(B,1)\top} \bm Z_k^{(2)} \bm Z_k^{(2)\top} \bm v^{(B,1)}}{1 + \bm Z_k^{(2)\top} \bm A^{(B,1)} \bm Z_k^{(2)}} + 
    \frac{\bm u^{(B,2)\top} \bm Z_k^{(2)} \bm Z_k^{(2)\top} \bm v^{(B,2)}}{1 + \bm Z_k^{(2)\top} \bm A^{(B,2)} \bm Z_k^{(2)}} \\
    &\qquad +
    \frac{\bm u^{(B,3)\top} \bm Z_k^{(2)} (\bm Z_k^{(2)\top} \bm A^{(B,3)} \bm Z_k^{(2)}) \bm Z_k^{(2)\top} \bm v^{(B,3)}}{(1 + \bm Z_k^{(2)\top} \bm B^{(B,3)} \bm Z_k^{(2)})^2}
\end{align*}
where
\begin{gather*}
    \bm u^{(B,1)} = \frac{2\lambda}{n^2} (\bSigma^{(2)})^{1/2} (\tilde{\bSigma}_{\setminus k} + \lambda \bm I)^{-1} \bSigma^{(2)} (\tilde \bSigma_{\setminus k} + \lambda \bm I)^{-1} \bbeta^{(2)} \\ \bm v^{(B,1)} = (\bSigma^{(2)})^{1/2} (\tilde \bSigma_{\setminus k} + \lambda\bm I)^{-1} (\tilde{\bm X}^{(1)\top} \tilde{\bm X}^{(1)}) \tilde \bbeta \\
    \bm A^{(B,1)} = \frac{1}{n} (\bSigma^{(2)})^{1/2} (\tilde \bSigma_{\setminus k} + \lambda\bm I)^{-1} (\bSigma^{(2)})^{1/2} \\
    \bm u^{(B,2)} = \frac{2\lambda}{n^2} (\bSigma^{(2)})^{1/2}  (\tilde \bSigma_{\setminus k} + \lambda \bm I)^{-1} \bbeta^{(2)} \\ \bm v^{(B,2)} = (\bSigma^{(2)})^{1/2}(\tilde{\bSigma}_{\setminus k} + \lambda \bm I)^{-1} \bSigma^{(2)} (\tilde \bSigma_{\setminus k} + \lambda\bm I)^{-1} (\tilde{\bm X}^{(1)\top} \tilde{\bm X}^{(1)}) \tilde \bbeta \\
    \bm A^{(B,2)} = \frac{1}{n} (\bSigma^{(2)})^{1/2} (\tilde \bSigma_{\setminus k} + \lambda\bm I)^{-1} (\bSigma^{(2)})^{1/2} \\
    \bm u^{(B,3)} = -\frac{2\lambda}{n^3} (\bSigma^{(2)})^{1/2}  (\tilde \bSigma_{\setminus k} + \lambda \bm I)^{-1} \bbeta^{(2)} \\
    \bm v^{(B,2)} = (\bSigma^{(2)})^{1/2}(\tilde{\bSigma}_{\setminus k} + \lambda \bm I)^{-1}  (\tilde{\bm X}^{(1)\top} \tilde{\bm X}^{(1)}) \tilde \bbeta \\
    \bm A^{(B,1)} = (\bSigma^{(2)})^{1/2} (\tilde \bSigma_{\setminus k} + \lambda\bm I)^{-1} \bSigma^{(2)}  (\tilde \bSigma_{\setminus k} + \lambda\bm I)^{-1}(\bSigma^{(2)})^{1/2} \\ \bm B^{(B,1)} = \frac{1}{n} (\bSigma^{(2)})^{1/2} (\tilde \bSigma_{\setminus k} + \lambda\bm I)^{-1} (\bSigma^{(2)})^{1/2}
\end{gather*}

\subsubsection{Expanding \texorpdfstring{$B_2$}{B2} terms}
Now, recall that
\begin{equation*}
    B_2(\bm Z^{(1)}, \bm Z^{(2)}) = \tilde \bbeta^\top \biggl(\frac{\bm X^{(1)\top} \bm X^{(1)}}{n}\biggr) (\hat \bSigma+ \lambda \bm I)^{-1} \bSigma^{(2)} (\hat \bSigma+ \lambda \bm I)^{-1} \biggl(\frac{\bm X^{(1)\top} \bm X^{(1)}}{n} \biggr) \tilde \bbeta
\end{equation*}
an using the result from earlier, we have
\begin{align*}
    &B_2(\bm Z^{(1),k}, \bm Z^{(2)}) \\&= \tilde \bbeta^\top \biggl(\frac{\bm X^{(1),k\top} \bm X^{(1),k}}{n}\biggr) (\hat \bSigma_k + \lambda \bm I)^{-1} \bSigma^{(2)} (\hat \bSigma_k + \lambda \bm I)^{-1} \biggl(\frac{\bm X^{(1),k\top} \bm X^{(1),k}}{n} \biggr) \tilde \bbeta \\
    &= \tilde \bbeta^\top \biggl(\frac{\bm X^{(1),\setminus k\top} \bm X^{(1),\setminus k}}{n} + \frac{\tilde {\bm X}^{(1)}_k \tilde{\bm X}_k^{(1)\top}}{n}\biggr)  \\
    &\qquad \biggl((\hat \bSigma_{\setminus k} + \lambda \bm I)^{-1} - \frac{1}{n} \frac{(\hat \bSigma_{\setminus k} + \lambda \bm I)^{-1} \tilde{\bm X}_k^{(1)} \tilde{\bm X}_k^{(1)\top} (\hat \bSigma_{\setminus k} + \lambda \bm I)^{-1}}{1 + \frac{1}{n} \tilde{\bm X}_k^{(1)\top} (\hat \bSigma_{\setminus k} + \lambda \bm I)^{-1} \tilde{\bm X}_k^{(1)}} \biggr)  \bSigma^{(2)} \\
    &\qquad \biggl((\hat \bSigma_{\setminus k} + \lambda \bm I)^{-1} - \frac{1}{n} \frac{(\hat \bSigma_{\setminus k} + \lambda \bm I)^{-1} \tilde{\bm X}_k^{(1)} \tilde{\bm X}_k^{(1)\top} (\hat \bSigma_{\setminus k} + \lambda \bm I)^{-1}}{1 + \frac{1}{n} \tilde{\bm X}_k^{(1)\top} (\hat \bSigma_{\setminus k} + \lambda \bm I)^{-1} \tilde{\bm X}_k^{(1)}} \biggr) \\
    &\qquad \biggl(\frac{\bm X^{(1),\setminus k\top} \bm X^{(1),\setminus k}}{n} + \frac{\tilde {\bm X}^{(1)}_k \tilde{\bm X}_k^{(1)\top}}{n}\biggr) \tilde \bbeta
\end{align*}
Therefore, we finally have
\begin{align*}
    &B_2(\bm Z^{(1),k}, \bm Z^{(2)}) - B_2( \bm Z^{(1),\setminus k}, \bm Z^{(2)}) \\
    &= 2\tilde \bbeta^\top \biggl(\frac{\tilde {\bm X}^{(1),\setminus k\top } \tilde {\bm X}^{(1), \setminus k}}{n}\biggr) (\hat \bSigma_{\setminus k} + \lambda \bm I)^{-1} \bSigma^{(2)} (\hat \bSigma_{\setminus k} + \lambda \bm I)^{-1} \biggl(\frac{\tilde{\bm X}_k^{(1)} \tilde{\bm X}_k^{(1)\top}}{n} \biggr) \tilde \bbeta \\
    &\qquad - 2\tilde \bbeta^\top \biggl(\frac{\tilde {\bm X}^{(1),\setminus k\top } \tilde {\bm X}^{(1), \setminus k}}{n}\biggr) (\hat \bSigma_{\setminus k} + \lambda \bm I)^{-1} \bSigma^{(2)}  \\
    &\qquad \qquad \cdot \biggl(\frac{1}{n} \frac{(\hat \bSigma_{\setminus k} +  \lambda \bm I)^{-1} \tilde{\bm X}^{(1)}_k \tilde{\bm X}^{(1)\top}_k (\hat \bSigma_{\setminus k} + \lambda \bm I)^{-1}}{1 +  \frac{1}{n}\tilde{\bm X}^{(1)\top}_k (\hat \bSigma_{\setminus k} + \lambda \bm I)^{-1} \tilde{\bm X}^{(1)}_k}\biggr)\biggl(\frac{\bm X^{(1),\setminus k\top} \bm X^{(1),\setminus k}}{n}\biggr) \tilde \bbeta \\
    &\qquad - 2\tilde \bbeta^\top \biggl(\frac{\tilde {\bm X}^{(1),\setminus k\top } \tilde {\bm X}^{(1), \setminus k}}{n}\biggr) (\hat \bSigma_{\setminus k} + \lambda \bm I)^{-1} \bSigma^{(2)} \\
    &\qquad \qquad \cdot \biggl(\frac{1}{n} \frac{(\hat \bSigma_{\setminus k} +  \lambda \bm I)^{-1} \tilde{\bm X}^{(1)}_k \tilde{\bm X}^{(1)\top}_k (\hat \bSigma_{\setminus k} + \lambda \bm I)^{-1}}{1 +  \frac{1}{n}\tilde{\bm X}^{(1)\top}_k (\hat \bSigma_{\setminus k} + \lambda \bm I)^{-1} \tilde{\bm X}^{(1)}_k}\biggr) \biggl(\frac{\tilde {\bm X}^{(1)}_k \tilde {\bm X}_k^{(1)\top}}{n}\biggr) \tilde \bbeta \\
    &\qquad - 2\tilde \bbeta^\top \biggl(\frac{\tilde {\bm X}^{(1),\setminus k\top } \tilde {\bm X}^{(1), \setminus k}}{n}\biggr) \biggl(\frac{1}{n} \frac{(\hat \bSigma_{\setminus k} +  \lambda \bm I)^{-1} \tilde{\bm X}^{(1)}_k \tilde{\bm X}^{(1)\top}_k (\hat \bSigma_{\setminus k} + \lambda \bm I)^{-1}}{1 +  \frac{1}{n}\tilde{\bm X}^{(1)\top}_k (\hat \bSigma_{\setminus k} + \lambda \bm I)^{-1} \tilde{\bm X}^{(1)}_k}\biggr) \\
    &\qquad \qquad \cdot \bSigma^{(2)} (\hat \bSigma_{\setminus k} + \lambda \bm I)^{-1} \biggl(\frac{\tilde {\bm X}^{(1)}_k \tilde {\bm X}_k^{(1)\top}}{n}\biggr) \tilde \bbeta \\
    &\qquad + \tilde \bbeta^\top \biggl(\frac{\tilde {\bm X}^{(1),\setminus k\top } \tilde {\bm X}^{(1), \setminus k}}{n}\biggr) \biggl(\frac{1}{n} \frac{(\hat \bSigma_{\setminus k} +  \lambda \bm I)^{-1} \tilde{\bm X}^{(1)}_k \tilde{\bm X}^{(1)\top}_k (\hat \bSigma_{\setminus k} + \lambda \bm I)^{-1}}{1 +  \frac{1}{n}\tilde{\bm X}^{(1)\top}_k (\hat \bSigma_{\setminus k} + \lambda \bm I)^{-1} \tilde{\bm X}^{(1)}_k}\biggr) \bSigma^{(2)} \\
    &\hspace{4em} \biggl(\frac{1}{n} \frac{(\hat \bSigma_{\setminus k} +  \lambda \bm I)^{-1} \tilde{\bm X}^{(1)}_k \tilde{\bm X}^{(1)\top}_k (\hat \bSigma_{\setminus k} + \lambda \bm I)^{-1}}{1 +  \frac{1}{n}\tilde{\bm X}^{(1)\top}_k (\hat \bSigma_{\setminus k} + \lambda \bm I)^{-1} \tilde{\bm X}^{(1)}_k}\biggr) \biggl(\frac{\tilde {\bm X}^{(1),\setminus k\top } \tilde {\bm X}^{(1), \setminus k}}{n}\biggr) \tilde \bbeta \\
    &\qquad + 2\tilde \bbeta^\top\biggl(\frac{\tilde {\bm X}^{(1),\setminus k\top } \tilde {\bm X}^{(1), \setminus k}}{n}\biggr) \biggl(\frac{1}{n} \frac{(\hat \bSigma_{\setminus k} +  \lambda \bm I)^{-1} \tilde{\bm X}^{(1)}_k \tilde{\bm X}^{(1)\top}_k (\hat \bSigma_{\setminus k} + \lambda \bm I)^{-1}}{1 +  \frac{1}{n}\tilde{\bm X}^{(1)\top}_k (\hat \bSigma_{\setminus k} + \lambda \bm I)^{-1} \tilde{\bm X}^{(1)}_k}\biggr) \bSigma^{(2)} \\
    &\hspace{4em} \biggl(\frac{1}{n} \frac{(\hat \bSigma_{\setminus k} +  \lambda \bm I)^{-1} \tilde{\bm X}^{(1)}_k \tilde{\bm X}^{(1)\top}_k (\hat \bSigma_{\setminus k} + \lambda \bm I)^{-1}}{1 +  \frac{1}{n}\tilde{\bm X}^{(1)\top}_k (\hat \bSigma_{\setminus k} + \lambda \bm I)^{-1} \tilde{\bm X}^{(1)}_k}\biggr) \biggl(\frac{\tilde {\bm X}^{(1)}_k \tilde {\bm X}_k^{(1)\top}}{n}\biggr) \tilde \bbeta \\
    &\qquad + \tilde \bbeta^\top \biggl(\frac{\tilde {\bm X}^{(1)}_k \tilde {\bm X}_k^{(1)\top}}{n}\biggr) (\hat \bSigma_{\setminus k} + \lambda \bm I)^{-1} \bSigma^{(2)} (\hat \bSigma_{\setminus k} + \lambda \bm I)^{-1} \biggl(\frac{\tilde{\bm X}_k^{(1)} \tilde{\bm X}_k^{(1)\top}}{n} \biggr) \tilde \bbeta \\
    &\qquad - 2\tilde \bbeta^\top \biggl(\frac{\tilde {\bm X}^{(1)}_k \tilde {\bm X}_k^{(1)\top}}{n}\biggr) \biggl(\frac{1}{n} \frac{(\hat \bSigma_{\setminus k} +  \lambda \bm I)^{-1} \tilde{\bm X}^{(1)}_k \tilde{\bm X}^{(1)\top}_k (\hat \bSigma_{\setminus k} + \lambda \bm I)^{-1}}{1 +  \frac{1}{n}\tilde{\bm X}^{(1)\top}_k (\hat \bSigma_{\setminus k} + \lambda \bm I)^{-1} \tilde{\bm X}^{(1)}_k}\biggr) \bSigma^{(2)}\\
    &\hspace{4em} (\hat \bSigma_{\setminus k} + \lambda \bm I)^{-1} \biggl(\frac{\tilde {\bm X}^{(1)}_k \tilde {\bm X}_k^{(1)\top}}{n}\biggr) \tilde \bbeta \\
    &\qquad + \tilde \bbeta^\top\biggl(\frac{\tilde {\bm X}^{(1)}_k \tilde {\bm X}_k^{(1)\top}}{n}\biggr) \biggl(\frac{1}{n} \frac{(\hat \bSigma_{\setminus k} +  \lambda \bm I)^{-1} \tilde{\bm X}^{(1)}_k \tilde{\bm X}^{(1)\top}_k (\hat \bSigma_{\setminus k} + \lambda \bm I)^{-1}}{1 +  \frac{1}{n}\tilde{\bm X}^{(1)\top}_k (\hat \bSigma_{\setminus k} + \lambda \bm I)^{-1} \tilde{\bm X}^{(1)}_k}\biggr) \bSigma^{(2)} \\
    &\hspace{4em} \biggl(\frac{1}{n} \frac{(\hat \bSigma_{\setminus k} +  \lambda \bm I)^{-1} \tilde{\bm X}^{(1)}_k \tilde{\bm X}^{(1)\top}_k (\hat \bSigma_{\setminus k} + \lambda \bm I)^{-1}}{1 +  \frac{1}{n}\tilde{\bm X}^{(1)\top}_k (\hat \bSigma_{\setminus k} + \lambda \bm I)^{-1} \tilde{\bm X}^{(1)}_k}\biggr) \biggl(\frac{\tilde {\bm X}^{(1)}_k \tilde {\bm X}_k^{(1)\top}}{n}\biggr) \tilde \bbeta \\
    &:= \bm u^{(c,1)\top} \tilde{\bm Z}_k^{(1)} \tilde{\bm Z}_k^{(1)\top} \bm v^{(c,1)} 
    + \frac{\bm u^{(c,2)\top} \tilde{\bm Z}_k^{(1)} \tilde{\bm Z}_k^{(1)\top} \bm v^{(c,2)}}{1 + \tilde{\bm Z}_k^{(1)\top} \bm A^{(c,2)} \tilde{\bm Z}_k^{(1)}} \\
    &\qquad + \frac{\bm u^{(c,3)\top} \tilde{\bm Z}_k^{(1)} (\tilde{\bm Z}_k^{(1)\top} \bm A^{(c,3)} \tilde{\bm Z}_k^{(1)}) \tilde{\bm Z}_k^{(1)\top} \bm v^{(c,3)}}{1 + \tilde{\bm Z}_k^{(1)\top} \bm B^{(c,3)} \tilde{\bm Z}_k^{(1)}} \\
    &\qquad + \frac{\bm u^{(c,4)\top} \tilde{\bm Z}_k^{(1)} (\tilde{\bm Z}_k^{(1)\top} \bm A^{(c,4)} \tilde{\bm Z}_k^{(1)}) \tilde{\bm Z}_k^{(1)\top} \bm v^{(c,4)}}{1 + \tilde{\bm Z}_k^{(1)\top} \bm B^{(c,4)} \tilde{\bm Z}_k^{(1)}} \\
    &\qquad+ \frac{\bm u^{(c,5)\top} \tilde{\bm Z}_k^{(1)} (\tilde{\bm Z}_k^{(1)\top} \bm A^{(c,5)} \tilde{\bm Z}_k^{(1)}) \tilde{\bm Z}_k^{(1)\top} \bm v^{(c,5)}}{(1 + \tilde{\bm Z}_k^{(1)\top} \bm B^{(c,5)} \tilde{\bm Z}_k^{(1)})^2} \\
    &\qquad + \frac{\bm u^{(c,6)\top} \tilde{\bm Z}_k^{(1)} (\tilde{\bm Z}_k^{(1)\top} \bm A^{(c,6)} \tilde{\bm Z}_k^{(1)})(\tilde{\bm Z}_k^{(1)\top} \bm B^{(c,6)} \tilde{\bm Z}_k^{(1)}) \tilde{\bm Z}_k^{(1)\top} \bm v^{(c,6)}}{(1 + \tilde{\bm Z}_k^{(1)\top} \bm C^{(c,6)} \tilde{\bm Z}_k^{(1)})^2} \\
    &\qquad + \bm u^{(c,7)\top} \tilde{\bm Z}_k^{(1)} (\tilde{\bm Z}_k^{(1)\top} \bm A^{(c,7)} \tilde{\bm Z}_k^{(1)}) \tilde{\bm Z}_k^{(1)\top} \bm v^{(c,7)} \\
    &\qquad + \frac{(\bm u^{(c,8)\top} \tilde{\bm Z}_k^{(1)})^2 (\tilde{\bm Z}_k^{(1)\top} \bm A^{(c,8)} \tilde{\bm Z}_k^{(1)})(\tilde{\bm Z}_k^{(1)\top} \bm B^{(c,6)} \tilde{\bm Z}_k^{(1)})}{1 + \tilde{\bm Z}_k^{(1)\top} \bm C^{(c,8)} \tilde{\bm Z}_k^{(1)}} \\
    &\qquad + \frac{(\bm u^{(c,9)\top} \tilde{\bm Z}_k^{(1)})^2 (\tilde{\bm Z}_k^{(1)\top} \bm A^{(c,9)} \tilde{\bm Z}_k^{(1)})(\tilde{\bm Z}_k^{(1)\top} \bm B^{(c,9)} \tilde{\bm Z}_k^{(1)})(\tilde{\bm Z}_k^{(1)\top} \bm C^{(c,9)} \tilde{\bm Z}_k^{(1)})}{(1 + \tilde{\bm Z}_k^{(1)\top} \bm D^{(c,9)} \tilde{\bm Z}_k^{(1)})^2}
\end{align*}
and similarly 
\begin{align*}
    &B_2(\bm Z^{(1),k-1}, \bm Z^{(2)}) - B_2( \bm Z^{(1),\setminus k}, \bm Z^{(2)}) \\
    &:= \bm u^{(c,1)\top} {\bm Z}_k^{(1)} {\bm Z}_k^{(1)\top} \bm v^{(c,1)} 
    + \frac{\bm u^{(c,2)\top} {\bm Z}_k^{(1)} {\bm Z}_k^{(1)\top} \bm v^{(c,2)}}{1 + {\bm Z}_k^{(1)\top} \bm A^{(c,2)} {\bm Z}_k^{(1)}} \\
    &\qquad + \frac{\bm u^{(c,3)\top} {\bm Z}_k^{(1)} ({\bm Z}_k^{(1)\top} \bm A^{(c,3)} {\bm Z}_k^{(1)}) {\bm Z}_k^{(1)\top} \bm v^{(c,3)}}{1 + {\bm Z}_k^{(1)\top} \bm B^{(c,3)} {\bm Z}_k^{(1)}} \\
    &\qquad + \frac{\bm u^{(c,4)\top} {\bm Z}_k^{(1)} ({\bm Z}_k^{(1)\top} \bm A^{(c,4)} {\bm Z}_k^{(1)}) {\bm Z}_k^{(1)\top} \bm v^{(c,4)}}{1 + {\bm Z}_k^{(1)\top} \bm B^{(c,4)} {\bm Z}_k^{(1)}} \\
    &\qquad+ \frac{\bm u^{(c,5)\top} {\bm Z}_k^{(1)} ({\bm Z}_k^{(1)\top} \bm A^{(c,5)} {\bm Z}_k^{(1)}) {\bm Z}_k^{(1)\top} \bm v^{(c,5)}}{(1 + {\bm Z}_k^{(1)\top} \bm B^{(c,5)} {\bm Z}_k^{(1)})^2} \\
    &\qquad + \frac{\bm u^{(c,6)\top} {\bm Z}_k^{(1)} ({\bm Z}_k^{(1)\top} \bm A^{(c,6)} {\bm Z}_k^{(1)})({\bm Z}_k^{(1)\top} \bm B^{(c,6)} {\bm Z}_k^{(1)}) {\bm Z}_k^{(1)\top} \bm v^{(c,6)}}{(1 + {\bm Z}_k^{(1)\top} \bm C^{(c,6)} {\bm Z}_k^{(1)})^2} \\
    &\qquad + \bm u^{(c,7)\top} {\bm Z}_k^{(1)} ({\bm Z}_k^{(1)\top} \bm A^{(c,7)} {\bm Z}_k^{(1)}) {\bm Z}_k^{(1)\top} \bm v^{(c,7)} \\
    &\qquad + \frac{(\bm u^{(c,8)\top} {\bm Z}_k^{(1)})^2 ({\bm Z}_k^{(1)\top} \bm A^{(c,8)} {\bm Z}_k^{(1)})({\bm Z}_k^{(1)\top} \bm B^{(c,6)} {\bm Z}_k^{(1)})}{1 + {\bm Z}_k^{(1)\top} \bm C^{(c,8)} {\bm Z}_k^{(1)}} \\
    &\qquad + \frac{(\bm u^{(c,9)\top} {\bm Z}_k^{(1)})^2 ({\bm Z}_k^{(1)\top} \bm A^{(c,9)} {\bm Z}_k^{(1)})({\bm Z}_k^{(1)\top} \bm B^{(c,9)} {\bm Z}_k^{(1)})({\bm Z}_k^{(1)\top} \bm C^{(c,9)} {\bm Z}_k^{(1)})}{(1 + {\bm Z}_k^{(1)\top} \bm D^{(c,9)} {\bm Z}_k^{(1)})^2}
\end{align*}
where
\begin{gather*}
    \bm u^{(c,1)} = \frac{2}{n^2} (\bSigma^{(1)})^{1/2} (\hat \bSigma_{\setminus k} + \lambda \bm I)^{-1} \bSigma^{(2)} (\hat \bSigma_{\setminus k} + \lambda \bm I)^{-1} (\tilde{\bm X}^{(1),\setminus k\top} \tilde{\bm X}^{(1),\setminus k}) \tilde \bbeta \\ \bm v^{(c,1)} = (\bSigma^{(1)})^{1/2} \tilde \bbeta \\
    \bm u^{(c,2)} = -\frac{2}{n^3} (\bSigma^{(1)})^{1/2} (\hat \bSigma_{\setminus k} + \lambda \bm I)^{-1} \bSigma^{(2)} (\hat \bSigma_{\setminus k} + \lambda \bm I)^{-1} (\tilde{\bm X}^{(1),\setminus k\top} \tilde{\bm X}^{(1),\setminus k}) \tilde \bbeta \\ \bm v^{(c,2)} = (\bSigma^{(1)})^{1/2} (\hat \bSigma_{\setminus k} + \lambda \bm I)^{-1} (\bm X^{(1), \setminus k\top} \bm X^{(1),\setminus k}) \tilde \bbeta \\
    \bm A^{(c,2)}= \frac{1}{n} (\bSigma^{(1)})^{1/2}(\hat \bSigma_{\setminus k} + \lambda \bm I)^{-1} (\bSigma^{(1)})^{1/2}, \bm v^{(c,3)} = (\bSigma^{(1)})^{1/2} \tilde \bbeta \\
    \bm u^{(c,3)} = -\frac{2}{n^3} (\bSigma^{(1)})^{1/2} (\hat \bSigma_{\setminus k} + \lambda \bm I)^{-1} \bSigma^{(2)} (\hat \bSigma_{\setminus k} + \lambda \bm I)^{-1} (\tilde{\bm X}^{(1),\setminus k\top} \tilde{\bm X}^{(1),\setminus k}) \tilde \bbeta \\
    \bm A^{(c,3)} = (\bSigma^{(1)})^{1/2} (\hat \bSigma_{\setminus k} + \lambda \bm I)^{-1} (\bSigma^{(1)})^{1/2}, \quad \bm B^{(c,3)} = \frac{1}{n} (\bSigma^{(1)})^{1/2} (\hat \bSigma_{\setminus k} + \lambda \bm I)^{-1} (\bSigma^{(1)})^{1/2} \\
    \bm u^{(c,4)} = -\frac{2}{n^3} (\bSigma^{(1)})^{1/2} (\hat \bSigma_{\setminus k} + \lambda \bm I)^{-1} (\tilde{\bm X}^{(1),\setminus k\top} \tilde{\bm X}^{(1),\setminus k}) \tilde \bbeta, \quad \bm v^{(c,4)} = (\bSigma^{(1)})^{1/2} \tilde \bbeta \\
    \bm A^{(c,4)} = (\bSigma^{(1)})^{1/2}(\hat \bSigma_{\setminus k} + \lambda \bm I)^{-1} \bSigma^{(2)}  (\hat \bSigma_{\setminus k} + \lambda \bm I)^{-1} (\bSigma^{(1)})^{1/2} \\ \bm B^{(c,4)} = \frac{1}{n} (\bSigma^{(1)})^{1/2} (\hat \bSigma_{\setminus k} + \lambda \bm I)^{-1} (\bSigma^{(1)})^{1/2} \\
    \bm u^{(c,5)} = \frac{1}{n^4} (\bSigma^{(1)})^{1/2} (\hat \bSigma_{\setminus k} + \lambda \bm I)^{-1} (\tilde{\bm X}^{(1),\setminus k\top} \tilde{\bm X}^{(1),\setminus k}) \tilde \bbeta \\ 
    \bm v^{(c,5)} = (\bSigma^{(1)})^{1/2} (\hat \bSigma_{\setminus k} + \lambda \bm I)^{-1} (\tilde{\bm X}^{(1),\setminus k\top} \tilde{\bm X}^{(1),\setminus k}) \tilde \bbeta \\
    \bm A^{(c,5)} = (\bSigma^{(1)})^{1/2} (\hat \bSigma_{\setminus k} + \lambda \bm I)^{-1}  \bSigma^{(2)} (\hat \bSigma_{\setminus k} + \lambda \bm I)^{-1} (\bSigma^{(1)})^{1/2} \\
    \bm B^{(c,5)} = \frac{1}{n} (\bSigma^{(1)})^{1/2} (\hat \bSigma_{\setminus k} + \lambda \bm I)^{-1} (\bSigma^{(1)})^{1/2} \\
    \bm u^{(c,6)} = \frac{2}{n^4} (\bSigma^{(1)})^{1/2} (\hat \bSigma_{\setminus k} + \lambda \bm I)^{-1} (\tilde{\bm X}^{(1),\setminus k\top} \tilde{\bm X}^{(1),\setminus k}) \tilde \bbeta, \quad \bm v^{(c,6)} = (\bSigma^{(1)})^{1/2} \tilde \bbeta \\
    \bm A^{(c,6)} = (\bSigma^{(1)})^{1/2} (\hat \bSigma_{\setminus k} + \lambda \bm I)^{-1}  \bSigma^{(2)} (\hat \bSigma_{\setminus k} + \lambda \bm I)^{-1} (\bSigma^{(1)})^{1/2} \\
    \bm B^{(c,6)} = (\bSigma^{(1)})^{1/2} (\hat \bSigma_{\setminus k} + \lambda \bm I)^{-1} (\bSigma^{(1)})^{1/2} \\
    \bm C^{(c,6)} = \frac{1}{n} (\bSigma^{(1)})^{1/2} (\hat \bSigma_{\setminus k} + \lambda \bm I)^{-1} (\bSigma^{(1)})^{1/2} \\
    \bm u^{(c,7)} = \frac{1}{n^2} (\bSigma^{(1)})^{1/2}\tilde \bbeta, \quad \bm v^{(c,7)} = (\bSigma^{(1)})^{1/2}\tilde \bbeta \\
    \bm A^{(c,7)} = (\bSigma^{(1)})^{1/2} (\hat \bSigma_{\setminus k} + \lambda \bm I)^{-1}  \bSigma^{(2)} (\hat \bSigma_{\setminus k} + \lambda \bm I)^{-1} (\bSigma^{(1)})^{1/2} \\
    \bm u^{(c,8)} =  (\bSigma^{(1)})^{1/2}  \tilde \bbeta, \quad \bm A^{(c,8)} = -\frac{2}{n^3} (\bSigma^{(1)})^{1/2} (\hat \bSigma_{\setminus k} + \lambda \bm I)^{-1} (\bSigma^{(1)})^{1/2}  \\
    \bm B^{(c,8)} = (\bSigma^{(1)})^{1/2} (\hat \bSigma_{\setminus k} + \lambda \bm I)^{-1}  \bSigma^{(2)} (\hat \bSigma_{\setminus k} + \lambda \bm I)^{-1} (\bSigma^{(1)})^{1/2} \\
    \bm C^{(c,8)} = \frac{1}{n} (\bSigma^{(1)})^{1/2} (\hat \bSigma_{\setminus k} + \lambda \bm I)^{-1} (\bSigma^{(1)})^{1/2} \\
    \bm u^{(c,9)} = (\bSigma^{(1)})^{1/2}  \tilde \bbeta, \quad  \bm A^{(c,9)} = \frac{1}{n^4} (\bSigma^{(1)})^{1/2} (\hat \bSigma_{\setminus k} + \lambda \bm I)^{-1} (\bSigma^{(1)})^{1/2} \\ \bm B^{(c,9)} = (\bSigma^{(1)})^{1/2} (\hat \bSigma_{\setminus k} + \lambda \bm I)^{-1}  \bSigma^{(2)} (\hat \bSigma_{\setminus k} + \lambda \bm I)^{-1} (\bSigma^{(1)})^{1/2} \\
    \bm C^{(c,9)} = (\bSigma^{(1)})^{1/2} (\hat \bSigma_{\setminus k} + \lambda \bm I)^{-1} (\bSigma^{(1)})^{1/2} \\ \bm D^{(c,9)} = \frac{1}{n} (\bSigma^{(1)})^{1/2} (\hat \bSigma_{\setminus k} + \lambda \bm I)^{-1} (\bSigma^{(1)})^{1/2}
\end{gather*}
Similarly, we have
\begin{align*}
    &B_2(\tilde{\bm Z}^{(1)}, \bm Z^{(2),k}) \\
    &= \tilde \bbeta^\top \biggl(\frac{\tilde{\bm X}^{(1)\top} \tilde{\bm X}^{(1)}}{n}\biggr) (\hat \bSigma_k + \lambda \bm I)^{-1} \bSigma^{(2)} (\hat \bSigma_k + \lambda \bm I)^{-1} \biggl(\frac{\tilde{\bm X}^{(1)\top} \tilde{\bm X}^{(1)}}{n}\biggr) \tilde \bbeta \\
    &= \tilde \bbeta^\top \biggl(\frac{\tilde{\bm X}^{(1)\top} \tilde{\bm X}^{(1)}}{n}\biggr)  \biggl((\hat \bSigma_{\setminus k} + \lambda \bm I)^{-1} - \frac{1}{n} \frac{(\hat \bSigma_{\setminus k} + \lambda \bm I)^{-1} \tilde{\bm X}_k^{(1)} \tilde{\bm X}_k^{(1)\top} (\hat \bSigma_{\setminus k} + \lambda \bm I)^{-1}}{1 + \frac{1}{n} \tilde{\bm X}_k^{(1)\top} (\hat \bSigma_{\setminus k} + \lambda \bm I)^{-1} \tilde{\bm X}_k^{(1)}} \biggr)  \bSigma^{(2)} \\
    &\qquad \biggl((\hat \bSigma_{\setminus k} + \lambda \bm I)^{-1} - \frac{1}{n} \frac{(\hat \bSigma_{\setminus k} + \lambda \bm I)^{-1} \tilde{\bm X}_k^{(1)} \tilde{\bm X}_k^{(1)\top} (\hat \bSigma_{\setminus k} + \lambda \bm I)^{-1}}{1 + \frac{1}{n} \tilde{\bm X}_k^{(1)\top} (\hat \bSigma_{\setminus k} + \lambda \bm I)^{-1} \tilde{\bm X}_k^{(1)}} \biggr) 
\biggl(\frac{\tilde{\bm X}^{(1)\top} \tilde{\bm X}^{(1)}}{n}\biggr) \tilde \bbeta
\end{align*}
and therefore
\begin{align*}
    &B_2(\tilde{\bm Z}^{(1)}, \bm Z^{(2),k})  - B_2(\tilde{\bm Z}^{(1)}, \bm Z^{(2),\setminus k}) \\
    &= -2\tilde \bbeta^\top \biggl(\frac{\tilde{\bm X}^{(1)\top} \tilde{\bm X}^{(1)}}{n}\biggr)  \biggl(\frac{1}{n} \frac{(\tilde \bSigma_{\setminus k} + \lambda \bm I)^{-1} \tilde{\bm X}_k^{(2)} \tilde{\bm X}_k^{(2)\top} (\tilde \bSigma_{\setminus k} + \lambda \bm I)^{-1}}{1 + \frac{1}{n} \tilde{\bm X}_k^{(2)\top} (\tilde \bSigma_{\setminus k} + \lambda \bm I)^{-1} \tilde{\bm X}_k^{(2)}}\biggr) \\
    &\hspace{4em}\bSigma^{(2)} (\tilde \bSigma_{\setminus k} + \lambda \bm I)^{-1}\biggl(\frac{\tilde{\bm X}^{(1)\top} \tilde{\bm X}^{(1)}}{n}\biggr) \tilde \bbeta \\
    &\qquad + \tilde \bbeta^\top \biggl(\frac{\tilde{\bm X}^{(1)\top} \tilde{\bm X}^{(1)}}{n}\biggr)  \biggl(\frac{1}{n} \frac{(\tilde \bSigma_{\setminus k} + \lambda \bm I)^{-1} \tilde{\bm X}_k^{(2)} \tilde{\bm X}_k^{(2)\top} (\tilde \bSigma_{\setminus k} + \lambda \bm I)^{-1}}{1 + \frac{1}{n} \tilde{\bm X}_k^{(2)\top} (\tilde \bSigma_{\setminus k} + \lambda \bm I)^{-1} \tilde{\bm X}_k^{(2)}}\biggr)  \bSigma^{(2)} \\
    &\hspace{4em}\biggl(\frac{1}{n} \frac{(\tilde \bSigma_{\setminus k} + \lambda \bm I)^{-1} \tilde{\bm X}_k^{(2)} \tilde{\bm X}_k^{(2)\top} (\tilde \bSigma_{\setminus k} + \lambda \bm I)^{-1}}{1 + \frac{1}{n} \tilde{\bm X}_k^{(2)\top} (\tilde \bSigma_{\setminus k} + \lambda \bm I)^{-1} \tilde{\bm X}_k^{(2)}}\biggr)\biggl(\frac{\tilde{\bm X}^{(1)\top} \tilde{\bm X}^{(1)}}{n}\biggr) \tilde \bbeta \\
    &:= \frac{\bm u^{(C,1)\top} \bm Z_k^{(2)} \bm Z_k^{(2)\top} \bm v^{(C,1)}}{1 + \bm Z_k^{(2)\top} \bm A^{(C,1)} \bm Z_k^{(2)}} + \frac{\bm u^{(C,2)\top} \bm Z_k^{(2)} (\bm Z_k^{(2)\top} \bm A^{(C,2)} \bm Z_k^{(2)}) \bm Z_k^{(2)\top} \bm v^{(C,2)}}{(1 + \bm Z_k^{(2)\top} \bm B^{(C,2)} \bm Z_k^{(2)})^2}
\end{align*}
where 
\begin{gather*}
    \bm u^{(C,1)} = -\frac{2}{n^3} (\bSigma^{(2)})^{1/2} (\tilde \bSigma_{\setminus k} + \lambda \bm I)^{-1} (\tilde{\bm X}^{(1)\top} \tilde{\bm X}^{(1)}) \tilde \bbeta\\
    \bm v^{(C,1)} =  (\bSigma^{(2)})^{1/2} (\hat \bSigma_{\setminus k} + \lambda \bm I)^{-1}\bSigma^{(2)} (\hat \bSigma_{\setminus k} + \lambda \bm I)^{-1} (\tilde{\bm X}^{(1)\top} \tilde{\bm X}^{(1)} )\tilde \bbeta \\
    \bm A^{(C,1)} = \frac
    1   n(\bSigma^{(2)})^{1/2} (\hat \bSigma_{\setminus k} + \lambda \bm I)^{-1} (\bSigma^{(2)})^{1/2} \\
    \bm u^{(C,2)} = \frac{1}{n^4} (\bSigma^{(2)})^{1/2} (\hat \bSigma_{\setminus k} + \lambda \bm I)^{-1} (\tilde{\bm X}^{(1)\top} \tilde{\bm X}^{(1)}) \tilde \bbeta \\
    \bm v^{(C,2)} = (\bSigma^{(2)})^{1/2} (\hat \bSigma_{\setminus k} + \lambda \bm I)^{-1} (\tilde{\bm X}^{(1)\top} \tilde{\bm X}^{(1)}) \bbeta^{(2)} \\
    \bm A^{(C,2)} = (\bSigma^{(2)})^{1/2} (\hat \bSigma_{\setminus k} + \lambda \bm I)^{-1} \bSigma^{(2)} (\hat \bSigma_{\setminus k} + \lambda \bm I)^{-1} (\bSigma^{(2)})^{1/2} \\
    \bm B^{(C,2)} = \frac
    1   n(\bSigma^{(2)})^{1/2} (\hat \bSigma_{\setminus k} + \lambda \bm I)^{-1} (\bSigma^{(2)})^{1/2}
\end{gather*}
\subsubsection{Common terms}
For fixed vectors $\bm u$, $\bm v$ and fixed PSD matrices $\bm A$, $\bm B$, $\bm C$, and $\bm D$, let $\bm S = \{\bm u, \bm v, \bm A, \bm B, \bm C, \bm D\}$ and consider the scalar-valued functions
\begin{align*}
    f_1(\bm w; \bm S) &= \bm u^\top \bm w \bm w^\top \bm v, \\ 
    f_2(\bm w; \bm S) &= \frac{\bm u^\top \bm w \bm w^\top \bm v}{1 + \bm w^\top \bm A \bm w}, \\
    f_3(\bm w; \bm S) &= \bm u^\top \bm w(\bm w^\top \bm A \bm w) \bm w^\top \bm v, \\
    f_4(\bm w; \bm S) &= \frac{\bm u^\top \bm w (\bm w^\top \bm A \bm w)  \bm w^\top \bm v}{1 + \bm w^\top \bm B \bm w} \\
    f_5(\bm w; \bm S) &= \frac{\bm u^\top \bm w (\bm w^\top \bm A \bm w) \bm w^\top \bm v}{(1 + \bm w^\top \bm B \bm w)^2}, \\
    f_6(\bm w; \bm S)&= \frac{(\bm u^\top \bm w)^2 (\bm w^\top \bm A \bm w) (\bm w^\top \bm B \bm w)}{1 + \bm w^\top \bm C \bm w} \\
    f_7(\bm w; \bm S) &= \frac{\bm u^\top \bm w (\bm w^\top \bm A \bm w) (\bm w^\top \bm B \bm w) \bm w^\top \bm v}{(1 + \bm w^\top \bm C \bm w)^2}, \\
    f_8(\bm w; \bm S) &= \frac{(\bm u^\top \bm w)^2 (\bm w^\top \bm A \bm w) (\bm w^\top \bm B \bm w)(\bm w^\top \bm C \bm w)}{(1 + \bm w^\top \bm D \bm w)^2}.
\end{align*}

From the above expansions, we see that, for each $j = 1,2,3$:
\begin{itemize}
    \item There are some $m_j^{(1)} \in \mathbb Z_{> 0}$, some $i_{j,1}^{(1)}, \dots, i_{j,m_j^{(1)}}^{(1)} \in \mathbb Z_{> 0}$, and some collections of vectors and matrices $\bm S_{j,1}^{(1)}, \dots, \bm S_{j,m_j^{(1)}}^{(1)}$ that are independent of $\bm Z_k^{(1)}, \tilde{\bm Z}_k^{(1)}$ such that
    \begin{align*}
        B_j(\bm Z^{(1),k}, \bm Z^{(2)}) - B_j(\bm Z^{(1),\setminus k}, \bm Z^{(2)}) &= \sum_{\ell=1}^{m_j^{(1)}} f_{i_{j,\ell}^{(1)}}(\tilde{\bm Z}^{(1)}_k; \bm S_{j,\ell}^{(1)}) \\
        B_j(\bm Z^{(1),k-1}, \bm Z^{(2)}) - B_j(\bm Z^{(1),\setminus k}, \bm Z^{(2)}) &= \sum_{\ell=1}^{m_j^{(1)}} f_{i_{j,\ell}^{(1)}}(\bm Z^{(1)}_k; \bm S_{j,\ell}^{(1)})
    \end{align*}
    \item There are some $m_j^{(2)} \in \mathbb Z_{> 0}$, some $i_{j,1}^{(2)}, \dots, i_{j,m_j^{(2)}}^{(2)} \in \mathbb Z_{> 0}$, and some collections of vectors and matrices $\bm S_{j,1}^{(2)}, \dots, \bm S_{j,m_j^{(2)}}^{(2)}$ that are independent of $\bm Z_k^{(2)}, \tilde{\bm Z}_k^{(2)}$ such that
    \begin{align*}
        B_j(\tilde{\bm Z}^{(1)}, \bm Z^{(2),k}) - B_j(\tilde{\bm Z}^{(1)}, \bm Z^{(2),\setminus k}) &= \sum_{\ell=1}^{m_j^{(2)}} f_{i_{j,\ell}^{(2)}}(\tilde{\bm Z}^{(2)}_k; \bm S_{j,\ell}^{(2)}) \\
        B_j(\tilde{\bm Z}^{(1)}, \bm Z^{(2),k-1}) - B_j(\tilde{\bm Z}^{(1)}, \bm Z^{(2),\setminus k}) &= \sum_{\ell=1}^{m_j^{(2)}} f_{i_{j,\ell}^{(2)}}(\bm Z^{(2)}_k; \bm S_{j,\ell}^{(2)})
    \end{align*}
\end{itemize}
Therefore, by Cauchy-Schwarz, for each $j = 1,2,3$:
\begin{itemize}
    \item We can bound
    \begin{align*}
        \E_{1,k}[(B_j(\bm Z^{(1),k}, \bm Z^{(2)}) - B_j(\bm Z^{(1),\setminus k}, \bm Z^{(2)}))^2] &\leq m_j^{(1)} \sum_{\ell=1}^{m_j^{(1)}} \E_{1,k}[f_{i_{j,\ell}^{(1)}}(\tilde{\bm Z}^{(1)}_k; \bm S_{j,\ell}^{(1)})^2] \\
        \E_{1,k}[(B_j(\bm Z^{(1),k-1}, \bm Z^{(2)}) - B_j(\bm Z^{(1),\setminus k}, \bm Z^{(2)}))^2] &\leq m_j^{(1)} \sum_{\ell=1}^{m_j^{(1)}} \E_{1,k}[f_{i_{j,\ell}^{(1)}}(\bm Z^{(1)}_k; \bm S_{j,\ell}^{(1)})^2] \\
        \E_{2,k}[(B_j(\tilde{\bm Z}^{(1)}, \bm Z^{(2),k}) - B_j(\tilde{\bm Z}^{(1)}, \bm Z^{(2),\setminus k}))^2] &\leq m_j^{(2)} \sum_{\ell=1}^{m_j^{(2)}} \E_{2,k}[f_{i_{j,\ell}^{(2)}}(\tilde{\bm Z}^{(2)}_k; \bm S_{j,\ell}^{(2)})^2] \\
        \E_{2,k}[(B_j(\tilde{\bm Z}^{(1)}, \bm Z^{(2),k-1}) - B_j(\tilde{\bm Z}^{(1)}, \bm Z^{(2),\setminus k}))^2] &\leq m_j^{(2)} \sum_{\ell=1}^{m_j^{(2)}} f_{i_{j,\ell}^{(2)}}\E_{2,k}[(\bm Z^{(2)}_k; \bm S_{j,\ell}^{(2)})^2]
    \end{align*}
    \item We can bound
    \begin{align*}
        &|\E_{1,k}[B_j(\bm Z^{(1),k}, \bm Z^{(2)}) - B_j(\bm Z^{(1), k-1}, \bm Z^{(2)})]| \\&\leq  \sum_{\ell=1}^{m_j^{(1)}} |\E_{1,k}[f_{i_{j,\ell}^{(1)}}(\tilde{\bm Z}^{(1)}_k; \bm S_{j,\ell}^{(1)}) - f_{i_{j,\ell}^{(1)}}(\bm Z^{(1)}_k; \bm S_{j,\ell}^{(1)})]| \\
        &|\E_{2,k}[B_j(\tilde{\bm Z}^{(1)}, \bm Z^{(2),k}) - B_j(\tilde{\bm Z}^{(1)}, \bm Z^{(2), k-1})]| \\&\leq \sum_{\ell=1}^{m_j^{(2)}} |\E_{2,k}[f_{i_{j,\ell}^{(2)}}(\tilde{\bm Z}^{(2)}_k; \bm S_{j,\ell}^{(2)}) - f_{i_{j,\ell}^{(2)}}(\bm Z^{(2)}_k; \bm S_{j,\ell}^{(2)})]|
    \end{align*}
\end{itemize}
Therefore, proving our universality result reduces to bounding terms of the form 
\begin{gather*}
    \E_{i,k}[f_\ell(\tilde{\bm Z}_k^{(i)}; \bm S)^2],  \quad \E_{i,k}[f_\ell({\bm Z}_k^{(i)}; \bm S)^2], \quad |\E_{i,k}[f_\ell(\tilde{\bm Z}_k^{(i)}; \bm S) - f_\ell(\bm Z_k^{(i)}; \bm S)]|
\end{gather*}
for $\ell = 1,\dots, 8$, $i = 1,2$, and any $\bm S$ (with the bounds depending on $\bm S$).

\subsection{Second moment bounds}
In this section, we bound the second moment terms $\E_{i,k}[f_\ell(\tilde{\bm Z}_k^{(i)}; \bm S)^2]$ and $\E_{i,k}[f_\ell({\bm Z}_k^{(i)}; \bm S)^2]$ for $\ell = 1, \dots, 8$, $i = 1,2$, and any $\bm S$. 

We repeatedly apply H\"older's inequality to yield the desired bounds. In bounding many of the terms, we also make use of the fact that $\bm A, \bm B, \bm C, \bm D$ are PSD by assumption and thus \begin{equation*}
    0 < \frac{1}{1 + \bm w^\top \bm A \bm w}, \frac{1}{1 + \bm w^\top \bm B \bm w}, \frac{1}{1 + \bm w^\top \bm C \bm w}, \frac{1}{1 + \bm w^\top \bm D \bm w} \leq 1
\end{equation*} for all $\bm w \in \R^p$.
\begin{lemma}\label{lm:second_moment_bounds}
Let $\bm W$ be a random $\R^p$-valued random variable consisting of i.i.d. entries such that $\E[W_1] = 0$, $\Var(W_1) = 1$, and $\E[W_1^m] < \infty$ for all $m \geq 1$. Then, we can bound
\begin{equation*}
\begin{alignedat}{2}
    \E[f_1(\bm W; \bm S)^2] &\lcon \|\bm u\|_2^2 \|\bm v\|_2^2, &\quad \E[f_2(\bm W; \bm S)^2] &\lcon \|\bm u\|_2^2 \|\bm v\|_2^2 \\
    \E[f_3(\bm W; \bm S)^2] &\lcon \|\bm u\|_2^2 \|\bm v\|_2^2 \|\bm A\|_{\text{op}}^2 p^2, &\quad \E[f_4(\bm W; \bm S)^2] &\lcon \|\bm u\|_2^2 \|\bm v\|_2^2 \|\bm A\|_{\text{op}}^2  p^2 \\
    \E[f_5(\bm W; \bm S)^2] &\lcon \|\bm u\|_2^2 \|\bm v\|_2^2 \|\bm A\|_{\text{op}}^2  p^2, &\quad \E[f_6(\bm W; \bm S)^2]&\lcon\|\bm u\|_2^4 \|\bm A\|_{\text{op}}^2 \|\bm B\|_{\text{op}}^2 p^4 \\
    \E[f_7(\bm W; \bm S)^2] &\lcon \|\bm u\|_2^2 \|\bm v\|_2^2 \|\bm A\|_{\text{op}}^2  \|\bm B\|_{\text{op}}^2   p^4, &\quad \E[f_8(\bm W; \bm S)^2] &\lcon \|\bm u\|_2^4 \|\bm A\|_{\text{op}}^2 \|\bm B\|_{\text{op}}^2 \|\bm C\|_{\text{op}}^2 p^6.
\end{alignedat}
\end{equation*}
\end{lemma}
\begin{proof}
    We can first bound the products
    \begin{align*}
        \E[&(\bm u^\top \bm W \bm W^\top \bm v)^2] 
        \\&\qquad \leq \E[|\bm u^\top \bm W|^4]^{1/2} \E[|\bm W^\top \bm v|^4]^{1/2} \\
        &\qquad \lcon \|\bm u\|_2^2 \|\bm v\|_2^2 \\
        \E[&(\bm u^\top \bm W(\bm W^\top \bm A \bm W)\bm W^\top \bm v)^2]  \\
        &\qquad \leq \E[|\bm u^\top \bm W|^8]^{1/4} \E[(\bm W^\top \bm A \bm W)^4]^{1/2} \E[|\bm W^\top \bm v|^8]^{1/4} \\
        &\qquad \lcon \|\bm u\|_2^2 \|\bm A\|_{\text{op}}^2 p^2 \|\bm v\|_2^2 \\
        \E[&(\bm u^\top \bm W(\bm W^\top \bm A \bm W)(\bm W^\top \bm B \bm W)\bm W^\top \bm v)^2]\\
        &\qquad\leq \E[|\bm u^\top \bm W|^8]^{1/4} \E[(\bm W^\top \bm A \bm W)^8]^{1/4} \E[(\bm W^\top \bm B \bm W)^8]^{1/4} \E[|\bm W^\top \bm v|^8]^{1/4} \\
        &\qquad\lcon \|\bm u\|_2^2 \|\bm A\|_{\text{op}}^2 p^2 \|\bm B\|_{\text{op}}^2 p^2 \|\bm v\|_2^2 \\
        \E[&((\bm u^\top \bm W)^2(\bm W^\top \bm A \bm W)(\bm W^\top \bm B \bm W)(\bm W^\top \bm C \bm W))^2]  \\
        &\qquad\leq \E[|\bm u^\top \bm W|^{16}]^{1/4} \E[(\bm W^\top \bm A \bm W)^8]^{1/4} \E[(\bm W^\top \bm B \bm W)^8]^{1/4} \E[(\bm W^\top \bm C \bm W)^8]^{1/4} \\
        &\qquad\lcon \|\bm u\|_2^4 \|\bm A\|_{\text{op}}^2 p^2 \|\bm B\|_{\text{op}}^2 p^2 \|\bm C\|_{\text{op}}^2 p^2 
    \end{align*}
    and the bounds then follow immediately.
\end{proof}
Note that the bounds in this lemma directly translate to bounds on $\E_{i,k}[f_\ell(\tilde{\bm Z}_k^{(i)}; \bm S)^2]$ and $\E_{i,k}[f_\ell(\bm Z_k^{(i)}; \bm S)^2]$ for $\ell = 1, \dots, 8$ and $i = 1,2$ and any $\bm S$ because we only make use of (a) matching first and second moments and (b) the mere existence of all other moments.

\subsection{First moment bounds}
In this section, we bound $|\E_{i,k}[f_\ell(\tilde{\bm Z}^{(i)}_k; \bm S) - f_\ell(\bm Z^{(i)}_k; \bm S))]|$ for $\ell = 1, \dots, 8$, $i = 1,2$, and any $\bm S$. As long as our proof only makes use of (a) matching first and second moments between $\tilde{\bm W}$ and $\bm W$, (b) the mere existence of all other moments of $\bm W$, and (c) the assumption $\tilde{\bm W} \sim \mathcal{N}(0, \bm I)$, it suffices to consider bounding
\begin{equation*}
    |\E_k[f_\ell(\tilde{\bm W}; \bm S) - f_\ell(\bm W; \bm S)]|
\end{equation*}
for $\ell = 1, \dots, 8$ and any fixed $\bm S$, where $\tilde{\bm W} \sim \mathcal{N}(0, \bm I)$ and where $\bm W$ is a random $\R^p$-valued random variables consisting of i.i.d. entries with $\E[W_1] = \E[\tilde W_1] = 0$ and $\Var(W_1) = \Var(\tilde{W}_1) = 1$. Importantly, we do \emph{not} assume that $W_1 \overset{d}{=} \tilde{W}_1$ or that $\E[W_1^m] = \E[\tilde W_1^m]$ for any $m > 2$. To further simplify our notation, we omit $\bm S$ when the dependence is clear. Therefore, the goal of this section is to bound
\begin{equation*}
    |\E_k[f_\ell(\tilde{\bm W}) - f_\ell(\bm W)]|
\end{equation*}
in terms of $\bm S$, for each $\ell = 1, \dots, 8$.

At a high level, our strategy for bounding these terms involves interpolating between $\tilde{\bm W}$ and $\bm W$ and bounding the change in the value of our expression along the curve. To bound the value of the expression, we perform multiple Taylor expansions and invoke Stein's identity to simplify the resulting expansions.

The following statement is the basis of our interpolation strategy, which we use bound the terms resulting from the Sherman-Morrison decompositions:
\begin{lemma}\label{lm:interp_edge}
    Let $f$ be a function that, for all multi-indices $\alpha$ with $|\alpha| \leq 3$, we have $|\partial^\alpha f(\bm w)| \leq C (1 + \|\bm w\|^m)$ for some $m > 0$. Let $\tilde{\bm W} \sim \mathcal{N}(0, \bm I_p)$ and let $\bm W$ consist of independent random entries with $\E[W_1] = 0$, $\Var(W_1) = 1$, and $\E[|W_1|^3] < \infty$. Assume $\bm W$ and $\tilde{\bm W}$ are independent. For $t \in (0,1)$, let
    \begin{align*}
        \bm W_{t} &= \sqrt{t} \tilde{\bm W} + \sqrt{1 - t} \bm W
    \end{align*}
    and for each $t,u \in (0,1)$ and $i = 1, \dots, p$, let
    \begin{equation*}
        \bm W_{t,u,i} = \bm W_t - (1 - u)\sqrt{1 - t} W_{i} e_i.
    \end{equation*}
    Then, we have
    \begin{align*}
        &|\E[f(\tilde{\bm W})] - \E[f(\bm W)]| \\
        &\leq \sup_{t \in (0,1)} \E\biggl[\sum_{i=1}^p(|W_i|  +  |W_i|^3 )\cdot \sup_{u \in (0,1)}| \partial_i^3 f(\bm W_{t,u,i})|\biggr]
    \end{align*}
\end{lemma}
\begin{proof}
Define $F(t) = \E[f(\bm W_{t})]$ so that bounding the above is equivalent to bounding
\begin{align*}
    |\E&[f(\bm W_{1})] - \E[f(\bm W_{0})]|\\
    &= |F(1) - F(0)| = \biggl|\int_0^1 F'(t) dt\biggr| \\
    &= \frac{1}{2}\biggl|\int_0^1 \biggl( \sum_{i=1}^p \frac{1}{\sqrt{t}} \E[\tilde W_i \partial_i f(\bm W_{t})] - \frac{1}{\sqrt{1-t}} \E[W_i \partial_i f(\bm W_{t})] \biggr)dt \biggr|  \\
    &= \frac{1}{2} \biggl| \int_0^1 \frac{1}{\sqrt{1-t}} \biggl(\sum_{i=1}^p \sqrt{1-t} \E[\partial_i^2 f(\bm W_{t})] - \E[W_i \partial_i f(\bm W_{t})]\biggr) dt \biggr|  \\
    &\leq \sup_{t \in (0,1)} \biggl|\sum_{i=1}^p \sqrt{1-t} \E[\partial_i^2 f(\bm W_{t})] - \E[W_i \partial_i f(\bm W_{t})] \tag{\theequation}\label{eq:lind_interp_B1}\biggr|
\end{align*}
where the third line follows from Leibniz's rule and the fourth line follows from applying Stein's identity to yield
\begin{equation*}
    \E[\tilde{W}_i \partial_i f(\bm W_{t})] = \sqrt{t} \E[\partial_i^2 f(\bm W_t)].
\end{equation*}
To analyze the latter term, we first express
\begin{equation*}
    \bm W_t = \sqrt{t} \tilde{\bm W} + \sqrt{1-t} (\bm W - W_i e_i) + \sqrt{1-t} W_i e_i := \bm W_{t,-i} + \sqrt{1-t} W_i e_i.
\end{equation*}
Then, letting $h_{i}: \R \to \R$ be defined 
\begin{equation*}
    h_{i}(w) = \partial_i f(\bm W_{t,-i} + \sqrt{1 - t} we_i)
\end{equation*}
we can apply Taylor's theorem to yield
\begin{align*}
    h_i(w) &= h_i(0) + h_i'(0)w +\int_0^w (w-s) h_i''(s) ds \\
    &= h_i(0) + h_i'(0)w + w^2\int_0^1 (1-u)h_i''(wu)du \\
    h_i'(w) &= h_i'(0) + \int_0^w h_i''(s) ds \\
    &= h_i'(0) + w \int_0^1 h_i''(wu) du
\end{align*}
Using these expansions, we have
\begin{align*}
    &\sqrt{1-t} \E[\partial_i^2 f(\bm W_t)] - \E[W_i \partial_i f(\bm W_t)] \\
    &= \E_k[h_i'(W_i)] - \E[W_i h_i(W_i)] \\
    &= h_i'(0) + \E\biggl[W_i \int_0^1 h''(W_i u) du \biggr] - h_i(0) \E[W_i] - h_i'(0) \E[W_i^2] \\
    &\qquad - \E\biggl[W_i^3 \int_0^1 (1-u) h_i''(W_iu)du \biggr] \\
    &= \E\biggl[ \int_0^1 (W_i - (1-u)W_i^3) h_i''(W_i u) du \biggr]
\end{align*}
Now, defining $\bm W_{t,u,i} = \bm W_t - (1 - u)\sqrt{1 - t} W_i e_i$, note that
\begin{equation*}
    h_i''(W_i u) = (1 - t) \partial_i^3 f(\bm W_{t,u,i})
\end{equation*}
The result then follows.
\end{proof}

Therefore, with this lemma, it suffices to bound
\begin{equation*}
     \E\biggl[\sum_{i=1}^p |W_i|^r \cdot \sup_{u \in (0,1)} |\partial_i^3 f_\ell(\bm W_{t,u,i})| \biggr]
\end{equation*}
for $r = 1,3$, $\ell = 1, \dots, 8$, and $t \in (0,1)$. We now embark on these bounds.

In bounding terms of the first kind, we prove analogous bounds to Lemma \ref{lm:basic_S_moments} for $\bm W_t$ and $\bm W_{t,u,i}$:
\begin{lemma}\label{lm:t_S_moments}
    Let $\bm W, \tilde{\bm W}$ be random $\R^p$-valued random variables, each consisting of i.i.d. entries such that $\E[W_1] = \E[\tilde W_1] = 0$, $\Var(W_1) = \Var(\tilde W_1)= 1$, and $\E[W_1^{2m}], \E[\tilde W_1^{2m}] < \infty$. For $t \in (0,1)$, let
    \begin{equation*}
        \bm W_t = \sqrt{t} \tilde{\bm W} + \sqrt{1-t} \bm W
    \end{equation*}
    and for $t, u \in (0,1)$ and $i = 1, \dots, p$, let 
    \begin{equation*}
        \bm W_{t,u,i} = \bm W_t - (1 - u)\sqrt{1 - t} W_i e_i
    \end{equation*}
    Then, for any $i = 1, \dots, p$, $t \in (0,1)$, a vector $\bm u$, positive semidefinite matrix $\bm A$, and a positive integer $m \geq 2$, we have
    \begin{gather*}
        \E[ \|\bm W_{t}\|_2^m] \lcon p^{m/2}, \quad \E[ |\bm u^\top \bm W_{t}|^m] \lcon \|\bm u\|_2^m, \quad  \E[  (\bm W_{t}^\top \bm A \bm W_{t})^m] \lcon \|\bm A\|_{\text{op}}^m p^m \\
        \E\biggl[\sup_{u \in (0,1)} \|\bm W_{t,u,i}\|_2^m\biggr] \lcon p^{m/2}, \quad \E\biggl[ \sup_{u \in (0,1)} |\bm u^\top \bm W_{t,u,i}|^m\biggr] \lcon \|\bm u\|_2^m   \\
        \E\biggl[ \sup_{u \in (0,1)} (\bm W_{t,u,i}^\top \bm A \bm W_{t,u,i})^m\biggr] \lcon \|\bm A\|_{\text{op}}^m p^m, \quad \sum_{i=1}^p \E\biggl[\sup_{u \in (0,1)} |(\bm A \bm W_{t,u,i})_i|^m \biggr] \lcon \|\bm A\|_{\text{op}}^m p^{m/2}
    \end{gather*}
    where $\lcon$ is taken with respect to $m$ and the moments of $W_1$, $\tilde{W}_1$.
\end{lemma}
\begin{proof}
    We can first write
    \begin{equation*}
        \|\bm W_t\|_2^m \le (\|\tilde {\bm W}\|_2 + \|\bm W\|_2)^m \lcon \|\tilde{\bm W}\|_2^m + \|\bm W\|_2^m
    \end{equation*}
    where the last inequality follows by the inequality $(a + b)^m \leq 2^{m-1} (a^m + b^m)$. Therefore, we have 
    \begin{equation*}    
        \E[\|\bm W_t\|_2^m] \leq 2\max\{\E[\|\tilde{\bm W}\|_2^m], \E[\|\bm W\|_2^m]\} \lcon p^{m/2}
    \end{equation*}
    by Lemma \ref{lm:basic_S_moments}. Similarly,
    \begin{equation*}
        |\bm u^\top \bm W|^m \leq (|\bm u^\top \tilde{\bm W}| + |\bm u^\top \bm W|)^m \lcon |\bm u^\top \tilde{\bm W}|^m + |\bm u^\top \bm W|^m
    \end{equation*}
    and thus
    \begin{equation*}
        \E[|\bm u^\top \bm W_t|^m] \leq 2 \max\{\E[|\bm u^\top \tilde{\bm W}|^m], \E[|\bm u^\top \bm W|^m]\} \lcon \|\bm u\|_2^m.
    \end{equation*}
    Then, we have
    \begin{equation*}
        \E[(\bm W_t^\top \bm A \bm W_t)^m] \leq \|\bm A\|_{\text{op}}^m\E[\|\bm W_t\|_2^{2m}] \lcon \|\bm A\|_{\text{op}}^m p^m
    \end{equation*}
    by applying the first bound in this lemma.

    Now, we prove the second set of inequalities in this lemma. Note that
    \begin{equation*}
        \|\bm W_{t,u,i}\|_2^m \leq (\|\bm W_t\|_2 + |W_i|)^m \leq (2 \|\bm W_t\|_2)^m \lcon \|\bm W_t\|_2^m
    \end{equation*}
    and thus
    \begin{equation*}
        \E\biggl[\sup_{u \in (0,1)} \|\bm W_{t,u,i}\|_2^m\biggr] \lcon\E[\|\bm W_t\|_2^m] \lcon p^{m/2}.
    \end{equation*}
    Next, we can bound
    \begin{equation*}
        |\bm u^\top \bm W_{t,u,i}|^m \leq (|\bm u^\top \bm W_t| + |W_i u_i|)^m \lcon |\bm u^\top \bm W_t|^m + \|\bm u\|_2^m |W_i|^m
    \end{equation*}
    and thus
    \begin{equation*}
        \E\biggl[ \sup_{u \in (0,1)} |\bm u^\top \bm W_{t,u,i}|^m\biggr] \lcon \|\bm u\|_2^m.
    \end{equation*}
    Then, we can bound
    \begin{align*}
        (\bm W_{t,u,i}^\top \bm A \bm W_{t,u,i})^m &= (\bm W_t^\top \bm A \bm W_t - 2(1-u)\sqrt{1-t} W_i e_i^\top \bm A \bm W_t + (1-u)^2 (1-t) W_i^2 A_{ii})^m \\
        &\lcon (\bm W_t^\top \bm A \bm W_t)^m + |W_i|^m |e_i^\top \bm A \bm W_t|^m + W_i^{2m} A_{ii}^m
    \end{align*}
    where
    \begin{equation*}
        \E[W_i^m (e_i^\top \bm A \bm W_t)^m] \leq \|\bm A\|_{\text{op}}^m \E[|W_i|^m \|\bm W_t\|_2^m] \leq \|\bm A\|_{\text{op}}^m \E[|W_i|^{2m}]^{1/2} \E[\|\bm W\|_2^{2m}]^{1/2}
    \end{equation*}
    and 
    \begin{equation*}
        \E[W_i^{2m} A_{ii}^m] \leq \|\bm A\|_{\text{op}}^m \E[W_i^{2m}].
    \end{equation*}
    It thus follows from the earlier bounds that
    \begin{equation*}
        \E\biggl[ \sup_{u \in (0,1)} (\bm W_{t,u,i}^\top \bm A \bm W_{t,u,i})^m\biggr] \lcon \|\bm A\|_{\text{op}} p^m
    \end{equation*}
    Finally, note that
    \begin{align*}
        &(\bm A \bm W_{t,u,i})_i^m = (\bm A_i^\top \bm W_{t,u,i})^m = (\bm A_i^\top \bm W_t - (1- u)\sqrt{1-t} W_i \bm A_{ii})^m \\
        &\lcon  (\bm A_i^\top \bm W_t)^m + (W_i \bm A_{ii})^m \\
        &= (\bm A \bm W_t)^m_i + |W_i|^m \bm A_{ii}^m
    \end{align*}
    which implies
    \begin{align*}
        \sum_{i=1}^p \E\biggl[ \sup_{u \in (0,1)} |\bm A \bm W_{t,u,i}|_i^m \biggr]  &\lcon \sum_{i=1}^p \biggl(\E[(\bm A \bm W_t)_i^m] + (\bm A_{ii})^m\E[|W_i|^m]\biggr)  \\
        &= \E[\|\bm A \bm W_t\|_m^m] + \sum_{i=1}^p(\bm A_{ii})^m\E[|W_i|^m]  \\
        &\leq \E[\|\bm A \bm W_t\|_2^m] + p \|\bm A\|_{\text{op}}^m \E[|W_1|^m] \\
        &\leq \|\bm A\|_{\text{op}}^m (\E[\|\bm W_t\|_2^m] + p \E[|W_1|^m] ) \\
        &\lcon \|\bm A\|_{\text{op}}^m (p^{m/2} + p\E[|W_1|^m] ) \\
        &\lcon \|\bm A\|_{\text{op}}^m p^{m/2}
    \end{align*}
\end{proof}

We now begin bounding each $f_1, \dots, f_8$. We prove the following bounds:
\begin{lemma}\label{lm:first_moment_bounds}
\begin{align*}
    |\E[f_1(\tilde{\bm W}) - f_1(\bm W)]| &= 0 \\
    |\E[f_2(\tilde{\bm W}) - f_2(\bm W)]| &\lcon \|\bm u\|_2 \|\bm v\|_2 ( \|\bm A\|_{\text{op}} p^{1/2} + \|\bm A\|_{\text{op}}^2 p + \|\bm A\|_{\text{op}}^3 p^2) \\
    |\E[f_3(\tilde{\bm W}) - f_3(\bm W)]| &\lcon \|\bm u\|_2 \|\bm v\|_2  \|\bm A\|_{\text{op}} p^{1/2} \\
    |\E[f_4(\tilde{\bm W}) - f_4(\bm W)]| &\lcon \|\bm u\|_2 \|\bm v\|_2 \|\bm A\|_{\text{op}}( p^{1/2} + \|\bm B\|_{\text{op}} p^{3/2} + \|\bm B\|_{\text{op}}^2 p^2 + \|\bm B\|_{\text{op}}^3 p^3) \\
    |\E[f_5(\tilde{\bm W}) - f_5(\bm W)]| &\lcon \|\bm u\|_2 \|\bm v\|_2 \|\bm A\|_{\text{op}}(p^{1/2} + \|\bm B\|_{\text{op}}^2 p^{5/2} + \|\bm B\|_{\text{op}}^4 p^4 + \|\bm B\|_{\text{op}}^6 p^6) \\
    |\E[f_6(\tilde{\bm W}) - f_6(\bm W)]| &\lcon \|\bm u\|_2^2\|\bm A\|_{\text{op}} \|\bm B\|_{\text{op}} ( p^{3/2} + \|\bm C\|_{\text{op}} p^{5/2} + \|\bm C\|_{\text{op}}^2 p^3 + \|\bm C\|_{\text{op}}^3 p^{7/2}) \\
    |\E[f_7(\tilde{\bm W}) - f_7(\bm W)]| &\lcon \|\bm u\|_2 \|\bm v\|_2 \|\bm A\|_{\text{op}} \|\bm B\|_{\text{op}} ( p^{3/2} +  \|\bm C\|_{\text{op}}^2 p^{7/2} +  \|\bm C\|_{\text{op}}^4 p^5 + \|\bm C\|_{\text{op}}^6 p^7) \\
    |\E[f_8(\tilde{\bm W}) - f_8(\bm W)]| &\lcon \|\bm u\|_2^2 \|\bm A\|_{\text{op}} \|\bm B\|_{\text{op}} \|\bm C\|_{\text{op}} \\
    &\qquad \cdot (p^{5/2} + \|\bm D\|_{\text{op}}^2 p^{9/2} + \|\bm D\|_{\text{op}}^4 p^{6} + \|\bm D\|_{\text{op}}^6 p^{8})
\end{align*}
\end{lemma}

Note that the first function $f_1$ can be handled directly because its expectation depends only on the first and second moments of $\bm W$ and $\tilde{\bm W}$. That is,
\begin{align*}
    \E[f_1(\tilde{\bm W}) - f_1(\bm W)] 
    = \bm u^\top (\E[\tilde{\bm W} \tilde{\bm W}^\top] - \E[\bm W\bm W^\top]) \bm v = 0
\end{align*}
because $\E[\tilde{\bm W} \tilde{\bm W}^\top] = \E[\bm W \bm W^\top] = \bm I$. The rest of this section is dedicated to proving the remaining bounds.

\subsection{Third derivative bounds}
To use Lemma \ref{lm:t_S_moments}, we must obtain bounds on the third derivative of each term $f_i(\bm W)$. Here, we obtain pointwise bounds and will later further bound these terms in expectation.

\subsubsection{First bound for \texorpdfstring{$f_2$}{f2}}
Note that we can express
\begin{equation*}
    f_2(\bm w) = \frac{g_{\bm u}(\bm w) g_{\bm v}(\bm w)}{g_{1, \bm A} (\bm w)}
\end{equation*}
Noting that $g_{1, \bm A} \neq 0$, we can write $f_2(\bm w) g_{1, \bm A}(\bm w) = g_{\bm u}(\bm w) g_{\bm v}(\bm w)$ and differentiating each side yields
\begin{equation*}
    f_2'g_{1, \bm A} + f_2g_{1, \bm A}' = g_{\bm u}' g_{\bm v} + g_{\bm u} g_{\bm v}'
\end{equation*}
where $f_2' = \partial_i f_2$, etc. This yields, as usual, 
\begin{equation*} 
f_2' = \frac{g_{\bm u}' g_{\bm v} + g_{\bm u} g_{\bm v}' - f_2g_{1, \bm A}'}{g_{1, \bm A}} = \frac{u_i g_{\bm v} + g_{\bm u} v_i - 2f_2\cdot (\bm A \bm w)_i}{g_{1, \bm A}}
\end{equation*} 
Differentiating each side again yields
\begin{equation*}
    f_2'' g_{1, \bm A} + 2f_2' g_{1, \bm A}'+f_2 g_{1, \bm A}'' = g_{\bm u}'' g_{\bm v} + 2 g_{\bm u}' g_{\bm v}' + g_{\bm u} g_{\bm v}''
\end{equation*}
and rearranging yields
\begin{align*}
    f_2'' &= \frac{g_{\bm u}''g_{\bm v} + 2g_{\bm u}'g_{\bm v}' + g_{\bm u}g_{\bm v}'' - 2f_2'g_{1, \bm A}' - f_2g_{1, \bm A}''}{g_{1, \bm A}} = \frac{2u_i v_i - 4f_2'\cdot (\bm A \bm w)_i - 2f_2\cdot\bm A_{ii}}{g_{1, \bm A}} 
\end{align*}
Again, we can differentiate each side again to yield
\begin{equation*}
    f_2''' g_{1, \bm A} + 3f_2'' g_{1, \bm A}' + 3f_2'g_{1, \bm A}'' + f_2g_{1, \bm A}''' = g_{\bm u}'''  g_{\bm v} + 3g_{\bm u}''g_{\bm v}' + 3g_{\bm u}'g_{\bm v}'' + g_{\bm u}g_{\bm v}'''
\end{equation*}
which yields
\begin{align*}
    f_2''' &= \frac{g_{\bm u}'''  g_{\bm v} + 3g_{\bm u}''g_{\bm v}' + 3g_{\bm u}'g_{\bm v}'' + g_{\bm u}g_{\bm v}''' - 3f_2'' g_{1, \bm A}' - 3f_2'g_{1, \bm A}'' - f_2g_{1, \bm A}'''}{g_{1, \bm A}} \\
    &=\frac{ - 6f_2'' \cdot (\bm A \bm w)_i - 6f_2'\cdot \bm A_{ii}}{g_{1, \bm A}} 
\end{align*}
For each $i$, noting that $g_{1, \bm A} \geq 1$, we have
\begin{equation*}
    |f_2'| \lcon |u_i g_{\bm v}| + |v_i g_{\bm u}| + |f_2 \cdot (\bm A \bm w)_i|
\end{equation*}
and
\begin{align*}
    |f_2''| &\lcon |u_i v_i| + |f_2' \cdot (\bm A \bm w)_i| + |f_2 \cdot \bm A_{ii}| \\
    &\lcon |u_i v_i| + (|u_i g_{\bm v} (\bm A \bm w)_i| + |v_i g_{\bm u}(\bm A \bm w)_i| + |f_2 \cdot (\bm A \bm w)_i^2|) + |f_2 \cdot \bm A_{ii}| \\
\end{align*}
and finally
\begin{align*}
    \biggl|\partial_i^3 f_2(\bm w) \biggr| &\lcon |f_2'' \cdot (\bm A \bm w)_i| + |f_2' \cdot \bm A_{ii}| \\
    &\lcon [|u_i v_i (\bm A \bm w)_i| + (|u_i g_{\bm v} (\bm A \bm w)_i^2| + |v_i g_{\bm u}(\bm A \bm w)_i^2| + |f_2 \cdot (\bm A \bm w)_i^3|) \\
    &\hspace{4em}+ |f_2 \cdot \bm A_{ii} (\bm A \bm w)_i|] \\
    &\qquad + (|u_i g_{\bm v} \bm A_{ii}| + |v_i g_{\bm u} \bm A_{ii}| + |f_2 \cdot (\bm A \bm w)_i \bm A_{ii}|) \\
    &\lcon [|u_i v_i (\bm A \bm w)_i| + (|g_{\bm v} \cdot u_i  (\bm A \bm w)_i^2| + |g_{\bm u} \cdot v_i (\bm A \bm w)_i^2| + |g_{\bm u} g_{\bm v} \cdot (\bm A \bm w)_i^3|) \\
    &\hspace{4em}+ |\bm A_{ii} g_{\bm u} g_{\bm v} \cdot  (\bm A \bm w)_i|] \\
    &\qquad + (|\bm A_{ii} g_{\bm v} \cdot u_i| + |\bm A_{ii} g_{\bm u} \cdot  v_i  | + |\bm A_{ii} g_{\bm u} g_{\bm v} \cdot (\bm A \bm w)_i |)
\end{align*} 

\subsubsection{First bound for \texorpdfstring{$f_3$}{f3}}
We can express
\begin{equation*}
    f_3(\bm w; \bm S) = g_{\bm u}(\bm w) g_{\bm v}(\bm w) g_{\bm A}(\bm w)
\end{equation*}
where $g_{\bm u}(\bm w) = \bm u^\top \bm w$,  $g_{\bm v}(\bm w) = \bm v^\top \bm w$, and $g_{\bm A}(\bm w) = \bm w^\top \bm A \bm w$. By Lemma \ref{lm:t_S_moments}, we see that, for $m \geq 2$,
\begin{gather*}
    \sup_i \E\biggl[ \sup_{u \in (0,1)} g_{\bm u}^m\biggr] \lcon \|\bm u\|_2^m, \quad  \sup_i \E\biggl[ \sup_{u \in (0,1)} g_{\bm v}^m\biggr] \lcon \|\bm v\|_2^m
\end{gather*}
By the product rule,
\begin{align*}
    f_3''' &=g_{\bm u}''' g_{\bm v} g_{\bm A} + g_{\bm u} g_{\bm v}''' g_{\bm A} + g_{\bm u} g_{\bm v} g_{\bm A}''' \\
    &\qquad  + 3 (g_{\bm u}'' g_{\bm v}' g_{\bm A} + g_{\bm u}'' g_{\bm v} g_{\bm A}' + g_{\bm u}' g_{\bm v}'' g_{\bm A} +  g_{\bm u} g_{\bm v}'' g_{\bm A}' + g_{\bm u}' g_{\bm v} g_{\bm A}'' + g_{\bm u} g_{\bm v}' g_{\bm A}'') \\
     &\qquad + 6 g_{\bm u}' g_{\bm v}' g_{\bm A}' \\
     &= 6 (u_i \bm A_{ii} g_{\bm v} + v_i \bm A_{ii} g_{\bm u}  ) + 12 u_i v_i (\bm A \bm w)_i
\end{align*}
That is,
\begin{align*}
    \biggl|\partial_i^3 f_3(\bm w)\biggr| &\lcon |u_i \bm A_{ii} g_{\bm v}| + |v_i \bm A_{ii} g_{\bm u}| + |u_i v_i \cdot (\bm A \bm w)_i| \\
    &= |\bm A_{ii} g_{\bm v}  \cdot u_i| + |\bm A_{ii} g_{\bm u} \cdot v_i| + |u_i v_i (\bm A \bm w)_i|
\end{align*}
\subsubsection{First bound for \texorpdfstring{$f_4$}{f4}}
We can express
\begin{equation*}
    f_4(\bm w; \bm S) = \frac{g_{\bm u}(\bm w) g_{\bm v}(\bm w) g_{\bm A}(\bm w)}{g_{1,\bm B}(\bm w)}
\end{equation*}
Noting that $g_{1,\bm B} \neq 0$, we can write $f_4(\bm w) g_{1,\bm B}(\bm w) = g_{\bm u}(\bm w) g_{\bm v}(\bm w) g_{\bm A}(\bm w)$ and differentiating each side yields
\begin{equation*}
    f_4'g_{1,\bm B} + f_4 g_{1,\bm B}' = g_{\bm u}' g_{\bm v} g_{\bm A} + g_{\bm u} g_{\bm v}' g_{\bm A} + g_{\bm u} g_{\bm v} g_{\bm A}'
\end{equation*}
where $f_4' = \partial_i f_4$, etc. This yields, as usual, 
\begin{align*} 
f_4' &= \frac{g_{\bm u}' g_{\bm v} g_{\bm A} + g_{\bm u} g_{\bm v}' g_{\bm A} + g_{\bm u} g_{\bm v} g_{\bm A}' - f_4 g_{1,\bm B}'}{g_{1,\bm B}} \\
&= \frac{u_i g_{\bm v} g_{\bm A} + v_i g_{\bm u} g_{\bm A} + 2 (\bm A \bm w)_i g_{\bm u} g_{\bm v} - 2 f_4 \cdot (\bm B \bm w)_i}{g_{1,\bm B}}.
\end{align*} 
Differentiating each side again yields
\begin{align*}
    f_4'' g_{1,\bm B} + 2f_4' g_{1,\bm B}'+f_4 g_{1,\bm B}'' &=  g_{\bm u}'' g_{\bm v} g_{\bm A} + g_{\bm u} g_{\bm v}'' g_{\bm A} + g_{\bm u} g_{\bm v} g_{\bm A}'' \\
    &\qquad + 2(g_{\bm u}' g_{\bm v}' g_{\bm A} + g_{\bm u}' g_{\bm v} g_{\bm A}' + g_{\bm u} g_{\bm v}' g_{\bm A}') \\
    &= g_{\bm u} g_{\bm v} g_{\bm A}'' + 2(g_{\bm u}' g_{\bm v}' g_{\bm A} + g_{\bm u}' g_{\bm v} g_{\bm A}' + g_{\bm u} g_{\bm v}' g_{\bm A}')
\end{align*}
and rearranging yields
\begin{align*}
    f_4'' &= \frac{g_{\bm u} g_{\bm v} g_{\bm A}'' + 2(g_{\bm u}' g_{\bm v}' g_{\bm A} + g_{\bm u}' g_{\bm v} g_{\bm A}' + g_{\bm u} g_{\bm v}' g_{\bm A}') - 2f_{4}' g_{1,\bm B}' - f_4 g_{1,\bm B}''}{g_{1,\bm B}} \\
    &= \frac{2\bm A_{ii}g_{\bm u} g_{\bm v} + 2 u_i v_i g_{\bm A} + 4u_i (\bm A \bm w)_i g_{\bm v} + 4v_i (\bm A \bm w)_i g_{\bm u} - 4f_4' \cdot (\bm B \bm w)_i - 2 f_4 \cdot \bm B_{ii}}{g_{1,\bm B}}
\end{align*}
Again, we can differentiate each side again to yield
\begin{align*}
    f_4''' g_{1,\bm B} + 3f_4'' g_{1,\bm B}' + 3f_4'g_{1,\bm B}'' + f_4 g_{1,\bm B}''' &= g_{\bm u}''' g_{\bm v} g_{\bm A} + g_{\bm u} g_{\bm v}''' g_{\bm A} + g_{\bm u} g_{\bm v} g_{\bm A}''' \\
    &\qquad + 3(g_{\bm u}'' g_{\bm v}' g_{\bm A} + g_{\bm u}' g_{\bm v}'' g_{\bm A} + g_{\bm u}'' g_{\bm v} g_{\bm A}' \\
    &\qquad \qquad + g_{\bm u}' g_{\bm v} g_{\bm A}'' + g_{\bm u} g_{\bm v}'' g_{\bm A}' + g_{\bm u} g_{\bm v}' g_{\bm A}'') \\
    &\qquad + 6 g_{\bm u}' g_{\bm v}' g_{\bm A}' \\
    &= 3(g_{\bm u}' g_{\bm v} g_{\bm A}'' + g_{\bm u}g_{\bm v}' g_{\bm A}'') + 6 g_{\bm u}' g_{\bm v}' g_{\bm A}'
\end{align*}
which yields
\begin{align*}
    f_4''' &= \frac{3(g_{\bm u}' g_{\bm v} g_{\bm A}'' + g_{\bm u}g_{\bm v}' g_{\bm A}'') + 6 g_{\bm u}' g_{\bm v}' g_{\bm A}' - 3f_{4}'' g_{1,\bm B}' - 3f_4' g_{1,\bm B}''}{g_{1,\bm B}} \\
    &= \frac{6u_i \bm A_{ii} g_{\bm v}  + 6 v_i \bm A_{ii} g_{\bm u}  + 12 u_i v_i (\bm A \bm w)_i - 6f_4'' \cdot (\bm B \bm w)_i - 6 f_4' \cdot \bm B_{ii}}{g_{1,\bm B}}
\end{align*}
For each $i$, noting that $g_{1,\bm B} \geq 1$, we have
\begin{equation*}
    |f_4'| \lcon |u_i g_{\bm v} g_{\bm A}| + |v_i g_{\bm u} g_{\bm A}| + |(\bm A \bm w)_i g_{\bm u} g_{\bm v}| + |f_4 \cdot (\bm B \bm w)_i|
\end{equation*}
and 
\begin{align*}
    |f_4''| &\lcon |\bm A_{ii} g_{\bm u} g_{\bm v} |+ |u_i v_i g_{\bm A}| + |u_i (\bm A \bm w)_i g_{\bm v}| + |v_i (\bm A \bm w)_i g_{\bm u}| + |f_4' \cdot (\bm B \bm w)_i| + |f_4 \cdot \bm B_{ii}| \\
    &\lcon |\bm A_{ii} g_{\bm u} g_{\bm v} |+ |u_i v_i g_{\bm A}| + |u_i (\bm A \bm w)_i g_{\bm v}| + |v_i (\bm A \bm w)_i g_{\bm u}|  + |f_4 \cdot \bm B_{ii}|\\
    &\qquad + (|u_i g_{\bm v} g_{\bm A} (\bm B \bm w)_i| + |v_i g_{\bm u} g_{\bm A}(\bm B \bm w)_i| + |(\bm A \bm w)_i g_{\bm u} g_{\bm v}(\bm B \bm w)_i| + |f_4 \cdot (\bm B \bm w)_i^2|)
\end{align*}
and thus finally
\begin{align*}
    &|\partial_i^3 f_4(\bm w) | \\
    &\lcon |u_i \bm A_{ii} g_{\bm v}| + |v_i \bm A_{ii} g_{\bm u}| + |u_i v_i (\bm A \bm w)_i| + |f_4'' \cdot (\bm B \bm w)_i| + |f_4' \cdot \bm B_{ii}| \\
    &\lcon |u_i \bm A_{ii} g_{\bm v}| + |v_i \bm A_{ii} g_{\bm u}| + |u_i v_i (\bm A \bm w)_i| \\
    &\qquad + [|\bm A_{ii} g_{\bm u} g_{\bm v} (\bm B \bm w)_i|+ |u_i v_i g_{\bm A} (\bm B \bm w)_i| + |u_i (\bm A \bm w)_i g_{\bm v} (\bm B \bm w)_i| + |v_i (\bm A \bm w)_i g_{\bm u} (\bm B \bm w)_i| \\
    &\qquad\quad  + (|u_i g_{\bm v} g_{\bm A} (\bm B \bm w)_i^2| + |v_i g_{\bm u} g_{\bm A}(\bm B \bm w)_i^2| + |(\bm A \bm w)_i g_{\bm u} g_{\bm v}(\bm B \bm w)_i^2|  \\
    &\qquad \quad +|f_4 \cdot (\bm B \bm w)_i^3|) + |f_4 \cdot \bm B_{ii}(\bm B \bm w)_i|] \\
    &\qquad + (|u_i g_{\bm v} g_{\bm A} \bm B_{ii}| + |v_i g_{\bm u} g_{\bm A} \bm B_{ii}| + |(\bm A \bm w)_i g_{\bm u} g_{\bm v}\bm B_{ii}| + |f_4 \cdot (\bm B \bm w)_i\bm B_{ii}|)  \\
    &\lcon |\bm A_{ii} g_{\bm v} \cdot u_i | + |\bm A_{ii} g_{\bm u} \cdot v_i | + |u_i v_i (\bm A \bm w)_i| \\
    &\qquad + [|\bm A_{ii} g_{\bm u} g_{\bm v} \cdot (\bm B \bm w)_i|+ |g_{\bm A} \cdot u_i v_i  (\bm B \bm w)_i| + |g_{\bm v} \cdot u_i (\bm A \bm w)_i  (\bm B \bm w)_i| \\
    &\qquad + |g_{\bm u} \cdot v_i (\bm A \bm w)_i  (\bm B \bm w)_i| \\
    & \qquad+ (|g_{\bm v} g_{\bm A} \cdot u_i (\bm B \bm w)_i^2| + |g_{\bm u} g_{\bm A} \cdot v_i (\bm B \bm w)_i^2| + |g_{\bm u} g_{\bm v} \cdot (\bm A \bm w)_i (\bm B \bm w)_i^2| \\
    &\qquad+ |g_{\bm u} g_{\bm v} g_{\bm A} \cdot (\bm B \bm w)_i^3|) + |\bm B_{ii}g_{\bm u} g_{\bm v} g_{\bm A} \cdot (\bm B \bm w)_i|] \\
    & \qquad+ (|\bm B_{ii}g_{\bm v} g_{\bm A} \cdot u_i  | + |\bm B_{ii}g_{\bm u} g_{\bm A} \cdot v_i | + |\bm B_{ii}g_{\bm u} g_{\bm v} \cdot (\bm A \bm w)_i | + |\bm B_{ii}g_{\bm u} g_{\bm v} g_{\bm A} \cdot (\bm B \bm w)_i|)
\end{align*}

\subsubsection{First bound for \texorpdfstring{$f_5$}{f5}}
We can express
\begin{equation*}
    f_5(\bm w; \bm S) = \frac{g_{\bm u}(\bm w) g_{\bm v}(\bm w) g_{\bm A}(\bm w)}{h_{1,\bm B}(\bm w)}
\end{equation*}
Noting that $h_{1,\bm B} \neq 0$, we can write $f_5(\bm w) h_{1,\bm B}(\bm w) = g_{\bm u}(\bm w) g_{\bm v}(\bm w) g_{\bm A}(\bm w)$ and differentiating each side yields
\begin{equation*}
    f_5'h_{1,\bm B} + f_5 h_{1,\bm B}' = g_{\bm u}' g_{\bm v} g_{\bm A} + g_{\bm u} g_{\bm v}' g_{\bm A} + g_{\bm u} g_{\bm v} g_{\bm A}'
\end{equation*}
where $f_5' = \partial_i f_5$, etc. This yields, as usual, 
\begin{align*} 
f_5' &= \frac{g_{\bm u}' g_{\bm v} g_{\bm A} + g_{\bm u} g_{\bm v}' g_{\bm A} + g_{\bm u} g_{\bm v} g_{\bm A}' - f_5 h_{1,\bm B}'}{h_{1,\bm B}} \\
&= \frac{u_i g_{\bm v} g_{\bm A} + v_i g_{\bm u} g_{\bm A} + 2 (\bm A \bm w)_i g_{\bm u} g_{\bm v} -4 f_5 \cdot (1 + \bm w^\top \bm B \bm w)(\bm B \bm w)_i }{h_{1,\bm B}}
\end{align*} 
Differentiating each side again yields
\begin{align*}
    f_5'' h_{1,\bm B} + 2f_5' h_{1,\bm B}'+f_5 h_{1,\bm B}'' &=  g_{\bm u}'' g_{\bm v} g_{\bm A} + g_{\bm u} g_{\bm v}'' g_{\bm A} + g_{\bm u} g_{\bm v} g_{\bm A}'' \\
    &\qquad + 2(g_{\bm u}' g_{\bm v}' g_{\bm A} + g_{\bm u}' g_{\bm v} g_{\bm A}' + g_{\bm u} g_{\bm v}' g_{\bm A}') \\
    &= g_{\bm u} g_{\bm v} g_{\bm A}'' + 2(g_{\bm u}' g_{\bm v}' g_{\bm A} + g_{\bm u}' g_{\bm v} g_{\bm A}' + g_{\bm u} g_{\bm v}' g_{\bm A}')
\end{align*}
and rearranging yields
\begin{align*}
    f_5'' &= \frac{g_{\bm u} g_{\bm v} g_{\bm A}'' + 2(g_{\bm u}' g_{\bm v}' g_{\bm A} + g_{\bm u}' g_{\bm v} g_{\bm A}' + g_{\bm u} g_{\bm v}' g_{\bm A}') - 2f_5' h_{1,\bm B}' - f_5 h_{1,\bm B}''}{h_{1,\bm B}} \\
    &= \frac{2\bm A_{ii}g_{\bm u} g_{\bm v} + 2 u_i v_i g_{\bm A} + 4u_i (\bm A \bm w)_i g_{\bm v} + 4v_i (\bm A \bm w)_i g_{\bm u} }{h_{1,\bm B}} \\
    &\qquad + \frac{- 8f_5' \cdot (1 + \bm w^\top \bm B\bm w)(\bm B \bm w)_i - 8 f_5 \cdot (\bm B \bm w)_i^2 - 4 f_5 \cdot (1 + \bm w^\top \bm B \bm w) \bm B_{ii}}{h_{1,\bm B}}
\end{align*}
Again, we can differentiate each side again to yield
\begin{align*}
    f_5''' h_{1,\bm B} + 3f_5'' h_{1,\bm B}' + 3f_5'h_{1,\bm B}'' + f_5 h_{1,\bm B}''' &= g_{\bm u}''' g_{\bm v} g_{\bm A} + g_{\bm u} g_{\bm v}''' g_{\bm A} + g_{\bm u} g_{\bm v} g_{\bm A}''' \\
    &\qquad + 3(g_{\bm u}'' g_{\bm v}' g_{\bm A} + g_{\bm u}' g_{\bm v}'' g_{\bm A} + g_{\bm u}'' g_{\bm v} g_{\bm A}' \\
    &\qquad \qquad + g_{\bm u}' g_{\bm v} g_{\bm A}'' + g_{\bm u} g_{\bm v}'' g_{\bm A}' + g_{\bm u} g_{\bm v}' g_{\bm A}'') \\
    &\qquad + 6 g_{\bm u}' g_{\bm v}' g_{\bm A}' \\
    &= 3(g_{\bm u}' g_{\bm v} g_{\bm A}'' + g_{\bm u}g_{\bm v}' g_{\bm A}'') + 6 g_{\bm u}' g_{\bm v}' g_{\bm A}'
\end{align*}
which yields
\begin{align*}
    f_5''' &= \frac{3(g_{\bm u}' g_{\bm v} g_{\bm A}'' + g_{\bm u}g_{\bm v}' g_{\bm A}'') + 6 g_{\bm u}' g_{\bm v}' g_{\bm A}' - 3f_5'' h_{1,\bm B}' - 3f_5' h_{1,\bm B}'' - f_5 h_{1,\bm B}'''}{h_{1,\bm B}} \\
    &= \frac{6u_i g_{\bm v} \bm A_{ii} + 6 v_i g_{\bm u} \bm A_{ii} + 12 u_i v_i (\bm A \bm w)_i}{h_{1,\bm B}} \\
    &\qquad - \frac{12f_5'' \cdot (1 + \bm w^\top \bm B \bm w)(\bm B \bm w)_i + 24f_5' \cdot (\bm B \bm w)_i^2 }{h_{1,\bm B}} \\
    &\qquad - \frac{12 f_5' \cdot  (1 + \bm w^\top \bm B \bm w) \bm B_{ii} + 24 f_5 \cdot (\bm B \bm w)_i \bm B_{ii}}{h_{1,\bm B}}
\end{align*}
For each $i$, noting that $h_{1,\bm B} \geq 1$, we have
\begin{equation*}
    |f_5'| \lcon |u_i g_{\bm v} g_{\bm A}| + |v_i g_{\bm u} g_{\bm A}| + |(\bm A \bm w)_i g_{\bm u} g_{\bm v}| + |f_5 \cdot (1 + \bm w^\top \bm B \bm w) (\bm B \bm w)_i|
\end{equation*}
and 
\begin{align*}
    |f_5''| &\lcon |\bm A_{ii} g_{\bm u} g_{\bm v} |+ |u_i v_i g_{\bm A}| + |u_i (\bm A \bm w)_i g_{\bm v}| + |v_i (\bm A \bm w)_i g_{\bm u}| \\
    &\qquad + |f_5' \cdot (1 + \bm w^\top \bm B \bm w)(\bm B \bm w)_i| + |f_5 \cdot \bm (\bm B \bm w)_i^2| + |f_5 \cdot (1 + \bm w^\top \bm B \bm w) \bm B_{ii}| \\
    &\lcon |\bm A_{ii} g_{\bm u} g_{\bm v} |+ |u_i v_i g_{\bm A}| + |u_i (\bm A \bm w)_i g_{\bm v}| + |v_i (\bm A \bm w)_i g_{\bm u}| \\
    &\qquad + (|u_i g_{\bm v} g_{\bm A} (1 + \bm w^\top \bm B \bm w)(\bm B \bm w)_i| + |v_i g_{\bm u} g_{\bm A}(1 + \bm w^\top \bm B \bm w)(\bm B \bm w)_i| \\
    &\qquad + |(\bm A \bm w)_i g_{\bm u} g_{\bm v} (1 + \bm w^\top \bm B \bm w)(\bm B \bm w)_i| + |f_5 \cdot (1 + \bm w^\top \bm B \bm w)^2 (\bm B \bm w)_i^2|) \\
    &\qquad + |f_5 \cdot \bm (\bm B \bm w)_i^2| + |f_5 \cdot (1 + \bm w^\top \bm B \bm w) \bm B_{ii}|
\end{align*}
and thus finally
\begin{align*}
    &\biggl|\partial_i^3 f_5(\bm w) \biggr| \lcon |u_i \bm A_{ii} g_{\bm v}| + |v_i \bm A_{ii} g_{\bm u}| + |u_i v_i (\bm A \bm w)_i| + |f_5'' \cdot (1 + \bm w^\top \bm B \bm w) (\bm B \bm w)_i| \\
    &\qquad + |f_5' \cdot (\bm B \bm w)_i^2| + |f_5' \cdot (1 + \bm w^\top \bm B \bm w)\bm B_{ii}| + |f_5 \cdot (\bm B \bm w)_i \bm B_{ii}| \\
    &\lcon |u_i \bm A_{ii} g_{\bm v}| + |v_i \bm A_{ii} g_{\bm u}| + |u_i v_i (\bm A \bm w)_i| \\
    &+ [|\bm A_{ii} g_{\bm u} g_{\bm v} (1 + \bm w^\top \bm B \bm w) (\bm B \bm w)_i|+ |u_i v_i g_{\bm A}(1 + \bm w^\top \bm B \bm w) (\bm B \bm w)_i| \\
    &+ |u_i (\bm A \bm w)_i g_{\bm v} (1 + \bm w^\top \bm B \bm w) (\bm B \bm w)_i| + |v_i (\bm A \bm w)_i g_{\bm u} (1 + \bm w^\top \bm B \bm w) (\bm B \bm w)_i| \\
    & + (|u_i g_{\bm v} g_{\bm A} (1 + \bm w^\top \bm B \bm w)^2(\bm B \bm w)_i^2| + |v_i g_{\bm u} g_{\bm A}(1 + \bm w^\top \bm B \bm w)^2(\bm B \bm w)_i^2| \\
    & + |(\bm A \bm w)_i g_{\bm u} g_{\bm v} (1 + \bm w^\top \bm B \bm w)^2(\bm B \bm w)_i^2| + |f_5 \cdot (1 + \bm w^\top \bm B \bm w)^3 (\bm B \bm w)_i^3|) \\
    & + |f_5 \cdot \bm (\bm B \bm w)_i^3 (1 + \bm w^\top \bm B \bm w) | + |f_5 \cdot (1 + \bm w^\top \bm B \bm w)^2 \bm B_{ii}(\bm B \bm w)_i|] \\
    & + (|u_i g_{\bm v} g_{\bm A} (\bm B \bm w)_i^2| + |v_i g_{\bm u} g_{\bm A} (\bm B \bm w)_i^2| + |(\bm A \bm w)_i g_{\bm u} g_{\bm v} (\bm B \bm w)_i^2| \\
    &\qquad + |f_5 \cdot (1 + \bm w^\top \bm B \bm w) (\bm B \bm w)_i^3|) \\
    & + (|u_i g_{\bm v} g_{\bm A} (1 + \bm w^\top \bm B \bm w) \bm B_{ii}| + |v_i g_{\bm u} g_{\bm A} (1 + \bm w^\top \bm B \bm w) \bm B_{ii}|  \\
    &\qquad + |(\bm A \bm w)_i g_{\bm u} g_{\bm v}(1 + \bm w^\top \bm B \bm w) \bm B_{ii}|+ |f_5 \cdot (1 + \bm w^\top \bm B \bm w)^2 (\bm B \bm w)_i \bm B_{ii}|)  \\
    &+ |f_5 \cdot (\bm B \bm w)_i \bm B_{ii}| \\
    &\lcon |\bm A_{ii} g_{\bm v} \cdot u_i| + |\bm A_{ii} g_{\bm u} \cdot v_i| + |u_i v_i (\bm A \bm w)_i| \\
    & + [|\bm A_{ii} g_{\bm u} g_{\bm v} g_{1, \bm B} \cdot (\bm B \bm w)_i|+ |g_{\bm A}g_{1, \bm B} \cdot u_i v_i (\bm B \bm w)_i| \\
    &+ | g_{\bm v} g_{1, \bm B} \cdot u_i (\bm A \bm w)_i(\bm B \bm w)_i| + | g_{\bm u} g_{1, \bm B} \cdot v_i (\bm A \bm w)_i (\bm B \bm w)_i| \\
    &+ (| g_{\bm v} g_{\bm A} g_{1, \bm B}^2\cdot u_i(\bm B \bm w)_i^2| + |g_{\bm u} g_{\bm A}g_{1, \bm B}^2\cdot v_i(\bm B \bm w)_i^2| \\
    & + | g_{\bm u} g_{\bm v} g_{1, \bm B}^2 \cdot (\bm A \bm w)_i(\bm B \bm w)_i^2| + |g_{\bm u} g_{\bm v} g_{\bm A} g_{1, \bm B}^3 \cdot (\bm B \bm w)_i^3|) \\
    & + |g_{\bm u} g_{\bm v} g_{\bm A} g_{1, \bm B} \cdot \bm (\bm B \bm w)_i^3 | + |\bm B_{ii}g_{\bm u} g_{\bm v} g_{\bm A} g_{1, \bm B}^2 \cdot (\bm B \bm w)_i|] \\
    & + (| g_{\bm v} g_{\bm A} \cdot u_i (\bm B \bm w)_i^2| + | g_{\bm u} g_{\bm A} \cdot v_i(\bm B \bm w)_i^2| + |g_{\bm u} g_{\bm v} \cdot (\bm A \bm w)_i  (\bm B \bm w)_i^2| \\
    &\qquad + |g_{\bm u} g_{\bm v} g_{\bm A} \cdot g_{1, \bm B} \cdot (\bm B \bm w)_i^3|) \\
    & + (|\bm B_{ii} g_{\bm v} g_{\bm A} g_{1, \bm B} \cdot u_i| + |\bm B_{ii} g_{\bm u} g_{\bm A} g_{1, \bm B} \cdot v_i| + |\bm B_{ii} g_{\bm u} g_{\bm v}g_{1, \bm B} \cdot  (\bm A \bm w)_i| \\
    & + |\bm B_{ii} g_{\bm u} g_{\bm v} g_{\bm A} g_{1, \bm B}^2 \cdot  (\bm B \bm w)_i |)  \\
    & + | \bm B_{ii} g_{\bm u} g_{\bm v} g_{\bm A} \cdot (\bm B \bm w)_i|
\end{align*}

\subsubsection{First bound for \texorpdfstring{$f_6$}{f6}}
We can express
\begin{equation*}
    f_6(\bm w; \bm S) = \frac{h_{\bm u}(\bm w) g_{\bm A}(\bm w) g_{\bm B}(\bm w)}{g_{1,\bm C}(\bm w)}.
\end{equation*}

Noting that $g_{1,\bm C} \neq 0$, we can write $f_6(\bm w) g_{1,\bm C}(\bm w) = h_{\bm u}(\bm w) g_{\bm A}(\bm w) g_{\bm B}(\bm w)$ and differentiating each side yields
\begin{equation*}
    f_6'g_{1,\bm C} + f_6 g_{1,\bm C}' = h_{\bm u}' g_{\bm A} g_{\bm B} + h_{\bm u} g_{\bm A}' g_{\bm B} + h_{\bm u} g_{\bm A} g_{\bm B}'
\end{equation*}
where $f_6' = \partial_i f_6$, etc. This yields, as usual, 
\begin{align*} 
f_6' &= \frac{h_{\bm u}' g_{\bm A} g_{\bm B} + h_{\bm u} g_{\bm A}' g_{\bm B} + h_{\bm u} g_{\bm A} g_{\bm B}' - f_6 g_{1,C}'}{g_{1,\bm C}} \\
&= \frac{2u_i (\bm u^\top \bm w) g_{\bm A} g_{\bm B} + 2 (\bm A \bm w)_ih_{\bm u} g_{\bm B} + 2 (\bm B \bm w)_i h_{\bm u} g_{\bm A} - 2 f_6 \cdot (\bm C \bm w)_i}{g_{1, \bm C}}
\end{align*} 
Differentiating each side again yields
\begin{align*}
    f_6'' g_{1, \bm C} + 2f_6' g_{1, \bm C}' + f_6 g_{1, \bm C}'' &= h_{\bm u}'' g_{\bm A} g_{\bm B} + h_{\bm u} g_{\bm A}'' g_{\bm B} + h_{\bm u} g_{\bm A} g_{\bm B}'' \\
    &\qquad + 2(h_{\bm u}' g_{\bm A}' g_{\bm B} + h_{\bm u}' g_{\bm A} g_{\bm B}' + h_{\bm u} g_{\bm A}' g_{\bm B}')
\end{align*}
and rearranging yields
\begin{align*}
    f_6'' &= \frac{h_{\bm u}'' g_{\bm A} g_{\bm B} + h_{\bm u} g_{\bm A}'' g_{\bm B} + h_{\bm u} g_{\bm A} g_{\bm B}'' }{g_{1, \bm C}} \\
    &\qquad +\frac{2(h_{\bm u}' g_{\bm A}' g_{\bm B} + h_{\bm u}' g_{\bm A} g_{\bm B}' + h_{\bm u} g_{\bm A}' g_{\bm B}') - 2f_6' g_{1, \bm C}' - f_6 g_{1, \bm C}''}{g_{1, \bm C}} \\
    &= \frac{2 u_i^2 g_{\bm A} g_{\bm B} + 2 \bm A_{ii} h_{\bm u} g_{\bm B} + 2 \bm B_{ii} h_{\bm u} g_{\bm A} +  8 u_i (\bm u^\top \bm w) (\bm A \bm w)_i g_{\bm B}}{g_{1, \bm C}} \\
    &\qquad + \frac{8 u_i (\bm u^\top \bm w) (\bm B \bm w)_i g_{\bm A} + 8 (\bm A \bm w)_i (\bm B \bm w)_i h_{\bm u} }{g_{1, \bm C}} \\
    &\qquad - \frac{4 f_6' \cdot ( \bm C \bm w)_i + 2 f_6 \cdot \bm C_{ii}}{g_{1, \bm C}}
\end{align*}
Again, we can differentiate each side again to yield
\begin{align*}
    f_6''' &g_{1, \bm C} + 3f_6'' g_{1, \bm C}' + 3f_6' g_{1,\bm C}'' + f_6 g_{1, \bm C}''' \\&= h_{\bm u}''' g_{\bm A} g_{\bm B} + h_{\bm u} g_{\bm A}''' g_{\bm B} + h_{\bm u} g_{\bm A} g_{\bm B}''' \\
    &\qquad + 3 (h_{\bm u}'' g_{\bm A}' g_{\bm B} +  h_{\bm u}' g_{\bm A}'' g_{\bm B} + h_{\bm u}'' g_{\bm A} g_{\bm B}' + h_{\bm u}' g_{\bm A} g_{\bm B}'' + h_{\bm u} g_{\bm A}'' g_{\bm B}' + h_{\bm u} g_{\bm A}' g_{\bm B}'') \\
    &\qquad + 6 h_{\bm u}' g_{\bm A}' g_{\bm B}' \\
    &= 3 (h_{\bm u}'' g_{\bm A}' g_{\bm B} +  h_{\bm u}' g_{\bm A}'' g_{\bm B} + h_{\bm u}'' g_{\bm A} g_{\bm B}' + h_{\bm u}' g_{\bm A} g_{\bm B}'' + h_{\bm u} g_{\bm A}'' g_{\bm B}' + h_{\bm u} g_{\bm A}' g_{\bm B}'') \\
    &\qquad + 6 h_{\bm u}' g_{\bm A}' g_{\bm B}'
\end{align*}
which yields
\begin{align*}
    f_6''' &= \frac{3 (h_{\bm u}'' g_{\bm A}' g_{\bm B} +  h_{\bm u}' g_{\bm A}'' g_{\bm B} + h_{\bm u}'' g_{\bm A} g_{\bm B}' + h_{\bm u}' g_{\bm A} g_{\bm B}'' + h_{\bm u} g_{\bm A}'' g_{\bm B}' + h_{\bm u} g_{\bm A}' g_{\bm B}'') + 6 h_{\bm u}' g_{\bm A}' g_{\bm B}'}{g_{1, \bm C}} \\
    &\qquad - \frac{3f_6'' g_{1,\bm C}'  + 3 f_6' g_{1, \bm C}''}{g_{1, \bm C}} \\
    &= \frac{3 (4u_i^2 (\bm A \bm w)_i g_{\bm B} +  4 u_i (\bm u^\top \bm w) \bm A_{ii} g_{\bm B} + 4 u_i^2 (\bm B\bm w)_i g_{\bm A}}{g_{1, \bm C}}  \\
    &\qquad + \frac{4 u_i (\bm u^\top \bm w) \bm B_{ii} g_{\bm A}  + 4 (\bm B \bm w)_i \bm A_{ii} h_{\bm u}  + 4 (\bm A \bm w)_i \bm B_{ii} h_{\bm u}) }{g_{1, \bm C}} \\
    &\qquad + \frac{48 u_i (\bm u^\top \bm w) (\bm A \bm w)_i (\bm B \bm w)_i}{g_{1, \bm C}} - \frac{6f_6'' \cdot (\bm C \bm w)_i  + 6 f_6' \bm C_{ii}}{g_{1, \bm C}} 
\end{align*}
For each $i$, noting that $g_{1,\bm C} \geq 1$, we have 
\begin{equation*}
    |f_6'| \lcon |u_i (\bm u^\top \bm w) g_{\bm A} g_{\bm B}| + |(\bm A \bm w)_ih_{\bm u} g_{\bm B}| + |(\bm B \bm w)_i h_{\bm u} g_{\bm A}| + |f_6 \cdot (\bm C \bm w)_i|
\end{equation*}
and 
\begin{align*}
    |f_6''| &\lcon  |u_i^2 g_{\bm A} g_{\bm B}| +  |\bm A_{ii} h_{\bm u} g_{\bm B}| + |\bm B_{ii} h_{\bm u} g_{\bm A}| +  |u_i (\bm u^\top \bm w) (\bm A \bm w)_i g_{\bm B}| \\
    &\qquad + |u_i (\bm u^\top \bm w) (\bm B \bm w)_i g_{\bm A}| + | (\bm A \bm w)_i (\bm B \bm w)_i h_{\bm u}| + |f_6'\cdot (\bm C\bm w)_i| + |f_6 \bm C_{ii}| \\
    &\lcon |u_i^2 g_{\bm A} g_{\bm B}| +  |\bm A_{ii} h_{\bm u} g_{\bm B}| + |\bm B_{ii} h_{\bm u} g_{\bm A}| +  |u_i (\bm u^\top \bm w) (\bm A \bm w)_i g_{\bm B}|\\
    &\qquad + |u_i (\bm u^\top \bm w) (\bm B \bm w)_i g_{\bm A}| + | (\bm A \bm w)_i (\bm B \bm w)_i h_{\bm u}| \\
    &\qquad + (|u_i (\bm u^\top \bm w) g_{\bm A} g_{\bm B} \cdot (\bm C\bm w)_i| + |(\bm A \bm w)_ih_{\bm u} g_{\bm B}\cdot (\bm C\bm w)_i| \\
    &\hspace{4em}+ |(\bm B \bm w)_i h_{\bm u} g_{\bm A}\cdot (\bm C\bm w)_i| + |f_6 \cdot (\bm C \bm w)_i^2|)
\end{align*}
and thus finally
\begin{align*}
    \biggl|\partial_i^3 f_6(\bm w)\biggr| &\lcon |u_i^2 (\bm A \bm w)_i g_{\bm B}| +  |u_i (\bm u^\top \bm w) \bm A_{ii} g_{\bm B}| + |u_i^2 (\bm B\bm w)_i g_{\bm A}|  + |u_i (\bm u^\top \bm w) \bm B_{ii} g_{\bm A}| \\
    &\qquad  + |(\bm B \bm w)_i \bm A_{ii} h_{\bm u}|  + |(\bm A \bm w)_i \bm B_{ii} h_{\bm u}|  + |u_i(\bm u^\top \bm w) (\bm A \bm w)_i (\bm B \bm w)_i| \\
    &\qquad + |f_6'' \cdot (\bm C \bm w)_i | + |f_6' \bm C_{ii}| \\
    &\lcon |u_i^2 (\bm A \bm w)_i g_{\bm B}| +  |u_i (\bm u^\top \bm w) \bm A_{ii} g_{\bm B}| + |u_i^2 (\bm B\bm w)_i g_{\bm A}|  + |u_i (\bm u^\top \bm w) \bm B_{ii} g_{\bm A}|   \\
    &\qquad + |(\bm B \bm w)_i \bm A_{ii} h_{\bm u}|  + |(\bm A \bm w)_i \bm B_{ii} h_{\bm u}| + |u_i(\bm u^\top \bm w) (\bm A \bm w)_i (\bm B \bm w)_i| \\
    &\qquad + [|u_i^2 g_{\bm A} g_{\bm B}\cdot (\bm C \bm w)_i| +  |\bm A_{ii} h_{\bm u} g_{\bm B}\cdot (\bm C \bm w)_i| + |\bm B_{ii} h_{\bm u} g_{\bm A}\cdot (\bm C \bm w)_i|  \\
    &\qquad \quad +  |u_i (\bm u^\top \bm w) (\bm A \bm w)_i g_{\bm B}\cdot (\bm C \bm w)_i|+ |u_i (\bm u^\top \bm w) (\bm B \bm w)_i g_{\bm A}\cdot (\bm C \bm w)_i|  \\
    &\qquad \quad + | (\bm A \bm w)_i (\bm B \bm w)_i h_{\bm u} \cdot (\bm C \bm w)_i| + (|u_i (\bm u^\top \bm w) g_{\bm A} g_{\bm B} \cdot (\bm C\bm w)_i^2|  \\
    &\qquad \quad + |(\bm A \bm w)_ih_{\bm u} g_{\bm B}\cdot (\bm C\bm w)_i^2| + |(\bm B \bm w)_i h_{\bm u} g_{\bm A}\cdot (\bm C\bm w)_i^2| + |f_6 \cdot (\bm C \bm w)_i^3|)]  \\
    &\qquad  + (|u_i (\bm u^\top \bm w) g_{\bm A} g_{\bm B}\bm C_{ii}| + |(\bm A \bm w)_ih_{\bm u} g_{\bm B}\bm C_{ii}| \\
    &\hspace{4em}+ |(\bm B \bm w)_i h_{\bm u} g_{\bm A}\bm C_{ii}| + |f_6 \cdot (\bm C \bm w)_i\bm C_{ii}|)  \\
    &\lcon | g_{\bm B} \cdot u_i^2 (\bm A \bm w)_i| +  |\bm A_{ii} g_{\bm u}  g_{\bm B}\cdot u_i| + |g_{\bm A} \cdot u_i^2 (\bm B\bm w)_i |  + |\bm B_{ii}g_{\bm u}  g_{\bm A} \cdot u_i |   \\
    &\qquad + | \bm A_{ii} g_{\bm u}^2 \cdot (\bm B \bm w)_i|  + | \bm B_{ii} g_{\bm u}^2 \cdot (\bm A \bm w)_i| + |g_{\bm u} \cdot u_i (\bm A \bm w)_i (\bm B \bm w)_i| \\
    &\qquad + [|g_{\bm A} g_{\bm B}\cdot u_i^2 (\bm C \bm w)_i| +  |\bm A_{ii} g_{\bm u}^2 g_{\bm B}\cdot (\bm C \bm w)_i| + |\bm B_{ii} g_{\bm u}^2 g_{\bm A}\cdot (\bm C \bm w)_i| \\
    &\qquad +  |g_{\bm u} g_{\bm B}\cdot u_i (\bm A \bm w)_i  (\bm C \bm w)_i| \\
    &\qquad \quad + | g_{\bm u} g_{\bm A}\cdot u_i (\bm B \bm w)_i  (\bm C \bm w)_i| + |g_{\bm u}^2 \cdot (\bm A \bm w)_i (\bm B \bm w)_i  (\bm C \bm w)_i| \\
    &\qquad \quad + (|g_{\bm u} g_{\bm A} g_{\bm B} \cdot u_i (\bm C\bm w)_i^2| + | g_{\bm u}^2 g_{\bm B}\cdot (\bm A \bm w)_i (\bm C\bm w)_i^2| \\
    &\qquad \quad + |g_{\bm u}^2 g_{\bm A}\cdot (\bm B \bm w)_i  (\bm C\bm w)_i^2| + |g_{\bm u}^2 g_{\bm A} g_{\bm B} \cdot (\bm C \bm w)_i^3|)]  \\
    &\qquad  + (|\bm C_{ii} g_{\bm u} g_{\bm A} g_{\bm B} \cdot u_i| + |\bm C_{ii} g_{\bm u}^2 g_{\bm B} \cdot (\bm A \bm w)_i|  \\
    &\hspace{4em}+ |\bm C_{ii} g_{\bm u}^2 g_{\bm A} \cdot (\bm B \bm w)_i| + |\bm C_{ii} g_{\bm u}^2 g_{\bm A} g_{\bm B} \cdot (\bm C \bm w)_i|) 
\end{align*}
We now bound each of the corresponding terms, as we did earlier.

\subsubsection{First bound for \texorpdfstring{$f_7$}{f7}}
We can express
\begin{equation*}
    f_7(\bm w; \bm S) = \frac{g_{\bm u}(\bm w) g_{\bm v}(\bm w) g_{\bm A}(\bm w) g_{\bm B}(\bm w)}{h_{1,\bm C}(\bm w)}.
\end{equation*}
Noting that $h_{1,\bm C} \neq 0$, we can write $f_7(\bm w) h_{1,\bm C}(\bm w) = g_{\bm u}(\bm w) g_{\bm v}(\bm w) g_{\bm A}(\bm w) g_{\bm B}(\bm w)$ and differentiating each side yields
\begin{align*}
    f_7'h_{1,\bm C} + f_7 h_{1,\bm C}' &= g_{\bm u}' g_{\bm v} g_{\bm A} g_{\bm B} + g_{\bm u} g_{\bm v}' g_{\bm A} g_{\bm B} + g_{\bm u} g_{\bm v} g_{\bm A}' g_{\bm B} + g_{\bm u} g_{\bm v} g_{\bm A} g_{\bm B} '
\end{align*}
where $f_7' = \partial_i f_7$, etc. This yields, as usual, 
\begin{align*} 
f_7' &= \frac{g_{\bm u}' g_{\bm v} g_{\bm A} g_{\bm B} + g_{\bm u} g_{\bm v}' g_{\bm A} g_{\bm B} + g_{\bm u} g_{\bm v} g_{\bm A}' g_{\bm B} + g_{\bm u} g_{\bm v} g_{\bm A} g_{\bm B} ' - f_7 h_{1,\bm C}'}{h_{1,\bm C}} \\
&= \frac{u_i g_{\bm v} g_{\bm A} g_{\bm B} + v_i g_{\bm u}  g_{\bm A} g_{\bm B} + 2 (\bm A \bm w)_ig_{\bm u} g_{\bm v} g_{\bm B} }{h_{1,\bm C}}  \\
&\qquad+ \frac{2 (\bm B \bm w)_ig_{\bm u} g_{\bm v} g_{\bm A} - 4f_7 \cdot (1 + \bm w^\top \bm C \bm w) (\bm C \bm w)_i}{h_{1,\bm C}} 
\end{align*} 
Differentiating each side again yields
\begin{align*}
    f_7'' h_{1, \bm C} + 2f_7' h_{1, \bm C}' + f_7 h_{1, \bm C}'' &= g_{\bm u}'' g_{\bm v} g_{\bm A} g_{\bm B} + g_{\bm u} g_{\bm v}'' g_{\bm A} g_{\bm B} + g_{\bm u} g_{\bm v} g_{\bm A}'' g_{\bm B} + g_{\bm u} g_{\bm v} g_{\bm A} g_{\bm B}'' \\
    &\qquad + 2(g_{\bm u}' g_{\bm v}' g_{\bm A} g_{\bm B} + g_{\bm u}' g_{\bm v} g_{\bm A}' g_{\bm B} + g_{\bm u}' g_{\bm v} g_{\bm A} g_{\bm B}' + g_{\bm u} g_{\bm v}' g_{\bm A}' g_{\bm B} \\
    &\qquad \qquad + g_{\bm u} g_{\bm v}' g_{\bm A} g_{\bm B}' + g_{\bm u} g_{\bm v} g_{\bm A}' g_{\bm B}') \\
    &=g_{\bm u} g_{\bm v} g_{\bm A}'' g_{\bm B} + g_{\bm u} g_{\bm v} g_{\bm A} g_{\bm B}'' \\
    &\qquad + 2(g_{\bm u}' g_{\bm v}' g_{\bm A} g_{\bm B} + g_{\bm u}' g_{\bm v} g_{\bm A}' g_{\bm B} + g_{\bm u}' g_{\bm v} g_{\bm A} g_{\bm B}' + g_{\bm u} g_{\bm v}' g_{\bm A}' g_{\bm B} \\
    &\qquad \qquad + g_{\bm u} g_{\bm v}' g_{\bm A} g_{\bm B}' + g_{\bm u} g_{\bm v} g_{\bm A}' g_{\bm B}')
\end{align*}
and rearranging yields
\begin{align*}
    &f_7'' = \frac{2(g_{\bm u}' g_{\bm v}' g_{\bm A} g_{\bm B} + g_{\bm u}' g_{\bm v} g_{\bm A}' g_{\bm B} + g_{\bm u}' g_{\bm v} g_{\bm A} g_{\bm B}' + g_{\bm u} g_{\bm v}' g_{\bm A}' g_{\bm B} + g_{\bm u} g_{\bm v}' g_{\bm A} g_{\bm B}' + g_{\bm u} g_{\bm v} g_{\bm A}' g_{\bm B}')}{h_{1, \bm C}} \\
    &\qquad 
    + \frac{g_{\bm u} g_{\bm v} g_{\bm A}'' g_{\bm B} + g_{\bm u} g_{\bm v} g_{\bm A} g_{\bm B}'' - 2f_7' h_{1,\bm C}' - f_7 h_{1,\bm C}''}{h_{1, \bm C}} \\
    &=\frac{2(u_i v_i g_{\bm A} g_{\bm B} + 2 u_i (\bm A \bm w)_i g_{\bm v} g_{\bm B} + 2 u_i (\bm B \bm w)_i  g_{\bm v} g_{\bm A})}{h_{1, \bm C}} \\
    &\qquad + \frac{2( 2 v_i (\bm A \bm w)_i g_{\bm u}  g_{\bm B} + 2 v_i (\bm B \bm w)_i g_{\bm u} g_{\bm A} + 4 (\bm A \bm w)_i (\bm B \bm w)_i g_{\bm u} g_{\bm v})}{h_{1, \bm C}} \\
    &\qquad 
    + \frac{2 \bm A_{ii} g_{\bm u} g_{\bm v}  g_{\bm B} + 2 \bm B_{ii} g_{\bm u} g_{\bm v} g_{\bm A} - 8 f_7' \cdot (1 + \bm w^\top \bm C \bm w)(\bm C \bm w)_i }{h_{1, \bm C}} \\
    &\qquad 
    + \frac{- 8 f_7 \cdot (\bm C \bm w)_i^2 - 4 f_7 \cdot (1 + \bm w^\top \bm C \bm w) \bm C_{ii}}{h_{1, \bm C}}
\end{align*}
Again, we can differentiate each side again to yield
\begin{align*}
    &f_7''' h_{1, \bm C} + 3f_7'' h_{1, \bm C}' + 3f_7' h_{1, \bm C}'' + f_7 h_{1, \bm C}''' \\&= g_{\bm u}''' g_{\bm v} g_{\bm A} g_{\bm B} + g_{\bm u} g_{\bm v}''' g_{\bm A} g_{\bm B} + g_{\bm u} g_{\bm v} g_{\bm A}''' g_{\bm B} + g_{\bm u} g_{\bm v} g_{\bm A} g_{\bm B}''' \\
    &\qquad + 3 (g_{\bm u}'' g_{\bm v}' g_{\bm A} g_{\bm B} + g_{\bm u}' g_{\bm v}'' g_{\bm A} g_{\bm B} + g_{\bm u}'' g_{\bm v} g_{\bm A}' g_{\bm B} + g_{\bm u}' g_{\bm v} g_{\bm A}'' g_{\bm B}\\
    &\qquad \qquad + g_{\bm u}'' g_{\bm v} g_{\bm A} g_{\bm B}' + g_{\bm u}' g_{\bm v} g_{\bm A} g_{\bm B}'' + g_{\bm u} g_{\bm v}'' g_{\bm A}' g_{\bm B} + g_{\bm u} g_{\bm v}' g_{\bm A}'' g_{\bm B} \\
    &\qquad \qquad + g_{\bm u} g_{\bm v}'' g_{\bm A} g_{\bm B}' + g_{\bm u} g_{\bm v}' g_{\bm A} g_{\bm B}'' + g_{\bm u} g_{\bm v} g_{\bm A}'' g_{\bm B}' + g_{\bm u} g_{\bm v} g_{\bm A}' g_{\bm B}'') \\
    &\qquad + 6 (g_{\bm u}' g_{\bm v}' g_{\bm A}' g_{\bm B} + g_{\bm u}' g_{\bm v}' g_{\bm A} g_{\bm B}' + g_{\bm u}' g_{\bm v} g_{\bm A}' g_{\bm B}' + g_{\bm u} g_{\bm v}' g_{\bm A}' g_{\bm B}') \\
    &= 3 ( g_{\bm u}' g_{\bm v} g_{\bm A}'' g_{\bm B} + g_{\bm u}' g_{\bm v} g_{\bm A} g_{\bm B}'' + g_{\bm u} g_{\bm v}' g_{\bm A}'' g_{\bm B} \\
    &\qquad \qquad + g_{\bm u} g_{\bm v}' g_{\bm A} g_{\bm B}'' + g_{\bm u} g_{\bm v} g_{\bm A}'' g_{\bm B}' + g_{\bm u} g_{\bm v} g_{\bm A}' g_{\bm B}'') \\
    &\qquad + 6 (g_{\bm u}' g_{\bm v}' g_{\bm A}' g_{\bm B} + g_{\bm u}' g_{\bm v}' g_{\bm A} g_{\bm B}' + g_{\bm u}' g_{\bm v} g_{\bm A}' g_{\bm B}' + g_{\bm u} g_{\bm v}' g_{\bm A}' g_{\bm B}') 
\end{align*}
which yields
\begin{align*}
    f_7''' &= \frac{3 ( g_{\bm u}' g_{\bm v} g_{\bm A}'' g_{\bm B} + g_{\bm u}' g_{\bm v} g_{\bm A} g_{\bm B}'' + g_{\bm u} g_{\bm v}' g_{\bm A}'' g_{\bm B} + g_{\bm u} g_{\bm v}' g_{\bm A} g_{\bm B}'' + g_{\bm u} g_{\bm v} g_{\bm A}'' g_{\bm B}' + g_{\bm u} g_{\bm v} g_{\bm A}' g_{\bm B}'')}{h_{1,\bm C}} \\
    &\quad + \frac{6 (g_{\bm u}' g_{\bm v}' g_{\bm A}' g_{\bm B} + g_{\bm u}' g_{\bm v}' g_{\bm A} g_{\bm B}' + g_{\bm u}' g_{\bm v} g_{\bm A}' g_{\bm B}' + g_{\bm u} g_{\bm v}' g_{\bm A}' g_{\bm B}')}{h_{1, \bm C}} \\
    & \quad- \frac{3f_7'' h_{1,\bm C}' + 3f_7' h_{1,\bm C}'' + f_7 h_{1,\bm C}'''}{h_{1, \bm C}} \\
    &= \frac{ 6 (u_i \bm A_{ii} g_{\bm v} g_{\bm B}+ u_i \bm B_{ii} g_{\bm v} g_{\bm A}  + v_i \bm A_{ii} g_{\bm u} g_{\bm B})}{h_{1, \bm C}} \\
    &\quad + \frac{ 6 (v_i \bm B_{ii} g_{\bm u} g_{\bm A}  + 2 \bm A_{ii} (\bm B \bm w)_i g_{\bm u} g_{\bm v} + 2 \bm B_{ii} (\bm A \bm w)_i g_{\bm u} g_{\bm v})}{h_{1, \bm C}} \\
    &\quad + \frac{12 (u_i v_i (\bm A \bm w)_i g_{\bm B} + u_i v_i (\bm B \bm w)_i  g_{\bm A}+ 2 u_i   (\bm A \bm w)_i (\bm B \bm w)_i g_{\bm v} + 2  v_i (\bm A \bm w)_i (\bm B \bm w)_i g_{\bm u})}{h_{1, \bm C}} \\
    &\quad - \frac{12f_7'' (1 + \bm w^\top \bm C \bm w)(\bm C \bm w)_i + 24f_7' \cdot (\bm C \bm w)_i^2}{h_{1, \bm C}} \\
    &\quad + \frac{12 f_7' \cdot (1 + \bm w^\top \bm C \bm w)\bm C_{ii}  + 24 f_7 \cdot (\bm C \bm w)_i \bm C_{ii}}{h_{1, \bm C}}
\end{align*}
For each $i$, noting that $h_{1,\bm C} \geq 1$, we have
\begin{align*}
    |f_7'| &\lcon |u_i g_{\bm v} g_{\bm A} g_{\bm B}| + |v_i g_{\bm u}  g_{\bm A} g_{\bm B}| + |(\bm A \bm w)_ig_{\bm u} g_{\bm v} g_{\bm B}| \\
    &\qquad + |(\bm B \bm w)_ig_{\bm u} g_{\bm v} g_{\bm A}| + |f_7 \cdot (1 + \bm w^\top \bm C \bm w) (\bm C \bm w)_i|
\end{align*}
and 
\begin{align*}
    |f_7''| &\lcon |u_i v_i g_{\bm A} g_{\bm B}| + |u_i (\bm A \bm w)_i g_{\bm v} g_{\bm B}| + |u_i (\bm B \bm w)_i  g_{\bm v} g_{\bm A}|  + |v_i (\bm A \bm w)_i g_{\bm u}  g_{\bm B}|   \\
    &\qquad + |v_i (\bm B \bm w)_i g_{\bm u} g_{\bm A}| + |(\bm A \bm w)_i (\bm B \bm w)_i g_{\bm u} g_{\bm v}| + |\bm A_{ii} g_{\bm u} g_{\bm v}  g_{\bm B}| + |\bm B_{ii} g_{\bm u} g_{\bm v} g_{\bm A}|  \\
    &\qquad + |f_7' \cdot (1 + \bm w^\top \bm C \bm w)(\bm C \bm w)_i| + |f_7 \cdot (\bm C \bm w)_i^2| + |f_7 \cdot (1 + \bm w^\top \bm C \bm w) \bm C_{ii}| \\
    &\lcon |u_i v_i g_{\bm A} g_{\bm B}| + |u_i (\bm A \bm w)_i g_{\bm v} g_{\bm B}| + |u_i (\bm B \bm w)_i  g_{\bm v} g_{\bm A}|  + |v_i (\bm A \bm w)_i g_{\bm u}  g_{\bm B}| \\
    &\qquad + |v_i (\bm B \bm w)_i g_{\bm u} g_{\bm A}| + |(\bm A \bm w)_i (\bm B \bm w)_i g_{\bm u} g_{\bm v}| + |\bm A_{ii} g_{\bm u} g_{\bm v}  g_{\bm B}| + |\bm B_{ii} g_{\bm u} g_{\bm v} g_{\bm A}| \\
    &\qquad + (|u_i g_{\bm v} g_{\bm A} g_{\bm B} \cdot (1 + \bm w^\top \bm C \bm w)(\bm C \bm w)_i| + |v_i g_{\bm u}  g_{\bm A} g_{\bm B} \cdot (1 + \bm w^\top \bm C \bm w)(\bm C \bm w)_i| \\
    &\qquad \quad + |(\bm A \bm w)_ig_{\bm u} g_{\bm v} g_{\bm B} \cdot (1 + \bm w^\top \bm C \bm w)(\bm C \bm w)_i| \\
    &\qquad \quad + |(\bm B \bm w)_ig_{\bm u} g_{\bm v} g_{\bm A} \cdot (1 + \bm w^\top \bm C \bm w)(\bm C \bm w)_i| \\
    &\qquad \quad + |f_7 \cdot (1 + \bm w^\top \bm C \bm w)^2 (\bm C \bm w)_i^2|) \\
    &\qquad + |f_7 \cdot (\bm C \bm w)_i^2| + |f_7 \cdot (1 + \bm w^\top \bm C \bm w) \bm C_{ii}|
\end{align*}
and thus finally
\begin{align*}
    &\biggl|\partial_i^3 f_7(\bm W_{t,u,i})\biggr| \\
    &\lcon |u_i \bm A_{ii} g_{\bm v} g_{\bm B}| + |u_i \bm B_{ii} g_{\bm v} g_{\bm A}|  + |v_i \bm A_{ii} g_{\bm u} g_{\bm B}| + |v_i \bm B_{ii} g_{\bm u} g_{\bm A}|  +  |\bm A_{ii} (\bm B \bm w)_i g_{\bm u} g_{\bm v}|  \\
    &\qquad + |\bm B_{ii} (\bm A \bm w)_i g_{\bm u} g_{\bm v}| + |u_i v_i (\bm A \bm w)_i g_{\bm B}| + |u_i v_i (\bm B \bm w)_i  g_{\bm A}| + |u_i   (\bm A \bm w)_i (\bm B \bm w)_i g_{\bm v}|  \\
    &\qquad + | v_i (\bm A \bm w)_i (\bm B \bm w)_i g_{\bm u}| +| f_7'' \cdot (1 + \bm w^\top \bm C \bm w)(\bm C \bm w)_i| + |f_7' \cdot (\bm C \bm w)_i^2| \\
    &\qquad + |f_7' \cdot (1 + \bm w^\top \bm C \bm w)\bm C_{ii}| + |f_7 \cdot (\bm C \bm w)_i \bm C_{ii}| \\
    &\lcon |u_i \bm A_{ii} g_{\bm v} g_{\bm B}| + |u_i \bm B_{ii} g_{\bm v} g_{\bm A}|  + |v_i \bm A_{ii} g_{\bm u} g_{\bm B}| + |v_i \bm B_{ii} g_{\bm u} g_{\bm A}|  +  |\bm A_{ii} (\bm B \bm w)_i g_{\bm u} g_{\bm v}|  \\
    &\qquad + |\bm B_{ii} (\bm A \bm w)_i g_{\bm u} g_{\bm v}|+ |u_i v_i (\bm A \bm w)_i g_{\bm B}| + |u_i v_i (\bm B \bm w)_i  g_{\bm A}| + |u_i   (\bm A \bm w)_i (\bm B \bm w)_i g_{\bm v}|  \\
    &\qquad + | v_i (\bm A \bm w)_i (\bm B \bm w)_i g_{\bm u}| \\
    &\qquad + [|u_i v_i g_{\bm A} g_{\bm B}(1 + \bm w^\top \bm C \bm w)(\bm C \bm w)_i| + |u_i (\bm A \bm w)_i g_{\bm v} g_{\bm B}(1 + \bm w^\top \bm C \bm w)(\bm C \bm w)_i| \\
    &\qquad \quad + |u_i (\bm B \bm w)_i  g_{\bm v} g_{\bm A}(1 + \bm w^\top \bm C \bm w)(\bm C \bm w)_i| \\
    &\qquad \quad + |v_i (\bm B \bm w)_i g_{\bm u} g_{\bm A} (1 + \bm w^\top \bm C \bm w)(\bm C \bm w)_i| \\
    &\qquad \quad + |(\bm A \bm w)_i (\bm B \bm w)_i g_{\bm u} g_{\bm v} (1 + \bm w^\top \bm C \bm w)(\bm C \bm w)_i| \\
    &\qquad \quad + |\bm A_{ii} g_{\bm u} g_{\bm v}  g_{\bm B} (1 + \bm w^\top \bm C \bm w)(\bm C \bm w)_i| + |\bm B_{ii} g_{\bm u} g_{\bm v} g_{\bm A} (1 + \bm w^\top \bm C \bm w)(\bm C \bm w)_i| \\
    &\qquad \quad  + (|u_i g_{\bm v} g_{\bm A} g_{\bm B} \cdot (1 + \bm w^\top \bm C \bm w)^2(\bm C \bm w)_i^2| + |v_i g_{\bm u}  g_{\bm A} g_{\bm B} \cdot (1 + \bm w^\top \bm C \bm w)^2(\bm C \bm w)_i^2| \\
    &\qquad \qquad + |(\bm A \bm w)_ig_{\bm u} g_{\bm v} g_{\bm B} \cdot (1 + \bm w^\top \bm C \bm w)^2(\bm C \bm w)_i^2| \\
    &\qquad \qquad + |(\bm B \bm w)_ig_{\bm u} g_{\bm v} g_{\bm A} \cdot (1 + \bm w^\top \bm C \bm w)^2(\bm C \bm w)_i^2| \\
    &\qquad \qquad + |f_7 \cdot (1 + \bm w^\top \bm C \bm w)^3 (\bm C \bm w)_i^3|) \\
    &\qquad \quad + |f_7 \cdot (\bm C \bm w)_i^3 (1 + \bm w^\top \bm C \bm w)| + |f_7 \cdot (1 + \bm w^\top \bm C \bm w)^2 \bm C_{ii} (\bm C \bm w)_i|] \\
    &\qquad + (|u_i g_{\bm v} g_{\bm A} g_{\bm B} \cdot (\bm C \bm w)_i^2| + |v_i g_{\bm u}  g_{\bm A} g_{\bm B} \cdot (\bm C \bm w)_i^2| + |(\bm A \bm w)_ig_{\bm u} g_{\bm v} g_{\bm B} \cdot (\bm C \bm w)_i^2| \\
    &\qquad \quad + |(\bm B \bm w)_ig_{\bm u} g_{\bm v} g_{\bm A} \cdot (\bm C \bm w)_i^2| + |f_7 \cdot (1 + \bm w^\top \bm C \bm w) (\bm C \bm w)_i^3|) \\
    &\qquad + (|u_i g_{\bm v} g_{\bm A} g_{\bm B} \cdot (1 + \bm w^\top \bm C \bm w) \bm C_{ii}| + |v_i g_{\bm u}  g_{\bm A} g_{\bm B} \cdot (1 + \bm w^\top \bm C \bm w) \bm C_{ii}| \\
    &\qquad \quad + |(\bm A \bm w)_ig_{\bm u} g_{\bm v} g_{\bm B} \cdot (1 + \bm w^\top \bm C \bm w) \bm C_{ii}| \\
    &\qquad \quad + |(\bm B \bm w)_ig_{\bm u} g_{\bm v} g_{\bm A} \cdot (1 + \bm w^\top \bm C \bm w) \bm C_{ii}| + |f_7 \cdot (1 + \bm w^\top \bm C \bm w)^2 (\bm C \bm w)_i \bm C_{ii}|) \\
    &\qquad + |f_7 \cdot (\bm C \bm w)_i \bm C_{ii}| \\
    &\lcon | \bm A_{ii} g_{\bm v} g_{\bm B} \cdot u_i| + | \bm B_{ii} g_{\bm v} g_{\bm A} \cdot u_i|  + | \bm A_{ii} g_{\bm u} g_{\bm B} \cdot v_i| + | \bm B_{ii} g_{\bm u} g_{\bm A} \cdot v_i|  +  |\bm A_{ii} \cdot g_{\bm u} g_{\bm v}  (\bm B \bm w)_i| \\
    &\qquad + |\bm B_{ii}  g_{\bm u} g_{\bm v} \cdot (\bm A \bm w)_i| + | g_{\bm B} \cdot u_i v_i (\bm A \bm w)_i| + | g_{\bm A} \cdot u_i v_i (\bm B \bm w)_i | + | g_{\bm v} \cdot u_i   (\bm A \bm w)_i (\bm B \bm w)_i| \\
    &\qquad + |  g_{\bm u} \cdot v_i (\bm A \bm w)_i (\bm B \bm w)_i| \\
    &\qquad + [|g_{\bm A} g_{\bm B}g_{1, \bm C}\cdot u_i v_i (\bm C \bm w)_i| + | g_{\bm v} g_{\bm B}g_{1, \bm C}\cdot u_i (\bm A \bm w)_i(\bm C \bm w)_i| \\
    &\qquad \quad + |  g_{\bm v} g_{\bm A}g_{1, \bm C}\cdot u_i (\bm B \bm w)_i(\bm C \bm w)_i| \\
    &\qquad \quad + | g_{\bm u} g_{\bm A} g_{1, \bm C}\cdot v_i (\bm B \bm w)_i(\bm C \bm w)_i| + |g_{\bm u} g_{\bm v} g_{1, \bm C}\cdot (\bm A \bm w)_i (\bm B \bm w)_i (\bm C \bm w)_i| \\
    &\qquad \quad + |\bm A_{ii} g_{\bm u} g_{\bm v}  g_{\bm B} g_{1, \bm C} \cdot (\bm C \bm w)_i| + |\bm B_{ii} g_{\bm u} g_{\bm v} g_{\bm A}g_{1, \bm C} \cdot  (\bm C \bm w)_i| \\
    &\qquad \quad  + (|g_{\bm v} g_{\bm A} g_{\bm B}  g_{1,\bm C}^2 \cdot u_i (\bm C \bm w)_i^2| + | g_{\bm u}  g_{\bm A} g_{\bm B}  g_{1,\bm C}^2 \cdot v_i (\bm C \bm w)_i^2| \\
    &\qquad \qquad + |g_{\bm u} g_{\bm v} g_{\bm B} g_{\bm C}^2 \cdot (\bm A \bm w)_i(\bm C \bm w)_i^2| + |g_{\bm u} g_{\bm v} g_{\bm A}  g_{1,\bm C}^2\cdot (\bm B \bm w)_i(\bm C \bm w)_i^2| \\
    &\qquad \qquad + |g_{\bm u} g_{\bm v} g_{\bm A} g_{\bm B} g_{1,\bm C}^3 \cdot (\bm C \bm w)_i^3|) \\
    &\qquad \quad + |g_{\bm u} g_{\bm v} g_{\bm A} g_{\bm B} g_{1,\bm C}\cdot (\bm C \bm w)_i^3 | + |\bm C_{ii} g_{\bm u} g_{\bm v} g_{\bm A} g_{\bm B} g_{1,\bm C}^2 \cdot  (\bm C \bm w)_i|] \\
    &\qquad + (|g_{\bm v} g_{\bm A} g_{\bm B} \cdot u_i (\bm C \bm w)_i^2| + |g_{\bm u}  g_{\bm A} g_{\bm B} \cdot v_i (\bm C \bm w)_i^2| + |g_{\bm u} g_{\bm v} g_{\bm B} \cdot(\bm A \bm w)_i (\bm C \bm w)_i^2| \\
    &\qquad \quad + |g_{\bm u} g_{\bm v} g_{\bm A} \cdot (\bm B \bm w)_i(\bm C \bm w)_i^2| + |g_{\bm u} g_{\bm v} g_{\bm A} g_{\bm B} g_{1,\bm C} \cdot (\bm C \bm w)_i^3|) \\
    &\qquad + (|\bm C_{ii}g_{\bm v} g_{\bm A} g_{\bm B} g_{1,\bm C} \cdot u_i  | + |\bm C_{ii} g_{\bm u}  g_{\bm A} g_{\bm B} g_{1,\bm C} \cdot   v_i| \\
    &\qquad \quad + |\bm C_{ii} g_{\bm u} g_{\bm v} g_{\bm B}g_{1,\bm C} \cdot  (\bm A \bm w)_i| \\
    &\qquad \quad + |\bm C_{ii} g_{\bm u} g_{\bm v} g_{\bm A}  g_{1,\bm C} \cdot  (\bm B \bm w)_i| + |\bm C_{ii}g_{\bm u} g_{\bm v} g_{\bm A} g_{\bm B} g_{1,\bm C}^2 \cdot  (\bm C \bm w)_i|) \\
    &\qquad + | \bm C_{ii} g_{\bm u} g_{\bm v} g_{\bm A} g_{\bm B} \cdot (\bm C \bm w)_i| 
\end{align*}
We now bound each of the corresponding terms, as we did earlier.

\subsubsection{First bound for \texorpdfstring{$f_8$}{f8}}
We can express
\begin{equation*}
    f_8(\bm w; \bm S) = \frac{h_{\bm u}(\bm w) g_{\bm A}(\bm w) g_{\bm B}(\bm w) g_{\bm C}(\bm w)}{h_{1,\bm D}(\bm w)}.
\end{equation*}
Noting that $h_{1,\bm D} \neq 0$, we can write $f_8(\bm w) h_{1,\bm D}(\bm w) = h_{\bm u}(\bm w) g_{\bm A}(\bm w) g_{\bm B}(\bm w) g_{\bm C}(\bm w)$ and differentiating each side yields
\begin{equation*}
    f_8' h_{1,\bm D} + f_8 h_{1,\bm D}' = h_{\bm u}' g_{\bm A} g_{\bm B} g_{\bm C} + h_{\bm u} g_{\bm A}' g_{\bm B} g_{\bm C} + h_{\bm u} g_{\bm A} g_{\bm B}' g_{\bm C} + h_{\bm u} g_{\bm A} g_{\bm B} g_{\bm C}'
\end{equation*}
where $f_8' = \partial_i f_8$, etc. This yields, as usual, 
\begin{align*} 
f_8' &= \frac{h_{\bm u}' g_{\bm A} g_{\bm B} g_{\bm C} + h_{\bm u} g_{\bm A}' g_{\bm B} g_{\bm C} + h_{\bm u} g_{\bm A} g_{\bm B}' g_{\bm C} + h_{\bm u} g_{\bm A} g_{\bm B} g_{\bm C}' - f_8 h_{1, \bm D}'}{h_{1, \bm D}} \\
&= \frac{2 u_i (\bm u^\top \bm w) g_{\bm A} g_{\bm B} g_{\bm C} + 2 (\bm A \bm w)_i h_{\bm u} g_{\bm B} g_{\bm C} + 2 (\bm B \bm w)_i h_{\bm u} g_{\bm A} g_{\bm C}}{h_{1, \bm D}} \\
&\qquad + \frac{ 2 (\bm C \bm w)_ih_{\bm u} g_{\bm A} g_{\bm B} - 4 f_8 \cdot (1 + \bm w^\top \bm D \bm w)(\bm D \bm w)_i}{h_{1, \bm D}}
\end{align*} 
Differentiating each side again yields
\begin{align*}
    f_8'' h_{1, \bm D} + 2f_8' h_{1, \bm D}' + f_8 h_{1, \bm D}'' &= h_{\bm u}'' g_{\bm A} g_{\bm B} g_{\bm C} + h_{\bm u} g_{\bm A}'' g_{\bm B} g_{\bm C} + h_{\bm u} g_{\bm A} g_{\bm B}'' g_{\bm C} + h_{\bm u} g_{\bm A} g_{\bm B} g_{\bm C}'' \\
    &\qquad + 2( h_{\bm u}' g_{\bm A}' g_{\bm B} g_{\bm C} + h_{\bm u}' g_{\bm A} g_{\bm B}' g_{\bm C} + h_{\bm u}' g_{\bm A} g_{\bm B} g_{\bm C}' \\
    &\qquad \qquad + h_{\bm u} g_{\bm A}' g_{\bm B}' g_{\bm C} + h_{\bm u} g_{\bm A}' g_{\bm B} g_{\bm C}' + h_{\bm u} g_{\bm A} g_{\bm B}' g_{\bm C}') 
\end{align*}
and rearranging yields
\begin{align*}
    f_{8}'' &= \frac{2( h_{\bm u}' g_{\bm A}' g_{\bm B} g_{\bm C} + h_{\bm u}' g_{\bm A} g_{\bm B}' g_{\bm C} + h_{\bm u}' g_{\bm A} g_{\bm B} g_{\bm C}' ) }{h_{1, \bm D}} \\
    &\qquad + \frac{2(  h_{\bm u} g_{\bm A}' g_{\bm B}' g_{\bm C} + h_{\bm u} g_{\bm A}' g_{\bm B} g_{\bm C}' + h_{\bm u} g_{\bm A} g_{\bm B}' g_{\bm C}') }{h_{1, \bm D}} \\
    &\qquad + \frac{h_{\bm u}'' g_{\bm A} g_{\bm B} g_{\bm C} + h_{\bm u} g_{\bm A}'' g_{\bm B} g_{\bm C} + h_{\bm u} g_{\bm A} g_{\bm B}'' g_{\bm C} + h_{\bm u} g_{\bm A} g_{\bm B} g_{\bm C}'' - 2f_8' h_{1, \bm D}' - f_8 h_{1,\bm D}''}{h_{1, \bm D}} \\
    &= \frac{8( u_i (\bm u^\top \bm w) (\bm A \bm w)_i g_{\bm B} g_{\bm C} + u_i (\bm u^\top \bm w) (\bm B\bm w)_i g_{\bm A} g_{\bm C} + u_i (\bm u^\top \bm w) (\bm C \bm w)_i g_{\bm A} g_{\bm B})}{h_{1, \bm D}} \\
    &\qquad + \frac{8((\bm A \bm w)_i(\bm B\bm w)_ih_{\bm u} g_{\bm C} + (\bm A \bm w)_i (\bm D \bm w)_i h_{\bm u}  g_{\bm B}  + (\bm B \bm w)_i (\bm C \bm w)_i h_{\bm u} g_{\bm A} ) }{h_{1, \bm D}} \\
    &\qquad + \frac{2u_i^2 g_{\bm A} g_{\bm B} g_{\bm C} + 2 \bm A_{ii} h_{\bm u} g_{\bm B} g_{\bm C} + 2 \bm B_{ii} h_{\bm u} g_{\bm A} g_{\bm C} + 2 \bm C_{ii} h_{\bm u} g_{\bm A} g_{\bm B}}{h_{1, \bm D}} \\
    &\qquad - \frac{8f_8' \cdot (1 + \bm w^\top \bm D \bm w)(\bm D \bm w)_i + 8 f_8 \cdot (\bm D \bm w)_i^2 + 4 f_8 \cdot (1 + \bm w^\top \bm D \bm w) \bm D_{ii}}{h_{1, \bm D}}
\end{align*}
Again, we can differentiate each side again to yield
\begin{align*}
    &f_8''' h_{1, \bm D} + 3 f_8'' h_{1, \bm D}' + 3 f_8' h_{1, \bm D}'' + f_8 h_{1, \bm D}''' \\&= h_{\bm u}''' g_{\bm A} g_{\bm B} g_{\bm C} + h_{\bm u} g_{\bm A}''' g_{\bm B} g_{\bm C} + h_{\bm u} g_{\bm A} g_{\bm B}''' g_{\bm C} + h_{\bm u} g_{\bm A} g_{\bm B} g_{\bm C}''' \\
    &\qquad + 3(h_{\bm u}'' g_{\bm A}' g_{\bm B} g_{\bm C} + h_{\bm u}' g_{\bm A}'' g_{\bm B} g_{\bm C} + h_{\bm u}'' g_{\bm A} g_{\bm B}' g_{\bm C} + h_{\bm u}' g_{\bm A} g_{\bm B}'' g_{\bm C} \\
    &\qquad \qquad + h_{\bm u}'' g_{\bm A} g_{\bm B} g_{\bm C}' + h_{\bm u}' g_{\bm A} g_{\bm B} g_{\bm C}'' + h_{\bm u} g_{\bm A}'' g_{\bm B}' g_{\bm C} + h_{\bm u} g_{\bm A}' g_{\bm B}'' g_{\bm C} \\
    &\qquad \qquad + h_{\bm u} g_{\bm A}'' g_{\bm B} g_{\bm C}' + h_{\bm u} g_{\bm A}' g_{\bm B} g_{\bm C}''+  h_{\bm u} g_{\bm A} g_{\bm B}'' g_{\bm C}' + h_{\bm u} g_{\bm A} g_{\bm B}' g_{\bm C}'') \\
    &\qquad + 6 (h_{\bm u} g_{\bm A}' g_{\bm B}' g_{\bm C}' + h_{\bm u}' g_{\bm A} g_{\bm B}' g_{\bm C}' + h_{\bm u}' g_{\bm A}' g_{\bm B} g_{\bm C}' + h_{\bm u}' g_{\bm A}' g_{\bm B}' g_{\bm C}) \\
    &= 3(h_{\bm u}'' g_{\bm A}' g_{\bm B} g_{\bm C} + h_{\bm u}' g_{\bm A}'' g_{\bm B} g_{\bm C} + h_{\bm u}'' g_{\bm A} g_{\bm B}' g_{\bm C} + h_{\bm u}' g_{\bm A} g_{\bm B}'' g_{\bm C} \\
    &\qquad \qquad + h_{\bm u}'' g_{\bm A} g_{\bm B} g_{\bm C}' + h_{\bm u}' g_{\bm A} g_{\bm B} g_{\bm C}'' + h_{\bm u} g_{\bm A}'' g_{\bm B}' g_{\bm C} + h_{\bm u} g_{\bm A}' g_{\bm B}'' g_{\bm C} \\
    &\qquad \qquad + h_{\bm u} g_{\bm A}'' g_{\bm B} g_{\bm C}' + h_{\bm u} g_{\bm A}' g_{\bm B} g_{\bm C}''+  h_{\bm u} g_{\bm A} g_{\bm B}'' g_{\bm C}' + h_{\bm u} g_{\bm A} g_{\bm B}' g_{\bm C}'') \\
    &\qquad + 6 (h_{\bm u} g_{\bm A}' g_{\bm B}' g_{\bm C}' + h_{\bm u}' g_{\bm A} g_{\bm B}' g_{\bm C}' + h_{\bm u}' g_{\bm A}' g_{\bm B} g_{\bm C}' + h_{\bm u}' g_{\bm A}' g_{\bm B}' g_{\bm C})
\end{align*}
which yields
\begin{align*}
    f_8''' &= \frac{3(h_{\bm u}'' g_{\bm A}' g_{\bm B} g_{\bm C} + h_{\bm u}' g_{\bm A}'' g_{\bm B} g_{\bm C} + h_{\bm u}'' g_{\bm A} g_{\bm B}' g_{\bm C} + h_{\bm u}' g_{\bm A} g_{\bm B}'' g_{\bm C})}{h_{1, \bm D}} \\
    &\qquad + \frac{3(h_{\bm u}'' g_{\bm A} g_{\bm B} g_{\bm C}' + h_{\bm u}' g_{\bm A} g_{\bm B} g_{\bm C}'' + h_{\bm u} g_{\bm A}'' g_{\bm B}' g_{\bm C} + h_{\bm u} g_{\bm A}' g_{\bm B}'' g_{\bm C})}{h_{1, \bm D}} \\
    &\qquad + \frac{3(h_{\bm u} g_{\bm A}'' g_{\bm B} g_{\bm C}' + h_{\bm u} g_{\bm A}' g_{\bm B} g_{\bm C}''+  h_{\bm u} g_{\bm A} g_{\bm B}'' g_{\bm C}' + h_{\bm u} g_{\bm A} g_{\bm B}' g_{\bm C}'')}{h_{1,\bm D}} \\
    &\qquad + \frac{6 (h_{\bm u} g_{\bm A}' g_{\bm B}' g_{\bm C}' + h_{\bm u}' g_{\bm A} g_{\bm B}' g_{\bm C}' + h_{\bm u}' g_{\bm A}' g_{\bm B} g_{\bm C}' + h_{\bm u}' g_{\bm A}' g_{\bm B}' g_{\bm C})}{h_{1, \bm D}} \\
    &\qquad - \frac{3f_8'' h_{1,\bm D}' + 3f_8' h_{1, \bm D}'' + f_8 h_{1, \bm D}'''}{h_{1, \bm D}} \\
    &= \frac{12(u_i^2 (\bm A \bm w)_i g_{\bm B} g_{\bm C} + u_i (\bm u^\top \bm w) \bm A_{ii} g_{\bm B} g_{\bm C} + u_i^2 (\bm B \bm w)_i g_{\bm A} g_{\bm C} + u_i (\bm u^\top \bm w) \bm B_{ii} g_{\bm A} g_{\bm C})}{h_{1, \bm D}} \\
    &\qquad + \frac{12(u_i^2 (\bm C \bm w)_i g_{\bm A} g_{\bm B} + u_i (\bm u^\top \bm w) \bm C_{ii} g_{\bm A} g_{\bm B} + \bm A_{ii} (\bm B \bm w)_i h_{\bm u} g_{\bm C} + \bm B_{ii} (\bm A \bm w)_i h_{\bm u} g_{\bm C})}{h_{1, \bm D}} \\
    &\qquad + \frac{12( \bm A_{ii} (\bm C \bm w)_i h_{\bm u}  g_{\bm B}  + \bm C_{ii} (\bm A \bm w)_i h_{\bm u} g_{\bm B} +  \bm B_{ii} (\bm C \bm w)_i h_{\bm u} g_{\bm A} + \bm C_{ii} (\bm B \bm w)_i h_{\bm u} g_{\bm A})}{h_{1,\bm D}} \\
    &\qquad + \frac{48 ( (\bm A \bm w)_i (\bm B \bm w)_i (\bm C \bm w)_i h_{\bm u} + u_i (\bm u^\top \bm w) (\bm B \bm w)_i (\bm C \bm w)_i g_{\bm A})}{h_{1, \bm D}} \\
    &\qquad + \frac{48(u_i (\bm u^\top \bm w) (\bm A \bm w)_i (\bm C \bm w)_i  g_{\bm B} + u_i (\bm u^\top \bm w)(\bm A \bm w)_i (\bm B \bm w)_i  g_{\bm C})}{h_{1, \bm D}} \\
    &\qquad - \frac{12f_8'' \cdot (1 + \bm w^\top \bm D \bm w)(\bm D \bm w)_i + 24f_8' \cdot (\bm D \bm w)_i^2}{h_{1, \bm D}}  \\
    &\qquad + \frac{12f_8' \cdot (1 + \bm w^\top \bm D \bm w) \bm D_{ii} + 24 f_8 \cdot (\bm D \bm w)_i \bm D_{ii}}{h_{1, \bm D}} 
\end{align*}
For each $i$, noting that $h_{1,\bm D} \geq 1$, we have 
\begin{align*}
    |f_8'| &\lcon |u_i (\bm u^\top \bm w) g_{\bm A} g_{\bm B} g_{\bm C}| + |(\bm A \bm w)_i h_{\bm u} g_{\bm B} g_{\bm C}| +  |(\bm B \bm w)_i h_{\bm u} g_{\bm A} g_{\bm C}| \\
    &\qquad +  |(\bm C \bm w)_ih_{\bm u} g_{\bm A} g_{\bm B}| +  |f_8 \cdot (1 + \bm w^\top \bm D \bm w)(\bm D \bm w)_i|
\end{align*}
and
\begin{align*}
    |f_8''| &\lcon |u_i (\bm u^\top \bm w) (\bm A \bm w)_i g_{\bm B} g_{\bm C}| + |u_i (\bm u^\top \bm w) (\bm B\bm w)_i g_{\bm A} g_{\bm C}| + |u_i (\bm u^\top \bm w) (\bm C \bm w)_i g_{\bm A} g_{\bm B}| \\
    &\qquad + |(\bm A \bm w)_i(\bm B\bm w)_ih_{\bm u} g_{\bm C}| + |(\bm A \bm w)_i (\bm C \bm w)_i h_{\bm u}  g_{\bm B}|  + |(\bm B \bm w)_i (\bm C \bm w)_i h_{\bm u} g_{\bm A}| \\
    &\qquad + |u_i^2 g_{\bm A} g_{\bm B} g_{\bm C}| + |\bm A_{ii} h_{\bm u} g_{\bm B} g_{\bm C}| +  |\bm B_{ii} h_{\bm u} g_{\bm A} g_{\bm C}| +  |\bm C_{ii} h_{\bm u} g_{\bm A} g_{\bm B}| \\
    &\qquad + |f_8' \cdot (1 + \bm w^\top \bm D \bm w)(\bm D \bm w)_i| + |f_8 \cdot (\bm D \bm w)_i^2| + |f_8 \cdot (1 + \bm w^\top \bm D \bm w) \bm D_{ii}| \\
    &\lcon |u_i (\bm u^\top \bm w) (\bm A \bm w)_i g_{\bm B} g_{\bm C}| + |u_i (\bm u^\top \bm w) (\bm B\bm w)_i g_{\bm A} g_{\bm C}| + |u_i (\bm u^\top \bm w) (\bm C \bm w)_i g_{\bm A} g_{\bm B}| \\
    &\qquad + |(\bm A \bm w)_i(\bm B\bm w)_ih_{\bm u} g_{\bm C}| + |(\bm A \bm w)_i (\bm C \bm w)_i h_{\bm u}  g_{\bm B}|  + |(\bm B \bm w)_i (\bm C \bm w)_i h_{\bm u} g_{\bm A}| \\
    &\qquad + |u_i^2 g_{\bm A} g_{\bm B} g_{\bm C}| + |\bm A_{ii} h_{\bm u} g_{\bm B} g_{\bm C}| +  |\bm B_{ii} h_{\bm u} g_{\bm A} g_{\bm C}| +  |\bm C_{ii} h_{\bm u} g_{\bm A} g_{\bm B}| \\
    &\qquad + (|u_i (\bm u^\top \bm w) g_{\bm A} g_{\bm B} g_{\bm C} (1 + \bm w^\top \bm D \bm w)(\bm D \bm w)_i| \\
    &\qquad \quad + |(\bm A \bm w)_i h_{\bm u} g_{\bm B} g_{\bm C}(1 + \bm w^\top \bm D \bm w)(\bm D \bm w)_i| \\
    &\qquad \quad +  |(\bm B \bm w)_i h_{\bm u} g_{\bm A} g_{\bm C} (1 + \bm w^\top \bm D \bm w)(\bm D \bm w)_i| \\
    &\qquad \quad + |(\bm C \bm w)_ih_{\bm u} g_{\bm A} g_{\bm B} (1 + \bm w^\top \bm D \bm w)(\bm D \bm w)_i| \\
    &\qquad \quad  +  |f_8 \cdot (1 + \bm w^\top \bm D \bm w)^2(\bm D \bm w)_i^2|) \\
    &\qquad + |f_8 \cdot (\bm D \bm w)_i^2| + |f_8 \cdot (1 + \bm w^\top \bm D \bm w) \bm D_{ii}|
\end{align*}
and thus finally
\begin{align*}
    &\biggl|\partial_i^3 f_8(\bm w)\biggr | \\
    &\lcon |u_i^2 (\bm A \bm w)_i g_{\bm B} g_{\bm C}| + |u_i (\bm u^\top \bm w) \bm A_{ii} g_{\bm B} g_{\bm C}| + |u_i^2 (\bm B \bm w)_i g_{\bm A} g_{\bm C}| + |u_i (\bm u^\top \bm w) \bm B_{ii} g_{\bm A} g_{\bm C}| \\
    &\qquad + |u_i^2 (\bm C \bm w)_i g_{\bm A} g_{\bm B}| + |u_i (\bm u^\top \bm w) \bm C_{ii} g_{\bm A} g_{\bm B}| + |\bm A_{ii} (\bm B \bm w)_i h_{\bm u} g_{\bm C}| + |\bm B_{ii} (\bm A \bm w)_i h_{\bm u} g_{\bm C}| \\
    &\qquad + |\bm A_{ii} (\bm C \bm w)_i h_{\bm u}  g_{\bm B}|  + |\bm C_{ii} (\bm A \bm w)_i h_{\bm u} g_{\bm B}| + |\bm B_{ii} (\bm C \bm w)_i h_{\bm u} g_{\bm A}| + |\bm C_{ii} (\bm B \bm w)_i h_{\bm u} g_{\bm A}| \\
    &\qquad + |(\bm A \bm w)_i (\bm B \bm w)_i (\bm C \bm w)_i h_{\bm u}| + |u_i (\bm u^\top \bm w) (\bm B \bm w)_i (\bm C \bm w)_i g_{\bm A}| \\
    &\qquad + |u_i (\bm u^\top \bm w) (\bm A \bm w)_i (\bm C \bm w)_i  g_{\bm B}| + |u_i (\bm u^\top \bm w)(\bm A \bm w)_i (\bm B \bm w)_i  g_{\bm C}| \\
    &\qquad + |f_8'' \cdot (1 + \bm w^\top \bm D \bm w)(\bm D \bm w)_i| + |f_8' \cdot (\bm D \bm w)_i^2| + |f_8' \cdot (1 + \bm w^\top \bm D \bm w) \bm D_{ii}|\\
    &\qquad + |f_8 \cdot (\bm D \bm w)_i \bm D_{ii}| \\
    &\lcon |u_i^2 (\bm A \bm w)_i g_{\bm B} g_{\bm C}| + |u_i (\bm u^\top \bm w) \bm A_{ii} g_{\bm B} g_{\bm C}| + |u_i^2 (\bm B \bm w)_i g_{\bm A} g_{\bm C}| + |u_i (\bm u^\top \bm w) \bm B_{ii} g_{\bm A} g_{\bm C}| \\
    &\qquad + |u_i^2 (\bm C \bm w)_i g_{\bm A} g_{\bm B}| + |u_i (\bm u^\top \bm w) \bm C_{ii} g_{\bm A} g_{\bm B}| + |\bm A_{ii} (\bm B \bm w)_i h_{\bm u} g_{\bm C}| + |\bm B_{ii} (\bm A \bm w)_i h_{\bm u} g_{\bm C}| \\
    &\qquad + |\bm A_{ii} (\bm C \bm w)_i h_{\bm u}  g_{\bm B}|  + |\bm C_{ii} (\bm A \bm w)_i h_{\bm u} g_{\bm B}| + |\bm B_{ii} (\bm C \bm w)_i h_{\bm u} g_{\bm A}| + |\bm C_{ii} (\bm B \bm w)_i h_{\bm u} g_{\bm A}| \\
    &\qquad + |(\bm A \bm w)_i (\bm B \bm w)_i (\bm C \bm w)_i h_{\bm u}| + |u_i (\bm u^\top \bm w) (\bm B \bm w)_i (\bm C \bm w)_i g_{\bm A}| \\
    &\qquad + |u_i (\bm u^\top \bm w) (\bm A \bm w)_i (\bm C \bm w)_i  g_{\bm B}| + |u_i (\bm u^\top \bm w)(\bm A \bm w)_i (\bm B \bm w)_i  g_{\bm C}| \\
    &\qquad + [|u_i (\bm u^\top \bm w) (\bm A \bm w)_i g_{\bm B} g_{\bm C} (1 + \bm w^\top \bm D \bm w)(\bm D \bm w)_i| \\
    &\qquad \quad + |u_i (\bm u^\top \bm w) (\bm B\bm w)_i g_{\bm A} g_{\bm C} (1 + \bm w^\top \bm D \bm w)(\bm D \bm w)_i| \\
    &\qquad \quad + |u_i (\bm u^\top \bm w) (\bm C \bm w)_i g_{\bm A} g_{\bm B }(1 + \bm w^\top \bm D \bm w)(\bm D \bm w)_i| \\
    &\qquad \quad + |(\bm A \bm w)_i(\bm B\bm w)_ih_{\bm u} g_{\bm C} (1 + \bm w^\top \bm D \bm w)(\bm D \bm w)_i| \\
    &\qquad \quad + |(\bm A \bm w)_i (\bm C \bm w)_i h_{\bm u}  g_{\bm B} (1 + \bm w^\top \bm D \bm w)(\bm D \bm w)_i|  \\
    &\qquad \quad + |(\bm B \bm w)_i (\bm C \bm w)_i h_{\bm u} g_{\bm A} (1 + \bm w^\top \bm D \bm w)(\bm D \bm w)_i| \\
    &\qquad \quad + |u_i^2 g_{\bm A} g_{\bm B} g_{\bm C} (1 + \bm w^\top \bm D \bm w)(\bm D \bm w)_i| + |\bm A_{ii} h_{\bm u} g_{\bm B} g_{\bm C} (1 + \bm w^\top \bm D \bm w)(\bm D \bm w)_i| \\
    &\qquad \quad +  |\bm B_{ii} h_{\bm u} g_{\bm A} g_{\bm C} (1 + \bm w^\top \bm D \bm w)(\bm D \bm w)_i| +  |\bm C_{ii} h_{\bm u} g_{\bm A} g_{\bm B} (1 + \bm w^\top \bm D \bm w)(\bm D \bm w)_i| \\
    &\qquad \quad  + (|u_i (\bm u^\top \bm w) g_{\bm A} g_{\bm B} g_{\bm C} (1 + \bm w^\top \bm D \bm w)^2(\bm D \bm w)_i^2| \\
    &\qquad \qquad + |(\bm A \bm w)_i h_{\bm u} g_{\bm B} g_{\bm C}(1 + \bm w^\top \bm D \bm w)^2(\bm D \bm w)_i^2| \\
    &\qquad \qquad +  |(\bm B \bm w)_i h_{\bm u} g_{\bm A} g_{\bm C} (1 + \bm w^\top \bm D \bm w)^2(\bm D \bm w)_i^2| \\
    &\qquad \qquad + |(\bm C \bm w)_i h_{\bm u} g_{\bm A} g_{\bm B} (1 + \bm w^\top \bm D \bm w)^2(\bm D \bm w)_i^2| \\
    &\qquad \qquad  +  |f_8 \cdot (1 + \bm w^\top \bm D \bm w)^3(\bm D \bm w)_i^3|) \\
    &\qquad \quad  + |f_8 \cdot (\bm D \bm w)_i^3 (1 + \bm w^\top \bm D \bm w)| + |f_8 \cdot (1 + \bm w^\top \bm D \bm w)^2 \bm D_{ii} (\bm D \bm w)_i|] \\
    &\qquad + (|u_i (\bm u^\top \bm w) g_{\bm A} g_{\bm B} g_{\bm C} (\bm D \bm w)_i^2| + |(\bm A \bm w)_i h_{\bm u} g_{\bm B} g_{\bm C}(\bm D \bm w)_i^2| +  |(\bm B \bm w)_i h_{\bm u} g_{\bm A} g_{\bm C}(\bm D \bm w)_i^2| \\
    &\qquad \quad +  |(\bm C \bm w)_ih_{\bm u} g_{\bm A} g_{\bm B} (\bm D \bm w)_i^2| +  |f_8 \cdot (1 + \bm w^\top \bm D \bm w)(\bm D \bm w)_i^3|) \\
    &\qquad + (|u_i (\bm u^\top \bm w) g_{\bm A} g_{\bm B} g_{\bm C} (1 + \bm w^\top \bm D \bm w) \bm D_{ii}| + |(\bm A \bm w)_i h_{\bm u} g_{\bm B} g_{\bm C} (1 + \bm w^\top \bm D \bm w) \bm D_{ii}| \\
    &\qquad \quad +  |(\bm B \bm w)_i h_{\bm u} g_{\bm A} g_{\bm C}(1 + \bm w^\top \bm D \bm w) \bm D_{ii}| +  |(\bm C \bm w)_ih_{\bm u} g_{\bm A} g_{\bm B}(1 + \bm w^\top \bm D \bm w) \bm D_{ii}| \\
    &\qquad \quad +  |f_8 \cdot (1 + \bm w^\top \bm D \bm w)^2(\bm D \bm w)_i \bm D_{ii}|) \\
    &\qquad + |f_8 \cdot (\bm D \bm w)_i \bm D_{ii}|
\end{align*}
or rather
\begin{align*}
    &\biggl|\partial_i^3 f_8(\bm w)\biggr|   \\
    &\lcon |g_{\bm B} g_{\bm C} \cdot u_i^2 (\bm A \bm w)_i | + |\bm A_{ii} g_{\bm u}  g_{\bm B} g_{\bm C} \cdot u_i| + |g_{\bm A} g_{\bm C} \cdot u_i^2 (\bm B \bm w)_i | + |\bm B_{ii} g_{\bm u}  g_{\bm A} g_{\bm C} \cdot u_i| \\
    &\qquad + |g_{\bm A} g_{\bm B}\cdot u_i^2 (\bm C \bm w)_i | + | \bm C_{ii} g_{\bm u} g_{\bm A} g_{\bm B}\cdot u_i| + |\bm A_{ii} g_{\bm u}^2 g_{\bm C}\cdot (\bm B \bm w)_i | + |\bm B_{ii} g_{\bm u}^2 g_{\bm C}\cdot  (\bm A \bm w)_i| \\
    &\qquad + |\bm A_{ii} g_{\bm u}^2  g_{\bm B}\cdot (\bm C \bm w)_i |  + |\bm C_{ii} g_{\bm u}^2 g_{\bm B}\cdot (\bm A \bm w)_i | + |\bm B_{ii} g_{\bm u}^2 g_{\bm A}\cdot (\bm C \bm w)_i | \\
    &\qquad + |\bm C_{ii} g_{\bm u}^2 g_{\bm A}\cdot (\bm B \bm w)_i | \\
    &\qquad + |g_{\bm u}^2\cdot (\bm A \bm w)_i (\bm B \bm w)_i (\bm C \bm w)_i | + | g_{\bm u} g_{\bm A}\cdot u_i(\bm B \bm w)_i (\bm C \bm w)_i | \\
    &\qquad + | g_{\bm u} g_{\bm B}\cdot u_i(\bm A \bm w)_i (\bm C \bm w)_i  | + | g_{\bm u}g_{\bm C}\cdot u_i(\bm A \bm w)_i (\bm B \bm w)_i  | \\
    &\qquad + [| g_{\bm u}  g_{\bm B} g_{\bm C} g_{1,\bm D} \cdot u_i(\bm A \bm w)_i (\bm D \bm w)_i| + | g_{\bm u} g_{\bm A} g_{\bm C} g_{1,\bm D}\cdot u_i(\bm B\bm w)_i (\bm D \bm w)_i| \\
    &\qquad \quad + | g_{\bm u} g_{\bm A} g_{\bm B } g_{1,\bm D} \cdot u_i(\bm C \bm w)_i(\bm D \bm w)_i|+ |g_{\bm u}^2 g_{\bm C} g_{1,\bm D} \cdot (\bm A \bm w)_i(\bm B\bm w)_i(\bm D \bm w)_i| \\
    &\qquad \quad + |g_{\bm u}^2  g_{\bm B}g_{1,\bm D}\cdot (\bm A \bm w)_i (\bm C \bm w)_i (\bm D \bm w)_i|  + |g_{\bm u}^2 g_{\bm A} g_{1,\bm D}\cdot (\bm B \bm w)_i (\bm C \bm w)_i (\bm D \bm w)_i | \\
    &\qquad \quad + |g_{\bm A} g_{\bm B} g_{\bm C} g_{1,\bm D} \cdot u_i^2 (\bm D \bm w)_i| + |\bm A_{ii} g_{\bm u}^2 g_{\bm B} g_{\bm C} g_{1,\bm D} \cdot (\bm D \bm w)_i | \\
    &\qquad \quad +  |\bm B_{ii} g_{\bm u}^2 g_{\bm A} g_{\bm C} g_{1,\bm D}\cdot(\bm D \bm w)_i | +  |\bm C_{ii} g_{\bm u}^2 g_{\bm A} g_{\bm B} g_{1,\bm D} \cdot (\bm D \bm w)_i | \\
    &\qquad \quad  + (| g_{\bm u} g_{\bm A} g_{\bm B} g_{\bm C} g_{1,\bm D}^2  \cdot u_i(\bm D \bm w)_i^2| + |g_{\bm u}^2 g_{\bm B} g_{\bm C}g_{1,\bm D}^2\cdot (\bm A \bm w)_i (\bm D \bm w)_i^2 | \\
    &\qquad \qquad +  |g_{\bm u}^2 g_{\bm A} g_{\bm C} g_{1,\bm D}^2 \cdot (\bm B \bm w)_i (\bm D \bm w)_i^2| + | g_{\bm u}^2 g_{\bm A} g_{\bm B} g_{1,\bm D}^2 \cdot (\bm C \bm w)_i(\bm D \bm w)_i^2| \\
    &\qquad \qquad  +  |g_{\bm u}^2 g_{\bm A} g_{\bm B} g_{\bm C}  g_{1,\bm D}^3 \cdot (\bm D \bm w)_i^3|) \\
    &\qquad \quad  + |g_{\bm u}^2 g_{\bm A} g_{\bm B} g_{\bm C} g_{1,\bm D}\cdot (\bm D \bm w)_i^3 | + |\bm D_{ii}g_{\bm u}^2 g_{\bm A} g_{\bm B} g_{\bm C} g_{1,\bm D}^2 \cdot  (\bm D \bm w)_i|] \\
    &\qquad + (|u_i g_{\bm u} g_{\bm A} g_{\bm B} g_{\bm C}  \cdot (\bm D \bm w)_i^2| + |g_{\bm u}^2 g_{\bm B} g_{\bm C}\cdot (\bm A \bm w)_i  (\bm D \bm w)_i^2| +  | g_{\bm u}^2 g_{\bm A} g_{\bm C} \cdot (\bm B \bm w)_i (\bm D \bm w)_i^2| \\
    &\qquad \quad +  |g_{\bm u}^2 g_{\bm A} g_{\bm B} \cdot (\bm C \bm w)_i (\bm D \bm w)_i^2 | +  |g_{\bm u}^2 g_{\bm A} g_{\bm B} g_{\bm C}  g_{1,\bm D} \cdot (\bm D \bm w)_i^3|) \\
    &\qquad + (|\bm D_{ii}  g_{\bm u} g_{\bm A} g_{\bm B} g_{\bm C} g_{1,\bm D} \cdot u_i| + |\bm D_{ii}  g_{\bm u}^2 g_{\bm B} g_{\bm C} g_{1,\bm D} \cdot (\bm A \bm w)_i| \\
    &\qquad \quad +  |\bm D_{ii} g_{\bm u}^2 g_{\bm A} g_{\bm C}g_{1,\bm D} \cdot (\bm B \bm w)_i | +  |\bm D_{ii} h_{\bm u} g_{\bm A} g_{\bm B}g_{1,\bm D}  \cdot (\bm C \bm w)_i| \\
    &\qquad \quad +  | \bm D_{ii} g_{\bm u}^2 g_{\bm A} g_{\bm B} g_{\bm C}  g_{1,\bm D}^2 \cdot (\bm D \bm w)_i|) \\
    &\qquad + |\bm D_{ii} g_{\bm u}^2 g_{\bm A} g_{\bm B} g_{\bm C} \cdot (\bm D \bm w)_i|
\end{align*}

\subsubsection{Second stage bounding}
We are now ready to further bound these terms in expectation. Note that every term $|\partial_i^3 f_j(\bm w)|$ is bounded by a sum of terms of the form
\begin{equation*}
|\text{(zero or one matrix diagonal element}) \cdot (\text{product of functions}) \cdot (\text{product of vector elements})|
\end{equation*}
Therefore, to bound $\sum_{i=1}^p\E[ |W_i|^r \cdot \sup_{u \in (0,1)} |\partial_i^3 f_j(\bm W_{t,u,i})|]$, we do the following to each term that appears:
\begin{enumerate}
    \item Bound matrix diagonal elements by the operator norm of the matrix.
    \item Use Cauchy-Schwarz to separate the expectation of $|W_i|^r \cdot \text{product of functions}$ from the expectation of the product of vector elements.
    \item Use H\"older's inequality to split the expectation of $|W_i|^r \cdot \text{product of functions}$ into multiple expectations.
    \item Bound each of these expectations using Lemma \ref{lm:t_S_moments}.
    \item Use H\"older's inequality on the expectation of the product of vector elements.
    \item Use Cauchy-Schwarz (potentially repeatedly) on the sum (over $i = 1, \dots, p$) of products of vector elements. If there is only one vector element, use Cauchy-Schwarz with said vector element and the constant $1$ vector, thus picking up an additional $\sqrt{p}$ factor.
    \item For random vector elements involving $\bm W_{t,u,i}$, use Lemma \ref{lm:t_S_moments} to bound each resulting vector norm expectation.
\end{enumerate}
This process of repeatedly applying Cauchy-Schwarz and H\"older's inequality on various terms suffices for our universality proof because they ultimately maintain the scale of the resulting bounds (with respect to $p$ and norms of elements of $\bm S$). As an example, we can bound
\begin{align*}
    \E\biggl[\sum_{i=1}^p &|W_i|^r \cdot \sup_{u \in (0,1)} |\bm D_{ii} g_{\bm u}^2 g_{\bm A} g_{\bm B} g_{\bm C} \cdot (\bm D \bm W_{t,u,i})_i|\biggr] \\
    &\leq \|\bm D\|_{\text{op}}\sum_{i=1}^p\E\biggl[ |W_i|^r \cdot \sup_{u \in (0,1)} |\bm D_{ii} g_{\bm u}^2 g_{\bm A} g_{\bm B} g_{\bm C} \cdot (\bm D \bm W_{t,u,i})_i|\biggr] \\
    &\leq \|\bm D\|_{\text{op}} \sup_i \E\biggl[|W_i|^{2r}\sup_{u \in (0,1)} |g_{\bm u}^4 g_{\bm A}^2 g_{\bm B}^2 g_{\bm C}^2| \biggr]^{1/2} \sum_{i=1}^p\E\biggl[ \sup_{u \in (0,1)} |(\bm D \bm W_{t,u,i})_i^2|\biggr]^{1/2} \\
    &\leq \|\bm D\|_{\text{op}} \sup_i \E[|W_i|^{2r}]^{1/2} \sup_i\E\biggl[\sup_{u \in (0,1)} |g_{\bm u}^4 | \biggr]^{1/2} \sup_i \E\biggl[\sup_{u \in (0,1)} |g_{\bm A}^2 | \biggr]^{1/2} \\
    &\qquad \sup_i \E\biggl[\sup_{u \in (0,1)} | g_{\bm B}^2 | \biggr]^{1/2} \sup_i \E\biggl[\sup_{u \in (0,1)} |g_{\bm C}^2| \biggr]^{1/2} \cdot \sum_{i=1}^p\E\biggl[ \sup_{u \in (0,1)} |(\bm D \bm W_{t,u,i})_i^2|\biggr]^{1/2} \\
    &\lcon \|\bm D\|_{\text{op}} \|\bm u\|_2^2 \|\bm A\|_{\text{op}} \|\bm B\|_{\text{op}} \|\bm C\|_{\text{op}} p^3 \cdot \sum_{i=1}^p\E\biggl[ \sup_{u \in (0,1)} |(\bm D \bm W_{t,u,i})_i^2|\biggr]^{1/2} \\
    &\leq \|\bm D\|_{\text{op}} \|\bm u\|_2^2 \|\bm A\|_{\text{op}} \|\bm B\|_{\text{op}} \|\bm C\|_{\text{op}} p^3 \cdot \sqrt{p} \biggl(\sum_{i=1}^p\E\biggl[ \sup_{u \in (0,1)} |(\bm D \bm W_{t,u,i})_i^2|\biggr]\biggr)^{1/2} \\
    &\lcon \|\bm D\|_{\text{op}} \|\bm u\|_2^2 \|\bm A\|_{\text{op}} \|\bm B\|_{\text{op}} \|\bm C\|_{\text{op}} p^3 \cdot \sqrt{p} \cdot \|\bm D\|_{\text{op}} \sqrt{p} \\
    &= \|\bm u\|_2^2 \|\bm A\|_{\text{op}} \|\bm B\|_{\text{op}} \|\bm C\|_{\text{op}} \|\bm D\|_{\text{op}}^2 p^4
\end{align*}
More generally, for some matrix $\bm M$, some function $h$ and some vector $\bm z$ depending on $\bm W_{t,u,i}$, we have
\begin{align*}
    \E\biggl[\sum_{i=1}^p |W_i|^r  \cdot \sup_{u \in (0,1)} |\bm M_{ii} h \cdot \bm z_i|\biggr] \lcon \|\bm M\|_{\text{op}} \sup_i\E\biggl[ \sup_{u \in (0,1)} h^4\biggr]^{1/4}\cdot \sum_{i=1}^p \E\biggl[ \sup_{u \in (0,1)}  z_i^2\biggr]^{1/2}
\end{align*}
Because all such functions $h$ for us consist of some product of the functions $g_{\bm u}$, $g_{\bm v}$, $g_{\bm A}$, $g_{\bm B}$, $g_{\bm C}$, we immediately have that our bound for $\sup_i \E[\sup_u h^4]^{1/4}$ is the same as the bound we would yield above from $\sup_i \E[\sup_u |h|]$.

Now, we handle the last term $\sum_{i=1}^p \E[\sup_u z_i^2]^{1/2}$ for all possible $\bm z$ in our bounds:
\begin{itemize}
    \item If there is only one element (either random or deterministic; e.g., $z_i =v_i^m$ or $z_i = (\bm A \bm w)_i^m$), we apply Cauchy-Schwarz with the $1$ vector. For random quantities, we then apply Lemma \ref{lm:t_S_moments}. The resulting expression is of the form $\sqrt{p} \cdot \|\bm v\|_{2m}^m \leq \sqrt{p} \|\bm v\|_2^m$ if the vector is deterministic or $\sqrt{p} \|\bm A\|_{\text{op}}^m p^{m/2}$ if it is deterministic: i.e.,
    \begin{align*}
        \sum_{i=1}^p \E\biggl[\sup_{u \in (0,1)} |(\bm A \bm w)_i^{2m}| \biggr]^{1/2} &\leq \biggl(\sum_{i=1}^p 1^2 \biggr)^{1/2}\biggl(\sum_{i=1}^p \E\biggl[\sup_{u \in (0,1)} |(\bm A \bm w)_i^{2m}|\biggr] \biggr)^{1/2} \\
        &\lcon \sqrt{p} (\|\bm A\|_{\text{op}}^{2m} p^{m})^{1/2} = \sqrt{p} \|\bm A\|_{\text{op}}^m p^{m/2}
    \end{align*}
    and similarly for deterministic vectors, where we crucially use the fact that $\|\bm v\|_m \leq \|\bm v\|_q$ for $m \geq q$.
    \item If there is one deterministic element and one random element (e.g., $v_i^q (\bm A \bm w)_i^m$), then we can apply Cauchy-Schwarz and then bound the random quantity using Lemma \ref{lm:t_S_moments}. The resulting expression is of the form $\|\bm v\|_2^q \|\bm A\|_{\text{op}}^m p^{m/2}$: i.e.,
    \begin{align*}
        \sum_{i=1}^p \E\biggl[\sup_{u \in (0,1)} |v_i^{2q}(\bm A \bm w)_i^{2m}| \biggr]^{1/2} &= \sum_{i=1}^p v_i^q \cdot \E\biggl[\sup_{u \in (0,1)} |(\bm A \bm w)_i^{2m}| \biggr]^{1/2}\\ &\leq \biggl(\sum_{i=1}^p  v_i^{2q}\biggr)^{1/2}\biggl(\sum_{i=1}^p \E\biggl[\sup_{u \in (0,1)} |(\bm A \bm w)_i^{2m}|\biggr] \biggr)^{1/2} \\
        &\lcon \|\bm v\|_{2q}^{q}(\|\bm A\|_{\text{op}}^{2m} p^{m})^{1/2} \leq \|\bm v\|_2^q \|\bm A\|_{\text{op}}^m p^{m/2}
    \end{align*}
    \item If there are two deterministic elements and one random element (e.g., $u_i^t v_i^q (\bm A \bm w)_i^m$), then a second application of Cauchy-Schwarz on the deterministic element yields the bound $\|\bm u\|_2^t \|\bm v\|_2^q \|\bm A\|_{\text{op}}^m p^{m/2}$: i.e., 
    \begin{align*}
        &\sum_{i=1}^p \E\biggl[\sup_{u \in (0,1)} |u_i^{2t} v_i^{2q}(\bm A \bm w)_i^{2m}| \biggr]^{1/2} \\&= \sum_{i=1}^p u_i^t v_i^q \cdot \E\biggl[\sup_{u \in (0,1)} |(\bm A \bm w)_i^{2m}| \biggr]^{1/2}\\ &\leq \biggl(\sum_{i=1}^p  u_i^{2t}v_i^{2q}\biggr)^{1/2}\biggl(\sum_{i=1}^p \E\biggl[\sup_{u \in (0,1)} |(\bm A \bm w)_i^{2m}|\biggr] \biggr)^{1/2} \\
        &\leq \biggl(\sum_{i=1}^p  u_i^{4t}\biggr)^{1/4}\biggl(\sum_{i=1}^p v_i^{4q}\biggr)^{1/4}\biggl(\sum_{i=1}^p \E\biggl[\sup_{u \in (0,1)} |(\bm A \bm w)_i^{2m}|\biggr] \biggr)^{1/2} \\
        &\lcon \|\bm u\|_{4t}^t\|\bm v\|_{4q}^{q}(\|\bm A\|_{\text{op}}^{2m} p^{m})^{1/2} \leq \|\bm u\|_2^t \|\bm v\|_2^q \|\bm A\|_{\text{op}}^m p^{m/2}
    \end{align*}
    \item If there is one deterministic element and two random elements (e.g., $v_i^q (\bm A \bm w)_i^m (\bm B \bm w)_i^t$), then (1) a second application of Cauchy-Schwarz on the shared expectation and (2) a third application of Cauchy-Schwarz on the sum yields the bound $\|\bm v\|_2^q \|\bm A\|_{\text{op}}^m p^{m/2} \|\bm B\|_{\text{op}}^t p^{t/2}$: i.e.,
    \begin{align*}
        \sum_{i=1}^p &\E\biggl[\sup_{u \in (0,1)} |v_i^{2q}(\bm A \bm w)_i^{2m} (\bm B \bm w)_i^{2t}| \biggr]^{1/2} \\&= \sum_{i=1}^p v_i^q \cdot \E\biggl[\sup_{u \in (0,1)} |(\bm A \bm w)_i^{2m} (\bm B \bm w)_i^{2t}| \biggr]^{1/2}\\ &\leq \biggl(\sum_{i=1}^p  v_i^{2q}\biggr)^{1/2}\biggl(\sum_{i=1}^p \E\biggl[\sup_{u \in (0,1)} |(\bm A \bm w)_i^{2m} (\bm B \bm w)_i^{2t}|\biggr] \biggr)^{1/2} \\
        &\leq \biggl(\sum_{i=1}^p v_i^{2q}\biggr)^{1/2}\biggl(\sum_{i=1}^p \E\biggl[\sup_{u \in (0,1)} |(\bm A \bm w)_i^{4m}|\biggr]^{1/2} \E\biggl[\sup_{u \in (0,1)} |(\bm B \bm w)_i^{4t}|\biggr]^{1/2}  \biggr)^{1/2} \\
        &\leq \biggl(\sum_{i=1}^p v_i^{2q}\biggr)^{1/2}\biggl(\sum_{i=1}^p \E\biggl[\sup_{u \in (0,1)} |(\bm A \bm w)_i^{4m}|\biggr]\biggr)^{1/4} \biggl(\sum_{i=1}^p \E\biggl[\sup_{u \in (0,1)} |(\bm B \bm w)_i^{4t}|\biggr]  \biggr)^{1/4} \\
        &\lcon \|\bm v\|_{2q}^{q}(\|\bm A\|_{\text{op}}^{4m} p^{2m})^{1/4} (\|\bm B\|_{\text{op}}^{4t} p^{2t})^{1/4} \leq \|\bm v\|_2^q \|\bm A\|_{\text{op}}^m p^{m/2} \|\bm B\|_{\text{op}}^t p^{t/2}
    \end{align*}
    \item Similarly, using the same techniques (i.e., Cauchy-Schwarz with the 1 vector, Cauchy-Schwarz in the expectation, Cauchy-Schwarz in the sum),
    \begin{equation*}
        \sum_{i=1}^p \E\biggl[\sup_{u \in (0,1)} |(\bm A\bm w)_i^{2m} (\bm B \bm w)_i^{2t}|\biggr]^{1/2} \leq \sqrt{p} \|\bm A\|_{\text{op}}^m p^{m/2} \|\bm B\|_{\text{op}}^t p^{t/2}
    \end{equation*}
    \item For no deterministic elements but three random elements (e.g., $(\bm A \bm w)_i^q (\bm B \bm w)_i^m (\bm C \bm w)_i^t$), we can apply Cauchy-Schwarz once and apply H\"older's inequality twice to yield the bound $\sqrt{p} \|\bm A\|_{\text{op}}^q p^{q/2} \|\bm B\|_{\text{op}}^m p^{m/2} \|\bm C\|_{\text{op}}^t p^{t/2}$: i.e.,
    \begin{align*}
        &\sum_{i=1}^p \E\biggl[\sup_{u \in (0,1)} |(\bm A \bm w)_i^{2q} (\bm B \bm w)_i^{2m}(\bm C \bm w)_i^{2t}| \biggr]^{1/2} \\
        &\leq\sqrt{p}\biggl(\sum_{i=1}^p \E\biggl[\sup_{u \in (0,1)} |(\bm A \bm w)_i^{2q} (\bm B \bm w)_i^{2m} (\bm C \bm w)_i^{2t}|\biggr] \biggr)^{1/2} \\
        &\leq \sqrt{p} \biggl(\sum_{i=1}^p \E\biggl[\sup_{u \in (0,1)} (\bm A \bm w)_i^{6q}\biggr]^{1/3} \E\biggl[\sup_{u \in (0,1)} (\bm B \bm w)_i^{6m}\biggr]^{1/3}  \E\biggl[\sup_{u \in (0,1)} (\bm C \bm w)_i^{6t}\biggr]^{1/3}\biggr)^{1/2} \\
        &\leq \sqrt{p} \biggl(\sum_{i=1}^p \E\biggl[\sup_{u \in (0,1)} (\bm A \bm w)_i^{6q}\biggr]\biggr)^{1/6} \biggl(\sum_{i=1}^p \E\biggl[\sup_{u \in (0,1)} (\bm B \bm w)_i^{6m}\biggr]\biggr)^{1/6}  \\
        &\qquad \biggl(\sum_{i=1}^p \E\biggl[\sup_{u \in (0,1)} (\bm C \bm w)_i^{6t}\biggr]\biggr)^{1/6} \\
        &\leq \sqrt{p} \|\bm A\|_{\text{op}}^q p^{q/2} \|\bm B\|_{\text{op}}^m p^{m/2} \|\bm C\|_{\text{op}}^t p^{t/2}.
    \end{align*}
\end{itemize}

Therefore, we have
\begin{align*}
    &\E\biggl[\sum_{i=1}^p |W_i|^r\sup_{u \in (0,1)}|\partial_i^3 f_2(\bm W_{t,u,i}) | \\
    &\lcon \|\bm u\|_2 \|\bm v\|_2 \|\bm A\|_{\text{op}} p^{1/2} + \|\bm v\|_2 \cdot \|\bm u\|_2 \|\bm A\|_{\text{op}}^2 p + \|\bm u\|_2\cdot \|\bm v\|_2 \|\bm A\|_{\text{op}}^2 p \\
    &\qquad + \|\bm u\|_2 \|\bm v\|_2 \cdot \sqrt{p} \|\bm A\|_{\text{op}}^3 p^{3/2} + \|\bm A\|_{\text{op}} \|\bm u\|_2\|\bm v\|_2 \cdot \sqrt{p}\|\bm A\|_{\text{op}} p^{1/2} \\
    &\qquad + \|\bm A\|_{\text{op}} \|\bm v\|_2 \cdot \sqrt{p}\|\bm u\|_2 + \|\bm A\|_{\text{op}}\|\bm u\|_2 \cdot \sqrt{p} \|\bm v\|_2 \\
    &\qquad + \|\bm A\|_{\text{op}} \|\bm u\|_2\|\bm v\|_2 \cdot \sqrt{p}\|\bm A\|_{\text{op}} p^{1/2} \\
    &\lcon \|\bm u\|_2 \|\bm v\|_2 (\|\bm A\|_{\text{op}} p^{1/2} + \|\bm A\|_{\text{op}}^2 p + \|\bm A\|_{\text{op}}^3 p^2)
\end{align*} 
Similarly,
\begin{align*}
   & \E\biggl[\sum_{i=1}^p |W_i|^r \sup_{u \in (0,1)}|\partial_i^3 f_3(\bm w)|\biggr] \\ 
    &\lcon \|\bm A\|_{\text{op}} \|\bm v\|_2 \cdot \sqrt{p} \|\bm u\|_2 + \|\bm A\|_{\text{op}} \|\bm u\|_2 \cdot \sqrt{p} \|\bm v\|_2 + \|\bm u\|_2 \|\bm v\|_2 \|\bm A\|_{\text{op}} p^{1/2} \\
    &\lcon \|\bm u\|_2\|\bm v\|_2\|\bm A||_{\text{op}} p^{1/2}
\end{align*}
and
\begin{align*}
    &\E\biggl[\sum_{i=1}^p |W_i|^r \sup_{u \in (0,1)}|\partial_i^3 f_4(\bm w)|\biggr] \\
    &\lcon \|\bm A\|_{\text{op}} \|\bm v\|_2 \cdot \sqrt{p} \|\bm u\|_2 + \|\bm A\|_{\text{op}} \|\bm u\|_2 \cdot \sqrt{p} \|\bm v\|_2 + \|\bm u\|_2 \|\bm v\|_2 \|\bm A\|_{\text{op}} p^{1/2} \\
    &\qquad + \|\bm A\|_{\text{op}} \|\bm u\|_2 \|\bm v\|_2 \cdot \sqrt{p}\|\bm B\|_{\text{op}} p^{1/2} + \|\bm A\|_{\text{op}} p \cdot \|\bm u\|_2 \|\bm v\|_2 \|\bm B\|_{\text{op}} p^{1/2} +\\
    &\qquad \|\bm v\|_2 \cdot \|\bm u\|_2 \|\bm A\|_{\text{op}} p^{1/2} \|\bm B\|_{\text{op}} p^{1/2} + \|\bm u\| \cdot \|\bm v\|_2 \|\bm A\|_{\text{op}} p^{1/2} \|\bm B\|_{\text{op}} p^{1/2}  \\
    &\qquad  + \|\bm v\|_2 \|\bm A\|_{\text{op}} p \cdot \|\bm u\|_2 \|\bm B\|_{\text{op}}^2 p + \|\bm u\|_2 \|\bm A\|_{\text{op}} p \cdot \|\bm v\|_2 \|\bm B\|_{\text{op}}^2 p \\
    &\qquad + \|\bm u\|_2 \|\bm v\|_2 \cdot \sqrt{p} \|\bm A\|_{\text{op}} p^{1/2} \|\bm B\|_{\text{op}}^2 p + \|\bm u\|_2 \|\bm v\|_2 \|\bm A\|_{\text{op}} p \cdot \sqrt{p} \|\bm B\|_{\text{op}}^3 p^{3/2} \\
    &\qquad  + \|\bm B \|_{\text{op}} \|\bm u\|_2 \|\bm v\|_2 \|\bm A\|_{\text{op}} p \cdot \sqrt{p} \|\bm B\|_{\text{op}} p^{1/2} + \|\bm B\|_{\text{op}} \|\bm v\|_{2} \|\bm A\|_{\text{op}} p \cdot \sqrt{p} \|\bm u\|_2  \\
    &\qquad  + \|\bm B\|_{\text{op}} \|\bm u\|_2 \|\bm A\|_{\text{op}} p \cdot \sqrt{p} \|\bm v\|_2  + \|\bm B\|_{\text{op}} \|\bm u\|_2 \|\bm v\|_2 \cdot \sqrt{p} \|\bm A\|_{\text{op}} p^{1/2} \\
    &\qquad + \|\bm B\|_{\text{op}} \|\bm u\|_2 \|\bm v\|_2 \|\bm A\|_{\text{op}} p \cdot \sqrt{p}\|\bm B\|_{\text{op}} p^{1/2} \\
    &\lcon \|\bm u\|_2 \|\bm v\|_2 (\|\bm A\|_{\text{op}} p^{1/2} + \|\bm A\|_{\text{op}} \|\bm B\|_{\text{op}} p^{3/2} + \|\bm A\|_{\text{op}} \|\bm B\|_{\text{op}}^2 p^2 + \|\bm A\|_{\text{op}} \|\bm B\|_{\text{op}}^3 p^3)
\end{align*}
and
\begin{align*}
    &\E\biggl[\sum_{i=1}^p |W_i|^r \sup_{u \in (0,1)}|\partial_i^3 f_5(\bm w)|\biggr] \\
    &\lcon \|\bm A\|_{\text{op}} \|\bm v\|_2 \cdot \sqrt{p} \|\bm u\|_2 + \|\bm A\|_{\text{op}} \|\bm u\|_2 \cdot \sqrt{p} \|\bm v\|_2 + \|\bm u\|_2 \|\bm v\|_2 \|\bm A\|_{\text{op}} p^{1/2} \\
    &\qquad + \|\bm A\|_{\text{op}} \|\bm u\|_2 \|\bm v\|_2 \|\bm B\|_{\text{op}} p \sqrt{p} \|\bm B\|_{\text{op}} p^{1/2} + \|\bm A\|_{\text{op}} p \|\bm B\|_{\text{op}} p \|\bm u\|_2 \|\bm v\|_2 \|\bm B\|_{\text{op}} p^{1/2} \\
    &\qquad + \|\bm v\|_2 \|\bm B\|_{\text{op}} p \|\bm u\|_2 \|\bm A\|_{\text{op}} p^{1/2} \|\bm B\|_{\text{op}} p^{1/2} + \|\bm u\|_2 \|\bm B\|_{\text{op}} p \|\bm v\|_2 \|\bm A\|_{\text{op}} p^{1/2} \|\bm B\|_{\text{op}} p^{1/2} \\
    &\qquad + \|\bm v\|_2 \|\bm A\|_{\text{op}} p \|\bm B\|_{\text{op}}^2 p^2 \|\bm u\|_2 \|\bm B\|_{\text{op}}^2 p + \|\bm u\|_2 \|\bm A\|_{\text{op}} p \|\bm B\|_{\text{op}}^2 p^2 \|\bm v\|_2 \|\bm B\|_{\text{op}}^2 p \\
    &\qquad + \|\bm u\|_2 \|\bm v\|_2 \|\bm B\|_{\text{op}}^2 p^2 \sqrt{p} \|\bm A\|_{\text{op}} p^{1/2} \|\bm B\|_{\text{op}}^2 p  \\
    &\qquad + \|\bm u\|_2 \|\bm v\|_2 \|\bm A\|_{\text{op}} p \|\bm B\|_{\text{op}}^3 p^3  \sqrt{p} \|\bm B\|_{\text{op}}^3 p^{3/2} \\
    &\qquad + \|\bm u\|_2 \|\bm v\|_2 \|\bm A\|_{\text{op}} p \|\bm B\|_{\text{op}} p  \sqrt{p} \|\bm B\|_{\text{op}}^3 p^{3/2} \\
    &\qquad + \|\bm B\|_{\text{op}} \|\bm u\|_2 \|\bm v\|_2 \|\bm A\|_{\text{op}} p \|\bm B\|_{\text{op}}^2 p^2 \sqrt{p} \|\bm B\|_{\text{op}} p^{1/2} \\
    &\qquad + \|\bm v\|_2 \|\bm A\|_{\text{op}} p \|\bm u\|_2 \|\bm B\|_{\text{op}}^2 p + \|\bm u\|_2 \|\bm A\|_{\text{op}} p \|\bm v\|_2 \|\bm B\|_{\text{op}}^2 p \\
    &\qquad + \| \bm u\|_2 \|\bm v\|_2 \sqrt{p} \|\bm A\|_{\text{op}} p^{1/2} \|\bm B\|_{\text{op}}^2 p \\
    &\qquad + \|\bm u\|_2 \|\bm v\|_2 \|\bm A\|_{\text{op}} p \|\bm B\|_{\text{op}} p \sqrt{p} \|\bm B\|_{\text{op}}^3 p^{3/2} + \|\bm B\|_{\text{op}} \|\bm v\|_2 \|\bm A\|_{\text{op}} p \|\bm B\|_{\text{op}} p  \sqrt{p} \|\bm u\|_2 \\
    &\qquad + \|\bm B\|_{\text{op}} \|\bm u\|_2 \|\bm A\|_{\text{op}} p \|\bm B\|_{\text{op}} p  \sqrt{p} \|\bm v\|_2  + \|\bm B\|_{\text{op}} \|\bm u\|_2 \|\bm v\|_2 \|\bm B\|_{\text{op}} p  \sqrt{p} \|\bm A\|_{\text{op}} p^{1/2} \\
    &\qquad + \|\bm B\|_{\text{op}} \|\bm u\|_2 \|\bm v\|_2 \|\bm A\|_{\text{op}} p \|\bm B\|_{\text{op}}^2 p^2  \sqrt{p} \|\bm B\|_{\text{op}} p^{1/2} \\
    &\qquad + \|\bm B\|_{\text{op}} \|\bm u\|_2 \|\bm v\|_2 \|\bm A\|_{\text{op}} p  \sqrt{p} \|\bm B\|_{\text{op}} p^{1/2} \\
    &\lcon \|\bm u\|_2 \|\bm v\|_2 (\|\bm A\|_{\text{op}} p^{1/2} + \|\bm A\|_{\text{op}} \|\bm B\|_{\text{op}}^2 p^{5/2} + \|\bm A\|_{\text{op}} \|\bm B\|_{\text{op}}^4 p^4 + \|\bm A\|_{\text{op}} \|\bm B\|_{\text{op}}^6 p^6)
\end{align*}
and
\begin{align*}
    &\E\biggl[\sum_{i=1}^p |W_i|^r \sup_{u \in (0,1)}|\partial_i^3 f_6(\bm w)|\biggr] \\
    &\lcon \|\bm B\|_{\text{op}} p \cdot \|\bm u\|_2^2 \|\bm A\|_{\text{op}} p^{1/2} + \|\bm A\|_{\text{op}} \|\bm u\|_2 \|\bm B\|_{\text{op}} p \cdot \sqrt{p} \|\bm u\|_2 + \|\bm A\|_{\text{op}} p \cdot \|\bm u\|_2^2 \|\bm B\|_{\text{op}} p^{1/2} \\
    &\qquad + \|\bm B\|_{\text{op}} \|\bm u\|_2 \|\bm A\|_{\text{op}} p \cdot \sqrt{p} \|\bm u\|_2 + \|\bm A\|_{\text{op}} \|\bm u\|_2^2 \cdot \sqrt{p} \|\bm B\|_{\text{op}} p^{1/2} \\
    &\qquad + \|\bm B\|_{\text{op}} \|\bm u\|_2^2 \cdot \sqrt{p}\|\bm A\|_{\text{op}} p^{1/2} \\
    &\qquad + \|\bm u\|_2 \cdot \|\bm u\|_2 \|\bm A\|_{\text{op}} p^{1/2} \|\bm B\|_{\text{op}} p^{1/2} + \|\bm A\|_{\text{op}} p \|\bm B\|_{\text{op}} p \cdot \|\bm u\|_2^2 \|\bm C\|_{\text{op}} p^{1/2} \\
    &\qquad + \|\bm A\|_{\text{op}} \|\bm u\|_2^2 \|\bm B\|_{\text{op}} p \cdot \sqrt{p}\|\bm C\|_{\text{op}} p^{1/2} + \|\bm B\|_{\text{op}} \|\bm u\|_2^2 \|\bm A\|_{\text{op}} p \cdot \sqrt{p}\|\bm C\|_{\text{op}} p^{1/2} \\
    &\qquad + \|\bm u\|_2 \|\bm B\|_{\text{op}} p \cdot \|\bm u\|_2 \|\bm A\|_{\text{op}} p^{1/2} \|\bm C\|_{\text{op}} p^{1/2} \\
    &\qquad + \|\bm u\| \|\bm A\|_{\text{op}} p \cdot \|\bm u\|_2 \|\bm B\|_{\text{op}} p^{1/2} \|\bm C\|_{\text{op}} p^{1/2} \\
    &\qquad + \|\bm u\|_2^2 \cdot \sqrt{p} \|\bm A\|_{\text{op}} p^{1/2} \|\bm B\|_{\text{op}} p^{1/2} \|\bm C\|_{\text{op}} p^{1/2} + \|\bm u\|_2 \|\bm A\|_{\text{op}} p \|\bm B\|_{\text{op}} p \cdot \|\bm u\|_2 \|\bm C\|_{\text{op}}^2 p \\
    &\qquad + \|\bm u\|_2^2 \|\bm B\|_{\text{op}} p \cdot \sqrt{p} \|\bm A\|_{\text{op}} p^{1/2} \|\bm C\|_{\text{op}}^2 p + \|\bm u\|_2^2 \|\bm A\|_{\text{op}} p \cdot \sqrt{p} \|\bm B\|_{\text{op}} p^{1/2} \|\bm C\|_{\text{op}}^2 p \\
    &\qquad + \|\bm u\|_2^2 \|\bm A\|_{\text{op}} p \|\bm B\|_{\text{op}} p \cdot \sqrt{p} \|\bm C\|_{\text{op}}^3 p^{3/2} + \|\bm C\|_{\text{op}} \|\bm u\|_2 \|\bm A\|_{\text{op}} p \|\bm B\|_{\text{op}} p \cdot \sqrt{p} \|\bm u\|_2 \\
    &\qquad + \|\bm C\|_{\text{op}} \|\bm u\|_2^2 \|\bm B\|_{\text{op}} p \cdot \sqrt{p} \|\bm A\|_{\text{op}} p^{1/2} + \|\bm C\|_{\text{op}} \|\bm u\|_2^2 \|\bm A\|_{\text{op}} p \cdot \sqrt{p} \|\bm B\|_{\text{op}} p^{1/2}  \\
    &\qquad + \|\bm C\|_{\text{op}} \|\bm u\|_2^2 \|\bm A\|_{\text{op}} p \|\bm B\|_{\text{op}} p \cdot \sqrt{p} \|\bm C\|_{\text{op}} p^{1/2} \\
    &\lcon \|\bm u\|_2^2 (\|\bm A\|_{\text{op}} \|\bm B\|_{\text{op}} p^{3/2} + \|\bm A\|_{\text{op}} \|\bm B\|_{\text{op}} \|\bm C\|_{\text{op}} p^{5/2} + \|\bm A\|_{\text{op}} \|\bm B\|_{\text{op}} \|\bm C\|_{\text{op}}^2 p^3 \\
    &\qquad + \|\bm A\|_{\text{op}} \|\bm B\|_{\text{op}} \|\bm C\|_{\text{op}}^3 p^{7/2})
\end{align*}
and
\begin{align*}
    &\E\biggl[\sum_{i=1}^p |W_i|^r \sup_{u \in (0,1)}|\partial_i^3 f_7(\bm w)|\biggr] \\
    &\lcon \|\bm A\|_{\text{op}} \|\bm v\|_2 \|\bm B\|_{\text{op}} p \cdot \sqrt{p} \|\bm u\|_2 + \|\bm B\|_{\text{op}} \|\bm v\|_2 \|\bm A\|_{\text{op}} p \cdot \sqrt{p} \|\bm u\|_2 \\
    &\qquad + \|\bm A\|_{\text{op}} \|\bm u\|_2\|\bm B\|_{\text{op}} p \cdot \sqrt{p} \|\bm v\|_2 + \|\bm B\|_{\text{op}} \|\bm u\|_2 \|\bm A\|_{\text{op}} p \cdot \sqrt{p} \|\bm v\|_2  \\
    &\qquad + \|\bm A\|_{\text{op}} \|\bm u\|_2 \|\bm v\|_2 \cdot \sqrt{p} \|\bm B\|_{\text{op}} p^{1/2} + \|\bm B\|_{\text{op}} \|\bm u\|_2 \|\bm v\|_2 \cdot \sqrt{p} \|\bm A\|_{\text{op}} p^{1/2} \\
    &\qquad + \|\bm B\|_{\text{op}} p \cdot \|\bm u\|_2 \|\bm v\|_2 \|\bm A\|_{\text{op}} p^{1/2} + \|\bm A\|_{\text{op}} p \cdot \|\bm u\|_2 \|\bm v\|_2 \|\bm B\|_{\text{op}} p^{1/2} \\
    &\qquad + \|\bm v\|_2 \cdot \|\bm u\|_2 \|\bm A\|_{\text{op}} p^{1/2} \|\bm B\|_{\text{op}} p^{1/2} + \|\bm u\|_2 \cdot \|\bm v\|_2 \|\bm A\|_{\text{op}} p^{1/2} \|\bm B\|_{\text{op}} p^{1/2}  \\
    &\qquad + \|\bm A\|_{\text{op}} p \|\bm B\|_{\text{op}} p \|\bm C\|_{\text{op}} p \cdot \|\bm u\|_2 \|\bm v\|_2 \|\bm C\|_{\text{op}} p^{1/2} \\
    &\qquad + \|\bm v\|_2 \|\bm B\|_{\text{op}} p \|\bm C\|_{\text{op}} p \cdot \|\bm u\|_2 \|\bm A\|_{\text{op}} p^{1/2} \|\bm C\|_{\text{op}} p^{1/2} \\
    &\qquad + \|\bm v\|_2 \|\bm A\|_{\text{op}} p \|\bm C\|_{\text{op}} p  \cdot \|\bm u\|_2 \|\bm B\|_{\text{op}}p^{1/2} \|\bm C\|_{\text{op}} p^{1/2} \\
    &\qquad + \|\bm u\|_2 \|\bm A\|_{\text{op}} p \|\bm C\|_{\text{op}} p \cdot \|\bm v\|_2 \|\bm B\|_{\text{op}} p^{1/2} \|\bm C\|_{\text{op}} p^{1/2} \\
    &\qquad + \|\bm u\|_2 \|\bm v\|_2 \|\bm C\|_{\text{op}} p \cdot \|\bm A\|_{\text{op}} p^{1/2} \|\bm B\|_{\text{op}} p^{1/2} \|\bm C\|_{\text{op}} p^{1/2} \\
    &\qquad + \|\bm A\|_{\text{op}} \|\bm u\|_2 \|\bm v\|_2 \|\bm B\|_{\text{op}} p \|\bm C\|_{\text{op}} p \cdot \sqrt{p} \|\bm C\|_{\text{op}} p^{1/2}  \\
    &\qquad + \|\bm B\|_{\text{op}} \|\bm u\|_2 \|\bm v\|_2 \|\bm A\|_{\text{op}} p \|\bm C\|_{\text{op}} p \cdot \sqrt{p} \|\bm C\|_{\text{op}} p^{1/2} \\
    &\qquad + \|\bm v\|_2 \|\bm A\|_{\text{op}} p \|\bm B\|_{\text{op}} p \|\bm C\|_{\text{op}}^2 p^2 \cdot \|\bm u\|_2 \|\bm C\|_{\text{op}}^2 p \\
    &\qquad + \|\bm u\|_2 \|\bm A\|_{\text{op}} p \|\bm B\|_{\text{op}} p \|\bm C\|_{\text{op}}^2 p^2 \cdot \|\bm v\|_2 \|\bm C\|_{\text{op}}^2 p \\
    &\qquad + \|\bm u\|_2 \|\bm v\|_2 \|\bm B\|_{\text{op}} p \|\bm C\|_{\text{op}}^2 p^2 \cdot \sqrt{p} \|\bm A\|_{\text{op}} p^{1/2} \|\bm C\|_{\text{op}}^2 p \\
    &\qquad + \|\bm u\|_2 \|\bm v\|_2 \|\bm A\|_{\text{op}} p \|\bm C\|_{\text{op}}^2 p^2 \cdot \sqrt{p} \|\bm B\|_{\text{op}} p^{1/2} \|\bm C\|_{\text{op}}^2 p \\
    &\qquad + \|\bm u\|_2 \|\bm v\|_2 \|\bm A\|_{\text{op}} p \|\bm B\|_{\text{op}} p \|\bm C\|_{\text{op}}^3 p^3 \cdot \sqrt{p} \|\bm C\|_{\text{op}}^3 p^{3/2} \\
    &\qquad + \|\bm u\|_2 \|\bm v\|_2 \|\bm A\|_{\text{op}} p \|\bm B\|_{\text{op}} p \|\bm C\|_{\text{op}} p \cdot \sqrt{p} \|\bm C\|_{\text{op}}^3 p^{3/2} \\
    &\qquad + \|\bm C\|_{\text{op}}\|\bm u\|_2 \|\bm v\|_2 \|\bm A\|_{\text{op}} p \|\bm B\|_{\text{op}} p \|\bm C\|_{\text{op}}^2 p^2 \cdot \sqrt{p} \|\bm C\|_{\text{op}} p^{1/2} \\
    &\qquad + \|\bm v\|_2 \|\bm A\|_{\text{op}} p \|\bm B\|_{\text{op}} p \cdot \|\bm u\|_2 \|\bm C\|_{\text{op}}^2 p + \|\bm u\|_2 \|\bm A\|_{\text{op}} p \|\bm B\|_{\text{op}} p \cdot \|\bm v\|_2 \|\bm C\|_{\text{op}}^2 p \\
    &\qquad + \|\bm u\|_2 \|\bm v\|_2 \|\bm B\|_{\text{op}} p \cdot \sqrt{p} \|\bm A\|_{\text{op}} p^{1/2} \|\bm C\|_{\text{op}}^2 p \\
    &\qquad + \|\bm u\|_2 \|\bm v\|_2 \|\bm A\|_{\text{op}} p \cdot \sqrt{p} \|\bm B\|_{\text{op}} p^{1/2} \|\bm C\|_{\text{op}}^2 p  \\
    &\qquad + \|\bm u\|_2 \|\bm v\|_2 \|\bm A\|_{\text{op}} p \|\bm B\|_{\text{op}} p \|\bm C\|_{\text{op}} p \cdot \sqrt{p} \|\bm C\|_{\text{op}}^3 p^{3/2} \\
    &\qquad + \|\bm C\|_{\text{op}} \|\bm v\|_{\text{op}} \|\bm A\|_{\text{op}} p \|\bm B\|_{\text{op}} p \|\bm C\|_{\text{op}} p \cdot \sqrt{p} \|\bm u\|_2 \\
    &\qquad + \|\bm C\|_{\text{op}} \|\bm u\|_{\text{op}} \|\bm A\|_{\text{op}} p \|\bm B\|_{\text{op}} p \|\bm C\|_{\text{op}} p \cdot \sqrt{p} \|\bm v\|_2 \\
    &\qquad + \|\bm C\|_{\text{op}} \|\bm u\|_2 \|\bm v\|_2 \|\bm B\|_{\text{op}} p \|\bm C\|_{\text{op}} p \cdot \sqrt{p} \|\bm A\|_{\text{op}} p^{1/2} \\
    &\qquad + \|\bm C\|_{\text{op}} \|\bm u\|_2 \|\bm v\|_2 \|\bm A\|_{\text{op}} p \|\bm C\|_{\text{op}} p \cdot \sqrt{p} \|\bm B\|_{\text{op}} p^{1/2} \\
    &\qquad +\|\bm C\|_{\text{op}}\|\bm u\|_2 \|\bm v\|_2 \|\bm A\|_{\text{op}} p \|\bm B\|_{\text{op}} p \|\bm C\|_{\text{op}}^2 p^2 \cdot \sqrt{p} \|\bm C\|_{\text{op}} p^{1/2} \\
    &\qquad + \|\bm C\|_{\text{op}} \|\bm u\|_2 \|\bm v\|_2 \|\bm A\|_{\text{op}} p \|\bm B\|_{\text{op}} p \cdot \sqrt{p}\|\bm C\|_{\text{op}} p^{1/2} \\
    &\lcon \|\bm u\|_2 \|\bm v\|_2 (\|\bm A\|_{\text{op}} \|\bm B\|_{\text{op}} p^{3/2} + \|\bm A\|_{\text{op}} \|\bm B\|_{\text{op}} \|\bm C\|_{\text{op}}^2 p^{7/2} \\
    &\qquad + \|\bm A\|_{\text{op}} \|\bm B\|_{\text{op}} \|\bm C\|_{\text{op}}^4 p^5 + \|\bm A\|_{\text{op}} \|\bm B\|_{\text{op}} \|\bm C\|_{\text{op}}^6 p^7
\end{align*}
and finally 
\begin{align*}
    &\E\biggl[\sum_{i=1}^p |W_i|^r \sup_{u \in (0,1)}|\partial_i^3 f_3(\bm w)|\biggr] \\
    &\lcon \|\bm B\|_{\text{op}} p \|\bm C\|_{\text{op}} p \cdot \|\bm u\|_2^2 \|\bm A\|_{\text{op}} p^{1/2} + \|\bm A\|_{\text{op}} \|\bm u\|_2 \|\bm B\|_{\text{op}} p \|\bm C\|_{\text{op}} p \cdot \sqrt{p} \|\bm u\|_2 \\
    & + \|\bm A\|_{\text{op}} p \|\bm C\|_{\text{op}} p \cdot \|\bm u\|_2^2 \|\bm B\|_{\text{op}} p^{1/2} + \|\bm B\|_{\text{op}} \|\bm u\|_2 \|\bm A\|_{\text{op}} p \|\bm C\|_{\text{op}} p \cdot \sqrt{p}\|\bm u\|_2 \\
    & + \|\bm A\|_{\text{op}} p \|\bm B\|_{\text{op}} p \cdot \|\bm u\|_2^2 \|\bm C\|_{\text{op}} p^{1/2} + \|\bm C\|_{\text{op}} \|\bm u\|_2 \|\bm A\|_{\text{op}} p \|\bm B\|_{\text{op}} p \cdot \sqrt{p} \|\bm u\|_2 \\
    & + \|\bm A\|_{\text{op}} \|\bm u\|_2^2 \|\bm C\|_{\text{op}} p \cdot \sqrt{p}\|\bm B\|_{\text{op}} p^{1/2} + \|\bm B\|_{\text{op}} \|\bm u\|_2^2 \|\bm C\|_{\text{op}} p \cdot \sqrt{p} \|\bm A\|_{\text{op}} p^{1/2} \\
    & + \|\bm A\|_{\text{op}} \|\bm u\|_2^2 \|\bm B\|_{\text{op}} p \cdot \sqrt{p} \|\bm C\|_{\text{op}} p^{1/2} + \|\bm C\|_{\text{op}} \|\bm u\|_2^2 \|\bm B\|_{\text{op}} p \cdot \sqrt{p} \|\bm A\|_{\text{op}} p^{1/2} \\
    & + \|\bm B\|_{\text{op}} \|\bm u\|_2^2 \|\bm A\|_{\text{op}} p \cdot \sqrt{p}\|\bm C\|_{\text{op}} p^{1/2} + \|\bm C\|_{\text{op}} \|\bm u\|_2^2 \|\bm A\|_{\text{op}} p \cdot \sqrt{p} \|\bm B\|_{\text{op}} p^{1/2} \\
    & + \|\bm u\|_2^2 \cdot \sqrt{p} \|\bm A\|_{\text{op}} p^{1/2} \|\bm B\|_{\text{op}} p^{1/2} \|\bm C\|_{\text{op}} p^{1/2} + \|\bm u\|_2 \|\bm A\|_{\text{op}} p \cdot \|\bm u\|_2 \|\bm B\|_{\text{op}} p^{1/2} \|\bm C\|_{\text{op}} p^{1/2} \\
    & + \|\bm u\|_2 \|\bm B\|_{\text{op}} p \cdot \|\bm u\|_2 \|\bm A\|_{\text{op}} p^{1/2} \|\bm C\|_{\text{op}} p^{1/2} + \|\bm u\|_2 \|\bm C\|_{\text{op}}p \cdot \|\bm u\|_2 \|\bm A\|_{\text{op}} p^{1/2} \|\bm B\|_{\text{op}} p^{1/2} \\
    & + \|\bm u\|_2 \|\bm B\|_{\text{op}} p \|\bm C\|_{\text{op}} p \|\bm D\|_{\text{op}} p \cdot \|\bm u\|_2 \|\bm A\|_{\text{op}} p^{1/2} \|\bm D\|_{\text{op}} p^{1/2} \\
    &+ \|\bm u\|_2 \|\bm A\|_{\text{op}} p \|\bm C\|_{\text{op}} p \|\bm D\|_{\text{op}} p \cdot \|\bm u\|_2 \|\bm B\|_{\text{op}} p^{1/2} \|\bm D\|_{\text{op}} p^{1/2} \\
    & + \|\bm u\|_2 \|\bm A\|_{\text{op}} p \|\bm B\|_{\text{op}} p \|\bm D\|_{\text{op}} p \cdot \|\bm u\|_2 \|\bm C\|_{\text{op}} p^{1/2} \|\bm D\|_{\text{op}} p^{1/2} \\
    &+ \|\bm u\|_2^2 \|\bm C\|_{\text{op}} p \|\bm D\|_{\text{op}} p \cdot \sqrt{p} \|\bm A\|_{\text{op}} p^{1/2} \|\bm B\|_{\text{op}} p^{1/2} \|\bm D\|_{\text{op}} p^{1/2} \\
    & + \|\bm u\|_2^2 \|\bm B\|_{\text{op}} p \|\bm D\|_{\text{op}} p \cdot \sqrt{p} \|\bm A\|_{\text{op}} p^{1/2} \|\bm C\|_{\text{op}} p^{1/2} \|\bm D\|_{\text{op}} p^{1/2} \\ &+ \|\bm u\|_2^2 \|\bm A\|_{\text{op}} p \|\bm D\|_{\text{op}} p \cdot \sqrt{p} \|\bm B\|_{\text{op}} p^{1/2} \|\bm C\|_{\text{op}} p^{1/2} \|\bm D\|_{\text{op}} p^{1/2} \\
    & + \|\bm A\|_{\text{op}} p \|\bm B\|_{\text{op}} p \|\bm C\|_{\text{op}} p \|\bm D\|_{\text{op}} p \cdot \|\bm u\|_2^2 \|\bm D\|_{\text{op}} p^{1/2} \\
    &+ \|\bm A\|_{\text{op}} \|\bm u\|_2^2 \|\bm B\|_{\text{op}} p \|\bm C\|_{\text{op}} p \|\bm D\|_{\text{op}} p \cdot \sqrt{p} \|\bm D\|_{\text{op}} p^{1/2} \\
    & + \|\bm B\|_{\text{op}} \|\bm u\|_2^2 \|\bm A\|_{\text{op}} p \|\bm C\|_{\text{op}} p \|\bm D\|_{\text{op}} p \cdot \sqrt{p} \|\bm D\|_{\text{op}} p^{1/2} \\
    &+ \|\bm C\|_{\text{op}} \|\bm u\|_2^2 \|\bm A\|_{\text{op}} p \|\bm B\|_{\text{op}} p \|\bm D\|_{\text{op}} p \cdot \sqrt{p} \|\bm D\|_{\text{op}} p^{1/2} \\
    &+ \|\bm u\|_2 \|\bm A\|_{\text{op}} p \|\bm B\|_{\text{op}} p \|\bm C\|_{\text{op}} p \|\bm D\|_{\text{op}}^2 p^2 \cdot \|\bm u\|_2 \|\bm D\|_{\text{op}}^2 p \\
    &+ \|\bm u\|_2^2 \|\bm B\|_{\text{op}} p \|\bm C\|_{\text{op}} p \|\bm D\|_{\text{op}}^2 p^2 \cdot \sqrt{p} \|\bm A\|_{\text{op}} p^{1/2} \|\bm D\|_{\text{op}}^2 p \\
    & + \|\bm u\|_2^2 \|\bm A\|_{\text{op}} p \|\bm C\|_{\text{op}} p \|\bm D\|_{\text{op}}^2 p^2 \cdot \sqrt{p} \|\bm B\|_{\text{op}} p^{1/2} \|\bm D\|_{\text{op}}^2 p\\
    &+ \|\bm u\|_2^2 \|\bm A\|_{\text{op}} p \|\bm B\|_{\text{op}} p \|\bm D\|_{\text{op}}^2 p^2 \cdot \sqrt{p} \|\bm C\|_{\text{op}} p^{1/2} \|\bm D\|_{\text{op}}^2 p \\
    &+ \|\bm u\|_2^2 \|\bm A\|_{\text{op}} p \|\bm B\|_{\text{op}} p \|\bm C\|_{\text{op}} p \|\bm D\|_{\text{op}}^3 p^3 \cdot \sqrt{p} \|\bm D\|_{\text{op}}^3 p^{3/2} \\
    &+ \|\bm u\|_2^2 \|\bm A\|_{\text{op}} p \|\bm B\|_{\text{op}} p \|\bm C\|_{\text{op}} p \|\bm D\|_{\text{op}} p \cdot \sqrt{p} \|\bm D\|_{\text{op}}^3 p^{3/2} \\ 
    &+ \|\bm D\|_{\text{op}} \|\bm u\|_2^2 \|\bm A\|_{\text{op}} p \|\bm B\|_{\text{op}} p \|\bm C\|_{\text{op}} p \|\bm D\|_{\text{op}}^2 p^2 \cdot \sqrt{p} \|\bm D\|_{\text{op}} p^{1/2} \\
    &+ \|\bm u\|_2 \|\bm A\|_{\text{op}} p \|\bm B\|_{\text{op}} p \|\bm C\|_{\text{op}} p \cdot \|\bm u\|_2 \|\bm D\|_{\text{op}}^2 p  \\
    & + \|\bm u\|_2^2 \|\bm B\|_{\text{op}} p \|\bm C\|_{\text{op}} p \cdot \sqrt{p} \|\bm A\|_{\text{op}} p^{1/2} \|\bm D\|_{\text{op}}^2 p \\
    &+ \|\bm u\|_2^2 \|\bm A\|_{\text{op}} p \|\bm C\|_{\text{op}} p \cdot \sqrt{p} \|\bm B\|_{\text{op}} p^{1/2} \|\bm D\|_{\text{op}}^2 p \\
    & + \|\bm u\|_2^2 \|\bm A\|_{\text{op}} p \|\bm B\|_{\text{op}} p \cdot \sqrt{p} \|\bm C\|_{\text{op}} p^{1/2} \|\bm D\|_{\text{op}}^2 p \\
    &+ \|\bm u\|_2^2 \|\bm A\|_{\text{op}} p \|\bm B\|_{\text{op}} p \|\bm C\|_{\text{op}} p \|\bm D\|_{\text{op}} p  \cdot \sqrt{p} \|\bm D\|_{\text{op}}^3 p^{3/2} \\
    & + \|\bm D\|_{\text{op}} \|\bm u\|_2 \|\bm A\|_{\text{op}} p \|\bm B\|_{\text{op}} p \|\bm C\|_{\text{op}} p \|\bm D\|_{\text{op}} p \cdot \sqrt{p} \|\bm u\|_2 \\
    &+ \|\bm D\|_{\text{op}} \|\bm u\|_2^2 \|\bm B\|_{\text{op}} p \|\bm C\|_{\text{op}} p \|\bm D\|_{\text{op}} p \cdot \sqrt{p} \|\bm A\|_{\text{op}} p^{1/2} \\
    & + \|\bm D\|_{\text{op}} \|\bm u\|_2^2 \|\bm A\|_{\text{op}} p \|\bm C\|_{\text{op}} p \|\bm D\|_{\text{op}} p \cdot \sqrt{p} \|\bm B\|_{\text{op}} p^{1/2} \\
    &+ \|\bm D\|_{\text{op}} \|\bm u\|_2^2 \|\bm A\|_{\text{op}} p \|\bm B\|_{\text{op}} p \|\bm D\|_{\text{op}} p \cdot \sqrt{p} \|\bm C\|_{\text{op}} p^{1/2} \\
    & + \|\bm D\|_{\text{op}} \|\bm u\|_2^2 \|\bm A\|_{\text{op}} p \|\bm B\|_{\text{op}} p \|\bm C\|_{\text{op}} p \|\bm D\|_{\text{op}}^2 p^2 \cdot \sqrt{p} \|\bm D\|_{\text{op}} p^{1/2} \\
    &+ \|\bm D\|_{\text{op}} \|\bm u\|_2^2 \|\bm A\|_{\text{op}} p \|\bm B\|_{\text{op}} p \|\bm C\|_{\text{op}} p \cdot \sqrt{p} \|\bm D\|_{\text{op}} p^{1/2} \\
    &\lcon \|\bm u\|_2^2(\|\bm A\|_{\text{op}} \|\bm B\|_{\text{op}} \|\bm C\|_{\text{op}} p^{5/2} + \|\bm A\|_{\text{op}} \|\bm B\|_{\text{op}} \|\bm C\|_{\text{op}} \|\bm D\|_{\text{op}}^2 p^{9/2}  \\
    &\qquad \qquad + \|\bm A\|_{\text{op}} \|\bm B\|_{\text{op}} \|\bm C\|_{\text{op}} \|\bm D\|_{\text{op}}^4 p^{6} + \|\bm A\|_{\text{op}} \|\bm B\|_{\text{op}} \|\bm C\|_{\text{op}} \|\bm D\|_{\text{op}}^6 p^{8})
\end{align*}

We are now ready to bound each term from the Taylor expansions.

\subsection{Bounding \texorpdfstring{$B_1$}{B1} terms}
\subsubsection{Swapping \texorpdfstring{$\bm Z^{(1)}$}{Z1}}
Note that
\begin{gather*}
    \|\bm u^{(a,1)}\|_2 \lcon \lambda n^{-1} \|\bbeta^{(2)}\|_2, \quad 
    \|\bm v^{(a,1)}\|_2 \lcon \lambda^{-2} \|\bbeta^{(2)}\|_{\text{op}}, \quad \|\bm A^{(a,1)}\|_{\text{op}} \lcon n^{-1} \lambda^{-1} \\
    \|\bm u^{(a,2)}\|_2 \lcon \lambda n^{-2} \|\bbeta^{(2)}\|_2, \quad 
    \|\bm v^{(a,2)}\|_2 \lcon \lambda^{-1} \|\bbeta^{(2)}\|_2 \\
    \|\bm A^{(a,2)}\|_{\text{op}} \lcon \lambda^{-2}, \quad 
    \|\bm B^{(a,2)}\|_{\text{op}} \lcon n^{-1} \lambda^{-1}
\end{gather*}
\paragraph{First term}
From the first moment bounds for $f_2$, we have
\begin{align*}
    &|\E_{1,k}[f_2(\bm Z_k^{(1)}; \bm S^{(a,1)}) - f_2(\tilde{\bm Z}_k^{(1)}; \bm S^{(a,1)})]|
    \\
    &\lcon \|\bm u^{(a,1)}\|_2 \|\bm v^{(a,1)}\|_2 (\|\bm A^{(a,1)}\|_{\text{op}} p^{1/2} + \|\bm A^{(a,1)}\|_{\text{op}}^2 p + \|\bm A^{(a,1)}\|_{\text{op}}^3 p^2)  \\
    &\lcon \lambda^{-2} n^{-3/2}  \|\bbeta^{(2)}\|_2^2 (1   + \lambda^{-1} n^{-1/2} + \lambda^{-2} n^{-1/2})
\end{align*}
and similarly, from the second moment bounds for $f_2$,
\begin{align*}
    \max\{\E_{1,k}[f_2(\bm Z_k^{(1)}; \bm S^{(a,1)})^2], \E_{1,k}[f_2(\tilde{\bm Z}_k^{(1)}; \bm S^{(a,1)})^2]\} &\lcon \|\bm u^{(a,1)}\|_2^2 \|\bm v^{(a,1)}\|_2^2 \\
    &\lcon \lambda^{-2} n^{-2} \|\bbeta^{(2)}\|_2^4 
\end{align*}
\paragraph{Second term}
From the first moment bounds for $f_5$, we have
\begin{align*}
    |&\E_{1,k}[f_5(\bm Z_k^{(1)}; \bm S^{(a,2)}) - f_5(\tilde{\bm Z}_k^{(1)}; \bm S^{(a,2)})]| \\
    &\lcon \|\bm u^{(a,2)}\|_2 \|\bm v^{(a,2)}\|_2 \|\bm A^{(a,2)}\|_{\text{op}} ( p^{1/2} +  \|\bm B^{(a,2)}\|_{\text{op}}^2 p^{5/2} +  \|\bm B^{(a,2)}\|_{\text{op}}^4 p^4 +  \|\bm B^{(a,2)}\|_{\text{op}}^6 p^6) \\
    &\lcon \lambda n^{-2} \|\bbeta^{(2)}\|_2 \cdot \lambda^{-1}\|\bbeta^{(2)}\|_2 \lambda^{-2}( p^{1/2} + n^{-2} \lambda^{-2} p^{5/2} + n^{-4} \lambda^{-4} p^4 + n^{-6}\lambda^{-6} p^6) \\
    &\lcon  \lambda^{-2} n^{-3/2} \|\bbeta^{(2)}\|_2^2 (1 + \lambda^{-2}  + \lambda^{-4} n^{-1/2} + \lambda^{-6} n^{-1/2}) 
\end{align*}
and similarly, from the second moment bounds for $f_5$, we have
\begin{align*}
    \max\{\E_{1,k}[f_2(\bm Z_k^{(1)}; \bm S^{(a,2)})^2], \E_{1,k}[f_2(\tilde{\bm Z}_k^{(1)}; \bm S^{(a,2)})^2]\} &\lcon \|\bm u^{(a,2)}\|_2^2 \|\bm v^{(a,2)}\|_2^2 \|\bm A^{(a,2)}\|_{\text{op}}^2 p^2 \\
    &\lcon \lambda^{-4}   n^{-2} \|\bbeta^{(2)}\|_2^4 
\end{align*}
\paragraph{Combined}
Combining all of these bounds with the bounds above demonstrate that
\begin{align*}
    |&\E[\varphi(B_1(\bm Z^{(1)}, \bm Z^{(2)}))] - \E[\varphi(B_1(\tilde{\bm Z}^{(1)}, \bm Z^{(2)}))]| \\
    &\lcon \lambda^{-2} n^{-1/2}\|\bbeta^{(2)}\|_2^2(  (1   + \lambda^{-1} n^{-1/2} + \lambda^{-2} + \lambda^{-4} n^{-1/2} + \lambda^{-6} n^{-1/2}) \\
    &\hspace{10em}+  n^{-1/2}(1 + \lambda^{-2}) \|\bbeta^{(2)}\|_2^2)  
\end{align*}

\subsubsection{Swapping \texorpdfstring{$\bm Z^{(2)}$}{Z2}}
The same bounds hold. That is,
\begin{align*}
    |&\E[\varphi(B_1(\tilde{\bm Z}^{(1)}, \bm Z^{(2)}))] - \E[\varphi(B_1(\tilde{\bm Z}^{(1)}, \tilde{\bm Z}^{(2)}))]| \\
    &\lcon \lambda^{-2} n^{-1/2}\|\bbeta^{(2)}\|_2^2(  (1   + \lambda^{-1} n^{-1/2} + \lambda^{-2} + \lambda^{-4} n^{-1/2} + \lambda^{-6} n^{-1/2}) \\
    &\hspace{10em}+  n^{-1/2}(1 + \lambda^{-2}) \|\bbeta^{(2)}\|_2^2)  
\end{align*}

\subsubsection{Universality for \texorpdfstring{$B_1$}{B1}}
Therefore, for any bounded function $\varphi$ with bounded first and second derivatives, we have
\begin{align*}
    |&\E[\varphi(B_1(\bm Z^{(1)}, \bm Z^{(2)}))] - \E[\varphi(B_1(\tilde{\bm Z}^{(1)}, \tilde{\bm Z}^{(2)}))]| \\
    &\lcon \lambda^{-2} n^{-1/2}\|\bbeta^{(2)}\|_2^2(  (1   + \lambda^{-1} n^{-1/2} + \lambda^{-2} + \lambda^{-4} n^{-1/2} + \lambda^{-6} n^{-1/2}) \\
    &\hspace{10em}+  n^{-1/2}(1 + \lambda^{-2}) \|\bbeta^{(2)}\|_2^2)  
\end{align*}

\subsection{Bounding \texorpdfstring{$B_3$}{B3} terms}
\subsubsection{Swapping \texorpdfstring{$\bm Z^{(1)}$}{Z1}}
We first bound
\begin{gather*}
    \|\bm u^{(b,1)}\|_2 \lcon \lambda^{-1} n^{-1} \|\bbeta^{(2)}\|_2, \quad \|\bm v^{(b,1)}\|_2 \lcon \|\tilde \bbeta\|_2 \\
    \|\bm u^{(b,2)}\|_2 \lcon \lambda^{-1} n^{-2} \|\bbeta^{(2)}\|_2, \quad 
    \|\bm v^{(b,2)}\|_2 \lcon \lambda^{-1} \|\bm X^{(1),\setminus k}\|_{\text{op}}^2 \|\tilde \bbeta\|_2, \\
    \|\bm A^{(b,2)}\|_{\text{op}} \lcon \lambda^{-1} n^{-1}, \quad  
    \|\bm u^{(b,3)}\|_2 \lcon \lambda^{-1} n^{-2} \|\bbeta^{(2)}\|_2, \quad 
    \|\bm v^{(b,3)}\|_2 \lcon \|\tilde \bbeta\|_2 \\
    \|\bm A^{(b,3)}\|_{\text{op}} \lcon \lambda^{-1}, \quad 
    \|\bm B^{(b,3)}\|_{\text{op}} \lcon \lambda^{-1} n^{-1} \\
    \|\bm u^{(b,4)}\|_2 \lcon n^{-2} \|\bbeta^{(2)}\|_2, \quad 
    \|\bm v^{(b,4)}\|_2 \lcon \lambda^{-2} \|\bm X^{(1),\setminus k}\|_{\text{op}}^2 \|\tilde \bbeta\|_2, \quad
    \|\bm A^{(b,4)}\|_{\text{op}} \lcon \lambda^{-1} n^{-1} \\
    \|\bm u^{(b,5)}\|_2 \lcon n^{-2} \|\bbeta^{(2)}\|_2, \quad 
    \|\bm v^{(b,5)}\|_2 \lcon \|\tilde \bbeta\|_2 \\
    \|\bm A^{(b,5)}\|_{\text{op}} \lcon \lambda^{-2}, \quad 
    \|\bm B^{(b,5)}\|_{\text{op}} \lcon \lambda^{-1} n^{-1}  \\
    \|\bm u^{(b,6)}\|_2 \lcon n^{-3} \|\bbeta^{(2)}\|_2, \quad 
    \|\bm v^{(b,6)}\|_2 \lcon \lambda^{-1} \|\bm X^{(1)\setminus k}\|_{\text{op}}^2 \|\tilde \bbeta\|_{2}\\
    \|\bm A^{(b,6)}\|_{\text{op}} \lcon \lambda^{-2}, \quad 
    \|\bm B^{(b,6)}\|_{\text{op}} \lcon \lambda^{-1} n^{-1} \\
    \|\bm u^{(b,7)}\|_2 \lcon n^{-3} \|\bbeta^{(2)}\|_2, \quad 
    \|\bm v^{(b,7)}\|_2 \lcon \|\tilde \bbeta\|_2 \\
    \|\bm A^{(b,7)}\|_{\text{op}} \lcon \lambda^{-2}, \quad 
    \|\bm B^{(b,7)}\|_{\text{op}} \lcon \lambda^{-1} \\
    \|\bm C^{(b,7)}\|_{\text{op}} \lcon \lambda^{-1} n^{-1}
\end{gather*}
\paragraph{First term}
Now, using the first moment bound for $f_1$, we have
\begin{equation*}
    |\E_{1,k}[f_1(\bm Z_k^{(1)}; \bm S^{(b,1)}) - f_1(\tilde{\bm Z}_k^{(1)}; \bm S^{(b,1)})]| = 0
\end{equation*}
and similarly, using the second moment bound for $f_1$, we have 
\begin{align*}
    \max&\{\E_{1,k}[f_1(\bm Z_k^{(1)}; \bm S^{(b,1)})^2], \E_{1,k}[f_1(\tilde{\bm Z}_k^{(1)}; \bm S^{(b,1)})^2]\} \\
    &\lcon \|\bm u^{(b,1)}\|_2^2 \|\bm v^{(b,1)}\|_2^2 \\
    &\leq \lambda^{-2} n^{-2}\|\bbeta^{(2)}\|_2^2 \|\tilde \bbeta\|_2^2
\end{align*}
\paragraph{Second term} Using the first moment bound for $f_2$, we have
\begin{align*}
    &|\E_{1,k}[f_2(\bm Z_k^{(1)}; \bm S^{(b,2)}) - f_2(\tilde{\bm Z}_k^{(1)}; \bm S^{(b,2)})]| \\&\lcon \|\bm u^{(b,2)}\|_2 \|\bm v^{(b,2)}\|_2 ( \|\bm A^{(b,2)}\|_{\text{op}} p^{1/2} + \|\bm A^{(b,2)}\|_{\text{op}}^2 p + \|\bm A^{(b,2)}\|_{\text{op}}^3 p^2) \\
    &\lcon \lambda^{-1} n^{-2} \|\bbeta^{(2)}\|_2 \cdot \lambda^{-1} \|\bm X^{(1),\setminus k}\|_{\text{op}}^2 \|\tilde \bbeta\|_2 \\
    &\qquad \cdot (\lambda^{-1} n^{-1} p^{1/2} + \lambda^{-2} n^{-2} p + \lambda^{-3} n^{-3} p^2) \\
    &\lcon \lambda^{-3} n^{-5/2} \|\bbeta^{(2)}\|_2 \|\tilde \bbeta\|_2  \|\bm X^{(1),\setminus k}\|_{\text{op}}^2  \\
    &\qquad \cdot (1+ \lambda^{-1} n^{-1/2} + \lambda^{-2} n^{-1/2})
\end{align*}
and using the second moment bound for $f_2$ yields
\begin{align*}
    \max\{&\E_{1,k}[f_2(\bm Z_k^{(1)}; \bm S^{(b,2)})^2], \E_{1,k}[f_2(\tilde{\bm Z}_k^{(1)}; \bm S^{(b,2)})^2]\} \\
    &\lcon \|\bm u^{(b,2)}\|_2^2 \|\bm v^{(b,2)}\|_2^2 \\
    &\lcon \lambda^{-4} n^{-4} \|\bbeta^{(2)}\|_2^2 \|\tilde \bbeta\|_2^2 \cdot \|\bm X^{(1),\setminus k}\|_{\text{op}}^4.
\end{align*}
\paragraph{Third term} Using the first moment bound for $f_4$, we have
\begin{align*}
    |&\E_{1,k}[f_4(\bm Z_k^{(1)}; \bm S^{(b,3)}) - f_4(\tilde{\bm Z}_k^{(1)}; \bm S^{(b,3)})]| \\
    &\lcon \|\bm u^{(b,3)}\|_2 \|\bm v^{(b,3)}\|_2 \|\bm A^{(b,3)}\|_{\text{op}}( p^{1/2} + \|\bm B^{(b,3)}\|_{\text{op}} p^{3/2} + \|\bm B^{(b,3)}\|_{\text{op}}^2 p^2 + \|\bm B^{(b,3)}\|_{\text{op}}^3 p^3) \\
    &\lcon \lambda^{-1} n^{-2} \|\bbeta^{(2)}\|_2 \|\tilde \bbeta\|_2   \lambda^{-1} (p^{1/2} + \lambda^{-1} n^{-1} p^{3/2} + \lambda^{-2} n^{-2} p^2 + \lambda^{-3} n^{-3} p^3) \\
    &\lcon \lambda^{-2} n^{-3/2} \|\bbeta^{(2)}\|_2 \|\tilde \bbeta\|_2 (1 + \lambda^{-1} + \lambda^{-2} n^{-1/2} + \lambda^{-3} n^{-1/2})
\end{align*}
and using the second moment bound for $f_4$ yields
\begin{align*}
    \max\{&\E_{1,k}[f_4(\bm Z_k^{(1)}; \bm S^{(b,3)})^2], \E_{1,k}[f_4(\tilde{\bm Z}_k^{(1)}; \bm S^{(b,3)})^2]\} \\
    &\lcon \|\bm u^{(b,3)}\|_2^2 \|\bm v^{(b,3)}\|_2^2 \|\bm A^{(b,3)}\|_{\text{op}}^2  p^2 \\
    &\lcon (\lambda^{-1} n^{-2} \|\bbeta^{(2)}\|_2)^2 \|\tilde \bbeta\|_2^2 \lambda^{-2} p^2 \\
    &= \lambda^{-4} n^{-2} \|\bbeta^{(2)}\|_2^2 \|\tilde \bbeta\|_2^2
\end{align*}
\paragraph{Fourth term} Using the first moment bound for $f_2$, we have
\begin{align*}
    |&\E_{1,k}[f_2(\bm Z_k^{(1)}; \bm S^{(b,4)}) - f_2(\tilde{\bm Z}_k^{(1)}; \bm S^{(b,4)})]| \\
    &\lcon \|\bm u^{(b,4)}\|_2 \|\bm v^{(b,4)}\|_2 ( \|\bm A^{(b,4)}\|_{\text{op}} p^{1/2} + \|\bm A^{(b,4)}\|_{\text{op}}^2 p + \|\bm A^{(b,4)}\|_{\text{op}}^3 p^2) \\
    &\lcon n^{-2} \|\bbeta^{(2)}\|_2 \lambda^{-2} \|\bm X^{(1),\setminus k}\|_{\text{op}}^2 \|\tilde \bbeta\|_2 (\lambda^{-1} n^{-1} p^{1/2} + \lambda^{-2} n^{-2} p + \lambda^{-3} n^{-3} p^2) \\
    &\lcon  \lambda^{-3} n^{-5/2}  \|\bbeta^{(2)}\|_2 \|\tilde \bbeta\|_2 \|\bm X^{(1),\setminus k}\|_{\text{op}}^2 (1 + \lambda^{-1} n^{-1/2} + \lambda^{-2} n^{-1/2})
\end{align*}
and using the second moment bound for $f_2$ yields
\begin{align*}
    \max\{&\E_{1,k}[f_2(\bm Z_k^{(1)}; \bm S^{(b,4)})^2], \E_{1,k}[f_2(\tilde{\bm Z}_k^{(1)}; \bm S^{(b,4)})^2]\} \\
    &\lcon \|\bm u^{(b,4)}\|_2^2 \|\bm v^{(b,4)}\|_2^2 \\
    &\lcon \lambda^{-4} n^{-4} \|\bbeta^{(2)}\|_2^2    \|\tilde \bbeta\|_2^2 \|\bm X^{(1),\setminus k}\|_{\text{op}}^4
\end{align*}
\paragraph{Fifth term} Using the first moment bound for $f_4$, we have
\begin{align*}
    |&\E_{1,k}[f_4(\bm Z_k^{(1)}; \bm S^{(b,5)}) - f_4(\tilde{\bm Z}_k^{(1)}; \bm S^{(b,5)})]| \\
    &\lcon \|\bm u^{(b,5)}\|_2 \|\bm v^{(b,5)}\|_2 \|\bm A^{(b,5)}\|_{\text{op}}( p^{1/2} + \|\bm B^{(b,5)}\|_{\text{op}} p^{3/2} + \|\bm B^{(b,5)}\|_{\text{op}}^2 p^2 + \|\bm B^{(b,5)}\|_{\text{op}}^3 p^3) \\
    &\lcon  n^{-2} \|\bbeta^{(2)}\|_2 \|\tilde \bbeta\|_2  \lambda^{-2} ( p^{1/2} + \lambda^{-1} n^{-1} p^{3/2} + \lambda^{-2} n^{-2} p^2 + \lambda^{-3} n^{-3} p^3) \\
    &\lcon  \lambda^{-2} n^{-3/2} \|\bbeta^{(2)}\|_2 \|\tilde \bbeta\|_2  ( 1 + \lambda^{-1}   + \lambda^{-2} n^{-1/2}  + \lambda^{-3} n^{-1/2})
\end{align*}
and using the second moment bound for $f_4$ yields
\begin{align*}
    \max\{&\E_{1,k}[f_4(\bm Z_k^{(1)}; \bm S^{(b,5)})^2], \E_{1,k}[f_4(\tilde{\bm Z}_k^{(1)}; \bm S^{(b,5)})^2]\} \\
    &\lcon \|\bm u^{(b,5)}\|_2^2 \|\bm v^{(b,5)}\|_2^2 \|\bm A^{(b,5)}\|_{\text{op}}^2  p^2 \\
    &\lcon \lambda^{-4}  n^{-2} \|\bbeta^{(2)}\|_2^2 \|\tilde \bbeta^{(2)}\|_2^2 
\end{align*}

\paragraph{Sixth term} Using the first moment bound for $f_5$, we have
\begin{align*}
    |&\E_{1,k}[f_5(\bm Z_k^{(1)}; \bm S^{(b,6)}) - f_5(\tilde{\bm Z}_k^{(1)}; \bm S^{(b,6)})]| \\
    &\lcon \|\bm u^{(b,6)}\|_2 \|\bm v^{(b,6)}\|_2 \|\bm A^{(b,6)}\|_{\text{op}}(p^{1/2} + \|\bm B^{(b,6)}\|_{\text{op}}^2 p^{5/2} + \|\bm B^{(b,6)}\|_{\text{op}}^4 p^4 + \|\bm B^{(b,6)}\|_{\text{op}}^6 p^6) \\
    &\lcon n^{-3} \|\bbeta^{(2)}\|_2 \lambda^{-1} \|\bm X^{(1)\setminus k}\|_{\text{op}}^2 \|\tilde \bbeta\|_{2} \lambda^{-2}(p^{1/2} + \lambda^{-2} n^{-2} p^{5/2} + \lambda^{-4} n^{-4} p^4 + \lambda^{-6} n^{-6} p^6) \\
    &\lcon \lambda^{-3} n^{-5/2} \|\bbeta^{(2)}\|_2 \|\tilde \bbeta\|_{2} \|\bm X^{(1)\setminus k}\|_{\text{op}}^2  (1 + \lambda^{-2}  + \lambda^{-4}  n^{-1/2}+ \lambda^{-6} n^{-1/2})
\end{align*}
and using the second moment bound for $f_5$ yields
\begin{align*}
    \max\{&\E_{1,k}[f_5(\bm Z_k^{(1)}; \bm S^{(b,6)})^2], \E_{1,k}[f_5(\tilde{\bm Z}_k^{(1)}; \bm S^{(b,6)})^2]\} \\
    &\lcon \|\bm u^{(b,6)}\|_2^2 \|\bm v^{(b,6)}\|_2^2 \|\bm A^{(b,6)}\|_{\text{op}}^2  p^2 \\
    &\lcon n^{-6} \|\bbeta^{(2)}\|_2^2 \cdot (\lambda^{-1} \|\bm X^{(1)\setminus k}\|_{\text{op}}^2 \|\tilde \bbeta\|_{2})^2 \lambda^{-4} p^2 \\
    &\lcon \lambda^{-6} n^{-4} \|\bbeta^{(2)}\|_2^2 \|\tilde \bbeta\|_{2}^2 \cdot \|\bm X^{(1)\setminus k}\|_{\text{op}}^4 
\end{align*}

\paragraph{Seventh term} Using the first moment bound for $f_7$, we have
\begin{align*}
    |&\E_{1,k}[f_7(\bm Z_k^{(1)}; \bm S^{(b,7)}) - f_7(\tilde{\bm Z}_k^{(1)}; \bm S^{(b,7)})]| \\
    &\lcon  \|\bm u^{(b,7)}\|_2 \|\bm v^{(b,7)}\|_2 \|\bm A^{(b,7)}\|_{\text{op}} \|\bm B^{(b,7)}\|_{\text{op}} \\
    &\qquad ( p^{3/2} +  \|\bm C^{(b,7)}\|_{\text{op}}^2 p^{7/2} +  \|\bm C^{(b,7)}\|_{\text{op}}^4 p^5 + \|\bm C^{(b,7)}\|_{\text{op}}^6 p^7) \\
    &\lcon n^{-3} \|\bbeta^{(2)}\|_2  \|\tilde \bbeta\|_2 \lambda^{-2} \lambda^{-1} ( p^{3/2} +  \lambda^{-2} n^{-2} p^{7/2} +  \lambda^{-4} n^{-4} p^5 + \lambda^{-6} n^{-6} p^7) \\
    &\lcon \lambda^{-3} n^{-3/2} \|\bbeta^{(2)}\|_2  \|\tilde \bbeta\|_2  ( 1 +  \lambda^{-2}  +  \lambda^{-4} n^{-1/2} + \lambda^{-6} n^{-1/2}) 
\end{align*}
and using the second moment bound for $f_7$ yields
\begin{align*}
    \max\{&\E_{1,k}[f_7(\bm Z_k^{(1)}; \bm S^{(b,7)})^2], \E_{1,k}[f_7(\tilde{\bm Z}_k^{(1)}; \bm S^{(b,7)})^2]\} \\
    &\lcon \|\bm u^{(b,7)}\|_2^2 \|\bm v^{(b,7)}\|_2^2 \|\bm A^{(b,7)}\|_{\text{op}}^2  \|\bm B^{(b,7)}\|_{\text{op}}^2   p^4 \\
    &\lcon \lambda^{-6} n^{-2} \|\bbeta^{(2)}\|_2^2 \|\tilde \bbeta\|_2^2 
\end{align*}

\paragraph{Combined}
It follows from the above bounds that
\begin{align*}
    |&\E[\varphi(B_3(\bm Z^{(1)}, \bm Z^{(2)}))] - \E[\varphi(B_3(\tilde{\bm Z}^{(1)}, \bm Z^{(2)}))]| \\
    &\lcon \E[\lambda^{-2} n^{-1}\|\bbeta^{(2)}\|_2^2 \|\tilde \bbeta\|_2^2 \\
    &\qquad + \lambda^{-3} n^{-3/2} \|\bbeta^{(2)}\|_2 \|\tilde \bbeta\|_2  \|\bm X^{(1),\setminus k}\|_{\text{op}}^2  \cdot (1+ \lambda^{-1} n^{-1/2} + \lambda^{-2} n^{-1/2}) \\
    &\qquad + \lambda^{-4} n^{-3} \|\bbeta^{(2)}\|_2^2 \|\tilde \bbeta\|_2^2 \cdot \|\bm X^{(1),\setminus k}\|_{\text{op}}^4 \\
    &\qquad + \lambda^{-2} n^{-1/2} \|\bbeta^{(2)}\|_2 \|\tilde \bbeta\|_2 (1 + \lambda^{-1} + \lambda^{-2} n^{-1/2} + \lambda^{-3} n^{-1/2}) \\
    &\qquad + \lambda^{-4} n^{-1} \|\bbeta^{(2)}\|_2^2 \|\tilde \bbeta\|_2^2 \\
    &\qquad + \lambda^{-3} n^{-3/2}  \|\bbeta^{(2)}\|_2 \|\tilde \bbeta\|_2 \|\bm X^{(1),\setminus k}\|_{\text{op}}^2 (1 + \lambda^{-1} n^{-1/2} + \lambda^{-2} n^{-1/2}) \\
    &\qquad + \lambda^{-4} n^{-3} \|\bbeta^{(2)}\|_2^2    \|\tilde \bbeta\|_2^2 \|\bm X^{(1),\setminus k}\|_{\text{op}}^4 \\
    &\qquad + \lambda^{-2} n^{-1/2} \|\bbeta^{(2)}\|_2 \|\tilde \bbeta\|_2  ( 1 + \lambda^{-1}   + \lambda^{-2} n^{-1/2}  + \lambda^{-3} n^{-1/2}) \\
    &\qquad + \lambda^{-4}  n^{-1} \|\bbeta^{(2)}\|_2^2 \|\tilde \bbeta^{(2)}\|_2^2 \\
    &\qquad + \lambda^{-3} n^{-3/2} \|\bbeta^{(2)}\|_2 \|\tilde \bbeta\|_{2} \|\bm X^{(1)\setminus k}\|_{\text{op}}^2  (1 + \lambda^{-2}  + \lambda^{-4}  n^{-1/2}+ \lambda^{-6} n^{-1/2}) \\
    &\qquad + \lambda^{-6} n^{-3} \|\bbeta^{(2)}\|_2^2 \|\tilde \bbeta\|_{2}^2 \cdot \|\bm X^{(1)\setminus k}\|_{\text{op}}^4 \\
    &\qquad + \lambda^{-3} n^{-1/2} \|\bbeta^{(2)}\|_2  \|\tilde \bbeta\|_2  ( 1 +  \lambda^{-2}  +  \lambda^{-4} n^{-1/2} + \lambda^{-6} n^{-1/2}) \\
    &\qquad +  \lambda^{-6} n^{-1} \|\bbeta^{(2)}\|_2^2 \|\tilde \bbeta\|_2^2] \\
    &\lcon \lambda^{-2} n^{-1}\|\bbeta^{(2)}\|_2^2 \|\tilde \bbeta\|_2^2  (1 + \lambda^{-2} + \lambda^{-4} + (\lambda^{-2} + \lambda^{-4}) n^{-2} \E[\|\bm X^{(1),\setminus k}\|_{\text{op}}^4]) \\
    &\qquad + \lambda^{-3} n^{-1/2} \|\bbeta^{(2)}\|_2\|\tilde \bbeta\|_2 ((1 + \lambda^{-1} + \lambda^{-2} + (\lambda^{-3} + \lambda^{-4}+ \lambda^{-6}) n^{-1/2}) \\
    &\qquad \qquad + (1 + \lambda^{-1} n^{-1/2} + \lambda^{-2} + \lambda^{-4} n^{-1/2} + \lambda^{-6} n^{-1/2})  n^{-1} \E[\|\bm X^{(1),\setminus k}\|_{\text{op}}^2]) \\
    &\lcon \lambda^{-2} n^{-1} \log^2 n \|\bbeta^{(2)}\|_2^2 \|\tilde \bbeta\|_2^2  (1 + \lambda^{-2} + \lambda^{-4}) \\
    &\qquad + \lambda^{-3} n^{-1/2} \log n \|\bbeta^{(2)}\|_2\|\tilde \bbeta\|_2 (1 + \lambda^{-1} + \lambda^{-2} + (\lambda^{-3} + \lambda^{-4}+ \lambda^{-6}) n^{-1/2})
\end{align*}

\subsubsection{Swapping \texorpdfstring{$\bm Z^{(2)}$}{Z2}}
Note that
\begin{gather*}
    \|\bm u^{(B,1)}\|_2 \lcon \lambda^{-1} n^{-2} \|\bbeta^{(2)}\|_2, \quad 
    \|\bm v^{(B,1)}\|_2 \lcon \lambda^{-1} \|\tilde{\bm X}^{(1)}\|_{\text{op}}^2 \|\tilde \bbeta\|_2 \\
    \|\bm A^{(B,1)}\|_{\text{op}} \lcon \lambda^{-1} n^{-1} \\
    \|\bm u^{(B,2)}\|_2 \lcon n^{-2} \|\bbeta^{(2)}\|_2, \quad 
    \|\bm v^{(B,2)}\|_2 \lcon \lambda^{-2} \|\tilde{\bm X}^{(1)}\|_{\text{op}}^2 \|\tilde \bbeta\|_2 \\
    \|\bm A^{(B,2)}\|_{\text{op}} \lcon \lambda^{-1} n^{-1} \\
    \|\bm u^{(B,3)}\|_2 \lcon n^{-3} \|\bbeta^{(2)}\|_2, \quad 
    \|\bm v^{(B,3)}\|_2 \lcon \lambda^{-1} \|\tilde{\bm X}^{(1)}\|_{\text{op}}^2 \|\tilde \bbeta\|_2 \\
    \|\bm A^{(B,3)}\|_{\text{op}} \lcon \lambda^{-2}, \quad 
    \|\bm B^{(B,3)}\|_{\text{op}} \lcon \lambda^{-1} n^{-1}
\end{gather*}
\paragraph{First term} Using the first moment bound for $f_2$, we have
\begin{align*}
    |&\E_{2,k}[f_2(\bm Z_k^{(2)}; \bm S^{(B,1)}) - f_2(\tilde{\bm Z}_k^{(2)}; \bm S^{(B,1)})]| \\
    &\lcon \|\bm u^{(B,1)}\|_2 \|\bm v^{(B,1)}\|_2 ( \|\bm A^{(B,1)}\|_{\text{op}} p^{1/2} + \|\bm A^{(B,1)}\|_{\text{op}}^2 p + \|\bm A^{(B,1)}\|_{\text{op}}^3 p^2)  \\
    &\lcon \lambda^{-1} n^{-2} \|\bbeta^{(2)}\|_2 \lambda^{-1} \|\tilde{\bm X}^{(1)}\|_{\text{op}}^2 \|\tilde \bbeta\|_2 ( \lambda^{-1} n^{-1} p^{1/2} + \lambda^{-2} n^{-2} p + \lambda^{-3} n^{-3} p^2) \\
    &\lcon \lambda^{-3} n^{-5/2} \|\bbeta^{(2)}\|_2   \|\tilde \bbeta\|_2 \|\tilde{\bm X}^{(1)}\|_{\text{op}}^2 ( 1 + \lambda^{-1} n^{-1/2} + \lambda^{-2} n^{-1/2})
\end{align*}
and using the second moment bound for $f_2$ yields
\begin{align*}
    \max\{&\E_{2,k}[f_2(\bm Z_k^{(2)}; \bm S^{(B,1)})^2], \E_{2,k}[f_2(\tilde{\bm Z}_k^{(2)}; \bm S^{(B,1)})^2]\} \\
    &\lcon \|\bm u^{(B,1)}\|_2^2 \|\bm v^{(B,1)}\|_2^2  \\
    &\lcon \lambda^{-4} n^{-4} \|\bbeta^{(2)}\|_2^2  \|\tilde \bbeta\|_2^2 \|\tilde{\bm X}^{(1)}\|_{\text{op}}^4
\end{align*}
\paragraph{Second term} Using the first moment bound for $f_2$, we have
\begin{align*}
    |&\E_{2,k}[f_2(\bm Z_k^{(2)}; \bm S^{(B,2)}) - f_2(\tilde{\bm Z}_k^{(2)}; \bm S^{(B,2)})]| \\
    &\lcon \|\bm u^{(B,2)}\|_2 \|\bm v^{(B,2)}\|_2 ( \|\bm A^{(B,2)}\|_{\text{op}} p^{1/2} + \|\bm A^{(B,2)}\|_{\text{op}}^2 p + \|\bm A^{(B,2)}\|_{\text{op}}^3 p^2)  \\
    &\lcon  n^{-2} \|\bbeta^{(2)}\|_2\lambda^{-2} \|\tilde{\bm X}^{(1)}\|_{\text{op}}^2 \|\tilde \bbeta\|_2 (\lambda^{-1} n^{-1} p^{1/2} + \lambda^{-2} n^{-2} p + \lambda^{-3} n^{-3} p^2) \\
    &\lcon \lambda^{-3} n^{-5/2} \|\bbeta^{(2)}\|_2 \|\tilde \bbeta\|_2 \|\tilde{\bm X}^{(1)}\|_{\text{op}}^2  (1 + \lambda^{-1} n^{-1/2} + \lambda^{-2} n^{-1/2}) 
\end{align*}
and using the second moment bound for $f_2$ yields
\begin{align*}
    \max\{&\E_{2,k}[f_2(\bm Z_k^{(2)}; \bm S^{(B,2)})^2], \E_{2,k}[f_2(\tilde{\bm Z}_k^{(2)}; \bm S^{(B,2)})^2]\} \\
    &\lcon \|\bm u^{(B,2)}\|_2^2 \|\bm v^{(B,2)}\|_2^2 \\
    &\lcon  \lambda^{-4} n^{-4} \|\bbeta^{(2)}\|_2^2  \|\tilde \bbeta\|_2^2 \|\tilde{\bm X}^{(1)}\|_{\text{op}}^4
\end{align*}
\paragraph{Third term} Using the first moment bound for $f_5$, we have
\begin{align*}
    |&\E_{2,k}[f_5(\bm Z_k^{(2)}; \bm S^{(B,3)}) - f_5(\tilde{\bm Z}_k^{(2)}; \bm S^{(B,3)})]| \\
    &\lcon \|\bm u^{(B,3)}\|_2 \|\bm v^{(B,3)}\|_2 \|\bm A^{(B,3)}\|_{\text{op}}(p^{1/2} + \|\bm B^{(B,3)}\|_{\text{op}}^2 p^{5/2} + \|\bm B^{(B,3)}\|_{\text{op}}^4 p^4 + \|\bm B^{(B,3)}\|_{\text{op}}^6 p^6) \\
    &\lcon n^{-3} \|\bbeta^{(2)}\|_2 \lambda^{-1} \|\tilde{\bm X}^{(1)}\|_{\text{op}}^2 \|\tilde \bbeta\|_2 \lambda^{-2} (p^{1/2} + \lambda^{-2} n^{-2} p^{5/2} + \lambda^{-4} n^{-4} p^4 + \lambda^{-6} n^{-6} p^6) \\
    &\lcon \lambda^{-3} n^{-5/2} \|\bbeta^{(2)}\|_2  \|\tilde \bbeta\|_2 \|\tilde{\bm X}^{(1)}\|_{\text{op}}^2  (1 + \lambda^{-2} + \lambda^{-4} n^{-1/2} + \lambda^{-6} n^{-1/2})
\end{align*}
and using the second moment bound for $f_5$ yields
\begin{align*}
    \max\{&\E_{2,k}[f_5(\bm Z_k^{(2)}; \bm S^{(B,3)})^2], \E_{2,k}[f_5(\tilde{\bm Z}_k^{(2)}; \bm S^{(B,3)})^2]\} \\
    &\lcon \lambda^{-6} n^{-4} \|\bbeta^{(2)}\|_2^2 \|\tilde \bbeta\|_2^2 \|\tilde{\bm X}^{(1)}\|_{\text{op}}^4 
    \end{align*}
\paragraph{Combined} It follows from the above bounds that
\begin{align*}
    |&\E[\varphi(B_3(\tilde{\bm Z}^{(1)}, \bm Z^{(2)}))] - \E[\varphi(B_3(\tilde{\bm Z}^{(1)}, \tilde{\bm Z}^{(2)}))]| \\
    &\lcon \E[\lambda^{-3} n^{-3/2} \|\bbeta^{(2)}\|_2   \|\tilde \bbeta\|_2 \|\tilde{\bm X}^{(1)}\|_{\text{op}}^2 ( 1 + \lambda^{-1} n^{-1/2} + \lambda^{-2} n^{-1/2}) \\
    &\qquad + \lambda^{-4} n^{-3} \|\bbeta^{(2)}\|_2^2  \|\tilde \bbeta\|_2^2 \|\tilde{\bm X}^{(1)}\|_{\text{op}}^4 \\
    &\qquad + \lambda^{-3} n^{-3/2} \|\bbeta^{(2)}\|_2 \|\tilde \bbeta\|_2 \|\tilde{\bm X}^{(1)}\|_{\text{op}}^2  (1 + \lambda^{-1} n^{-1/2} + \lambda^{-2} n^{-1/2}) \\
    &\qquad + \lambda^{-4} n^{-3} \|\bbeta^{(2)}\|_2^2  \|\tilde \bbeta\|_2^2 \|\tilde{\bm X}^{(1)}\|_{\text{op}}^4 \\
    &\qquad + \lambda^{-3} n^{-3/2} \|\bbeta^{(2)}\|_2  \|\tilde \bbeta\|_2 \|\tilde{\bm X}^{(1)}\|_{\text{op}}^2  (1 + \lambda^{-2} + \lambda^{-4} n^{-1/2} + \lambda^{-6} n^{-1/2})\\
    &\qquad + \lambda^{-6} n^{-3} \|\bbeta^{(2)}\|_2^2 \|\tilde \bbeta\|_2^2 \|\tilde{\bm X}^{(1)}\|_{\text{op}}^4 ] \\
    &\lcon \lambda^{-4} n^{-3} \|\bbeta^{(2)}\|_2^2 \|\tilde \bbeta\|_2^2 \E[\|\tilde{\bm X}^{(1)}\|_{\text{op}}^4] (1 + \lambda^{-2}) \\
    &\qquad + \lambda^{-3} n^{-3/2} \|\bbeta^{(2)}\|_2 \|\tilde \bbeta\|_2 \E[\|\tilde{\bm X}^{(1)}\|_{\text{op}}^2] \\
    &\hspace{4em}(1 + \lambda^{-1} n^{-1/2} + \lambda^{-2} + \lambda^{-4} n^{-1/2} + \lambda^{-6} n^{-1/2}) \\
    &\lcon \lambda^{-4} n^{-1} \log^2 n \|\bbeta^{(2)}\|_2^2 \|\tilde \bbeta\|_2^2 (1 + \lambda^{-2}) \\
    &\qquad + \lambda^{-3} n^{-1/2} \log n \|\bbeta^{(2)}\|_2 \|\tilde \bbeta\|_2  (1 + \lambda^{-1} n^{-1/2} + \lambda^{-2} + \lambda^{-4} n^{-1/2} + \lambda^{-6} n^{-1/2}) 
    \end{align*}

\subsubsection{Universality for \texorpdfstring{$B_3$}{B3}}
Therefore, for any bounded function $\varphi$ with bounded first and second derivatives, we have
\begin{align*}
    |&\E[\varphi(B_3(\bm Z^{(1)}, \bm Z^{(2)}))] - \E[\varphi(B_3(\tilde{\bm Z}^{(1)}, \tilde{\bm Z}^{(2)}))]| \\
    &\lcon \lambda^{-2} n^{-1} \log^2 n \|\bbeta^{(2)}\|_2^2 \|\tilde \bbeta\|_2^2  (1 + \lambda^{-2} + \lambda^{-4}) \\
    &\qquad + \lambda^{-3} n^{-1/2} \log n \|\bbeta^{(2)}\|_2\|\tilde \bbeta\|_2 (1 + \lambda^{-1} + \lambda^{-2} + (\lambda^{-3} + \lambda^{-4}+ \lambda^{-6}) n^{-1/2})\\
    &\qquad + \lambda^{-4} n^{-1} \log^2 n \|\bbeta^{(2)}\|_2^2 \|\tilde \bbeta\|_2^2 (1 + \lambda^{-2}) \\
    &\qquad + \lambda^{-3} n^{-1/2} \log n \|\bbeta^{(2)}\|_2 \|\tilde \bbeta\|_2  (1 + \lambda^{-1} n^{-1/2} + \lambda^{-2} + \lambda^{-4} n^{-1/2} + \lambda^{-6} n^{-1/2})  \\
    &\lcon \lambda^{-2} n^{-1} \log^2 n \cdot \|\bbeta^{(2)}\|_2^2 \|\tilde \bbeta\|_2^2  (1 + \lambda^{-2} + \lambda^{-4}) \\
    &\qquad + \lambda^{-3} n^{-1/2} \log n \cdot \|\bbeta^{(2)}\|_2\|\tilde \bbeta\|_2 (1 + \lambda^{-1} + \lambda^{-2} + (\lambda^{-3} + \lambda^{-4}+ \lambda^{-6}) n^{-1/2})
    \end{align*}

\subsection{Bounding \texorpdfstring{$B_2$}{B2} terms}
\subsubsection{Swapping \texorpdfstring{$\bZ^{(1)}$}{Z1}}
We first bound
\begin{gather*}
    \|\bm u^{(c,1)}\|_2 \lcon n^{-2} \lambda^{-2} \|\tilde{\bm X}^{(1),\setminus k}\|_{\text{op}}^2 \|\tilde \bbeta\|_2, \quad 
    \|\bm v^{(c,1)}\|_2 \lcon \|\tilde \bbeta\|_2 \\
    \|\bm u^{(c,2)}\|_2 \lcon n^{-3} \lambda^{-2}  \|\tilde{\bm X}^{(1),\setminus k}\|_{\text{op}}^2 \|\tilde \bbeta\|_2, \quad 
    \|\bm v^{(c,2)}\|_2 = \lambda^{-1} \|\tilde{\bm X}^{(1),\setminus k}\|_{\text{op}}^2 \|\tilde \bbeta\|_2,  \\
    \|\bm A^{(c,2)}\|_{\text{op}}= n^{-1} \lambda^{-1}, \quad \|\bm u^{(c,3)}\|_2 = n^{-3}\lambda^{-2} \|\tilde{\bm X}^{(1),\setminus k}\|_{\text{op}}^2 \|\tilde \bbeta\|_2, \quad 
    \|\bm v^{(c,3)}\|_2 = \|\tilde \bbeta\|_2 \\
    \|\bm A^{(c,3)}\|_{\text{op}} = \lambda^{-1} , \quad 
    \|\bm B^{(c,3)}\|_{\text{op}} = n^{-1} \lambda^{-1} \\
    \|\bm u^{(c,4)}\|_2 = n^{-3} \lambda^{-1} \|\tilde{\bm X}^{(1),\setminus k}\|_{\text{op}}^2 \|\tilde \bbeta\|_2, \quad 
    \|\bm v^{(c,4)}\|_2 = \|\tilde \bbeta\|_2 \\
    \|\bm A^{(c,4)}\|_{\text{op}} = \lambda^{-2} , \quad 
    \|\bm B^{(c,4)}\|_{\text{op}} = n^{-1} \lambda^{-1} \\
    \|\bm u^{(c,5)}\|_2 = n^{-4} \lambda^{-1} \|\tilde{\bm X}^{(1),\setminus k}\|_{\text{op}}^2 \|\tilde \bbeta\|_2, \quad 
    \|\bm v^{(c,5)}\|_2 = \lambda^{-1} \|\tilde{\bm X}^{(1),\setminus k}\|_{\text{op}}^2 \|\tilde \bbeta\|_2 \\
    \|\bm A^{(c,5)}\|_{\text{op}} = \lambda^{-2}, \quad 
    \|\bm B^{(c,5)}\|_{\text{op}} = n^{-1} \lambda^{-1} \\
    \|\bm u^{(c,6)}\|_2 = n^{-4} \lambda^{-1} \|\tilde{\bm X}^{(1),\setminus k}\|_{\text{op}}^2 \|\tilde \bbeta\|_2, \quad 
    \|\bm v^{(c,6)}\|_2 = \|\tilde \bbeta\|_2 \\
    \|\bm A^{(c,6)}\|_{\text{op}} = \lambda^{-2}  , \quad 
    \|\bm B^{(c,6)}\|_{\text{op}} = \lambda^{-1}, \quad 
    \|\bm C^{(c,6)}\|_{\text{op}} = n^{-1} \lambda^{-1} \\
    \|\bm u^{(c,7)}\|_2 = n^{-2} \|\tilde \bbeta\|_2, \quad 
    \|\bm v^{(c,7)}\|_2 = \|\tilde \bbeta\|_2 ,\quad
    \|\bm A^{(c,7)}\|_{\text{op}} = \lambda^{-2}  \\
    \|\bm u^{(c,8)}\|_2 =   \|\tilde \bbeta\|_2, \quad 
    \|\bm A^{(c,8)}\|_{\text{op}} = n^{-3} \lambda^{-1}, \quad 
    \|\bm B^{(c,8)}\|_{\text{op}} = \lambda^{-2}, \quad 
    \|\bm C^{(c,8)}\|_{\text{op}} = n^{-1} \lambda^{-1} \\
    \|\bm u^{(c,9)}\|_2 =  \|\tilde \bbeta\|_2, \quad  
    \|\bm A^{(c,9)}\|_{\text{op}} = n^{-4} \lambda^{-1} \\ 
    \|\bm B^{(c,9)}\|_{\text{op}} = \lambda^{-2}, \quad 
    \|\bm C^{(c,9)}\|_{\text{op}} = \lambda^{-1} , \quad 
    \|\bm D^{(c,9)}\|_{\text{op}} = n^{-1} \lambda^{-1} 
\end{gather*}
\paragraph{First term}
Now, using the first moment bound for $f_1$, we have
\begin{equation*}
    |\E_{1,k}[f_1(\bm Z_k^{(1)}; \bm S^{(c,1)}) - f_1(\tilde{\bm Z}_k^{(1)}; \bm S^{(c,1)})]| = 0
\end{equation*}
and similarly, using the second moment bound for $f_1$, we have 
\begin{align*}
    \max&\{\E_{1,k}[f_1(\bm Z_k^{(1)}; \bm S^{(c,1)})^2], \E_{1,k}[f_1(\tilde{\bm Z}_k^{(1)}; \bm S^{(c,1)})^2]\} \\
    &\lcon \|\bm u^{(c,1)}\|_2^2 \|\bm v^{(c,1)}\|_2^2 \\
    &\lcon \lambda^{-4}n^{-4}  \|\tilde{\bm X}^{(1),\setminus k}\|_{\text{op}}^4 \|\tilde \bbeta\|_2^4
\end{align*}

\paragraph{Second term} Using the first moment bound for $f_2$, we have
\begin{align*}
    |&\E_{1,k}[f_2(\bm Z_k^{(1)}; \bm S^{(c,2)}) - f_2(\tilde{\bm Z}_k^{(1)}; \bm S^{(c,2)})]| \\
    &\lcon \|\bm u^{(c,2)}\|_2 \|\bm v^{(c,2)}\|_2 ( \|\bm A^{(c,2)}\|_{\text{op}} p^{1/2} + \|\bm A^{(c,2)}\|_{\text{op}}^2 p + \|\bm A^{(c,2)}\|_{\text{op}}^3 p^2) \\
    &\lcon n^{-3} \lambda^{-2}  \|\tilde{\bm X}^{(1),\setminus k}\|_{\text{op}}^2 \|\tilde \bbeta\|_2 \lambda^{-1} \|\tilde{\bm X}^{(1),\setminus k}\|_{\text{op}}^2 \|\tilde \bbeta\|_2 ( n^{-1} \lambda^{-1} p^{1/2} + n^{-2} \lambda^{-2} p + n^{-3} \lambda^{-3} p^2) \\
    &\lcon n^{-7/2} \lambda^{-3}   \|\tilde \bbeta\|_2^2  \|\tilde{\bm X}^{(1),\setminus k}\|_{\text{op}}^4 ( \lambda^{-1} + \lambda^{-2}n^{-1/2}  + \lambda^{-3} n^{-1/2} )
\end{align*}
and using the second moment bound for $f_2$ yields
\begin{align*}
    \max\{&\E_{1,k}[f_2(\bm Z_k^{(1)}; \bm S^{(c,2)})^2], \E_{1,k}[f_2(\tilde{\bm Z}_k^{(1)}; \bm S^{(c,2)})^2]\} \\
    &\lcon \|\bm u^{(c,2)}\|_2^2 \|\bm v^{(c,2)}\|_2^2  \\
    &\lcon ( n^{-3} \lambda^{-2}  \|\tilde{\bm X}^{(1),\setminus k}\|_{\text{op}}^2 \|\tilde \bbeta\|_2)^2 \cdot (\lambda^{-1} \|\tilde{\bm X}^{(1),\setminus k}\|_{\text{op}}^2 \|\tilde \bbeta\|_2)^2 \\
    &\lcon \lambda^{-6} n^{-6} \|\tilde \bbeta\|_2^4 \|\tilde{\bm X}^{(1),\setminus k}\|_{\text{op}}^8
\end{align*}
\paragraph{Third term} Using the first moment bound for $f_4$, we have
\begin{align*}
    |&\E_{1,k}[f_4(\bm Z_k^{(1)}; \bm S^{(c,3)}) - f_4(\tilde{\bm Z}_k^{(1)}; \bm S^{(c,3)})]| \\
    &\lcon \|\bm u^{(c,3)}\|_2 \|\bm v^{(c,3)}\|_2 \|\bm A^{(c,3)}\|_{\text{op}}( p^{1/2} + \|\bm B^{(c,3)}\|_{\text{op}} p^{3/2} + \|\bm B^{(c,3)}\|_{\text{op}}^2 p^2 + \|\bm B^{(c,3)}\|_{\text{op}}^3 p^3) \\
    &\lcon n^{-3}\lambda^{-2} \|\tilde{\bm X}^{(1),\setminus k}\|_{\text{op}}^2 \|\tilde \bbeta\|_2 \|\tilde \bbeta\|_2 \lambda^{-1}( p^{1/2} + n^{-1} \lambda^{-1} p^{3/2} + n^{-2} \lambda^{-2} p^2 + n^{-3} \lambda^{-3} p^3) \\
    &\lcon n^{-5/2}\lambda^{-3} \|\tilde \bbeta\|_2^2 \|\tilde{\bm X}^{(1),\setminus k}\|_{\text{op}}^2 ( 1 +  \lambda^{-1} + n^{-1/2} \lambda^{-2} + n^{-1/2} \lambda^{-3}) \\
\end{align*}
and using the second moment bound for $f_4$ yields
\begin{align*}
    \max\{&\E_{1,k}[f_4(\bm Z_k^{(1)}; \bm S^{(c,3)})^2], \E_{1,k}[f_4(\tilde{\bm Z}_k^{(1)}; \bm S^{(c,3)})^2]\} \\
    &\lcon \|\bm u^{(c,3)}\|_2^2 \|\bm v^{(c,3)}\|_2^2 \|\bm A^{(c,3)}\|_{\text{op}}^2  p^2 \\
    &\lcon (n^{-3}\lambda^{-2} \|\tilde{\bm X}^{(1),\setminus k}\|_{\text{op}}^2 \|\tilde \bbeta\|_2)^2 (\|\tilde \bbeta\|_2)^2 \lambda^{-2} p^2 \\
    &\lcon n^{-4}\lambda^{-6} \|\tilde \bbeta\|_2^4  \|\tilde{\bm X}^{(1),\setminus k}\|_{\text{op}}^4 
\end{align*}

\paragraph{Fourth term} Using the first moment bound for $f_4$, we have
\begin{align*}
    |&\E_{1,k}[f_4(\bm Z_k^{(1)}; \bm S^{(c,4)}) - f_4(\tilde{\bm Z}_k^{(1)}; \bm S^{(c,4)})]| \\
    &\lcon \|\bm u^{(c,4)}\|_2 \|\bm v^{(c,4)}\|_2 \|\bm A^{(c,4)}\|_{\text{op}}( p^{1/2} + \|\bm B^{(c,4)}\|_{\text{op}} p^{3/2} + \|\bm B^{(c,4)}\|_{\text{op}}^2 p^2 + \|\bm B^{(c,4)}\|_{\text{op}}^3 p^3) \\
    &\lcon n^{-3} \lambda^{-1} \|\tilde{\bm X}^{(1),\setminus k}\|_{\text{op}}^2 \|\tilde \bbeta\|_2 \|\tilde \bbeta\|_2 \lambda^{-2} ( p^{1/2} + n^{-1} \lambda^{-1} p^{3/2} + n^{-2} \lambda^{-2} p^2 + n^{-3} \lambda^{-3} p^3) \\
    &\lcon n^{-5/2} \lambda^{-3} \|\tilde \bbeta\|_2^2 \|\tilde{\bm X}^{(1),\setminus k}\|_{\text{op}}^2   ( 1 +  \lambda^{-1} + n^{-1/2} \lambda^{-2} + n^{-1/2} \lambda^{-3})
\end{align*}
and using the second moment bound for $f_4$ yields
\begin{align*}
    \max\{&\E_{1,k}[f_4(\bm Z_k^{(1)}; \bm S^{(c,4)})^2], \E_{1,k}[f_4(\tilde{\bm Z}_k^{(1)}; \bm S^{(c,4)})^2]\} \\
    &\lcon \|\bm u^{(c,4)}\|_2^2 \|\bm v^{(c,4)}\|_2^2 \|\bm A^{(c,4)}\|_{\text{op}}^2  p^2 \\
    &\lcon (n^{-3} \lambda^{-1} \|\tilde{\bm X}^{(1),\setminus k}\|_{\text{op}}^2 \|\tilde \bbeta\|_2)^2 \|\tilde \bbeta\|_2^2 \lambda^{-4}  p^2 \\
    &\lcon n^{-4} \lambda^{-6} \|\tilde{\bm X}^{(1),\setminus k}\|_{\text{op}}^4 \|\tilde \bbeta\|_2^4   
\end{align*}

\paragraph{Fifth term} Using the first moment bound for $f_5$, we have
\begin{align*}
    |&\E_{1,k}[f_5(\bm Z_k^{(1)}; \bm S^{(c,5)}) - f_5(\tilde{\bm Z}_k^{(1)}; \bm S^{(c,5)})]| \\
    &\lcon \|\bm u^{(c,5)}\|_2 \|\bm v^{(c,5)}\|_2 \|\bm A^{(c,5)}\|_{\text{op}}(p^{1/2} + \|\bm B^{(c,5)}\|_{\text{op}}^2 p^{5/2} + \|\bm B^{(c,5)}\|_{\text{op}}^4 p^4 + \|\bm B^{(c,5)}\|_{\text{op}}^6 p^6) \\
    &\lcon n^{-4} \lambda^{-1} \|\tilde{\bm X}^{(1),\setminus k}\|_{\text{op}}^2 \|\tilde \bbeta\|_2 \lambda^{-1} \|\tilde{\bm X}^{(1),\setminus k}\|_{\text{op}}^2 \|\tilde \bbeta\|_2 \lambda^{-2}  \\
    &\qquad (p^{1/2} + n^{-2} \lambda^{-2} p^{5/2} + n^{-4} \lambda^{-4} p^4 + n^{-6} \lambda^{-6} p^6) \\
    &\lcon n^{-7/2} \lambda^{-4}  \|\tilde \bbeta\|_2^2 \|\tilde{\bm X}^{(1),\setminus k}\|_{\text{op}}^4  (1 + \lambda^{-2} + n^{-1/2} \lambda^{-4}  + n^{-1/2} \lambda^{-6}) 
\end{align*}
and using the second moment bound for $f_5$ yields
\begin{align*}
    \max\{&\E_{1,k}[f_5(\bm Z_k^{(1)}; \bm S^{(c,5)})^2], \E_{1,k}[f_5(\tilde{\bm Z}_k^{(1)}; \bm S^{(c,5)})^2]\} \\
    &\lcon \|\bm u^{(c,5)}\|_2^2 \|\bm v^{(c,5)}\|_2^2 \|\bm A^{(c,5)}\|_{\text{op}}^2  p^2 \\
    &\lcon (n^{-4} \lambda^{-1} \|\tilde{\bm X}^{(1),\setminus k}\|_{\text{op}}^2 \|\tilde \bbeta\|_2)^2 (\lambda^{-1} \|\tilde{\bm X}^{(1),\setminus k}\|_{\text{op}}^2 \|\tilde \bbeta\|_2)^2\lambda^{-4} p^2 \\
    &\lcon n^{-6} \lambda^{-8}  \|\tilde \bbeta\|_2^4  \|\tilde{\bm X}^{(1),\setminus k}\|_{\text{op}}^8 
\end{align*}

\paragraph{Sixth term} Using the first moment bound for $f_7$, we have
\begin{align*}
    |&\E_{1,k}[f_7(\bm Z_k^{(1)}; \bm S^{(c,6)}) - f_7(\tilde{\bm Z}_k^{(1)}; \bm S^{(c,6)})]| \\
    &\lcon  \|\bm u^{(c,6)}\|_2 \|\bm v^{(c,6)}\|_2 \|\bm A^{(c,6)}\|_{\text{op}} \|\bm B^{(c,6)}\|_{\text{op}}  \\
    &\qquad ( p^{3/2} +  \|\bm C^{(c,6)}\|_{\text{op}}^2 p^{7/2} +  \|\bm C^{(c,6)}\|_{\text{op}}^4 p^5 + \|\bm C^{(c,6)}\|_{\text{op}}^6 p^7) \\
    &\lcon  n^{-4} \lambda^{-1} \|\tilde{\bm X}^{(1),\setminus k}\|_{\text{op}}^2 \|\tilde \bbeta\|_2 \|\tilde \bbeta\|_2 \lambda^{-2} \lambda^{-1} ( p^{3/2} +  n^{-2} \lambda^{-2} p^{7/2} +  n^{-4} \lambda^{-4} p^5 + n^{-6} \lambda^{-6} p^7) \\
    &\lcon  n^{-5/2} \lambda^{-4}   \|\tilde \bbeta\|_2^2 \|\tilde{\bm X}^{(1),\setminus k}\|_{\text{op}}^2 ( 1 +  \lambda^{-2}  +  n^{-1/2} \lambda^{-4}  + n^{-1/2} \lambda^{-6} ) 
\end{align*}
and using the second moment bound for $f_7$ yields
\begin{align*}
    \max\{&\E_{1,k}[f_7(\bm Z_k^{(1)}; \bm S^{(c,6)})^2], \E_{1,k}[f_7(\tilde{\bm Z}_k^{(1)}; \bm S^{(c,6)})^2]\} \\
    &\lcon \|\bm u^{(c,6)}\|_2^2 \|\bm v^{(c,6)}\|_2^2 \|\bm A^{(c,6)}\|_{\text{op}}^2  \|\bm B^{(c,6)}\|_{\text{op}}^2   p^4 \\
    &\lcon (n^{-4} \lambda^{-1} \|\tilde{\bm X}^{(1),\setminus k}\|_{\text{op}}^2 \|\tilde \bbeta\|_2)^2 \|\tilde \bbeta\|_2^2 \lambda^{-4} \lambda^{-2} p^4 \\
    &\lcon n^{-4} \lambda^{-8} \|\tilde{\bm X}^{(1),\setminus k}\|_{\text{op}}^4 \|\tilde \bbeta\|_2^4  
\end{align*}

\paragraph{Seventh term} Using the first moment bound for $f_3$, we have
\begin{align*}
    |&\E_{1,k}[f_3(\bm Z_k^{(1)}; \bm S^{(c,7)}) - f_3(\tilde{\bm Z}_k^{(1)}; \bm S^{(c,7)})]| \\
    &\lcon  \|\bm u^{(c,7)}\|_2 \|\bm v^{(c,7)}\|_2 \|\bm A^{(c,7)}\|_{\text{op}} p^{1/2} \\
    &\lcon n^{-2} \|\tilde \bbeta\|_2 \|\tilde \bbeta\|_2 \lambda^{-2} p^{1/2} \\
    &\lcon n^{-3/2} \lambda^{-2} \|\tilde \bbeta\|_2^2 
\end{align*}
and using the second moment bound for $f_3$ yields
\begin{align*}
    \max\{&\E_{1,k}[f_3(\bm Z_k^{(1)}; \bm S^{(c,7)})^2], \E_{1,k}[f_3(\tilde{\bm Z}_k^{(1)}; \bm S^{(c,7)})^2]\} \\
    &\lcon \|\bm u^{(c,7)}\|_2^2 \|\bm v^{(c,7)}\|_2^2 \|\bm A^{(c,7)}\|_{\text{op}}^2 p^2 \\
    &\lcon n^{-4} \|\tilde \bbeta\|_2^2 \|\tilde \bbeta\|_2^2 \lambda^{-4} p^2 \\
    &\lcon n^{-2} \lambda^{-4} \|\tilde \bbeta\|_2^4
\end{align*}

\paragraph{Eighth term} Using the first moment bound for $f_6$, we have
\begin{align*}
    |&\E_{1,k}[f_6(\bm Z_k^{(1)}; \bm S^{(c,8)}) - f_6(\tilde{\bm Z}_k^{(1)}; \bm S^{(c,8)})]| \\
    &\lcon \|\bm u^{(c,8)}\|_2^2\|\bm A^{(c,8)}\|_{\text{op}} \|\bm B^{(c,8)}\|_{\text{op}} ( p^{3/2} + \|\bm C^{(c,8)}\|_{\text{op}} p^{5/2} + \|\bm C^{(c,8)}\|_{\text{op}}^2 p^3 + \|\bm C^{(c,8)}\|_{\text{op}}^3 p^{7/2})  \\
    &\lcon \|\tilde \bbeta\|_2^2 n^{-3} \lambda^{-1} \lambda^{-2} ( p^{3/2} + n^{-1} \lambda^{-1} p^{5/2} + n^{-2} \lambda^{-2} p^3 + n^{-3} \lambda^{-3} p^{7/2})  \\
    &\lcon n^{-3/2} \lambda^{-3} \|\tilde \bbeta\|_2^2 (1 + \lambda^{-1}  + n^{-1/2} \lambda^{-2} + n^{-1} \lambda^{-3}) 
\end{align*}
and using the second moment bound for $f_6$ yields
\begin{align*}
    \max\{&\E_{1,k}[f_6(\bm Z_k^{(1)}; \bm S^{(c,8)})^2], \E_{1,k}[f_6(\tilde{\bm Z}_k^{(1)}; \bm S^{(c,8)})^2]\} \\
    &\lcon \|\bm u^{(c,8)}\|_2^4 \|\bm A^{(c,8)}\|_{\text{op}}^2 \|\bm B^{(c,8)}\|_{\text{op}}^2 p^4 \\
    &\lcon \|\tilde \bbeta\|_2^4 (n^{-3} \lambda^{-1})^2 \lambda^{-4} p^4 \\
    &\lcon n^{-2} \lambda^{-6} \|\tilde \bbeta\|_2^4 
\end{align*}

\paragraph{Ninth term} Using the first moment bound for $f_8$, we have
\begin{align*}
    |&\E_{1,k}[f_8(\bm Z_k^{(1)}; \bm S^{(c,8)}) - f_8(\tilde{\bm Z}_k^{(1)}; \bm S^{(c,8)})]| \\
    &\lcon \|\bm u\|_2^2 \|\bm A\|_{\text{op}} \|\bm B\|_{\text{op}} \|\bm C\|_{\text{op}} (p^{5/2} + \|\bm D\|_{\text{op}}^2 p^{9/2} + \|\bm D\|_{\text{op}}^4 p^{6} + \|\bm D\|_{\text{op}}^6 p^{8}) \\
    &\lcon \|\tilde \bbeta\|_2^2 n^{-4} \lambda^{-1}  \lambda^{-2}  \lambda^{-1} (p^{5/2} + n^{-2} \lambda^{-2}  p^{9/2} + n^{-4} \lambda^{-4} p^{6} + n^{-6} \lambda^{-6} p^{8}) \\
    &\lcon n^{-3/2} \lambda^{-4} \|\tilde \bbeta\|_2^2 (1 + \lambda^{-2}  + n^{-1/2} \lambda^{-4}  + n^{-1/2} \lambda^{-6} )
\end{align*}
and using the second moment bound for $f_8$ yields
\begin{align*}
    \max\{&\E_{1,k}[f_8(\bm Z_k^{(1)}; \bm S^{(c,8)})^2], \E_{1,k}[f_8(\tilde{\bm Z}_k^{(1)}; \bm S^{(c,8)})^2]\} \\
    &\lcon \|\bm u\|_2^4 \|\bm A\|_{\text{op}}^2 \|\bm B\|_{\text{op}}^2 \|\bm C\|_{\text{op}}^2 p^6 \\
    &\lcon \|\tilde \bbeta\|_2^4 n^{-8} \lambda^{-2} \lambda^{-4} \lambda^{-2} p^6 \\
    &\lcon n^{-2} \lambda^{-8} \|\tilde \bbeta\|_2^4 
\end{align*}
\paragraph{Combined} It follows from the above bounds that
\begin{align*}
    |&\E[\varphi(B_3(\bm Z^{(1)}, \bm Z^{(2)}))] - \E[\varphi(B_3(\tilde{\bm Z}^{(1)}, \bm Z^{(2)}))]| \\
    &\lcon \E[\lambda^{-4}n^{-3}  \|\tilde{\bm X}^{(1),\setminus k}\|_{\text{op}}^4 \|\tilde \bbeta\|_2^4 \\
    &\qquad + n^{-5/2} \lambda^{-3}   \|\tilde \bbeta\|_2^2  \|\tilde{\bm X}^{(1),\setminus k}\|_{\text{op}}^4 ( \lambda^{-1} + \lambda^{-2}n^{-1/2}  + \lambda^{-3} n^{-1/2} ) \\
    &\qquad + \lambda^{-6} n^{-5} \|\tilde \bbeta\|_2^4 \|\tilde{\bm X}^{(1),\setminus k}\|_{\text{op}}^8 \\
    &\qquad + n^{-3/2}\lambda^{-3} \|\tilde \bbeta\|_2^2 \|\tilde{\bm X}^{(1),\setminus k}\|_{\text{op}}^2 ( 1 +  \lambda^{-1} + n^{-1/2} \lambda^{-2} + n^{-1/2} \lambda^{-3}) \\
    &\qquad + n^{-3}\lambda^{-6} \|\tilde \bbeta\|_2^4  \|\tilde{\bm X}^{(1),\setminus k}\|_{\text{op}}^4 \\
    &\qquad + n^{-3/2} \lambda^{-3} \|\tilde \bbeta\|_2^2 \|\tilde{\bm X}^{(1),\setminus k}\|_{\text{op}}^2   ( 1 +  \lambda^{-1} + n^{-1/2} \lambda^{-2} + n^{-1/2} \lambda^{-3})\\
    &\qquad + n^{-3} \lambda^{-6} \|\tilde{\bm X}^{(1),\setminus k}\|_{\text{op}}^4 \|\tilde \bbeta\|_2^4 \\
    &\qquad + n^{-5/2} \lambda^{-4}  \|\tilde \bbeta\|_2^2 \|\tilde{\bm X}^{(1),\setminus k}\|_{\text{op}}^4  (1 + \lambda^{-2} + n^{-1/2} \lambda^{-4}  + n^{-1/2} \lambda^{-6}) \\
    &\qquad + n^{-5} \lambda^{-8}  \|\tilde \bbeta\|_2^4  \|\tilde{\bm X}^{(1),\setminus k}\|_{\text{op}}^8 \\
    &\qquad + n^{-3/2} \lambda^{-4}   \|\tilde \bbeta\|_2^2 \|\tilde{\bm X}^{(1),\setminus k}\|_{\text{op}}^2 ( 1 +  \lambda^{-2}  +  n^{-1/2} \lambda^{-4}  + n^{-1/2} \lambda^{-6} \\
    &\qquad + n^{-3} \lambda^{-8} \|\tilde{\bm X}^{(1),\setminus k}\|_{\text{op}}^4 \|\tilde \bbeta\|_2^4   \\
    &\qquad +n^{-1/2} \lambda^{-2} \|\tilde \bbeta\|_2^2  +  n^{-1} \lambda^{-4} \|\tilde \bbeta\|_2^4 \\
    &\qquad + n^{-1/2} \lambda^{-3} \|\tilde \bbeta\|_2^2 (1 + \lambda^{-1}  + n^{-1/2} \lambda^{-2} + n^{-1} \lambda^{-3})  + n^{-1} \lambda^{-6} \|\tilde \bbeta\|_2^4 \\
    &\qquad + n^{-1/2} \lambda^{-4} \|\tilde \bbeta\|_2^2 (1 + \lambda^{-2}  + n^{-1/2} \lambda^{-4}  + n^{-1/2} \lambda^{-6} ) + n^{-1} \lambda^{-8} \|\tilde \bbeta\|_2^4 ] \\
    &\lcon n^{-1} (\lambda^{-4} + \lambda^{-6} + \lambda^{-8}) \|\tilde \bbeta\|_2^4 + n^{-3} (\lambda^{-4} + \lambda^{-6} + \lambda^{-8}) \|\tilde \bbeta\|_2^4 \E[\|\tilde{\bm X}^{(1),\setminus k}\|_{\text{op}}^4] \\
    &\qquad + n^{-5} (\lambda^{-6} + \lambda^{-8}) \|\tilde \bbeta\|_2^4 \E[\|\tilde{\bm X}^{(1),\setminus k}\|_{\text{op}}^8] \\
    &\qquad + n^{-1/2} \lambda^{-2} \|\tilde \bbeta\|_2^2 (1 + \lambda^{-1} + \lambda^{-2} + \lambda^{-4} + n^{-1/2}(\lambda^{-3} + \lambda^{-6} + \lambda^{-7} + \lambda^{-8})) \\
    &\qquad + n^{-3/2} \lambda^{-3} \|\tilde \bbeta\|_2^2 (1 + \lambda^{-1} + \lambda^{-3} + n^{-1/2} (\lambda^{-2} + \lambda^{-5} + \lambda^{-7})) \E[\|\tilde{\bm X}^{(1),\setminus k}\|_{\text{op}}^2] \\
    &\qquad + n^{-5/2} \lambda^{-3} \|\tilde \bbeta\|_2^2 (\lambda^{-1} + \lambda^{-3} + n^{-1/2} (\lambda^{-2} + \lambda^{-5} + \lambda^{-7})) \E[\|\tilde{\bm X}^{(1),\setminus}\|_{\text{op}}^4] \\
    &\lcon n^{-1} \log^2 n \lambda^{-4}  \|\tilde \bbeta\|_2^4  + n^{-1} \log^4 n (\lambda^{-6} + \lambda^{-8}) \|\tilde \bbeta\|_2^4  \\
    &\qquad + n^{-1/2} \lambda^{-2} \|\tilde \bbeta\|_2^2 (1 + n^{-1/2} \lambda^{-7} ) \\
    &\qquad + n^{-1/2} \log n \lambda^{-3} \|\tilde \bbeta\|_2^2  \\
    &\qquad + n^{-1/2} \log^2 n \lambda^{-3} \|\tilde \bbeta\|_2^2 (\lambda^{-1} + \lambda^{-3} + n^{-1/2} (\lambda^{-2} + \lambda^{-5} + \lambda^{-7}))
\end{align*}

\subsubsection{Swapping \texorpdfstring{$\bZ^{(2)}$}{Z2}}
We first bound
\begin{gather*}
    \|\bm u^{(C,1)}\|_2 \lcon n^{-3}  \lambda^{-1} \|\tilde{\bm X}^{(1)}\|_{\text{op}}^2 \|\tilde \bbeta\|_2, \quad 
    \|\bm v^{(C,1)}\|_2 =   \lambda^{-2}  \|\tilde{\bm X}^{(1)}\|_{\text{op}}^2 \|\tilde \bbeta\|_2 \\
    \|\bm A^{(C,1)}\|_{\text{op}} = n^{-1} \lambda^{-1}   \\
    \|\bm u^{(C,2)}\|_2 = n^{-4}  \lambda^{-1}  \|\tilde{\bm X}^{(1)}\|_{\text{op}}^2 \|\tilde \bbeta\|_2, \quad 
    \|\bm v^{(C,2)}\|_2 =  \lambda^{-1}  \|\tilde{\bm X}^{(1)}\|_{\text{op}}^2 \|\tilde \bbeta\|_2 \\
    \|\bm A^{(C,2)}\|_{\text{op}} =  \lambda^{-2}, \quad 
    \|\bm B^{(C,2)}\|_{\text{op}} = n^{-1} \lambda^{-1}  
\end{gather*}
\paragraph{First term}
From the first moment bounds for $f_2$, we have
\begin{align*}
    |&\E_{1,k}[f_2(\bm Z_k^{(1)}; \bm S^{(C,1)}) - f_2(\tilde{\bm Z}_k^{(1)}; \bm S^{(C,1)})]| \\
    &\lcon \|\bm u^{(C,1)}\|_2 \|\bm v^{(C,1)}\|_2 (\|\bm A^{(C,1)}\|_{\text{op}} p^{1/2} + \|\bm A^{(C,1)}\|_{\text{op}}^2 p + \|\bm A^{(C,1)}\|_{\text{op}}^3 p^2)  \\
    &\lcon n^{-3}  \lambda^{-1} \|\tilde{\bm X}^{(1)}\|_{\text{op}}^2 \|\tilde \bbeta\|_2 \lambda^{-2}  \|\tilde{\bm X}^{(1)}\|_{\text{op}}^2 \|\tilde \bbeta\|_2 (n^{-1} \lambda^{-1} p^{1/2} + n^{-2} \lambda^{-2} p + n^{-3} \lambda^{-3} p^2) \\
    &\lcon n^{-7/2}  \lambda^{-4} \|\tilde{\bm X}^{(1)}\|_{\text{op}}^4 \|\tilde \bbeta\|_2^2 (1 + n^{-1/2} \lambda^{-1} + n^{-1/2} \lambda^{-2})
\end{align*}
and similarly, from the second moment bounds for $f_2$,
\begin{align*}
    \max\{&\E_{1,k}[f_2(\bm Z_k^{(1)}; \bm S^{(C,1)})^2], \E_{1,k}[f_2(\tilde{\bm Z}_k^{(1)}; \bm S^{(C,1)})^2]\} \\
    &\lcon \|\bm u^{(C,1)}\|_2^2 \|\bm v^{(C,1)}\|_2^2 \\
    &\lcon n^{-6}  \lambda^{-6} \|\tilde{\bm X}^{(1)}\|_{\text{op}}^8 \|\tilde \bbeta\|_2^4 
\end{align*}
\paragraph{Second term}
From the first moment bounds for $f_5$, we have
\begin{align*}
    |&\E_{1,k}[f_5(\bm Z_k^{(1)}; \bm S^{(C,2)}) - f_5(\tilde{\bm Z}_k^{(1)}; \bm S^{(C,2)})]| \\
    &\lcon \|\bm u^{(C,2)}\|_2 \|\bm v^{(C,2)}\|_2 \|\bm A^{(C,2)}\|_{\text{op}} ( p^{1/2} +  \|\bm B^{(C,2)}\|_{\text{op}}^2 p^{5/2} +  \|\bm B^{(C,2)}\|_{\text{op}}^4 p^4 +  \|\bm B^{(C,2)}\|_{\text{op}}^6 p^6) \\
    &\lcon n^{-4}  \lambda^{-1}  \|\tilde{\bm X}^{(1)}\|_{\text{op}}^2 \|\tilde \bbeta\|_2 \lambda^{-1}  \|\tilde{\bm X}^{(1)}\|_{\text{op}}^2 \|\tilde \bbeta\|_2 \lambda^{-2}  \\
    &\qquad ( p^{1/2} +  n^{-2} \lambda^{-2} p^{5/2} +  n^{-4} \lambda^{-4} p^4 +  n^{-6} \lambda^{-6} p^6) \\
    &\lcon n^{-7/2}  \lambda^{-4}  \|\tilde{\bm X}^{(1)}\|_{\text{op}}^4 \|\tilde \bbeta\|_2^2   ( 1 +   \lambda^{-2}  +  n^{-1/2} \lambda^{-4} +  n^{-1/2} \lambda^{-6})
\end{align*}
and similarly, from the second moment bounds for $f_5$, we have
\begin{align*}
    \max\{&\E_{1,k}[f_2(\bm Z_k^{(1)}; \bm S^{(C,2)})^2], \E_{1,k}[f_2(\tilde{\bm Z}_k^{(1)}; \bm S^{(C,2)})^2]\} \\
    &\lcon \|\bm u^{(C,2)}\|_2^2 \|\bm v^{(C,2)}\|_2^2 \|\bm A^{(C,2)}\|_{\text{op}}^2 p^2 \\
    &\lcon (n^{-4}  \lambda^{-1}  \|\tilde{\bm X}^{(1)}\|_{\text{op}}^2 \|\tilde \bbeta\|_2)^2(\lambda^{-1}  \|\tilde{\bm X}^{(1)}\|_{\text{op}}^2 \|\tilde \bbeta\|_2)^2 \lambda^{-4} p^2 \\
    &\lcon n^{-6}  \lambda^{-8}  \|\tilde{\bm X}^{(1)}\|_{\text{op}}^8 \|\tilde \bbeta\|_2^4   
\end{align*}
\paragraph{Combined}
Combining all of these bounds with the bounds above demonstrate that
\begin{align*}
    |&\E[\varphi(B_1(\tilde{\bm Z}^{(1)}, \bm Z^{(2)}))] - \E[\varphi(B_1(\tilde{\bm Z}^{(1)}, \tilde{\bm Z}^{(2)}))]| \\
    &\lcon n^{-5/2} \lambda^{-4} \|\tilde \bbeta\|_2^2 \E[\|\tilde{\bm X}^{(1)}\|_{\text{op}}^4] (1 +  \lambda^{-2} n^{-1/2} (\lambda^{-1} + \lambda^{-4} + \lambda^{-6}))  \\
    &\qquad + n^{-5} (\lambda^{-6} + \lambda^{-8}) \|\tilde \bbeta\|_2^4 \E[\|\tilde{\bm X}^{(1)}\|_{\text{op}}^8] \\
    &\lcon n^{-1/2} \log^2 n \cdot  \lambda^{-4} \|\tilde \bbeta\|_2^2 (1 +  \lambda^{-2} n^{-1/2} (\lambda^{-1} + \lambda^{-4} + \lambda^{-6}))  \\
    &\qquad + n^{-1} \log^4 n \cdot  (\lambda^{-6} + \lambda^{-8}) \|\tilde \bbeta\|_2^4 
\end{align*}

\subsubsection{Universality for \texorpdfstring{$B_2$}{B2}}
Therefore, for any bounded function $\varphi$ with bounded first and second derivatives, we have
\begin{align*}
    |&\E[\varphi(B_2(\bm Z^{(1)}, \bm Z^{(2)}))] - \E[\varphi(B_2(\tilde{\bm Z}^{(1)}, \tilde{\bm Z}^{(2)}))]| \\
    &\lcon n^{-1} \log^2 n \lambda^{-4}  \|\tilde \bbeta\|_2^4  + n^{-1} \log^4 n (\lambda^{-6} + \lambda^{-8}) \|\tilde \bbeta\|_2^4  \\
    &\qquad + n^{-1/2} \lambda^{-2} \|\tilde \bbeta\|_2^2 (1 + n^{-1/2} \lambda^{-7} ) \\
    &\qquad + n^{-1/2} \log n \lambda^{-3} \|\tilde \bbeta\|_2^2  \\
    &\qquad + n^{-1/2} \log^2 n \lambda^{-3} \|\tilde \bbeta\|_2^2 (\lambda^{-1} + \lambda^{-3} + n^{-1/2} (\lambda^{-2} + \lambda^{-5} + \lambda^{-7})) \\
    &\qquad + n^{-1/2} \log^2 n \cdot  \lambda^{-4} \|\tilde \bbeta\|_2^2 (1 +  \lambda^{-2} n^{-1/2} (\lambda^{-1} + \lambda^{-4} + \lambda^{-6}))  \\
    &\qquad + n^{-1} \log^4 n \cdot  (\lambda^{-6} + \lambda^{-8}) \|\tilde \bbeta\|_2^4  \\
    &\lcon n^{-1} \log^2 n \lambda^{-4}  \|\tilde \bbeta\|_2^4  + n^{-1} \log^4 n (\lambda^{-6} + \lambda^{-8}) \|\tilde \bbeta\|_2^4  \\
    &\qquad + n^{-1/2} \lambda^{-2} \|\tilde \bbeta\|_2^2 (1 + n^{-1/2} \lambda^{-7} ) \\
    &\qquad + n^{-1/2} \log n \lambda^{-3} \|\tilde \bbeta\|_2^2  \\
    &\qquad + n^{-1/2} \log^2 n \lambda^{-3} \|\tilde \bbeta\|_2^2 (\lambda^{-1} + \lambda^{-3} + n^{-1/2} (\lambda^{-2} + \lambda^{-5} + \lambda^{-7})) 
\end{align*}

\subsection{Variance term}
The proof of universality for the variance term follows by the same procedure as the bias terms; notably, the variance term is remarkably similar to $B_3$ and thus its proof is also quite similar. We include it here for completeness.

Let $\bm U = \tilde{\bm U} \cdot \frac{p}{n}$, where $\tilde{\bm U} \sim \text{Unif}(\mathcal{S}^{d-1})$ is drawn independently of everything else. Define $V(\bW^{(1)}, \bW^{(2)})$ analogously to the bias quantities $B_j(\bW^{(1)}, \bW^{(2)})$ and then define $\tilde V(\bW^{(1)}, \bW^{(2)}) = \sigma^2 \bm U^\top (\hat \bSigma + \lambda \bm I)^{-2} \hat \bSigma \bSigma^{(2)} \bm U$.
With these quantities, we can express
\begin{align*}
    V(\bm Z^{(1)}, \bm Z^{(2)}) &= \frac{\sigma^2}{n} \Tr[(\hat \bSigma + \lambda \bm I)^{-2} \hat \bSigma \bSigma^{(2)}] \\
    &= \E_U[\sigma^2 \bm U^\top (\hat \bSigma + \lambda \bm I)^{-2} \hat \bSigma \bSigma^{(2)} \bm U] \\
    &= \E_{U}[\tilde V(\bm Z^{(1)}, \bm Z^{(2)})]
\end{align*}
and similarly $V(\tilde \bZ^{(1)}, \tilde \bZ^{(2)}) = \E_U[\tilde V(\tilde \bZ^{(1)}, \tilde \bZ^{(2)})]$,
where the expectation $\E_U$ is taken with respect to $\bm U$.
Then, it suffices to show that, for each $i = 1, \dots, n$ and bounded function $\varphi$ with bounded first and second derivative, we have
\begin{equation*}
    |\E[\varphi(\E_{U}[\tilde V(\bm Z^{(1)}, \bm Z^{(2)})] - \E[\varphi(\E_{U}[\tilde V(\bm Z^{(1)}, \bm Z^{(2)})])]| \to 0.
\end{equation*}
\subsubsection{Lindeberg and Sherman-Morrison expansions}
As before, we have
\begin{align*}
    &|\E[\varphi(\E_{U}[\tilde V(\bm Z^{(1)}, \bm Z^{(2)})])] - \E[\varphi(\E_{U}[\tilde V(\bm Z^{(1)}, \bm Z^{(2)})])]|\\
    &\leq \|\varphi'\|_\infty \sum_{k=1}^{n_1} \E[|\E_U\E_{1,k}[\tilde V(\bm Z^{(1),k}, \bm Z^{(2)}) - \tilde V(\bm Z^{(1),k-1}, \bm Z^{(2)})]|] \\
    &\qquad + \|\varphi''\|_\infty \sum_{k=1}^{n_1} \E[\E_U\E_{1,k}[(\tilde V(\bm Z^{(1),k}, \bm Z^{(2)}) - \tilde V(\bm Z^{(1),\setminus k}, \bm Z^{(2)}))^2]] \\
    &\qquad + \|\varphi''\|_\infty  \sum_{k=1}^{n_1} \E[\E_U\E_{1,k}[(\tilde V(\bm Z^{(1),k-1}, \bm Z^{(2)}) - \tilde V(\bm Z^{(1),\setminus k}, \bm Z^{(2)}))^2]] \\
    &|\E[\varphi(\E_U[\tilde V(\bm Z^{(1)}, \bm Z^{(2)})])] - \E[\varphi(\E_U[\tilde V(\tilde{\bm Z}^{(1)}, \tilde{\bm Z}^{(2)})])]| \\
    &\leq \|\varphi'\|_\infty \sum_{k=1}^{n_2} \E[|\E_U\E_{2,k}[(\tilde V(\tilde{\bm Z}^{(1)}, \bm Z^{(2),k}))] - \E[(\tilde V(\tilde{\bm Z}^{(1)}, \bm Z^{(2),k-1}))]|] \\
    &\qquad + \|\varphi''\|_\infty \sum_{k=1}^{n_2} \E[\E_U\E_{2,k}[(\tilde V(\tilde{\bm Z}^{(1)}, \bm Z^{(2),k}) - \tilde V(\tilde{\bm Z}^{(1)}, \bm Z^{(2),\setminus k}))^2]] \\
    &\qquad + \|\varphi''\|_\infty  \sum_{k=1}^{n_2} \E[\E_U\E_{2,k}[(\tilde V(\tilde{\bm Z}^{(1)}, \bm Z^{(2),k-1}) - \tilde V(\tilde{\bm Z}^{(1)}, \bm Z^{(2),\setminus k}))^2]]
\end{align*}
Recall that, by Sherman-Morrison, we can express
\begin{align*}
    (\hat \bSigma_k + \lambda \bm I)^{-1} &= (\hat \bSigma_{\setminus k} + \tfrac{1}{n} \tilde{\bm X}_k^{(1)} \tilde{\bm X}_k^{(1)\top} + \lambda \bm I)^{-1} \\
    &= (\hat \bSigma_{\setminus k} + \lambda \bm I)^{-1} - \frac{1}{n} \frac{(\hat \bSigma_{\setminus k} + \lambda \bm I)^{-1} \tilde{\bm X}_k^{(1)} \tilde{\bm X}_k^{(1)\top} (\hat \bSigma_{\setminus k} + \lambda \bm I)^{-1}}{1 + \frac{1}{n} \tilde{\bm X}_k^{(1)\top} (\hat \bSigma_{\setminus k} + \lambda \bm I)^{-1} \tilde{\bm X}_k^{(1)}}.
\end{align*}
and therefore
\begin{align*}
    \tilde V(\bm Z^{(1),k}, \bm Z^{(2)}) &= \sigma^2 \bm U^\top \biggl((\hat \bSigma_{\setminus k} + \lambda \bm I)^{-1} - \frac{1}{n} \frac{(\hat \bSigma_{\setminus k} + \lambda \bm I)^{-1} \tilde{\bm X}_k^{(1)} \tilde{\bm X}_k^{(1)\top} (\hat \bSigma_{\setminus k} + \lambda \bm I)^{-1}}{1 + \frac{1}{n} \tilde{\bm X}_k^{(1)\top} (\hat \bSigma_{\setminus k} + \lambda \bm I)^{-1} \tilde{\bm X}_k^{(1)}} \biggr)^2 \\
    &\qquad \cdot \biggl(\frac{\bm X^{(1),\setminus k\top} \bm X^{(1),\setminus k}}{n} +  \frac{\bm X^{(2)\top} \bm X^{(2)}}{n} + \frac{\tilde{\bm X}_k^{(1)} \tilde{\bm X}_k^{(1)\top}}{n}\biggr) \bSigma^{(2)} \bm U
\end{align*}
and therefore
\begin{align*}
    &V_1(\bm Z^{(1),k}, \bm Z^{(2)}) - V_1(\bm Z^{(1),\setminus k}, \bm Z^{(2)}) \\
    &=  \sigma^2 \bm U^\top (\hat \bSigma_{\setminus k} + \lambda \bm I)^{-2} \biggl(\frac{\tilde{\bm X}_k^{(1)} \tilde{\bm X}_k^{(1)\top}}{n}\biggr) \bSigma^{(2)} \bm U \\
    &\qquad - \sigma^2 \bm U^\top (\hat \bSigma_{\setminus k} + \lambda \bm I)^{-1} \biggl(\frac{1}{n} \frac{(\hat \bSigma_{\setminus k} + \lambda \bm I)^{-1} \tilde{\bm X}_k^{(1)} \tilde{\bm X}_k^{(1)\top} (\hat \bSigma_{\setminus k} + \lambda \bm I)^{-1}}{1 + \frac{1}{n} \tilde{\bm X}_k^{(1)\top} (\hat \bSigma_{\setminus k} + \lambda \bm I)^{-1} \tilde{\bm X}_k^{(1)}}\biggr) \\
    &\hspace{4em} \biggl(\frac{\bm X^{(1),\setminus k\top} \bm X^{(1),\setminus k}}{n} +  \frac{\bm X^{(2)\top} \bm X^{(2)}}{n} \biggr) \bSigma^{(2)} \bm U \\
    &\qquad - \sigma^2 \bm U^\top (\hat \bSigma_{\setminus k} + \lambda \bm I)^{-1} \biggl(\frac{1}{n} \frac{(\hat \bSigma_{\setminus k} + \lambda \bm I)^{-1} \tilde{\bm X}_k^{(1)} \tilde{\bm X}_k^{(1)\top} (\hat \bSigma_{\setminus k} + \lambda \bm I)^{-1}}{1 + \frac{1}{n} \tilde{\bm X}_k^{(1)\top} (\hat \bSigma_{\setminus k} + \lambda \bm I)^{-1} \tilde{\bm X}_k^{(1)}}\biggr) \\
    &\hspace{4em}\biggl(\frac{\tilde{\bm X}_k^{(1)} \tilde{\bm X}_k^{(1)\top}}{n}\biggr) \bSigma^{(2)} \bm U \\
    &\qquad - \sigma^2 \bm U^\top \biggl(\frac{1}{n} \frac{(\hat \bSigma_{\setminus k} + \lambda \bm I)^{-1} \tilde{\bm X}_k^{(1)} \tilde{\bm X}_k^{(1)\top} (\hat \bSigma_{\setminus k} + \lambda \bm I)^{-1}}{1 + \frac{1}{n} \tilde{\bm X}_k^{(1)\top} (\hat \bSigma_{\setminus k} + \lambda \bm I)^{-1} \tilde{\bm X}_k^{(1)}}\biggr) (\hat \bSigma_{\setminus k} + \lambda \bm I)^{-1} \\
    &\hspace{4em}\biggl(\frac{\bm X^{(1),\setminus k\top} \bm X^{(1),\setminus k}}{n} +  \frac{\bm X^{(2)\top} \bm X^{(2)}}{n} \biggr) \bSigma^{(2)} \bm U \\
    &\qquad - \sigma^2 \bm U^\top \biggl(\frac{1}{n} \frac{(\hat \bSigma_{\setminus k} + \lambda \bm I)^{-1} \tilde{\bm X}_k^{(1)} \tilde{\bm X}_k^{(1)\top} (\hat \bSigma_{\setminus k} + \lambda \bm I)^{-1}}{1 + \frac{1}{n} \tilde{\bm X}_k^{(1)\top} (\hat \bSigma_{\setminus k} + \lambda \bm I)^{-1} \tilde{\bm X}_k^{(1)}}\biggr) \\
    &\hspace{4em} (\hat \bSigma_{\setminus k} + \lambda \bm I)^{-1} \biggl(\frac{\tilde{\bm X}_k^{(1)} \tilde{\bm X}_k^{(1)\top}}{n}\biggr) \bSigma^{(2)} \bm U \\
    &\qquad + \sigma^2 \bm U^\top \biggl(\frac{1}{n} \frac{(\hat \bSigma_{\setminus k} + \lambda \bm I)^{-1} \tilde{\bm X}_k^{(1)} \tilde{\bm X}_k^{(1)\top} (\hat \bSigma_{\setminus k} + \lambda \bm I)^{-1}}{1 + \frac{1}{n} \tilde{\bm X}_k^{(1)\top} (\hat \bSigma_{\setminus k} + \lambda \bm I)^{-1} \tilde{\bm X}_k^{(1)}}\biggr) \\
    &\hspace{4em} \biggl(\frac{1}{n} \frac{(\hat \bSigma_{\setminus k} + \lambda \bm I)^{-1} \tilde{\bm X}_k^{(1)} \tilde{\bm X}_k^{(1)\top} (\hat \bSigma_{\setminus k} + \lambda \bm I)^{-1}}{1 + \frac{1}{n} \tilde{\bm X}_k^{(1)\top} (\hat \bSigma_{\setminus k} + \lambda \bm I)^{-1} \tilde{\bm X}_k^{(1)}}\biggr) \\
    &\hspace{4em}\biggl(\frac{\bm X^{(1),\setminus k\top} \bm X^{(1),\setminus k}}{n} +  \frac{\bm X^{(2)\top} \bm X^{(2)}}{n} \biggr) \bSigma^{(2)} \bm U \\
    &\qquad + \sigma^2 \bm U^\top \biggl(\frac{1}{n} \frac{(\hat \bSigma_{\setminus k} + \lambda \bm I)^{-1} \tilde{\bm X}_k^{(1)} \tilde{\bm X}_k^{(1)\top} (\hat \bSigma_{\setminus k} + \lambda \bm I)^{-1}}{1 + \frac{1}{n} \tilde{\bm X}_k^{(1)\top} (\hat \bSigma_{\setminus k} + \lambda \bm I)^{-1} \tilde{\bm X}_k^{(1)}}\biggr) \\
    &\hspace{4em} \biggl(\frac{1}{n} \frac{(\hat \bSigma_{\setminus k} + \lambda \bm I)^{-1} \tilde{\bm X}_k^{(1)} \tilde{\bm X}_k^{(1)\top} (\hat \bSigma_{\setminus k} + \lambda \bm I)^{-1}}{1 + \frac{1}{n} \tilde{\bm X}_k^{(1)\top} (\hat \bSigma_{\setminus k} + \lambda \bm I)^{-1} \tilde{\bm X}_k^{(1)}}\biggr) \bSigma^{(2)} \biggl(\frac{\tilde{\bm X}_k^{(1)} \tilde{\bm X}_k^{(1)\top}}{n}\biggr) \bm U \\
    &:= \bm u^{(d,1)\top} \tilde{\bm Z}_k^{(1)} \tilde{\bm Z}_k^{(1)\top} \bm v^{(d,1)} 
    + \frac{\bm u^{(d,2)\top} \tilde{\bm Z}_k^{(1)} \tilde{\bm Z}_k^{(1)\top} \bm v^{(d,2)}}{1 + \tilde{\bm Z}_k^{(1)\top} \bm A^{(d,2)} \tilde{\bm Z}_k^{(1)}} 
    \\
    &\qquad + \frac{\bm u^{(d,3)\top} \tilde{\bm Z}_k^{(1)} (\tilde{\bm Z}_k^{(1)\top} \bm A^{(d,3)} \tilde{\bm Z}_k^{(1)}) \tilde{\bm Z}_k^{(1)\top} \bm v^{(d,3)}}{1 + \tilde{\bm Z}_k^{(1)\top} \bm B^{(d,3)} \tilde{\bm Z}_k^{(1)}} \\
    &\qquad + \frac{\bm u^{(d,4)\top} \tilde{\bm Z}_k^{(1)}  \tilde{\bm Z}_k^{(1)\top} \bm v^{(d,4)}}{1 + \tilde{\bm Z}_k^{(1)\top} \bm A^{(d,4)} \tilde{\bm Z}_k^{(1)}} + \frac{\bm u^{(d,5)\top} \tilde{\bm Z}_k^{(1)} (\tilde{\bm Z}_k^{(1)\top} \bm A^{(d,5)} \tilde{\bm Z}_k^{(1)}) \tilde{\bm Z}_k^{(1)\top} \bm v^{(d,5)}}{1 + \tilde{\bm Z}_k^{(1)\top} \bm B^{(d,5)} \tilde{\bm Z}_k^{(1)}} \\
    &\qquad + \frac{\bm u^{(d,6)\top} \tilde{\bm Z}_k^{(1)} (\tilde{\bm Z}_k^{(1)\top} \bm A^{(d,6)} \tilde{\bm Z}_k^{(1)}) \tilde{\bm Z}_k^{(1)\top} \bm v^{(d,6)}}{(1 + \tilde{\bm Z}_k^{(1)\top} \bm B^{(d,6)} \tilde{\bm Z}_k^{(1)})^2} \\
    &\qquad + \frac{\bm u^{(d,7)\top} \tilde{\bm Z}_k^{(1)} (\tilde{\bm Z}_k^{(1)\top} \bm A^{(d,7)} \tilde{\bm Z}_k^{(1)})(\tilde{\bm Z}_k^{(1)\top} \bm B^{(d,7)} \tilde{\bm Z}_k^{(1)}) \tilde{\bm Z}_k^{(1)\top} \bm v^{(d,7)}}{(1 + \tilde{\bm Z}_k^{(1)\top} \bm C^{(d,7)} \tilde{\bm Z}_k^{(1)})^2}
\end{align*}
and similarly
\begin{align*}
    &V_1(\bm Z^{(1),k-1}, \bm Z^{(2)}) - V_1(\bm Z^{(1),\setminus k}, \bm Z^{(2)}) \\
    &:= \bm u^{(d,1)\top} {\bm Z}_k^{(1)} {\bm Z}_k^{(1)\top} \bm v^{(d,1)} 
    + \frac{\bm u^{(d,2)\top} {\bm Z}_k^{(1)} {\bm Z}_k^{(1)\top} \bm v^{(d,2)}}{1 + {\bm Z}_k^{(1)\top} \bm A^{(d,2)} {\bm Z}_k^{(1)}} 
  \\
  &\qquad+ \frac{\bm u^{(d,3)\top} {\bm Z}_k^{(1)} ({\bm Z}_k^{(1)\top} \bm A^{(d,3)} {\bm Z}_k^{(1)}) {\bm Z}_k^{(1)\top} \bm v^{(d,3)}}{1 + {\bm Z}_k^{(1)\top} \bm B^{(d,3)} {\bm Z}_k^{(1)}} \\
    &\qquad + \frac{\bm u^{(d,4)\top} {\bm Z}_k^{(1)}  {\bm Z}_k^{(1)\top} \bm v^{(d,4)}}{1 + {\bm Z}_k^{(1)\top} \bm A^{(d,4)} {\bm Z}_k^{(1)}} + \frac{\bm u^{(d,5)\top} {\bm Z}_k^{(1)} ({\bm Z}_k^{(1)\top} \bm A^{(d,5)} {\bm Z}_k^{(1)}) {\bm Z}_k^{(1)\top} \bm v^{(d,5)}}{1 + {\bm Z}_k^{(1)\top} \bm B^{(d,5)} {\bm Z}_k^{(1)}} \\
    &\qquad + \frac{\bm u^{(d,6)\top} {\bm Z}_k^{(1)} ({\bm Z}_k^{(1)\top} \bm A^{(d,6)} {\bm Z}_k^{(1)}) {\bm Z}_k^{(1)\top} \bm v^{(d,6)}}{(1 + {\bm Z}_k^{(1)\top} \bm B^{(d,6)} {\bm Z}_k^{(1)})^2}\\
    &\qquad+ \frac{\bm u^{(d,7)\top} {\bm Z}_k^{(1)} ({\bm Z}_k^{(1)\top} \bm A^{(d,7)} {\bm Z}_k^{(1)})({\bm Z}_k^{(1)\top} \bm B^{(d,7)} {\bm Z}_k^{(1)}) {\bm Z}_k^{(1)\top} \bm v^{(d,7)}}{(1 + {\bm Z}_k^{(1)\top} \bm C^{(d,7)} {\bm Z}_k^{(1)})^2}
\end{align*}
where
\begin{gather*}
    \bm u^{(d,1)} = \frac{\sigma^2}{n} (\bSigma^{(1)})^{1/2} (\hat \bSigma_{\setminus k} + \lambda \bm I)^{-2} \bm U, \quad 
    \bm v^{(d,1)} = (\bSigma^{(1)})^{1/2} \bSigma^{(2)} \bm U  \\
    \bm u^{(d,2)} = \frac{\sigma^2}{n^2} (\bSigma^{(1)})^{1/2} (\hat \bSigma_{\setminus k} + \lambda \bm I)^{-2} \bm U \\ 
    \bm v^{(d,2)} = (\bSigma^{(1)})^{1/2} (\hat \bSigma_{\setminus k} + \lambda \bm I)^{-1} (\bm X^{(1),\setminus k\top} \bm X^{(1),\setminus k} + \bm X^{(2)\top} \bm X^{(2)}) \bSigma^{(2)} \bm U \\
    \bm A^{(d,2)}= \frac{1}{n} (\bSigma^{(1)})^{1/2} (\hat \bSigma_{\setminus k} + \lambda \bm I)^{-1} (\bSigma^{(1)})^{1/2}  \\
    \bm u^{(d,3)} = \frac{\sigma^2}{n^2} (\bSigma^{(1)})^{1/2} (\hat \bSigma_{\setminus k} + \lambda \bm I)^{-2} \bm U, \quad 
    \bm v^{(d,3)} = (\bSigma^{(1)})^{1/2} \bSigma^{(2)} \bm U \\
    \bm A^{(d,3)} = (\bSigma^{(1)})^{1/2} (\hat \bSigma_{\setminus k} + \lambda \bm I)^{-1} (\bSigma^{(1)})^{1/2} \\
    \bm B^{(d,3)} = \frac{1}{n} (\bSigma^{(1)})^{1/2} (\hat \bSigma_{\setminus k} + \lambda \bm I)^{-1} (\bSigma^{(1)})^{1/2} \\
    \bm u^{(d,4)} = \frac{\sigma^2}{n^2} (\bSigma^{(1)})^{1/2} (\hat \bSigma_{\setminus k} + \lambda \bm I)^{-1} \bm U \\ 
    \bm v^{(d,4)} = (\bSigma^{(1)})^{1/2} (\hat \bSigma_{\setminus k} + \lambda \bm I)^{-2}  (\bm X^{(1),\setminus k\top} \bm X^{(1),\setminus k} + \bm X^{(2)\top} \bm X^{(2)}) \bSigma^{(2)} \bm U \\
    \bm A^{(d,4)} = \frac{1}{n} (\bSigma^{(1)})^{1/2} (\hat \bSigma_{\setminus k} + \lambda \bm I)^{-1} (\bSigma^{(1)})^{1/2} \\
    \bm u^{(d,5)} = \frac{\sigma^2}{n^2} (\bSigma^{(1)})^{1/2} (\hat \bSigma_{\setminus k} + \lambda \bm I)^{-1} \bm U, \quad 
    \bm v^{(d,5)} = (\bSigma^{(1)})^{1/2} \bSigma^{(2)} \bm U \\
    \bm A^{(d,5)} = (\bSigma^{(1)})^{1/2} (\hat \bSigma_{\setminus k} + \lambda \bm I)^{-2} (\bSigma^{(1)})^{1/2} \\ 
    \bm B^{(d,5)} = \frac{1}{n} (\bSigma^{(1)})^{1/2} (\hat \bSigma_{\setminus k} + \lambda \bm I)^{-1} (\bSigma^{(1)})^{1/2}  \\
    \bm u^{(d,6)} = \frac{\sigma^2}{n^3} (\bSigma^{(1)})^{1/2} (\hat \bSigma_{\setminus k} + \lambda \bm I)^{-1} \bm U \\ 
    \bm v^{(d,6)} = (\bSigma^{(1)})^{1/2} (\hat \bSigma_{\setminus k} + \lambda \bm I)^{-1} (\bm X^{(1),\setminus k\top} \bm X^{(1),\setminus k} + \bm X^{(2)\top} \bm X^{(2)}) \bSigma^{(2)} \bm U \\
    \bm A^{(d,6)} = (\bSigma^{(1)})^{1/2} (\hat \bSigma_{\setminus k} + \lambda \bm I)^{-2} (\bSigma^{(1)})^{1/2} \\ 
    \bm B^{(d,6)} = \frac{1}{n} (\bSigma^{(1)})^{1/2} (\hat \bSigma_{\setminus k} + \lambda \bm I)^{-1} (\bSigma^{(1)})^{1/2} \\
    \bm u^{(d,7)} = \frac{\sigma^2}{n^3} (\bSigma^{(1)})^{1/2} (\hat \bSigma_{\setminus k} + \lambda \bm I)^{-1} \bm U, \quad 
    \bm v^{(d,7)} = (\bSigma^{(1)})^{1/2} \bm U \\
    \bm A^{(d,7)} = (\bSigma^{(1)})^{1/2} (\hat \bSigma_{\setminus k} + \lambda \bm I)^{-2} (\bSigma^{(1)})^{1/2} \\ 
    \bm B^{(d,7)} = (\bSigma^{(1)})^{1/2} (\hat \bSigma_{\setminus k} + \lambda \bm I)^{-1} \bSigma^{(2)} (\bSigma^{(1)})^{1/2} \\
    \bm C^{(d,7)} = \frac{1}{n} (\bSigma^{(1)})^{1/2} (\hat \bSigma_{\setminus k} + \lambda \bm I)^{-1} (\bSigma^{(1)})^{1/2}
\end{gather*}
By the symmetry of $V_1$ in $\bm Z^{(1)}$ and $\bm Z^{(2)}$, we see that
\begin{align*}
    &V_1(\tilde{\bm Z}^{(1)}, \bm Z^{(2),k}) - V_1(\tilde{\bm Z}^{(1)}, \bm Z^{(2), \setminus k}) \\
    &:= \bm u^{(D,1)\top} \tilde{\bm Z}_k^{(2)} \tilde{\bm Z}_k^{(2)\top} \bm v^{(D,1)} 
    + \frac{\bm u^{(D,2)\top} \tilde{\bm Z}_k^{(2)} \tilde{\bm Z}_k^{(2)\top} \bm v^{(D,2)}}{1 + \tilde{\bm Z}_k^{(2)\top} \bm A^{(D,2)} \tilde{\bm Z}_k^{(2)}} 
  \\
  &\qquad + \frac{\bm u^{(D,3)\top} \tilde{\bm Z}_k^{(2)} (\tilde{\bm Z}_k^{(2)\top} \bm A^{(D,3)} \tilde{\bm Z}_k^{(2)}) \tilde{\bm Z}_k^{(2)\top} \bm v^{(D,3)}}{1 + \tilde{\bm Z}_k^{(2)\top} \bm B^{(D,3)} \tilde{\bm Z}_k^{(2)}} \\
    &\qquad + \frac{\bm u^{(D,4)\top} \tilde{\bm Z}_k^{(2)}  \tilde{\bm Z}_k^{(2)\top} \bm v^{(D,4)}}{1 + \tilde{\bm Z}_k^{(2)\top} \bm A^{(D,4)} \tilde{\bm Z}_k^{(2)}} + \frac{\bm u^{(D,5)\top} \tilde{\bm Z}_k^{(2)} (\tilde{\bm Z}_k^{(2)\top} \bm A^{(D,5)} \tilde{\bm Z}_k^{(2)}) \tilde{\bm Z}_k^{(2)\top} \bm v^{(D,5)}}{1 + \tilde{\bm Z}_k^{(2)\top} \bm B^{(D,5)} \tilde{\bm Z}_k^{(2)}} \\
    &\qquad + \frac{\bm u^{(D,6)\top} \tilde{\bm Z}_k^{(2)} (\tilde{\bm Z}_k^{(2)\top} \bm A^{(D,6)} \tilde{\bm Z}_k^{(2)}) \tilde{\bm Z}_k^{(2)\top} \bm v^{(D,6)}}{(1 + \tilde{\bm Z}_k^{(2)\top} \bm B^{(D,6)} \tilde{\bm Z}_k^{(2)})^2} \\
    &\qquad + \frac{\bm u^{(D,7)\top} \tilde{\bm Z}_k^{(2)} (\tilde{\bm Z}_k^{(2)\top} \bm A^{(D,7)} \tilde{\bm Z}_k^{(2)})(\tilde{\bm Z}_k^{(2)\top} \bm B^{(D,7)} \tilde{\bm Z}_k^{(2)}) \tilde{\bm Z}_k^{(2)\top} \bm v^{(D,7)}}{(1 + \tilde{\bm Z}_k^{(2)\top} \bm C^{(D,7)} \tilde{\bm Z}_k^{(2)})^2}
\end{align*}
and similarly
\begin{align*}
    &V_1(\bm Z^{(1),k}, \bm Z^{(2)}) - V_1(\bm Z^{(1),\setminus k}, \bm Z^{(2)}) \\
    &:= \bm u^{(D,1)\top} {\bm Z}_k^{(2)} {\bm Z}_k^{(2)\top} \bm v^{(D,1)} 
    + \frac{\bm u^{(D,2)\top} {\bm Z}_k^{(2)} {\bm Z}_k^{(2)\top} \bm v^{(D,2)}}{1 + {\bm Z}_k^{(2)\top} \bm A^{(D,2)} {\bm Z}_k^{(2)}} 
    \\
    &\qquad + \frac{\bm u^{(D,3)\top} {\bm Z}_k^{(2)} ({\bm Z}_k^{(2)\top} \bm A^{(D,3)} {\bm Z}_k^{(2)}) {\bm Z}_k^{(2)\top} \bm v^{(D,3)}}{1 + {\bm Z}_k^{(2)\top} \bm B^{(D,3)} {\bm Z}_k^{(2)}} \\
    &\qquad + \frac{\bm u^{(D,4)\top} {\bm Z}_k^{(2)}  {\bm Z}_k^{(2)\top} \bm v^{(D,4)}}{1 + {\bm Z}_k^{(2)\top} \bm A^{(D,4)} {\bm Z}_k^{(2)}} + \frac{\bm u^{(D,5)\top} {\bm Z}_k^{(2)} ({\bm Z}_k^{(2)\top} \bm A^{(D,5)} {\bm Z}_k^{(2)}) {\bm Z}_k^{(2)\top} \bm v^{(D,5)}}{1 + {\bm Z}_k^{(2)\top} \bm B^{(D,5)} {\bm Z}_k^{(2)}} \\
    &\qquad + \frac{\bm u^{(D,6)\top} {\bm Z}_k^{(2)} ({\bm Z}_k^{(2)\top} \bm A^{(D,6)} {\bm Z}_k^{(2)}) {\bm Z}_k^{(2)\top} \bm v^{(D,6)}}{(1 + {\bm Z}_k^{(2)\top} \bm B^{(D,6)} {\bm Z}_k^{(2)})^2} \\
    &\qquad + \frac{\bm u^{(D,7)\top} {\bm Z}_k^{(2)} ({\bm Z}_k^{(2)\top} \bm A^{(D,7)} {\bm Z}_k^{(2)})({\bm Z}_k^{(2)\top} \bm B^{(D,7)} {\bm Z}_k^{(2)}) {\bm Z}_k^{(2)\top} \bm v^{(D,7)}}{(1 + {\bm Z}_k^{(2)\top} \bm C^{(D,7)} {\bm Z}_k^{(2)})^2}
\end{align*}
where
\begin{gather*}
    \bm u^{(D,1)} = \frac{\sigma^2}{n} (\bSigma^{(2)})^{1/2} (\tilde \bSigma_{\setminus k} + \lambda \bm I)^{-2} \bm U, \quad 
    \bm v^{(D,1)} = (\bSigma^{(2)})^{3/2}  \bm U  \\
    \bm u^{(D,2)} = \frac{\sigma^2}{n^2} (\bSigma^{(2)})^{1/2} (\tilde \bSigma_{\setminus k} + \lambda \bm I)^{-2} \bm U \\
    \bm v^{(D,2)} = (\bSigma^{(2)})^{1/2} (\tilde \bSigma_{\setminus k} + \lambda \bm I)^{-1} (\tilde{\bm X}^{(1)\top} \tilde{\bm X}^{(1)} + \bm X^{(2), \setminus k\top} \bm X^{(2),\setminus k}) \bSigma^{(2)} \bm U \\
    \bm A^{(D,2)}= \frac{1}{n} (\bSigma^{(2)})^{1/2} (\tilde \bSigma_{\setminus k} + \lambda \bm I)^{-1} (\bSigma^{(2)})^{1/2}  \\
    \bm u^{(D,3)} = \frac{\sigma^2}{n^2} (\bSigma^{(2)})^{1/2} (\tilde \bSigma_{\setminus k} + \lambda \bm I)^{-2} \bm U, \quad 
    \bm v^{(D,3)} = (\bSigma^{(2)})^{3/2} \bm U \\
    \bm A^{(D,3)} = (\bSigma^{(2)})^{1/2} (\tilde \bSigma_{\setminus k} + \lambda \bm I)^{-1} (\bSigma^{(2)})^{1/2} \\ 
    \bm B^{(D,3)} = \frac{1}{n} (\bSigma^{(2)})^{1/2} (\tilde \bSigma_{\setminus k} + \lambda \bm I)^{-1} (\bSigma^{(2)})^{1/2} \\
    \bm u^{(D,4)} = \frac{\sigma^2}{n^2} (\bSigma^{(2)})^{1/2} (\tilde \bSigma_{\setminus k} + \lambda \bm I)^{-1} \bm U \\ 
    \bm v^{(D,4)} = (\bSigma^{(2)})^{1/2} (\tilde \bSigma_{\setminus k} + \lambda \bm I)^{-2}  (\tilde{\bm X}^{(1)\top} \tilde{\bm X}^{(1)} + \bm X^{(2), \setminus k\top} \bm X^{(2),\setminus k}) \bSigma^{(2)} \bm U \\
    \bm A^{(D,4)} = \frac{1}{n} (\bSigma^{(2)})^{1/2} (\tilde \bSigma_{\setminus k} + \lambda \bm I)^{-1} (\bSigma^{(2)})^{1/2} \\
    \bm u^{(D,5)} = \frac{\sigma^2}{n^2} (\bSigma^{(2)})^{1/2} (\tilde \bSigma_{\setminus k} + \lambda \bm I)^{-1} \bm U, \quad 
    \bm v^{(D,5)} = (\bSigma^{(2)})^{3/2} \bm U \\
    \bm A^{(D,5)} = (\bSigma^{(2)})^{1/2} (\tilde \bSigma_{\setminus k} + \lambda \bm I)^{-2} (\bSigma^{(2)})^{1/2} \\ 
    \bm B^{(D,5)} = \frac{1}{n} (\bSigma^{(2)})^{1/2} (\tilde \bSigma_{\setminus k} + \lambda \bm I)^{-1} (\bSigma^{(2)})^{1/2}  \\
    \bm u^{(D,6)} = \frac{\sigma^2}{n^3} (\bSigma^{(2)})^{1/2} (\tilde \bSigma_{\setminus k} + \lambda \bm I)^{-1} \bm U \\ 
    \bm v^{(D,6)} = (\bSigma^{(2)})^{1/2} (\tilde \bSigma_{\setminus k} + \lambda \bm I)^{-1} (\tilde{\bm X}^{(1)\top} \tilde{\bm X}^{(1)} + \bm X^{(2), \setminus k\top} \bm X^{(2),\setminus k}) \bSigma^{(2)} \bm U \\
    \bm A^{(D,6)} = (\bSigma^{(2)})^{1/2} (\tilde \bSigma_{\setminus k} + \lambda \bm I)^{-2} (\bSigma^{(2)})^{1/2} \\ 
    \bm B^{(D,6)} = \frac{1}{n} (\bSigma^{(2)})^{1/2} (\tilde \bSigma_{\setminus k} + \lambda \bm I)^{-1} (\bSigma^{(2)})^{1/2} \\
    \bm u^{(D,7)} = \frac{\sigma^2}{n^3} (\bSigma^{(2)})^{1/2} (\tilde \bSigma_{\setminus k} + \lambda \bm I)^{-1} \bm U, \quad 
    \bm v^{(D,7)} = (\bSigma^{(2)})^{1/2} \bm U \\
    \bm A^{(D,7)} = (\bSigma^{(2)})^{1/2} (\tilde \bSigma_{\setminus k} + \lambda \bm I)^{-2} (\bSigma^{(2)})^{1/2} \\ 
    \bm B^{(D,7)} = (\bSigma^{(2)})^{1/2} (\tilde \bSigma_{\setminus k} + \lambda \bm I)^{-1} \bSigma^{(2)} (\bSigma^{(2)})^{1/2} \\
    \bm C^{(D,7)} = \frac{1}{n} (\bSigma^{(2)})^{1/2} (\tilde \bSigma_{\setminus k} + \lambda \bm I)^{-1} (\bSigma^{(2)})^{1/2}
\end{gather*}   
\subsubsection{Bounding first and second moments}
Now, we can apply Lemmas \ref{lm:first_moment_bounds} and \ref{lm:second_moment_bounds}. 
We first bound
\begin{gather*}
    \|\bm u^{(d,1)}\|_2 = \sigma^2 n^{-1} \lambda^{-2} \|\bm U\|_2, \quad 
    \|\bm v^{(d,1)}\|_2 = \|\bm U\|_2  \\
    \|\bm u^{(d,2)}\|_2 = \sigma^2 n^{-2} \lambda^{-2} \|\bm U\|_2, \quad 
    \|\bm v^{(d,2)}\|_2 = \lambda^{-1} (\|\bm X^{(1),\setminus k}\|_{\text{op}}^2 + \|\bm X^{(2)}\|_{\text{op}}^2) \|\bm U\|_2 \\
    \|\bm A^{(d,2)}\|_{\text{op}}= n^{-1} \lambda^{-1}  \\
    \|\bm u^{(d,3)}\|_2 = \sigma^2 n^{-2} \lambda^{-2} \|\bm U\|_2, \quad 
    \|\bm v^{(d,3)}\|_2 = \|\bm U\|_2 \\
    \|\bm A^{(d,3)}\|_{\text{op}} = \lambda^{-1} , \quad 
    \|\bm B^{(d,3)}\|_{\text{op}} = n^{-1} \lambda^{-1} \\
    \|\bm u^{(d,4)}\|_2 = \sigma^2 n^{-2} \lambda^{-1} \|\bm U\|_2, \quad 
    \|\bm v^{(d,4)}\|_2 = \lambda^{-2}  (\|\bm X^{(1),\setminus k}\|_{\text{op}}^2 + \|\bm X^{(2)}\|_{\text{op}}^2) \|\bm U\|_2 \\
    \|\bm A^{(d,4)}\|_{\text{op}} = n^{-1} \lambda^{-1} \\
    \|\bm u^{(d,5)}\|_2 = \sigma^2 n^{-2} \lambda^{-1} \|\bm U\|_2, \quad 
    \|\bm v^{(d,5)}\|_2 = \|\bm U\|_2 \\
    \|\bm A^{(d,5)}\|_{\text{op}} = \lambda^{-2} , \quad 
    \|\bm B^{(d,5)}\|_{\text{op}} = n^{-1} \lambda^{-1}  \\
    \|\bm u^{(d,6)}\|_2 = \sigma^2 n^{-3} \lambda^{-1} \|\bm U\|_2, \quad 
    \|\bm v^{(d,6)}\|_2 = \lambda^{-1} (\|\bm X^{(1),\setminus k}\|_{\text{op}}^2 + \|\bm X^{(2)}\|_{\text{op}}^2) \|\bm U\|_2  \\
    \|\bm A^{(d,6)}\|_{\text{op}} = \lambda^{-2} , \quad 
    \|\bm B^{(d,6)}\|_{\text{op}} = n^{-1} \lambda^{-1} \\
    \|\bm u^{(d,7)}\|_2 = \sigma^2 n^{-3} \lambda^{-1} \|\bm U\|_2, \quad 
    \|\bm v^{(d,7)}\|_2 = \|\bm U\|_2 \\
    \|\bm A^{(d,7)}\|_{\text{op}} = \lambda^{-2} , \quad 
    \|\bm B^{(d,7)}\|_{\text{op}} = \lambda^{-1}  \\
    \|\bm C^{(d,7)}\|_{\text{op}} = n^{-1} \lambda^{-1} 
\end{gather*}
\paragraph{First term}
Now, using the first moment bound for $f_1$, we have
\begin{equation*}
    |\E_{1,k}[f_1(\bm Z_k^{(1)}; \bm S^{(d,1)}) - f_1(\tilde{\bm Z}_k^{(1)}; \bm S^{(d,1)})]| = 0
\end{equation*}
and similarly, using the second moment bound for $f_1$, we have 
\begin{align*}
    \max&\{\E_{1,k}[f_1(\bm Z_k^{(1)}; \bm S^{(d,1)})^2], \E_{1,k}[f_1(\tilde{\bm Z}_k^{(1)}; \bm S^{(d,1)})^2]\} \\
    &\lcon \|\bm u^{(d,1)}\|_2^2 \|\bm v^{(d,1)}\|_2^2 \\
    &\lcon \sigma^4 n^{-2} \lambda^{-4} \|\bm U\|_2^4
\end{align*}
\paragraph{Second term} Using the first moment bound for $f_2$, we have
\begin{align*}
    |&\E_{1,k}[f_2(\bm Z_k^{(1)}; \bm S^{(d,2)}) - f_2(\tilde{\bm Z}_k^{(1)}; \bm S^{(d,2)})]| \\
    &\lcon \|\bm u^{(d,2)}\|_2 \|\bm v^{(d,2)}\|_2 ( \|\bm A^{(d,2)}\|_{\text{op}} p^{1/2} + \|\bm A^{(d,2)}\|_{\text{op}}^2 p + \|\bm A^{(d,2)}\|_{\text{op}}^3 p^2) \\
    &\lcon \sigma^2 n^{-2} \lambda^{-2} \lambda^{-1} (\|\bm X^{(1),\setminus k}\|_{\text{op}}^2 + \|\bm X^{(2)}\|_{\text{op}}^2) \|\bm U\|_2^2 ( n^{-1} \lambda^{-1} p^{1/2} + n^{-2} \lambda^{-2} p + n^{-3} \lambda^{-3} p^2)  \\
    &\lcon \sigma^2 n^{-5/2} \lambda^{-4} \|\bm U\|_2^2(\|\bm X^{(1),\setminus k}\|_{\text{op}}^2 + \|\bm X^{(2)}\|_{\text{op}}^2) (1 + n^{-1/2} \lambda^{-1}  + n^{-1/2} \lambda^{-2}) 
\end{align*}
and using the second moment bound for $f_2$ yields
\begin{align*}
    \max\{&\E_{1,k}[f_2(\bm Z_k^{(1)}; \bm S^{(d,2)})^2], \E_{1,k}[f_2(\tilde{\bm Z}_k^{(1)}; \bm S^{(d,2)})^2]\} \\
    &\lcon \sigma^4 n^{-4} \lambda^{-6} (\|\bm X^{(1),\setminus k}\|_{\text{op}}^4 + \|\bm X^{(2)}\|_{\text{op}}^4) \|\bm U\|_2^4.
\end{align*}
\paragraph{Third term} Using the first moment bound for $f_4$, we have
\begin{align*}
    |&\E_{1,k}[f_4(\bm Z_k^{(1)}; \bm S^{(d,3)}) - f_4(\tilde{\bm Z}_k^{(1)}; \bm S^{(d,3)})]| \\
    &\lcon \|\bm u^{(d,3)}\|_2 \|\bm v^{(d,3)}\|_2 \|\bm A^{(d,3)}\|_{\text{op}}( p^{1/2} + \|\bm B^{(d,3)}\|_{\text{op}} p^{3/2} + \|\bm B^{(d,3)}\|_{\text{op}}^2 p^2 + \|\bm B^{(d,3)}\|_{\text{op}}^3 p^3) \\
    &\lcon \sigma^2 n^{-2} \lambda^{-2} \lambda^{-1} \|\bm U\|_2^2 ( p^{1/2} + n^{-1} \lambda^{-1} p^{3/2} + n^{-2} \lambda^{-2} p^2 + n^{-3} \lambda^{-3} p^3) \\
    &\lcon \sigma^2 n^{-3/2} \|\bm U\|_2^2 \lambda^{-3} ( 1 +  \lambda^{-1} + n^{-1/2} \lambda^{-2} + n^{-1/2} \lambda^{-3}) 
\end{align*}
and using the second moment bound for $f_4$ yields
\begin{align*}
    \max\{&\E_{1,k}[f_4(\bm Z_k^{(1)}; \bm S^{(d,3)})^2], \E_{1,k}[f_4(\tilde{\bm Z}_k^{(1)}; \bm S^{(d,3)})^2]\} \\
    &\lcon \sigma^4 n^{-2} \lambda^{-6} \|\bm U\|_2^4
\end{align*}
\paragraph{Fourth term} Using the first moment bound for $f_2$, we have
\begin{align*}
    |&\E_{1,k}[f_2(\bm Z_k^{(1)}; \bm S^{(d,4)}) - f_2(\tilde{\bm Z}_k^{(1)}; \bm S^{(d,4)})]| \\
    &\lcon \|\bm u^{(d,4)}\|_2 \|\bm v^{(d,4)}\|_2 ( \|\bm A^{(d,4)}\|_{\text{op}} p^{1/2} + \|\bm A^{(d,4)}\|_{\text{op}}^2 p + \|\bm A^{(d,4)}\|_{\text{op}}^3 p^2) \\
    &\lcon \sigma^2 n^{-2} \lambda^{-1} \lambda^{-2}  (\|\bm X^{(1),\setminus k}\|_{\text{op}}^2 + \|\bm X^{(2)}\|_{\text{op}}^2) \|\bm U\|_2^2 ( n^{-1} \lambda^{-1} p^{1/2} + n^{-2} \lambda^{-2} p + n^{-3} \lambda^{-3} p^2) \\
    &\lcon \sigma^2 n^{-5/2} \lambda^{-4}  (\|\bm X^{(1),\setminus k}\|_{\text{op}}^2 + \|\bm X^{(2)}\|_{\text{op}}^2) \|\bm U\|_2^2 ( 1 + n^{-1/2} \lambda^{-1}  + n^{-1/2} \lambda^{-2} ) 
\end{align*}
and using the second moment bound for $f_2$ yields
\begin{align*}
    \max\{&\E_{1,k}[f_2(\bm Z_k^{(1)}; \bm S^{(d,4)})^2], \E_{1,k}[f_2(\tilde{\bm Z}_k^{(1)}; \bm S^{(d,4)})^2]\} \\
    &\lcon \|\bm u^{(d,4)}\|_2^2 \|\bm v^{(d,4)}\|_2^2 \\
    &\lcon (\sigma^2 n^{-2} \lambda^{-1})^2 \|\bm U\|_2^4 (\lambda^{-2}  (\|\bm X^{(1),\setminus k}\|_{\text{op}}^2 + \|\bm X^{(2)}\|_{\text{op}}^2))^2 \\
    &\lcon \sigma^4 n^{-4} \lambda^{-6}  (\|\bm X^{(1),\setminus k}\|_{\text{op}}^4 + \|\bm X^{(2)}\|_{\text{op}}^4) \|\bm U\|_2^4.
\end{align*}
\paragraph{Fifth term} Using the first moment bound for $f_4$, we have
\begin{align*}
    |&\E_{1,k}[f_4(\bm Z_k^{(1)}; \bm S^{(d,5)}) - f_4(\tilde{\bm Z}_k^{(1)}; \bm S^{(d,5)})]| \\
    &\lcon \|\bm u^{(d,5)}\|_2 \|\bm v^{(d,5)}\|_2 \|\bm A^{(d,5)}\|_{\text{op}} \|\bm U\|_2^2 \\
    &\qquad ( p^{1/2} + \|\bm B^{(d,5)}\|_{\text{op}} p^{3/2} + \|\bm B^{(d,5)}\|_{\text{op}}^2 p^2 + \|\bm B^{(d,5)}\|_{\text{op}}^3 p^3) \\
    &\lcon \sigma^2 n^{-3/2} \lambda^{-3} \|\bm U\|_2^2 ( 1 +  \lambda^{-1} + n^{-1/2} \lambda^{-2}  + n^{-1/2} \lambda^{-3}) 
\end{align*}
and using the second moment bound for $f_4$ yields
\begin{align*}
    \max\{&\E_{1,k}[f_4(\bm Z_k^{(1)}; \bm S^{(d,5)})^2], \E_{1,k}[f_4(\tilde{\bm Z}_k^{(1)}; \bm S^{(d,5)})^2]\} \\
    &\lcon \sigma^4 n^{-2} \lambda^{-6}   \|\bm U\|_2^4
\end{align*}

\paragraph{Sixth term} Using the first moment bound for $f_5$, we have
\begin{align*}
    |&\E_{1,k}[f_5(\bm Z_k^{(1)}; \bm S^{(d,6)}) - f_5(\tilde{\bm Z}_k^{(1)}; \bm S^{(d,6)})]| \\
    &\lcon \|\bm u^{(d,6)}\|_2 \|\bm v^{(d,6)}\|_2 \|\bm A^{(d,6)}\|_{\text{op}}(p^{1/2} + \|\bm B^{(d,6)}\|_{\text{op}}^2 p^{5/2} + \|\bm B^{(d,6)}\|_{\text{op}}^4 p^4 + \|\bm B^{(d,6)}\|_{\text{op}}^6 p^6)  \\
    &\lcon \sigma^2 n^{-3} \lambda^{-1} \lambda^{-1} (\|\bm X^{(1),\setminus k}\|_{\text{op}}^2 + \|\bm X^{(2)}\|_{\text{op}}^2)  \lambda^{-2} \|\bm U\|_2^2 \\
    &\qquad (p^{1/2} + n^{-2} \lambda^{-2} p^{5/2} + n^{-4} \lambda^{-4} p^4 + n^{-6} \lambda^{-6} p^6) \\
    &\lcon \sigma^2 n^{-5/2} \lambda^{-4} \|\bm U\|_2^2 (\|\bm X^{(1),\setminus k}\|_{\text{op}}^2 + \|\bm X^{(2)}\|_{\text{op}}^2)  (1 + \lambda^{-2}  + n^{-1/2} \lambda^{-4}  + n^{-1/2} \lambda^{-6} ) 
\end{align*}
and using the second moment bound for $f_5$ yields
\begin{align*}
    \max\{&\E_{1,k}[f_5(\bm Z_k^{(1)}; \bm S^{(d,6)})^2], \E_{1,k}[f_5(\tilde{\bm Z}_k^{(1)}; \bm S^{(d,6)})^2]\} \\
    &\lcon \|\bm u^{(d,6)}\|_2^2 \|\bm v^{(d,6)}\|_2^2 \|\bm A^{(d,6)}\|_{\text{op}}^2  p^2 \\
    &\lcon \sigma^4 n^{-4} \lambda^{-8} (\|\bm X^{(1),\setminus k}\|_{\text{op}}^4 + \|\bm X^{(2)}\|_{\text{op}}^4)   \|\bm U\|_2^4
\end{align*}

\paragraph{Seventh term} Using the first moment bound for $f_7$, we have
\begin{align*}
    |&\E_{1,k}[f_7(\bm Z_k^{(1)}; \bm S^{(d,7)}) - f_7(\tilde{\bm Z}_k^{(1)}; \bm S^{(d,7)})]| \\
    &\lcon  \|\bm u^{(d,7)}\|_2 \|\bm v^{(d,7)}\|_2 \|\bm A^{(d,7)}\|_{\text{op}} \|\bm B^{(d,7)}\|_{\text{op}} \\
    &\qquad ( p^{3/2} +  \|\bm C^{(d,7)}\|_{\text{op}}^2 p^{7/2} +  \|\bm C^{(d,7)}\|_{\text{op}}^4 p^5 + \|\bm C^{(d,7)}\|_{\text{op}}^6 p^7) \\
    &\lcon  \sigma^2 n^{-3/2} \lambda^{-4} \|\bm U\|_2^2 ( 1 +  \lambda^{-2}  +  n^{-1/2} \lambda^{-4} + n^{-1/2} \lambda^{-6} ) 
\end{align*}
and using the second moment bound for $f_7$ yields
\begin{align*}
    \max\{&\E_{1,k}[f_7(\bm Z_k^{(1)}; \bm S^{(d,7)})^2], \E_{1,k}[f_7(\tilde{\bm Z}_k^{(1)}; \bm S^{(d,7)})^2]\} \\
    &\lcon \|\bm u^{(d,7)}\|_2^2 \|\bm v^{(d,7)}\|_2^2 \|\bm A^{(d,7)}\|_{\text{op}}^2  \|\bm B^{(d,7)}\|_{\text{op}}^2   p^4 \\
    &\lcon \sigma^4 n^{-2} \lambda^{-8}  \|\bm U\|_2^4
\end{align*}
\paragraph{Combined}
It thus follows from the above bounds that
\begin{align*}
    |&\E[\varphi(\E_U[\tilde V(\bm Z^{(1)}, \bm Z^{(2)})])] - \E[\varphi(\E_U[\tilde V(\tilde{\bm Z}^{(1)}, \bm Z^{(2)})])]| \\
    &\lcon \sigma^4 n^{-1} \lambda^{-4} \E[\|\bm U\|_2^4] \\
    &\qquad + \sigma^2 n^{-3/2} \lambda^{-4} (\E[\|\bm X^{(1),\setminus k}\|_{\text{op}}^2] + \E[\|\bm X^{(2)}\|_{\text{op}}^2])  \\
    &\hspace{4em}(1 + n^{-1/2} \lambda^{-1}  + n^{-1/2} \lambda^{-2}) \E[\|\bm U\|_2^2]  \\
    &\qquad + \sigma^4 n^{-3} \lambda^{-6} (\E[\|\bm X^{(1),\setminus k}\|_{\text{op}}^4] + \E[\|\bm X^{(2)}\|_{\text{op}}^4])\E[\|\bm U\|_2^4] \\
    &\qquad + \sigma^2 n^{-1/2} \lambda^{-3} ( 1 +  \lambda^{-1} + n^{-1/2} \lambda^{-2} + n^{-1/2} \lambda^{-3}) \E[\|\bm U\|_2^2] \\
    &\qquad + \sigma^4 n^{-1} \lambda^{-6} \E[\|\bm U\|_2^4] \\
    &\qquad + \sigma^2 n^{-3/2} \lambda^{-4}  (\E[\|\bm X^{(1),\setminus k}\|_{\text{op}}^2] + \E[\|\bm X^{(2)}\|_{\text{op}}^2]) \\
    &\hspace{4em}( 1 + n^{-1/2} \lambda^{-1}  + n^{-1/2} \lambda^{-2} ) \E[\|\bm U\|_2^2] \\
    &\qquad + \sigma^4 n^{-3} \lambda^{-6}  (\E[\|\bm X^{(1),\setminus k}\|_{\text{op}}^4] + \E[\|\bm X^{(2)}\|_{\text{op}}^4]) \E[\|\bm U\|_2^4] \\
    &\qquad +  \sigma^2 n^{-1/2} \lambda^{-3} ( 1 +  \lambda^{-1} + n^{-1/2} \lambda^{-2}  + n^{-1/2} \lambda^{-3}) \E[\|\bm U\|_2^2] \\
    &\qquad + \sigma^4 n^{-1} \lambda^{-6} \E[\|\bm U\|_2^4]  \\
    &\qquad + \sigma^2 n^{-3/2} \lambda^{-4} (\E[\|\bm X^{(1),\setminus k}\|_{\text{op}}^2] + \E[\|\bm X^{(2)}\|_{\text{op}}^2])  \\
    &\hspace{4em}(1 + \lambda^{-2}  + n^{-1/2} \lambda^{-4}  + n^{-1/2} \lambda^{-6} ) \E[\|\bm U\|_2^2] \\
    &\qquad + \sigma^4 n^{-3} \lambda^{-8} (\E[\|\bm X^{(1),\setminus k}\|_{\text{op}}^4] + \E[\|\bm X^{(2)}\|_{\text{op}}^4])  \E[\|\bm U\|_2^4] \\
    &\qquad + \sigma^2 n^{-1/2} \lambda^{-4} ( 1 +  \lambda^{-2}  +  n^{-1/2} \lambda^{-4} + n^{-1/2} \lambda^{-6} ) \E[\|\bm U\|_2^2] \\
    &\qquad + \sigma^4 n^{-1} \lambda^{-8} \E[\|\bm U\|_2^4]\\
    &\lcon \sigma^2 n^{-1/2} \log n \cdot \lambda^{-4} (1 + \lambda^{-2} + n^{-1/2}(\lambda^{-1} + \lambda^{-4} + \lambda^{-6})) \\
    &\qquad + \sigma^2 n^{-1/2} \lambda^{-3} \\
    &\qquad + \sigma^4 n^{-1} \log^2n \cdot (\lambda^{-6} + \lambda^{-8}) \\
    &\qquad + \sigma^4 n^{-1} \lambda^{-4}
\end{align*}
\subsubsection{Universality for \texorpdfstring{$\E[\tilde V]$}{E[V]}}
Because of the symmetry between in $\E[\tilde V]$ between $\bm Z^{(1)}$ and $\bm Z^{(2)}$, the same bound holds for the difference $|\E[\varphi(\E_U[\tilde V(\tilde{\bm Z}^{(1)}, \bm Z^{(2)})])] - \E[\varphi(\E_U[\tilde V(\tilde{\bm Z}^{(1)}, \tilde{\bm Z}^{(2)})])]|$. Therefore,
\begin{align*}
    |&\E[\varphi(V(\bZ^{(1)}, \bZ^{(2)}))] - \E[\varphi(V(\tilde \bZ^{(1)}, \tilde \bZ^{(2)}))]| \\
    &=|\E[\varphi(\E_U[\tilde V({\bm Z}^{(1)}, {\bm Z}^{(2)})])] - \E[\varphi(\E_U[\tilde V(\tilde{\bm Z}^{(1)}, \tilde{\bm Z}^{(2)})])]| \\
    &\lcon \sigma^2 n^{-1/2} \log n \cdot \lambda^{-4} (1 + \lambda^{-2} + n^{-1/2}(\lambda^{-1} + \lambda^{-4} + \lambda^{-6})) + \sigma^2 n^{-1/2} \lambda^{-3} \\
    &\qquad + \sigma^4 n^{-1} \log^2n \cdot (\lambda^{-6} + \lambda^{-8}) + \sigma^4 n^{-1} \lambda^{-4}.
\end{align*}
which concludes the proof of Proposition \ref{pp:universality_ridge}.

\subsection{Proof of Theorem \ref{thm:universality}}
We now use Proposition \ref{pp:universality_ridge} to finish our proof of Theorem \ref{thm:universality}. In particular, we use largely the same ideas as Appendix \ref{section:theorem_model_shift_final}, where we control the difference between the corresponding ridge and ridgeless quantities of interest. Recall that, by Lemmas \ref{lm:b2_ridge_limit_bound} and \ref{lm:b3_ridge_limit_bound}, for any small constant $c > 0$, we have 
\begin{align*}
    |B_2(\hat \bbeta_\lambda; \bbeta^{(2)}) - B_2(\hat \bbeta; \bbeta^{(2)})| &= O(n^c \lambda \|\tilde \bbeta\|_2^2) \\
    |B_3(\hat \bbeta_\lambda; \bbeta^{(2)}) - B_3(\hat \bbeta; \bbeta^{(2)})| &= O(n^c\lambda \|\tilde \bbeta\|_2\|\bbeta^{(2)}\|)
\end{align*}
with high probability over $(\bm Z^{(1)}, \bm Z^{(2)})$. Additionally, it follows from the second result in Lemma \ref{lm:pseudoinverse_trace_ridge_bound} that
\begin{equation*}
    |V(\hat \bbeta_\lambda; \bbeta^{(2)}) - V(\hat \bbeta; \bbeta^{(2)})| = O(n^c\lambda)
\end{equation*}
with high probability.
Finally, we prove the analogous bound for the first bias term. 
\begin{lemma}\label{lm:b1_ridge_limit_bound}
    We have
    \begin{equation*}
        |B_1(\hat \bbeta_\lambda; \bbeta^{(2)}) - B_1(\hat \bbeta; \bbeta^{(2)})| =O_\prec(\lambda \|\bbeta^{(2)}\|_2^2).
    \end{equation*}
\end{lemma}
\begin{proof}
    Using the triangle inequality, we have
    \begin{align*}
        &|B_1(\hat \bbeta_\lambda; \bbeta^{(2)}) - B_1(\hat \bbeta; \bbeta^{(2)})| \\
        &= |\lambda^2 \bbeta^{(2)\top} (\hat \bSigma + \lambda \bm I)^{-1} \bSigma (\hat \bSigma + \lambda \bm I)^{-1} \bbeta^{(2)} - \bbeta^{(2)\top} (\bm I - \hat \bSigma^\dagger \hat \bSigma) \bSigma (\bm I - \hat \bSigma^\dagger \hat \bSigma)   \bbeta^{(2)} | \\
        &\leq |\lambda\bbeta^{(2)\top} (\lambda (\hat \bSigma + \lambda \bm I)^{-1} - (\bm I - \hat \bSigma^\dagger \hat \bSigma))   \bSigma (\hat \bSigma + \lambda \bm I)^{-1}   \bbeta^{(2)}| \\
        &\qquad + | \bbeta^{(2)\top} (\bm I - \hat \bSigma^\dagger \hat \bSigma) \bSigma (\lambda(\hat \bSigma + \lambda \bm I)^{-1} - (\bm I - \hat \bSigma^\dagger \hat \bSigma))  \bbeta^{(2)} | \\
        &\lesssim \|\bbeta^{(2)}\|_2^2 \|\bSigma\|_{\text{op}} \cdot \|\lambda (\hat \bSigma + \lambda \bm I)^{-1} - (\bm I - \hat \bSigma^\dagger \hat \bSigma)\|_{\text{op}} (\|\bm I - \hat \bSigma^\dagger \hat \bSigma\|_{\text{op}} + \|\lambda (\hat \bSigma + \lambda \bm I)^{-1}\|_{\text{op}}) \\
        &\lesssim \|\bbeta^{(2)}\|_2^2 \|\bSigma\|_{\text{op}} \cdot \frac{\lambda}{d_{\text{nz}}^2} \\
        &= O_\prec(\lambda \|\bbeta^{(2)}\|_2^2 )
    \end{align*}
    where again the last two lines follow from Lemmas \ref{lemma:ESD_bounded} and \ref{lm:pseudoinverse_ridge_diff_bound}.
\end{proof}
That is, as $\lambda \downarrow 0$ sufficiently slowly and $n \to \infty$, the following four converge in probability: (1) $B_{\lambda,j}(\bZ^{(1)}, \bZ^{(2)})$ converges in probability to $B_{0,j}(\bZ^{(1)}, \bZ^{(2)})$ for $j = 1,2,3$, (2) $B_{\lambda,j}(\tilde \bZ^{(1)}, \tilde \bZ^{(2)})$ converges in probability to $B_{0,j}(\tilde \bZ^{(1)}, \tilde \bZ^{(2)})$ for $j = 1,2,3$, (3) $V_{\lambda}(\bZ^{(1)}, \bZ^{(2)})$ converges in probability to $B_{0,j}(\bZ^{(1)}, \bZ^{(2)})$, and (4) $V_{\lambda}(\tilde \bZ^{(1)}, \tilde \bZ^{(2)})$ converges in probability to $V_{0}(\tilde \bZ^{(1)}, \tilde \bZ^{(2)})$. Recalling that $\varphi$ is a continuous function, it follows that (1) $\varphi(B_{\lambda,j}(\bZ^{(1)}, \bZ^{(2)}))$ converges in probability to $\varphi(B_{0,j}(\bZ^{(1)}, \bZ^{(2)}))$ for $j = 1,2,3$, (2) $\varphi(B_{\lambda,j}(\tilde \bZ^{(1)}, \tilde \bZ^{(2)}))$ converges in probability to $\varphi(B_{0,j}(\tilde \bZ^{(1)}, \tilde \bZ^{(2)}))$ for $j = 1,2,3$, (3) $\varphi(V_{\lambda}(\bZ^{(1)}, \bZ^{(2)}))$ converges in probability to $\varphi(B_{0,j}(\bZ^{(1)}, \bZ^{(2)}))$, and (4) $\varphi(V_{\lambda}(\tilde \bZ^{(1)}, \tilde \bZ^{(2)}))$ converges in probability to $\varphi(V_{0}(\tilde \bZ^{(1)}, \tilde \bZ^{(2)}))$

Now, recalling that $\varphi$ is bounded, it follows from the dominated convergence theorem that, as $\lambda\downarrow 0$ sufficiently slowly and $n \to \infty$, we have
\begin{align*}
    |\E[\varphi(B_{\lambda,j}(\bZ^{(1)}, \bZ^{(2)}))] - \E[\varphi(B_{0,j}(\bZ^{(1)}, \bZ^{(2)}))]| &\to 0,\quad j = 1,2,3 \\
    |\E[\varphi(B_{\lambda,j}(\tilde \bZ^{(1)}, \tilde \bZ^{(2)}))] - \E[\varphi(B_{0,j}(\tilde \bZ^{(1)}, \tilde \bZ^{(2)}))]| &\to 0,\quad j = 1,2,3 \\
    |\E[\varphi(V_{\lambda}(\bZ^{(1)},  \bZ^{(2)}))] - \E[\varphi(V_{0}( \bZ^{(1)}, \bZ^{(2)}))]| &\to 0 \\
    |\E[\varphi(V_{\lambda}(\tilde \bZ^{(1)}, \tilde \bZ^{(2)}))] - \E[\varphi(V_{0}(\tilde \bZ^{(1)}, \tilde \bZ^{(2)}))]| &\to 0.
\end{align*}
Therefore, we can conclude that, taking $\lambda\downarrow 0$ sufficiently slowly and taking $n \to \infty$, we get
\begin{align*}
    &|\E[\varphi(B_{0,1}(\bZ^{(1)}, \bZ^{(2)}))] - \E[\varphi(B_{0,1}(\tilde \bZ^{(1)}, \tilde \bZ^{(2)}))]| \\
    &\le |\E[\varphi(B_{0,1}(\bZ^{(1)}, \bZ^{(2)}))] - \E[\varphi(B_{\lambda,1}(\bZ^{(1)}, \bZ^{(2)}))]|  \\
    &\qquad + |\E[\varphi(B_{\lambda,1}(\bZ^{(1)}, \bZ^{(2)}))] - \E[\varphi(B_{\lambda,1}(\tilde \bZ^{(1)}, \tilde \bZ^{(2)}))]| \\
    &\qquad + |\E[\varphi(B_{\lambda,1}(\tilde \bZ^{(1)}, \tilde \bZ^{(2)}))] - \E[\varphi(B_{0,1}(\tilde \bZ^{(1)}, \tilde \bZ^{(2)}))]|  \\
    &\to 0
\end{align*}
The same reasoning follows for $j = 2,3$ and for the variance terms, from which we conclude 
\begin{align*}
    |\E[\varphi(B_{0,j}(\bZ^{(1)}, \bZ^{(2)}))] - \E[\varphi(B_{0,j}(\tilde \bZ^{(1)}, \tilde \bZ^{(2)}))]| &\to 0,\quad j = 1,2,3 \\
    |\E[\varphi(V_{0}(\bZ^{(1)}, \bZ^{(2)}))] - \E[\varphi(V_{0}(\tilde \bZ^{(1)}, \tilde \bZ^{(2)}))]| &\to 0.
\end{align*}
Theorem \ref{thm:universality} then follows immediately by the standard Lindeberg argument (cf.~ \cite{korado2011lindeberg, hu2023universality}).

\section{Additional Proofs}\label{sec:other_extension_proofs}
In this section, we prove Theorems \ref{thm:misspecification} and \ref{thm:bias_corrected_risk}.

\subsection{Proof of Theorem \ref{thm:misspecification}}
In this section, we provide a proof of Theorem \ref{thm:misspecification}. Note that it suffices to characterize the behavior of $R_1(\hat \bbeta; \bbeta^{(2)}, \btheta^{(2)})$ and $R_2(\hat \bbeta; \bbeta^{(2)}, \btheta^{(2)})$ separately: we handle these two terms in Sections \ref{sec:misspecification_r1_proof} and \ref{sec:misspecification_r2_proof}, respectively.

\subsubsection{First term}\label{sec:misspecification_r1_proof}
Crucially, note that we can express
\begin{align*}
    \hat \bbeta &= (\bm X^\top \bm X)^\dagger (\bm X^{(1)\top} (\bm X^{(1)} \bbeta^{(1)} + \bep^{(1)}) + \bm X^{(2)\top} (\bm X^{(2)} \bbeta^{(2)} + \bep^{(2)})) \\
    &\qquad + (\bm X^\top \bm X)^\dagger (\bm X^{(1)\top}\bm W^{(1)} \btheta^{(1)} + \bX^{(2)\top} \bW^{(2)} \btheta^{(2)})
\end{align*}
where the first term $\hat \bbeta_{\bm X} := (\bm X^\top \bm X)^\dagger (\bm X^{(1)\top} (\bm X^{(1)} \bbeta^{(1)} + \bep^{(1)}) + \bm X^{(2)\top} (\bm X^{(2)} \bbeta^{(2)} + \bep^{(2)}))$ matches the min-$\ell_2$-norm interpolator and setting studied in Theorem \ref{thm:model_shift} and the second term $\bep_{\bm W} := (\bm X^\top \bm X)^\dagger (\bm X^{(1)\top}\bm W^{(1)} \btheta^{(1)} + \bX^{(2)\top} \bW^{(2)} \btheta^{(2)})$ is mean zero, conditional on $\bm X$. Additionally, by assumption, we have that $\hat \bbeta_{\bm X}$ and $\bep_{\bm W}$ are independent.

Now, using the fact that $\bm x_0$ and $\bm w_0$ are independent, we have
\begin{align*}
    R_1(\hat \bbeta; \bbeta^{(2)}, \btheta^{(2)}) &= \E[(\bm x_0^\top \hat \bbeta - \bm x_0^\top \bbeta^{(2)} - \E[\bm w_0 \mid \bm x_0]^\top \btheta^{(2)})^2 \mid \bm X] \\
    &= \E[(\bm x_0^\top \hat \bbeta - x_0^\top \bbeta^{(2)})^2 \mid \bm X] \\
    &= \E[\|\hat \bbeta - \bbeta^{(2)}\|_{\Sigma^{(2)}}^2 \mid \bm X] \\
    &= \| \E[\hat \bbeta \mid \bm X] - \bbeta^{(2)}\|_{\bSigma^{(2)}}^2 + \Tr[\Cov(\hat \bbeta \mid \bm X) \bSigma^{(2)}] \\
    &= \|\E[\hat \bbeta_{\bm X} \mid \bm X] - \bbeta^{(2)}\|_{\bSigma^{(2)}}^2 + \Tr[\Cov(\hat \bbeta_{\bm X} \mid \bm X) \bSigma^{(2)}] \\
    &\qquad + \Tr[\Cov(\bep_{\bm W} \mid \bm X) \bSigma^{(2)}].
\end{align*}
The behavior of the first two summands follows directly from our model shift results.
\begin{lemma}
    Suppose the setting of Theorem \ref{thm:misspecification} holds. Then, for any small constant $c > 0$, with high probability over $(\bm Z^{(1)}, \bm Z^{(2)})$, we have
    \begin{align*}
        &\|\E[\hat \bbeta_{\bm X} \mid \bm X] - \bbeta^{(2)}\|_{\bSigma^{(2)}}^2 + \Tr[\Cov(\hat \bbeta_{\bm X} \mid \bm X) \bSigma^{(2)}] \\
        &= \mathcal V(\hat H_n, \gamma) + \mathcal B_1(\hat H_n, \hat G_n^{\bbeta^{(2)}}, \gamma) + \mathcal B_2(\hat H_n, \hat G_n^{\tilde \bbeta}, \gamma) + \mathcal B_3(\hat H_n, \hat G_n^{(b)}, \gamma) \\
        &\qquad+ O(p^{-1/10+c} \|\bbeta^{(2)}\|_2 \|\tilde \bbeta\|_2)
    \end{align*}
\end{lemma}
\begin{proof}
    Recall that, in the setting of \ref{thm:misspecification}, the term $\hat \bbeta_{\bm X}$ precisely matches the min-$\ell_2$-norm interpolator in the setting of Theorem \ref{thm:model_shift}. Therefore, the result in Theorem \ref{thm:model_shift} applies, from which the lemma follows.
\end{proof}
The remaining term $\Tr[\Cov(\bep_{\bm W} \mid \bm X) \bSigma^{(2)}]$ requires more care: we use the same symmetrization techniques as in Appendix \ref{sec:proof:model_shift}. First, note that
\begin{align*}
        \Cov(\bep_{\bm W} \mid \bm X) &= \Cov((\bm X^\top \bm X)^\dagger \bX^{(1)\top} \bW^{(1)} \btheta^{(1)} \mid \bm X) \\
        &\qquad + \Cov((\bm X^\top \bm X)^\dagger \bX^{(2)\top} \bW^{(2)} \btheta^{(2)} \mid \bm X) \\
        &= (\bm X^\top \bm X)^\dagger \bm X^{(1)\top} \bm X^{(1)} (\bm X^\top \bm X)^\dagger \cdot \|\btheta^{(1)}\|_{\bSigma_m^{(1)}}^2 \\
        &\qquad + (\bm X^\top \bm X)^\dagger \bm X^{(2)\top} \bm X^{(2)} (\bm X^\top \bm X)^\dagger \cdot \|\btheta^{(2)}\|_{\bSigma_m^{(2)}}^2.
    \end{align*}
    and thus
\begin{align*}
    \Tr[\Cov(\bep_{\bm W} \mid \bm X)] &= \|\btheta^{(1)}\|_{\bSigma_m^{(1)}}^2 \cdot \frac{1}{n} \Tr\biggl[\hat \bSigma^\dagger \biggl(\frac{\bm X^{(1)\top} \bm X^{(1)}}{n}\biggr) \hat \bSigma^\dagger \bSigma^{(2)} \biggr] \\
    &\qquad + \|\btheta^{(2)}\|_{\bSigma_m^{(2)}}^2 \cdot \frac{1}{n} \Tr\biggl[\hat \bSigma^\dagger \biggl(\frac{\bm X^{(2)\top} \bm X^{(2)}}{n}\biggr) \hat \bSigma^\dagger \bSigma^{(2)} \biggr]
\end{align*}
It suffices to characterize $\frac{1}{n} \Tr[\hat \bSigma^\dagger (\frac{\bm X^{(1)\top} \bm X^{(1)}}{n}) \hat \bSigma^\dagger \bSigma^{(2)}]$, since the behavior of the latter term $\frac{1}{n} \Tr[\hat \bSigma^\dagger (\frac{\bm X^{(2)\top} \bm X^{(2)}}{n}) \hat \bSigma^\dagger \bSigma^{(2)}]$ follows by symmetry in this setting.

We begin by defining
\begin{align*}
    \bm V_M &:= \biggl(\frac{\bm Z^{(1)\top} \bm Z^{(1)}}{n}\biggr) (\hat \bSigma_Z + \lambda \bSigma^{-1})^{-2} \\
    \bm V_M^F &:= \hat \bSigma_Z  (\hat \bSigma + \lambda \bSigma^{-1})^{-2}
\end{align*}
\begin{lemma}\label{lm:pseudoinverse_trace_ridge_bound}
We have 
    \begin{align*}
        \biggl|\frac{1}{n} \Tr\biggl[\hat \bSigma^\dagger \biggl( \frac{\bm X^{(1)\top} \bm X^{(1)}}{n} \biggr) \hat \bSigma^\dagger \bSigma^{(2)}\biggr] - \frac{1}{n} \Tr[\bm V_M]\biggr| &=O_\prec( \lambda) \\
        \biggl|\frac{1}{n} \Tr[\hat \bSigma^\dagger \bSigma^{(2)}] - \frac{1}{n} \Tr[\bm V_M^F]\biggr| &= O_\prec(\lambda) 
    \end{align*}
\end{lemma}
\begin{proof}
    We follows the same reasoning as Section \ref{subsubsec:finite_sample_model_diff}. 

    For the first line, note that
    \begin{align*}
        &\biggl|\frac{1}{n} \Tr\biggl[\hat \bSigma^\dagger \biggl( \frac{\bm X^{(1)\top} \bm X^{(1)}}{n} \biggr) \hat \bSigma^\dagger \bSigma^{(2)}\biggr] - \frac{1}{n} \Tr\biggl[ (\hat \bSigma + \lambda \bm I)^{-1} \biggl( \frac{\bm X^{(1)\top} \bm X^{(1)}}{n} \biggr) (\hat \bSigma + \lambda \bm I)^{-1} \bSigma^{(2)}\biggr]\biggr| \\
        &= \biggl|\frac{1}{n} \Tr\biggl[\frac{\bm X^{(1)}}{\sqrt{n}} \hat \bSigma^\dagger \bSigma^{(2)} \bSigma^\dagger \frac{\bm X^{(1)\top}}{\sqrt{n}} \biggr] -\frac{1}{n} \Tr\biggl[   \frac{\bm X^{(1)}}{\sqrt n}  (\hat \bSigma + \lambda \bm I)^{-1} \bSigma^{(2)} (\hat \bSigma + \lambda \bm I)^{-1} \frac{\bm X^{(1)\top}}{\sqrt{n}}\biggr] \biggr| \\
        &\le \biggl|\frac{1}{n} \Tr\biggl[\frac{\bm X^{(1)}}{\sqrt{n}} \hat \bSigma^\dagger \bSigma^{(2)} (\bSigma^\dagger - (\hat \bSigma + \lambda \bm I)^{-1}) \frac{\bm X^{(1)\top}}{\sqrt{n}} \biggr] \biggr| \\
        &\qquad + \biggl|\frac{1}{n} \Tr\biggl[   \frac{\bm X^{(1)}}{\sqrt n}  (\hat \bSigma^\dagger - (\hat \bSigma + \lambda \bm I)^{-1}) \bSigma^{(2)} (\hat \bSigma + \lambda \bm I)^{-1} \frac{\bm X^{(1)\top}}{\sqrt{n}}\biggr] \biggr| \\
        &\le \biggl\|\frac{\bm X^{(1)}}{\sqrt{n}} \hat \bSigma^\dagger \bSigma^{(2)} (\bSigma^\dagger - (\hat \bSigma + \lambda \bm I)^{-1}) \frac{\bm X^{(1)\top}}{\sqrt{n}} \biggr\|_{\text{op}} \\
        &\qquad + \biggl\|\frac{\bm X^{(1)}}{\sqrt n}  (\hat \bSigma^\dagger - (\hat \bSigma + \lambda \bm I)^{-1}) \bSigma^{(2)} (\hat \bSigma + \lambda \bm I)^{-1} \frac{\bm X^{(1)\top}}{\sqrt{n}}\biggr\|_{\text{op}}  \\
        &\le \|\bSigma^{(2)}\|_{\text{op}} \biggl\|(\hat \bSigma^\dagger - (\hat \bSigma + \lambda\bm I)^{-1}) \frac{\bm X^{(1)\top}}{\sqrt{n}} \biggr\|_{\text{op}} \biggl(\biggl\| \hat \bSigma^\dagger \frac{\bm X^{(1)\top}}{\sqrt{n}}\biggr\|_{\text{op}} + \biggl\| (\hat \bSigma + \lambda \bm I)^{-1} \frac{\bm X^{(1)\top}}{\sqrt{n}}\biggr\|_{\text{op}} \biggr) \\
        &\le \|\bSigma^{(2)}\|_{\text{op}} \cdot \frac{2\lambda}{d_{\text{nz}}^4} \\
        &= O_\prec( \lambda \|\bSigma^{(2)}\|_{\text{op}})
    \end{align*}
    where the last two lines follow from Lemmas  \ref{lemma:ESD_bounded}, \ref{lm:pseudoinverse_ridge_diff_bound}, and \ref{lm:pseudoinverse_ridge_indiv_bound}. The first line follows by recognizing that
    \begin{equation*}
        \frac{1}{n} \Tr\biggl[(\hat \bSigma + \lambda \bm I)^{-1} \biggl( \frac{\bm X^{(1)\top} \bm X^{(1)}}{n}\biggr)(\hat \bSigma + \lambda \bm I)^{-1} \bSigma^{(2)}\biggr] = \frac{1}{n} \Tr[\bm V_M].
    \end{equation*}
    The same bound and reasoning applies for the second line, recalling that $\hat \bSigma^\dagger = \hat \bSigma^\dagger \hat \bSigma \hat \bSigma^\dagger$.
\end{proof}

\begin{lemma}\label{lm:trace_symmetry_mis}
Let $c > 0$ be a small constant. Then,
    \begin{align*}
        &\biggl|\frac{1}{n} \Tr[\bm V_M] - \frac{n_1}{n} \cdot \frac
        1   n\ \Tr[\bm V_M^F]\biggr| = O_p(\lambda^{-3} n^{-1+c} )
    \end{align*}
    with high probability over $(\bZ^{(1)}, \bZ^{(2)})$.
\end{lemma}
\begin{proof}
    First, we note that the same symmetry argument as \eqref{eq:symmetry_B3}, we have
    \begin{align}
        \E[\Tr[\bm V_M]] &= \Tr\biggl[\E\biggl[\biggl(\frac{\bm Z^{(1)\top} \bm Z^{(1)}}{n}\biggr) (\hat \bSigma_Z + \lambda \bSigma^{-1})^{-2} \biggr] \biggr]  \\
        &= \Tr\biggl[\E\biggl[\biggl(\frac{1}{n} \sum_{i \in I} \bm Z_i \bm Z_i^\top \biggr) (\hat \bSigma_Z + \lambda \bSigma^{-1})^{-2} \biggr] \biggr] \\
        &= \frac{n_1}{n} \Tr[\E[\hat \bSigma_Z(\hat \bSigma_Z + \lambda\bSigma^{-1})^{-2}]] \\
        &= \frac{n_1}{n} \E[\Tr[\bm V_M^F]] \label{eq:symmetry_vm}.
    \end{align}
    
    As in Lemma \ref{lm:b2p_conc}, we now use the Gaussian Poincar\'e inequality to prove concentration of $\Tr[\bm V_M]$ and $\Tr[\bm V_M^F]$ around their respective expectations.

    For $i = 1, \dots, n_1$ and $p = 1, \dots, p$, we can similarly compute
    \begin{align*}
        &\frac{\partial}{\partial \bm Z_{ij}^{(1)}} \Tr\biggl[\frac{\bm Z^{(1)\top} \bm Z^{(1)}}{n} (\hat \bSigma_Z + \lambda \bSigma^{-1})^{-2} \biggr] \\
        &= \frac{1}{n} \Tr\biggl[(e_j (\bm Z_i^{(1)})^\top + \bm Z_i^{(1)} e_j^\top) \biggl( (\hat \bSigma_Z + \lambda \bSigma^{-1})^{-2} \\
        &\hspace{12em}- (\hat \bSigma_Z + \lambda \bSigma^{-1})^{-1} \biggl(\frac{\bm Z^{(1)\top} \bm Z^{(1)}}{n}\biggr) (\hat \bSigma_Z + \lambda \bSigma^{-1})^{-2} \\
        &\hspace{12em} - (\hat \bSigma_Z + \lambda \bSigma^{-1})^{-2} \biggl(\frac{\bm Z^{(1)\top} \bm Z^{(1)}}{n}\biggr) (\hat \bSigma_Z + \lambda \bSigma^{-1})^{-1}\biggr) \biggr] \\
        &= \frac{2}{n} \biggl(\bm Z^{(1)} \biggl( (\hat \bSigma_Z + \lambda \bSigma^{-1})^{-2} - (\hat \bSigma_Z + \lambda \bSigma^{-1})^{-1} \biggl(\frac{\bm Z^{(1)\top} \bm Z^{(1)}}{n}\biggr) (\hat \bSigma_Z + \lambda \bSigma^{-1})^{-2} \\
        &\hspace{7em} - (\hat \bSigma_Z + \lambda \bSigma^{-1})^{-2} \biggl(\frac{\bm Z^{(1)\top} \bm Z^{(1)}}{n}\biggr) (\hat \bSigma_Z + \lambda \bSigma^{-1})^{-1}\biggr)\biggr)_{ij}
    \end{align*}
    and similarly, for $i = 1, \dots, n_2$ and $p= 1,\dots, p$, we have
    \begin{align*}
        &\frac{\partial}{\partial \bm Z_{ij}^{(2)}} \Tr\biggl[\frac{\bm Z^{(1)\top} \bm Z^{(1)}}{n} (\hat \bSigma_Z + \lambda \bSigma^{-1})^{-2} \biggr] \\
        &= -\frac{2}{n} \biggl(\bm Z^{(2)} \biggl( (\hat \bSigma_Z + \lambda \bSigma^{-1})^{-1} \biggl(\frac{\bm Z^{(1)\top} \bm Z^{(1)}}{n}\biggr) (\hat \bSigma_Z + \lambda \bSigma^{-1})^{-2} \\
        &\hspace{7em} + (\hat \bSigma_Z + \lambda \bSigma^{-1})^{-2} \biggl(\frac{\bm Z^{(1)\top} \bm Z^{(1)}}{n}\biggr) (\hat \bSigma_Z + \lambda \bSigma^{-1})^{-1}\biggr)\biggr)_{ij}.
    \end{align*}
    Therefore, the  Gaussian Poincar\'e inequality implies that
    \begin{align*}
        \Var\biggl(\frac{1}{n} \Tr[\bm V_M]\biggr) &\le \frac{4}{n^4} \cdot \E\biggl[\biggl\|\bm Z^{(1)} \biggl( (\hat \bSigma_Z + \lambda \bSigma^{-1})^{-2} \\
        &\hspace{6em}- (\hat \bSigma_Z + \lambda \bSigma^{-1})^{-1} \biggl(\frac{\bm Z^{(1)\top} \bm Z^{(1)}}{n}\biggr) (\hat \bSigma_Z + \lambda \bSigma^{-1})^{-2} \\
        &\hspace{6em} - (\hat \bSigma_Z + \lambda \bSigma^{-1})^{-2} \biggl(\frac{\bm Z^{(1)\top} \bm Z^{(1)}}{n}\biggr) (\hat \bSigma_Z + \lambda \bSigma^{-1})^{-1}\biggr) \biggr\|_F^2\biggr] \\
        &\qquad + \frac{4}{n^4} \cdot \E\biggl[\biggl\|\bm Z^{(2)} \biggl( (\hat \bSigma_Z + \lambda \bSigma^{-1})^{-1} \biggl(\frac{\bm Z^{(1)\top} \bm Z^{(1)}}{n}\biggr) (\hat \bSigma_Z + \lambda \bSigma^{-1})^{-2} \\
        &\hspace{6em} + (\hat \bSigma_Z + \lambda \bSigma^{-1})^{-2} \biggl(\frac{\bm Z^{(1)\top} \bm Z^{(1)}}{n}\biggr) (\hat \bSigma_Z + \lambda \bSigma^{-1})^{-1}\biggr) \biggr\|_F^2\biggr] \\
        &\lesssim \frac{1}{n^3} \E[ \|\bm Z^{(1)} (\hat \bSigma_Z + \lambda \bSigma^{-1})^{-2}\|_{\text{op}}^2] \\
        &\qquad + \frac{1}{n^3} \E\biggl[ \biggl\|\bm Z^{(1)} (\hat \bSigma_Z + \lambda \bSigma^{-1})^{-1} \biggl(\frac{\bm Z^{(1)\top} \bm Z^{(1)}}{n}\biggr) (\hat \bSigma_Z + \lambda \bSigma^{-1})^{-2}\biggr\|_{\text{op}}^2\biggr] \\
        &\qquad + \frac{1}{n^3} \E\biggl[ \biggl\|\bm Z^{(2)} (\hat \bSigma_Z + \lambda \bSigma^{-1})^{-1} \biggl(\frac{\bm Z^{(1)\top} \bm Z^{(1)}}{n}\biggr) (\hat \bSigma_Z + \lambda \bSigma^{-1})^{-2}\biggr\|_{\text{op}}^2\biggr] \\
        &\qquad + \frac{1}{n^3} \E\biggl[ \biggl\|\bm Z^{(1)} (\hat \bSigma_Z + \lambda \bSigma^{-1})^{-2} \biggl(\frac{\bm Z^{(1)\top} \bm Z^{(1)}}{n}\biggr) (\hat \bSigma_Z + \lambda \bSigma^{-1})^{-1}\biggr\|_{\text{op}}^2\biggr] \\
        &\qquad + \frac{1}{n^3} \E\biggl[ \biggl\|\bm Z^{(2)} (\hat \bSigma_Z + \lambda \bSigma^{-1})^{-2} \biggl(\frac{\bm Z^{(1)\top} \bm Z^{(1)}}{n}\biggr) (\hat \bSigma_Z + \lambda \bSigma^{-2})^{-1}\biggr\|_{\text{op}}^2\biggr] \\
        &\le \frac{\|\bSigma\|_{\text{op}}^6}{\lambda^6 n^2} \biggl(\E\biggl[ \biggl\| \frac{\bm Z^{(1)\top} \bm Z^{(1)}}{n} \biggr\|_{\text{op}}^2 \biggl(\biggl\|\frac{\bm Z^{(1)\top} \bm Z^{(1)}}{n}\biggr\|_{\text{op}} + \biggl\|\frac{\bm Z^{(2)\top} \bm Z^{(2)}}{n}\biggr\|_{\text{op}} \biggr)\biggr] \biggr) \\
        &\qquad + \frac{\|\bSigma\|_{\text{op}}^4}{\lambda^4 n^2} \E\biggl[\biggl\|\frac{\bm Z^{(1)\top} \bm Z^{(1)}}{n}\biggr\|_{\text{op}} \biggr] \\
        &\lesssim \frac{\|\bSigma\|_{\text{op}}^6}{\lambda^6 n^2} + \frac{\|\bSigma\|_{\text{op}}^4}{\lambda^4 n^2} 
    \end{align*}
    where the last line follows from applying H\"older's inequality and Lemma \ref{lm:gaussian_matrix_norm_expec}.
    The same argument demonstrates that the same bound holds for $\Var(\frac{1}{n} \Tr[\bm V_M^F])$. The conclusion then follows from Chebyshev's inequality and \eqref{eq:symmetry_vm}.
\end{proof}

Thus, we yield the following characterization:
\begin{lemma}
Let $c > 0$ be a small constant. Then,
    \begin{align*}
    \Tr[\Cov(\bep_{\bm W} \mid \bm X)] 
    &= \biggl(\frac{n_1}{n} \cdot \frac{\|\btheta^{(1)}\|_{\bSigma_m^{(1)}}^2}{\sigma^2}  + \frac{n_2}{n} \cdot \frac{\|\btheta^{(2)}\|_{\bSigma_m^{(2)}}^2}{\sigma^2}\biggr) \cdot \mathcal V(\hat H_n, \gamma) + O(p^{-1/10+c}).
\end{align*}
with high probability over $(\bZ^{(1)}, \bZ^{(2)})$
\end{lemma}
\begin{proof}
    Taking $\lambda = p^{-1/10}$ in Lemmas \ref{lm:pseudoinverse_trace_ridge_bound} and \ref{lm:trace_symmetry_mis}, we have
    \begin{equation*}
        \biggl|\frac{1}{n} \Tr\biggl[\hat \bSigma^\dagger \biggl(\frac{\bm X^{(1)\top} \bm X^{(1)}}{n}\biggr) \hat \bSigma^\dagger \bSigma^{(2)} \biggr] - \frac{n_1}{n} \cdot \frac{1}{n} \Tr[\hat \bSigma^\dagger \bSigma^{(2)}]\biggr| = O(p^{-1/10+c}).
    \end{equation*}
    Recognizing that $\frac{1}{\sigma^2}V(\hat \bbeta; \bbeta^{(2)}) = \frac{1}{n} \Tr[\hat \bSigma^\dagger \bSigma^{(2)}]$, it follows from Theorem \ref{thm:model_shift} that
    \begin{equation*}
        \biggl|\frac{1}{n} \Tr\biggl[\hat \bSigma^\dagger \biggl(\frac{\bm X^{(1)\top} \bm X^{(1)}}{n}\biggr) \hat \bSigma^\dagger \bSigma^{(2)} \biggr] - \frac{n_1}{\sigma^2 n} \cdot \mathcal V(\hat H_n, \gamma) \biggr| = O(p^{-1/10+c}).
    \end{equation*}
    By symmetry, we have 
    \begin{equation*}
        \biggl|\frac{1}{n} \Tr\biggl[\hat \bSigma^\dagger \biggl(\frac{\bm X^{(2)\top} \bm X^{(2)}}{n}\biggr) \hat \bSigma^\dagger \bSigma^{(2)} \biggr] - \frac{n_2}{\sigma^2 n} \cdot \mathcal V(\hat H_n, \gamma) \biggr| = O(p^{-1/10+c}),
    \end{equation*}
    from which the conclusion follows.
\end{proof}

\subsubsection{Second term}\label{sec:misspecification_r2_proof}
Now, consider the second term $R_2(\hat \bbeta; \bbeta^{(2)}, \btheta^{(2)})$. Note that $\E[\bm w_0 \mid \bm x_0] = \E[\bm w_0] = 0$ by independence of $\bm x_0$ and $\bm w_0$, which implies that
\begin{align*}
    R_2(\hat \bbeta; \bbeta^{(2)}, \btheta^{(2)}) &= \E[(\E[\bm w_0 \mid \bm x_0]^\top \btheta^{(2)} - \bm w_0^\top \btheta^{(2)})^2] \\
    &= \E[(\bm w_0^\top \btheta^{(2)})^2] = \|\btheta^{(2)}\|_{\bSigma_m^{(2)}}^2
\end{align*}
as desired.

\subsection{Proof of Theorem \ref{thm:bias_corrected_risk}}
In this section, we provide a proof of Theorem \ref{thm:bias_corrected_risk}. 

Denote $n_\kappa = n_1 + \lfloor (1 - \kappa ) n_2 \rfloor$. For a given subset $\mathcal I \subseteq [n_2]$, we denote $\hat \bSigma_{\mathcal I} = \frac{\bm X^{(1)\top} \bm X^{(1)}}{n_\kappa} + \frac{\bm X^{(2)\top} \bm X^{(2)}}{n_\kappa}$ and $\hat \bSigma_{\mathcal I^c} = \frac{\bm X^{(2)\top } \bm X^{(2)}}{\kappa n_2}$, in which case
\begin{align*}
    \hat \bbeta_{\text{BC}} &= \hat \bbeta_{\mathcal I} + \hat \bdelta_{\lambda_\delta} \\
    &= \hat \bbeta_{\mathcal I} + \frac{1}{\kappa n_2} (\hat \bSigma_{\mathcal I^c} + \lambda_\delta \bm I)^{-1} \bm X_{\mathcal I^c}^{(2)\top} (\bm y_{\mathcal I^c}^{(2)} - \bm X_{\mathcal I^c}^{(2)} \hat \bbeta_{\mathcal I}) \\
    &= (\hat \bSigma_{\mathcal I^c} + \lambda _\delta \bm I)^{-1} \frac{\bm X^{(2)\top}_{\mathcal I^c} \bm y_{\mathcal I^c}^{(2)}}{\kappa n_2} + \lambda_\delta(\hat \bSigma_{\mathcal I^c} + \lambda_\delta)^{-1} \hat \bbeta_{\mathcal I}.
 \end{align*}
 Therefore, we have
 \begin{align*}
     \hat \bbeta_{\text{BC}} - \bbeta^{(2)} &= (\hat \bSigma_{\mathcal I^c} + \lambda _\delta \bm I)^{-1} \biggl(\frac{\bm X^{(2)\top}_{\mathcal I^c} \bm y_{\mathcal I^c}^{(2)}}{\kappa n_2} - \hat \bSigma_{\mathcal I^c} \bbeta^{(2)}\biggr) + \lambda_\delta(\hat \bSigma_{\mathcal I^c} + \lambda_\delta)^{-1} (\hat \bbeta_{\mathcal I} - \bbeta^{(2)}) \\
     &= (\hat \bSigma_{\mathcal I^c} + \lambda _\delta \bm I)^{-1} \biggl(\frac{\bm X^{(2)\top}_{\mathcal I^c} \bep_{\mathcal I^c}^{(2)}}{\kappa n_2} \biggr) + \lambda_\delta(\hat \bSigma_{\mathcal I^c} + \lambda_\delta)^{-1} (\hat \bbeta_{\mathcal I} - \bbeta^{(2)})
 \end{align*}
In the above expression, note that the terms $(\hat \bSigma_{\mathcal I^c} + \lambda _\delta \bm I)^{-1} (\frac{\bm X^{(2)\top}_{\mathcal I^c} \bep_{\mathcal I^c}^{(2)}}{\kappa n_2})$ and $\lambda_\delta(\hat \bSigma_{\mathcal I^c} + \lambda_\delta)^{-1} (\hat \bbeta_{\mathcal I} - \bbeta^{(2)})$ are independent, given $\bm X$, and that the first term has mean $0$, given $\bm X$. Therefore, we can decompose
\begin{align*}
    &R(\hat \bbeta_{\text{BC}}; \bbeta^{(2)}) \\&= \E[\|\hat \bbeta_{\text{BC}} - \bbeta^{(2)}\|_2^2 \mid \bm X]  \\
    &= \underbrace{\E\biggl[\biggl\|(\hat \bSigma_{\mathcal I^c} + \lambda _\delta \bm I)^{-1} \biggl(\frac{\bm X^{(2)\top}_{\mathcal I^c} \bep_{\mathcal I^c}^{(2)}}{\kappa n_2}\biggr)\biggr\|_2^2\biggm\vert \bm X \biggr]}_{R_1(\hat \bbeta_{\text{BC}}; \bbeta^{(2)})} + \underbrace{\E[\|\lambda_\delta(\hat \bSigma_{\mathcal I^c} + \lambda_\delta)^{-1} (\hat \bbeta_{\mathcal I} - \bbeta^{(2)})\|_2^2 \mid \bm X]}_{R_2(\hat \bbeta_{\text{BC}}; \bbeta^{(2)})}.
\end{align*}
The analysis of the first term is relatively straightforward.
\begin{lemma}
    Define the quantity
    \begin{equation*}
        \mathcal R_1(\hat \bbeta_{\text{BC}}; \bbeta^{(2)}) = \sigma^2 \gamma_\kappa \cdot \frac{\lambda_\delta + 1 - \gamma_\kappa + \gamma_\kappa \lambda_\delta m_n(-\lambda_\delta) - \lambda_\delta (1 + \gamma_\kappa m_{n,1}(-\lambda_\delta))}{(\lambda_\delta + 1 - \gamma_\kappa + \gamma_\kappa \lambda_\delta m_n(-\lambda_\delta))^2} 
    \end{equation*}
    In the context of Theorem \ref{thm:bias_corrected_risk}, for any small $c > 0$, we have
    \begin{equation}
        |R_1(\hat \bbeta_{BC}; \bbeta^{(2)}) - \mathcal R_1(\hat \bbeta_{\mathrm{BC}}; \bbeta^{(2)})| = O(n^{-1/2 + c})
    \end{equation}
    with high probability over $(\bZ^{(1)}, \bZ^{(2)})$.
\end{lemma}
\begin{proof}
    We can first compute
    \begin{align*}
    \E\biggl[\biggl\|(\hat \bSigma_{\mathcal I^c} + \lambda _\delta \bm I)^{-1} \biggl(\frac{\bm X^{(2)\top}_{\mathcal I^c} \bep_{\mathcal I^c}^{(2)}}{\kappa n_2}\biggr)\biggr\|_2^2\biggm\vert \bm X \biggr] &= \frac{\sigma^2}{\kappa n_2} \Tr[(\hat \bSigma_{\mathcal I^c} + \lambda_\delta \bm I)^{-2} \hat \bSigma_{\mathcal I^c}] \\
    &=\frac{\sigma^2}{\kappa n_2} \Tr[(\hat \bSigma_{\mathcal I^c} + \lambda_\delta \bm I)^{-1}] \\
    &\qquad - \frac{\sigma^2}{\kappa n_2} \Tr[\lambda_\delta(\hat \bSigma_{\mathcal I^c} + \lambda_\delta \bm I)^{-2}]
\end{align*}
Now, as in Lemma \ref{lm:b3f_local}, we can apply \cite[Theorem 3.16]{knowles2017anisotropic} and \cite[Section A.2]{hastie2022surprises} to yield
    \begin{align}
        &\biggl|\frac{1}{p} \Tr[(\hat \bSigma_{\mathcal I^c} + \lambda_\delta \bm I)^{-1}] - \frac{1}{\lambda_\delta + 1 - \gamma_\kappa + \gamma_\kappa \lambda_\delta m_n(-\lambda_\delta)} \biggr| = O_\prec(n^{-1/2} \lambda^{-1})\\
        &\biggl|\frac1p \Tr[\lambda_\delta  (\hat \bSigma_{\mathcal I^c} + \lambda_\delta \bm I)^{-2}]  - \frac{\lambda_\delta (1 + \gamma_\kappa m_{n,1}(-\lambda_\delta))}{(\lambda_\delta + 1 - \gamma_\kappa + \gamma_\kappa \lambda m_n(-\lambda_\delta))^2} \biggr| = O_\prec(n^{-1} \lambda^{-1})
    \end{align}
    from which the conclusion follows.
\end{proof}
Now, we can similarly use a local law to analyze the latter term.
\begin{lemma}
    Define the quantity
    \begin{equation*}
        \mathcal R_2(\hat \bbeta_{\text{BC}}; \bbeta^{(2)}) = \frac{\lambda_\delta^2(1 + \gamma_\kappa m_{n,1}(-\lambda_\delta))}{(\lambda_\delta + 1 - \gamma_\kappa + \gamma_\kappa \lambda_\delta m_n(-\lambda_\delta))^2} \cdot R(\hat \bbeta_{\mathcal I}; \bbeta^{(2)}).
    \end{equation*}
    Then, for any small constant $c > 0$, we have
    \begin{equation*}
        |R_2(\hat \bbeta_{\mathrm{BC}}; \bbeta^{(2)}) -  \mathcal R_2(\hat \bbeta_{\mathrm{BC}}; \bbeta^{(2)}) | = O_p(n^{-1/2 + c})
    \end{equation*}
    with high probability over $(\bZ^{(1)}, \bZ^{(2)})$
\end{lemma}
\begin{proof}
    Note that the matrix $\lambda_\delta (\hat \bSigma_{\mathcal I^c} + \lambda_\delta)^{-1}$ is both $\sigma(\bm X)$-measurable and independent of $\hat \bbeta_{\mathcal I} - \bbeta^{(2)}$. Thus, conditional on $\hat \bbeta_{\mathcal I} - \bbeta^{(2)}$, we can apply the local law from  \cite[Section A.2]{hastie2022surprises} (cf.~\cite[Theorem 3.16]{knowles2017anisotropic}) to yield that there exists a fixed constant $C$ (not dependent on $\hat \bbeta_{\mathcal I} - \bbeta^{(2)}$ such that
    \begin{align*}
        &\biggl|\lambda_\delta^2\frac{(\hat \bbeta_{\mathcal I} - \bbeta^{(2)})^\top}{\|\hat \bbeta_{\mathcal I} - \bbeta^{(2)}\|_2} (\hat \bSigma_{\mathcal I^c} + \lambda_\delta)^{-2} \frac{(\hat \bbeta_{\mathcal I} - \bbeta^{(2)})}{\|\hat \bbeta_{\mathcal I} - \bbeta^{(2)}\|_2} - \frac{\lambda_\delta^2(1 + \gamma_\kappa m_{n,1}(-\lambda_\delta))}{(\lambda_\delta + 1 - \gamma_\kappa + \gamma_\kappa \lambda_\delta m_n(-\lambda_\delta))^2}  \biggr| \\
        &\le O_\prec(n^{-1/2} )
    \end{align*}
    from which the conclusion follows.
\end{proof}

\section{Sketch of a Covariate Shift Result Without Simultaneous Diagonalizability}\label{sec:sketch_nonsimul}
In Theorem \ref{thm:design_shift}, we imposed Assumption \ref{as:simul_diag}; that is, the covariance matrices $\bSigma^{(1)}$ and $\bSigma^{(2)}$ are diagonalizable in a shared eigenbasis $\bm V$. In the underparametrized regime in \cite{yang2020analysis}, such an assumption can often be avoided by whitening with respect to one of the covariance matrices. However, in the overparametrized regime, this technique fails because whitening the design also changes the geometry of the resulting interpolator: in particular, the min-$\ell_2$-norm interpolator is transformed into a min-norm interpolator with respect to a different, covariance-dependent norm.

Fully overcoming Assumption \ref{as:simul_diag} requires substantial technical work beyond the work in this manuscript. To this end, we sketch a promising pathway for generalizing our results beyond simultaneous diagonalizability while deferring a rigorous proof to future work.

Our sketch relies upon the observation that we may apply anisotropic local laws from \cite{knowles2017anisotropic} repeatedly, conditional on the source- or target-only datasets. In this appendix, we focus on the first bias term
\begin{equation*}
    B_1(\hat \bbeta; \bbeta^{(2)}) = \lambda^2 \bbeta^{(2)\top} (\hat \bSigma + \lambda \bm I)^{-1} \bSigma^{(2)} (\hat \bSigma + \lambda \bm I)^{-1} \bbeta^{(2)}
\end{equation*}
and note that our approach may be adapted to the other terms. We omit exact technical rigor (e.g., we assume that $\lambda \in \R$) and instead focus on the general ideas and challenges of our pathway.

Let $\hat \bSigma^{(i)} = \frac{\bX^{(i)\top} \bm X^{(i)}}{n}$ for $i = 1,2$. We begin by first expressing, as in the proof of Theorem \ref{thm:design_shift}, that
\begin{align*}
    B_1(\hat \bbeta; \bbeta^{(2)}) &= \lambda^2 \bbeta^{(2)\top} (\hat \bSigma + \lambda \bm I)^{-1} \bSigma^{(2)} (\hat \bSigma + \lambda \bm I)^{-1} \bbeta^{(2)} \\
    &= \lambda^2 \bbeta^{(2)\top} (\hat \bSigma^{(1)} + \hat \bSigma^{(2)} + \lambda \bm I)^{-1} \bSigma^{(2)} (\hat \bSigma^{(1)} + \hat \bSigma^{(2)} + \lambda\bm I )^{-1} \bbeta^{(2)} \\
    &=- \frac{\partial}{\partial t} \bigg\vert_{t=0} \lambda \bbeta^{(2)\top} (\hat \bSigma^{(1)} + \hat \bSigma^{(2)} + \lambda \bm I + \lambda t\bSigma^{(2)})^{-1} \bbeta^{(2)} \\
    &= -\frac{\partial}{\partial t} \bigg\vert_{t=0} \lambda \bbeta^{(2)\top} (\hat \bSigma^{(1)} + \lambda \bm R_1)^{-1} \bbeta^{(2)} \\
    &= -\frac{\partial}{\partial t} \bigg\vert_{t=0} \lambda \bbeta^{(2)\top} \bm R_1^{-1/2} (\bm R_1^{-1/2} \hat \bSigma^{(1)} \bm R_1^{-1/2} + \lambda \bm I)^{-1} \bm R_1^{-1/2} \bbeta^{(2)}
\end{align*}
where we let $\bm R_1 = \bm I + t \bSigma^{(2)} + \frac{1}{\lambda} \hat \bSigma^{(2)}$. Then, conditional on $\hat \bSigma^{(2)}$, the term $\bm R_1^{-1/2} \hat \bSigma^{(1)} \bm R_1^{-1/2}$ corresponds to a sample covariance matrix with covariance $\bm R_1^{-1/2} \bSigma^{(1)} \bm R_1^{-1/2}$. Therefore, by applying an anisotropic local law (e.g., \cite{knowles2017anisotropic}), we have
\begin{align*}
    &\lambda \bbeta^{(2)\top} \bm R_1^{-1/2} (\bm R_1^{-1/2} \hat \bSigma^{(1)} \bm R_1^{-1/2} + \lambda \bm I)^{-1} \bm R_1^{-1/2} \bbeta^{(2)} \\
    &\approx  \lambda \bbeta^{(2)\top} (\lambda \bm R_1 + \tfrac{n_1}{n} \lambda r_n(-\lambda; \gamma_1, \tfrac{n_1}{n} (\bSigma^{(1)})^{1/2} \bm R_1^{-1}  (\bSigma^{(1)})^{1/2}) \bSigma^{(1)})^{-1} \bbeta^{(2)}
\end{align*}
where $r_n(-\lambda; \gamma_1, \tfrac{n_1}{n} (\bSigma^{(1)})^{1/2} \bm R_1^{-1}  (\bSigma^{(1)})^{1/2})$ is the companion Stieltjes transform (cf.~\cite[Lemma 2.2]{knowles2017anisotropic}) corresponding to the empirical spectral distribution of the matrix 
\begin{align*}
    \frac{n_1}{n} (\bSigma^{(1)})^{1/2} \bm R_1^{-1}  (\bSigma^{(1)})^{1/2} &= \frac{\lambda n_1}{n} (\bSigma^{(1)})^{1/2} (\hat \bSigma^{(2)} + \lambda t \bSigma^{(2)} + \lambda \bm I )^{-1} (\bSigma^{(1)})^{1/2}.
\end{align*}
If it can be shown that the empirical spectral distribution of $\frac{\lambda n_1}{n} (\bSigma^{(1)})^{1/2} (\hat \bSigma^{(2)} + \lambda t \bSigma^{(2)} + \lambda \bm I )^{-1} (\bSigma^{(1)})^{1/2}$ converges in probability to some probability measure $\mu_{1,t}$ (with respect to a suitable metric on probability measures), then we have 
\begin{align*}
    &\lambda \bbeta^{(2)\top} (\lambda \bm R_1 + \tfrac{n_1}{n} \lambda r_n(-\lambda; \gamma_1, \tfrac{n_1}{n} (\bSigma^{(1)})^{1/2} \bm R_1^{-1}  (\bSigma^{(1)})^{1/2}) \bSigma^{(1)})^{-1} \bbeta^{(2)} \\
    &\approx \lambda \bbeta^{(2)\top} (\lambda \bm R_1 + \tfrac{n_1}{n} \lambda r(-\lambda; \gamma_1, \mu_{1,t}) \bSigma^{(1)})^{-1} \bbeta^{(2)} \\
    &= \lambda \bbeta^{(2)\top} (\hat \bSigma^{(2)} + t \lambda \bSigma^{(2)} + \lambda \bm I + \tfrac{n_1}{n} \lambda r(-\lambda; \gamma_1, \mu_{1,t}) \bSigma^{(1)})^{-1} \bbeta^{(2)} \\
    &= \lambda \bbeta^{(2)} (\hat \bSigma^{(2)} + \lambda \bm R_2)^{-1} \bbeta^{(2)} \\
    &= \lambda \bbeta^{(2)} \bm R_2^{-1/2}(\bm R_2^{-1/2} \hat \bSigma^{(2)} \bm R_2^{-1/2} + \lambda\bm I)^{-1} \bm R_2^{-1/2} \bbeta^{(2)}
\end{align*}
where we let  $\bm R_2 = t  \bSigma^{(2)} +  \bm I + \tfrac{n_1}{n}  r(-\lambda; \gamma_1, \mu_{1,t}) \bSigma^{(1)}$. Then, we can once again apply an anisotropic local law (e.g., \cite{knowles2017anisotropic}) to yield
\begin{align*}
    &\lambda \bbeta^{(2)} \bm R_2^{-1/2}(\bm R_2^{-1/2} \hat \bSigma^{(2)} \bm R_2^{-1/2} + \lambda\bm I)^{-1} \bm R_2^{-1/2} \bbeta^{(2)} \\
    &\approx \lambda \bbeta^{(2)\top} (\lambda \bm R_2 + \tfrac{n_2}{n} \lambda r_n(-\lambda; \gamma_1, \tfrac{n_2}{n} (\bSigma^{(2)})^{1/2} \bm R_2^{-1}  (\bSigma^{(2)})^{1/2}) \bSigma^{(2)})^{-1} \bbeta^{(2)}
\end{align*}
which is a deterministic quantity.

There remain two major technical challenges to completing such a proof: (1) demonstrating that the empirical spectral distribution of $\frac{\lambda n_1}{n} (\bSigma^{(1)})^{1/2} (\hat \bSigma^{(2)} + \lambda t \bSigma^{(2)} + \lambda \bm I )^{-1} (\bSigma^{(1)})^{1/2}$ converges in probability to some deterministic probability measure $\mu_{1,t}$ and (2) characterizing the measure $\mu_{1,t}$ is nontrivial. With the diagonalizations $\bSigma^{(i)} = \bV^{(i)} \bm D^{(i)} \bm V^{(i)\top}$ for $i = 1,2$,  the empirical spectral distribution is equal in distribution to that of 
\begin{equation*}
    \frac{\lambda n_1}{n} (\bD^{(1)})^{1/2} \bm V^{(1)\top} \bm V^{(2)} (\hat \bSigma^{(2)}_Z + \lambda t \bD^{(2)} + \lambda \bm I )^{-1} \bm V^{(2)\top} \bm V^{(1)}(\bD^{(1)})^{1/2}
\end{equation*}
where $\hat \bSigma^{(2)}_Z = \frac{\bm Z^{(2)\top} \bm Z^{(2)}}{n}$. While analyzing the empirical spectral distribution of $(\hat \bSigma^{(2)}_Z + \lambda t \bD^{(2)} + \lambda \bm I )^{-1}$ is possible with existing tools, a major challenge is introduced through the conjugation by $\bm V^{(2)\top} \bm V^{(1)}(\bm D^{(1)})^{1/2}$.  We believe that linearization techniques (such as Definition \ref{def:resolvent}) may be promising for tackling this challenge.

The second major technical challenge is related to analyzing the resulting object, the $t$-derivative of the object
\begin{equation*}
    \lambda \bbeta^{(2)\top} (\lambda \bm R_2 + \tfrac{n_2}{n} \lambda r_n(-\lambda; \gamma_1, \tfrac{n_2}{n} (\bSigma^{(2)})^{1/2} \bm R_2^{-1}  (\bSigma^{(2)})^{1/2}) \bSigma^{(2)})^{-1} \bbeta^{(2)}.
\end{equation*}
Even if one is able to determine that $\lambda^2 \bbeta^{(2)\top} (\hat \bSigma + \lambda \bm I)^{-1} \bSigma^{(2)} (\hat \bSigma + \lambda \bm I)^{-1} \bbeta^{(2)}$ converges to such an object, it is less clear how to make sense of it in a way that allows for insights similar to those of our main text. However, we believe that the resulting objects may reduce significantly for special cases of covariance matrices which are of interest (e.g., when $\bm V^{(1)} = \bm Q \bm V^{(2)}$ for a  rotation matrix $\bm Q$ that only acts on a subspace of dimension $O(1)$) while still yielding  insights into how to treat general covariance matrices.

\section{Additional Simulations}\label{sec:additional_simulations}
\subsection{Universality}\label{sec:university_sims}
In this section, we examine the robustness of our risk characterizations beyond the Gaussian setting. Theorem \ref{thm:universality} reveals that, under Assumption \ref{as:design_all_moments}, the asymptotic prediction risk of the pooled min-$\ell_2$-norm interpolator agrees with its Gaussian counterpart under model shift. Here, we test whether the risk characterization in Theorem \ref{thm:model_shift} continues to provide accurate predictions for a range of non-Gaussian covariates, including some that fall outside the scope of Assumption \ref{as:shared} or \ref{as:design_all_moments}.

\subsubsection{Synthetic covariates}
In Figure \ref{fig:universality_version_isotropic}, we reproduce Figures \ref{fig:isotropic_model_shift}(a) and \ref{fig:isotropic_model_shift}(b) under three synthetic non-Gaussian covariate distributions. Figures \ref{fig:universality_version_isotropic}(a)-(d) fall directly within the scope of Theorem \ref{thm:universality}. In Figures \ref{fig:universality_version_isotropic}(a) and \ref{fig:universality_version_isotropic}(b), our covariates are drawn as i.i.d.~Rademacher entries. In Figures \ref{fig:universality_version_isotropic}(c) and \ref{fig:universality_version_isotropic}(d), the entries are generated from standardized Hardy-Weinberg distributions: for each coordinate $j = 1, \dots, p$, we first fix an allele frequency $p_j$ drawn uniformly from $[0.25, 0.75]$. Then, for a given observation, the $j$th entry (i.e., the $j$th genotype count) is drawn from a rescaled and centered $\text{Bin}(2, p_j)$ distribution. In all four of these figures, we see very close alignment between the empirical risk values and the Gaussian risk predictions, which provides strong evidence for the finite-sample applicability of our asymptotic universality result in Theorem \ref{thm:universality}.

In Figures \ref{fig:universality_version_isotropic}(e) and \ref{fig:universality_version_isotropic}(f), the entries are drawn from a standardized $t_4$ distribution. Recalling that a $t_4$ distribution lacks a fourth moment, it certainly follows that Assumption \ref{as:design_all_moments} fails to hold. However, we note that the empirical risk values and Gaussian risk predictions in Figures \ref{fig:universality_version_isotropic}(e) and \ref{fig:universality_version_isotropic}(f) still match tightly, which suggests that the universality phenomenon described in Figure \ref{thm:universality} may extend to heavier-tailed covariates, though formally establishing such a result requires additional arguments.

\begin{figure}
    \centering
    \begin{subfigure}[b]{0.48\linewidth}
    \includegraphics[width=\linewidth]{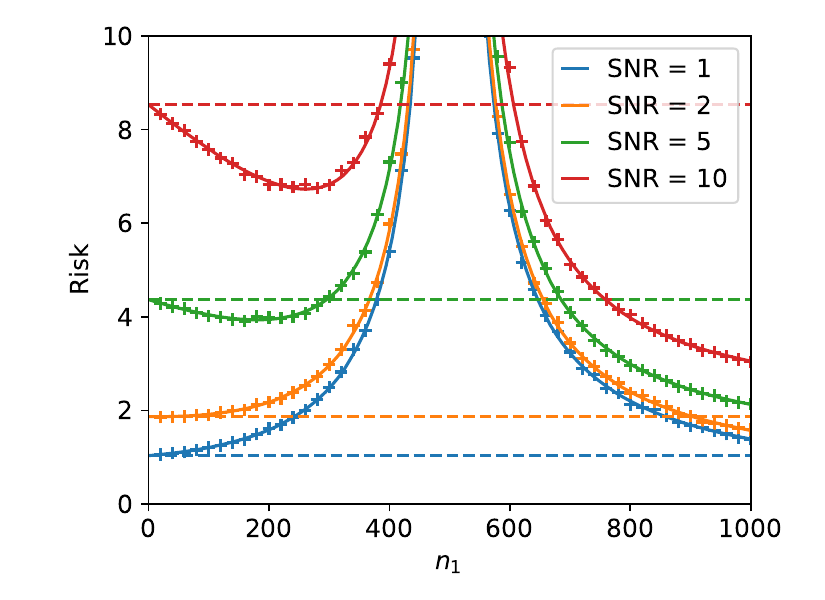}
    \caption{Fixing $\textnormal{SSR}=0.2$}
    \end{subfigure}
    \begin{subfigure}[b]{0.48\linewidth}
    \includegraphics[width=\linewidth]{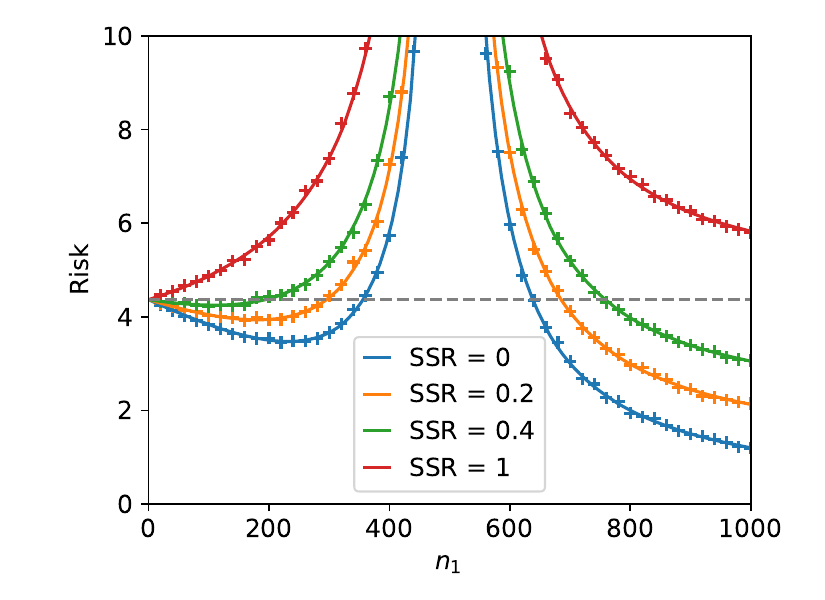}
    \caption{Fixing $\textnormal{SNR}=5$}
    \end{subfigure}
    \begin{subfigure}[b]{0.48\linewidth}
    \includegraphics[width=\linewidth]{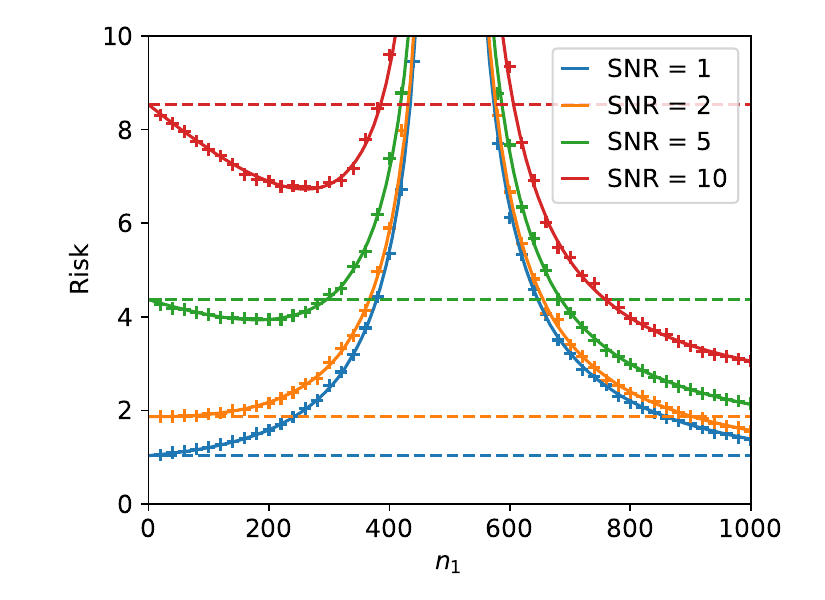}
    \caption{Fixing $\textnormal{SSR}=0.2$}
    \end{subfigure}
    \begin{subfigure}[b]{0.48\linewidth}
    \includegraphics[width=\linewidth]{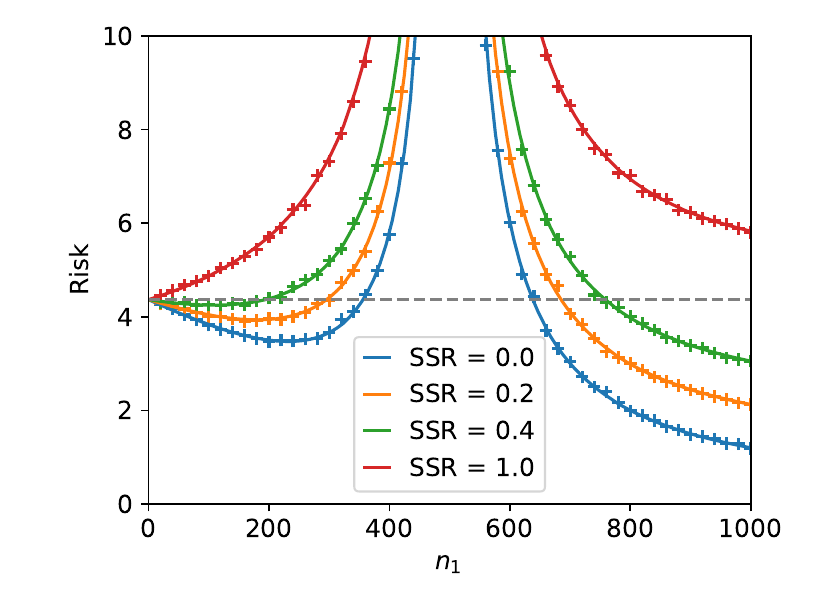}
    \caption{Fixing $\textnormal{SNR}=5$}
    \end{subfigure}
    \begin{subfigure}[b]{0.48\linewidth}
    \includegraphics[width=\linewidth]{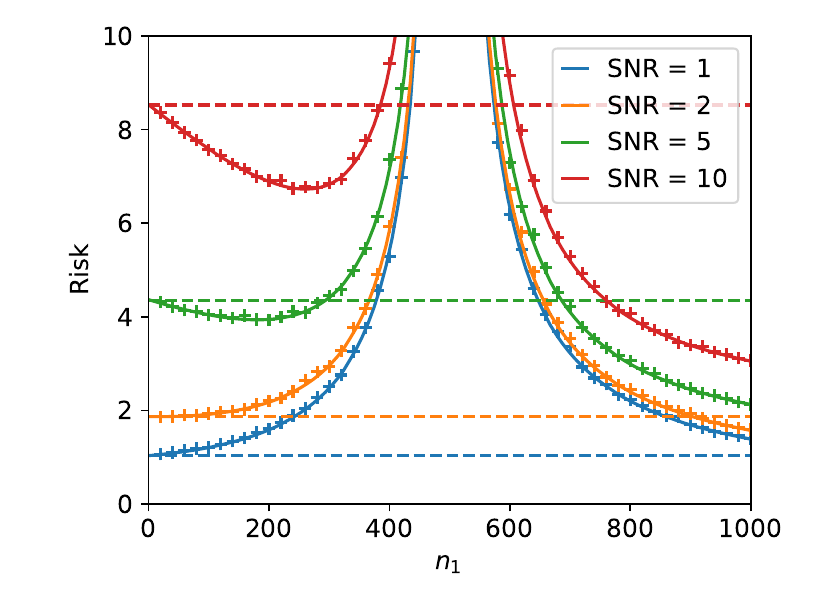}
    \caption{Fixing $\textnormal{SSR}=0.2$}
    \end{subfigure}
    \begin{subfigure}[b]{0.48\linewidth}
    \includegraphics[width=\linewidth]{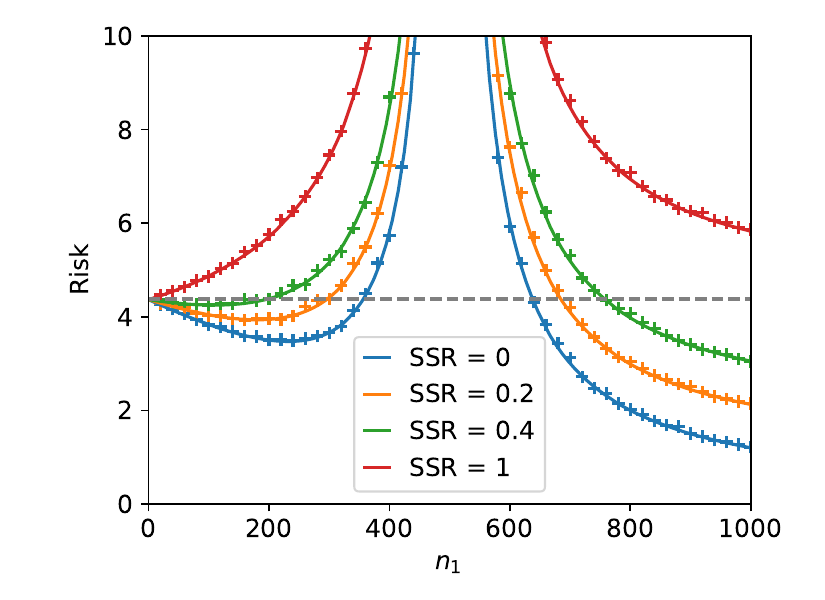}
    \caption{Fixing $\textnormal{SNR}=5$}
    \end{subfigure}
    \caption{Generalization error of the pooled min-$\ell_2$-norm interpolator under isotropic model shift with Rademacher, Hardy-Weinberg, and $t_4$ covariates, with $n_2 = 100$ and $p = 600$. A fixed realization of $\bbeta^{(2)}$ and $\tilde \bbeta$ are used throughout, where $\bbeta^{(2)}$ is drawn uniformly from the sphere of radius $\sqrt{\textnormal{SNR}}$ and $\tilde \bbeta$ is drawn uniformly from the sphere of radius $\sqrt{\textnormal{SSR} \cdot \textnormal{SNR}}$. All six panels redraw i.i.d. $\mathcal N(0,1)$ noise across trials.
    Panels (a) and (b) use i.i.d. Rademacher covariates. Panels (c) and (d) use Hardy-Weinberg covariates: for each coordinate $j = 1, \dots, p$, we first draw and fix $p_j \sim \text{Unif}(0.25, 0.75)$, sample entries from $\text{Bin}(2, p_j)$, and then center and standardize the results. Panels (e) and (f) use i.i.d. standardized $t_4$ covariates.
    Solid curves denote theoretical predictions, obtained from Theorem \ref{thm:model_shift} for $n_1 + n_2 < p$ and from \cite{yang2020analysis} for $n_1 + n_2 > p$. $+$ markers denote simulation averages across $50$ trials, and dashed horizontal lines denote the simulation average risk of the target-only interpolator. }
    \label{fig:universality_version_isotropic}
\end{figure}

\subsubsection{Semi-synthetic covariates}
In Figure \ref{fig:semisynthetic}, we evaluate our risk predictions using semi-synthetic covariates derived from real human genotype data. We select 600 SNPs from the 1000 Genomes Phase 3 panel \cite{1000genomes} and fit the HMM implemented in the  fastPHASE software \cite{scheet06fastphase}. With the fitted parameters, we then use the SNPknock package from \cite{sessia2019hmm} to sample new SNPs. When whitened, the sampled SNPs are not expected to be i.i.d., as Assumption \ref{as:shared} required. Therefore, like the $t_4$ covariates from Figures \ref{fig:universality_version_isotropic}(e) and \ref{fig:universality_version_isotropic}(f), the covariates in Figure \ref{fig:semisynthetic} once again fall outside the theoretical assumptions for our results. Nonetheless, in Figures \ref{fig:semisynthetic}(a)-(c), we again see very close alignment between the empirical risk values and the Gaussian risk predictions, thus validating our risk predictions for a set of semi-synthetic covariates.

\begin{figure}
    \centering
    \begin{subfigure}[b]{0.48\linewidth}
    \includegraphics[width=\linewidth]{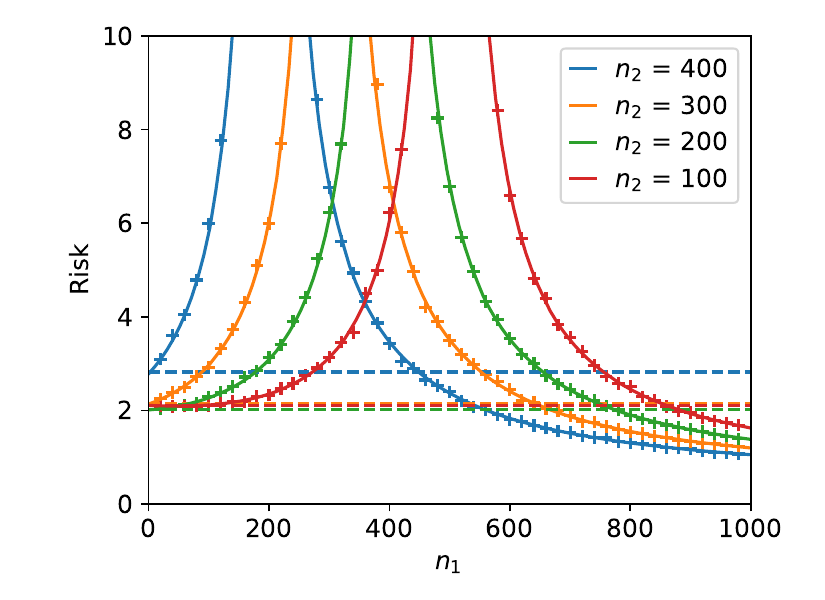}
    \caption{Fixing $\textnormal{SNR}=5, \textnormal{SSR}=0.2$}
    \end{subfigure}
    \begin{subfigure}[b]{0.48\linewidth}
    \includegraphics[width=\linewidth]{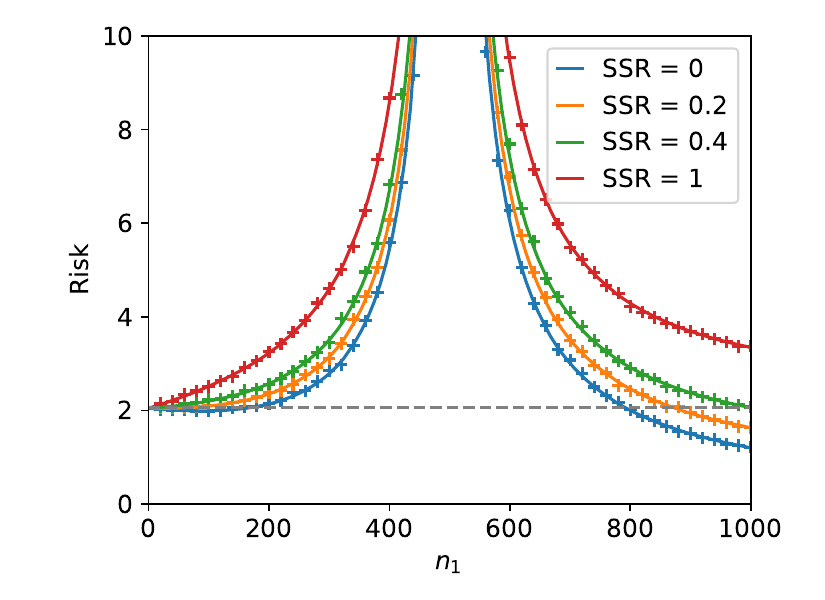}
    \caption{Fixing $n_2=100,\textnormal{SNR}=5$}
    \end{subfigure}
    \begin{subfigure}[b]{0.48\linewidth}
    \includegraphics[width=\linewidth]{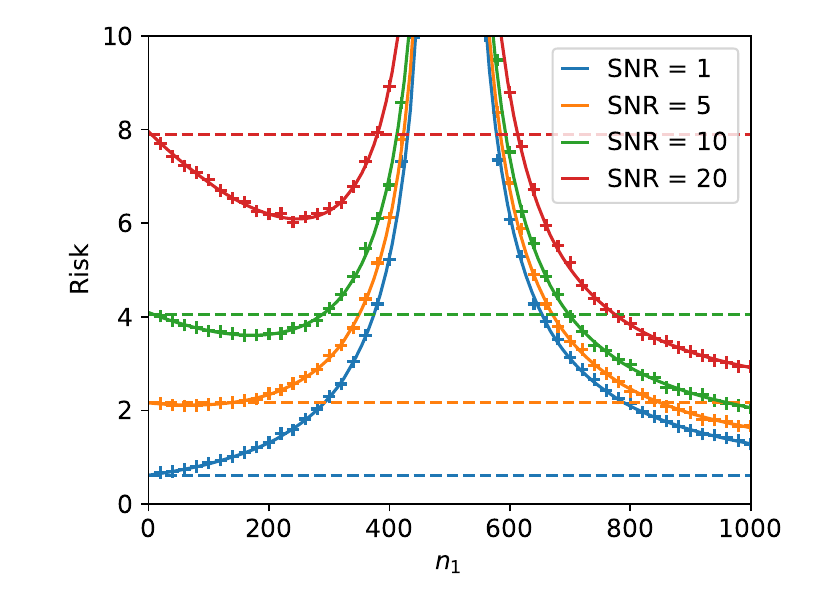}
    \caption{Fixing $n_2=100,\textnormal{SSR}=0.2$}
    \end{subfigure}
    \caption{Generalization error of the pooled min-$\ell_2$-norm interpolator under model shift with semi-synthetic data for $p = 600$ SNPs. A fixed realization of $\bbeta^{(2)}$ and $\tilde \bbeta$ are used throughout, where $\bbeta^{(2)}$ is drawn uniformly from the sphere of radius $\sqrt{\textnormal{SNR}}$ and $\tilde \bbeta$ is drawn uniformly from the sphere of radius $\sqrt{\textnormal{SSR} \cdot \textnormal{SNR}}$. SNPs are independently redrawn across trials from an HMM fit to SNPs from the 1000 Genomes Phase 3 panel. Panels (a), (b), and (c) show the risk of the pooled interpolator over varying source sample sizes $n_1$ for varying target sample sizes $n_2$, SSR, and SNR, respectively. Solid curves denote theoretical predictions, obtained from Theorem \ref{thm:model_shift} for $n_1 + n_2 < p$ and from \cite{yang2020analysis} for $n_1 + n_2 > p$. $+$ markers denote simulation averages across $50$ trials, and dashed horizontal lines denote the simulation average risk of the target-only interpolator. }
    \label{fig:semisynthetic}
\end{figure}

\subsection{Misspecification}\label{sec:misspec_sim} In Figure \ref{fig:misspecification}, we examine the effect of misspecification strength on the risk of the pooled min-$\ell_2$-norm interpolator under the misspecified model shift setting considered in Theorem \ref{thm:misspecification}. In Figure \ref{fig:misspecification}(a), we vary the source misspecification strength $\|(\bSigma_m^{(1)})^{1/2} \btheta^{(1)}\|_2^2$ while fixing the target misspecification strength $\|(\bSigma_m^{(2)})^{1/2} \btheta^{(2)}\|_2^2$ at $0$; conversely, in Figure \ref{fig:misspecification}(b), we vary the target misspecification strength $\|(\bSigma_m^{(2)})^{1/2} \btheta^{(2)}\|_2^2$ while fixing the source misspecification strength $\|(\bSigma_m^{(1)})^{1/2} \btheta^{(1)}\|_2^2$ at $0$. We note that, in both Figures \ref{fig:misspecification}(a) and \ref{fig:misspecification}(b), the empirical risk of the pooled min-$\ell_2$-norm interpolator closely matches the risk predicted by Theorem \ref{thm:misspecification}, thus again providing evidence for the strong finite-sample performance of our results.

In both Figures \ref{fig:misspecification}(a) and \ref{fig:misspecification}(b), we see that increasing the amount of misspecification increases the risk of the pooled min-$\ell_2$-norm interpolator, as expected. However, we note that different patterns emerge for the two sources of misspecification. As $n_1$ increases, the effect of source misspecification increases, as larger fractions of the pooled data are affected by the source misspecification. On the other hand, target misspecification raises the risk of the pooled min-$\ell_2$-norm interpolator both through the training labels (as captured by an increased value of $R_1(\hat \bbeta; \bbeta^{(2)}, \btheta^{(2)})$ in Theorem \ref{thm:misspecification}) and through the irreducible target prediction error (as captured by an increased value of $R_2(\hat \bbeta; \bbeta^{(2)}, \btheta^{(2)})$). Therefore, for small and moderate $n_1$, target misspecification is more harmful than analogous source misspecification. However, as $n_1$ grows, source misspecification becomes more amplified and can have a greater impact than analogous target misspecification.

\begin{figure}
    \centering
    \begin{subfigure}[b]{0.48\linewidth}
    \includegraphics[width=\linewidth]{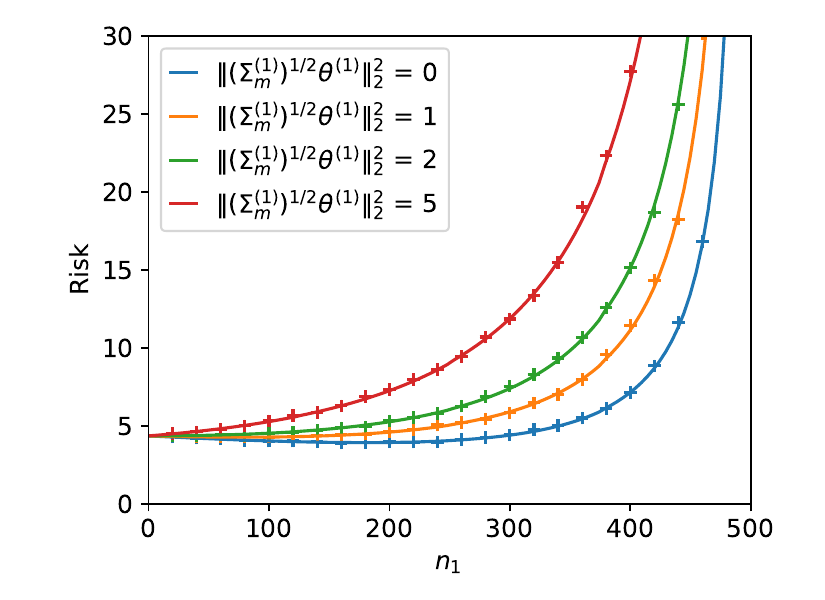}
    \caption{Fixing $\|(\bSigma^{(2)}_m)^{1/2} \btheta^{(2)}\|_2^2 = 0$}
    \end{subfigure}
    \begin{subfigure}[b]{0.48\linewidth}
    \includegraphics[width=\linewidth]{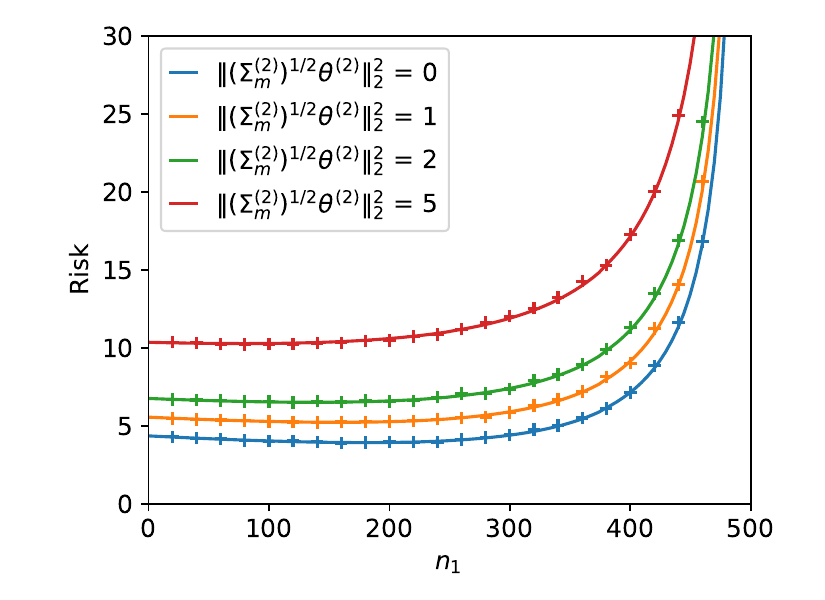}
    \caption{Fixing $\|(\bSigma^{(1)}_m)^{1/2} \btheta^{(1)}\|_2^2 = 0$}
    \end{subfigure}
    \caption{Comparison of the generalization error of the pooled min-$\ell_2$-norm interpolator under the misspecified model shift setting from Theorem \ref{thm:misspecification} as the source sample size $n_1$ varies. We fix $n_2 = 100$, $p = 600$,  $p_{m,1} = p_{m,2} = 200$, $\textnormal{SNR} = 5$, $\textnormal{SSR} = 0.2$, and $\bSigma^{(1)}_m = \bSigma^{(2)}_m = \bm I$. A fixed realization of $\bbeta^{(2)}$, $\tilde \bbeta$, $\btheta^{(1)}$, and $\btheta^{(2)}$ are used throughout, where $\bbeta^{(2)}$, $\tilde \bbeta$, $\btheta^{(1)}$, and $\btheta^{(2)}$ are drawn uniformly from spheres of radius $\sqrt{\textnormal{SNR}}$,  $\sqrt{\textnormal{SSR} \cdot \textnormal{SNR}}$, $\|(\bSigma_m^{(1)})^{1/2} \btheta^{(1)}\|_2$, and $\|(\bSigma_m^{(2)})^{1/2} \btheta^{(2)}\|_2$, respectively. Isotropic Gaussian covariates and i.i.d. $\mathcal N(0,1)$ noise are redrawn across trials. Panel (a) fixes $\|(\bSigma_m^{(2)})^{1/2} \btheta^{(2)}\|_2^2 = 0$, while Panel (b) fixes $\|(\bSigma_m^{(1)})^{1/2} \btheta^{(1)}\|_2^2 = 0$. Solid curves denote theoretical predictions, obtained from Theorem \ref{thm:misspecification}. $+$ markers denote simulation averages across $50$ trials.}
    \label{fig:misspecification}
\end{figure}

\subsection{Bias-corrected estimator}\label{sec:bias_corected_sims} We now provide empirical results related to the bias-corrected estimator in Algorithm \ref{alg:bias_corrected}.

First, in Figure \ref{fig:biascorrected}, we verify the risk predictions of Theorem \ref{thm:bias_corrected_risk} for the bias-corrected estimator across a range of SNR, SSR, sample-split parameter, and second-stage regularization parameter values. Across these parameter settings and over a range of source sample sizes $n_1$, the empirical risks closely match theoretical predictions from Theorem \ref{thm:bias_corrected_risk}, validating the strong finite-sample performance of our results for the bias-corrected estimator.
\begin{figure}
    \centering
    \begin{subfigure}[b]{0.48\linewidth}
    \includegraphics[width=\linewidth]{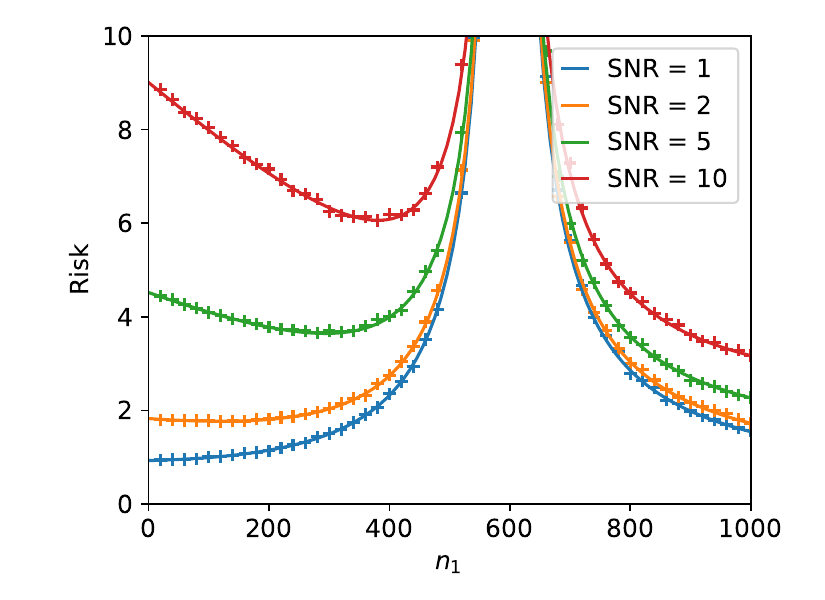}
    \caption{Fixing $\textnormal{SSR}=0.2$, $\kappa = 0.95$, $\lambda_\delta = 10$}
    \end{subfigure}
    \begin{subfigure}[b]{0.48\linewidth}
    \includegraphics[width=\linewidth]{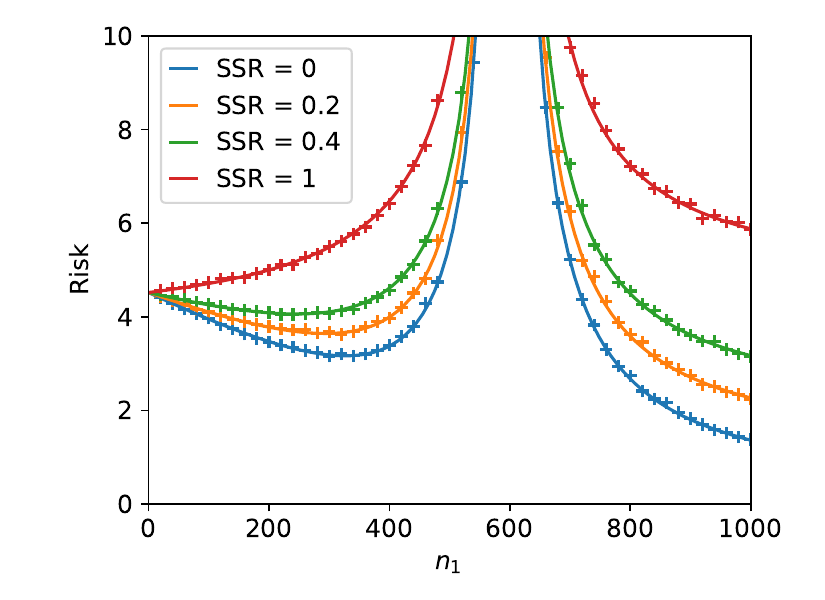}
    \caption{Fixing $\textnormal{SNR}=5$, $\kappa = 0.95$, $\lambda_\delta = 10$}
    \end{subfigure}
    \begin{subfigure}[b]{0.48\linewidth}
    \includegraphics[width=\linewidth]{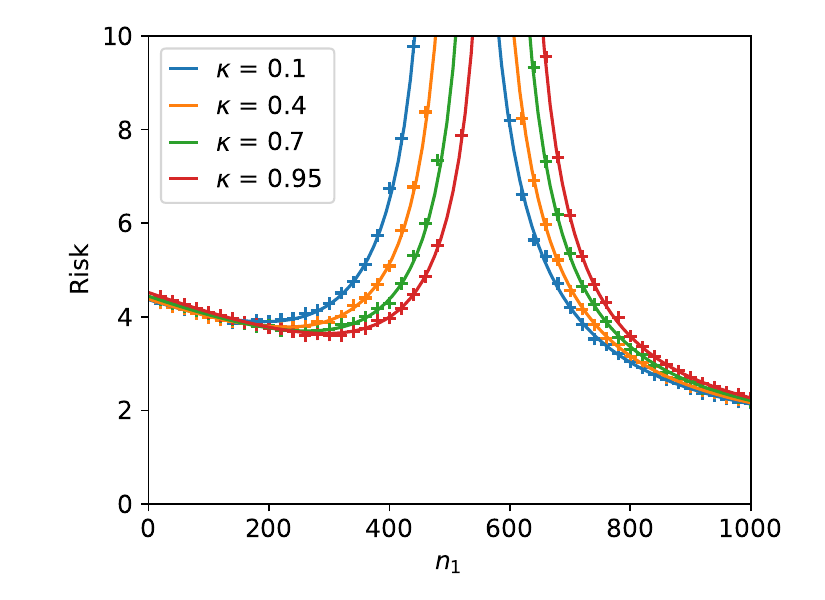}
    \caption{Fixing $\textnormal{SNR} = 5, \textnormal{SSR}=0.2$, $\lambda_\delta = 10$}
    \end{subfigure}
    \begin{subfigure}[b]{0.48\linewidth}
    \includegraphics[width=\linewidth]{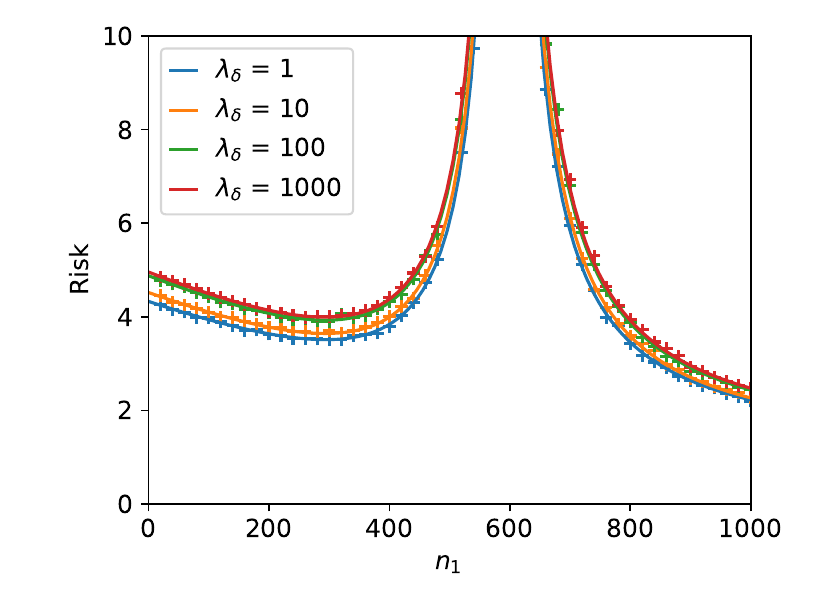}
    \caption{Fixing $\textnormal{SNR} = 5$, $\textnormal{SSR}=0.2$, $\kappa = 0.95$}
    \end{subfigure}
    \caption{Generalization error of the bias-corrected estimator under model shift with $n_2 = 100$ and $p = 600$. A fixed realization of $\bbeta^{(2)}$ and $\tilde \bbeta$ are used throughout, where $\bbeta^{(2)}$ is drawn uniformly from the sphere of radius $\sqrt{\textnormal{SNR}}$ and $\tilde \bbeta$ is drawn uniformly from the sphere of radius $\sqrt{\textnormal{SSR} \cdot \textnormal{SNR}}$. Isotropic Gaussian covariates and i.i.d. $\mathcal N(0,1)$ noise are redrawn across trials.
    Panels (a), (b), (c), and (d) show the risk of the bias-corrected estimator over varying sample sizes $n_1$ for varying SNR, SSR, $\kappa$, and $\lambda_\delta$. Solid curves denote theoretical predictions obtained from Theorem \ref{thm:bias_corrected_risk}, and $+$ markers denote simulation averages across $50$ trials.}
    \label{fig:biascorrected}
\end{figure}

Next, in Figure \ref{fig:biascorrected_compare}, we compare the performance of the bias-corrected estimator, the pooled min-$\ell_2$-norm interpolator, and the target-only min-$\ell_2$-norm interpolator across a range of source sample sizes $n_1$. Figures \ref{fig:biascorrected_compare}(a) and \ref{fig:biascorrected_compare}(b) demonstrate qualitatively similar but substantively different patterns for different SNR values. In both cases, the pooled min-$\ell_2$-norm interpolator outperforms the bias-corrected estimator for sufficiently small $n_1$, and this range increases as the SNR increases. Note also that the risk of the two transfer estimators peak at different values of $n_1$: relative to the bias-corrected estimator, the pooled estimator becomes unstable earlier.

\begin{figure}
    \centering
    \begin{subfigure}[b]{0.48\linewidth}
    \includegraphics[width=\linewidth]{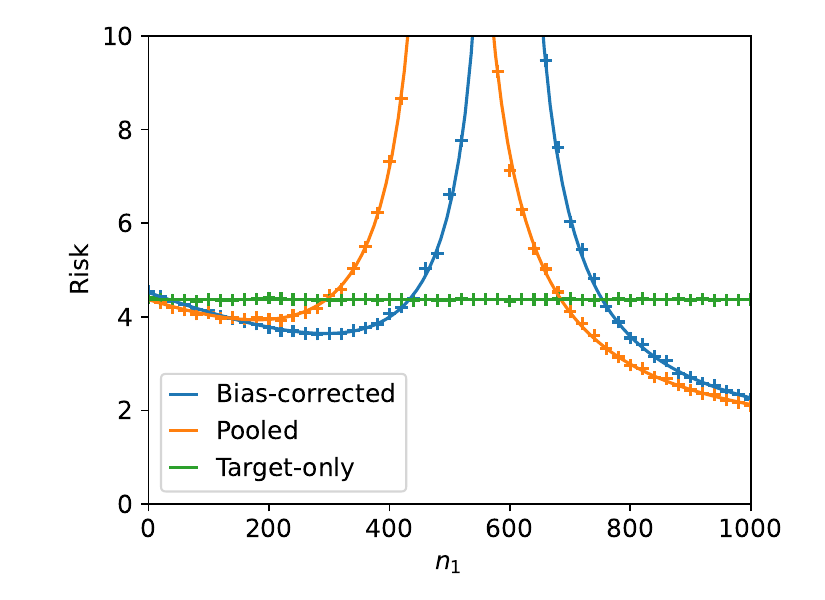}
    \caption{Fixing $\textnormal{SNR}=5$}
    \end{subfigure}
    \begin{subfigure}[b]{0.48\linewidth}
    \includegraphics[width=\linewidth]{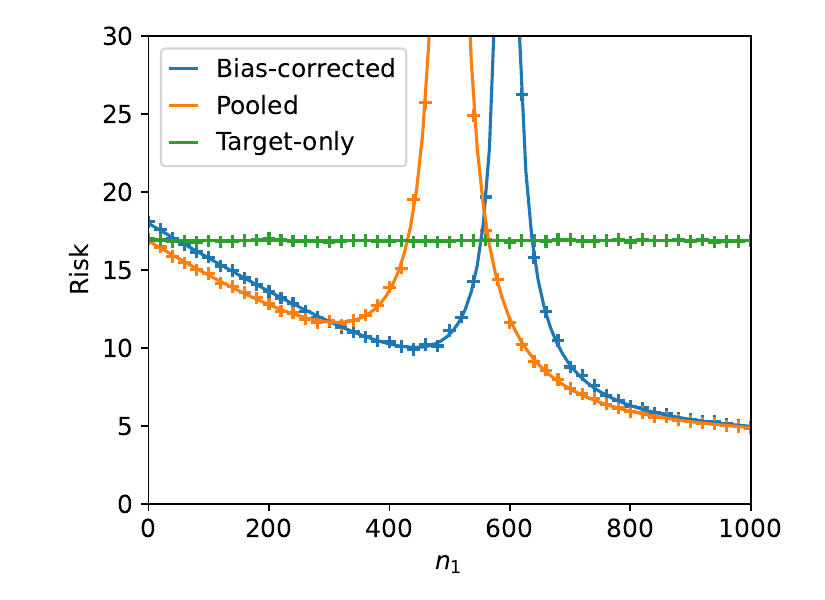}
    \caption{Fixing $\textnormal{SNR}=20$}
    \end{subfigure}
    \caption{Comparison of the generalization error of the bias-corrected estimator, pooled min-$\ell_2$-norm interpolator, and target-$\ell_2$-norm interpolator with $n_2 = 100$, $p = 600$, $\textnormal{SSR} = 0.2$, $\kappa = 0.95$, and $\lambda_\delta = 10$. A fixed realization of $\bbeta^{(2)}$ and $\tilde \bbeta$ are used throughout, where $\bbeta^{(2)}$ is drawn uniformly from the sphere of radius $\sqrt{\textnormal{SNR}}$ and $\tilde \bbeta$ is drawn uniformly from the sphere of radius $\sqrt{\textnormal{SSR} \cdot \textnormal{SNR}}$. Isotropic Gaussian covariates and i.i.d. $\mathcal N(0,1)$ noise are redrawn across trials. Panel (a) fixes $\textnormal{SNR} = 5$, while Panel (b) fixes $\textnormal{SNR} = 20$. Solid curves denote theoretical predictions, obtained from Theorem \ref{thm:model_shift}, Theorem \ref{thm:bias_corrected_risk}, and \cite[Theorem 1]{hastie2022surprises} for $n_1 + n_2 < p$ and from \cite{yang2020analysis} for $n_1 + n_2 > p$. $+$ markers denote simulation averages across $50$ trials, and dashed horizontal lines denote the simulation average risk of the target-only interpolator.}
    \label{fig:biascorrected_compare}
\end{figure}

\subsection{Comparison to null risk}\label{sec:null_comparison} We briefly note that, for sufficiently small SNR, the pooled min-$\ell_2$-norm interpolator is not guaranteed to outperform even the null estimator $\hat \bbeta_{\text{null}} = 0$. In Figure \ref{fig:null_compare}, we replicate Figures \ref{fig:isotropic_model_shift}(a) and \ref{fig:design_shift}(c), respectively, but include horizontal dotted lines to denote the risk of the null estimator, which always has risk $\E[\|\hat \bbeta_{\text{null}} - \bbeta^{(2)}\|_{\bSigma^{(2)}}^2] = \|\bbeta^{(2)}\|_2^2 = \text{SNR} \cdot \sigma^2$ in these two settings. Indeed, in both Figures \ref{fig:null_compare}(a) and \ref{fig:null_compare}(b), we see that there are typically at least some source sample sizes $n_1$ for which the pooled min-$\ell_2$-norm interpolator outperforms the null estimator. However, for $\text{SNR} = 1$ in the both figures, we see that the risk of the pooled min-$\ell_2$-norm interpolator is greater than that of the null estimator for the entire range of the overparametrized setting.

\begin{figure}
    \centering
    \begin{subfigure}[b]{0.48\linewidth}
    \includegraphics[width=\linewidth]{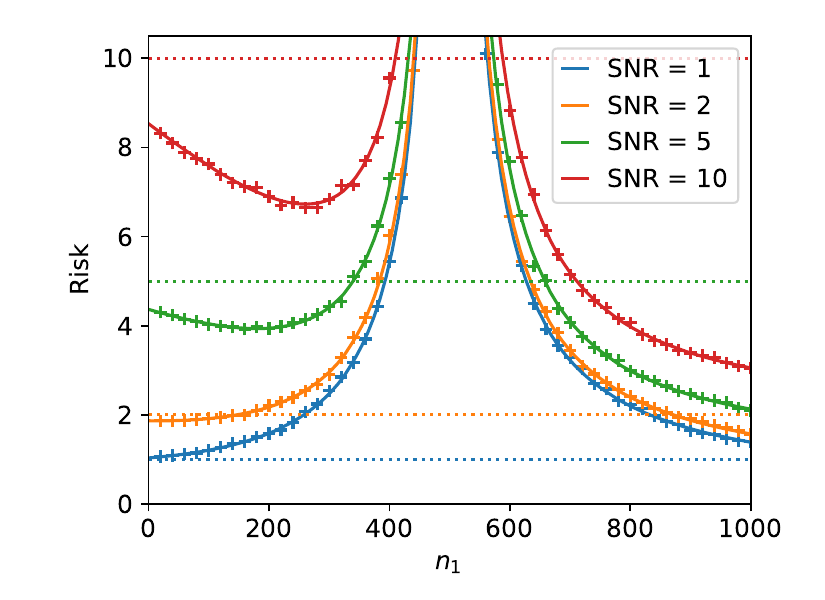}
    \caption{version of 1(a)}
    \end{subfigure}
    \begin{subfigure}[b]{0.48\linewidth}
    \includegraphics[width=\linewidth]{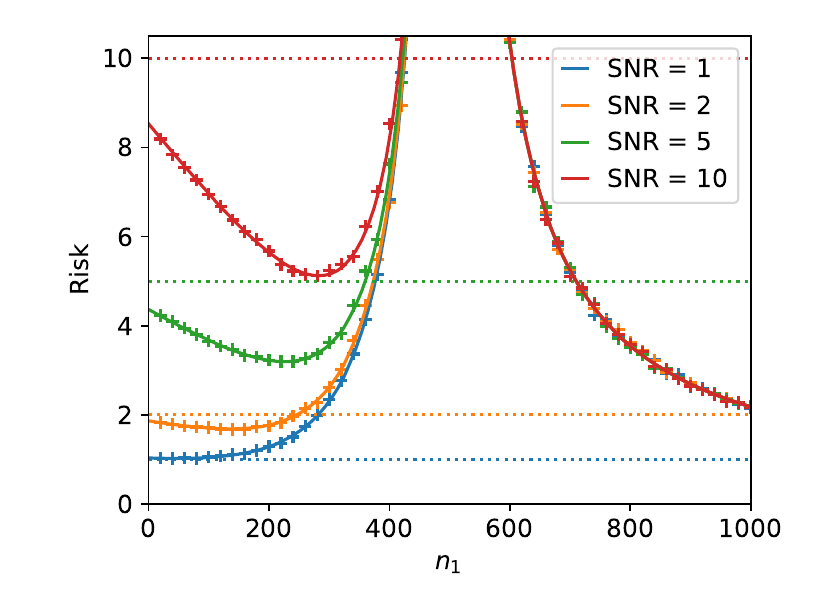}
    \caption{version of 3(c)}
    \end{subfigure}
    \caption{Replica of Figures \ref{fig:isotropic_model_shift}(a) and \ref{fig:design_shift}(c) with the null risk marked explicitly.}
    \label{fig:null_compare}
\end{figure}

\end{document}